\documentclass[11pt]{amsart}
\usepackage{mine,amsmath,amssymb,amsthm,amsfonts,mathrsfs}
\usepackage[T1]{fontenc}
\usepackage{tgtermes}
\title[Bloch-Kato for some four-dimensional symplectic Galois representations]{On the Bloch-Kato Conjecture for some four-dimensional symplectic Galois representations}

\author{Naomi Sweeting}
\address{Department of Mathematics\\Princeton University}
\email{naomiss@math.princeton.edu}
\date{2025}

\newcommand*{\home}{\mathcal H\! \mathit{om}}
\usepackage{enumitem}
\usepackage[margin=1in]{geometry}
\usepackage{url}
\DeclareMathOperator{\inner}{Int}
\DeclareMathOperator{\JL}{JL}
\DeclareMathOperator{\nerd}{Nrd}
\DeclareMathOperator{\image}{Im}
\DeclareMathOperator{\CH}{CH}
\DeclareMathOperator{\MP}{Mp}
\DeclareMathOperator{\volume}{Vol}
\newcommand{\PM}{\pm}
\let\And\relax

\DeclareMathOperator{\SPF}{Spf}
\DeclareMathOperator{\GSP}{GSp}
\DeclareMathOperator{\symmetric}{Sym}
\DeclareMathOperator{\VL}{VL}
\DeclareMathOperator{\SO}{SO}
\DeclareMathOperator{\SP}{Sp}

\DeclareMathOperator{\Lie}{Lie}
\DeclareMathOperator{\nilpotent}{Nilp}
\DeclareMathOperator{\debt}{det}
\DeclareMathOperator{\spin}{GSpin}
\DeclareMathOperator{\Shimura}{Sh}

\DeclareMathOperator{\Sym}{Sym}

\newcommand{\tors}{{\operatorname{tors}}}
\DeclareMathOperator{\CNL}{CNL}
\DeclareMathOperator{\smooth}{sm}
\DeclareMathOperator{\Res}{Res}
\DeclareMathOperator{\Restriction}{Res}

\DeclareMathOperator{\dimension}{dim}
\DeclareMathOperator{\And}{End}
\DeclareMathOperator{\Test}{Test}
\DeclareMathOperator{\degree}{deg}
\DeclareMathOperator{\Home}{Hom}
\DeclareMathOperator{\Induction}{Ind}
\DeclareMathOperator{\induction}{Ind}
\DeclareMathOperator{\kernel}{ker}
\DeclareMathOperator{\cokernel}{coker}
\DeclareMathOperator{\cind}{c-Ind}
\newcommand{\integral}{\int}
\newcommand{\IFF}{\iff}
\newcommand{\hook}{\hookrightarrow}
\newcommand{\blackboardone}{\mathbbm{1}}
\newcommand{\total}{{\operatorname{tot}}}
\newcommand{\ramified}{{\operatorname{ram}}}
\newcommand{\ram}{{\operatorname{ram}}}
\newcommand{\lr}{{\operatorname{lr}}}
\renewcommand{\check}{\vee}
\renewcommand{\special}{\star}

\DeclareMathOperator{\trace}{tr}
\DeclareMathOperator{\Sh}{Sh}
\let\ss\relax
\newcommand{\ss}{{ss}}
\DeclareMathAlphabet{\mathdutchcal}{U}{dutchcal}{m}{n}
\SetMathAlphabet{\mathdutchcal}{bold}{U}{dutchcal}{b}{n}
\DeclareMathAlphabet{\mathdutchbcal}{U}{dutchcal}{b}{n}
\newcommand{\paranormal}{\Sh_{K ^ qK_q ^ {\paramodular}} (V_{D/q})}
\newcommand{\funnyT}{\mathtt{T}}

\newcommand{\q}{{\mathdutchbcal{q}}}
\renewcommand{\p}{\mathdutchbcal{p}}
\renewcommand{\p}{\mathfrak{p}}
\newcommand{\product}{\prod}
\newcommand{\cl}{\operatorname{cl}}
\newcommand{\isomorphism}{\xrightarrow{\sim}}
\newcommand{\support}{{\operatorname{supp}}}
\newcommand{\perpendicular}{{\perp}}

\newcommand{\metaplectic}{{\operatorname{Mp}}}
\DeclareMathOperator{\WD}{WD}
\newcommand{\AJ}
{{\operatorname{AJ}}}
\newcommand{\derived}{{\operatorname{der}}}
\newcommand{\der}{{\operatorname{der}}}
\newcommand{\Leo}{\operatorname{Lie}}

\DeclareMathOperator{\SL}{SL}
\newcommand{\fitting}{{\operatorname{Fitt}}}

\newcommand{\Frobenius}{{\operatorname{Frob}}}

\newcommand{\unramified}{{\operatorname{unr}}}
\newcommand{\ordinary}{{\operatorname{ord}}}
\newcommand{\divisors}{{\operatorname{div}}}
\newcommand{\congruent}{{\operatorname{cong}}}
\newcommand{\Selmer}{{\operatorname{Sel}}}
\newcommand{\Summer}{{\operatorname{Sel}}}
\newcommand{\summer}{{\operatorname{Sel}}}
\newcommand{\adjoint}{{\operatorname{ad}}}
\newcommand{\Adjoint}
{{\operatorname{ad}}}
\newcommand{\rel}{{\operatorname{rel}}}
\newcommand{\length}{{\operatorname{lg}}}
\newcommand{\annihilator}{{\operatorname{Ann}}}
\newcommand{\twist}{{\operatorname{twist}}}
\DeclareMathOperator{\automorphisms}{Aut}
\DeclareMathOperator{\neck}{\neq}
\DeclareMathOperator{\PGL}{PGL}
\DeclareMathOperator{\Ind}{Ind}
\DeclareMathOperator{\BC}{BC}
\DeclareMathOperator{\Span}{Span}
\DeclareMathOperator{\Spf}{Spf}

\DeclareMathOperator{\determinant}{det}
\newcommand{\Extension}{{\operatorname{Ext}}}

\newcommand{\approximate}{{\approx}}
\newcommand{\et}{{\operatorname{\acute{e}t}}}
\newcommand{\identity}{{\operatorname{id}}}
\newcommand{\union}{{\cup}}

\newcommand{\cyc}{{\operatorname{cyc}}}
\newcommand{\sick}{{\operatorname{cyc}}}
\newcommand{\Siegel}{\operatorname{Sie}}
\DeclareMathOperator{\shimmer}{Sh}

\newcommand{\nomad}{\nmid}
\newcommand{\universal}{{\operatorname{univ}}}
\let\univ\relax
\newcommand{\univ}{{\operatorname{univ}}}
\DeclareMathOperator{\Max}{max}
\newcommand{\pseudodeformations}
{{\operatorname{PsDef}}}
\let\red\relax
\newcommand{\dR}{{\operatorname{dR}}}
\newcommand{\red}{{\operatorname{red}}}
\newcommand{\Prod}{\prod}
\DeclareMathOperator{\Ord}{ord}
\DeclareMathOperator{\order}{ord}

\newcommand{\tw}{{\operatorname{tw}}}
\newcommand{\sing}{{\operatorname{sing}}}
\DeclareMathOperator{\pr}{pr}
\newcommand{\Crystal}{{\operatorname{cris}}}
\newcommand{\crystal}{{\operatorname{cris}}}
\newcommand{\local}{\operatorname{loc}}
\newcommand{\half}{{\frac{1}{2}}}

\newcommand{\SC}{{\operatorname{SC}}}
\newcommand{\quiver}{\equiv}
\newcommand{\paramodular}{{\operatorname{Pa}}}
\newcommand{\multiplicity}{\operatorname{mult}}
\newcommand{\Siegelnormal}{\Sh_{K ^ qK_q ^ {\Siegel}} (V_{D/q})}
\newcommand{\hypernormal}{\Sh_{K ^ qK_q} (V_{D/q})}

\DeclareMathOperator{\inc}{inc}
\DeclareMathOperator{\Nm}{Nm}
\DeclareMathOperator{\HT}{HT}
\begin{document}
\begin{abstract}
    The Bloch-Kato conjecture predicts a far-reaching connection between orders of vanishing of $L$-functions and the ranks of Selmer groups of $p$-adic Galois representations. In this article, we consider the four-dimensional, symplectic  Galois representations arising from automorphic representations $\pi$ of $\GSP_4(\A_\Q)$ with trivial central character and with the lowest cohomological archimedean weight. Under mild technical conditions, we prove that the Selmer group vanishes when the central value $L(\pi,\operatorname{spin},1/2)$ is nonzero.    
    In the spirit of bipartite Euler systems, we bound the Selmer group by using level-raising congruences to construct ramified Galois cohomology classes. The relation to $L$-values comes via the $\spin_3\hookrightarrow \spin_5$ periods on a compact inner form of $\GSP_4$. We also prove a result towards the rank-one case:
     if the $\pi$-isotypic part of the  Abel-Jacobi image of any of Kudla's one-cycles on the Siegel threefold is nonzero, it generates the full Selmer group. These cycles are linear combinations of embedded quaternionic Shimura curves, and under the conjectural arithmetic Rallis inner product formula, their heights  are related to $L'(\pi,\operatorname{spin},1/2)$.
\end{abstract}
\maketitle
\tableofcontents

\numberwithin{equation}{section}
\newpage
\setcounter{section}{-1}
\section{Introduction}
\subsection{Main results}\label{subsec:intro_main}
Let $\pi=\otimes' \pi_v$ be a cuspidal automorphic representation of $\GSP_4(\A_\Q)$ of trivial central character and conductor $N(\pi)$, such that $\pi_\infty$ belongs to the discrete series $L$-packet of parallel weight $(3,3)$. In particular, $\pi$ appears in the \'etale cohomology of the $\GSP_4$ Shimura variety with trivial coefficients. Let $E$ be a sufficiently large coefficient field; then we have a compatible family of $p$-adic Galois representations
$$\rho_{\pi,\p}: G_\Q \to \GSP_4(E_\p)$$ 
indexed by primes $\p|p$ of $E$, and normalized so that the similitude character of $\rho_{\pi,\p}$ is cyclotomic. Let $V_{\pi,\p}$ be the underlying four-dimensional $E_\p[G_\Q]$-module of $\rho_{\pi,\p}$, and consider the Bloch-Kato Selmer group 
$$H^1_f(\Q, V_{\pi,\p}) = \ker\left(H^1(\Q, V_{\pi,\p}) \to H^1(\Q_p, V_{\pi,\p}\otimes B_\crystal) \times \prod_{\l\neq p} H^1(I_{\Q_\l}, V_{\pi,\p})\right),$$ where $I_{\Q_\l}$ is the local inertia subgroup. 
The Bloch-Kato conjecture applied to $V_{\pi,\p}$ predicts
$$\dim_{E_\p} H^1_f(\Q, V_{\pi,\p}) = \ord_{s = 1/2} L(\pi,\operatorname{spin}, 1/2),$$ where the $L$-function is normalized so that $s = 1/2$ is the central value.  Our first main result proves many cases of the Bloch-Kato conjecture for $V_{\pi,\p}$ in rank zero.
\begin{thmintro}[Theorem \ref{thm:rk_zero_main}]\label{thm:intro_rk0}
    Suppose $\pi$  is not CAP or endoscopic, and for some $\l|N(\pi)$, $\pi_\l$ has a local Jacquet-Langlands transfer to the compact inner form of $\GSP_{4,\Q_\l}$. Let $\p|p$ be a prime of $E$ such that:
    \begin{enumerate}

        \item\label{item:intro_rk0_1} $\pi_p$ is unramified.
        \item\label{item:intro_rk0_2} The residual representation $\overline \rho_{\pi,\p}$ is absolutely irreducible and generic (Definition \ref{def:generic}).
        \item \label{item:intro_rk0_3}There exists a prime $q\nmid N(\pi)$ such that $q^4\not\equiv 1\pmod p$, and $\overline\rho_{\pi,\p}(\Frob_q)$ has eigenvalues $$\set{q, 1, \alpha, q/\alpha}$$ with $\alpha\not\in \set{\pm 1, \pm q, q^2, q^{-1}}$.
    \end{enumerate}
    Then $$L(\pi, \operatorname{spin}, 1/2) \neq 0 \implies H^1_f(\Q, V_{\pi,\p})=0.$$
\end{thmintro}

  \begin{rmksunnumbered}
  \begin{enumerate}[label = (\roman*)]
      \item  The conditions (\ref{item:intro_rk0_1}) and (\ref{item:intro_rk0_2}) are always satisfied for cofinitely many $\p$. Condition (\ref{item:intro_rk0_3}) is always satisfied for the primes $\p$ lying over a set of rational primes of positive Dirichlet density; it is satisfied for all but finitely many $\p$ under conditions listed in Theorem \ref{thm:when_adm_primes}. 
                      \item When $p > 5$, which is necessary for (\ref{item:intro_rk0_3}),  $\pi$ is  CAP or endoscopic if and only if $V_{\pi,\p}$ is reducible.

      \item For all but finitely many of the primes $\p$ satisfying the conditions in Theorem \ref{thm:intro_rk0}, we are able to strengthen the result to the 
     vanishing of the  dual Selmer group $H^1_f(\Q, V_{\pi,\p}/T_{\pi,\p})$, where $T_{\pi,\p}\subset V_{\pi,\p}$ is a Galois-stable lattice; see Definition \ref{def:divisible_Selmer} and Corollary \ref{cor:rk_zero_integral}.
  \end{enumerate}
      \end{rmksunnumbered}
    For instance, from Theorem \ref{thm:intro_rk0} we can deduce the following: 
    \begin{corintro}[Corollary \ref{cor:rk_zero_integral_IIa}]
        Suppose $\pi$ is not CAP or endoscopic, and there exists a prime $\l|N(\pi)$ such that $\pi_\l$ is of type IIa in the notation of \cite{roberts2007local}. Then $$L(\pi,\operatorname{spin}, 1/2) \neq 0 \implies H^1_f(\Q, V_{\pi,\p}/T_{\pi,\p}) = 0$$ for all but finitely many $\p$. 
    \end{corintro}
    In \S\ref{subsec:corollaries_rk0}, we give  applications of Theorem \ref{thm:intro_rk0} to certain $\pi$ arising by automorphic induction; this gives new results towards Bloch-Kato for Hilbert modular forms over real quadratic fields of non-parallel weight $(2,4)$ (including CM forms), and for twists of classical modular forms of weight 3 by certain Hecke characters of imaginary quadratic fields.

    Our second main result concerns the rank one case of Bloch-Kato for $V_{\pi,\p}$. To state it,
    let $V$ be a quadratic space of signature $(3,2)$, and suppose $\pi$ admits a Jacquet-Langlands transfer $\Pi$ to the inner form $\spin(V)$ of $\GSP_4$; for instance, if $V$ is split, then $\spin(V) = \GSP_4$ and we can take $\Pi = \pi$. Let $K \subset \spin(V)(\A_f)$ be a neat compact open subgroup such that $\Pi_f^K \neq 0$, and let $\Sh_K(V)$ be the Shimura variety for $\spin(V)$ at level $K$, which is  a classical Siegel threefold when $V$ is split.  
 Let $\p|p$ be a prime of $E$; then we have
\begin{equation}\label{eq:intro_etale}H^i_\et(\Sh_K(V)_{\overline \Q}, E_\p)[\Pi_f] =\begin{cases} \Pi_f^K \otimes V_{\pi,\p}(-2), & i = 3,\\ 0, & i\neq 3.\end{cases}\end{equation}
In particular, if $\CH^2(\Sh_K(V))_{\m_{\pi_f}}$ denotes the Chow group of codimension-two algebraic cycles,  localized at the maximal ideal of an appropriate Hecke algebra corresponding to the eigenvalues of $\pi$, we have an Abel-Jacobi map
\begin{equation}
\begin{split}
   \partial_{ \AJ,{\Pi_f}}: \CH^2(\Sh_K(V))_{\m_{\pi_f}} &\to H^1_f(\Q, H^3_\et(\Sh_K(V)_{\overline \Q}, E_\p(2)) [\Pi_f]) \\ &= H^1_f(\Q, V_{\pi,\p}) \otimes \Pi_f^K.
   \end{split}
\end{equation}
We consider the codimension-two Kudla cycles
\begin{equation}
    \label{eq:intro_kudla}
    Z(T,\phi) \in \CH^2(\Sh_K(V)),
\end{equation} which are indexed by the data of a $2\times 2$ symmetric, positive-semidefinite matrix $T$ and a $K$-invariant test function $\phi \in \mathcal S(V^2\otimes \A_f, \Z).$
When $T$ is nondegenerate, $Z(T, \phi)$ is a linear combination of Shimura curves over $\Q$, embedded into $\Sh_K(V)$ by $\spin(V)(\A_f)$-translates of group maps of the form
$$\spin(1,2) \hookrightarrow\spin(3,2);$$
conversely, any such embedded Shimura curve can be written as a Kudla cycle for some $T$ and $\phi$.
\begin{thmintro} [Theorem \ref{thm:rk1_non_endoscopic_ultimate}]\label{thm:intro_rk1}
Suppose $\pi$ is not CAP or endoscopic, and its base change to $\GL_4(\A_\Q)$ does not arise by automorphic induction from an imaginary quadratic or quartic CM field.
Let $\p|p$ be a prime of $E$ satisfying (\ref{item:intro_rk0_1})-(\ref{item:intro_rk0_3}) from Theorem \ref{thm:intro_rk0}.
Then for any Kudla cycle $Z(T, \phi) \in \CH^2(\Sh_K(V))$,
$$\partial_{\AJ, \Pi_f} (Z(T, \phi)) \neq 0 \implies \dim_{E_\p} H^1_f(\Q, V_{\pi,\p}) = 1.$$
    
\end{thmintro}
We now explain the relation of the Bloch-Kato Conjecture to Theorem \ref{thm:intro_rk1}, in which $L$-values do not directly appear. According to the resolution by Bruinier and Westerholt-Raum of Kudla's modularity conjecture \cite{bruinier2015kudla}, the formal $q$-series
\begin{equation}
  \Theta_\phi\coloneqq   \sum_{T \in \Sym_2(\Q)_{\geq 0}} Z(T, \phi) q^T  
\end{equation}
lies in $\CH^2(\Sh_K(V)) \otimes_\Z M_{5/2,2}$,
where $M_{5/2, 2}$ is the space of holomorphic Siegel modular forms 
of degree 2 and weight 5/2. Let $f\in M_{5/2, 2}$ generate the automorphic representation of the metaplectic group $\MP_4(\A_\Q)$ corresponding to $\Pi$ under the generalized Shimura-Waldspurger correspondence of Gan-Li \cite{gan2018shimura}; then the Petersson inner product $\Theta_\phi(f) \coloneqq \langle \Theta_\phi, f\rangle $ lies in $\CH^2(\Sh_K(V))_\C$. The still-conjectural\footnote{But see the important progress towards this result by Li-Zhang \cite{li2022arithmetic}, and the unitary case proved by Li-Liu \cite{li2021chowI, li2022chowII}.} \emph{arithmetic Rallis inner product formula} proposes that, up to some local factors depending on $\phi$ and $f$, 
the derivative $L'(\pi, \operatorname{spin}, 1/2)$ is related to the height of the cycle $\Theta_\phi(f)$ (suitably interpreted). Assuming Beilinson's conjecture on the injectivity of the $p$-adic Abel-Jacobi map, Theorem \ref{thm:intro_rk1} could then be reformulated as 
$$L'(\pi,\operatorname{spin}, 1/2)\neq 0 \implies \dim_{E_\p} H^1_f(\Q, V_{\pi,\p}) = 1,$$
   which is consistent with the Bloch-Kato conjecture.
\subsection{Overview of the proofs}\label{subsec:overview_intro}
Before giving more detailed sketches below, we briefly indicate the methods of proof of Theorems \ref{thm:intro_rk0} and \ref{thm:intro_rk1}. Pursuing a strategy initiated by Bertolini and Darmon for elliptic curves \cite{bertolini2005iwasawa} -- and later extended to many new contexts by other authors
\cite{
chida2015anticyclotomic,corato2023spherical,
liu2016hirzebruch,
liu2019bounding,liu2020supersingular,liu2022beilinson,wang2022triple} -- we bound the Selmer group by constructing ramified Galois cohomology classes through level-raising congruences and special cycles on Shimura varieties for  ramified $\spin_5$ groups. Proving the classes we construct are ramified is the most delicate part: by a calculation on the special fiber, the ramification is essentially measured by linear combinations of compact $\spin_3-\spin_5$ periods for a Jacquet-Langlands transfer of $\pi$.\footnote{Or a congruent automorphic form to $\pi$, in the case of Theorem \ref{thm:intro_rk1}.} To access the underlying representation theory of these periods and relate them to $L$-values, we interpret them as the Fourier coefficients of certain theta lifts in $M_{5/2,2}$.

\subsection{Sketch of the proof in the rank zero case}\label{subsec:intro_sketch_rk0}
Let us explain the main ideas involved in the proof of Theorem \ref{thm:intro_rk0}. Fix a $G_\Q$-stable lattice $T_{\pi,\p} \subset V_{\pi,\p}$, and let $T_{\pi,n} \coloneqq T_{\pi,\p}/\p^n T_{\pi,\p}$ for $n \geq 1$.   The mechanism for bounding the Selmer group is a collection of auxiliary  Galois cohomology classes
$$\kappa_n(q) \in H^1(\Q, T_{\pi,n}),$$ indexed by primes $q$ satisfying appropriate level-raising congruence conditions modulo $\p^n$, and 
 having the following two properties:
\begin{enumerate}
    \item The restriction $\Res_\l \kappa_n(q)$ is unramified for all $\l\nmid N(\pi)p q$, and crystalline for $\l = p$.
    \item Under the assumption $L(\pi, \operatorname{spin}, 1/2) \neq 0$,  $\Res_q \kappa_n(q)$ is ramified (if $n$ is sufficiently large). 
\end{enumerate}
Given such a system of classes, a standard argument using Poitou-Tate duality shows that $H^1_f(\Q, V_{\pi,\p})$ vanishes. 

As indicated in \S\ref{subsec:overview_intro}, we obtain the classes $\kappa_n(q)$ by level-raising congruences from special cycles on  Shimura varieties for  ramified $\spin_5$ groups. More precisely, fix the prime $\l|N(\pi)$ such that $\pi_\l$ has a local Jacquet-Langlands transfer (we will see later that this is crucial for the argument). Let $V_{q\l}$ be the unique five-dimensional quadratic space of signature $(3,2)$ and trivial discriminant that ramifies precisely at $q$ and $\l$; one exists by the local-global classification of quadratic forms. On the corresponding $\spin$ Shimura variety $\Sh(V_{q\l})$ we have the codimension-two Kudla cycles $Z(T, \phi) \in \CH^2(\Sh(V_{q\l}))$ as in \S\ref{subsec:intro_main}. (For simplicity, we suppress  the choice of level structure.)
After localizing at the maximal ideal $\m$ of the Hecke algebra corresponding to $\overline\rho_{\pi,\p}$, the main result of \cite{hamann2023torsion} implies that $H^4(\Sh(V_{q\l}), O_\p)_\m = 0$, where $O \subset E$ is the ring of integers --  this is the most crucial way that the condition (\ref{item:intro_rk0_2}) of Theorem \ref{thm:intro_rk0} enters the argument. In particular, one has an Abel-Jacobi map
$$\partial_{\AJ, \m}: \CH^2(\Sh(V_{q\l}))_\m \to H^1(\Q, H^3_\et(\Sh(V_{q\l})_{\overline\Q}, O_\p(2))_\m).$$
However, it is no longer true that $V_{\pi,\p}$ appears in the \'etale cohomology $H^3_\et(\Sh(V_{q\l})_{\overline \Q}, E_\p)$, because, as $\pi_q$ is unramified, $\pi$ does not have a Jacquet-Langlands transfer to $\spin(V_{q\l})$.

Instead, we construct \emph{level-raising maps} (see below)
$$\alpha_n: H^3_\et(\Sh(V_{q\l})_{\overline\Q}, O_\p(2))_\m \to T_{\pi,n}.$$ Then we obtain a family of Galois cohomology classes $$\kappa_n(q, T, \phi, \alpha_n) \coloneqq \alpha_n \circ \partial_{\AJ,\m} (Z(T,\phi))\in H^1(\Q, T_{\pi,n}),$$ for varying $T$ and $\phi\in \mathcal S(V_{q\l}^2 \otimes \A_f, \Z)$. Any $\kappa_n(q, T, \phi, \alpha_n)$ will satisfy 
the property (1) above, as long as the level structure for $\Sh(V_{q\l})$ is chosen to be hyperspecial outside $N(\pi)q$. The proof of (2) is far more subtle, as we explain below, and we are only able to show that \emph{some} $\kappa_n(q, T,\phi,\alpha_{n})$ is ramified at $q$; in particular, the choice of $T$ cannot be made explicit. Once we have any ramified class, however, we can use it as the class called $\kappa_n(q)$ above.
\subsubsection{Ramification and the relation to $L$-values}

Let $V_\l$ be the unique five-dimensional, positive-definite quadratic space of trivial discriminant that ramifies only at $\l$. By the assumption that $\pi_\l$ has a local Jacquet-Langlands transfer, one can find a function
\begin{equation}\label{eq:beta_pi_intro}\beta_\pi: \spin(V_\l)(\Q) \backslash \spin(V_\l)(\A_f)/K \to O\end{equation} with the same spherical Hecke eigenvalues as $\pi$, where $K$ is a sufficiently small compact open subgroup, with $K_q$ hyperspecial. We abbreviate this double coset space as $\Sh(V_\l)$. Because $\pi$ has trivial central character, $\beta_\pi$ descends to an automorphic form on $\SO(V_\l)$; in particular, 
for any $\phi\in \mathcal S(V_\l^2\otimes \A_f , \Z)$, one can consider the classical theta lift
$$\Theta_\phi(\beta_\pi) \in M_{5/2,2}\otimes_\Z O.$$
The classical Rallis inner product formula computes the Petersson norm of $\Theta_\phi(\beta_\pi)$:
\begin{equation}
    \langle \Theta_\phi(\beta_\pi), \Theta_\phi(\beta_\pi) \rangle \doteq \langle \beta_\pi, \beta_\pi\rangle \cdot L(\pi, \operatorname{spin}, 1/2)
\end{equation}
 up to local factors depending on $\phi$ and $\beta_\pi$.
   
On the other hand, the Fourier coefficients of $\Theta_\phi(\beta_\pi)$ can be  computed as $\spin_3$-periods of $\beta_\pi$. In fact, for $T\in \Sym_{2}(\Q)_{\geq 0}$ and $\phi\in \mathcal S(V_\l^2\otimes \A_f, \Z)$, one can define $Z(T, \phi) \in \Z[\Sh(V_\l)]$ analogously to Kudla's cycles in (\ref{eq:intro_kudla}), arising from group embeddings $\spin_3\hookrightarrow \spin(V_\l)$. Then we have:
$$\Theta_\phi(\beta_\pi) 
= \sum_{T\in \Sym_2(\Q)_{\geq 0}}  \beta_\pi(Z(T, \phi))q^T.$$
In particular,
\begin{equation}\label{eq:intro_L_Z}L(\pi,\operatorname{spin}, 1/2) \neq 0 \iff\; \exists\, T, \phi \text{ s.t. } \beta_\pi(Z(T, \phi)) \neq 0.\end{equation}
The connection to Galois cohomology classes comes from the following (idealized) identity, for a certain choice of $\alpha_n$: 
\begin{equation}\label{eq:intro_magic_identity}\Res_q\kappa_n (q, T, \phi^q\otimes \phi_q^\ram, \alpha_n) = \beta_\pi(Z(T, \phi^q\otimes \phi_q^{\unr})) 
 \pmod {\p^n}.\end{equation}
 Here $\phi^q\in \mathcal S(V_{q\l}^2\otimes \A_f^q, \Z) \simeq \mathcal S(V_\l^2 \otimes \A_f^q, \Z)$ is any test function, $\phi_q^\ram\in \mathcal S(V_{q\l}^2 \otimes \Q_q, \Z)$ is  designed so that $Z(T, \phi^q\otimes \phi_q^\ram) \in \CH^2(\Sh(V_{q\l}))$ will have a reasonable integral model over $\Z_{(q)}$ -- in fact, $Z(T, \phi^q\otimes \phi_q^\ram)$ is a linear combination of quaternionic Shimura curves ramified at $q$, with the usual maximal compact level structure -- and $\phi_q^{\unr}\in \mathcal S(V_\l^2 \otimes \Q_q, \Z)$ is an explicit test function whose origin is described in the next paragraph.
Finally, to make sense of the identity (\ref{eq:intro_magic_identity}), we use that $H^1(I_{\Q_q}, T_{\pi,n})^{\Frob_q = 1}$ is free of rank one over $O/\p^n$, by the congruence conditions imposed on $q$; so we are viewing both sides  as elements of $O/\p^n$. 
 
The main tool for proving (\ref{eq:intro_magic_identity}) is an analysis of the weight spectral sequence for a semistable integral model of $\Sh(V_{q\l})$ over $\Z_{(q)}$; this is obtained by blowup from the canonical PEL model, which has ordinary quadratic singularities. The double coset space $\Sh(V_\l)$ enters the picture as the indexing set for the irreducible components of the (two-dimensional) supersingular locus of the special fiber of $\Sh(V_{q\l})$, and indeed $\phi_q^\unr$ arises from considering the intersections of the special fiber of $Z(T, \phi^q\otimes\phi_q^\ram)$ with various strata in the supersingular locus.
 (Of course, for this discussion to be accurate, the level structures for $\Sh(V_{q\l})$ and $\Sh(V_\l)$ are chosen to be compatible away from $q$.)

However, even once we have proved (\ref{eq:intro_magic_identity}), one major obstacle remains to proving  that at least one of the classes $\Res_q\kappa_n(q, T, \phi^q\otimes\phi_q^\ram, \alpha_n)$ is ramified.   From (\ref{eq:intro_L_Z}), one can deduce that there exists $\phi^q$ and $T$ such that $\beta_\pi(Z(T, \phi^q\otimes \phi_q^{\operatorname{sph}})) \neq 0 \pmod {\p^n}$ for sufficiently large $n$, where $\phi_q^{\operatorname{sph}}$ is the indicator function of a self-dual lattice. Unfortunately, although $\phi_q^\unr$ is invariant by a hyperspecial subgroup $K_q$ of $\spin(V_{\l})(\Q_q)$, it is not equal to $\phi_q^{\operatorname{sph}}$. To get around this, we give a  criterion on $K_q$-invariant functions $\phi_q\in \mathcal S(V_\l^2 \otimes \Q_q, \Z)$ under which we can prove:
\begin{equation}\label{eq:intro_change_test}\beta_\pi(Z(T, \phi^q\otimes\phi_q^{\operatorname{sph}})) \neq 0 \pmod {\p^n} \implies \exists\, T'\text{ s.t. }\beta_\pi(Z(T', \phi^q\otimes \phi_q))\neq 0 \pmod {\p^n}.\end{equation}
The function $\phi_q^{\unr}$ satisfies the criterion, so indeed we can use (\ref{eq:intro_L_Z}), (\ref{eq:intro_magic_identity}), and (\ref{eq:intro_change_test}) to obtain the local ramification of some class $\Res_q\kappa_n(q, T', \phi,\alpha_n)$. 
The lack of control over $T'$ (in particular its possible dependence on $q$) is the reason that we must consider the full family of classes $\kappa_n(q, T, \phi, \alpha_n)$.

The proof of (\ref{eq:intro_change_test}) uses that the theta lift map
$$\Theta_{-} (\beta_\pi): \mathcal S(V_\l^2 \otimes \A_f, \Z) \to M_{5/2,2} \otimes_\Z O$$ factors as a product of local maps, defined purely in terms of the local Weil representation. Thus (\ref{eq:intro_change_test}) can be reduced to a local question about the mod-$p$  Weil representation of $\MP_4(\Q_q)\times \SO(V_{\l})(\Q_q)$ on $\mathcal S(V_\l^2, \overline \F_p)$.

\subsubsection{Level-raising}\label{subsubsec:intro_lr}
In the discussion above, we asserted the existence of a level raising map \begin{equation}\label{eq:alpha_n_intro}
    \alpha_n : H^3_\et (\Sh(V_{q\l})_{\overline\Q}, O_\p(2))_\m \to T_{\pi,n}
\end{equation} such that the classes $\kappa_n(q, T, \phi^q\otimes \phi_q^\ram, \alpha_n)$ satisfy the identity (\ref{eq:intro_magic_identity}). We now explain the construction of $\alpha_n$ in more detail. 
The weight spectral sequence for $\Sh(V_{q\l})$ gives rise to a diagram of the following form:

\begin{equation}\label{eq:cd_intro}
\begin {tikzcd}
M_{-1} H ^ 1\left (I_{\Q_q}, H ^ 3_{\et} (\Shimura (V_{q\l})_{\overline\Q}, O_\p (2))_\m\right) \arrow [r, hook]\arrow [d, twoheadrightarrow, "\xi"] &H ^ 1\left (I_{\Q_q}, H ^ 3_{\et} (\Shimura(V_{q\l})_{\overline\Q}, O_\p (2))_\m\right) ^ {\Frobenius_q = 1}\arrow[ddl, dashed, "\alpha_{n,\ast}"]\\
O_\p\left [\Shimura(V_{\l})\right]_\m/(\funnyT^\lr_q) \arrow[d,"\beta_\pi"] \\
O/\p^n\simeq H^1(I_{\Q_q}, T_{\pi,n})^{\Frob_q = 1},
\end{tikzcd}
\end{equation}
where $\funnyT^\lr_q$ is a certain spherical Hecke operator measuring the level-raising congruence at the prime $q$;  $M_{-1}$ refers to the monodromy filtration;  and the equality in the bottom row uses the level-raising condition on $q$. Any level-raising map as in (\ref{eq:alpha_n_intro}) induces a map
$\alpha_{n, \ast}$ as in (\ref{eq:cd_intro}), not necessarily making the diagram commute.

Using an idea of Scholze \cite{scholze2018lubintate} on ``typic-ness'' of Galois modules, we show that the data of a Hecke-equivariant $\alpha_{n, \ast}$ is actually \emph{equivalent} to the data of $\alpha_n$. Moreover, the identity
(\ref{eq:intro_magic_identity}) is more or less equivalent to the commutativity of the diagram (\ref{eq:cd_intro}). So our task is to find the Hecke-equivariant dashed arrow in (\ref{eq:cd_intro}) lifting $\beta_\pi\circ \xi$.

One approach to finding the lift in (\ref{eq:cd_intro}) would be Taylor-Wiles patching, in the spirit of \cite{liu2024deformation}.
This would show that $H^3_\et(\Sh(V_{q\l})_{\overline\Q}, O_\p)_\m$ is free as a Hecke module, and as a byproduct that 
the top arrow in (\ref{eq:cd_intro}) is an isomorphism.
However, the drawback of the patching argument is that it requires some strong ``residual ramification'' conditions on $\overline\rho_{\pi,\p}|_{G_{\Q_\l}}$, for primes $\l|N(\pi)$; in general, one expects that the Hecke-module structure of $H^3_\et(\Sh(V_{q\l})_{\overline\Q}, O_\p)_\m$ may be more complicated.

Instead of taking this route, our method still uses Galois deformation theory, but now only to obtain a rather coarse quantitative control on the possible congruences between $ \beta_\pi\circ \xi$ and other Hecke eigensystems in $H^1(I_{\Q_q}, H^3_\et(\Sh(V_{q\l})_{\overline\Q}, O_\p(2))_\m)$. Indeed, if there were no congruences at all, then (\ref{eq:cd_intro}) would just be a diagram of $O/\p^n$-modules, and the  lift $\alpha_{n, \ast}$ would exist for trivial reasons. In general, we are able to lift $\beta_\pi\circ \xi$ only after multiplying by a generator of  $\p^C$, where $C$ is essentially the length of a certain adjoint Selmer group for $T_{\pi,n}$.
We can control $C$  using a theorem of Newton and Thorne \cite{newton2023thorne, thorne2022vanishing} that $H^1_f(\Q, \ad^0\rho_{\pi,\p}) = 0;$ since $C$ is independent of $n$ and $q$ (and actually 0 for cofinitely many $\p$) its appearance causes no harm in the argument.\footnote{We remark that the theorem of Newton and Thorne still uses Taylor-Wiles patching, but in a more flexible context.} 
\subsection{Sketch of the proof in the rank one case}\label{subsec:intro_sketch_rk1}
Let $V$ be the quadratic space from \S\ref{subsec:intro_main}. For an auxiliary  prime $q_1$ satisfying a level-raising congruence condition, we consider the ``nearby''  quadratic space $V_{q_1}$, which has trivial discriminant and signature $(5,0)$, and is ramified precisely at $q_1$ and the primes of ramification for $V$. 

 The key point to prove Theorem \ref{thm:intro_rk1} is to produce a Hecke-equivariant map of the form \begin{equation}\label{eq:intro_beta}\beta_{q_1}: \Sh(V_{q_1}) \to O/\p^n\; \text{ s.t. }\; \exists\, T, \phi \text{ with } \beta_{q_1} (Z(T, \phi)) \neq 0,\end{equation}
where $\phi$ lies in $\mathcal S(V_{q_1}^2 \otimes \A_f, \Z)$, $\Sh(V_{q_1})$ is a finite double coset space like the one in (\ref{eq:beta_pi_intro}), the Hecke module structure on $O/\p^n$ is given by $\pi$, and again the choice of level structure is suppressed for simplicity. Once we have $\beta_{q_1}$, a similar argument to the proof of Theorem \ref{thm:intro_rk0} allows us to produce a Galois cohomology class
$\kappa_n(q_1q_2) \in H^1(\Q, T_{\pi,n})$ which is ramified at $q_2$, crystalline at $p$, and unramified at $\l$ for $\l\nmid N(\pi) pq_1q_2$. This class is then used as the input to the duality argument to bound $H^1_f(\Q, V_{\pi,\p})$. 

Let $Z(T, \phi)\in \CH^2(\Sh(V))$ be the class from Theorem \ref{thm:intro_rk1} with $\partial_{\AJ, \Pi_f}(Z(T, \phi)) \neq 0$. 
The idea to find $\beta_{q_1}$ is to study the special fiber of $\Sh(V)$ at $q_1$,  a prime of good reduction. The supersingular locus is purely one-dimensional, indexed by a Shimura set $\Sh(V_{q_1})$.  
We have the Abel-Jacobi map on the special fiber:
\begin{equation}\label{eq:AJ_intro}
    \partial_{\AJ, \m}: \CH^2(\Sh(V)_{\F_{q_1}}) \to H^1(\F_{q_1}, H^3_{\et}(\Sh(V)_{\overline\F_{q_1}}, O_\p(2))_\m)= H^1_\unr(\Q_{q_1}, H^3_{\et}(\Sh(V)_{\overline\Q}, O_\p(2))_\m).
\end{equation}
Restricting (\ref{eq:AJ_intro}) to supersingular cycles gives a map
\begin{equation}\label{eq:partial_on_ss}\partial_{\AJ, \m, \ss}: \Z[\Sh(V_{q_1})]\to H^1_\unr(\Q_{q_1}, H^3_{\et}(\Sh(V)_{\overline\Q}, O_\p(2))_\m).\end{equation}
Composing with a map $H^3_{\et} (\Sh(V)_{\overline\Q}, O_\p(2))_\m \to T_{\pi, \p}$, cf. (\ref{eq:intro_etale}), and noting that $H^1_\unr(\Q_{q_1}, T_{\pi,n})$ is free of rank one over $O/\p^n$ when $q_1$ satisfies the  level-raising congruence condition modulo $\p^n$,  we obtain our $\beta_{q_1}$. 

Using the Chebotarev density theorem and the local-global compatibility of the Abel-Jacobi map, one can ensure $\partial_{\AJ,\m} (Z(T, \phi)_{\F_{q_1}})$ is nonzero for a good choice of $q_1$. Now, if the special fiber  $Z(T, \phi)_{\F_{q_1}}$ were purely supersingular, then we could identify this local Abel-Jacobi class with the image of a special cycle in $\Z[\Sh(V_{q_1})]$ under $\beta_{q_1}$, and obtain the nonvanishing (\ref{eq:intro_beta}).
(This is the strategy of  \cite{bertolini2005iwasawa}, where the role of $Z(T, \phi)$ is played by a Heegner point.) Unfortunately, such is not the case for general $Z(T, \phi)$.\footnote{One could imagine modifying $\phi_q$ to ensure a purely supersingular special fiber, and applying a version of  (\ref{eq:intro_change_test}) to ensure that $Z(T, \phi)$ remains nontrivial under this change of test function. Somewhat to our surprise, this strategy fails: any appropriate choice of $\phi_q$ fails the criterion under which we prove (\ref{eq:intro_change_test}).} 

Instead, we use an auxiliary $\spin_4$ Shimura variety $S$ such that $Z(T, \phi) \subset S \subset \Sh(V)$. On the special fiber of $S$, the supersingular locus is one-dimensional, and -- by \cite{tian2019tate} -- any one-cycle  is cohomologically equivalent to a linear combination of supersingular cycles, at least after the application of a certain $\spin_4$-Hecke operator $\funnyT_{q_1}$.  We leverage this   to rewrite $\partial_{\AJ,\m} (\funnyT_{q_1} \cdot Z(T, \phi)_{\F_{q_1}})$ as the image of some $Z(T, \phi')$ under  $\partial_{\AJ,\m,\ss}$, with $\phi' \in \mathcal S(V_{q_1} ^2 \otimes \A_f, \Z)$. 

Now it remains to  choose $q_1$ so that $\partial_{\AJ,\m}(\funnyT_{q_1}\cdot  Z(T,\phi)_{\F_{q_1}})$ is nontrivial. 
For this choice to be possible, we need some  non-entanglement between $\rho_{\pi,\p}$ and the Galois representations appearing in the cohomology of $S$, which are closely related to Hilbert modular forms. In particular, the support of $Z(T, \phi)$  in the cohomology of $S$ needs to avoid some problematic Hecke eigensystems associated to CM forms.
To ensure this, we use a version of the modifying-the-test-function trick from (\ref{eq:intro_change_test}), this time at (non-level-raising) auxiliary primes $\l\neq q_1,q_2$. Since the trick uses representation theory, it relies on the modularity conjecture recalled after the statement of Theorem \ref{thm:intro_rk1}. (The subtleties that appear in this part of the argument prevent us from bounding the dual Selmer group $H^1_f(\Q, V_{\pi,\p}/T_{\pi,\p})$ in the context of Theorem \ref{thm:intro_rk1}, in contrast to the rank zero case.)

One final difficulty that arises in the proof of Theorem \ref{thm:intro_rk1} is the application of the level-raising arguments in \S\ref{subsubsec:intro_lr} to $\beta_{q_1}$. Since $\beta_{q_1}$ does not lift to a characteristic-zero Hecke eigenfunction in general, we cannot directly apply the theorem of Newton and Thorne to produce the lift in the diagram (\ref{eq:cd_intro}). Instead, we use the relative deformation theory developed by Fakhruddin, Khare, and Patrikis, along with a certain control of level-raised adjoint Selmer groups, to show that either $\beta_{q_1}$ lifts to an eigenfunction, or the Hecke eigensystem of $\pi$ is congruent modulo $\p^n$ to that of an automorphic representation $\pi'$ ramified at $q_1$ \emph{and} $q_2$;  in either case, we leverage the congruence to solve the lifting problem in (\ref{eq:cd_intro}). It is at this step that the extra hypothesis in Theorem \ref{thm:intro_rk1} (that $\pi$ does not arise from  certain automorphic inductions) enters the argument; if the image of the Galois representation $\rho_{\pi,\p}$ is too small, then we are unable to control the  adjoint Selmer groups needed to apply Fakhruddin, Khare, and Patrikis' method.

\subsection{Comparison to other work}
When $\pi_p$ is Borel-ordinary with respect to $\p$, and under a slightly different large-image condition, Theorem \ref{thm:intro_rk0} can be deduced from the main result of \cite{loeffler2023birchswinnertondyerconjecturemodularabelian}, which also covers more general weights. However, the existence of ordinary primes is still an open problem for general automorphic representations of $\GSP_4$, so the ordinarity hypothesis is potentially a serious one.

As indicated in \S\ref{subsec:overview_intro}, this article fits into an extensive literature of bounding Selmer groups using level-raising congruences (the method of ``bipartite Euler systems''). Compared to these works, one of the main underlying difficulties to prove Theorem \ref{thm:intro_rk0} is the non-factorizability of $\spin_3-\spin_5$ period integrals,  which is ultimately responsible for the appearance of the exotic test function $\phi_q^\unr$ (rather than the indictor function of a self-dual lattice or a translate of such by the spherical Hecke algebra) in the identity (\ref{eq:intro_magic_identity}). Another way to formulate this obstruction is in terms of spherical functions and Hironaka's conjecture \cite{hironaka1989spherical,coratozanarella2023sphericalfunctionssymmetricforms}. To our knowledge, the only prior work on bipartite Euler systems facing this challenge was the recently appeared PhD thesis of Corato Zanarella \cite{corato2024phd}. Interestingly, rather than our change-of-test-functions technique, the issue in \emph{op. cit.} is resolved using derived algebraic geometry; in our context, the analogous idea would be to work with $Z(T, \phi^q\otimes \phi_q^\ram)$ for more general $\phi_q^\ram$, even in the absence of a good integral model. 

Another novel aspect in the present work is in finding the $\beta_{q_1}$ in (\ref{eq:intro_beta}). To our knowledge, the idea of using intersection theory on the special fiber of an auxiliary Shimura variety $S$, in which the supersingular locus and the special cycle are of complementary dimension, is a new one. 

Finally, we warn the reader that the ``first and second reciprocity laws'' (corresponding to finding $\kappa_n(q)$ and $\beta_{q_1}$, respectively, in the discussions above) in the table of contents are named only by analogy to \cite{bertolini2005iwasawa}. In both cases, the full reciprocity law would constitute an equality, whereas we only prove an inequality, and that only up to a bounded error. The exact statements are Theorems \ref{thm:1ERL} and \ref{ultimate theorem for second TRL}.

\subsection{Comments on the endoscopic case}

An automorphic representation $\pi$ as in the beginning of \S\ref{subsec:intro_main} is called \emph{endoscopic} associated to a pair $(\pi_1,\pi_2)$ of
 cuspidal automorphic representations of $\GL_2(\A_\Q)$ -- necessarily arising from classical modular forms of weights 2 and 4, respectively, and trivial nebentypus  characters -- 
 if  the associated Galois representations satisfy \begin{equation}\label{eq:intro_endoscopic}\rho_{\pi,\p} = \rho_{\pi_1,\p} \boxplus \rho_{\pi_2,\p}\end{equation} for one, or equivalently all, primes $\p$ of $E$.  Here we normalize $\rho_{\pi_i,\p}:G_\Q\to \GL_2(E_\p)$  to have cyclotomic determinant; also let $V_{\pi_i,\p}$ be the underlying Galois module. The set of $\pi$ satisfying (\ref{eq:intro_endoscopic}) is by definition the endoscopic $L$-packet $\Pi(\pi_1,\pi_2)$. If $V$ is a quadratic space of signature $(3,2)$ and $\Pi=\Pi_f\otimes\Pi_\infty$ is an automorphic representation of $\spin(V)(\A_\Q)$ nearly equivalent to the members of $\Pi(\pi_1,\pi_2)$, then we have
 \begin{equation}H^i_\et(\Sh_K(V)_{\overline\Q}, E_\p)[\Pi_f] = \begin{cases}
  \Pi_f^K \otimes  V_{\pi_1,\p}(-2), & i = 3,\, \Pi_\infty \text{ generic},
\\\Pi_f^K\otimes  V_{\pi_2,\p}(-2), & i = 3, \, \Pi_\infty \text{ holomorphic},\\ 0 & i \neq 3.
 \end{cases}\end{equation}
Note here that $\Pi_\infty$ is uniquely determined by $\Pi_f$ via the Arthur multiplicity formula, which is known in this case \cite{chan2015local}.
In particular, we still have a well-defined map
\begin{equation}\label{eq:intro_AJ_endo}
    \partial_{\AJ,\Pi_f}: \CH^2(\Sh_K(V)) \to \Pi_f^K \otimes \left(H^1_f(\Q, V_{\pi_1,\p}\oplus V_{\pi_2,\p})\right).
\end{equation}
The analogue of Theorem \ref{thm:intro_rk1} is now:
\begin{propintro}[Theorem \ref{thm:endoscopic_rk1}]\label{prop:intro_endoscopic_rk1}
    Suppose $\pi_1$ and $\pi_2$ are non-CM, and $\p|p$ is a prime of $E$ such that:
    \begin{enumerate}
        \item $\pi_{1,p}$ and $\pi_{2,p}$ are unramified.
        \item The residual representations $\overline\rho_{\pi_1,\p}$ and $\overline\rho_{\pi_2,\p}$ are both absolutely irreducible, and $\overline\rho_{\pi_1,\p}\oplus \overline\rho_{\pi_2,\p}$ is generic (Definition \ref{def:generic}).
        \item For both $i = 1, 2$, there exists a prime $q \nmid N(\pi)$ such that $q^4\not\equiv 1\pmod p$, $\overline\rho_{\pi_i,\p}(\Frob_q)$ has eigenvalues $\set{1, q}$, and $\overline\rho_{\pi_{3-i},\p}(\Frob_q)$ has eigenvalues $\set{\alpha, q/\alpha}$ with $\alpha\not\in \set{\pm 1, \pm q, q^2, q^{-1}}.$ 
    \end{enumerate}
    Suppose as well that  \begin{equation*}
        \tag{$\ast$}H^1_f(\Q, V_{\pi_1,\p}\otimes V_{\pi_2,\p}(-1)) = 0.
    \end{equation*} Then for any Kudla cycle $Z(T, \phi) \in \CH^2(\Sh_K(V))$,
    $$\partial_{\AJ,\Pi_f} (Z(T, \phi)) \neq 0 \implies \dim_{ E_\p} H^1_f(\Q, V_{\pi_1,\p}) + \dim_{E_\p} H^1_f(\Q, V_{\pi_2,\p}) = 1.$$
\end{propintro}
\begin{rmksunnumbered}
    \begin{enumerate}[label = (\roman*)]
    \item The conditions (\ref{item:intro_rk0_1})-(\ref{item:intro_rk0_3}) hold for cofinitely many $\p$ (Lemma \ref{lem:easy_assumptions_cofinite} and Proposition \ref{prop:adm_endosc}).  
    \item The condition $(\ast)$ is always expected to hold, by a classical result of Shahidi on nonvanishing of Rankin-Selberg $L$-values \cite{shahidi1981certain}. Unfortunately, we remark that $(\ast)$ does \emph{not} follow from the main result of \cite{leiloefflerzerbes2014RSC}, because the necessary large-image condition is never satisfied for $V_{\pi_1,\p}\otimes V_{\pi_2,\p}(-1)$. The role of $(\ast)$ in the  proof  is to control congruences to non-endoscopic automorphic representations of $\spin_5$ groups.   
        \item Consider the case when $V$ is split, so that $\Sh_K(V)$ is a classical Siegel threefold, and $\pi_1$ is associated to an elliptic curve $E$ over $\Q$. 
        In particular, the Abel-Jacobi map (\ref{eq:intro_AJ_endo}) gives a way to construct classes in $H^1_f(\Q, T_pE)$ from special cycles on Siegel threefolds. Weissauer asked in \cite{weissauer2005four} 
how such classes related to the classical theory of Heegner points, a question to which Proposition 
     \ref{prop:intro_endoscopic_rk1} provides a partial answer: like the Selmer classes coming from Heegner points, these Abel-Jacobi classes can be nonzero only if $E$ has rank one (at least modulo the assumption on Rankin-Selberg convolutions).
     \item In the text, we prove a stronger version of Proposition \ref{prop:intro_endoscopic_rk1} that includes some CM cases.
    \end{enumerate}
\end{rmksunnumbered}

Similarly, the methods used to Prove Theorem \ref{thm:intro_rk0} can also be used for endoscopic representations under the condition $(\ast)$. This case is included in Theorem \ref{thm:rk_zero_endoscopic} for completeness, but note that the result gives nothing new beyond Kato's work  \cite{kato2004padic}.

\subsection{Organization of the paper}
In \S\ref{sec:prelim_BQAs}, we lay out the notational conventions for the article and cover preliminary materials on involutions, Clifford algebras, and Selmer groups. In \S\ref{sec:prelim_ARs}, we collect the necessary  background results related to automorphic representations, Galois representations, and Shimura varieties. The most important role of this section is to make some of the results of \cite{rosner2024global}  (on global Jacquet-Langlands transfers for inner forms of $\GSP_4$) unconditional on Conjecture 7.5 of \emph{op. cit.} In \S\ref{sec:sp_cycles_theta}, we define Kudla's cycles $Z(T, \phi)$ and their analogues on compact $\spin$ groups, and explain their relation to classical theta lifts. In \S\ref{sec:construction}, we define the Galois cohomology classes and special periods used in the Euler system arguments. In  \S\ref{sec:rep_theory}, we use the mod-$p$ theory of the Weil representation to prove the change-of-test-functions criterion explained above in (\ref{eq:intro_change_test}); this is one of the most crucial (and technical) parts of the argument.

In \S\ref{sec:RZ}, we turn to geometry, studying a ramified Rapoport-Zink space for $\spin_5$. This section is based on \cite{wang2020bruhat}, although we need more details on  intersection theory for our applications to special cycles. 
In \S\ref{sec:1ERL_geom}, these results are applied to study the special fiber of the ramified $\spin_5$ Shimura variety, and compute the local part of the Abel-Jacobi image of special cycles. (Some of the results in this section are generalizations of those in \cite{wang2022arithmetic}.) In \S\ref{sec:1erl_continued}, we perform the level-raising and complete the program described in \S\ref{subsec:intro_sketch_rk0} above. In \S\ref{sec:main_rk0}, we prove the main results in the rank zero case. In \S\ref{sec:2ERL_geometry}, we study special cycles on the special fibers of $\spin_4$ and $\spin_5$ Shimura varieties with good reduction. In \S\ref{sec:2ERL_contd} and \S\ref{sec:rk1}, we  prove our main result in the rank one case.

The paper contains three appendices: in the short Appendix \ref{sec:appendix_spin4}, we explain the relation of
the cohomology of $\spin_4$ Shimura varieties to Hilbert modular forms, which is well-known but for which we were unable to find a suitable reference. In Appendix \ref{sec:appendix_def_theory}, we develop a general framework for deformation-theoretic characteristic-zero level raising of $G$-valued Galois representations, using  ideas from \cite{fakhruddin2021relative}. 
These results are most important for the proof of the rank one case. In Appendix \ref{sec:large_image}, we prove some large-image results for  the $p$-adic Galois representations that appear in the article, which are necessary for various Chebotarev arguments.

\subsection{Acknowledgements}
I wish to thank my PhD advisor, Mark Kisin, for his encouragement and advice while I worked on this project. 
I am also grateful to   Aaron Landesman, Alice Lin, Sam Mundy, Chris Skinner, Shiang Tang, Salim Tayou, Wei Zhang,  and especially Murilo Corato Zanarella for helpful conversations. Thank you to Karen Schlain for help with typesetting. This work was supported by NSF grants DGE-1745303 and DMS-2401823.
\section{Preliminaries}\label{sec:prelim_BQAs}
\subsection{Notation}

\subsubsection{Number fields and Galois groups}\label{subsubsec:number_field_notation}\label{subsubsec:chi_cyc}
Let $L \subset \overline\Q$ be a number field.
\begin{itemize}
\item We denote by $O_L$ the ring of integers. If $\p\subset O_L$ is a prime, then we write $O_\p$ and $L_\p$ for the respective completions, and $\varpi_\p$ for a uniformizer of $O_{\p}$. 
\item Let $G_L\coloneqq \Gal(\overline \Q/L)$ be the absolute Galois group. If $S$ is a finite set of places of $L$, we denote by $L^S$ the maximal unramified-outside-$S$ extension of $L$, and set $G_{L,S} \coloneqq \Gal(L^S/L)\subset G_L$. If $\p$ is a prime of $L$, we also write $G_{L_\p}$ for the absolute Galois group of $L_\p$, with inertia subgroup $I_{L_\p}\subset G_{L_\p}$.
\item For any $p$, $\chi_{p,\cyc}: G_L \to \Z_p^\times$ denotes the $p$-adic cyclotomic character. We normalize the definition of Hodge-Tate weights so that $\chi_{p,\cyc}|_{G_{\Q_p}}$ has weight one.
  \end{itemize}

\subsubsection{Ad\`ele groups and class field theory}
\begin{itemize}
    \item \label{subsubsec:adeles_hecke}
For a number field $K$, let $\A_K$ be the ad\`ele ring; when $K = \Q$ we typically write $\A = \A_\Q$. For a finite set $S$ of places of $\Q$, we write 
$$\A_S = \prod_{v\in S} \Q_v,\;\; \A^S = \prod_{v\not\in S}'\Q_v,\;\; \A_f^S = \prod_{v\not\in S \cup \set{\infty}}' \Q_v.$$
\item If $\pi = \otimes'_v \pi_v$ is an irreducible admissible representation of $G(\A)$ for some $\Q$-group $G$, then for a squarefree integer $D$, $\pi_f^D$ denotes $\otimes'_{\l\nmid D} \pi_\l$, and $\pi_D$ denotes $\otimes_{\l|D}\pi_\l$. The similar notations $\pi_f^S$, $\pi_S$ hold for finite sets of primes $S$.  
\item For any prime $p$, let $\langle p \rangle \in \A_f^\times$ be the image of $p$ under the natural inclusion $\Q_p^\times \hookrightarrow \A_f^\times$. 
\item 
We always normalize the reciprocity maps of class field theory
to send uniformizers to geometric Frobenius.
\item For any number field $L$ and finite-order character $\chi: L^\times \backslash \A_L^\times \to k^\times$, with $k$ a field, we write $\rec(\chi): G_L \to k^\times $ for the pullback by the reciprocity map.

\end{itemize}

\subsubsection{Coefficient rings}\label{subsubsec:notation_coefficients}\label{subsubsec:CNLO_notation}
For coefficient rings, we will often use a discrete valuation ring $O$ which is a finite flat extension of $\Z_p$, with uniformizer $\varpi$. In this context:
\begin{itemize}
    \item We denote by $\CNL_O$ the category of complete local Noetherian $O$-algebras with residue field $O/\varpi$.
    \item For any torsion $O$-module $M$ and any $m\in M$, $\ord_\varpi (m)\geq 0$ denotes the least integer such that $\varpi^{\ord_\varpi(m)} m = 0.$
\end{itemize}
\subsubsection{Symplectic groups}\label{subsubsec:coordinates_gsp2n}
\begin{itemize}
\item For any $n$, let $\Omega_n = \begin{pmatrix}
       \mathbf{0}_n  & I_n \\ -I_n & \mathbf{0}_n
    \end{pmatrix}\in \GL_{2n}(\Z)$, 
    where $\mathbf{0}_n$ is the $n\times n$ matrix of zeros and $I_n$ is the $n\times n$ identity matrix.
    \item 
Define the algebraic group $\GSP_{2n}$ over $\Z$ by
\begin{equation}\label{eq:GSP4_embedded}\GSP_{2n}(R) = \set{(g,\lambda)\in \GL_{2n}(R)\times R^\times  \,: \, g \Omega g^t =  \lambda\Omega}\end{equation}
for any ring $R$. 
\item We have the natural similitude character $\nu: \GSP_{2n} \to \mathbb G_m,$ and the symplectic group $\SP_{2n}\subset \GSP_{2n}$ is the kernel of $\nu$.
\item For any prime $p$, we write $\langle p \rangle \in \GSP_{2n}(\A_f)$ for the scalar matrix corresponding to $\langle p \rangle\in \A_f^\times$.
\end{itemize}
\subsubsection{GSpin groups}\label{subsubsec:spin_groups_notation}
Let $V$ be a quadratic space over a field $F$. 
\begin{itemize}
    \item We denote the pairing $V\times V \to F$ by $(v, w) \mapsto v\cdot w$. 
    \item The Clifford algebra of $V$ is the associative $F$-algebra $C(V)$ generated by $v\in V$, subject to the relation
    $$v^2 = (v\cdot v)1.$$ There is a natural $\Z/2\Z$ grading on $C(V)$, with respect to which the plus part is denoted $C^+(V)$. 
    \item We denote by $\ast$ the natural involution on $C(V)$, determined by $v^\ast = v$. Then $\spin(V)$ is the algebraic group over $F$ defined by 
    $$\spin(V)(R) = \set{(g, \lambda) \in \left(C^+(V)\otimes R\right)^\times\times R^\times\,: \, gg^\ast = \lambda}.$$ The natural similitude character is again denoted $\nu$.
    \item If $F = \Q$, then for any prime $p$ we again denote by $\langle p \rangle \in \spin(V)(\A_f)$ the scalar element $p \in C^+(V)^\times.$
  
\end{itemize}
\subsubsection{Quaternion algebras and quadratic spaces}
\begin{itemize}
    \item \label{subsubsec:B_D_notation}
For a squarefree integer $D\geq 1$, let $B_D$ be the quaternion algebra over $\Q$ which ramifies at the factors of $D$, and possibly at infinity. Let $\ast_D$ be the standard involution on $B_D$. We also denote by $\ast_D$ the involution on $M_n(B_D)$ given by the composite of $\ast_D$ and transposition. 

\item For all squarefree $D$,  let\begin{equation}\label{eq:V_D}
    V_D \coloneqq M_2(B_D)^{\ast_D = 1,\tr = 0},
\end{equation}
which is a 5-dimensional quadratic space of trivial discriminant, whose Hasse invariant coincides with that of $B_D$. The signature of $V_D$ is $(5,0)$ or $(3,2)$ when $B_D$ is ramified or split at infinity, respectively.
\item We note that  $\spin(V_D)$ is an inner form of $\spin(V_1) \cong \GSP_4$.
\end{itemize}
\subsubsection{Algebraic geometry}
\begin{itemize}
    \item For a closed subscheme $X$ of a scheme $Y$, we denote by $\mathcal N_{X/Y}$ the normal sheaf. 
    \item If $X$ is a scheme, $X_\red\subset X$ denotes the maximal reduced subscheme.
\end{itemize}
\subsubsection{Miscellaneous}
\begin{itemize}
    \item For a squarefree 
integer $D \geq 1$, we denote by $\div(D)$ its set of prime factors, and let $\sigma(D)\coloneqq \#\div(D)$.

    \item \label{subsubsec:lattices_n} Suppose $V$ is an $F$-vector space for a nonarchimedean local field $F$ with ring of integers $O_F$. For two $O_F$-lattices $\Lambda \subset \Lambda'$ of $V$, the notation $\Lambda\subset_n \Lambda'$ means that $\Lambda'/\Lambda$ has $O_F$-length $n$. 
    \item 
 For any group $G$, we let $Z_G\subset G$ denote the center. 
 \item For a prime $p$, let $\breve \Z_p$ be the Witt vectors of $\overline \F_p$. 
\end{itemize}

\subsection{Central simple algebras and involutions}
\subsubsection{}
Let $M$ be a central simple algebra over a field $F$ of characteristic zero. Throughout this paper an involution on $M$ will be understood to mean an involution of the first kind, i.e. fixing $F$. Such involutions fall into two categories: main type or nebentype (also called symplectic and orthogonal type \cite{knus1998book}).  For example, the transpose involution on $M = M_2(F)$ is of nebentype; the standard involution on a quaternion algebra is of main type.
For later use, we now recall some basic facts.
\begin{prop}
\label{prop:tensoring involutions}
    Let $(M_1, \ast_1)$ and $(M_2, \ast_2)$ be two central simple algebras equipped with involutions. 
The induced involution $\ast_1\otimes \ast_2$ on $M$ is of nebentype if and only if $\ast_1$ and $\ast_2$ are of the \emph{same} type.
\end{prop}
\begin{proof}
    This is \cite[Proposition 2.23]{knus1998book}.
\end{proof}
\subsubsection{}
Suppose $F = \R$ or $\Q$, and recall that an element $\mu \in M$ is called \emph{totally positive} if it has strictly positive eigenvalues as an endomorphism of $M$. 
\begin{prop}\label{prop:positive involutions}
    Let $M=M_n(\R)$ or $M_n(\mathbb H)$, and suppose $M$ is equipped with a \emph{positive} involution $\ast$, i.e. such that $\tr(a^\ast a) > 0$ for all $0 \neq a\in M$. Then the positive involutions on $M$ are those of the form
    $$a\mapsto \mu a^\ast \mu^{-1},$$ where $\mu^\ast = \mu$ and either $\mu$ or $-\mu$ is totally positive. 
    In particular, all positive involutions on $M$ have the same type.
\end{prop}
\begin{proof}
This is \cite[Proposition 8.4.7 and Lemma 8.4.12]{voight2021quaternion}.
\end{proof}
\subsubsection{}
The tensor product of two quaternion algebras is called a \emph{biquaternion algebra} (BQA). Involutions of nebentype on rational BQAs are particularly simple:
\begin{lemma}\label{lem:BQA_symplectic_agree}
Let $M$ be a BQA over $\Q$, and let $\ast_1,$ $\ast_2$ be two involutions of $M$ of main type. If either $M\otimes \R$ is split, or $\ast_1,$ $\ast_2$ are both positive, then there exists an element $g\in M ^\times $
such that $$\inner (g)\circ\ast_1 =\ast_2\circ\inner (g). $$
In particular, conjugation by $g $ defines an isomorphism $(M, \ast_1)\simeq (M, \ast_2)$.
\end{lemma}
Here, $\inner(g)$ is the automorphism $a \mapsto gag^{-1}$ of $M$.
\begin{proof}
  By the Skolem-Noether theorem, there exists $h\in M ^\times $
    such that $\inner (h)\circ\ast_1 =\ast_2. $
    The condition on $g $
    in the lemma is equivalent to $h =\lambda g ^{\ast_1} g $, for some $\lambda\in\Q ^\times $.
   By \cite[Theorem 16.19]{knus1998book}, there exists such a $g $
if and only if the Pfaffian norm of $h $ (with respect to $\ast_1 $)
belongs to $(\Q ^\times) ^ 2\nerd (M ^\times) $. If $M\otimes\R $
is split, then $\nerd (M ^\times) =\Q ^\times $, so we are done. If $M\otimes\R $
is not split, then $\nerd (M ^\times) =\Q ^\times_{>0} $, and it suffices to show that the equation $h =\lambda g^{\ast_1} g $
has a solution over $\R$. Assuming without loss of generality that $\ast_1$ is the standard involution on $M\otimes \R \simeq M_2(\mathbb H)$, this can be checked directly using Proposition \ref{prop:positive involutions}.
\end{proof}

\subsubsection {}
Let $F$ be a nonarchimedean local field with ring of integers $O_F$, and let $B$ be the unique nonsplit quaternion algebra over $O_F$. We denote by $O_B$ the unique maximal $O_F$-order in $B$, with uniformizer $\pi\in O_B$ and natural valuation $\ord_\pi$. 
\begin {prop}\label{proposition needed to define unit and non-unit type}
Suppose $g\in\GL_n (B) $
satisfies $$gM_n (O_B) g ^ {-1} = M_n (O_B). $$
Then, up to rescaling by an element of $F ^\times $, we have either $g\in\GL_n (O_B) $
or $g\in\pi\GL_n (O_B) $.
\end {prop}
\begin {proof}
Suppose $g $
is given by the matrix $(g_{ij}) $
and $g ^ {-1} $
is given by the matrix $(h_{kl}) $. Let $\alpha_{jk} $
be the matrix with a 1 in the $jk $
position and 0s elsewhere; since $g\alpha_{jk} g ^ {-1}\in M_n (O_B) $, we see that \begin{equation}\label{eq:for_unit_type}g_{ij} h_{kl}\in O_B,\;\;\text {for all}\, i, j, k, l. \end{equation}
Without loss of generality, by rescaling $g $, we may assume $g_{ij}\in O_B $
for all $i, j $
but $\ord_\pi (g_{ij})\leq 1 $
for some $i, j $. If $h_{ij}\in O_B $
for all $i, j $, then $g\in\GL_n (O_B) $, so suppose without loss of generality that $\ord_\pi (h_{kl}) <0 $
for some $k, l $.
It follows from (\ref{eq:for_unit_type}) that $\ord_\pi (g_{ij})\geq 1 $
for all $i, j $, with equality holding for some $i, j $, and $\ord_\pi (h_{kl})\geq - 1 $
for all $k, l $; in particular, $g\in\pi\GL_n (O_B) $.
\end {proof}
Motivated by Proposition \ref{proposition needed to define unit and non-unit type}, we make the following definition.
\begin {definition}\label{unit and non-unit type definition}
An involution $\ast$
of $M_n (B) $
stabilizing $M_n (O_B) $
is called of
\emph {unit type} if it is of the form $\alpha ^\ast= g\overline\alpha ^ tg ^ {-1} $
for some $g\in\GL_n (O_B) $, and of
\emph {non-unit type} if it is of the form $\alpha ^\ast= g\overline\alpha ^ tg ^ {-1} $
for some $g\in\pi\GL_n (O_B) $.
\end {definition}
Similarly, if $B $
is a quaternion algebra over $\Q $
and $p $
is a prime such that $B\otimes\Q_p $
is not split, let $O_B\subset B $
be the unique maximal $\Z_{(p)} $-order. Then an involution $\ast$
of $M_n (B) $
stabilizing $M_n (O_B) $is called of unit or non-unit type according to whether the induced involution on $M_n (B)\otimes\Q_p $
is of unit or non-unit type.
\begin {remark}
By Proposition \ref{proposition needed to define unit and non-unit type}, an involution $\ast$
of $M_n (B) $
stabilizing $M_n (O_B) $
induces an involution on $M_n (O_B/\pi) $, which acts trivially on the center of $M_n (O_B/\pi) $
if and only if $\ast$
is of non-unit type.\end {remark}
\subsubsection{}
The same proof as Proposition \ref{proposition needed to define unit and non-unit type} also shows:
\begin{prop}\label{prop:conj_OF}
Let $F$ be a nonarchimedean local field with ring of integers $O_F$. 
If $g\in \GL_n(F)$ satisfies
    $$gM_n(O_F) g^{-1} = M_n(O_F),$$ then we have $g\in F^\times\GL_n(O_F)$.
\end{prop}
\subsection{PEL data}
\subsubsection{}
Recall the notion of a PEL datum $\mathcal D = (B, \ast, V, \psi)$ from \cite[Chapter 8]{milne2005introduction}, and  set $C =\End_B (V) $. Then $C $ is equipped with an involution $c\mapsto c' $, the adjoint with respect to $\psi$, and the $\Q $-group $G_{\mathcal D} $
associated to $\mathcal D $
is defined by $$G_{\mathcal D} (R) =\set {(g,\lambda)\in (C\otimes R) ^\times \times R^\times \,:\, gg'= \lambda}. $$ To $\mathcal D $, there is associated the reflex field $E = E_{\mathcal D} $, and, for any compact open subgroup $K\subset G_{\mathcal D} (\A_f) $, a moduli functor $M_{K} $ over $E$. Let us briefly recall the definition of $M_{K}$, for which more details can be found in  \cite[p. 390]{kottwitz1992points}. 
For a connected scheme $S\to \Spec E$, $M_{ K} (S) $
is the set of isomorphism classes of tuples $(A,\iota,\lambda,\eta)$ where:
\begin{itemize}
    \item $A/S $ is an abelian scheme up to isogeny;
    \item $\iota: B \hookrightarrow \End^0(A/S)$  is an embedding satisfying the {Kottwitz determinant condition}  derived from $\mathcal D$;
    \item $\lambda: A\to A ^\check $
    is a quasi-polarization such that $\iota (b^\ast) ^\check\circ\lambda =\lambda\circ\iota (b) $
    for all $b\in B $;
    \item $\eta$ is a $K $-level structure  for $A$, i.e., for any geometric point $s$ of $S$, a $\pi_1(S, s)$-stable $K$-orbit of isomorphisms 
 $$\eta:H_{1,\et} (A_s, \A_f)\xrightarrow{\sim} V\otimes_\Q \A_f$$ respecting the actions of $B $ and the symplectic pairings on both sides up to a scalar.
    \end{itemize}
    If $K $ is neat, then $M_K $ is represented by a smooth quasi-projective scheme over $E $.
    \subsubsection{}
    Let $p$ be a prime. A self-dual $p $-integral refinement $\mathscr D $
    of $\mathcal D $
    is the additional data of a $\ast $-stable maximal $\Z_{(p)}$-order $O_B\subset B $
    and a self-dual $O_B $-stable $\Z_{(p)} $-lattice $\Lambda\subset V $. (This is a special case of the notion in \cite[\S6]{rapoport1996period}.)
 For a compact open subgroup $K ^ p\subset G_{\mathcal D} (\A_f ^ {p}) $, the corresponding moduli problem $\mathcal M_{K^p}$ is defined as follows. For a connected scheme $S\to\Spec\O_E\otimes \Z_{(p)} $, $\mathcal M_{K^p}(S)$ is the set of isomorphism classes of tuples $(A,\iota,\lambda,\eta^p) $ where:
    \begin {itemize}
    \item $A/S $ is an abelian scheme up to prime-to-$p $ isogeny;
    \item $\iota: O_B\hook\And(A/S)\otimes\Z_{(p)} $
    is an embedding satisfying the Kottwitz condition;
    \item $\lambda: A\rightarrow A ^\check $ is a prime-to-$p$ quasi-polarization such that $\iota (b^\ast)^\check\circ\lambda =\lambda\circ\iota (b) $
    for all $b\in O_B $;
    \item $\eta ^ p$ is a $K ^ p $-level structure, i.e., for any geometric point $s$ of $S$, a $\pi_1(S, s)$-stable $K^p$-orbit of isomorphisms 
 $$\eta:H_{1,\et} (A_s, \A_f^p)\xrightarrow{\sim} V\otimes_\Q \A_f^p$$ respecting the actions of $B $ and the symplectic pairings on both sides up to a scalar.
    \end {itemize}
    If $K ^ p $ is neat, then $\mathcal M_{K ^ p} $
    is represented by a quasi-projective scheme over $O_E\otimes\Z_{(p)} $.
    Its generic fiber is $M_{K ^ pK_p} $, where $K_p = \operatorname{Stab}_{G_{\mathcal D}(\Q_p)} (\Lambda)$.

\begin{lemma}\label{lem:g_polarization}
Let $S$ be a scheme and $B$ a  simple $\Q$-algebra with involution $\ast$ of the first kind. Suppose given an abelian scheme $A/S$ with an embedding of $\Q$-algebras
$$\iota: B\hookrightarrow\End^0 (A/S),$$ and a quasi-polarization $\lambda: A\to A^\vee$ such that
$$\iota(b^\ast)^\vee\circ \lambda = \lambda\circ\iota(b),\;\;\forall b\in B.$$

Then for any totally positive $g\in B^\times$ with $g^\ast = g$,  $\lambda\circ g$ defines a quasi-polarization of $A$.
\end{lemma}  
\begin{proof}
    It suffices to prove the lemma when $S = \Spec k$ with $k$ an algebraically closed field. In this case, by the discussion in \cite[Chapter 21, Application III]{mumford1974abelian}, it suffices to show that $g$ has strictly positive eigenvalues on $\End^0(A)$.  But this follows from the positivity of the eigenvalues on $B$, because $\End^0(A)$ is a semisimple $\Q$-algebra containing $B$, and any such is a product of copies of $B$ as a $B$-module.
\end{proof}
The following
corollary is immediate. 

\begin{cor}\label{cor:changing_polarizations_PEL}
    Let $\mathcal D = (B,\ast, V, \psi) $
    be a PEL datum, where $B$ is a simple $\Q$-algebra and $\ast$ is an involution of the first kind. Suppose given a totally positive $g\in B ^\times $
    such that $g^\ast = g$,  and let $\mathcal D_g = (B,\ast_g, V, \psi_g), $
    where $\ast_g \coloneqq g \circ \ast \circ g^{-1}$ and $$\psi_g( x, y)\coloneqq\psi( x, g ^ {-1} y),\;\; \forall x,y\in V. $$
    Then $G_{\mathcal D} = G_{\mathcal D_g}$, and, for all $K\subset G_{\mathcal D}(\A_f) $, there is a canonical isomorphism $$M_{\mathcal D, K}\xrightarrow{\sim}M_{\mathcal D_g, K} $$
    defined by $$(A,\iota,\lambda,\eta)\mapsto (A,\iota,\lambda\circ g,\eta). $$
    If $\mathscr D = (O_B,\ast,\Lambda, \psi) $
    is a self-dual $p$-integral refinement of $\mathcal D $
    and $g $ lies in $O_B ^\times $, then $\mathscr D_g = (O_B,\ast_g,\Lambda,\psi_g) $
    is a self-dual $p $-integral refinement of $\mathcal D_g $ and the above isomorphism extends to an isomorphism of integral models $$\mathcal M_{\mathscr D, K^p}\xrightarrow{\sim}\mathcal M_{\mathscr D_g, K ^ p} $$ for all compact open $K ^ p\subset G_{\mathcal D} (\A_f^p) $.
    \end {cor}
    \qed
    
We also deduce:
\begin{cor}\label{cor:exists_basepoint}
    Let $B$ be an indefinite quaternion algebra over $\Q$, with $O_B\subset B$ a maximal $\Z_{(p)}$-order, and
    fix an integer $n \geq 1$. 
    Let $\ast$ be a 
   positive involution on $M_n(O_B)$, of non-unit type if $B\otimes \Q_p$ is ramified. Then there exists an abelian scheme $A$ over $\Spec \breve{\Z}_p$ of dimension $2n$ with supersingular reduction, a
    prime-to-$p$ quasi-polarizon $\lambda: A \to A^\vee$, and an embedding $\iota: M_n(O_B) \hookrightarrow \End(A)\otimes \Z_{(p)},$ such that 
    $$\iota(b^\ast)^\vee\circ \lambda = \lambda\circ\iota(b),\;\;\forall b\in M_n(O_B).$$
\end{cor}

\begin{proof}
Using Lemma \ref{lem:g_polarization} along with Propositions \ref{proposition needed to define unit and non-unit type} and \ref{prop:conj_OF}, we  reduce to the case that $\alpha^\ast =  \alpha^{\dagger t}$ for all $\alpha\in M_n(O_B)$, where $\dagger$ is a positive involution on $O_B$, of non-unit type if $B$ ramifies at $p$. Thus it suffices to prove the corollary when $n = 1$, in which case it is well-known from the classical theory of Shimura curves; see the discussions in \cite[Chapter III]{boutot1991uniformisation}.
\end{proof}
\subsection{PEL data for $\spin_{3,2}$ groups}
\subsubsection{}
In this article, we will consider PEL data that arise in the following way. Let $D \geq 1$ be squarefree with $\sigma(D)$ \emph{even}, and let $q$ be a prime, possibly with $q|D$. Suppose fixed a maximal $\Z_{(q)}$-order $O_D \subset B_D$, with a nebentype involution $\ast$, and an embedding $\breve \Z_q \hookrightarrow \C$. 
\begin{definition}\label{def:O_Dtriple}
   \begin{enumerate}\item  An $(O_D, \ast)$-triple is a triple $(A_0, \iota_0, \lambda_0)$, where:
    \begin{enumerate}
        \item $A_0$ is an abelian scheme over $\Spec \Breve{\Z}_q$  of rank 4 with supersingular reduction;
        \item $\iota_0$ is an embedding $O_D\hookrightarrow\End(A_0) \otimes\Z_{(q)}$;
        \item $\lambda_0: A_0 \to A_0^\vee$ is a prime-to-$q$ polarization, such that 
    $$\iota_0(\alpha^\ast)^\vee \circ \lambda_0 = \lambda_0 \circ \iota_0(\alpha),\;\; \forall \alpha\in O_D.$$
    \end{enumerate}

    \item Given an $(O_D, \ast)$-triple as above, we set $H\coloneqq H_1(A_0(\C), \Q),$ with its canonical symplectic form $\psi$.  Let $\Lambda  \subset H$ be the lattice $ H_1(A_0(\C), \Z_{(q)})$. The PEL datum associated to $(A_0, \iota_0, \lambda_0)$ is defined by
$$\mathcal D = (B_D, \ast, H, \psi),$$ with self-dual $q$-integral refinement
$$\mathscr D = (O_D, \ast, \Lambda, \psi).$$
    \end{enumerate}
\end{definition}
\subsubsection{}
Let $D\cdot q \coloneqq Dq/\operatorname{gcd}(D, q)$.
Given an $(O_D, \ast)$-triple $(A_0, \iota_0, \lambda_0)$:
\begin{itemize}
 \item  $H$ is a $B_D$-module, and $\End(H, B_D)$ is isomorphic to $M_2(B_D)$. The adjoint involution $\dagger$ on $\End(H, B_D)$ is of main type (by Proposition  \ref{prop:tensoring involutions}, because the adjoint involution on $\End(H) = B_D\otimes \End(H, B_D)$ is of main type).

 \item Set $\overline A_0\coloneqq (A_0)_{\overline\F_q}$, with its induced $\O_D$-action $\overline \iota_0$ and polarization $\overline\lambda_0$. Then $\End^0(\overline A_0, \overline \iota_0)$ is isomorphic to $M_2(B_{D\cdot q})$, and its Rosati involution $\dagger$ is positive, hence of main type by Proposition \ref{prop:positive involutions}.
 \end{itemize}\begin{definition}\label{def:unif_datum_general}
 \leavevmode
 \begin{enumerate}
     \item \label{def:unif_datum_general_1}  A $q$-adic \emph{uniformization datum} $(A_0,\iota_0,\lambda_0, i_D, i_{D\cdot q})$
for $(O_D, \ast)$ is an $(O_D, \ast)$-triple $(A_0,\iota_0, \lambda_0)$ as above, along with a choice of isomorphisms of algebras with involution: $$i_D:\left(\End(H, B_D), \dagger\right) \isomorphism (M_2(B_D), \ast_D),$$ $$i_{D\cdot q}:\left(\End^0(\overline A_0, \overline\iota_0), \dagger\right) \isomorphism (M_2(B_{D\cdot q}), \ast_{D\cdot q}).$$
\item\label{def:unif_datum_general_2} A $q$-adic \emph{uniformization datum}  $(\ast, A_0, \iota_0, \lambda_0, i_D, i_{D\cdot q})$ for $V_D$ is a choice of positive nebentype involution $\ast$ on $O_D$ -- of unit type if $q|D$ -- and a uniformization datum $(A_0, \iota_0, \lambda_0, i_D, i_{D\cdot q})$ for $(O_D, \ast)$. 
 \end{enumerate}
\end{definition}
Recall here that $V_D$ was defined in (\ref{eq:V_D}).
\begin{rmk}\label{rmk:unif_datum}
    \begin{enumerate}
        \item \label{rmk:unif_datum_one}
        Given any $(O_D,\ast)$-triple $(A_0, \iota_0, \lambda_0)$, the choices 
         in Definition \ref{def:unif_datum_general}(\ref{def:unif_datum_general_1}) exist by Lemma \ref{lem:BQA_symplectic_agree}.
        \item \label{rmk:unif_datum_two}Given a $q$-adic uniformization datum $(\ast, A_0,\iota_0,\lambda_0, i_D, i_{D \cdot q})$ for $V_D$, we also obtain isomorphisms 
$$\End(H, B_D)^{\dagger = 1, \tr = 0} \isomorphism V_D, \; \End^0(\overline A_0, \overline \iota_0)^{\dagger = 1, \tr = 0} \isomorphism V_{D \cdot q}.$$
The former determines an isomorphism $$G_{\mathcal D} \isomorphism\spin(V_D).$$
Moreover,  the action of $\End(\overline A_0, \overline \iota_0)$ on  
$$H_{1,\et} (\overline A_0, \A_f^q) \cong H_1(A_0(\C), \A_f^q)$$ induces an isomorphism $V_{D \cdot q}\otimes_\Q \A_f^q\cong V_D \otimes_\Q \A_f^q$. In turn, this induces an isomorphism $\spin(V_D)(\A_f^q) \cong \spin(V_{D \cdot q})(\A_f^q)$.    
    \end{enumerate}
\end{rmk}

\subsection{Local conditions and Selmer groups}
\subsubsection{}
In this subsection, $O$ is the ring of integers of a finite extension  $E/\Q_p$, and $\varpi\in O$ is a uniformizer. 

\begin{notation}\label{notation:BK_etc}
\leavevmode
    \begin{enumerate}
        \item \label{notation:BK_part_tf}Suppose $M$ is a finite free $O$-module with an action of $G_{\Q_\l}$, where $\l$ may be equal to $p$. We consider the Bloch-Kato local conditions
\begin{equation*}
    H^1_f(\Q_\l, M) \coloneqq \ker\left(H^1(\Q_\l, M) \to \frac{H^1(\Q_\l, M\otimes\Q_p)}{H^1_f(\Q_\l, M\otimes\Q_p)}\right).
\end{equation*}
\item Globally, if $M$ is a finite free $O$-module with $G_\Q$-action, let
\begin{equation*}
    H^1_f(\Q, M) \coloneqq \ker\left(H^1(\Q, M) \to \prod_\l \frac{H^1(\Q_\l, M)}{H^1_f(\Q_\l, M)}\right).
\end{equation*}
\item\label{notation:BK_part_unr} Suppose $\l \neq p$. For an unramified, finitely generated $O$-module $M$ (either finite or infinite) with $G_{\Q_\l}$-action, we let
\begin{equation*}
    H^1_f(\Q_\l, M) \coloneqq H^1_{\unr} (\Q_\l, M).
\end{equation*}
    \end{enumerate}
\end{notation}

\begin{rmk}\label{rmk:compare_f_unr_BK}
    Notations \ref{notation:BK_etc}(\ref{notation:BK_part_tf},\ref{notation:BK_part_unr}) are consistent because, when $M$ is free over $O$ and unramified as a $G_{\Q_\l}$-module, the map
$H^1(I_\l, M) \to H^1(I_\l, M\otimes \Q_p)$ is injective. 
\end{rmk}

\subsubsection{}\label{subsubsec:torsion_cryst}
Fix integers $a\leq 0 \leq b$. 
Recall that a finite $\Z_p[G_{\Q_p}]$-module $M$ is said to be \emph{torsion crystalline} with Hodge-Tate weights in $[a, b]$ if there exists a crystalline $\Q_p[G_{\Q_p}]$-module $V$ with Hodge Tate weights in $[a,b]$, and two $G_{\Q_p}$-stable lattices $T_1 \subset T_2 \subset V$, such that $M = T_2/T_1.$ A finitely generated $\Z_p[G_{\Q_p}]$-module $M$ is torsion crystalline with Hodge-Tate weights in $[a, b]$ if $M/p^n$ is so for all $n \geq 1$. 

If $M$ is torsion crystalline with Hodge-Tate weights in $[a, b]$, let $$H^1_{f,\tors}(\Q_p, M) \subset H^1(\Q_p, M)$$ be the subspace of cohomology classes such that the corresponding extension
$$0 \to M \to \ast \to \Z_p\to 0$$ is torsion crystalline with Hodge-Tate weights in $[a,b]$. 
If $M$ is an $O[G_{\Q_p}]$-module, then $H^1_{f,\tors}(\Q_p, M)$ is an $O$-submodule of $H^1(\Q_p, M)$.
\begin{prop}\label{prop torsion crystalline iff crystalline}
A finite free $O[G_{\Q_p}]$-module $M$ is torsion crystalline with Hodge-Tate weights in $[a, b]$ if and only if $M \otimes \Q_p$ is crystalline with Hodge-Tate weights in $[a, b]$. In particular, if $M$ is a finite free $O[G_{\Q_p}]$-module, then $H^1_f(\Q_p, M) = H^1_{f,\tors} (\Q_p, M)$. 
\end{prop}
\begin{proof}
    This is  \cite[\S7.3]{liu2007torsion}.\footnote{In fact, for the main results it suffices to assume $b-a < p - 1$, in which case Proposition \ref{prop torsion crystalline iff crystalline} is  due to Breuil \cite[Proposition 6]{breuil1999remarque}.}
\end{proof}
\begin{notation}
We will from now on write $H^1_f(\Q_p, M)$ in place of $H^1_{f,\tors}(\Q_p, M)$ for any finitely generated $O$-module $M$ with $G_{\Q_p}$-action; by Proposition \ref{prop torsion crystalline iff crystalline}, this  is consistent with Notation \ref{notation:BK_etc}(\ref{notation:BK_part_tf}).
\end{notation}
\subsection{Duality}
\subsubsection{}
Let $\l$ be a prime, possibly equal to $p$.
If $M$ is a finite free $O$-module with $G_{\Q_\l}$-action, let $M^\ast= \Hom_O(M, O(1))$. Similarly, if $M$ is a  locally compact $O$-module with $G_{\Q_\l}$-action, let $M^\vee= \Hom_O(M, E/O(1))$ be the Cartier dual. 
We have the local Tate pairings
\begin{equation}\label{eq:tate_pairing_pontryagin}
 \langle \cdot, \cdot\rangle_\l:   H^1(\Q_\l, M) \times H^1(\Q_\l, M^\vee) \to H^2(\Q_\l, E/O(1)) = E/O
\end{equation}
and 
\begin{equation}\label{eq:tate_pairing_O}
   \langle \cdot, \cdot\rangle_\l:   H^1(\Q_\l, M) \times H^1(\Q_\l, M^\ast) \to H^2(\Q_\l, O(1)) = O.
\end{equation}
The former is perfect, and the latter is perfect modulo torsion.
We recall the following standard fact.

\begin{lemma}\label{lem:zero_local_pairing}
Suppose $M$ is a finite free $O$-module with $G_{\Q_\l}$-action. 
Then the subspaces
  $H^1_f(\Q_\l, M)$ and $H^1_f(\Q_\l, M^\ast)$ pair to zero under the local Tate pairing (\ref{eq:tate_pairing_O}).
\end{lemma}
\begin{proof}
When $O = \Z_p$, 
the result follows from  \cite[Proposition 3.8]{bloch1990lfunctions}.\footnote{Although \emph{loc. cit.} assumes $M\otimes \Q_p$ is de Rham when $\l = p$, this is needed to show that $H^1_f(\Q_\l, M)$ and $H^1_f(\Q_\l, M^\ast)$  are \emph{exact} annihilators; the proof shows that they annihilate each other in general.} 
To reduce to this case, note that $M^\ast$ is canonically isomorphic to $M'\coloneqq \Hom_{\Z_p} (M, \Z_p)$, and we have a commutative diagram:
\begin{center}
\begin{tikzcd}
H^1_f(\Q_\l, M) \times H^1_f(\Q_\l, M^\ast) \arrow [d, "\sim",{anchor=west}]\arrow [r, "{\langle\cdot,\cdot\rangle_p}"] & O\arrow[d, "\tr"]\\H ^ 1_f (\Q_\l, M)\times H^1_f(\Q_\l, M')\arrow [r, "{\langle\cdot,\cdot\rangle_p} = 0"] &\Z_p.
\end{tikzcd}
\end{center}
Since $H^1_f(\Q_\l, M^\ast)\subset H^1(\Q_\l, M^\ast)$ is $O$-stable and the trace pairing $O \times O \to \Z_p$ is nondegenerate, the lemma follows.
\end{proof}

\begin{lemma}\label{lemma kill pairing on crystalline}
    Fix integers $a\leq0 \leq b$, and let $M$ be a finite free $O$-module with $[G_{\Q_p}]$-action which is torsion crystalline with Hodge-Tate weights in $[a, b]$. Then:
    \begin{enumerate}
        \item\label{lemma kill pairing on crystalline part one} We have $\varprojlim H^1_{f} (\Q_p, M/\varpi^n) = H^1_f (\Q_p, M)$. 
        \item\label{lem:kill_cryst_1.5}
        For all $n \geq 1$, there exists $m \geq n$ such that $$\im \left(H^1_f(\Q_p, M/\varpi^m) \to H^1_f(\Q_p, M/\varpi^n)\right) = \im \left(H^1_f(\Q_p, M) \to H^1_f(\Q_p, M/\varpi^n)\right).$$
        \item\label{lemma kill pairing on crystalline part two}For all $n\geq 1$, there exists $m \geq n$ such that the local Tate pairing
        $$\langle \cdot, \cdot \rangle: H^1(\Q_p, M/\varpi^n) \times H^1(\Q_p, M^\ast/ \varpi^n) \to O/\varpi^n$$
        vanishes when restricted to 
        $$\im \left(H^1_{f}(\Q_p, M/\varpi^m) \to H^1_{f} (\Q_p, M/\varpi^n)\right)\times \im \left(H^1_{f}(\Q_p, M^\ast/\varpi^m) \to H^1_{f} (\Q_p, M^\ast/\varpi^n)\right).$$ 
    \end{enumerate}
\end{lemma}
\begin{rmk}
    When $2(b-a) \leq p -2$ and $O = \Z_p$, we may take $m = n$ in the final part by \cite[Corollary 6.1]{niziol1993cohomology}.
\end{rmk}
\begin{proof}
    The first part is clear from Proposition \ref{prop torsion crystalline iff crystalline}, and it follows that $$\cap _{m\geq n} \im \left(H^1_{f} (\Q_p, M/\varpi^m) \to H^1_{f} (\Q_p, M/\varpi^n)\right) = \im \left(H^1_f(\Q_p, M) \to H^1_{f} (\Q_p, M/\varpi^n\right).$$
     Since $H^1_{f} (\Q_p, M/\varpi^n)$ is a finite $O$-module, 
it is clear that (\ref{lem:kill_cryst_1.5}) holds.
    
    For (\ref{lemma kill pairing on crystalline part two}), take $m$ sufficiently large to satisfy (\ref{lem:kill_cryst_1.5})  for both $M$ and $M^\ast$. The assertion then follows from Lemma \ref{lem:zero_local_pairing} and the commutativity of the diagram:
\begin{center}
\begin{tikzcd}
    H^1(\Q_p, M) \times H^1(\Q_p, M^\ast) \arrow[r] \arrow[d] & O\arrow[d]\\
    H^1(\Q_p, M/\varpi^n) \times H^1(\Q_p, M^\ast/\varpi^n) \arrow[r] & O/\varpi^n.
\end{tikzcd}
\end{center}
\end{proof}

\section {Automorphic representations, Hecke algebras, and Shimura varieties}\label{sec:prelim_ARs}

\subsection{Hecke algebras and Galois representations}
Let $G$ be a split, connected, reductive algebraic group over $\Z[S_0^{-1}]$ for a finite set of primes $S_0$, with Borel subgroup $B = TU\subset G$ and Weyl group $W_G$. For simplicity, we assume that the derived subgroup of $G$ is simply connected. 
\begin{definition}
    \leavevmode
    \begin{enumerate}
        \item For a prime $\l\not\in S_0$ and a ring $R$, let $\T_{G, \l, R}$ denote the spherical Hecke algebra of compactly supported, $R$-valued, $G(\Z_\l)$-biinvariant functions on $G(\Q_\l)$.
 For a finite set $S\supset S_0$ and a ring $R$,  $$ \T_{G, R}^S \coloneqq \bigotimes'_{\l\not\in S} \T_{G, \l, R}.$$
\item Let $P = MN \subset G$ be a parabolic subgroup. For all $\l\not\in S_0$, we define a natural map $$S_M^G: \T_{G, \l, R}\to \T_{M, \l, R}$$ by 
$$S_M^G(f) (m) = \int_{N(\Q_\l)} f(mn) \d n,$$
where the Haar measure on $N(\Q_\l)$ gives volume 1 to $N(\Z_\l)$. 
    \end{enumerate}
\end{definition}
When $R = \Z$, we drop it from the notation. 

\begin{prop}\label{prop:Hecke_diagram_induction}
We have a commutative diagram of functors
\begin{equation}
    \begin{tikzcd}
        R[M(\Q_\l)]-\operatorname{Mod} \arrow[r, "\Gamma_{M(\Z_\l)}"] \arrow[d, "\natural-\Ind"] & \T_{M, \l, R}-\operatorname{Mod} \arrow[d, "(S_M^G)^\ast"]\\
        R[G(\Q_\l)]-\operatorname{Mod} \arrow[r, "\Gamma_{G(\Z_\l)}"] & \T_{G, \l, R}-\operatorname{Mod},
    \end{tikzcd}
\end{equation}
where $\natural-\Ind$ denotes the unnormalized parabolic induction. 
\end{prop}
\begin{proof}
    This follows from combining Lemmas 2.4, 2.7, and 2.9 of \cite{newton2016torsion}.
\end{proof}

\begin{notation}
    Let $\widehat T \subset \widehat G$ be the dual torus to $T$.     We write $X^\bullet$ and $X_\bullet$ for the character and cocharacter groups of any split algebraic torus.  
\end{notation}
\begin{definition}
Suppose now that $R$ is a $\Z[\l^{1/2},\l^{-1/2}]$-algebra
and recall the Satake transform
\begin{equation}
\begin{split}
    R[X^\bullet(\widehat T)]^{W_G} &\isomorphism \T_{G, \l, R}\\
    \lambda &\mapsto [c_\lambda].
    \end{split}
\end{equation}
\begin{enumerate}
\item 
    If $\m \subset \T_{G, \l, R}$ is a maximal ideal with residue field $k$, the \emph{Satake parameter} for $\m$ is the unique element
        $$\Sat_{G, \l}(\m ) \in \widehat T(\overline k)/W_G$$ such that, for any $\lambda \in R[X^\bullet (\widehat T)]^{W_G}$, 
        $\lambda (\Sat_{G,\l}(\m)) = [c_\lambda] \pmod \m.$ A maximal ideal $\m\subset  {\T}^S_{G, R}$ defines a maximal ideal of each $\T_{G, \l, R}$ with $\l\not\in S$; we  denote by $\Sat_{G, \l}(\m)$ the corresponding element of $\widehat T(\overline k)/W_G$ for each $\l\not\in S$. 
        \item Let $\rho_G\in X_\bullet (\widehat T) = X^\bullet (T)$ be the half-sum of positive roots.  A \emph{normalization} of the Satake transform for $G$ (cf. \cite[\S8]{gross1998satake}) is a choice of element $\omega_G \in X_\bullet( Z_{\widehat G})$ such that
    $$\omega_G \equiv \rho_G \pmod{2 X_\bullet (\widehat T)}.$$
    \end{enumerate}
\end{definition}

\subsubsection{}
For the rest of this subsection, we assume that $R$ is a $\Z_{p^2}$-algebra, with a fixed choice of square root $\l^{1/2} \in \Z_{p^2}$ for any $\l \neq p$.  
\begin{definition}\label{def:Galois_type}
Let $S\supset S_0$ be a finite set of primes. 
\begin{enumerate}
    \item 
    Given a maximal ideal $\m \subset   {\T}^S_{G, R}$ with residue field $k$, a semisimple Galois representation
    $$\overline\rho_{\m}: G_{\Q} \to \widehat G(\overline{k})$$ is said to be \emph{associated} to $\m$ (with respect to a normalization $\omega_G$ of the Satake transform) if  for all but finitely many $\l\not\in S$, $\overline\rho_\m|_{G_{\Q_\l}}$ is unramified and
    $$\overline\rho_\m (\Frob_{\l^{-1}})^{ss} \sim \omega_G(\l^{-1/2}) \cdot \Sat_{G,\l}(\m);$$ here $\sim$ denotes $\widehat G(\overline{k})$-conjugacy. 
    If this holds for \emph{all} $\l\not\in S\cup \set{p}$, then $\overline\rho_\m$ is said to be \emph{strongly associated} with $\m$. 
\item
    If there exists a Galois representation (strongly) associated  to $\m$, then $\m$ is said to be (strongly) of \emph{Galois type}.
    \end{enumerate}

\end{definition}
Note that whether $\m$ is (strongly) of Galois type is independent of the choice of normalization; the corresponding representations $\overline\rho_\m$ differ by the composite of $\chi_p^\cyc$ and an algebraic cocharacter of $Z_{\widehat G}$. 
    \begin{example}
    Suppose $\pi$ is an automorphic representation of $\GSP_4(\A)$, unramified outside $S$, and fix an isomorphism $\iota: \overline\Q_p \isomorphism \C$. This determines a distinguished square root $\l^{1/2} \in \overline\Q_p$ for all $\l$.  Moreover, the Hecke action on $\iota^{-1}\pi^S$ defines a maximal ideal $\m_\pi\subset 
   \T^S_{\GSP_4, \overline \Q_p}$.
 If $\pi$ is relevant in the sense of Definition \ref{def:relevant_GSP} below, then the representation $\rho_{\pi,\iota}$ of Theorem \ref{thm:rho_pi_LLC}  is strongly associated to $\m_\pi$ with respect to the normalization given by the scalar subgroup $\mathbb G_m \hookrightarrow \GSP_4$. 
\end{example}

\begin{definition}
    Let $S \supset S_0$ be a finite set of primes. A maximal ideal $\m\subset {\T}_{G, R}^S$ with residue field $k$ is said to be \emph{Eisenstein} if there exists an associated $\overline \rho_\m$ (with respect to any normalization) which factors as
    $$\overline\rho_\m: G_{\Q} \to \widehat M(\overline{k}) \hookrightarrow \widehat G(\overline{k})$$ for some standard parabolic subgroup $P = MN\subsetneq G$.
    \end{definition}
\begin{prop}\label{prop:m_pullback_eisenstein}
Let $S\supset S_0$ be a finite set of primes. 
Suppose for some  proper parabolic subgroup $P = MN\subsetneq G$, $\m_M \subset {\T}^S_{M, R}$ is a maximal ideal of Galois type. 
Then
    $\m= (S_M^G)^\ast \m_M \subset {\T}^S_{G, R}$ is Eisenstein.
\end{prop}
\begin{proof}
Fix normalizations $\omega_M$ and $\omega_G$ of the Satake transform for $M$ and $G$, respectively. Also let $d_P \in X_\bullet (Z_{\widehat M})$ be the 
dual of the character $$\det(\ad(-)| N): M \to \mathbb G_m.$$
By the well-known compatibility of the Satake transform with \emph{normalized} parabolic induction, combined with Proposition \ref{prop:Hecke_diagram_induction}, $\Sat_{G, \l}(\m) \in \widehat T(\overline k)/W_G$ is represented by any representative of $\Sat_{M, \l}(\m_M) \cdot d_P(\l^{1/2}) \in \widehat T(\overline k)/W_M.$ It follows that 
$$\overline\rho_\m \coloneqq \overline \rho_{\m_M} \otimes (\chi_p^\cyc)^\alpha: G_\Q \to \widehat M(\overline k) \hookrightarrow \widehat G(\overline k)$$ is associated to $\m$, where 
$$\alpha \coloneqq \frac{\omega_G -\omega_M - d_P}{2} \in X_\bullet(Z_{\widehat M}).$$ Hence $\m$ is Eisenstein. 

\end{proof}
\subsubsection{The case of $G= \GSP_4$}\label{subsubsec:GSP4_hecke}
When $G= \GSP_4$, the Hecke algebra $\T_{G,\l,R}$ can be identified with the polynomial algebra $R[T_{\l,1}, T_{\l,2}, \langle \l\rangle]$, where the generators correspond to the following double coset operators:

\begin{align*}
    T_{\l,1} &=  \mathbbm{1} \left(\GSP_4(\Z_\l)\begin{pmatrix} \l & && \\ & 1&&\\ && \l^{-1}&\\&&&1 \end{pmatrix}\GSP_4(\Z_\l)\right),\\
T_{\l,2} &=  \mathbbm{1} \left(\GSP_4(\Z_\l)\begin{pmatrix} \l & && \\ & \l&&\\ && 1&\\&&&1 \end{pmatrix} \GSP_4(\Z_\l)\right),\\
    \langle \l \rangle &= \mathbbm{1} \left(\GSP_4(\Z_\l)\begin{pmatrix} \l & &&\\ &\l&&\\&&\l&\\&&&\l\end{pmatrix}\GSP_4(\Z_\l)\right).
\end{align*}

The Satake parameter of a maximal ideal $\m\subset \T_{\GSP_4,\l, R}$ with residue field $k$
can be identified with the data of an element $\nu\in k$ -- the similitude factor of $\Sat_{\GSP_4,\l}(\m)$ -- together with the multi-set 
$\set{\alpha,\beta,\nu/\alpha,\nu/\beta}$ of elements of $\overline{k}$ --  the eigenvalues of $\Sat_{\GSP_4,\l}(\m)$ in the standard four-dimensional representation of $\widehat{\GSP_4} = \GSP_4$. 
We will abusively write the Satake parameter as $\set{\alpha,\beta,\nu/\alpha,\nu/\beta}$; in general, $\nu$ is not always determined by this unordered set.


The relation to Hecke eigenvalues is given explicitly in this case by
\begin{equation}\label{eq:hecke_eigenvalues_satake}
\begin{split}
T_{\l,1} &= \l^2\left(\alpha\beta/\nu + \alpha/\beta + \beta/\alpha + \nu/\alpha\beta \right)+ (\l^2 - 1) \pmod\m\\
    T_{\l, 2} &= \l^{3/2}\left(\alpha + \beta + \nu/\alpha + \nu/\beta \right)\pmod\m\\
    \langle \l\rangle &= \nu \pmod\m.
\end{split}
\end{equation}

When it is clear from context that $G = \GSP_4$, the subscript $G$ in Hecke algebras and Satake parameters may be omitted from the notation.

\subsection{Automorphic forms and Galois representations}
Let $F$ be a totally real field. If we fix a prime $p$ and an isomorphism $\iota: \overline\Q_p\isomorphism \C$, then to each archimedean place $v|\infty$ of $F$, we can associate an embedding $\iota^\ast v: F \hookrightarrow\overline\Q_p$. 
\begin{thm}\label{thm:rho_GL2_LLC}
Suppose $\pi$ is a unitary, cuspidal automorphic representation of $\GL_2(\A_F)$ associated to a Hilbert modular form of weight $(2k_v)_{v|\infty}$, where $v$ runs over archimedean places of $F$ and $k_v\geq 1$ for all $v$. Then for every isomorphism $\iota: \overline\Q_p \isomorphism \C$, with $p$ a prime, there exists a Galois representation
$$\rho_{\pi,\iota}: G_F \to \GL_2(\overline \Q_p)$$ with the following properties.
\begin{enumerate}
\item \label{part:rho_GL2_LLC_1}$\rho_{\pi,\iota}|_{G_{F_v}}$ is potentially semistable for all $v|p$, and for all nonarchimedean primes $v$ of $F$:
$$\iota WD(\rho_{\pi,\iota}|_{G_{F_v}})^{F-ss} \simeq \rec (\pi_v\otimes |\cdot |^{\frac{1}{2}}).$$
     Moreover this Weil-Deligne representation is pure of weight $-1$. 
\item\label{part:rho_GL2_LLC_HT} For each place $v|\infty$ of $F$, 
the Hodge-Tate weights of $\rho_{\pi, \iota}$ with respect to the embedding $\iota^\ast v: F\hookrightarrow\overline\Q_p$ are $1-{k_v}$ and ${k_v}$.
    \item \label{part:rho_GL2_LLC_char}The similitude character of $\rho_{\pi,\iota}$ is $\iota ^{-1}\rec(\omega_{\pi}) \chi_{p,\cyc}$, where $\omega_\pi$ is the central character. 
\end{enumerate}
\end{thm}
 Here $\rec$ is the usual local Langlands correspondence for $\GL_n$, normalized to coincide with the reciprocity map of local class field theory  when $n = 1$. 

\begin{proof}
Note that the purity in (\ref{part:rho_GL2_LLC_1}) follows immediately from the claimed identity of Weil-Deligne representations and the  Ramanujan conjecture for $\pi$ \cite{blasius2006ramanujan}. 

The existence of $\rho_{\pi,\iota}$ with the property (\ref{part:rho_GL2_LLC_1}) for all $v\nmid p$ was proved by Carayol \cite{carayol1986hilbert} under the assumption that either $[F:\Q]$ is odd, or $\pi_v$ is square-integrable for some finite prime $v$ of $F$. In general, $\rho_{\pi,\iota}$ was constructed by Taylor \cite{taylor1989hilbert}, and the proof of (\ref{part:rho_GL2_LLC_1}) for $v\nmid p$  follows the argument of \cite[Theorem 2.1.3]{wiles1988ordinary}. Finally, the property (\ref{part:rho_GL2_LLC_1}) for $v|p$, along with (\ref{part:rho_GL2_LLC_HT}), were established in general by Skinner \cite{skinner2009note} (except that our normalizations of Hodge-Tate weights and reciprocity maps are inverted from \emph{loc. cit.}). The property (\ref{part:rho_GL2_LLC_char}) is an immediate corollary of (\ref{part:rho_GL2_LLC_1}).
\end{proof}

\begin{definition}\label{def:coefficient_GL2}
    An automorphic representation $\pi$ of $\GL_2(\A_F)$ has (strong) coefficient field $E_0 \subset \C$ if $E_0$ is a number field and, for all primes $p$ and isomorphisms $\iota: \overline\Q_p \isomorphism \C$, $\rho_{\pi,\iota}$ is defined over the $p$-adic closure of $\iota^{-1} (E_0)$. In this case, it depends only on the prime $\p$ of $E_0$ induced by $\C$, and we obtain a well-defined $\rho_{\pi,\p} : G_F \to \GL_2(E_{0,\p})$ such that $\rho_{\pi,\iota}$ is the extension of scalars of $\rho_{\pi,\p}$.  
\end{definition}
\begin{rmk}\label{rmk:coeff_field_GL2}
    The argument of \cite[Proposition 3.2.5]{chenevier2013construction} shows that strong coeffient fields exist in the situation of Theorem \ref{thm:rho_GL2_LLC}.
\end{rmk}
\begin{notation}\label{notation:V_T_GL_2}
    Let $\pi$ be as in Theorem \ref{thm:rho_GL2_LLC}, and let $E_0$ be a strong coefficient field of $\pi$. Fix a prime $\p$ of $E_0$ with residue field $k(\p)$. 
    \begin{enumerate}
    \item Write $V_{\pi,\p}$ for the underlying $E_{0,\p}[G_F]$-module of $\rho_{\pi,\p}$. 
    \item Let $T_{\pi,\p} \subset V_{\pi,\p}$ be a $G_F$-stable $O_{E_0,\p}$-lattice. Then the $k(\p)[G_F]$-module $\overline T_{\pi,\p}$ depends only on $V_{\pi,\p}$ up to semisimplification. We write
    $$\overline\rho_{\pi,\p} : G_F \to \GL_2(k(\p))$$ for the corresponding semisimple Galois representation.
    \item When $\p$ is clear from context, it may be dropped from all subscripts in the above notations.
    \end{enumerate}
\end{notation}
We now turn to automorphic representations of $\GSP_4$. 

\begin{definition}\label{def:relevant_GSP}
    An automorphic representation $\pi$ of $\GSP_4(\A)$ will be called \emph{relevant} if:
    \begin{itemize}
        \item $\pi$ is cuspidal and not CAP, and has unitary central character.
        \item $\pi_\infty$ belongs to the discrete series $L$-packet of weight $(3,3)$.
    \end{itemize}
\end{definition}
\subsubsection{}
If $\pi$ is an automorphic representation of $\GSP_4(\A)$, then for all $\l$ such that $\pi_\l$ is unramified, recall that we write the Satake parameter of $\pi_\l$ as a multiset of four complex numbers of the form $\set{\alpha, \beta, \nu \alpha^{-1}, \nu\beta^{-1}}$. If $\pi$ has central character $\omega_\pi$, then $\nu  = \omega_\pi(\langle \l\rangle)= \rec(\omega_\pi)(\Frob_\l^{-1}).$

\begin{definition}\label{def:endoscopic_GSP}
    A cuspidal, non-CAP automorphic representation $\pi$ of $\GSP_4(\A)$ is \emph{endoscopic} associated to an unordered pair $(\pi_1,\pi_2)$ of cuspidal  automorphic representations of $\GL_2(\A)$ with the same central character if the following holds: for all but finitely many primes $\l$,    
    if $\pi_{i,\l}$ are both unramified with Satake parameters  $\set{\alpha_i, \nu/\alpha_i}$, then  $\pi_\l$ is unramified with   Satake parameter $\set{\alpha_1, \alpha_2, \nu/\alpha_1, \nu/\alpha_2}$. 
\end{definition}
\subsubsection{}
In this case, $\pi_1$ and $\pi_2$ are necessarily distinct by \cite[Lemma 5.2]{weissauer2009endoscopy}, and  the central character of $\pi$ is the common 
central character of the $\pi_i$. 
In the literature, the property in Definition \ref{def:endoscopic_GSP} is often called being \emph{weakly} endoscopic, and being endoscopic requires a stronger compatibility condition of local Langlands parameters at all finite places; however, the distinction is unimportant, cf. the results of \cite[\S5]{weissauer2009endoscopy}.
\begin{lemma}\label{lem:when_endoscopic_relevant}
    Suppose $\pi$ is an endoscopic automorphic representation of $\GSP_4(\A)$, associated to a pair $(\pi_1,\pi_2)$ of automorphic representations of $\GL_2(\A)$. Then $\pi$ is relevant if and only if $\pi_1$ and $\pi_2$ are unitary with discrete series archimedean components of weights 2 and 4 in some order.
\end{lemma}
\begin{proof}
Recall $\pi$ is cuspidal and not CAP by definition. Also, $\pi$ 
clearly has unitary central character if and only if $\pi_1$ 
and $\pi_2$ are unitary, so it suffices to consider the archimedean weights.

Consider the representation $\pi_{1,\infty}\boxtimes \pi_{2,\infty}$ of $\operatorname{GSO}_{2,2}(\R) = (\GL_2(\R)\times \GL_2(\R))/\R^\times$;
by \cite[Lemma 5.6]{weissauer2009endoscopy}, $\pi_\infty$ belongs to the local $L$-packet attached to $\pi_{1,\infty}\boxtimes\pi_{2,\infty}$ via the known Langlands parametrization for real groups and the map of dual groups 
$${}^L\operatorname{GSO}_{2,2} = (\GL_2\times_{\mathbb G_m} \GL_2)(\C) \hookrightarrow\GSP_4(\C) = {}^L\GSP_{4,\R}.$$ We conclude that $\pi_\infty$ belongs to the desired discrete series $L$-packet if and only if $\pi_{1,\infty}$ and $\pi_{2,\infty}$ are discrete series of weights 2 and 4 in some order.

    \end{proof}

\begin{thm}\label{thm:rho_pi_LLC}
    Let $\pi$ be a relevant automorphic representation of $\GSP_4(\A)$, and fix a prime $p$ along with an isomorphism $\iota:\overline\Q_p \isomorphism \C$. Then there exists a semisimple Galois representation
    $$\rho_{\pi,\iota}: G_\Q \to \GSP_4(\overline \Q_p)$$ with the following properties.
    \begin{enumerate}
        \item \label{part:rho_pi_LLC1} $\rho_{\pi,\iota}|_{G_{\Q_p}}$ is potentially semistable, and   for all primes $\l$, 
     $$\iota WD(\rho_{\pi,\iota}|_{G_{\Q_\l}})^{F-ss} \simeq \rec_{\operatorname{GT}} (\pi_\l\otimes |\cdot|^{\frac{1}{2}}).$$
     Moreover this Weil-Deligne representation is pure of weight $-1$. 
             \item \label{part:rho_pi_LLC_HT}The Hodge-Tate weights of $\rho_{\pi,\iota}|_{G_{\Q_p}}$ are $\set{-1,0,1,2}$. 

        \item \label{part:rho_pi_char}The similitude character of $\rho_{\pi,\iota}$ is $\iota^{-1}\rec(\omega_\pi)\chi_{p,\cyc}$, where $\omega_\pi$ is the central character. 
    \end{enumerate}
\end{thm}
Here $\rec_{\operatorname{GT}}$ is the Gan-Takeda local Langlands correspondence of \cite{gantakeda2011local}, which associates to an irreducible admissible representation of $\GSP_4(\Q_\l)$ a $\GSP_4(\C)$-valued Weil-Deligne representation.
\begin{proof}
Suppose first that $\pi$ is endoscopic associated to a pair $(\pi_1, \pi_2)$ of automorphic representations of $\GL_2(\A)$.  By Lemma \ref{lem:when_endoscopic_relevant} and Theorem \ref{thm:rho_GL2_LLC}, we can take $$\rho_{\pi,\iota} \coloneqq \rho_{\pi_1,\iota} \oplus \rho_{\pi_2,\iota}$$ under the natural embedding $\GL_2\times_{\mathbb G_m}\GL_2 \to \GSP_4$; this satisfies (\ref{part:rho_pi_LLC1}) by Theorem  \ref{thm:rho_GL2_LLC}(\ref{part:rho_GL2_LLC_1}) combined with \cite[Corollary 5.1, Theorem 5.2(3)]{weissauer2009endoscopy} and the construction of $\rec_{\operatorname{GT}}$ in \cite{gantakeda2011local}. Then (\ref{part:rho_pi_LLC_HT}) and (\ref{part:rho_pi_char}) are satisfied by Theorem \ref{thm:rho_GL2_LLC}(\ref{part:rho_GL2_LLC_HT}, \ref{part:rho_GL2_LLC_char}).

Now suppose $\pi$ is not endoscopic. By \cite[Proposition 10.1]{rosner2024global} combined with \cite[Theorem 1]{weissauer2008whittaker}, we may assume without loss of generality that $\pi$ is globally generic. 
Under the extra hypothesis that $\pi_v$ is Steinberg for some $v\neq p$, and omitting the case $\l = p$ of (\ref{part:rho_pi_LLC1}), the theorem then follows from the main result of \cite{sorensen2010galois}, except that we have twisted the $\rho_{\pi,\iota}$ of \emph{loc. cit.} by $(\chi_p^{\cyc})^2$. The extra hypothesis  in \cite{sorensen2010galois} is needed only to appeal to the results of  Taylor-Yoshida \cite{tayloryoshida2007}, 
 and has since been removed by work of Caraiani \cite{caraiani2012local}. Similarly, the proof in \cite{sorensen2010galois} of (\ref{part:rho_pi_LLC1}) for $\l \neq p$ extends to the case $\l=p$ with the additional input of \cite{caraiani2014monodromy}.
\end{proof}

\begin{notation}
In the setting of Theorem \ref{thm:rho_pi_LLC},
    we write $V_{\pi,\iota}$ for the underlying four-dimensional Galois module of $\rho_{\pi,\iota}$. 
\end{notation}
\begin{lemma}\label{lem:reducible_endoscopic}
    In the situation of Theorem \ref{thm:rho_pi_LLC}, suppose $p > 3$. Then $V_{\pi,\iota}$ is reducible if and only if $\pi$ is endoscopic. 
\end{lemma}
\begin{proof}
    This is \cite[Theorem 3.1]{weiss2022images}. 
\end{proof}

\begin{definition}\label{def:coefficient_pi}
    A relevant automorphic representation $\pi$ of $\GSP_4(\A)$ has (strong) coefficient field $E_0 \subset \C$ if $E_0$ is a number field and, for all primes $p$ and isomorphisms $\iota: \overline\Q_p \isomorphism \C$, $\rho_{\pi,\iota}$ is defined over the $p$-adic closure of $\iota^{-1} (E_0)$. In this case, it depends only on the prime $\p$ of $E_0$ induced by $\iota$, and we obtain a well-defined $\rho_{\pi,\p} : G_F \to \GSP_4(E_{0,\p})$ such that $\rho_{\pi,\iota}$ is the extension of scalars of $\rho_{\pi,\p}$.  
\end{definition}
\begin{lemma}\label{lem:coeff_field_exists}
    If $\pi$ is a relevant automorphic representation of $\GSP_4(\A)$, then a strong coefficient field $E_0$ exists for $\pi$.
\end{lemma}
\begin{proof}
If $\pi$ is endoscopic, this is clear from Remark \ref{rmk:coeff_field_GL2}, so assume otherwise.

    Let $r: \GSP_4 \hookrightarrow \GL_4$ be the natural embedding, and let $\overline{\iota^{-1}(E_0)}$ be the $p$-adic closure for any $\iota: \overline\Q_p \isomorphism \C$. By the argument of \cite[Proposition 3.2.5]{chenevier2013construction}, 
    there exists a number field $E_0$ such that, for all $\iota: \overline\Q_p \isomorphism \C$, $r\circ \rho_{\pi,\iota}$ is defined over  $\overline{\iota^{-1}(E_0)}$  and $\rho_{\pi,\iota}(g)$ has distinct eigenvalues in $\overline{\iota^{-1}(E_0)}^\times$ for some $g\in G_\Q$. 

It suffices to check 
 $\rho_{\pi,\iota}$ is defined over $\overline{\iota^{-1}(E_0)}$ whenever $p > 3$. For this, let $G$ be the absolute Galois group of $\overline{\iota^{-1}(E_0)}$. For all $\sigma \in G$, we have
    $$\rho_{\pi,\iota}^\sigma =h(\sigma ) \rho_{\pi,\iota} h(\sigma)^{-1}$$ for some $h(\sigma) \in \GL_4(\overline\Q_p)$. Using the absolute irreducibility from Lemma \ref{lem:reducible_endoscopic} and Schur's lemma, we conclude $h(\sigma) \in \GSP_4(\overline\Q_p)$, and $\sigma \mapsto h(\sigma)$ defines a cocycle $h\in H^1(G, \operatorname{PGSp}_4(\overline\Q_p))$. The class of $h$ determines an inner form $H$ of $\GSP_4$ over $\overline{\iota^{-1}(E_0)}$ such that $\rho_{\pi,\iota}$ can be conjugated to lie in $H(\overline{\iota^{-1}(E_0)})$. However, since $\rho_{\pi,\iota}(g)$ has distinct eigenvalues in 
    $\overline{\iota^{-1}(E_0)}$  for some $g\in G_\Q$, $H$ must be split, and this completes the proof. 
    \end{proof}
\subsubsection{}

Analogously to Notation \ref{notation:V_T_GL_2}, we make the following notations.
\begin{notation}\label{notation:V_T_pi}
    Let $\pi$ be as in Theorem \ref{thm:rho_pi_LLC}, and let $E_0$ be a strong coefficient field of $\pi$. Fix a prime $\p$ of $E_0$ with residue field $k(\p)$. 
    \begin{enumerate}
    \item Write $V_{\pi,\p}$ for the four-dimensional underlying $E_{0,\p}[G_\Q]$-module of $\rho_{\pi,\p}$.
    \item Let $T_{\pi,\p} \subset V_{\pi,\p}$ be any $G_\Q$-stable $O_{\p}$-lattice; 
    we define $\overline T_{\pi,\p} \coloneqq \left(T_{\pi,\p}/\varpi_\p\right)^{ss}$, which  depends only on $V_{\pi,\p}$.
    We also write
    $$\overline\rho_{\pi,\p} : G_F \to \GL_4(k(\p))$$ for the corresponding semisimple Galois representation.
    

    \item When $\p$ is clear from context, it may be dropped from all subscripts in the above notations.
    \end{enumerate}
\end{notation}


\begin{lemma}\label{lem:BC}
    Suppose $\pi$ is a relevant automorphic representation of $\GSP_4(\A)$. Then there exists a base change $\BC(\pi)$ to an automorphic representation of $\GL_4(\A)$, such that for each place $v$, $$\rec(\BC(\pi)_v) = r\circ\rec_{\operatorname{GT}}(\pi_v),$$ where $$r: {}^L \GSP_4 = \GSP_4(\C) \hookrightarrow \GL_4(\C) = {}^L \GL_4$$ is the natural embedding of dual groups. Moreover $\BC(\pi)$ is cuspidal if and only if $\pi$ is non-endoscopic. 
\end{lemma}
\begin{proof}
If $\pi$ is endoscopic associated to $(\pi_1,\pi_2)$, then there exists a non-cuspidal base change, which is the isobaric sum $\pi_1 \boxplus \pi_2$; the compatibility with local Langlands parameters is by the same reasoning as the endoscopic case of  Theorem \ref{thm:rho_pi_LLC}.

In the non-endoscopic case,  the lemma follows from \cite[Proposition 10.1]{rosner2024global} and its proof; note that the endoscopic transfer to $\GL_4\times \GL_1$ used in \emph{loc. cit.} is compatible with the Gan-Takeda local Langlands paremeters by the main result of \cite{chan2015local}.
\end{proof}

\begin{lemma}
    \label{lem:sym_cube_relevant}
        Suppose $\pi$ is a relevant automorphic representation of $\GSP_4(\A)$. If $\BC(\pi)$ is a symmetric cube lift of a non-CM automorphic representation $\pi_0$ of $\GL_2(\A)$, then $\pi_0$ has discrete series archimedean component of weight 2. 

\end{lemma}
\begin{proof}
    This follows from matching archimedean $L$-parameters using \cite[Theorem B]{kim2002functorial} and Lemma \ref{lem:BC}.
\end{proof}

\begin{lemma}\label{lem:RQ_induction_relevant}
    Suppose $\pi$ is a relevant automorphic representation of $\GSP_4(\A)$. If $\BC(\pi)$ is the automorphic induction of a non-CM automorphic representation $\pi_0$ of $\GL_2(\A_K)$ with $K$ real quadratic, then $\pi_0$ has discrete series archimedean components of weights 2 and 4, in some order. Moreover the central characters $\omega_{\pi_0}$ of $\pi_0$ and $\omega_\pi$ of $\pi$ satisfy
    $$\omega_\pi \circ \operatorname{Nm}_{K/\Q} = \omega_{\pi_0}.$$
\end{lemma}

\begin{proof}
    The assertion on archimedean components follows from Lemma \ref{lem:BC} and
     the compatibility of automorphic induction with local Langlands parameters \cite[Chapter 3, Theorem 5.1]{arthur1989basechange}. To check the relation of central characters, 
     fix some prime $p$ along with an isomorphism $\iota: \overline\Q_p \isomorphism \C$. Let $\pi_0^\tw$ denote the $\Gal(K/\Q)$-twist. 
     Then we have
     $$V_{\pi,\iota}|_{G_K} = \rho_{\pi_0,\iota} \oplus \rho_{\pi_0^\tw, \iota}.$$
     Using the identities
     $$V_{\pi,\iota} = V_{\pi,\iota}^\check \otimes \iota^{-1}\rec(\omega_\pi) \otimes \chi_{p,\cyc},\;\; \rho_{\pi_0,\iota} = \rho_{\pi_0,\iota}^\check \otimes \iota^{-1}\rec(\omega_{\pi_0})\otimes \chi_{p,\cyc},$$
and likewise for $\pi_0^\tw$, it follows that 
$$\rho_{\pi_0,\iota} \otimes \iota^{-1} \left(\rec(\omega_{\pi})|_{G_K}/\rec(\omega_{\pi_0})\right) \oplus \rho_{\pi_0^\tw,\iota} \otimes \iota^{-1} \left(\rec(\omega_{\pi})|_{G_K}/\rec(\omega_{\pi_0^\tw})\right)\cong \rho_{\pi_0,\iota} \oplus \rho_{\pi_0^\tw,\iota}.$$
     
     Since $\rho_{\pi_0,\iota}$ and $\rho_{\pi_0^\tw, \iota}$ are absolutely irreducible by Theorem \ref{thm:nekovar}, and have different Hodge-Tate weights by  Theorem \ref{thm:rho_GL2_LLC}(\ref{part:rho_GL2_LLC_HT}), this implies $\rec(\omega_\pi) |_{G_K} = \rec(\omega_{\pi_0})$, i.e. $\omega_\pi \circ \Nm_{K/\Q} = \omega_{\pi_0}$. 
     \end{proof}

     In the next lemma, if $K$ is a quadratic field, we write $\BC_{K/\Q}$ for the base change of an automorphic representation of $\GL_n(\A)$ to $\GL_n(\A_K)$, which exists by \cite{arthur1989basechange}.\begin{lemma}\label{lem:IQ_induction_relevant}
             Suppose $\pi$ is a relevant, non-endoscopic automorphic representation of $\GSP_4(\A)$. If 
 $\BC(\pi)$ is the automorphic induction of an automorphic representation $\pi_0$ of $\GL_{2}(\A_K)$  with $K$ imaginary quadratic and $\pi_0$ is not itself an automorphic induction, then 
$\pi_0$ is of the form $\BC_{K/\Q} (\sigma) \otimes \chi$, where $\sigma$ is the unitary automorphic representation of $\GL_2(\A)$ corresponding to a non-CM classical modular form of weight $k =  2$ or 3, and $\chi$ is a Hecke character of $K$.
     \end{lemma}
     \begin{proof}
         By hypothesis, $\BC_{K/\Q} \circ \BC(\pi) $ is an isobaric sum $\pi_0\,\boxplus\,\pi_0^\tw$, where $\tw$ denotes the $\Gal(K/\Q)$-twist.

         Considering archimedean $L$-parameters and using Lemma \ref{lem:BC} combined with \cite[Chapter 3, Theorem 5.1]{arthur1989basechange}, we see that the local $L$-parameter of $\pi_0$ at the unique archimedean place of $K$ is of the form 
         \begin{equation}\label{eq:arch_for_IQ}z \mapsto \begin{pmatrix}(z/\overline z)^{\epsilon_1\frac{1}{2}} & \\ & (z/\overline z)^{\epsilon _2\frac{3}{2}}\end{pmatrix}\end{equation}
         for $\epsilon_1,\epsilon_2 \in \set{\pm 1}$. Let $\omega_\pi: \A^\times \to \C^\times$ and $\omega_{\pi_0}$ be the central characters of $\pi$ and $\pi_0$, respectively; then because $\BC(\pi) \cong \BC(\pi)^\vee\otimes \omega_\pi$, we have
         \begin{align*}
             \pi_0^\vee\, \boxplus\,(\pi_0^\tw)^\vee&\cong \pi_0 \otimes \omega_\pi \circ \Nm_{K/\Q} \boxplus \,\pi_0^\tw \otimes \omega_\pi\circ \Nm_{K/\Q} \\
             &\cong \pi_0 \otimes \omega_{\pi_0}^{-1} \boxplus \pi_0^\tw \otimes \omega_{\pi_0}^{-\tw}.
         \end{align*} The archimedean $L$-parameter (\ref{eq:arch_for_IQ}) shows that $\pi_0^\vee \not\cong \pi_0 \otimes \omega_\pi\circ \Nm_{K/\Q} $, and hence \begin{equation}\label{eq:IQ_ind_mess}\pi_0^\vee \cong \pi_0^\tw \otimes \omega_\pi\circ\Nm_{K/\Q} \cong \pi_0\otimes \omega_{\pi_0}^{-1}.\end{equation}

         On the other hand, computing the central character of $\BC(\pi)$, we have:
\begin{equation}
    \omega_{\pi}^2 = \omega_{\pi_0}|_{\A^\times}. 
\end{equation}   
In particular, $(\omega_\pi\circ \Nm_{K/\Q})/\omega_{\pi_0}$ is trivial when restricted to $\A^\times$, hence of the form $\chi/\chi^\tw$ for an automorphic character $\chi$ of $\A_K^\times$. Then $(\pi_0 \otimes \chi^{-1})$ is isomorphic to its $\Gal(K/\Q)$-twist by (\ref{eq:IQ_ind_mess}), hence arises as the base change of a cuspidal automorphic representation $\sigma$ of $\GL_2(\A)$ by \cite[Chapter 3, Theorem 4.2]{arthur1989basechange}. Without loss of generality, we may assume $\sigma$ is unitary; 
 considering the archimedean local $L$-parameter of $\sigma$ and again using \cite[Chapter 3, Theorem 5.1]{arthur1989basechange}, the archimedean component of $\sigma$ is discrete series of weight 2 or 3. We have $\pi_0 =\BC_{K/\Q} (\sigma) \otimes \chi$ by construction. 
      \end{proof}


\subsection{Shimura varieties and Shimura sets}
\subsubsection{}
Let $V $
be a quadratic space over $\Q $, and recall the algebraic group $\spin(V)$ from (\ref{subsubsec:spin_groups_notation}). 

\subsubsection{Indefinite case}\label {subsubsection: indefinite case}
Suppose $V\otimes\R $
has signature $(n, 2) $. If $V ^ -\subset V_\R $
is a negative definite $2 $-plane, one obtains a map $$C ^ + (V ^ -)\simeq\C\rightarrow C ^ + (V_\R), $$
which induces a Shimura datum $$h:\Res_{\C/\R}\mathbb G_m\rightarrow\spin (V)_\R. $$
For a neat compact open subgroup $K\subset\spin (V) (\A_f) $, the resulting Shimura variety $\Shimura_K (V) $ is a smooth quasi-projective variety over $\Q $.
\subsubsection{Definite case}\label {subsubsection: definite case}
If $V $ is a positive definite quadratic space over $\Q $, then $\spin (V) (\R) $ is compact. For a compact open subgroup $K\subset\spin (V) (\A_f) $, let $\Shimura_K (V) $ denote the finite double coset space $$\Shimura_K (V)\coloneqq\spin (V) (\Q)\backslash\spin (V) (\A_f)/K. $$
 \subsubsection{Hecke algebras}
 Suppose $V = V_D$ is one of the quadratic spaces from (\ref{subsubsec:B_D_notation}), and let $O$ be a coefficient ring. 
 If
 $K = \prod_\l K_\l \subset \spin(V)(\A_f)$ is a neat compact open subgroup and $S$ is a finite set of primes of $\Q$ containing all those such that $K_\l$ is not hyperspecial, then $ {\T}^S_O= \T^S_{\GSP_4,O}$ acts on $H^\ast(\Sh_K(V)(\C), O)$; the cohomology is interpreted as $O[\Sh_K(V)]$ in the definite case. We denote by $\T^S_{ K,V_D,O}$ the quotient of $ {\T}^S_O$ defined by this action, and may drop the subscript $O$ when it is clear from context.


\subsection{Automorphic representations of $\spin_5$ groups}
\subsubsection{}
Recall the five-dimensional quadratic spaces $V_D$ from (\ref{subsubsec:B_D_notation}).
\begin{definition}\label{def:eisenstein}
    An automorphic representation $\pi$ of $\spin(V_D)(\A)$ is \emph{Eisenstein} if there exists a parabolic subgroup $P \subset \GSP_4$, with Levi factor $L$, and an automorphic representation $\sigma$ of $L(\A)$ such that $\pi$ is nearly equivalent to a constituent of $\Ind_{P(\A)}^{\GSP_4(\A)} \sigma$.
\end{definition}
The following  generalizes Definitions \ref{def:relevant_GSP} and \ref{def:endoscopic_GSP}. 
\begin{definition}\label{def:relevant}
    For a squarefree integer $D \geq 1$, an automorphic representation $\pi$ of $\spin(V_D)(\A)$ will be called \emph{relevant} if:
    \begin{enumerate}
    \item $\pi$ is  not Eisenstein (in particular is cuspidal), and has unitary central character.
        \item $\pi_\infty$ is 
        trivial if $\sigma(D)$ is odd (in which case $\spin(V_D)(\R)$ is a compact group); or 
              belongs to the discrete series $L$-packet of weight $(3,3)$ if $\sigma(D)$ is even (in which case $\spin(V_D)(\R) = \GSP_4(\R)$). 
    \end{enumerate}
    A non-Eisenstein automorphic representation $\pi$ of $\spin(V_D)(\A)$ is called \emph{endoscopic} associated to a pair $(\pi_1,\pi_2)$ of automorphic representations of $\GL_2(\A)$ if  the condition in Definition \ref{def:endoscopic_GSP} is satisfied for all but finitely many $\l\nmid D$. 
\end{definition}
\begin{lemma}\label{lem:coh_for_JL}
   Fix a squarefree $D\geq 1$ such that $\sigma(D)$ is {even},  and suppose $\pi$ is an automorphic representation of $\spin(V_D)(\A)$ that is not Eisenstein. 
 Then for any neat compact open subgroup $K= \prod K_\l$ of $\spin(V_D)(\A_f)$:
   \begin{enumerate}
   \item\label{item:coh_for_JL_c} The natural map  induces an isomorphism $$H^i_c(\Sh_K(V_D), \C)[\pi_f] \isomorphism H^i(\Sh_K(V_D), \C)[\pi_f].$$
   \item \label{item:coh_for_JL_matsushima}We have $$H^i(\Sh_K(V_D), \C)[\pi_f] = \bigoplus_{\pi_\infty'} m(\pi_f\otimes\pi_\infty') \cdot \pi_f^K\otimes  H^i(\mathfrak{gsp}_4, U(2); \pi_\infty') \neq 0,$$ where  $U(2) \subset \GSP_4(\R)$ is the maximal compact subgroup, $\pi_\infty'$ runs over cohomological representations of $\GSP_4(\R)$, and $m(\pi_f\otimes\pi_\infty')$ is the multiplicity in the discrete (equivalently cuspidal) automorphic spectrum of $\spin(V_D)(\A)$. 
       \item \label{item:coh_for_JL_3}  If $\pi$ is relevant and 
$H^i(\Sh_K(V_D), \C)[\pi_f] \neq 0$, then $i = 3.$
   \end{enumerate}
\end{lemma}
\begin{proof}
The first part is standard, cf. \cite[Chapter 9]{harder2019cohomology}. Then (\ref{item:coh_for_JL_matsushima}) is immediate from Matsushima's formula and the diagram in \cite[p. 293]{taylor1993ladic}. We now show (\ref{item:coh_for_JL_3}). Without loss of generality, we may assume  that $\pi_f^K\neq 0$; in particular, because $\pi$ is relevant, (\ref{item:coh_for_JL_matsushima}) implies that \begin{equation}\label{eq:contributes_H3}
    H^3(\Sh_K(V_D), \C)[\pi_f] \neq 0.
\end{equation}

    Now let $S$ be a finite set of places of $\Q$ such that $K_\l$ is hyperspecial for $\l\not\in S$, and fix a prime $p \not\in S \cup \set{2,3}$ along with 
    an isomorphism $\iota: \overline\Q_p \isomorphism \C$. It suffices to show \begin{equation}
    H^i_\et(\Sh_K(V_D)_{\overline\Q}, \overline \Q_p)[\pi_f] \neq 0\implies i = 3.
    \end{equation}
    For this, we argue as in \cite{taylor1993ladic}, with the additional input of the Fontaine-Mazur conjecture for $\GL_2$. 
Set $W^i$ to be the $\pi^S$-isotypic component\footnote{Defined using $\iota$. We elide $\iota$ from the discussion to ease the burden of notation.} of $
H^i_\et(\Sh_K(V_D)_{\overline\Q}, \overline\Q_p)$ with respect to the natural action of $ \T^S_{\overline\Q_p}$; then $W^i$ is a $G_\Q$-module.
We have the following facts:
\begin{enumerate}
\item\label{margulis_item} By   Margulis' superrigidity theorem  \cite[Chapter IX, Corollary 7.15(iii)]{margulis1991discrete}, $W^1 = 0$,  and clearly $W^0 = W^6 = 0$ because $\pi_f$ is not one-dimensional.
    \item\label{item:eichler} For each $\l\not\in S\cup\set{p}$,  $\Frob_\l$ satisfies the Eichler-Shimura relation on $H^\ast(\Sh_K(V_D)_{\overline\Q}, \overline \Q_p)$ by \cite{wedhorn2000congruence}; in particular,
    if $\pi_\l$ has Satake parameters $\set{\alpha,\beta,\nu\alpha^{-1},\nu\beta^{-1}}$, then    \begin{equation}\label{eq:eichler_shimura}(\Frob_\l^{-1} - \l^{3/2} \alpha) (\Frob_\l^{-1} - \l^{3/2} \beta) (\Frob_\l^{-1} - \l^{3/2}\nu\alpha^{-1}) (\Frob_\l^{-1} - \l^{3/2}\nu \beta^{-1}) = 0 \text{ on each } W^i.\end{equation}
    \item    \label{item:PD}  By Poincar\'e duality and part (\ref{item:coh_for_JL_c}) of the theorem there are perfect pairings
    $$W^i \times W^{6-i} \to \rec(\omega_\pi) (-3)$$ for all $i$, where $\omega_\pi$ is the central character of $\pi$, cf. \cite[p. 297]{taylor1993ladic}.
    \item\label{item:Weil} Again by (\ref{item:coh_for_JL_c}) of the theorem, for all $\l\not\in S\cup \set{p}$, the eigenvalues of $\Frob_\l$ on $W^i$ are Weil numbers of weight $i$. 
\end{enumerate}
Now suppose $W^i \neq 0$ for some $i \neq 3$. By (\ref{margulis_item}), $i \neq 0,1$, and then by (\ref{item:PD}) we conclude that $W^2 \neq 0$. If $\l\not\in S\cup \set{p}$  is a prime such that $\Frob_\l$ has $n$ distinct eigenvalues on $W^2$, it has $n$ distinct eigenvalues on $W^4$ as well by (\ref{item:PD}),
and at least one eigenvalue on $W^3$ by (\ref{eq:contributes_H3}).
In particular, by (\ref{item:Weil}) and the fact that $\Frob_\l$ has at most 4 total eigenvalues by (\ref{item:eichler}), 
 we conclude that $n = 1$ and $\Frob_\l$ also has at most two eigenvalues on $W^3$. Now the same argument as \cite[Proposition 3]{taylor1993ladic} shows that there is a two-dimensional representation $R: G_\Q \to \GL_2(\overline\Q_p)$ -- possibly reducible -- with distinct Hodge-Tate weights such that $(W^3)^{ss} = R^{\oplus e}$ for some $e\geq 1$, and   a character $\chi: G_\Q \to \overline\Q_p^\times$ such that  $(W^2)^{ss} =\chi^{\oplus d}$ and $(W^4)^{ss} =  (\rec(\omega_{\pi})\chi_{p,\cyc}^{-3}\chi^{-1})^{\oplus d} $ for some $d\geq 1$. Note that $\omega_\pi$ is even because the central character of $\pi_\infty$ is trivial, so $R$ is odd by (\ref{item:PD}) above.
Hence by \cite[Theorem 1.0.4]{pan2022fontaine}, $R$ is automorphic. 
Comparing with (\ref{item:eichler}), we see that there exists an automorphic character  $\omega$ of $\A^\times$ and an automorphic representation $\sigma$ of $\GL_2(\A)$ such that $\pi$ is nearly equivalent to a constituent of the representation $$\Ind_{P(\A)}^{\GSP_4(\A)} \sigma \boxtimes \omega,$$ where $P \subset \GSP_4$ is the Siegel parabolic subgroup with Levi factor $\GL_2\boxtimes \GL_1$. 
Hence $\pi$ is Eisenstein, which contradicts the assumption that $\pi$ is relevant and completes the proof. 
\end{proof}
We will require some information about Jacquet-Langlands transfers of relevant automorphic representations between the various $\spin(V_D)(\A)$. 
\begin{definition}\label{def:transferrable}
    For a tempered irreducible admissible representation $\pi_\l$ of $\GSP_4(\Q_\l)$, we say $\pi_\l$ is \emph{transferrable} if it does not belong to the types I, IIIa, VIa, VIb, VII, VIIIa, VIIIb, or IXa in the notation of \cite{roberts2007local}.
\end{definition}
In particular, if $\pi_\l$ is unramified, it is not transferrable.

\begin{thm}\label{thm:JL}
    Let $D\geq 1$ be squarefree. Then:
    \begin{enumerate}
    \item \label{part:JL_one} For each relevant, non-endoscopic automorphic representation $\Pi = \otimes'\Pi_v$ of $\spin(V_D)(\A)$, $\Pi^D_f \coloneqq \otimes'_{\l\nmid D } \Pi_\l$ can be completed to a relevant automorphic representation of $\GSP_4(\A)$. 
    \item \label{part:JL_two}Conversely, suppose $\pi=\otimes'\pi_v$ is a relevant, non-endoscopic automorphic representation of $\GSP_4(\A)$. Then $\pi^D_f\coloneqq \otimes'_{\l\nmid D}\pi_\l$ can be completed to 
 a relevant automorphic representation $\Pi$ of $\spin(V_D)(\A)$  if and only if $\pi_\l$ is transferrable for all primes $\l|D$. 
    
    \item\label{part:JL_four} Let $\pi=\otimes'\pi_\l$ be a relevant, non-endoscopic automorphic representation of $\GSP_4(\A)$, and assume $\pi_\l$ is transferrable for all $\l|D$. 
  Then any  relevant automorphic representation $\Pi$ of $\spin(V_D)(\A)$ with $\Pi^D_f \cong \pi^D_f$  has automorphic multiplicity one. 
  The set of all such $\Pi$
     is a Cartesian product of local $L$-packets for $v|D\infty$,  where the nonarchimedean $L$-packets are   determined only by the corresponding local factors of $\pi$.  The archimedean $L$-packet is the discrete series packet of weight $(3,3)$
     when $\sigma(D)$ is even and the trivial representation when $\sigma(D)$ is odd. 
        \end{enumerate}
\end{thm}
\begin{proof}
    When $\sigma(D)$ is odd, so that $\spin(V_D)(\R)$ is compact, this is \cite[Theorem 11.4]{rosner2024global}. 
    To prove these assertions when $\sigma(D)$ is even, we follow the sketch indicated in the discussion following \emph{loc. cit.}

The following fact will be used repeatedly:
\begin{equation}\label{eq:fact_coh_ds}
\begin{split}
    \text{If $\pi_\infty$ 
    is an irreducible admissible representation of $\GSP_4(\R)$ with $H^3(\mathfrak{gsp}_4, U(2); \pi_\infty) \neq 0$,}\\\text{ then $\pi_\infty$ lies in the discrete series $L$-packet of weight $(3, 3)$. 
  }
  \end{split}
\end{equation}
This fact follows from the calculations in \cite[p. 293]{taylor1993ladic}.

  To prove (\ref{part:JL_one}), suppose first that $\Pi= \otimes' \Pi_v$ is a relevant automorphic representation of $\spin(V_D)(\A)$.
   By \cite[Corollary 7.4]{rosner2024global} combined with Lemma \ref{lem:coh_for_JL}(\ref{item:coh_for_JL_3}), we conclude that $\Pi^D_f$ can be completed to a cohomological automorphic representation $\pi$ of $\GSP_4(\A)$. 
   Moreover $\pi_f$ necessarily contributes to cohomology in degree 3 by \cite[Proposition 8.2]{rosner2024global}, and in particular $\pi$ is relevant by (\ref{eq:fact_coh_ds}) combined with Lemma \ref{lem:coh_for_JL}(\ref{item:coh_for_JL_matsushima}), so this proves (\ref{part:JL_one}). 

   To prove (\ref{part:JL_two}) and (\ref{part:JL_four}), we need some preparation. Fix a relevant automorphic representation $\pi=\otimes'\pi_v$ of $\GSP_4(\A)$.
  For each prime $\l|D$ such that $\pi_\l$ is transferrable, there is a local $L$-packet of irreducible admissible representations of $\spin(V_D)(\Q_\l)$ determined by the character relations of Lemma 11.1 of \emph{op. cit.} Let $\A_D = \prod_{\l|D} \Q_\l$; by an irreducible admissible representation of $\spin(V_D)(\A_D)$ we mean a direct product of irreducible admissible representations of $\spin(V_D)(\Q_\l)$ for $\l|D$. 
  
  If $\pi_\l$ is transferrable for all $\l|D$,  let $S$ be the set of irreducible admissible representations of $\spin(V_D)(\A_D)$ obtained by taking the Cartesian product of the local $L$-packets. If $\pi_\l$ is not transferrable for some $\l|D$, then let $S$ be the empty set.

   The comparison of cohomological trace formulas from the proof of \cite[Theorem 11.4]{rosner2024global} shows the following identity: for all irreducible admissible representations $\Pi_D$ of $\spin(V_D)(\A_D)$, we have
      \begin{equation}\label{eq:trace_comparison}\frac{1}{2}\sum_{\pi_\infty} m (\pi^D\otimes \Pi_D\otimes \pi_\infty)\chi(\pi_\infty) = \begin{cases} 2, & \Pi_D \in S,\\ 0, & \Pi_D \not\in S.\end{cases}\end{equation}
      Here the notation is as follows:
      \begin{itemize}
          \item The sum over $\pi_\infty$ runs over cohomological representations of $\GSP_4(\R)$, and $\chi(\pi_\infty)$ 
is the negative Euler characteristic $\sum (-1)^{i+1} \dim H^i(\mathfrak{gsp}_4, U(2); \pi_\infty)$. 
\item $\pi^D$ is the representation $\otimes'_{\l\nmid D} \pi_\l$ of $\spin(V_D)(\A^D)$.
\item $m(\pi^D\otimes\Pi_D\otimes\pi_\infty)$ is the multiplicity in the discrete (equivalently cuspidal, since $\pi$ is relevant) automorphic spectrum of $\spin(V_D)(\A)$. 
      \end{itemize}

We can further manipulate the left-hand side of (\ref{eq:trace_comparison}). Let $\pi_\infty^W$ and $\pi_\infty^H$ be the generic and holomorphic members, respectively, of the discrete series $L$-packet of weight $(3,3)$. We then claim that
\begin{equation}\label{eq:mult_coh_WH}
 \frac{1}{2}   \sum_{\pi_\infty} m (\pi^D\otimes\Pi_D\otimes\pi_\infty) \chi(\pi_\infty) = m(\pi^D\otimes\Pi_D\otimes\pi_\infty^W)+ m(\pi^D\otimes\Pi_D\otimes\pi_\infty^H).
\end{equation}
To prove (\ref{eq:mult_coh_WH}), 
suppose first that the right-hand side is positive. Then any cohomological $\pi_\infty$ with $\Pi = \pi^D_f\otimes\Pi_D\otimes\pi_\infty$ automorphic can have Lie algebra cohomology only in degree 3 by Lemma \ref{lem:coh_for_JL}(\ref{item:coh_for_JL_matsushima}, \ref{item:coh_for_JL_3}). In particular, if $m(\pi^D\otimes\Pi_D\otimes\pi_\infty)\chi(\pi_\infty) \neq 0$ then $\pi_\infty$ is either $\pi_\infty^W$ or $\pi_\infty^H$ by (\ref{eq:fact_coh_ds}), so -- combined with the fact that $\chi(\pi_\infty^W) = \chi(\pi_\infty^H) = 2$ by \cite[p. 293]{taylor1993ladic} -- we have shown (\ref{eq:mult_coh_WH}) when the right-hand side is positive. On the other hand, suppose the right-hand side of (\ref{eq:mult_coh_WH}) vanishes; then each summand on the left-hand side of (\ref{eq:mult_coh_WH}) is non-positive by (\ref{eq:fact_coh_ds}) because Lie algebra cohomology vanishes for $\GSP_4(\R)$ outside degrees 0, 2, 3, 4, and 6. But the left-hand side of (\ref{eq:mult_coh_WH}) is also non-negative by (\ref{eq:trace_comparison}), so we conclude that it is zero, hence (\ref{eq:mult_coh_WH}) again holds.

In particular, by (\ref{eq:mult_coh_WH}) and (\ref{eq:trace_comparison}) together, we have
\begin{equation}\label{eq:new_trace_comparison}
    m(\pi^D\otimes\Pi_D\otimes\pi_\infty^W) + m(\pi^D\otimes\Pi_D\otimes\pi_\infty^H) = \begin{cases} 2, & \Pi_D \in S,\\ 0, & \Pi_D \not\in S.\end{cases}
\end{equation}

From (\ref{eq:new_trace_comparison}), it is clear that $\pi^D$ can be completed to a relevant automorphic representation of $\spin(V_D)(\A)$ if and only if $S$ is nonempty, i.e. if and only if $\pi_\l$ is transferrable for all $\l|D$; this shows (\ref{part:JL_two}).

We also see from (\ref{eq:new_trace_comparison}) that (\ref{part:JL_four}) is equivalent to the assertion that $m(\pi^D\otimes\Pi_D\otimes\pi_\infty^W) = m(\pi^D\otimes\Pi_D\otimes\pi_\infty^H) = 1$ for all $\Pi_D\in S$. Suppose for contradiction that  $m(\pi^D\otimes\Pi_D\otimes\pi_\infty^W)$ and $m(\pi^D\otimes\Pi_D\otimes\pi_\infty^H)$ are 0 and 2, in some order.
Let $K = \Prod K_\l\subset \spin(V_D)$ be a neat compact open subgroup such that $\Pi_f^K \neq 0$, and let $p > 3$ be a prime such that $K_p$ is hyperspecial. Fix as 
well an isomorphism $\iota: \overline\Q_p \isomorphism \C$.  Then by the discussion in \cite[p. 296]{taylor1993ladic}, $H^3_\et(\Sh_K(V_D)_{\overline\Q}, \overline \Q_p)[\Pi_f^K]$ is nonzero and has exactly two distinct Hodge-Tate weights as a representation of $G_{\Q_p}$. 

On the other hand, by (\ref{item:eichler}) from the proof of Lemma \ref{lem:coh_for_JL} combined with the irreducibility of $V_{\pi,\iota}$ (Lemma \ref{lem:reducible_endoscopic}), up to semisimplification $H^3_\et(\Sh_K(V_D)_{\overline\Q}, \overline \Q_p)[\Pi_f^K]$ is a sum of copies of $V_{\pi,\iota}$, which is a contradiction because $V_{\pi, \iota} $ has distinct Hodge-Tate weights by Theorem \ref{thm:rho_pi_LLC}(\ref{part:rho_pi_LLC_HT}). This completes the proof of  (\ref{part:JL_four}).
\end{proof}

\begin{rmk}
    One can similarly prove Theorem \ref{thm:JL} in all regular weights, by considering cohomology of Siegel threefolds with coefficients in more general local systems; we omit the details for concision.
\end{rmk}

\subsection{Relevant endoscopic automorphic representations}

\subsubsection{}
Relevant endoscopic automorphic representations of $\spin(V_D)(\A)$ can be constructed as follows. Fix unitary cuspidal automorphic representations $\pi_1$ and $\pi_2$ of $\GL_2(\A)$ with the same central character, and whose archimedean components are discrete series of weights 4 and 2, in some order. Let $D_1$ and $D_2$ be squarefree positive integers, such that $\pi_i$ admits a Jacquet-Langlands transfer $\pi_i^{D_i}$ to $B_{D_i}^\times(\A)$   for $i = 1, 2$ (notation as in (\ref{subsubsec:B_D_notation})). 

Define $D_1\ast D_2 \coloneqq D_1D_2/ \gcd(D_1,D_2)^2$, which is a squarefree positive integer. Then as in \cite[\S3.3]{chan2015local}, there is a global theta lift,
\begin{equation}
    \Theta(\pi_1^{D_1}\boxtimes \pi_2^{D_2}) = \Theta(\pi_2^{D_2}\boxtimes \pi_1^{D_1}),
\end{equation}
which is either zero or a cuspidal automorphic representation of $\spin(V_{D_1\ast D_2})(\A)$.
\begin{thm}\label{thm:endoscopic_packets}
With notation as above, relabel $\pi_1$ and $\pi_2$ if necessary so that $\pi_1$ has weight 2 and $\pi_2$ has weight 4. 
\begin{enumerate}
    \item \label{part:endoscopic_packets 1}The theta lift $\Theta(\pi_1^{D_1}\boxtimes\pi_2^{D_2})$ is nonzero if and only if either $\sigma(D_2)$ or $\sigma(D_1\ast D_2)$ is even. When nonzero, $\Theta(\pi_1^{D_1}\boxtimes \pi_2^{D_2})$ is always relevant and endoscopic associated to the pair $(\pi_1,\pi_2)$, and these representations are all distinct. 
    \item\label{part:endoscopic packets 2} Conversely, each relevant endoscopic automorphic representation of $\spin(V_D)(\A)$ arises in this way, and appears with automorphic multiplicity one in the discrete (equivalently cuspidal) spectrum. 
\end{enumerate}
\end{thm}
\begin{proof}
  This follows from the results of \cite[\S3]{chan2015local}, combined with the discussion of the local archimedean theta lift in \S2.4 of \emph{op. cit.}.
\end{proof}
\begin{cor}\label{cor:JL_general}
    For any relevant automorphic representation $\Pi$ of $\spin(V_D)(\A)$, there exists a relevant automorphic representation $\pi$ of $\GSP_4(\A)$ such that:
    \begin{enumerate}
        \item \label{cor:JL_general_one} For each prime $\l\nmid D$, $\Pi_\l \cong \pi_\l$.
        \item\label{cor:JL_general_two} For each prime $\l|D$, $\pi_\l$ is transferrable, and $\Pi_\l$ lies in the corresponding local packet of representations of $\spin(V_D)(\Q_\l)$ from Theorem \ref{thm:JL}(\ref{part:JL_four}). 
    \end{enumerate}
\end{cor}
\begin{proof}
If $\Pi$ is not endoscopic, this follows from Theorem \ref{thm:JL}. If $\Pi$ is endoscopic, (\ref{cor:JL_general_one}) is immediate  from Theorem \ref{thm:endoscopic_packets}. For (\ref{cor:JL_general_two}), 
it follows because the Langlands correspondences of \cite{gantakeda2011local, gantantono2011inner} are constructed to be compatible with theta lifting,  and the 
correspondence of local representations in \cite[Table 3]{rosner2024global} respects Langlands parameters by Lemma  11.1 of \emph{op. cit.}
    \end{proof}
\begin{rmk}
    \label{rmk:rho_Pi}
     By Corollary \ref{cor:JL_general}, we can associate to each relevant automorphic representation $\Pi$ of $\spin(V_D)(\A)$ a compatible system of Galois representations  $\rho_{\Pi, \iota}$ as in Theorem \ref{thm:rho_pi_LLC}.
\end{rmk}


\subsection{Local representations with paramodular fixed vectors}
\begin{notation}\label{notation:paramodular_subgroup}
    For all primes $q$ (whether or not $q|D$), the \emph{paramodular subgroup} 
of $\spin(V_D)(\Q_q)$ is a maximal compact subgroup described in \cite[p. 918]{sorensen2009level}. To avoid confusion, we denote this subgroup by $K_q^{\paramodular}$ when $q\nmid D$ and by $K_q^\ramified$ when $q|D$. 
\end{notation}

\begin{lemma}\label{lem:IIa_new}
    Let $\pi$ be a relevant automorphic representation of $\spin(V_D)(\A)$, and suppose $q\nmid D$ is a prime such that $\pi_q$  has a $K_q^{\paramodular}$-fixed vector. Then:
    \begin{enumerate}
        \item \label{part:IIa_type} $\pi_q$ is either spherical or of type IIa in the notation of 
\cite{roberts2007local}.
    \item \label{part:IIa_completion}
    $\pi$ is the unique completion of $\pi_f^q\otimes \pi_\infty$ to an automorphic representation of $\spin(V_D)(\A)$.
\item\label{part:IIa_Galois} If $\pi_q$ is of type IIa, then  for any $\iota: \overline \Q_p \isomorphism \C$ with $q \neq p$, the action of $I_{\Q_q}$ on $V_{\pi,\iota}$ is unipotent with monodromy of rank one. Moreover, the corresponding local packet of representations of $\spin(V_{Dq})(\Q_q)$ in Theorem \ref{thm:JL}(\ref{part:JL_four}) is a single representation with a unique $K_q^\ramified$-fixed vector. 
\end{enumerate}
\end{lemma}
\begin{proof}
By Corollary \ref{cor:JL_general}, $\pi_f^D= \otimes_{\l\nmid D} \pi_\l$ can be completed to a relevant automorphic representation of $\GSP_4(\A)$. In particular, $\pi_q$ is tempered (by Theorem \ref{thm:rho_pi_LLC}(\ref{part:rho_pi_LLC1})), so (\ref{part:IIa_type}) follows from \cite[Tables A.1, A.13]{roberts2007local}. 

If $\pi^q_f\otimes \pi_q'\otimes\pi_\infty$ is automorphic for some $\pi_q'$, then $\pi^{Dq}_f\otimes \pi_q'$ can be completed to an automorphic representation of $\GSP_4(\A)$ by the same reasoning as for $\pi_f^D$. Hence by Theorem \ref{thm:rho_pi_LLC}(\ref{part:rho_pi_LLC1}), $\pi_q'$ and $\pi_q$ belong to the same Gan-Takeda local $L$-packet. But the $L$-packets of type I and IIa are singletons, so this shows (\ref{part:IIa_completion}).

Finally, suppose
$\pi_q$ is the type IIa representation denoted $\chi \operatorname{St}_{\GL(2)} \rtimes \sigma$ in \emph{loc. cit.} Because $\pi_q$ has a $K_q^\paramodular$-fixed vector, 
 $\chi$ and $\sigma$ are unramified characters of $\Q_q^\times$. This implies the assertions on $\rho_{\pi,\iota}|_{I_{\Q_q}}$ in (\ref{part:IIa_Galois}) by Theorem \ref{thm:rho_pi_LLC}(\ref{part:rho_GL2_LLC_1}) and the explicit local Langlands paremeters found in \cite[Table A.7]{roberts2007local}. The final claim in (\ref{part:IIa_Galois})
 follows from \cite[Table 3]{rosner2024global} combined with \cite[Theorem B]{sorensen2009level}.
\end{proof}

\subsection{Generic maximal ideals and cohomology of $\spin_5$ Shimura varieties}

\subsubsection{}
For this subsection, fix a 
coefficient field $E\subset \overline\Q_p$ with $E$ a finite extension of $\Q_p$, and let $O\subset E$ be the ring of integers with uniformizer $\varpi$. Also fix an isomorphism $\iota: \overline\Q_p \isomorphism \C$, a
squarefree $D\geq 1$, a neat compact open subgroup $K = \prod_\l K_\l \subset \spin(V_D)(\A_f)$, and a set $S$ of places of $\Q$ containing all $\l$ such that $K_\l$ is not hyperspecial.

\begin{lemma}\label{lem:non-Eisenstein_coh}
    Suppose $\m\subset \T^S_O$ is non-Eisenstein and $\sigma(D)$ is even. Then for all $i$, the natural maps induce isomorphisms:
    \begin{equation}
    \begin{split}       
        H^i_c(\Sh_K(V_D), O)_\m&\isomorphism H^i(\Sh_K(V_D), O)_\m, \\
        H^i_c(\Sh_K(V_D), \overline\F_p) _\m&\isomorphism H^i(\Sh_K(V_D), \overline\F_p)_\m.
                \end{split}
    \end{equation}
\end{lemma}
\begin{proof}
We show the second isomorphism; the first follows formally. Let $P = MN \subset \spin(V_D)$ be a parabolic subgroup. Then we may fix an identification $\spin(V_D)(\A^S) \simeq \GSP_4(\A^S)$ such that $P(\A^S) = P_1(\A^S)$ for a parabolic subgroup $P_1 = M_1N_1\subset \GSP_4$.

The Levi factor $M$  is abstractly isomorphic to either $B_D^\times\times \GL_1$ or $\GL_1^3$ (the latter occurring only if $D = 1$). 
To any compact open subgroup $K_M \subset M(\A_f)$ we can associate a locally symmetric space $S_{K_M}(M) = 
M(\Q) \backslash M(\A) / K_M \cdot K_\infty
$, where $K_\infty\subset M(\R)$ is the product of the center and a maximal compact subgroup. In particular, each connected component of $S_{K_M}(M)$ is either a Shimura curve or an isolated point.

Using the Borel-Serre compactification of $\Sh_K(V_D)$ and the argument of \cite[\S4]{newton2016torsion}, it
suffices to show the following: for any parabolic subgroup $P = MN\subset \spin(V_D)$ as above and any compact open subgroup $K_M = \prod_\l K_{M, \l} \subset M(\A_f)$ with $K_{M, \l}$ hyperspecial for $\l\not\in S$, the support of the $
 \T^S_{\GSP_4, \overline\F_p}$-module
$(S_{M}^{\GSP_4})^\ast H^i(S_{K_M}(M), \overline \F_p)$ is Eisenstein for all $i$. But this follows from Proposition \ref{prop:m_pullback_eisenstein} because, as $M = B_D^\times \times \GL_1$ or $\GL_1^3$, every maximal ideal of $ \T^S_{M, \overline\F_p}$ in the support of $H^i(S_{K_M}(M), \overline\F_p)$ is clearly of Galois type. 
\end{proof}

\begin{definition}\label{def:generic}
\leavevmode
\begin{enumerate}
     \item     A maximal ideal $\mathfrak m \subset \T_{\l, O}$ is called \emph{generic} if $p\nmid 2\l(\l^4 - 1)$, and the Satake parameter $\set{\alpha,\beta,\nu/\alpha,\nu/\beta}$  of $\m$ is multiplicity-free with no two elements having ratio $\l$.
 \item     A maximal ideal $\mathfrak m \subset {\T}^S_{O}$ is called \emph{generic} if there exist infinitely many $\l\not\in S$ such that the induced maximal ideal of $\T_{\l, O}$ is generic. For any quotient $\T$ of ${\T}^S_O$, a maximal ideal $\m\subset \T$ is called generic if its pullback to ${\T}^S_O$ is so.
\item A Galois representation $\overline \rho: G_\Q \to \GSP_4(\overline \F_p)$, unramified outside a finite set $S$, is called \emph{generic} if there exists a prime $\l\not\in S$ with $p\nmid 2\l(\l^4 - 1)$, such that $\overline\rho(\Frob_\l)$ has distinct eigenvalues, no two having ratio $\l$.
\end{enumerate}
\end{definition}
\begin{rmk}
    If $\m\subset {\T}^S_O$ has an associated Galois representation $\overline\rho_\m$, clearly $\m$ is generic if and only if $\overline\rho_\m$ is so.
\end{rmk}

\begin{thm}\label{thm:generic}
    Suppose $\m \subset \T^S_O$ is a generic maximal ideal and $\sigma(D)$ is even. 
    Then:
    \begin{enumerate}
        \item\label{part:thm_generic_one}  For all $i < 3$, we have $$H^i(\Sh_K(V_D), \overline\F_p)_\m = H^{6-i} _c(\Sh_K(V_D), \overline\F_\p)_\m = 0.$$ 
        \item \label{part:thm_generic_two}If $\m$ is also non-Eisenstein, then 
        $$H^i_c(\Sh_K(V_D), O)_\m = H^i(\Sh_K(V_D), O)_\m$$ is $\varpi$-torsion-free for all $i$, and vanishes unless $i = 3$.  
    \end{enumerate}
\end{thm}
\begin{proof}
    Part (\ref{part:thm_generic_one}) is immediate from \cite[Theorem 1.16]{hamann2023torsion} and our definition of genericity. Using Lemma \ref{lem:non-Eisenstein_coh}, (\ref{part:thm_generic_two}) is a standard consequence of (\ref{part:thm_generic_one}).
\end{proof}

\begin{lemma}\label{lem:yucky_eisenstein_lemma}
Suppose $\m\subset \T_{K, V_D, O}^S$ is  non-Eisenstein and generic, and suppose
 $\pi$ is an automorphic representation of $\spin(V_D)(\A)$ such that  $\pi_f^K\neq 0$, and the action of $ \T^S_O$ on $\iota^{-1}\pi_f^K$ factors through $\T_{K,V_D, O,\m}^S$. Then $\pi$ is not Eisenstein (in the sense of Definition \ref{def:eisenstein}). 
\end{lemma}
With a bit more care, the genericity assumption can be dropped; we leave the details to the reader.
\begin{proof}
Suppose first that $D = 1$, so $\pi$ is an automorphic representation of $\GSP_4(\A)$, and let $\m_\pi \subset {\T}^S_{\GSP_4, \overline\Q_p}$ be the maximal ideal determined by the Hecke action on $\iota^{-1}\pi^S$. Then by hypothesis, $\m_\pi$ is in the  support of the Hecke module $H^3(\Sh_K(V_1), \overline \Q_p)_\m.$
By Lemma \ref{lem:non-Eisenstein_coh} combined with the diagram in \cite[p. 293]{taylor1993ladic}, we conclude that $\pi^S$ can be completed to an automorphic representation $\pi'$ of $\GSP_4(\A)$ that appears in the discrete spectrum. Suppose for contradiction that $\pi$ is Eisenstein; then $\pi'$ is either CAP or a residual representation. In either case, it follows that the Satake parameter of $\pi_\l$ contains a pair of the form $\set{\alpha, \alpha\l}$ for all  $\l\not\in S$: when $\pi'$ is CAP, this uses \cite[Theorems 2.5, 2.6]{Piatetski1983saito} and \cite[Theorem C]{Soudry1988CAP}, and when $\pi'$ is residual it uses well-known results on the irreducibility of principal series representations for $\GSP_4$. But these Satake parameters are inconsistent with the genericity of $\m$, so $\pi$ cannot be Eisenstein, as desired.

We now handle the case of general $D$; the argument is a mild refinement of the trace formula method used in \cite{rosner2024global}. We abbreviate $G_D \coloneqq \spin(V_D)$ and fix a minimal parabolic subgroup $P_0 \subset G_D$ (which will be all of $G_D$ if  $\sigma(D)$ is odd).
Let $f_D\in C_c^\infty(G_D(\A_f), \C)$ be a  test function with regular support, and define
\begin{equation}\label{eq:def_T_f_D}
    T(f_D) = \sum_{P_0 \subset P = MN \subset G_D} (-1)^{\operatorname{rank}(M) - \operatorname{rank}(G_D)}\sum_{w \in W^P} (-1)^{\l(w)} \tr \left(\overline f_D^P\cdot \chi_P^G, \varinjlim_{K_M} H^\ast(S_{K_M} (M), V_{w\cdot \rho_{G_D} - \rho_M})\right). 
\end{equation}
Here $W^P$ is the set of minimal-length coset representatives for the Weyl group of $G_D$ modulo that of $M$, $\tr$ is the supertrace,  $\overline f^P$ and $\chi_P^{G_D}$ are defined as in \cite[\S2.6]{weissauer2009endoscopy},  $S_{K_M}(M)$ is the symmetric space from the proof of Lemma \ref{lem:non-Eisenstein_coh}, $\rho_{G_D}$ and $\rho_M$ are the half-sums of positive roots, and $V_{w\cdot\rho_{G_D} - \rho_M}$ is the complex local system on $S_{K_M}(M)$ of weight $w\cdot \rho_{G_D} - \rho_M$. Although $w\cdot \rho_{G_D} -\rho_M$ might not be integral, we interpret $V_{w\cdot\rho_{G_D} - \rho_M}$ as the twist of $V_{w\cdot \rho_{G_D} - \rho_{G_D}}$ by the real character $\delta_P^{1/2}$ of $M(\R)$, i.e.
\begin{equation}\label{eq:coh_rho_twist}H^\ast(S_{K_M} (M), V_{w\cdot \rho_{G_D} - \rho_M}) = \delta_{P}^{-1/2} \otimes H^\ast(S_{K_M} (M), V_{w\cdot \rho_{G_D} - \rho_{G_D}})\end{equation} as Hecke modules.

For test functions $f_1\in C_c^\infty(G_1(\A_f), \C)$ with regular support, we define $T(f_1)$ analogously to (\ref{eq:def_T_f_D}).
If $f_D$ and $f_1\in C_c^\infty(G_1(\A_f), \C)$ are matching functions in the sense of \cite[\S5]{rosner2024global}, then it follows from combining \cite[Lemma 2.10]{weissauer2009endoscopy} with \cite[Theorem 5.3]{rosner2024global} that
\begin{equation}\label{eq:trace_comparison_2}
    T(f_D) = T(f_1) c_D
\end{equation}
for a nonzero constant $c_D$ depending only on $D$. 

Shrinking $K_S$ if necessary, we can fix a $K_S$-biinvariant test function $f_{D,S} \in C_c^\infty(G_D(\A_S))$ with regular support such that 
\begin{equation}\label{eq:S_test_trace}
    \tr\left(f_{D, S}, H^\ast(\Sh_K(V_D), \C)[\pi^S])\right) \neq 0;
\end{equation}
when $\sigma(D)$ is even,
we use Theorem \ref{thm:generic}(\ref{part:thm_generic_two}) to ensure that contributions from $\pi$ in different degrees do not cancel.
Let $f_{1,S} \in C_c^\infty(G_1(\A_S))$ be a matching function, and let $K_{1,S} \subset G_1(\A_S)$ be a compact open subgroup such that $f_{1,S}$ is $K_{1,S}$-biinvariant.  We also fix an isomorphism $G_D(\A_f^S) \simeq G_1(\A_f^S)$, 
and let $K_1 = K^SK_{1,S} \subset G_1(\A_f)$, which is a compact open subgroup. 
Let $$\Pi(D) = \bigoplus_{P_0 \subset P = MN \subset G_D} \bigoplus_{w \in W^P} \Ind_{P(\A_f)}^{G_D(\A_f)}\varinjlim_{K_M}  H^\ast(S_{K_M} (M), V_{w\cdot \rho_{G_D} - \rho_M}),$$and 
let ${\mathcal P}_D$ be the  set  of irreducible $G_D(\A_f)$-constituents occurring in $\Pi(D)$ with a $K$-fixed vector. Likewise, we define $\Pi(1)$ and ${\mathcal P}_1$, where now we consider constituents with $K_1$-fixed vectors, and set ${\mathcal P} = {\mathcal P}_D\sqcup  {\mathcal P}_1$.

Then $\mathcal P$ contains finitely many near equivalence classes, cf. 
 \cite[p. 45]{weissauer2009endoscopy}.
We can therefore fix a finite set $T$ of primes of $\Q$, disjoint from $S$, so that two representations $\sigma,\tau\in 
{\mathcal P}$ are nearly equivalent if and only if  $\sigma_T\cong \tau_T$. 
Now fix a $K_T= \prod_{\l\in T}K_\l$-biinvariant test function $f_T\in C_c^\infty (G_D(\A_T), \C)= C_c^\infty(G_1(\A_T), \C)$, such that for all $\tau\in{ \mathcal P}$,  $\tr (f_T | \tau_T) = 0$ unless $\tau_T\cong \pi_T$, and $\tr(f_T |\pi_T) = 1$.

Also fix an auxiliary prime $v_0 \not\in S\cup T$, and for a large constant $C >0$ to be chosen later, let $f_{v_0,C}\in \T_{\GSP_4, v_0, \C}$ be a test function satisfying the conclusion of \cite[Lemma 3.9]{rosner2024global} for the representation $\pi_{v_0}$; in particular, $\tr(f_{v_0, C}| \pi_{v_0}) = 1.$ 

We consider the global test function
$$f_D = f_{D, S} f_Tf_{v_0,C}f^{S\cup T\cup \set{v_0}}\in C_c^\infty(G_D(\A_f)),$$
 where $f^{S\cup T\cup \set{v_0}}$ is the indicator function of $K^{S\cup T \cup \set{v_0}}$. The matching function is 
 $$f_1 = f_{1,S}f_Tf_{v_0, C}f^{S\cup T\cup \set{v_0}}\in C_c^\infty(G_1(\A_f)).$$

 \begin{claim}
     For all $P_0 \subset P = MN \subsetneq G_{1}$ and all $w\in W^P$, we have
     $$\tr\left(\overline f_1^P \cdot \chi_P^G,\; \delta_{P}^{-1/2}\varinjlim_{K_M} H^\ast(S_{K_M}(M), V_{w\cdot \rho_{G_1} - \rho_{G_1}})\right)=0,$$
     and likewise for all $P_0 \subset P = MN \subsetneq G_D$.  
 \end{claim}
 \begin{proof}[Proof of claim]
     To ease notation, we prove the claim for $G_1$; the proof for $G_D$ is identical. 
 Let $\tau$ be an irreducible constituent of $\varinjlim H^\ast(S_{K_M}(M), V_{w\cdot \rho_{G_{1}} - \rho_{G_{1}}})$.
 It suffices to show that 
 $$\tr\left( \overline f^P_{1}\cdot \chi_P^{G_{1}}, \tau\otimes \delta_{P}^{-1/2}\right) = 0.$$ 
 
Now by the argument of \cite[Proposition 3.10]{rosner2024global}, if $C$ is chosen sufficiently large, then $$\tr\left(\overline f^P_1 \cdot \chi_P^G, \tau\otimes \delta_{P}^{-1/2}\right) = 
\tr\left(\overline{f_{1,S} f_T f^{S\cup T\cup \set{v_0}}}^P \cdot f^P_{v_0}, \tau\otimes \delta_{P}^{-1/2}\right) 
$$
for an auxiliary spherical test function $f^P_{v_0}$ on $M(\Q_{v_0})$. In particular, because $\overline f^P$ is invariant under $K_1 \cap M(\A_f)$, 
we may assume without loss of generality that $\tau$ contains a fixed vector for $K_1 \cap M(\A_f)$. Moreover it suffices to show
$$\tr\left(\overline f^P_T, \tau_T \otimes \delta_P^{-1/2}\right) = 0.$$
By \cite[Proposition 3.5(2)]{rosner2024global}, the latter trace coincides with 
\begin{equation}\label{eq:T_tr_ind}
    \tr\left(f_T, \natural-\Ind_{M(\A_T)}^{G_1(\A_T)} \tau\right).
\end{equation}
We will show (\ref{eq:T_tr_ind}) vanishes. Indeed, since each constituent of $\natural-\Ind_{M(\A_f)}^{G_1(\A_f)}\tau$ lies in $ \mathcal P_1$, if (\ref{eq:T_tr_ind}) is nonzero then $\pi$ is nearly equivalent to a constituent of $\natural-\Ind_{M(\A_f)}^{G_1(\A_f)} \tau$ by the choice of $f_T$. Now, the maximal ideal $\m_\tau\subset {\T}^S_{{M}, \overline\Q_p}$ defined by the Hecke action on $\iota^{-1}\tau$ is of Galois type because $\tau$ is cohomological and 
 $M = \GL_1^3$ or $\GL_2\times \GL_1$. Hence by Propositions \ref{prop:Hecke_diagram_induction} and \ref{prop:m_pullback_eisenstein}, the maximal ideal $\m_\pi\subset  {\T}^S_{\GSP_4,\overline\Q_p}$ defined by $\iota^{-1}\pi$ is Eisenstein, which one can easily check contradicts the hypothesis that $\m$ is non-Eisenstein.
\end{proof}

In particular, the claim combined with (\ref{eq:S_test_trace}) 
 and the choice of $f_T$ shows that $T(f_D) \neq 0$, so $T(f_1) \neq 0$ by (\ref{eq:trace_comparison_2}). Using the claim  again,
 we see that
$$\tr\left(f_1,  H^\ast(\Sh_{K_1}(V_1), \C)\right) \neq 0.$$
By Franke's theorem \cite{franke1998harmonic} and our choice of $f_T$, we conclude there exists an automorphic representation $\pi'$ of $\GSP_4(\A)$ which is nearly equivalent to $\pi$ and unramified outside $S$, such that the Hecke eigensystem of $\pi'^S$ appears in $H^\ast(\Sh_{K_1}(V_1), \C)$. 
Expand $S$ to a larger set $S'$ such that $\pi^{S'} \cong \pi'^{S'}$. Then the maximal ideal $\m^{S'} \subset {\T}^{S'}_{\GSP_4, O}$ formed by restricting $\m$ descends to $\T^{S'}_{K_1, V_1,O}$. By the case $D= 1$ of the lemma for $\m^{S'}$ and $\pi'$, $\pi'$ is not Eisenstein, hence $\pi$ is not either. 
\end{proof}
\begin{cor}\label{cor:coh_relevant}
    Suppose $\sigma(D)$ is even, and $\m\subset \T^S_{K,V_D,O}$ is a generic, non-Eisenstein maximal ideal. Also  fix an isomorphism $\iota: \overline \Q_p \isomorphism \C$. Then $$H^3_\et (\Sh_K(V_D)_{\overline\Q}, \overline\Q_p)_\m  = \bigoplus_{\pi_f} \iota^{-1}\pi_f^K\otimes\rho_{\pi_f},$$ where:
    \begin{itemize}
        \item $\pi_f$ runs over the finite parts of relevant automorphic representations $\pi$ of $\spin(V_D)(\A_f)$ such that the $\T^S_{O}$-action on $\pi_f^K$ factors through $\T^S_{K,V_D,O,\m}$. 
        \item If $\pi$ is not endoscopic, then $\rho_{\pi_f} = \rho_{\pi, \iota}(-2)$, cf. Remark \ref{rmk:rho_Pi}.
        \item If $\pi$ is endoscopic associated to a pair $\pi_1, \pi_2$ of automorphic representations of $\GL_2(\A)$ with discrete series archimedean components of weights 2 and 4, respectively, we can write $\pi = \Theta(\pi_1^{D_1}\boxtimes \pi_2^{D_2})$ by Theorem \ref{thm:endoscopic_packets}. Then
      $$  \rho_{\pi_f} = \begin{cases}
            \rho_{\pi_1,\iota}(-2), & \sigma(D_1) \text{ even},\\
            \rho_{\pi_2,\iota}(-2), & \sigma(D_1) \text{ odd}.
        \end{cases}$$
    \end{itemize}
\end{cor}
\begin{proof}
    It follows from Lemmas \ref{lem:yucky_eisenstein_lemma} and \ref{lem:coh_for_JL}(\ref{item:coh_for_JL_matsushima})
    that  $$H^3_\et(\Sh_K(V_D)_{\overline\Q}, \overline\Q_p)_\m \cong  \bigoplus_{\pi_f} \iota^{-1}\pi_f^K\otimes\rho_{\pi_f}$$ as Hecke modules,
    where $\pi_f$ runs over the finite parts of non-Eisenstein automorphic representations of $\spin(V_D)(\A_f)$ with Hecke action factoring through $\T^S_{K,V_D,O,\m}$, and  $\rho_{\pi_f}$ is some Galois representation with $$\dim\rho_{\pi_f} = \sum_{\pi_\infty'} m(\pi_f\otimes\pi_\infty') \dim H^3(\mathfrak{gsp}_4, U(2); \pi_\infty').$$
    In particular, by (\ref{eq:fact_coh_ds}), the only $\pi_f$ with $\rho_{\pi_f} \neq 0$ are the finite parts of \emph{relevant} automorphic representations $\pi$.

    We first consider the non-endoscopic case. 
  As in \cite[p. 296]{taylor1993ladic}, we see from Theorem \ref{thm:JL}(\ref{part:JL_four}) that $\rho_{\pi_f}$ is four-dimensional with   Hodge-Tate weights  $\set{0,-1,-2,-3}$ if $\pi$ is not endoscopic. 
Since for all but finitely many $\l\not\in S$, $\Frob_\l$ satisfies the Eichler-Shimura relation (\ref{eq:eichler_shimura}) on $\rho_{\pi_f}$, we conclude that $\rho_{\pi_f} = \rho_{\pi,\iota} (-2)$. 
It remains to consider the endoscopic case, when $\pi = \Theta(\pi_1^{D_1} \boxtimes \pi_2^{D_2}).$
Then $\pi$ is the unique completion of $\pi_f$ to an automorphic representation of 
$\spin(V_D)(\A)$ by Theorem \ref{thm:endoscopic_packets}. Moreover, by the local archimedean theta lift described in \cite[Proposition 4.3.1]{harris1992arithmetic}, $\pi_\infty$ is  the generic or holomorphic member of the discrete series $L$-packet in the case that $\sigma(D_1)$ is even or odd, respectively. Hence we conclude  as above that $\rho_{\pi_f}$ is two-dimensional, with Hodge-Tate weights $\set{-1,-2}$ or $\set{0,-3}$ when $\sigma(D_1)$ is even or odd, respectively. Since we still have the Eichler-Shimura relation (\ref{eq:eichler_shimura}), it follows that $\rho_{\pi_f} $ is either $\rho_{\pi_1, \iota}(-2)$ or $\rho_{\pi_2, \iota}(-2)$ depending on the Hodge-Tate weights, and the corollary follows. 
\end{proof}

\begin{cor}\label{cor:T_embedding}
    Let $\m\subset \T^S_{K,V_D, O}$ be a generic and non-Eisenstein maximal ideal, and let $\mathcal T$ be the set of relevant automorphic representations $\pi$ of $\spin(V_D)(\A)$ such that $\pi_f^K\neq 0$ and the Hecke action on $\iota^{-1} \pi_f^K$ factors through $\T^S_{K,V_D,O, \m}$. Then we 
 have a natural embedding of $\T^S_{K, V_D,O,\m}$-algebras
    $$\T^S_{K, V_D,O, \m} \hookrightarrow \bigoplus_{\pi\in \mathcal T} \overline{\Q}_p(\pi),$$
    where $\overline\Q_p(\pi)$ is $\overline\Q_p$ with Hecke action through the eigenvalues on $\iota^{-1}\pi_f^K$. 
\end{cor}
\begin{proof}
If $\sigma(D)$ is odd, this is immediate from  Lemma \ref{lem:yucky_eisenstein_lemma}. If $\sigma(D)$ is even, it follows from Theorem \ref{thm:generic}(\ref{part:thm_generic_two}) combined with Corollary \ref{cor:coh_relevant}.
\end{proof}


\section{Special cycles and theta lifts}\label{sec:sp_cycles_theta}

\subsection{Special cycles}
\subsubsection{}\label{subsubsec:basic_sp_cyc}

In this subsection, we explain the construction of special cycles $Z(T,\phi)$, due to Kudla  in the indefinite case \cite{kudla1997orthogonal}.
\begin{construction}\label{constr:Z_T_phi}
    Let $V$ be a quadratic space over $\Q$ of signature $(m, 2)$ or $(m,0)$. 
\begin{enumerate}
    \item If $V_0 \subset V$ is a positive definite subspace, then for any $g\in \spin(V_0)(\A_f)\backslash\spin(V)(\A_f)/K$,  
we obtain a canonical finite morphism
\begin{equation}\label{eq:finite_morphism_g}
    \Sh_{K_{0,g}}(V_0^\perp) \xrightarrow{\cdot g} \Sh_K(V),
\end{equation}
with $K_{0,g} \coloneqq gKg^{-1} \cap \spin(V_0^\perp)(\A_f).$
\begin{enumerate}[label=(\roman*)]
    \item If $V$ has signature $(m, 2)$ and $\dim V_0 = n$, then we write
\begin{equation*}
    Z(g, V_0, V)_K \in \CH^n(\Sh_K(V))
\end{equation*}
for the pushforward of the fundamental class on $\Sh_{K_{0,g}}(V_0^\perp)$ under (\ref{eq:finite_morphism_g}).
\item If $V$ has signature $(m, 0)$, then we write
\begin{equation*}
    Z(g, V_0, V)_K \in \Z\left[\Sh_K(V)\right]
\end{equation*}
for the pushforward of the constant function 1 on $\Sh_{K_{0,g}}(V_0^\perp)$ under (\ref{eq:finite_morphism_g}).

\end{enumerate}
\item  \label{constr:Z_T_phi_omega}For any $T\in \Sym_n(\Q)_{\geq 0}$, let
\begin{equation*}
    \Omega_{T, V}= \set{(x_1, \ldots, x_n)\in V^{ n} \,:\, x_i\cdot x_j = T_{ij} \;\forall\, 1\leq i,j\leq n},
\end{equation*}
viewed as an affine algebraic variety over $\Q$. 
\item Now suppose given a neat compact open subgroup $K \subset \spin(V)(\A_f)$, along with a test function $$\phi \in \mathcal S(V^{n} \otimes \A_f, R)^K$$ for some $n\leq m$ and some ring $R$. (The action of $K$ is the natural one, factoring through the map to $\SO(V)(\A_f)$.) For any   $T\in \Sym_n(\Q)_{\geq 0}$, if $\Omega_{T,V}(\Q) = \emptyset$, then the special cycle $Z(T, \phi)_K$, in $\CH^n(\Sh_K(V),R)\coloneqq \CH^n(\Sh_K(V))\otimes_\Z R$ or $R\left[\Sh_K(V)\right]$, is defined to vanish. 
Otherwise:
\begin{enumerate}[label=(\roman*)]
    \item If $V$ is positive definite, then fix a base point $ (x_1,\ldots, x_n) \in \Omega_{T, V}(\Q)$, and let $$V_0 = \Span_\Q\set{x_1, \ldots, x_n}\subset V.$$ Then we define
\begin{equation}\label{eq:Z_T_phi_K}
    Z(T, \phi)_K = \sum_{g\in \spin(V_0^\perp)(\A_f) \backslash \spin(V)(\A_f)/ K} \phi(g^{-1} x_1, \ldots, g^{-1} x_n) Z(g, V_0, V)_K \in R[\Sh_K(V)].
\end{equation}
\item If $V$ has signature $(m, 2)$ and $T$ is positive definite, we define $$Z(T, \phi)_K \in \CH^n(\Sh_K(V), R)$$ by the same formula (\ref{eq:Z_T_phi_K}). If $T$ is not positive definite (and still $\Omega_{T, V}(\Q) \neq 0)$, we define $Z(T, \phi)_K \in \CH^n(\Sh_K(V), R)$ by the recipe in \cite[p. 61]{kudla1997orthogonal}; the details will not be needed.
\end{enumerate}
\end{enumerate}
\end{construction}
\begin{rmk}
   By transitivity of the $\spin(V)(\Q)$-action on $\Omega_{T, V}(\Q)$, $Z(T, \phi)_K$ is independent of the choice of base point for $\Omega_{T, V}(\Q)$. 
\end{rmk}

\begin{prop}\label{prop:pullback_compatible_Z}
    For neat compact open subgroups $K' \subset K \subset \spin(V)(\A_f)$, if $\pr_{K, K'}:\Sh_{K'}(V) \to \Sh_{K}(V)$ is the natural map, then $\pr_{K,K'} ^\ast Z(T, \phi)_K = Z(T, \phi)_{K'}$.
\end{prop}
\begin{proof}
    This is \cite[Proposition 5.10]{kudla1997orthogonal} in the indefinite case; the definite case is proved in the same way.
\end{proof}
\begin{notation}
    \label{notation:SC}
    For any compact open subgroup $K \subset \spin(V)(\A_f)$, and any ring $R$, we define $\SC_K^n(V, R) $ to be the $R$-span of the special cycles $Z(T,\phi)_K$ for $T\in \Sym_n(\Q)_{\geq 0}$ and $\phi \in \mathcal S(V^n\otimes \A_f, R)$. When $R = \Z$ it may be dropped from the notation.
\end{notation}

\begin{rmk}\label{rmk:SC_definitions}
    Note that $\SC_K^n(V, \Z)$ contains all of the special cycles $z = Z(g, V_0, V)_K$ from (\ref{subsubsec:basic_sp_cyc}) with $\dim V_0 = n$. Indeed, choose a basis $\set{e_1, \ldots, e_n}$ for $V_0$, and set $T_{ij} \coloneqq e_i \cdot e_j$.  Then one can choose $\phi\in \mathcal S(V\otimes \A_f, \Z)^K$ such that $\phi|_{\Omega_{T,V}(\A_f)}$ is the indicator function of $K\cdot g^{-1}(e_1,\ldots, e_n)$, and  it follows that $Z(T, \phi)_K = Z(g, V_0, V)_K$. 
\end{rmk}

\subsubsection{}
We will later need the following proposition to understand the double coset space appearing in (\ref{eq:Z_T_phi_K}).
\begin{prop}\label{prop:zanarella}
    Suppose $K_\l\subset \SO(V)(\Q_\l)$ is the stabilizer of a self-dual lattice $L \subset V\otimes \Q_\l$, and $V = V_0 \oplus V_1$ be an orthogonal decomposition of $V$. 
    Then the natural map
    \begin{align*}
\SO(V_0)(\Q_\l)\backslash \SO(V)(\Q_\l) / K_\l &\to \set{\textrm{lattices }L_1 \subset V_1\otimes \Q_\l} \\
      g & \mapsto g\cdot L \cap V_1
    \end{align*}
    is injective. If $\dim(V_1) <\dim (V_0)$, then its image consists of all $L_1$ on which the pairing is $\Z_\l$-valued. 
\end{prop}
\begin{proof}
    This follows from (the proof of) \cite[Propositions 3.1.5, 3.1.6]{corato2023spherical}.
\end{proof}
\subsection{ Symplectic and metaplectic groups}
In this section, we set up basic notions for symplectic and metaplectic groups, mostly following the exposition of \cite{gan2012metaplectic}.
\begin{notation}\label{notation:symplectic}
    \begin{enumerate}
        \item \label{notation:symplectic_W}For an integer $n\geq 1 $, we define the standard symplectic lattice $W_{2 n} $
with basis $e_1,\ldots, e_n, e_1 ^\ast,\ldots, e_n ^\ast$
and pairing determined by
\begin{equation}
\langle e_i, e_j\rangle = 0,\;\;\langle e_i ^\ast, e_j ^\ast\rangle = 0,\;\;\langle e_i, e_j ^\ast\rangle =\delta_{ij}.
\end{equation}
The symplectic group $\SP_{2n}$ as defined in (\ref{subsubsec:coordinates_gsp2n}) is the isometry group of 
 $W_{2 n} $.
\item    The Siegel parabolic subgroup $P = MN\subset \SP_{2n}$ is the stabilizer of $\Span_\Z\set{e_1, \ldots, e_n} \subset W_{2n}$. 
We identify $N$ with $\Sym_n$, the space of $n\times n$ symmetric matrices, and $M$ with $\GL_n$ via the embedding $$g\mapsto\begin{pmatrix} g & 0\\0 & g ^ {- t}\end{pmatrix}\in\SP_{2 n},\;\; g\in\GL_n. $$
\item If $k $
is a local field of characteristic zero, the metaplectic group $\MP_{2 n} (k) $
is defined as the unique non-split central extension of $\SP_{2 n} (k) $
by $\mu_2 $:
\begin{equation*}
0\rightarrow\mu_2\rightarrow\MP_{2n} (k)\xrightarrow {g\mapsto\overline g}\SP_{2 n} (k)\rightarrow 0.
\end{equation*}

The double cover $\MP_{2n}(k) \to \SP_{2n}(k)$ splits uniquely over $N(k)$, with $P = MN$ the Siegel parabolic as above. The preimage $\widetilde P(k)$ of $P(k)$ therefore has a Levi decomposition
$$\widetilde P(k) = \widetilde M(k)N(k)$$ with $\widetilde M(k)$ a nonsplit double cover of $M(k)$. 
\item If $k $
is non-archimedean with ring of integers $\O $
and the residue characteristic of $\O $
is odd, let $\MP_{2n}(\O)\subset \MP_{2n}(k)$ be the unique lifting of $\SP_{2n}(\O)$ \cite[\S6]{gansavin2012metaplecticIIHecke}.
\item\label{notation:symplectic_part_cartan} Let $U(n) \hookrightarrow \SP_{2n}(\R)$ be the embedding defined by 
$$ A+ i B \mapsto \begin{pmatrix}
     A & B \\ -B & A
\end{pmatrix}.$$
We fix the Cartan decomposition
\begin{equation*}\label{cartan spr}
\mathfrak {s p}_{2 n,\C} =\mathfrak {u} (n)_\C\oplus\mathfrak {p} ^ +\oplus\mathfrak {p} ^ -,
\end{equation*}
such that $\mathfrak p^+$ is isomorphic to the symmetric square of the defining representation of $U(n)$. 
\item \label{notation:symplectic_part_weight}Let $\widetilde U(n)\subset \MP_{2n}(\R)$ be the preimage of $U(n)$. If $j^{1/2}(g, z)$ is the half-integral weight automorphy factor of \cite[p. 25]{shimura1995metaplectic}, then we let
$\det^{1/2}: \widetilde U(n) \to \C^\times$ be the restriction of $j^{1/2}(g, i)$, which is a square root of the determinant character. We set $\det^k\coloneqq (\det^{1/2})^{2k}$ for all $k \in \frac{1}{2} \Z$.
\item Globally, let $$\MP_{2 n} (\A) =\prod_v'\MP_{2 n} (\Q_v) $$
be the restricted  product with respect to the subgroups $\MP_{2 n} (\Z_v)\subset\MP_{2 n} (\Q_v) $
for $v\neq 2,\infty $. The inclusion $\SP_{2 n} (\Q)\hookrightarrow\SP_{2 n} (\A) $
lifts naturally to $$\SP_{2 n} (\Q)\hookrightarrow\MP_{2 n} (\A),$$ by which we will view $\SP_{2n}(\Q)$ as a subgroup of $\MP_{2n}(\A)$. 

    \end{enumerate}
\end{notation}

\begin{definition}
For $k\in\frac {1} {2}\Z $, 
 define the space 
  $M ^ {2 n}_k $  of adelic Siegel modular forms of degree $2 n $
and weight $k $, consisting of smooth functions
$$f:\SP_{2 n} (\Q)\backslash\MP_{2 n} (\A)\to\C $$
such that:
\begin{enumerate}
\item $f (g z) =\debt ^ k (z) f(g)$
for any $z\in\widetilde U (n)\subset\MP_{2 n} (\R) $.
\item $X\cdot f (g) = 0 $
for any $X\in\mathfrak p ^ -\subset\mathfrak {sp}_{2 n,\R} $.
\end{enumerate}
\end{definition}
\begin{rmk}
Note that $M_k ^ {2 n} $
is naturally an $\MP_{2 n} (\A _f) $-module.
\end{rmk}
\begin{notation}\label{notation:where_psi}
\leavevmode
    \begin{enumerate}
        \item Let $\psi $
be the additive character of $\Q\backslash\A $
which is unramified at all finite places and satisfies $$\psi (x_\infty) = e ^ {2\pi i x_\infty} $$
for $x_\infty\in\R\subset\A. $
\item For $T\in \Sym_n(\Q)$, define
$$\psi_T\coloneqq \psi\circ \left(\frac{1}{2}\tr(T-)\right),$$
a unitary character of $N(\Q)\backslash N(\A)$ (where $N$ is the unipotent radical of the Siegel parabolic).
    \end{enumerate}
\end{notation}

From basic Fourier analysis, one deduces the following proposition.
\begin{prop}
    Let $f\in M_k^{2n}$.
     Then we have an identity of functions on $N(\A) \subset \MP_{2n}(\A)$:
    $$f = \sum_{T\in \Sym_n(\Q)} a_T(f)q_\A^T,$$
    where $q_\A^T\coloneqq e^{-\tr (T)}\cdot \psi_T$ and \begin{equation*}
a_T (f)\coloneqq\frac {e ^ {\trace (T)}} {\volume ([ N])}\integral_{[ N]} f (n)\psi_T ^ {-1} (n)\d n.
\end{equation*}
\end{prop}



\begin{definition}
For any subring $R\subset\C $, define
$$M_{k, R} ^ {2 n} =\set {f\in M_k ^ {2 n}\,:\, a_T (f)\in R\text { for all } T\in\symmetric_n (\Q)}. $$
\end{definition}

\subsection{ Weil representation}
\subsubsection{}\label{weil rep}
Let $k $
be a local field, and let $V $
be a quadratic space over $k $
of
\emph {odd}
dimension and trivial discriminant. Also fix an even integer $2 n\geq 2 $. For any nontrivial additive character $\psi $
of $k $, the Weil representation $\omega_\psi $
of $\SO (V) (k)\times\MP_{2 n} (k) $
is realized on the complex Schwartz space $\mathcal S (V ^ n,\C) $, and is determined by the following formulas.

\begin{equation}\label{eq:Weil_formulas}
\begin {cases}
\omega_\psi (g, 1)\phi (x) =\phi (g ^ {-1} x), & g\in\SO (V) (k).\\\omega_\psi (1, u)\phi (x) =\psi (\frac {1} {2} u (x)\cdot x)\phi (x), & u\in N (k)\cong\symmetric_n (k).\\
\omega_\psi (1, m)\phi (x) =\chi_\psi (m) |\debt (\overline m) | ^ {\frac {\dim V} {2}}\phi (\overline m ^ t x), & m\in\widetilde M (k).\\
\omega_\psi (1, w)\phi (x) =\gamma_w\integral_{V ^ n}\phi (y)\psi (x\cdot y)\d y.
\end {cases}
\end{equation}
Here, the notation is as follows:
\begin{itemize}[label= $\circ$]
\item $P = M N $
is the Siegel parabolic.
\item $\chi_\psi $
is a $\mu_8$-valued genuine character of $\widetilde M (k) $
described in \cite[p. 1661]{gan2012metaplectic}.

\item $w\in\MP_{2 n} (k) $
is a certain Weyl element such that $\overline wP\overline w ^ {-1} = P ^ {\text {op}} $.
\item $\gamma_w $
is a certain eighth root of unity.
\item The pairing on $V ^ n $
in the second and fourth equations is $$(x_1,\ldots, x_n)\cdot (y_1,\ldots, y_n) =\sum x_i\cdot y_i. $$
\item $\d x $
in the fourth equation is a self-dual Haar measure.\end {itemize}
\subsubsection {}\label{subsubsec:coeffs_for_Weil}
If $k $
is non-archimedean of residue characteristic $\l $
and $R\subset\C $
is a $\Z\left [\frac {1} {\l}\right] $-algebra containing all eighth and $\l $th power roots of unity, then the same formulas give a well-defined action of $\SO (V) (k)\times\MP_{2 n} (k) $
on the space $\mathcal S (V ^ n, R) $
of $R $-valued Schwartz functions.
\subsubsection{}
Globally, if $V $
is a quadratic space over $\Q $, we have Weil representations $\mathcal S(V ^ n\otimes\Q_v,\C) $
for all places $v $, defined using the localizations of the fixed global additive character $\psi $
of $\Q\backslash\A $. Similarly, we have the Weil representations $\mathcal S(V ^ n\otimes\A,\C) $, $\mathcal S(V ^ n\otimes\A_f,\C) $, etc., defined as restricted tensor products.

\subsection{ Classical theta lifting}
\begin{notation}\label{notation:classical_theta}
For a Schwartz function $\phi\in \mathcal S(V ^ n\otimes\A,\C) $
and a cuspidal automorphic form $\alpha $
on $\SO (V)(\A) $, we write $\theta_\phi (\alpha) $
for the automorphic form on $\MP_{2 n} (\A) $
defined by
\begin{equation}
\theta_\phi (\alpha) (h) =\integral_{[\SO(V)]}\alpha (g)\sum_{x\in V ^ n}\omega_\psi (g, h)\phi (x)\d g.
\end{equation}
The normalization depends on a choice of Haar measure.
\end{notation}
\subsubsection{}
Suppose $V $
is positive definite and let $K\subset\spin (V) (\A_f) $
be a neat compact open subgroup. For $\alpha: \Sh_K(V)\to \C$, define an automorphic form $\overline \alpha$ on $\spin(V)(\A)$, which descends to $\SO(V)(\A)$, by
 $$\overline \alpha(g_fg_\infty) = \frac{1}{\volume (\widehat \Z^\times)}\integral_{\widehat \Z^\times} \alpha(g_fz) \d z,\;\;\forall g = g_fg_\infty\in \spin(V)(\A).$$
\begin{lemma}\label{theta fourier coeffs}
Let $\phi_\infty\in \mathcal S(V ^ n\otimes\R,\C) $
be the Gaussian $$\phi_\infty (x) = e ^ {- x\cdot x}. $$
Then for  $\alpha:\Shimura_K (V)\to \C $ 
and  $$\phi_f\in \mathcal S(V ^ n\otimes\A_f,\C)^K, $$
$\theta_{\phi_f\otimes\phi_\infty} (\overline\alpha) $
lies in $M ^ {2 n}_{\frac {\dim (V)} {2}}
$
and $$a_T\left (\theta_{\phi_f\otimes\phi_\infty} (\overline \alpha)\right) =C_K\cdot \alpha\left (Z (T,\phi_f)_K\right), $$
where the constant is $$ C_K = \frac{\volume\left(\SO(V)(\R)\cdot \image\left(K \to \SO(V)(\A_f)\right)\right)}{[K\cdot \widehat \Z^\times: K]}.$$
\end {lemma}
\begin{proof}
In the Fock model of the $\left(\mathfrak{sp}_{2n \dim(V)}, \widetilde U(n\dim(V))\right)$-module associated to $\mathcal S(V^n\otimes \R, \C)$, $\phi_\infty$ has degree zero, hence is annihilated by $\mathfrak p^-\subset \mathfrak {sp}_{2n}$; cf. \cite[(2.2)]{howe1989transcending}. By comparing degrees with  \cite[Proposition 2.1(2)]{adams1998genuine}, we also conclude that $\widetilde U(n)$ acts on $\theta_{\phi_f\otimes\phi_\infty}(\overline\alpha)$ by $\det^{\frac{\dim (V)}{2}},$ which proves
$\theta_{\phi_f\otimes\phi_\infty}(\overline \alpha) \in M^{2n}_{\frac{\dim(V)}{2}}$.

It remains to compute the Fourier coefficients. We calculate:
\begin{align*}
a_T (\theta_{\phi_f\otimes\phi_\infty} (\overline \alpha)) & =\frac {e ^ {\trace T}} {\volume ([N])}\integral_{[N]}\psi_T ^ {-1} (u)\integral_{[\SO (V)]}\overline\alpha (g)\sum_{x\in V ^ n (\Q)}\omega_\psi (g, u)\phi_f\otimes\phi_\infty (x)\d g\d u\\
& =\frac {e ^ {\trace T}} {\volume ([N])}\integral_{[\SO (V)]}\overline\alpha (g)\sum_{x\in V ^ n (\Q)}\integral_{[N]}\psi_T ^ {-1} (u)\psi \left(\frac {1} {2} u (x)\cdot x\right)\omega_\psi (g, 1)\phi_f\otimes\phi_\infty (x)\d g\d u\\
& =\integral_{[\SO (V)]}\overline\alpha (g)\sum_{x\in\Omega_{T, V}(\Q)}\omega_\psi (g, 1)\phi_f (x)\d g.
\end{align*}
Fix a base point $x= (x_1, \ldots, x_m) \in\Omega_{T, V}(\Q) $
and let $V_0 = \Span_\Q\set{x_1,\ldots, x_m}$. Then, since $\Omega_{T,V}(\Q) $
is a single $\SO (V) (\Q) $-orbit, we may rewrite the final equation as
$$a_T (\theta_{\phi_f\otimes\phi_\infty} (\alpha)) =\integral_{\SO (V_0^\perp) (\A)\backslash\SO (V) (\A)}\phi_f (g ^ {-1} x)\integral_{[\SO (V_0^\perp)]}\overline \alpha (hg)\d h\d g.$$ This coincides with the claimed formula by definition of $\alpha(Z(T, \phi_f)_K)$. 
\end{proof}
\subsubsection{}
As a special case, suppose $V =\Q $
is the one-dimensional space with quadratic form $a\mapsto a ^ 2 $. Then $\SO (V) $
is the trivial group and we take the unit Haar measure on $\SO (V) (\A) =\set {\operatorname {id}} $. For each prime $\l $, let $$\phi_\l\in \mathcal S (V ^ n\otimes\Q_\l,\Z) $$
be the indicator function of the lattice $$\Z_\l ^ n\subset\Q ^ n_\l = V ^ n\otimes\Q_\l, $$
and let $\phi_f =\otimes_\l\phi_\l $. Let $\blackboardone$ be the unit automorphic form on $\SO(V)(\A)$.  We write
\begin{equation}
\theta_{\frac {1} {2}}\coloneqq\theta_{\phi_f\otimes\phi_\infty} (\blackboardone).
\end{equation}

\begin{lemma}\label {action on wheat one half theta function lemma}
Let $\l $
be an odd prime, and suppose $R\subset\C $
is a $\Z\left [\frac {1} {\l}\right] $-algebra containing all eighth and $\l $th power roots of unity. Then for any $g\in\MP_{2 n} (\Q_\l) $, $g\cdot\theta_{\frac {1} {2}} $
lies in $M ^ {2 n}_{\frac {1} {2}, R} $, and the constant term $a_0 (g\cdot\theta_{\frac {1} {2}}) $
of its Fourier expansion lies in $R ^\times $.
\end{lemma}
\begin{proof}
By the obvious equivariance of the theta lift, we have $$g\cdot\theta_{\frac {1} {2}} =\theta_{\omega_\psi (1, g)\phi_f\otimes\phi_\infty} (\blackboardone). $$
Since $$\mathcal S(V ^ n\otimes\Q_\l, R)\subset \mathcal S(V ^ n,\C) $$
is $\MP_{2 n} (\Q_\l) $-stable, Lemma \ref{theta fourier coeffs} shows $$g\cdot\theta_{\frac {1} {2}}\in M ^ {2 n}_{\frac {1} {2}, R}. $$
Also,  $\theta_{\frac {1} {2}} $
is $\MP_{2 n} (\Z_\l) $-invariant  by construction.

By the Iwasawa decomposition
$$\SP_{2 n} (\Q_\l) = P (\Q_\l)\cdot\SP_{2 n} (\Z_\l), $$
it suffices to show $a_0 (g\cdot\theta_{\frac {1} {2}})\in R ^\times $
for $g\in\widetilde P (\Q_\l) $. Now, by Lemma \ref{theta fourier coeffs} again,
$$a_0 (g\cdot\theta_{\frac {1} {2}}) =\omega_\psi (1, g)\cdot\phi_f (0), $$
and it follows from the explicit formulas in (\ref{weil rep}) that $$\omega_\psi (1, g)\cdot\phi_f (0)\in R ^\times $$
for all $g\in\widetilde P (\Q_\l) $; this proves the lemma.
\end{proof}
\begin {prop}\label{prop:integral_structure_stable_metaplectic}
Let $R\subset\C $
be a $\Z\left [\frac {1} {\l}\right] $-algebra containing all eighth and $\l $th power roots of unity. Then for all $k\in\frac {1} {2}\Z $, $$M ^ {2 n}_{k, R}\subset M ^ {2 n}_k $$
is stable under the action of $\MP_{2 n} (\Q_\l) $.
\end{prop}
\begin{proof}
If $k\in\Z $
is integral, this is a consequence of the $q $-expansion principle for classical Siegel modular forms \cite[Chapter 5, Proposition 1.8]{faltings1990degeneration}. In general, for any $f\in M ^ {2 n}_{k, R} $
and $f'\in M ^ {2 n}_{k', R} $, the product $f f'\in M ^ {2 n}_{k + k', R} $
has formal $q $-expansion:
\begin{equation}\label {Q expansion equation for product of two modular forms}
(ff') |_{N (\A)} =\sum_{T\in\symmetric_n (\Q)}\sum_{S\in\symmetric_n (\Q)} a_T (f) a_S (f') q_{\A} ^ {S + T}.
\end{equation}
(This expression make sense because $a_T (f) $
and $a_S (f') $
are each supported on positive semi-definite matrices with bounded denominators in their entries, see \cite[Proposition 1.1]{shimura1995metaplectic}.)

We apply this to $f\in M ^ {2 n}_{k, R} $, with $k\in\frac {1} {2} +\Z $, and $f' =\theta_{\frac {1} {2}} $. Since $k +\frac {1} {2}\in\Z $, we have $$g (f\theta_{\frac {1} {2}}) = g (f) g(\theta_{\frac {1} {2}})\in M ^ {2 n}_{k +\frac {1} {2}, R} $$
for any $g\in\MP_{2 n} (\Q_\l) $.

Now, by Lemma \ref{action on wheat one half theta function lemma} above, the $q $-expansion of $g(\theta_{\frac {1} {2}} )$
has an inverse power series $$\sum_{T\in\symmetric_n (\Q)} b_T q ^ T_{\A} $$
with $b_T\in R $. Hence, by the uniqueness of $q $-expansions, the Fourier coefficients of $g( f) $
lie in $R $
as well.
\end{proof}

\subsection{ Formal theta lifts}\label{subsec:formal_theta}
 \subsubsection{}
 Suppose $V $
has signature $(m,0)$ or  $(m, 2) $
for some $m\geq 1 $, and let $$K =\prod K_\l\subset\spin (V) (\A_f) $$
be a neat compact open subgroup.
For any subring $R\subset\C $, we define $\Test_K (V, R) $
to be $$\operatorname{Hom}_\Z (\Z [\Shimura_K (V)], R)= \operatorname{Hom}_R (R [\Shimura_K (V)], R) $$
in the positive definite case, or $$\operatorname{Hom}_\Z (\CH ^\ast(\Shimura_K (V)), R) =\operatorname{Hom}_R (\CH ^\ast(\Shimura_K (V),R), R)  $$
in the indefinite case.

\subsubsection{} Now fix $1\leq n \leq m$. For any $K $-invariant Schwartz function $\phi\in \mathcal S(V ^ n\otimes\A_f, R) $
and any $\alpha\in\Test_K (V, R) $, we define the formal theta lift
$$\Theta (\alpha,\phi)_K\coloneqq\sum_{T\in\symmetric_n (\Q)_{\geq 0}}\alpha\left (Z (T,\phi)_K\right) q_\A ^ T. $$
The subscript $K $
will be omitted when there is no risk of confusion.
\subsubsection{}\label{subsubsec:test_convolution}
Let $\l $ be a prime
such that $K_\l $
has pro-order  invertible in $R $, and let
$$\Test_{K ^ \l} (V, R)\coloneqq\varinjlim_{K_\l'}\Test_{K ^ \l K_\l'} (V), $$
where the transition maps are induced by pushforward. Note that $\Test_{K ^ \l} (V, R) $
has a natural action of $\spin (V) (\Q_\l) $, dual to the one described in \cite[p. 41]{zhang2009modularity}. For $\alpha\in\Test_{K ^ \l} (V) $
and a $K ^ \l $-invariant Schwartz function $\phi\in \mathcal S(V ^ n\otimes\A_f, R) $, we define the renormalized formal theta lift
\begin{equation}
\Theta_{K ^ \l} (\alpha,\phi)\coloneqq\frac {\Theta (\alpha,\phi)_{K ^ \l K_\l'}} {[K_\l: K_\l']},
\end{equation}
for any $K_\l'\subset K_\l $
fixing both $\alpha$
and $\phi $. 
Because the cycles $Z (T,\phi)_{K ^ \l K_\l'} $
are compatible under pullback, $\Theta_{K ^ \l} (\alpha,\phi) $
does not depend on the choice of $K_\l' $.
\begin{prop}\label{prop:equivariance_and_modularity}
Suppose $R$ is a $\Z[1/\l]$-algebra containing all eighth and $\l$th power roots of unity, and the pro-order of $K_\l$ is invertible in $R$. Then $\Theta_{K ^ \l} $
defines a $\spin (V) (\Q_\l)\times\MP_{2 n} (\Q_\l) $-equivariant map
$$\Test_{K ^ \l} (V, R)\otimes \mathcal S(V ^ n\otimes\A_f, R)^{K^\l}\rightarrow M_{\frac {\dimension (V)} {2}, R} ^ {2 n}. $$
\end{prop}
\begin{proof}
Note that both modularity and equivariance can be checked after extending scalars, so without loss of generality suppose $R =\C $. In the definite case, the proposition is a formal consequence of Lemma \ref{theta fourier coeffs} above. In the indefinite case, the modularity of the formal theta lift is \cite[Theorem 6.2]{bruinier2015kudla} and the equivariance is \cite[Corollary 5.11]{kudla1997orthogonal} combined with \cite[Corollary 2.12]{zhang2009modularity}.
\end{proof}

\section {Construction of Galois cohomology classes and special periods}\label{sec:construction}

Before beginning this section, we establish the notation that will be in force for most of the rest of the paper.
\begin{notation}\label{notation:pi_basic}
\leavevmode
    \begin{enumerate}
        \item Fix a relevant automorphic representation $\pi$ of $\GSP_4$ with trivial central character. Fix as well a strong coefficient field $E_0$ for $\pi$ (Definition \ref{def:coefficient_pi}). If $\pi$ is endoscopic associated to a pair $(\pi_1,\pi_2)$ of automorphic representations of $\GL_2(\A)$, then we also assume without loss of generality that $E_0$ is a strong coefficient field for $\pi_1$ and $\pi_2$.\footnote{In fact, it is not difficult to check that this last assumption is automatic. The main point is to  use Hodge-Tate theory to verify that, for all $\iota: \overline \Q_p \isomorphism \C$, $\rho_{\pi_1,\iota}$ and $\rho_{\pi_2,\iota}$ cannot differ by a Galois twist.}
        \item We fix a finite set $S$ of finite primes, containing 2 and all $\l$ such that $\pi_\l$ is ramified. 
\item For all primes $\p$ of $E_0$, let $O_\p\subset E_{0,\p}$ be the ring of integers, and let $\varpi_\p\in O_\p$ be a uniformizer, with residue field $k_\p= O_\p/\varpi_\p$. We drop the subscript $\p$ when there is no risk of confusion.
For a finite set of primes $S' \supset S$, we write $\m_{\pi,\p}^{S'} \subset {\T}^{S'}_{O_\p}$ for the maximal ideal defined by the Hecke eigenvalues of $\pi$, and drop the decorations when they are clear from context. 
\item Notation \ref{notation:V_T_pi} remains in force.
    \end{enumerate}
\end{notation}
\begin{rmk}
    We assume $2\in S$ so that 2 is not admissible under Definition \ref{def:admissible} below, and to prove the (convenient but not essential) Lemma \ref{lem:tidy_suffices}.
\end{rmk}
\subsection{Assumptions on $\p$}
We now define some assumptions on primes $\p$ of $E_0$. Let  $p$ denote the  residue characteristic.
\begin{assumption}\label{assumptions_on_p}
\leavevmode
\begin{enumerate}
    \item\label{assume:pi_p_unr} $p$ does not lie in $S$.
    \item \label{assume:pi_generic}There exists a rational prime $\l\not\in S\cup \set{p}$ such that $\l^4 \not\equiv 1 \pmod p$ and $\overline\rho_{\pi,\p}(\Frob_\l)$ has distinct eigenvalues, no two having ratio $\l$. 
    
    The final assumption depends on whether $\pi$ is endoscopic. 
    \item\label{assume:irreducible}
\begin{itemize}
    \item     If $\pi$ is not endoscopic, then $\overline\rho_{\pi,\p}$ is absolutely irreducible. 
\item If $\pi$ is endoscopic associated to a pair $(\pi_1,\pi_2)$ of automorphic representations of $\GL_2(\A)$, then $\overline\rho_{\pi_1,\p}$ and $\overline\rho_{\pi_2,\p}$ are both absolutely irreducible.
\end{itemize}
\end{enumerate}
\end{assumption}
\begin{rmk}\label{rmk:p>5}
    Assumption \ref{assumptions_on_p}(\ref{assume:pi_generic}) implies that $p > 5$.
\end{rmk}
\begin{notation}
\leavevmode
Under Assumption \ref{assumptions_on_p}(\ref{assume:irreducible}):
\begin{enumerate}
    \item   Let  $T_{\pi,\p}$  be an $O_\p[G_\Q]$-module such that $T_{\pi,\p}\otimes \Q_p = V_{\pi,\p}$; Assumption \ref{assumptions_on_p}(\ref{assume:irreducible}) implies that $T_{\pi,\p}$ is unique up to isomorphism. 
\item For all $n \geq 1$, we write $T_{\pi,\p, n} \coloneqq T_{\pi,\p}/\varpi_\p^n T_{\pi,\p}$. Also let $\overline T_{\pi,\p}\coloneqq T_{\pi,\p,1}$. 
\item When $\p$ is clear from context, it may be dropped from the above notations. 
\end{enumerate}
\end{notation}

\begin{rmk}\label{rmk:rho_O_valued}
Under Assumption \ref{assumptions_on_p}(\ref{assume:irreducible}), $T_{\pi,\p}$ is isomorphic to its $O_\p$-dual; we use this to view $\rho_{\pi,\p}$ as valued in $\GSP_4(O_\p)$, and $\overline\rho_{\pi,\p}$ as valued in $\GSP_4(k_\p)$. 
\end{rmk}
\begin{lemma}\label{lem:easy_assumptions_cofinite}
    Assumption \ref{assumptions_on_p} holds for all but finitely many primes $\p$ of $E_0$.
\end{lemma}
    \begin{proof}
That Assumption \ref{assumptions_on_p}(\ref{assume:irreducible}) holds for all but finitely many primes $\p$ follows from \cite[Theorem 1.2(i)]{weiss2022images} in the non-endoscopic case; in the endoscopic case, it follows from \cite[Theorem 2.1]{ribet1985largeimage}. It is also obvious that Assumption \ref{assumptions_on_p}(\ref{assume:pi_p_unr}) holds for cofinitely many $\p$.
We consider Assumption \ref{assumptions_on_p}(\ref{assume:pi_generic}).
    
    First fix an arbitrary $\p$ with residue characteristic $p$. The image of $\rho_{\pi,\p}$ contains an element with distinct eigenvalues by \cite[Theorem 1]{sen1973liealgebras} and Theorem \ref{thm:rho_pi_LLC}(\ref{part:rho_pi_LLC_HT}). Hence by the Chebotarev density theorem, there exists a prime $\l\not\in S\cup\set{p}$ such that $\Frob_\l$ has distinct eigenvalues on $\rho_{\pi,\p}$. The Satake parameter  of $\pi_\l$ is therefore multiplicity-free and of the form $\set{\alpha,\beta,\alpha^{-1}, \beta^{-1}}$ with $|\alpha| = |\beta| = 1$. 
Let $E_1 = E_0(\alpha,\beta)$, which is a finite extension because the Hecke eigenvalues of $\pi_\l$ lie in $E_0$. For all but finitely many primes $\p_1\nmid \l$ of $E_1$, we will have
\begin{align*}
    \alpha^2&\not\equiv 1, \l, \l^{-1} \pmod {\p_1}\\
      \beta^2&\not\equiv 1, \l, \l^{-1} \pmod {\p_1}\\
    \alpha\beta &\not\equiv 1, \l, \l^{-1} \pmod {\p_1}\\
    \alpha/\beta &\not\equiv 1, \l, \l^{-1} \pmod {\p_1},\\
    \l^4  &\not\equiv 1, \pmod{\p_1}.
\end{align*}
For such a $\p_1$, let  $\p' = \p_1|_{E_0}$ and let $p'$ be the residue characteristic of $\p'$. Then we have $\l^4 \not\equiv 1 \pmod{ p'}$ and the eigenvalues of $\Frob_\l$ on $\overline \rho_{\pi,\p'}$ are
distinct and not of ratio $\l$, i.e. Assumption \ref{assumptions_on_p}(\ref{assume:pi_generic}) holds for $\p'$. 
\end{proof}

\begin{lemma}\label{lem:H1_tors_free_new}
    Assume $\p$ satisfies Assumption \ref{assumptions_on_p}(\ref{assume:irreducible}). Then:
    \begin{enumerate}
        \item \label{lem:H1_tors_free_part}$H^1(\Q, T_{\pi})$ is $\varpi$-torsion-free.
        \item \label{lem:H1_tf_order}
        Suppose given  $c\in H^1(\Q, T_{\pi})$ and $a\geq 1$
        such that $c\not\in \varpi^a H^1(\Q, T_{\pi})$. Then for all $n \geq 1$, the image $c_n \in H^1(\Q, T_{\pi,n})$ satisfies 
        $$\ord_{\varpi}  c_n > n - a.$$
    \end{enumerate}
\end{lemma}
\begin{proof}
    The assumption implies that $H^0(\Q, \overline T_{\pi}) = 0$. The long exact sequence in Galois cohomology associated to 
    $$0 \to T_{\pi} \xrightarrow{\varpi} T_{\pi} \to \overline T_{\pi} \to 0$$ therefore gives (\ref{lem:H1_tors_free_part}).
    For (\ref{lem:H1_tf_order}), a similar argument to (\ref{lem:H1_tors_free_part}) shows that the map $H^1(\Q, T_{\pi,a}) \xrightarrow{\varpi^{n-a}} H^1(\Q, T_{\pi,n})$ is injective, so the 
     kernel of $$\varpi^{n-a}: H^1(\Q, T_{\pi,n}) \to H^1(\Q, T_{\pi,n})$$ coincides with the kernel of $H^1(\Q, T_{\pi,n}) \to H^1(\Q, T_{\pi,a})$. Hence it suffices to show that $c_a \neq 0$, which is clear from the assumption $c\not\in \varpi^a H^1(\Q, T_\pi)$ and the long exact sequence in Galois cohomology associated to    
    $$0 \to T_{\pi} \xrightarrow{\varpi^a} T_{\pi} \to  T_{\pi,a} \to 0.$$
\end{proof}
\begin{lemma}\label{lem:m_is_nonEisgeneric}
    Suppose $\p$ satisfies Assumption \ref{assumptions_on_p}. Then $\m_{\pi, \p}\subset \T^S_{O}$ is non-Eisenstein and generic. 
\end{lemma}
\begin{proof}
Recall from Remark \ref{rmk:rho_O_valued} that $\overline\rho_{\pi}$ is valued in  $\GSP_4(k)$. Then $\m_{\pi,\p}$ is clearly of Galois type associated to $\overline\rho_{\pi}$. 
The genericity  of $\m_{\pi,\p}$  (Definition \ref{def:generic}) therefore follows from Assumption \ref{assumptions_on_p}(\ref{assume:pi_generic}).
From Assumption \ref{assumptions_on_p}(\ref{assume:irreducible}), it follows that $\overline T_{\pi,\p}\otimes \overline\F_p$ contains no Galois-stable line.  So if $\m_{\pi,\p}$ were Eisenstein, $\overline\rho_{\pi,\p}$ would have to factor through a Siegel parabolic subgroup. In particular, then $\overline T_{\pi,\p}\otimes \overline\F_p =  \rho_0 \oplus \rho_0 \otimes \det\rho_0^{-1} \otimes \omega_p$, where $\rho_0: G_\Q \to \GL_2(\overline \F_p)$ is some irreducible representation, and $\omega_p$ is the mod-$p$ cyclotomic character. This is clearly impossible by Assumption \ref{assumptions_on_p}(\ref{assume:irreducible}) and Lemma  \ref{lem:residually distinct} below. 
\end{proof}
\begin{lemma}\label{lem:residually distinct}
Let $\p|p$ be a prime of $E_0$ such that $p > 5$ and $\pi_p$ is unramified. If $\pi$ is endoscopic associated to $(\pi_1,\pi_2)$, then $\overline\rho_{\pi_1}$ and $\overline\rho_{\pi_2}$ are not isomorphic. 
\end{lemma}
\begin{proof}
    Recall that $\rho_{\pi_1}$ and $\rho_{\pi_2}$ have  Hodge-Tate weights $\set{-1, 2}$ and $\set{0, 1}$ up to reording, by Lemma \ref{lem:when_endoscopic_relevant} and Theorem \ref{thm:rho_GL2_LLC}(\ref{part:rho_GL2_LLC_HT}). The lemma then follows from  Fontaine-Laffaille theory \cite[Th\'eor\`eme 6.1]{fontaine1982laffaille}.
\end{proof}

\subsection{Admissible primes}
In this section, $\p|p$ is a prime of $E_0$.
\begin{definition}\label{def:admissible}
\leavevmode
\begin{enumerate}
    \item We say a prime $q\not\in S \cup \set{p}$ is \emph{admissible} for  $\rho_{\pi}= \rho_{\pi,\p}$ if $q^4\not\equiv 1 \pmod p$ and the generalized eigenvalues of $\overline\rho_{\pi}(\Frob_q)$ are of the form $\set{q,\alpha, 1,q/\alpha}$, with $\alpha\neq \pm q, \pm 1, q^2, q^{-1}$.
    \item If $q$ is admissible, define $n(q) \geq 1$ to be the greatest integer such that $\det(\Frob_q - q|V_\pi)\equiv 0 \pmod {\varpi^{n(q)}}$. 
\item We say $q$ is $n$-\emph{admissible} if  it is admissible and $n(q) \geq n$.
 \item If $Q \geq 1$ is squarefree, we say $Q$ is \emph{admissible} (resp. $n$-\emph{admissible}) if all primes $q|Q$ are so.
 \item
 Analogously, an element $g\in G_\Q$ is called \emph{admissible} for $\rho_{\pi}$ if  $\nu_g\coloneqq \chi_{p,\cyc} (g)$ satisfies $\nu_g ^4\not \equiv 1 \pmod p$, and
     $g$ acts on $V_{\pi}$ with generalized eigenvalues $\set{\nu_g, 1, \alpha, \nu_g/\alpha}$ for some 
     $$\alpha\not\equiv \pm \nu_g, \pm 1, \nu_g^2, \nu_g^{-1} \pmod{\varpi}.$$
\end{enumerate}
\end{definition}

\begin{lemma}\label{lem:M0_adm_def}
Suppose $\p$ satisfies 
 Assumption \ref{assumptions_on_p}(\ref{assume:irreducible}). Then 
a prime $q$ is $n$-admissible for $\rho_{\pi}=\rho_{\pi,\p}$ if and only if there exists a $G_{\Q_q}$-stable decomposition $$T_{\pi,n} = M_{0,n} \oplus M_{1,n}$$ such that:
\begin{enumerate}
    \item $M_{0,n}$ and $M_{1,n}$ are each free of rank two over $O/\varpi^n$, and 
    $\Frob_q|_{M_{0,n}} = \begin{pmatrix}
        q & \\ & 1
    \end{pmatrix}$ in some basis.
    \item $\Frob_q^2 - 1$, $\Frob_q^2 - q^2$, $\Frob_q - q^2$, and $\Frob_q - q^{-1}$ all act invertibly on $M_{1,n}$. 
\end{enumerate}
\end{lemma}
\begin{proof}
    Immediate from Definition \ref{def:admissible}.
\end{proof}

\begin{lemma}\label{lem:admissible_TFAE}
 Let $\p$ be a prime of $E_0$.    The following are equivalent:
    \begin{enumerate}
        \item\label{part:admissible_TFAE 1} There exist admissible primes for $\rho_{\pi}$. 
                \item \label{part:admissible_TFAE n}For all $n$, there exist $n$-admissible primes for $\rho_{\pi}.$

        \item \label{part:admissible_TFAE element}There exists an admissible element $g\in G_\Q$ for $\rho_{\pi}$. 

    \end{enumerate}
\end{lemma}
\begin{proof}
  Clearly  (\ref{part:admissible_TFAE n}) implies (\ref{part:admissible_TFAE 1}), and (\ref{part:admissible_TFAE element}) implies (\ref{part:admissible_TFAE n}) by the Chebotarev density theorem. The proof that (\ref{part:admissible_TFAE 1}) implies (\ref{part:admissible_TFAE n}) follows the same argument of \cite[Lemma 2.7.1]{liu2022beilinson}, and (\ref{part:admissible_TFAE n}) implies (\ref{part:admissible_TFAE element}) by compactness.
\end{proof}

\subsubsection{}
Now suppose $\pi$ is endoscopic associated to a pair $(\pi_1,\pi_2)$ of automorphic representations of $\GL_2(\A)$. 
\begin {definition}\label{def:admissible_endoscopic}  
\leavevmode
\begin{enumerate}
\item 
A prime $q\not\in S\union\set {p} $
is called \emph{BD-admissible} for $\rho_{\pi_i}= \rho_{\pi_i,\p} $, with $i = 1$ or $2$, 
if $q^2\not\equiv 1\pmod p$ and the generalized eigenvalues of $\overline\rho_{\pi_i}(\Frob_q)$ are $\set{1, q}$. 
\item If $q$ is BD-admissible for $\rho_{\pi_i}$, define $n(q)\geq 1$ to be the greatest integer such that $\debt(\Frob_q - q|V_{\pi_i}) =\equiv 0 \pmod {\varpi^{n(q)}}.$
\item We say $q$ is $n$\emph{-BD-admissible} if it is BD-admissible and $n(q) \geq 1$.
 \item If $Q \geq 1$ is squarefree, we say $Q$ is BD-admissible (resp $n$-BD-admissible) for $\rho_{\pi_i}$ if all $q|Q$ are so. 
\item Likewise, an element $g\in G_\Q$ is called BD-admissible for $\rho_{\pi_i}$ if $\chi_{p,\cyc}(g)^2 \not\equiv 1\pmod p$ and $g$ acts on $V_{\pi_i}$ with eigenvalues $\chi_{p,\cyc}(g)$ and 1.
\end{enumerate}
\end{definition}

\begin{rmk}\label{rmk:BD_adm}
    \leavevmode
    \begin{enumerate}
    \item 
Definition \ref{def:admissible_endoscopic} is adapted from \cite[p. 18]{bertolini2005iwasawa}, but there it is allowed that the eigenvalues of $\Frob_q$ on $\overline\rho_{\pi_i}$ are  $-1$ and $-q$.
\item   \label{rmk:BD_adm_dichotomy_part}  If $q $
is $n $-admissible for $\rho_{\pi} $, then it is $n $-BD-admissible for exactly one of $\rho_{\pi_1} $ and $\rho_{\pi_2} $; and if $g\in G_\Q$ is admissible for $\rho_{\pi}$, then it is BD-admissible for exactly one of $\rho_{\pi_1}$ and $\rho_{\pi_2}$.
In particular, if $Q \geq 1$ is admissible for $\rho_{\pi}$, there is a unique factorization $Q = Q_1\cdot Q_2$ with $Q_1,Q_2\geq 1$, such that all $q | Q_i$ are BD-admissible for $\rho_{\pi_i}$.
 \end{enumerate}
\end{rmk}

\subsubsection{}
For any prime $q\not\in S\cup \set{p}$, 
recall that $H^1_f(\Q_q, T_{\pi,n}) = H^1_{\unr} (\Q_q, T_{\pi,n})$,  and set $H^1_{/f} (\Q_q, T_{\pi,n}) = H^1(\Q_q, T_{\pi,n})/H^1_f(\Q_q, T_{\pi,n})$. 

\begin{prop}\label{prop:local_adm_free}
    Suppose $q$ is $n$-admissible. Then $H^1_f(\Q_q, T_{\pi,n})$ and $H^1_{/f}(\Q_q, T_{\pi,n})$ are each free of rank one over $\O/\varpi^n$,     and local Poitou-Tate duality induces a perfect pairing $$H^1_f(\Q_q, T_{\pi,n})\times H^1_{/f}(\Q_q, T_{\pi,n}) \to O/\varpi^n.$$
\end{prop}
\begin{proof}
First note that the induced pairing is perfect because $H^1_f(\Q_q, T_{\pi,n})$ is self-annihilating under the Tate pairing, and one can check $\lg_O H^1(\Q_q, T_{\pi,n}) = 2\lg_O H^1_f(\Q_q, T_{\pi,n})$ using the local Euler characteristic formula and local duality. So it suffices to prove that 
 $$H^1_f(\Q_q, T_{\pi,n}) = T_{\pi,n}/(\Frob_q - 1)T_{\pi,n}$$ is free of rank one over $O/\varpi^n$. 
Indeed, this is immediate from Lemma \ref{lem:M0_adm_def}.
\end{proof}
\begin{notation}\label{notation:loc_q}If $q$ is $n$-admissible and $S'\supset S$ is a finite set with $q\not\in S'$, then by Proposition \ref{prop:local_adm_free} we have the localization and residue maps:
\begin{align*}
    \loc_q: H^1(\Q^{S'}/\Q, T_{\pi,n}) &\to H^1_f(\Q_q, T_{\pi,n}) \simeq O/\varpi^n\\
    \partial_q: H^1(\Q, T_{\pi,n}) &\to H^1_{/f}(\Q_q, T_{\pi,n}) \simeq O/\varpi^n
\end{align*}
\end{notation}
\subsection{Level structures and test vectors}
Fix a prime $\p$ of $E_0$ of residue characteristic $p$. 
\begin{definition}\label{def:S_tidy}
For any squarefree $D\geq 1$, an \emph{$S$-level structure} for $\spin(V_D)$ is a compact open subgroup $K=\prod K_\l \subset \spin(V_D)(\A_f)$ such that:
  \begin{enumerate}
      \item $K$ is neat in the sense of \cite[\S0.1]{pink1990arithmetical}.
      \item For all $\l\not\in S\cup \div(D)$, $K_\l$ is hyperspecial.
      \end{enumerate}
  An \emph{$S$-tidy level structure} is an $S$-level structure satisfying:
  \begin{enumerate}
  \setcounter{enumi}{2}
      \item \label{def:S_tidy_three}If $K_Z \subset \A_f^\times$ is the intersection of $K$ with the center of $\spin(V_D)(\A_f)$, then  the finite group $\Q^\times \backslash \A_f^\times/ K_Z$ has order coprime to $p$. 
  \end{enumerate}
\end{definition}


 The reason for the final condition of Definition \ref{def:S_tidy} is the following convenient lemma:
\begin{lemma}\label{lem:S_tidy_upshot}
    Suppose $K$ is an $S$-tidy level structure for $\spin(V_D)$.
    Then for all finite sets $S' \supset S$  and  all $\l\not\in S\cup \div(D)$, we have
    $$\langle \l \rangle = 1 \text{ on } H^\ast(\Sh_K(V_D)(\C), O)_{\m^{S'}_{\pi,\p}}.$$
\end{lemma}
\begin{proof}
By Definition \ref{def:S_tidy}(\ref{def:S_tidy_three}), after replacing $O$ by a finite extension we can write
$$H^\ast(\Sh_K(V_D)(\C), O) = \bigoplus _\chi H^\ast(\Sh_K(V_D)(\C), O)_\chi,$$ where $\chi$ runs over $O$-valued characters of the finite group $\Q^\times \backslash \A_{f}^\times / K_Z$ and $$\langle \l \rangle = \chi(\l) \text{ on } H^\ast(\Sh_K(V_D)(\C), O)_\chi.$$
The characters $\chi$ are distinct modulo  $\varpi$, and the lemma follows because $\pi$ has trivial central character. 
\end{proof}

\begin{definition}
    \label{def:test_vectors}
    Let $D \geq 1$ be squarefree, let $K$ be 
   an $S$-level structure for $\spin(V_D)$, and let $R = O$ or $O/\varpi^n$, viewed as a $\T^S_O$-module via the Hecke eigenvalues of $\pi$. We define $\Test_K(V_D, \pi, R)$ as follows.
    \begin{enumerate}
        \item If $\sigma(D)$ is \emph{even} and $\p$ satisfies Assumption \ref{assumptions_on_p}, then $$\Test_K(V_D, \pi,R) = \Hom_{{\T}^{S\cup \div(D)}_O[G_\Q]}\left(H^3_\et (\Sh_K(V_D)_{\overline \Q}, O(2)) , T_{\pi}\otimes_O R\right).$$  
        \item If $\sigma(D)$ is \emph{odd}, then
       $$\Test_K(V_D, \pi,R) = \Hom_{{\T}^{S\cup \div(D)}_O}\left(O\left[\Sh_K(V_D)\right], R\right).$$
    \end{enumerate}
\end{definition}
\begin{rmk}
The definition of $\Test_K(V_D, \pi, R)$ depends only on $K$, and not on $S$; one can check this using Theorem \ref{thm:JL}(\ref{part:JL_four}), and Corollary \ref{cor:T_embedding} when $\sigma(D)$ is even.
\end{rmk}
\subsection{Constructions}
Fix a prime $\p$ of $E_0$ of residue characteristic $p$.

\begin{construction}\label{subsubsec:lambda_constr}
Let $D\geq 1$ be squarefree, and let $Q \geq 1$ be admissible and coprime to $D$, such that $\sigma(DQ)$ is \emph{odd}.
For any $S$-level structure $K$ for $\spin(V_{DQ})$:
\begin{enumerate}
    \item 
We define
$$\lambda_n^D(Q; K) 
\subset O/\varpi^n$$ to be the submodule spanned by the elements $\alpha(z)$, where:
\begin{itemize}[label = $\circ$]
    \item $\alpha$ lies in $\Test_K(V_{DQ}, \pi, O/\varpi^n)$.
    \item $z$ lies in $\SC^2_K(V_{DQ}, O)$ (Notation \ref{notation:SC}).
\end{itemize}
\item 
If $\phi\in \mathcal S(V_{DQ}^2 \otimes \A_f, O)^K$ is a test function, then we define
$$\lambda_n^D(Q, \phi; K) \subset \lambda_n^D(Q; K)$$ to be the submodule spanned by elements $\alpha(Z(T, \phi)_K)$, where:
\begin{itemize}[label = $\circ$]
    \item $\alpha$ lies in $\Test_K(V_{DQ}, \pi, O/\varpi^n)$. 
    \item $T$ lies in $\Sym_2(\Q)_{\geq 0}$, and $Z(T, \phi)_K$ was defined in Construction \ref{constr:Z_T_phi}.
\end{itemize}
\item If $Q = 1$, then we define $\lambda^D(1; K) \subset O$ and $\lambda^D(1,\phi; K)\subset \lambda^D(1; K)$ analogously, where now  $\alpha$ ranges over $\Test_K(V_{D}, \pi, O)$. 
\item We write $\lambda_n^D(Q) \subset O/\varpi^n$ for the submodule spanned by $\lambda_n^D(Q; K)$ as $K$ varies, and likewise $\lambda^D(1)$. 
\end{enumerate}
In all of the above constructions, we include a  subscript $\p$  only when it is necessary for clarity.

\subsubsection{}\label{subsubsec:where_AJ}
Let $D\geq 1$ be squarefree with $\sigma(D)$ \emph{even}, and suppose $\p$ satisfies Assumption \ref{assumptions_on_p}. Let $\m\coloneqq \m^{S\cup \div(D)}_{\pi,\p}\subset \T^{S\cup \div(D)}_O$.
It follows from Lemma \ref{lem:m_is_nonEisgeneric} and Theorem \ref{thm:generic}(\ref{part:thm_generic_two}) that  the \'etale realization map
$$\CH^2 (\Sh_K(V_D),O)_{\m} \to H^{4} _\et(\Sh_K(V_D)_{\overline\Q}, O(2))_\m^{G_\Q}$$ is trivial. We therefore obtain a well-defined Abel-Jacobi map
$$\partial _{\AJ,\m} : \CH^2(\Sh_K(V_D),O)_{\m} \to H^1(\Q, H^3_\et (\Sh_K(V_D)_{\overline\Q}, O(2))_\m).$$

For any $\alpha\in \Test_K(V_D, \pi, R)$ with $R = O$ or $O/\varpi^n$, we obtain an induced map
\begin{equation}\label{eq:in_constr_for_kappa}
    \CH^2(\Sh_K(V_D),O)_{\m} \xrightarrow{\partial_{\AJ, \m}}H^1(\Q, H^3_\et(\Sh_K(V_D)_{\overline\Q}, O(2))_\m ) \xrightarrow{\alpha_\ast} H^1(\Q, T_{\pi}\otimes_O R).
\end{equation}





\end{construction}
\begin{construction}\label{kappa_constr}
Let $D \geq 1$ be squarefree, and let $Q \geq 1$ be admissible and coprime to $D$, such that $\sigma(DQ)$ is \emph{even}. Suppose $\p$ satisfies Assumption \ref{assumptions_on_p}, and let $K$ be an $S$-level structure for $\spin(V_{DQ})$.
\begin{enumerate}
    \item 
We define $$
    \kappa_n^D(Q; K) \subset H^1(\Q, T_{\pi,n})$$
    to be the submodule spanned by $\alpha_\ast \circ \partial_{\AJ, \m} (z)$, where:
    \begin{itemize}[label = $\circ$]
        \item $\alpha$ lies in $\Test_K(V_{DQ}, \pi, O/\varpi^n)$.
        \item $z$ lies in $\SC^2_K(V_{DQ}, O)$ (Notation \ref{notation:SC}).
    \end{itemize}
    \item For any $\phi\in \mathcal S(V_{DQ}^2\otimes \A_f, O)^K$, we define
    $$\kappa_n^D(Q, \phi; K) \subset \kappa_n^D(Q; K)$$
    to be the submodule spanned by elements $\alpha_\ast \circ \partial_{\AJ,\m} (Z(T, \phi)_K)$, where:
    \begin{itemize}[label = $\circ$]
        \item $\alpha$ lies in $\Test_K(V_{DQ}, \pi, O/\varpi^n)$. 
        \item $T$ lies in $\Sym_2(\Q)_{\geq 0}$, and $Z(T, \phi)_K$ was defined in Construction \ref{constr:Z_T_phi}. 
        \end{itemize}
        \item If $Q = 1$, then we define $\kappa^D(1; K) \subset H^1(\Q, T_\pi)$ and $\kappa^D(1, \phi; K) \subset \kappa^D(1; K)$ analogously, with now $\alpha\in \Test_K(V_{D}, \pi, O).$
\item We write $\kappa_n^D(Q) \subset H^1(\Q, T_{\pi,n})$ for the submodule spanned by $\kappa_n^D(Q;K)$ as $K$ varies, and likewise $\kappa^D(1)$. 


\end{enumerate}
\end{construction}
In all of the above constructions, we include a subscript $\p$ only when it is necessary for clarity. 
\begin{rmk}\label{rmk:D_vs_Q_constr}
  The only reason to distinguish between $D$ and $Q$ in Constructions \ref{subsubsec:lambda_constr} and \ref{kappa_constr} is to define $\lambda^D(1)$ and $\kappa^D(1)$ when $\sigma(D)$ is odd and even, respectively; moreover, one can check using Corollary \ref{cor:JL_general} that $\lambda^D(1)$ or $\kappa^D(1)$ is trivial unless $\pi_\l$ is transferrable for all $\l|D$. 
\end{rmk}
Now we prove some basic properties of Constructions \ref{subsubsec:lambda_constr} and \ref{kappa_constr}.
\begin{prop}\label{L values related to lambda elements}
    Suppose $L(1/2,\pi, \operatorname{spin}) \neq 0$.
    If $D > 1$ is squarefree with $\sigma(D)$ odd and 
    $\pi_f^D$ can be completed to an automorphic representation of $\spin(V_D)(\A)$, then for any prime $\p$ of $E_0$, $\lambda^D(1)_\p \neq 0$. Moreover, for all but finitely many $\p$, we have
    $$\lambda^D(1)_\p \not\equiv 0 \pmod{\varpi_\p}.$$
\end{prop}
\begin{proof}
Let $\Pi$ be any automorphic representation of $\spin(V_D)(\A)$ with $\Pi_f^D\simeq \pi^D_f$. 
Because $\pi$ has trivial central character,  $\Pi$ descends to an automorphic representation of $\SO(V_D)(\A)$. 
By \cite[Theorem 1.1]{ginzburg2009poles}, the global theta lift of $\Pi$ to $\MP_4(\A)$ is nonzero; i.e., the map
\begin{equation}\label{eq:theta_pi}\begin{split}
\mathcal S(V_D^2\otimes \A, \C) \otimes \Pi &\to \operatorname{Fun}(\SP_4(\Q) \backslash \MP_4(\A), \C) \\
\phi \otimes \alpha&\mapsto \theta_\phi(\alpha)
\end{split}
\end{equation} is not identically zero (Notation \ref{notation:classical_theta}). Also, the image of (\ref{eq:theta_pi}) lies in the space of cusp forms by the global tower property of the theta lift \cite{rallis1984tower}: otherwise, $\Pi$ would occur in the restriction to $\SO(V_D)(\A)$ of the theta lift of an automorphic representation of $\MP_2(\A)$, which is ruled out by the Shimura-Waldspurger correspondence and the relevance of $\pi$. In particular, by the local-global compatibility of the theta correspondence (see the proof of \cite[Theorem I.2.2]{rallis1984tower}), if $\phi = \otimes\phi_v\in \mathcal S(V_D^2\otimes \A, \C)$ and $\alpha = \otimes \alpha_v\in \Pi$ are factorizable, then
\begin{equation}\label{eq:theta_local_global}\theta_\phi(\alpha) \neq 0 \iff \langle \rho_v(\phi_v), \alpha_v \rangle \neq 0 \;\forall v,\end{equation} where by definition
$$\rho_v: \mathcal S(V_D^2\otimes \Q_v, \C) \to (\Pi_v)^\vee \boxtimes \Theta_v(\Pi_v)$$ is the maximal $(\Pi_v)^\vee$-isotypic quotient of the Weil representation of $\SO(V_D)(\Q_v) \times \MP_4(\Q_v)$ on $\mathcal S(V_D^2\otimes \Q_v, \C)$. 

When $v= \infty$, then $\alpha_\infty \in \Pi_\infty$ is unique up to scalar, and we take $\phi_\infty\in \mathcal S(V_D^2\otimes\R, \C)$ to be the Gaussian from Lemma \ref{theta fourier coeffs}. Then $\alpha_\infty$ and $\phi_\infty$ satisfy the local condition in (\ref{eq:theta_local_global}) by the theory of joint harmonics \cite[Proposition 2.1(2), \S5]{adams1998genuine}. 
In particular, (\ref{eq:theta_local_global}) implies that we can fix an $S$-level structure $K$ and data $$\alpha \in \Hom_{\operatorname{set}}(\Sh_K(V_D), O_{E_0}) \cap \Pi \subset C_c^\infty(\spin(V_D)(\A), \C),\;\; \phi_f\in \mathcal S(V_D^2\otimes \A_f, O_{E_0})^K$$ such that $$\theta_{\phi_f\otimes\phi_\infty} (\alpha) \neq 0.$$
By Lemma \ref{theta fourier coeffs}, this means that $$0 \neq \alpha (Z(T, \phi_f)_K)\in O_{E_0}$$ for some $T\in \Sym_2(\Q)_{\geq 0}$. 
Now we note that $\alpha$ has the Hecke eigenvalues of $\pi$ away from $S$ by construction. In particular, for all primes $\p$ of $E_0$, $\alpha$ lies in $\Test_K(V_D, \pi, O_\p)$, 
and $$0 \neq \alpha(Z(T, \phi_f)_K) \in \lambda^D(1, \phi_f)_\p \subset \lambda^D(1)_\p.$$ 
Since $\alpha(Z(T,\phi_f)_K) \not\equiv 0\pmod \p$ for all but finitely many $\p$, this completes the proof. 
\end{proof}

\begin{prop}\label{prop: check local conditions for kappa}
Suppose $D\geq 1$ is squarefree and $\p$ satisfies Assumption \ref{assumptions_on_p}.

\begin{enumerate}
    \item \label{part:check_local_one}  For all admissible $Q$ coprime to $D$ with 
    $\sigma(DQ)$ even and all
    $\l\not\in S \cup \div(DQ)$, we have
    $$\Res_\l \kappa _n ^D(Q) \subset H^1_{f}(\Q_\l, T_{\pi,n}).$$
\item\label{part:check_local_two} If $\sigma(D)$ is even, then we have
$$\kappa^D(1) \subset H^1_f(\Q, T_\pi).$$
\end{enumerate}
\end{prop}
\begin{proof}
Write $Q = 1$ in case (\ref{part:check_local_two}).
We claim that, for all $S$-level structures $K$ and all  $z\in \CH^2(\Sh_K(V_{DQ}))$, \begin{equation}\label{eq:AJ_in_f}\partial_{\AJ, \m} (z) \in H^1_f(\Q, H^3_\et(\Sh_K(V_{DQ})_{\overline \Q}, O(2))_\m).\end{equation}
Note here that $H^3_\et(\Sh_K(V_{DQ})_{\overline \Q}, O(2))_\m$ is $p$-torsion-free by Lemma \ref{lem:m_is_nonEisgeneric} and Theorem \ref{thm:generic}(\ref{part:thm_generic_two}), so the Bloch-Kato Selmer group is defined as in Notation \ref{notation:BK_etc}(\ref{notation:BK_part_tf}). 
Observe as well that (\ref{eq:AJ_in_f}) implies the proposition: indeed, it clearly implies (\ref{part:check_local_two}) by the functoriality of the Bloch-Kato local conditions; and it also implies  (\ref{part:check_local_one}) by Remark \ref{rmk:compare_f_unr_BK} and 
 Proposition \ref{prop torsion crystalline iff crystalline}.

Now note that, for all primes $\l$, 
$H^3_\et(\Sh_K(V_{DQ})_{\overline \Q}, E_{0,\p}(2))_\m$ is pure of weight one as a $G_{\Q_\l}$-module by Corollary \ref{cor:coh_relevant} combined with Theorem \ref{thm:rho_pi_LLC}(\ref{part:rho_pi_LLC1}); hence $$H^1\left(\Q_\l,  H^3_\et(\Sh_K(V_{DQ})_{\overline \Q}, E_{0,\p}(2))_\m\right) = 0$$ for all $\l\neq p$
and $$H^1_g\left(\Q_p, H^3_\et(\Sh_K(V_{DQ})_{\overline \Q}, E_{0,\p}(2))_\m\right) = H^1_f\left(\Q_p, H^3_\et(\Sh_K(V_{DQ})_{\overline \Q}, E_{0,\p}(2))_\m\right).$$
Thus (\ref{eq:AJ_in_f}) follows from 
\cite[Theorem 5.9]{nekovar2016syntomic} combined with the proof of \cite[Theorem 3.1(ii)]{nekovar2000banff}. 
    \end{proof}
\begin{lemma}\label{lem:tidy_suffices}
Let $D \geq 1$ be squarefree, and let $Q \geq 1$ be admissible and coprime to $D$. Then:
\begin{enumerate}
    \item If $\sigma(DQ)$ is even and $\p$ satisfies Assumption \ref{assumptions_on_p}, then $\kappa_n^D(Q)$ is generated by $\kappa_n^D(Q;K)$ as $K$ ranges over $S$-tidy level structures for $\spin(V_{DQ})$.
    \item If $\sigma(DQ)$ is odd and $p > 5$, then $\lambda_n^D(Q)$ is generated by $\lambda_n^D(Q;K)$ as $K$ ranges over $S$-tidy level structures for $\spin(V_{DQ})$.
\end{enumerate}
\end{lemma}
\begin{proof}
    We prove the first statement, as the two are similar. Let $K$ be an $S$-level structure for $\spin(V_{DQ})$, and note that $K_2 = K_{\l = 2}$  has pro-order prime to $p$.\footnote{To see this, one first observes that $K_{2}$ stabilizes some lattice $L \subset V_{DQ}\otimes \Q_{2}$ such that $2L \subset L^\vee \subset L$, where $L^\vee$ is the $\Z_{2}$-linear dual.
    One obtains a  natural map $f: K_2 \to \SO(L^\vee/2L) \times \SO(L/L^\vee)$, where $L^\vee/2L$ and $L/L^\vee$ are naturally nondegenerate symmetric spaces over $\F_2$ of dimension at most 5. Since  $p > 5$ by Remark \ref{rmk:p>5}, the image of $f$ then has order prime to $p$, and the kernel of $f$ is clearly pro-$2$.}    
After fixing a sufficiently small compact open subgroup $K_{2}' \subset K_{2}$, we can ensure that $K' \coloneqq K_{2}'\cdot \prod_{\l\neq 2} (K_\l \Z_\l^\times)$ is neat; it is then clearly $S$-tidy as well.  Let $\pr_{K,K'\cap K} : \Sh_{K'\cap K}(V_{DQ}) \to \Sh_K (V_{DQ})$
and $\pr_{K', K'\cap K} : \Sh_{K'\cap K} (V_{DQ}) \to \Sh_{K'}(V_{DQ})$ be the natural maps. 
For $R= O$ or $O/\varpi^n$ and  $\alpha_K\in \Test_K(V_{DQ}, \pi, R)$, we set 
$$\alpha_{K'} \coloneqq \frac{\alpha_K \circ 
\pr_{K, K'\cap K, \ast} \circ \pr_{K', K' \cap K}^\ast}{[K: K' \cap K]}\in \Test_{K'}(V_{DQ}, \pi, R) .$$ If $\phi\in \mathcal S(V_{DQ}^2\otimes \A_f, O)$ is fixed by $K$, it is also fixed by $K'$, and we have $\alpha_{K',\ast} \circ \partial_{\AJ,\m} (Z(T, \phi)_{K'}) = \alpha_{K,\ast}\circ \partial_{\AJ,\m} (Z(T, \phi)_K)$ by Proposition \ref{prop:pullback_compatible_Z}; the lemma follows.
\end{proof}

    Finally, we introduce some notation that will be used in the endoscopic case.
    \begin{notation}
        Suppose $\pi$ is endoscopic associated to a pair $(\pi_1,\pi_2)$.         
        Then for any squarefree $D\geq 1$ with $\sigma(D)$ \emph{even}, any $S$-level structure $K$ for $\spin(V_D)$, any $\phi \in \mathcal S(V_D^2\otimes \A_f, O)^K$, and $j = 1$ or 2,  we define $\kappa^D(1, \phi; K)^{(j)}\subset H^1_f(\Q, T_{\pi_j})$
        as the image of $\kappa^D(1)$ under the natural projection 
        $H^1_f(\Q, T_{\pi}) \to H^1_f(\Q, T_{\pi_j}).$ We similarly write $\kappa^D(1; K)^{(j)}$ and $\kappa^D(1)^{(j)}$.
        It is easy to check 
that $$\kappa^D(1, \phi; K) = \kappa^D(1, \phi; K)^{(1)} \oplus \kappa^D(1, \phi; K)^{(2)},$$ etc.  As usual, a subscript $\p$ is included when necessary for clarity. 

\end{notation}

\section{Nonvanishing criteria for changing test functions}\label{sec:rep_theory}
\subsection{Setup and notation}

\subsubsection{}
Fix a prime $q$, and let $k = \C$ or $\overline\F_p$, for an odd prime $p \neq q$. In the latter case we also assume fixed an isomorphism $\iota: \overline\Q_p \isomorphism \C$, and assume throughout this section that\begin{equation}\label{week banal equation}
q-1 \neq 0 \text{ in } k.
\end{equation}
We denote by $|\cdot|^{1/2}: \Q_q^\times \to k^\times$ the unramified character such that $|q|^{1/2} = q^{-1/2}\in k^\times$, using the isomorphism $\iota$ when $k = \overline\F_p$.

\subsubsection{}\label{local notation for split space}
Let $V_m $
be the split quadratic space of dimension $2 m +1 $
over $\Q_q $, with basis $v_0, v_1,\ldots, v_m, v_1 ^\ast,\ldots, v_m ^\ast$
and pairing given by:
$$v_i \cdot v_j^\ast = \delta_{ij}, \;\;v_i^\ast \cdot v_j^\ast = 0,  \;\;
v_i \cdot v_j = \delta_{i0}\delta_{j0}.$$
Then $L\coloneqq \Span_{\Z_q}\set{v_0, v_1,\ldots, v_m, v_1 ^\ast,\ldots, v_m ^\ast}$ is a self-dual lattice $L\subset V_m $.
Abbreviate\begin{equation} G_m =\SO (V_m) (\Q_q),\;\; G_n' =\MP_{2 n} (\Q_q).\end{equation}
For any parabolic subgroup $P \subset G_m$ (resp. $P' \subset G_n'$), we write $R_P$  (resp. $R_{P'}$) for the normalized Jacquet module functor with respect to $P$ (resp. $P'$).
\subsubsection{}
Fix an integer $n\geq 1 $,
and consider the Weil representation on $\mathcal S(V^n, \C)$ with respect to the localization, also written $\psi$,
of our fixed global additive character 
of $\Q $ (Notation \ref{notation:where_psi}).   If $k = \overline \F_p$, we have fixed an isomorphism $\iota: \overline\Q_p \isomorphism \C$, and 
it follows from the discussion in (\ref{subsubsec:coeffs_for_Weil}) that $\mathcal S (V ^ n,\Breve\Z_p)\subset \mathcal S(V^n,\overline\Q_p) $ 
is stable under $G_m\times G_n'$; reducing modulo $p$, we obtain the Weil representation on $\mathcal S (V ^ n,\overline\F_p) $.
Whether $k = \C$ or $\overline\F_p$, we abbreviate
 $\Omega_{m,n}= \mathcal S(V^n, k)$.

\subsubsection{} If $\pi $
is an irreducible, admissible $k $-linear representation of $G_m $,  define $\Theta_{m,n} (\pi)$ to be the  $k[G_n']$-module such that $$\Omega_{m,n} \twoheadrightarrow\pi \boxtimes \Theta_{m,n}(\pi)$$ is the maximal $\pi$-isotypic quotient of $\Omega_{m,n}$. 
Similarly, if $\pi' $
is an irreducible, admissible, genuine $k$-linear representation of $G_n' $, define $\Theta_{n, m} (\pi')$ to be the $k[G_m]$-module such that $$\Omega_{m,n} \twoheadrightarrow\Theta_{n,m}(\pi')\boxtimes \pi'$$
is the maximal $\pi'$-isotypic quotient of $\Omega_{m,n}$. 
\subsubsection{}
Let \begin{equation}\widetilde\GL_1 (\Q_q)\xrightarrow {g\mapsto\overline g}\Q_q ^\times\end{equation}
be the double cover described in \cite[p. 1661]{gan2012metaplectic}, with canonical genuine character $\chi_\psi: \widetilde \GL_1 (\Q_q) \to \mu_8 \subset k^\times.$


\subsection{ The structure of the Weil representation}
Consider the Schwartz spaces $\mathcal S (\Q_q,k) $
and $\mathcal S (\Q_q ^\times,k) $, viewed as representations of $\Q_q ^\times\times\Q_q ^\times $
via $$(g_1, g_2)\cdot\phi (x) =\phi (g_1 ^ {-1} x g_2). $$
\begin{lemma}\label{preliminary calculation for theta}
Fix a character $\chi:\Q_q ^\times\rightarrow k ^\times $.
\begin{enumerate}
\item\label{preliminary calculation for theta part one} The maximal quotient of $\mathcal S (\Q_q ^\times,k) $
on which the first factor of $\Q_q ^\times $
acts by $\chi $
is realized by the map
\begin{align*}
\mathcal S (\Q_q ^\times,k) &\rightarrow\chi\boxtimes\chi ^ {-1}\\
\phi &\mapsto\integral_{\Q_q ^\times}\phi (t)\chi (t)\d ^\times t.
\end{align*}
\item\label{preliminary calculation for theta part two} Assume $\chi $
is nontrivial. Then  the map in (\ref{preliminary calculation for theta part one}) extends uniquely to a map $f_\chi: \mathcal S (\Q_q,k)\rightarrow\chi\boxtimes\chi ^ {-1} $
via
$$\phi\mapsto\frac {1} {1 -\chi (g_0)}\integral_{\Q_q ^\times}\left (\phi (t) -\phi (g_0 ^ {-1} t)\right)\chi (t)\d ^\times t, $$
where $g_0\in\Q_q ^\times $
is any element such that $\chi (g_0)\neq 1. $
\item \label{preliminary calculation for theta part three}For integers $a\leq b$, let $S_{a,b} \subset \mathcal S(q^a \Z_q, k)$ be the subspace of Schwartz functions that are invariant under multiplication by $\Z_q^\times$ and translation by $q^b\Z_q$. Then for distinct, nontrivial, and unramified characters $\chi_1,\ldots, \chi_m: \Q_q^\times \to k^\times$ with $m \leq b - a + 1$, $f_{\chi_1},\ldots, f_{\chi_m}$ are linearly independent as functions on $S_{a,b}$. 
\end{enumerate}
\end{lemma}
\begin{proof}
Part (\ref{preliminary calculation for theta part one}) is elementary. 
For (\ref{preliminary calculation for theta part two}), we have the exact sequence
$$0\rightarrow \mathcal S (\Q_q ^\times,k)\rightarrow \mathcal S (\Q_q,k)\xrightarrow {\phi\mapsto\phi (0)}k\rightarrow 0, $$
which is equivariant for the trivial $\Q_q ^\times\times\Q_q ^\times $-action on $k $. Since $$\Home_{\Q_q ^\times\times\Q_q ^\times} (k,\chi\boxtimes\chi ^ {-1}) = 0, $$
there is at most one extension of the map in (\ref{preliminary calculation for theta part one}) to $\mathcal S (\Q_q,k) $, and the formula given in (\ref{preliminary calculation for theta part two}) exhibits it.

For (\ref{preliminary calculation for theta part three}), let $x_i = \chi_i(q)$.  A direct calculation shows that, for $\phi\in S_{a,b}$, 
\begin{equation*}
\begin{split}
    f_{\chi_i} (\phi)= \frac{1}{1 - x_i} \operatorname{vol} (\Z_q^\times) \bigg(\phi(q^a)(x_i^a - x_i^{a+1}) + \phi (q^{a+1}) (x_i^{a+1} - x_i^{a+2}) \\ +\ldots + \phi(q^{b-1}) (x_i^{b-1} - x_i^b) + \phi(q^b) x_i^b\bigg).\end{split}
\end{equation*}
So it suffices to show that the matrix
\begin{equation*}
    \begin{pmatrix}
    x_1^a - x_1^{a+1} & \cdots & x_1^{b-1} - x_1^b & x_1^b \\ 
    x_2^a - x_2^{a+1} & \cdots & x_2^{b-1} - x_2^b &  x_2^b \\
     \vdots & \ddots & \vdots & \vdots  \\
     x_m^a - x_m^{a+1} & \cdots & x_m^{b-1} - x_m^b &  x_m^b 
    \end{pmatrix}
\end{equation*}
is nondegenerate, and this follows from the Vandermonde determinant formula.

\end{proof}
\subsubsection{}
Assuming that $m\geq 1 $, let $$P = MN\subset G_m $$
be the parabolic subgroup stabilizing the isotropic line $\langle v_1\rangle $. Then $M$
is isomorphic to $\Q_q ^\times\times G_{m -1} $; we normalize the isomorphism such that
\begin{equation} (\alpha, g)\cdot v_1 =\alpha v_1,\;\text {for}\; (\alpha, g)\in M.
\end{equation}
Similarly, let $P' = M' N'\subset G_n' $
be the preimage of the stabilizer of $e_1\in W_{2 n} $ (Notation \ref{notation:symplectic}(\ref{notation:symplectic_W})); then $M'$ is isomorphic to $\widetilde\GL_1 (\Q_q)\times_{\mu_2} G'_{n -1}. $
We normalize the isomorphism so that
\begin{equation}
(\alpha, g)\cdot e_1 =\overline\alpha e_1,\;\text {for}\; (\alpha, g)\in M'.
\end{equation}

\begin{lemma}\label{jacquet module filtration theta}
\leavevmode
\begin{enumerate}
\item 
The normalized Jacquet module $R_P (\Omega_{m,n}) $
fits into an exact sequence
$$0\rightarrow\Induction_{M\times P'} ^ {M\times G_n'}\chi_\psi\cdot \mathcal S (\Q_q ^\times,k)\boxtimes\Omega_{m-1,n -1}\rightarrow R_P (\Omega_{m, n})\rightarrow |\cdot | ^ {\frac {2n -2 m +1} {2}}\Omega_{m-1, n}\rightarrow 0, $$
where $|\cdot | $
is the canonical character of $M $, and $\chi_\psi\cdot \mathcal S (\Q_q ^\times,k) $
is a $\GL_1 (\Q_q)\times\widetilde\GL_1 (\Q_q) $-module with action defined by $(g, h)\cdot\phi (t) =\chi_\psi (h)\phi (g ^ {-1} t\overline h). $
\item Similarly, the normalized Jacquet module $R_{P'} (\Omega_{m,n}) $
fits into a canonical exact sequence
$$0\rightarrow\Induction_{P\times M'} ^ {G_m\times M'}\chi_\psi\cdot \mathcal S (\Q_q ^\times,k)\boxtimes\Omega_{m-1,n -1}\rightarrow R_{P'} (\Omega_{m, n})\rightarrow\chi_\psi |\cdot | ^ {\frac {2 m -2 n +1} {2}}\boxtimes\Omega_{m,n -1}\rightarrow 0. $$
\end{enumerate}
\end{lemma}
\begin{proof} When $k = \C$, this is  \cite[Theorem 2.8]{kudla1986local}; see also \cite[Proposition 7.3]{gan2012metaplectic}  for our more convenient normalizations. When $k = \overline\F_p$, the proof in \cite{kudla1986local} applies without change because $p \neq q$. 
\end{proof}
\begin {corollary}\label{induced theta corollary}
\leavevmode
\begin{enumerate}
\item \label{cor:induced_theta_one}Let $\pi_{m -1} $
be an irreducible admissible representation of $G_{m -1} $, and let $\chi_0: \Q_q^\times \to k^\times$ be a character with
$\chi_0\neq |\cdot | ^ {\frac {2 n -2 m +1} {2}} .$
Then for all admissible  $k[G'_n]$-modules $\mathcal M$,  $$\Home_{G_m\times G_n'}\left (\Omega_{m, n},\left (\Induction_P ^ {G_m}\chi_0\boxtimes\pi_{m -1}\right)\boxtimes \mathcal M\right) =\Home_{G_n'}\left (\Induction_{P'} ^ {G_n'}\chi_\psi\cdot\chi_0 ^ {-1}\boxtimes\Theta_{m -1, n -1} (\pi_{m -1}),\mathcal  M\right). $$
In particular, if $$\pi_m\coloneqq\Induction_P ^ {G_m}\chi_0\boxtimes\pi_{m -1} $$
is irreducible, then $$\Theta_{m,n} (\pi_m) =\Induction_{P'} ^ {G_n'}\chi_\psi\cdot\chi_0 ^ {-1}\boxtimes\Theta_{m -1, n -1} (\pi_{m -1}). $$
\item\label{cor:induced_theta_two} Similarly, let $\pi'_{n -1} $
be an irreducible admissible genuine representation of $G'_{n -1} $
and $\chi_0\neq |\cdot | ^ {\frac {2 m -2 n +1} {2}} $
a character of $\Q_q ^\times $.
Then for all $k[G_m]$-modules $\mathcal M $, $$\Home_{G_m\times G_n'}\left (\Omega_{m, n},\mathcal M \boxtimes \left (\Induction_{P'} ^ {G_n'}\chi_\psi\cdot\chi_0\boxtimes\pi_{n -1}'\right)\right) =\Home_{G_m}\left (\Induction_{P} ^ {G_m}\chi_0 ^{-1}\boxtimes\Theta_{m -1, n -1} (\pi_{n -1}'), \mathcal M\right). $$
In particular, if $$\pi'_n\coloneqq\Induction_{P'} ^ {G_n'}\chi_\psi\chi_0\boxtimes\pi_{n -1}' $$
is irreducible, then $$\Theta_{n, m} (\pi'_n) =\Induction_{P} ^ {G_m}\chi_0 ^{-1}\boxtimes\Theta_{n-1,m -1} (\pi'_{n -1}). $$
\end{enumerate}
\end{corollary}
\begin{proof}
This is immediate from Lemma \ref{jacquet module filtration theta} and Lemma \ref{preliminary calculation for theta}(\ref{preliminary calculation for theta part one}).
\end {proof}

\subsection{Principal series over $k$ for orthogonal and metaplectic groups}
Some of the arguments in this subsection were inspired by the work of Zorn \cite{zorn2010reducibility}. The results are new only if $k = \overline \F_p$. \subsubsection{} Continuing the notation of (\ref{local notation for split space}), we now assume $m = n$. 
\begin{notation}
\leavevmode
\begin{enumerate}
\item Let $B\subset G_n$
be the stabilizer of the maximal isotropic flag
\begin{equation}\label {orthogonal isotropic flag}
0\subset\langle v_1\rangle\subset\langle v_1, v_2\rangle\subset\cdots\subset\langle v_1,\ldots, v_n\rangle
\end{equation}
and let $B'\subset G'_n $
be the preimage of the stabilizer of the maximal isotropic flag
\begin{equation*}
0\subset\langle e_1\rangle\subset\langle e_1, e_2\rangle\subset\cdots\subset\langle e_1,\ldots, e_n\rangle.
\end{equation*}
\item The Levi factor $T $
of $B $
is identified with $(\Q_q ^\times) ^ n $
via its action on the associated graded of (\ref {orthogonal isotropic flag}), and similarly the Levi factor $T' $
of $B' $
is identified with
$$\underbrace {\widetilde\GL_1 (\Q_q)\times_{\mu_2}\cdots\times_{\mu_2}\widetilde\GL_1 (\Q_q)}_{n\;\text {times}}. $$
\item 
For any character $$\chi: (\Q_q ^\times) ^ n\rightarrow k ^\times, $$
we define (normalized) principal series representations
\begin{equation*}
I (\chi) =\Induction_B ^ {G_n}\chi,\;\; I' (\chi) =\Induction_{B'} ^ {G_n'}\chi_\psi\cdot\chi.
\end{equation*}
\item We write $W= S_n \rtimes (\Z/2\Z)^n$ for the Weyl group of $T$ in $G_n$, which is also the Weyl group of $T'$ in $G_n'$.
\end{enumerate}
\end{notation}
\begin{lemma}\label{jacquet module of induced}
The semi-simplified normalized Jacquet modules are $$R_B (I (\chi)) ^ {\ss} =\oplus_{w\in W}\chi ^ w,\;\;\; R_{B'} (I' (\chi)) ^ {\ss} =\oplus_{w\in W}\chi ^ w\chi_\psi.$$
\end{lemma}
\begin{proof}
Over $\overline\Q_p $, this follows from the well-known result for $\C $; see \cite[Lemma 4.8]{zorn2010reducibility}. Since $p\neq q $, the $p $-modular case follows by the proof of \cite[Lemme 34]{vigneras1989gl2}.
\end{proof}
\begin{lemma}\label{quotients are intertwining operators}
\leavevmode
\begin{enumerate}
\item Any nontrivial quotient $I (\chi)\twoheadrightarrow\pi $
extends to a nontrivial intertwining operator $$I (\chi)\twoheadrightarrow\pi\rightarrow I (\chi ^ w) $$
for some $w\in W $.
\item Similarly, any nontrivial quotient $I' (\chi)\twoheadrightarrow\pi' $
extends to a nontrivial intertwining operator $$I' (\chi)\twoheadrightarrow\pi'\rightarrow I' (\chi ^ w)$$
for some $w\in W$.
\end{enumerate}
\end{lemma}
\begin{proof}
Let $\overline B\subset G_n $
and $\overline B'\subset G_n' $
be the opposite Borel subgroups to $B $, $B' $. Recall the ``second Frobenius reciprocity''
\begin{equation}\label{second frobenius reciprocity}
\Home (I (\chi),\pi) =\Home (\chi, R_{\overline B} (\pi)),\;\;\;\Home (I' (\chi),\pi') =\Home (\chi_\psi\chi, R_{\overline B'} (\pi')).
\end {equation}
In the orthogonal case, this is \cite[II.3.8(2)]{vigneras1996representations}. Although the result there is only stated for reductive groups, the proof can be adapted verbatim to the metaplectic case. (The key technical points are the existence of arbitrarily small compact open subgroups admitting an Iwahori factorization, and the conditions in Lemma I.8.13 of \emph{op. cit.} All of the fundamental results on Hecke algebras in Chapter I of \emph{op. cit.} apply to general locally profinite groups.)

Returning to the proof of the lemma, note that $$R_{\overline B} (\pi)\neq 0\IFF R_B (\pi)\neq 0 $$
and likewise in the metaplectic case. Hence it follows from (\ref{second frobenius reciprocity}) that any quotient $I (\chi)\twoheadrightarrow\pi $
extends to a nontrivial map $$I (\chi)\to\pi\rightarrow I (\rho )$$
for some character $\rho $
of $(\Q_q ^\times) ^ n $; by Lemma \ref{jacquet module of induced}, $\rho $
is a Weyl conjugate of $\chi $, so (1) holds. The argument for (2) is identical.
\end {proof}
\begin{lemma}\label{rank one irreducible inductions}
Suppose $n = 1 $ and $\chi ^ 2\neq |\cdot | ^ {\pm 1}. $
Then $I (\chi) $
and $I' (\chi) $
are both irreducible and $\Theta_{1, 1} (I (\chi)) = I' (\chi). $
\end{lemma}
\begin{proof}
First of all, $I (\chi) $
is irreducible by \cite[Th\'eor\`eme 3]{vigneras1989gl2}. Then $$\Theta_{1, 1} (I (\chi)) = I' (\chi ^ {-1}) $$
by Corollary \ref{induced theta corollary}(\ref{cor:induced_theta_one}). Since the intertwining map $I (\chi)\rightarrow I (\chi ^{-1}) $
is an isomorphism by irreducibility, we have
$$I' (\chi ^{-1})\cong\Theta_{1, 1} (I (\chi))\cong\Theta_{1, 1} (I (\chi ^{-1}))\cong I' (\chi). $$
If $\chi ^ 2\neq \blackboardone $, this shows $I' (\chi) $
is irreducible by Lemma \ref{jacquet module of induced} and Lemma \ref{quotients are intertwining operators}. If $\chi ^ 2 = \blackboardone $, we instead use Corollary \ref{induced theta corollary}(\ref{cor:induced_theta_one},\ref{cor:induced_theta_two}) to obtain
\begin{align*}
\dimension\Home_{G_1'} (I' (\chi), I' (\chi)) & =\dimension\Home_{G_1\times G_1'} (\Omega_{1, 1}, I (\chi)\boxtimes I' (\chi))\\
& =\dimension\Home_{G_1} (I (\chi), I (\chi))\\
& = 1
\end{align*}
since $I (\chi) = I (\chi ^ {-1}) $
is irreducible. This shows that the intertwining operator $I' (\chi)\to I' (\chi) $
is unique, so the lemma follows from Lemma \ref{quotients are intertwining operators}.
\end{proof}
\begin{lemma}\label{jacquet module multiplicity}

Let $\chi =\chi_1\boxtimes\cdots\boxtimes\chi_n $
be a character such that $\chi_i ^ 2 =\blackboardone $
for some $1\leq i\leq n $. Then for any submodule $\pi\subset I (\chi) $, we have $$\chi ^ {\oplus 2}\subset R_B (\pi) ^ {\ss}. $$
Similarly, for any submodule $\pi'\subset I' (\chi) $, we have $$\chi_\psi\chi ^ {\oplus 2}\subset R_{B'} (\pi') ^ {\ss}. $$
\end{lemma}
\begin{proof}
The orthogonal and metaplectic cases are identical, so we just prove the result for $\pi'\subset I' (\chi). $
Let $Q_i = L_iU_i\subset G_n' $
be the rank one standard parabolic subgroup corresponding to the $i $th long root of $T' $. Then $$L_i\cong\underbrace {\widetilde\GL_1 (\Q_q)\times_{\mu_2}\cdots\times_{\mu_2}\widetilde\GL_1 (\Q_q)}_{i -1\;\text {times}}\times_{\mu_2} G_1'\times_{\mu_2}\underbrace {\widetilde\GL_1 (\Q_q)\times_{\mu_2}\cdots\times_{\mu_2}\widetilde\GL_1(\Q_q)}_{n - i\;\text {times}}. $$
Let $$\rho =\chi_\psi\cdot\left (\chi_1\boxtimes\cdots\boxtimes\chi_{i -1}\right),\;\;\sigma =\chi_\psi\cdot\left (\chi_{i +1}\boxtimes\cdots\boxtimes\chi_n\right). $$

Since $\pi' $
admits a nonzero map to $I' (\chi) $, $\chi $
is a quotient of $R_{B'} (\pi') = R_{B'\intersection L_i} R_{Q_i} (\pi'), $
so we have a nontrivial intertwining operator $$R_{Q_i} (\pi')\rightarrow\Induction_{B'\cap L_i} ^ {L_i}\chi_\psi\chi =\rho\boxtimes I' (\chi_i)\boxtimes\sigma. $$
This is surjective since $I' (\chi_i) $
is irreducible by Lemma \ref{rank one irreducible inductions} above. (By (\ref{week banal equation}), we have $\chi_i ^ 2 =\blackboardone \neq |\cdot | ^ {\pm 1}. $)
By the exactness of the Jacquet functor, we conclude $$R_{B'} (\pi') ^ {\ss}\twoheadrightarrow R_{B'\intersection L_i} (\rho\boxtimes I' (\chi_i)\boxtimes\sigma) ^ {\ss} =\chi_\psi\chi ^ {\oplus 2}. $$
\end{proof}


\begin{lemma}\label{home space dimension one lemma}
Suppose $\chi=\chi_1\boxtimes\cdots\chi_n:(\Q_q^\times)^n\to k^\times $ is a character such that $\chi_i$ are all distinct and $\chi_i^2=\mathbbm 1 $ for at most one $1\leq i\leq n$. Then $$\dimension\Home (I (\chi), I (\chi ^ w)) =\dimension\Home (I' (\chi), I' (\chi ^ w)) = 1 $$
for all $w\in W $.
\end{lemma}
\begin{proof}
If the Weyl conjugates of $\chi$ are all distinct, this is automatic from Lemma \ref{jacquet module of induced}; 
so assume without loss of generality that $\chi_i ^ 2 =\blackboardone $
for exactly one $1\leq i\leq n $. Then the stabilizer of $\chi$ in $W$ has order exactly two.
The argument is the same for $I (\chi) $
and $I' (\chi) $, so we consider $I (\chi) $
in order to minimize notation. 
\begin {claim}
Any nontrivial map $f: I (\chi)\rightarrow I (\chi) $
is an isomorphism.
\end {claim}
\begin{proof}[Proof of claim]
Indeed, if $f $
has nontrivial kernel, then $R_B (\kernel f) ^ {\ss} $
contains $\chi ^ {\oplus 2} $
by Lemma \ref{jacquet module multiplicity}. But $\chi $
appears with multiplicity exactly two in $R_B (I (\chi)) ^ {\ss} $
by Lemma \ref{jacquet module of induced}, so then $R_B (\image f) ^ {\ss} $
does not contain $\chi $, which is impossible since we are assuming that $f $
is nontrivial. Hence $f $
is injective, so $$R_B (f): R_B (I (\chi))\rightarrow R_B (I (\chi)) $$
is injective. Since $R_B (I (\chi)) $
has finite length, $R_B (f) $
is also surjective. But then $\cokernel f $
has trivial Jacquet module, which means $\cokernel f = 0 $
by Lemma \ref{quotients are intertwining operators}. So $f $
is also surjective.
\end{proof}
Now by the claim, $\End(I(\chi))$ is a division algebra over $k$, and 
also a $k $-vector space of dimension $$\dimension\Home (I (\chi), I (\chi))\leq 2.$$ Since
 $k$ is algebraically closed, we conclude\begin{equation}\label{eq:homsp_I_chi}\dimension\Home (I (\chi), I (\chi)) = 1.\end{equation}

Next observe that there are no non-split extensions of distinct characters of $T$ over $k $.
In particular, we may decompose $$R_B (I (\chi)) =\bigoplus_{\chi ^ w} R_B (I (\chi))_{\chi ^ w} ,$$
where $\chi ^ w $
runs over the (distinct) Weyl conjugates of $\chi $, and (\ref{eq:homsp_I_chi}) implies that $R_B (I (\chi))_\chi $
is a non-split extension of $\chi $
by $\chi $. The same argument applies to show $R_B (I (\chi ^ w))_{\chi ^ w} $
is a non-split extension of $\chi ^ w $
by $\chi ^ w $
for all $w\in W$.
Now for any nonzero intertwining map $$f: I (\chi)\rightarrow I (\chi ^ w), $$
$R_B (\image f) ^ {\ss} $
contains $(\chi ^ w) ^ {\oplus 2} $
by Lemma \ref{jacquet module multiplicity}. Hence the induced map $$R_B (f): R_B (I (\chi))_{\chi ^ w}\rightarrow R_B (I (\chi ^ w))_{\chi ^ w} $$
is surjective, in particular an isomorphism since both sides have dimension two over $k $.
Since $R_B (I (\chi ^ w))_{\chi ^ w} $
is non-split, so is $R_B (I (\chi))_
{\chi ^ w}. $
Hence $$\dimension\Home (I (\chi), I (\chi ^ w)) =\dimension\Home (R_B (I (\chi)),\chi ^ w) = 1 $$
for all $w\in W $.
\end{proof}
\subsubsection{}\label{weyl group generators}
To state the next lemma, we use the following explicit generators for 
the Weyl group $W $:
\begin{enumerate}[label = (\roman*)]
\item The inversion $s $
sending a character $\chi =\chi_1\boxtimes\cdots\boxtimes\chi_n $
to $\chi ^ s =\chi_1 ^{-1}\boxtimes\cdots\boxtimes\chi_n. $
\item For $1\leq i <n $, the transposition $w_i $
sending a character $\chi =\chi_1\boxtimes\cdots\boxtimes\chi_n $
to $\chi ^ {w_i} =\chi_1\boxtimes\cdots\boxtimes\chi_{i -1}\boxtimes\chi_{i +1}\boxtimes\chi_{i}\boxtimes\chi_{i +2}\boxtimes\cdots\boxtimes\chi_n $.
\end{enumerate}
\begin{lemma}\label{flippy induction lemma}
Let $\chi =\chi_1\boxtimes\cdots\boxtimes\chi_n: (\Q_q ^\times) ^ n\rightarrow k ^\times $
be a character.
\begin{enumerate}
\item Suppose $\chi_1 ^ 2\neq |\cdot | ^ {\pm 1} $. Then $$I (\chi)\cong I (\chi ^ s)\;\;\text {and}\;\; I' (\chi)\cong I' (\chi ^ s). $$
\item Suppose $\chi_i/\chi_{i +1}\neq |\cdot | ^ {\pm 1} $
for some $1\leq i <n $. Then $$I (\chi)\cong I (\chi ^ {w_i})\;\;\text {and}\;\; I' (\chi)\cong I' (\chi ^ {w_i}). $$
\end{enumerate}
\end{lemma}
\begin{proof}
Once again, the orthogonal and metaplectic cases are identical. We give the proof in the orthogonal case.
For (1), let $P = MN\subset G_n $
be the rank one standard parabolic subgroup such that the Weyl group of $M $
is generated by $s $. Then $$M\cong G_1\times (\Q_q) ^ {n -1}. $$
Also write $\rho =\chi_2\boxtimes\cdots\boxtimes\chi_n: (\Q_q ^\times) ^ {n -1}\rightarrow k ^\times. $
Then
\begin{equation}\label {1st induction equation proof of while invariance Lemma}
I (\chi) =\Induction_P ^ {G_n}\Induction_{B\intersection M} ^ {M}\chi =\Induction_P ^ {G_n} I (\chi_1)\boxtimes\rho
\end{equation}
and similarly
\begin{equation}\label {2nd induction equation proof of while invariance Lemma}
I (\chi ^ s) =\Induction_P ^ {G_n} I (\chi_1 ^{-1})\boxtimes\rho.
\end{equation}
By Lemma \ref{rank one irreducible inductions}, $I (\chi_1) $
is irreducible, with an intertwining isomorphism to $I (\chi_1 ^{-1}). $
By (\ref {1st induction equation proof of while invariance Lemma}) and (\ref {2nd induction equation proof of while invariance Lemma}), this induces an isomorphism $I (\chi)\cong I (\chi ^ s). $
The proof of (2) is similar: let $Q_i = L_i U_i\subset G_n $
be the rank one standard parabolic with Weyl group generated by $w_i $. Then $$L_i\cong\left (\Q_q ^\times\right) ^ {i -1}\times\GL_2 (\Q_q)\times\left (\Q_q ^\times\right) ^ {n - i -1}. $$
By \cite[Th\'eor\`eme 3]{vigneras1989gl2} applied to the $\GL_2 (\Q_q) $-factor of $L_i$, we conclude $$\Ind_{B\cap L_i}^{L_i} \chi \cong \Ind_{B\cap L_i} \chi^{w_i}.$$
(In the metaplectic case, $L_i $
has a $\widetilde\GL_2 (\Q_q) $
factor, to which we may still apply the results of \emph{loc. cit.} with a twist by $\chi_\psi $.)
Then as above we obtain an isomorphism $$I (\chi) =\Induction_{Q_i} ^ {G_n}\Induction_{B\intersection L_i} ^ {L_i}\chi\cong\Induction_{Q_i} ^ {G_n}\Induction_{B\intersection L_i} ^ {L_i}\chi ^ {w_i} = I (\chi ^ {w_i}). $$
\end{proof}
\begin{definition}\label{def:generic_character}
Let $\chi = \chi_1\boxtimes\cdots\boxtimes\chi_n: (\Q_q ^\times) ^ n\rightarrow k ^\times $ be a character.
\begin{enumerate}
\item We say $\chi$ is \emph{generic} if 
 $\chi_i\chi_j \not\in \set{|\cdot|^{\pm 1}, \blackboardone}$ for all $1\leq i, j \leq n$ and $\chi_i/\chi_j \not\in \set{|\cdot|^{\pm 1}, \blackboardone}$ for all $1\leq i < j \leq n$.
 \item We say $\chi$ is \emph{almost generic} if $\chi_i\chi_j,\chi_i/\chi_j\not\in\set {|\cdot | ^ {\pm 1}, \blackboardone}$ for all $1\leq i <j\leq n $, $\chi_i^2 \neq |\cdot|^{\pm 1}$ for all $1 \leq i \leq n$, and $\chi_i^2 = \blackboardone$ for at most one $1 \leq i \leq n$. 
\end{enumerate}
\end{definition}
\begin{cor}\label{all weyl conjugates isomorphic generic case}
Let $\chi : (\Q_q ^\times) ^ n\rightarrow k ^\times $
be generic or almost generic. Then:
\begin{enumerate}
\item $I (\chi)\cong I (\chi ^ w) $
and $I' (\chi)\cong I' (\chi ^ w) $
for all $w\in W $.
\item $I (\chi) $
and $I' (\chi) $
are both irreducible.
\item $\Theta (I (\chi)) = I' (\chi).$
\end{enumerate}
\end{cor}
\begin{proof}
(1) follows from writing $w $
as a product of generators and repeatedly applying Lemma \ref{flippy induction lemma}. (2) follows from (1) combined with Lemmas \ref{quotients are intertwining operators} and \ref{home space dimension one lemma}. Once we have (2), it follows from  repeated applications of Corollary \ref{induced theta corollary}(\ref{cor:induced_theta_one}) that $$\Theta (I (\chi)) = I' (\chi ^{-1}). $$
Then (3) follows from (1).
\end{proof}
\begin{definition}\label{def:level_raising_generic}
Let 
 $\chi =\chi_1\boxtimes\cdots\boxtimes\chi_n: (\Q_q) ^\times\rightarrow k ^\times $ be a character.
 \begin{enumerate}
 \item We say $\chi$ is
\emph {level-raising generic}
if:
\begin{enumerate}[label = (\roman*)]
\item \label{item:lrg_ptone}For exactly one $1\leq i_0\leq n, $
$\chi_{i_0} = |\cdot | ^ {\frac {1} {2}} $.
\item \label{item:lrg_pttwo}For all $1\leq i <j\leq n $, $\chi_i\chi_j,\chi_i/\chi_j\not\in\set {|\cdot | ^ {\pm 1}, \blackboardone} $.
\item For all $i\neq i_0, $
 $\chi_i ^ 2\not\in\set {|\cdot | ^ {\pm 1}, \blackboardone} $.
\end{enumerate}
\item 
We say $\chi $
is \emph {almost level-raising generic} if it satisfies \ref{item:lrg_ptone} and \ref{item:lrg_pttwo} above,  and moreover:
\begin{enumerate}[label = (iii')]
    \item For all $i \neq i_0$, $\chi_i^2 \neq|\cdot|^{\pm 1}$; and $\chi_i^2 = \blackboardone$ for at most $i \neq i_0$.
\end{enumerate}
\end{enumerate}
\end{definition}
\begin{notation}
For $\chi $
almost level-raising generic, the set of Weyl conjugates $W (\chi) $
is naturally divided into two subsets: those $\chi ^ w =\chi_1'\boxtimes\cdots\boxtimes\chi_n' $
such that $\chi_i' = |\cdot | ^ {\frac {1} {2}} $
for some $i $; and those such that $\chi_i' = |\cdot | ^ {-\frac {1} {2}} $
for some $i $. We denote these subsets by $W (\chi) ^ + $
and $W (\chi) ^ - $, respectively.
\end{notation}
\begin{lemma}\label{lem:two_Weyl_conjugates}
Fix $\delta = + $
or $- $. Then for any $\chi ^ {w_1},\chi ^ {w_2}\in W (\chi) ^\delta $, we have $I (\chi ^ {w_1})\cong I (\chi ^ {w_2}) $
and $I' (\chi ^ {w_1})\cong  I' (\chi ^ {w_2}) $.
\end{lemma}
\begin{proof}
For any $\chi ^ {w_1},\chi ^ {w_2}\in W (\chi) ^\delta $, we may write $w_1 ^ {-1} w_2 = s_1\cdots s_k $
where each $s_i $
is one of the generators in (\ref{weyl group generators}) and $\chi ^ {w_1 s_1\cdots s_i}\in W (\chi) ^\delta $
for all $1\leq i\leq k $. The lemma then follows from repeated applications of Lemma \ref{flippy induction lemma}.
\end{proof}
\begin{construction}

For any $\chi $
which is almost level-raising generic, there is a quotient $J (\chi) $
of $I (\chi) $
defined as follows. Let $Q_{i_0} = L_{i_0} U_{i_0}\subset G_n $
be the standard rank one parabolic corresponding to the short root indexed by $i_0 $. Then we have
\begin{equation}
L_{i_0}\cong (\Q_q ^\times) ^ {i_0-1}\times G_1\times (\Q_q ^\times) ^ {n - i_0},
\end{equation}
and by \cite[Th\'eor\`eme 3]{vigneras1989gl2}, $\Induction_T ^ {L_{i_0}}\chi $
has a one-dimensional quotient
\begin{equation}
J_{i_0} (\chi)\coloneqq\chi_1\boxtimes\cdots\boxtimes\chi_{i_0-1}\boxtimes\blackboardone\boxtimes\chi_{i_0+1}\boxtimes\cdots\boxtimes\chi_n,
\end{equation}
 where $\blackboardone $
denotes the trivial representation of $G_1 $.
We let\begin{equation} J (\chi) =\Induction_{Q_{i_0}} ^ {G_n} J_{i_0} (\chi).
\end{equation}
\end{construction}
To study the theta lift of $J (\chi) $, we first have the following calculation in the rank-one case.

\begin{lemma}\label{lem:theta_trivial}
If $\blackboardone $
is the trivial representation of $\SO (V_1) (\Q_q) =\operatorname {PGL}_2 (\Q_q) $, then $$\Theta_{1, 1} (\blackboardone) = I' (|\cdot|^{1/2}), $$
with the corresponding quotient $$\mathcal S (V_1,k)\rightarrow\blackboardone\boxtimes\Theta_{1, 1} (\blackboardone) $$
induced by $$\phi\mapsto\left (g\mapsto\omega_\psi (1, g)\phi (0)\right),\;\; g\in\MP_2 (\Q_q). $$
\end{lemma}
\begin{proof}
Using the injection 
 $\blackboardone \hookrightarrow I(|\cdot|^{-1/2})$, we obtain by Corollary \ref{induced theta corollary}(\ref{cor:induced_theta_one}) an embedding
\begin{equation*}
    \begin{split}
        \Hom_{G_1'} (\Theta_{1,1}(\blackboardone), \mathcal M) =\Hom_{G_1\times G_1'}(\Omega_{1,1}, \blackboardone\boxtimes \mathcal M) \hookrightarrow\Hom_{G_1\times G_1'} (\Omega_{1,1}, I(|\cdot|^{-1/2})\boxtimes \mathcal M) \\= \Hom_{G_1'} (I'(|\cdot|^{1/2}, \mathcal M)
    \end{split}
\end{equation*}
functorial in $k[G_1']$-modules $\mathcal M$. In particular, $\Theta_{1,1}(\blackboardone)$ is a quotient of $I'(|\cdot|^{1/2})$. On the other hand, it is clear that the map in the lemma defines a nontrivial homomorphism  $\Theta_{1,1}(\blackboardone) \to I'(|\cdot|^{1/2})$, so we have a nontrivial composite
$$I'(|\cdot|^{1/2}) \twoheadrightarrow \Theta_{1,1}(\blackboardone) \to I'(|\cdot|^{1/2}).$$ This must be an isomorphism by Lemma \ref{home space dimension one lemma}, and the lemma follows.
\end{proof}
\begin {corollary}\label{cor:J_chi_theta}
If $\chi $
is almost level-raising generic, then $J (\chi) $
is irreducible and $\Theta_{n, n} (J (\chi)) = I' (\chi). $
\end{corollary}
\begin{proof}
First note that, by \cite[p. 44]{vigneras1989gl2}, $J (\chi) $
is the image of a nonzero intertwining operator $I (\chi)\rightarrow I (\rho) $, where $$\rho =\chi_1\boxtimes\cdots\boxtimes\chi_{i_0-1}\boxtimes\chi ^{-1}_{i_0}\boxtimes\chi_{i_0+1}\boxtimes\cdots\boxtimes\chi_n\in W(\chi)^-.$$
If $J (\chi)\twoheadrightarrow\pi $
is a nonzero quotient, then we obtain, by Lemma \ref{quotients are intertwining operators}, a nonzero intertwining operator $$I (\chi)\twoheadrightarrow J (\chi)\twoheadrightarrow\pi\to I (\chi ^ w) $$
for some $w\in W $. However, Lemmas \ref{home space dimension one lemma} and  \ref{lem:two_Weyl_conjugates} show that this composite (since it is not an isomorphism) must coincide with the intertwining operator $I (\chi)\rightarrow I (\rho) $
whose image defines $J (\chi) $. In particular $J (\chi) =\pi $. So indeed $J (\chi) $
is irreducible. Then $\Theta_{n, n} (J (\chi)) = I' (\rho ^{-1}) $
by Lemma \ref{lem:theta_trivial} and repeated applications of Corollary \ref{induced theta corollary}(\ref{cor:induced_theta_one}). Since $\rho ^{-1}\in W (\chi) ^ + $, 
we also have $I' (\rho ^{-1}) = I' (\chi) $ by Lemma \ref{lem:two_Weyl_conjugates}, 
and this completes the proof.
\end{proof}
\subsection{ Explicit theta lifting over $k$ for principal series}
\begin{notation}\label{notation:phi_bar}Given $\phi\in\Omega_{n, n} $
and $t_1,\ldots, t_n\in\Q_q $, define
\begin{equation}
\overline\phi (t_1,\ldots, t_n) =\integral_{\Q_q ^ {\frac {n (n -1)} {2}}}\phi (t_1 v_1, t_2 v_2+ a_1 v_1, t_3 v_3+ a_2 v_1+ a_3 v_2,\ldots, t_n v_n +\cdots + a_{\frac {n (n -1)} {2}} v_{n -1})\d a_1\cdots\d a_{\frac {n (n -1)} {2}},
\end{equation}
with $v_1,\ldots, v_n $
as in (\ref{local notation for split space}) above. 
\end{notation}A direct calculation shows that:
\begin{lemma}\label{intertwining map from theta to induced}
The map $\phi\mapsto\overline\phi $
defines a morphism of $T\times T' $-modules
$$R_{B\times B'} (\Omega_{n, n})\rightarrow\left ((|\cdot | ^ {-\frac {1} {2}}) ^ {\boxtimes n}\boxtimes\chi_\psi\cdot (|\cdot | ^ {\frac {1} {2}}) ^ {\boxtimes n}\right)\otimes \mathcal S (\Q_q ^ n,k), $$
where $T\times T' $
acts on $\mathcal S (\Q_q ^ n,k) $
by
\begin{equation}
(x_1,\ldots, x_n)\times (y_1,\ldots, y_n) (f) (t_1,\ldots, t_n) = f (x_1 ^{-1} t_1\overline y_1,\ldots, x_n ^{-1} t_n\overline y_n).
\end{equation}
\end{lemma}
\begin{definition}\label{def:C_chi}
    For any character $\chi =\chi_1\boxtimes\cdots\boxtimes\chi_n: (\Q_q ^\times) ^ n\rightarrow k ^\times $
with $\chi_i\neq |\cdot | ^ {-\frac {1} {2}} $
for all $i $, we consider the following condition on $\phi\in\Omega_{n, n} $:
\begin{multline}\tag {$C_\chi $}
\text{
There exists $g\in G_n' $
such that $f_\chi (\overline {\omega_\psi (1, g)\phi})\neq 0 $, where}\\
f_\chi:\left ((|\cdot | ^ {-\frac {1} {2}}) ^ {\boxtimes n}\boxtimes\chi_\psi\cdot (|\cdot | ^ {\frac {1} {2}}) ^ {\boxtimes n}\right)\otimes \mathcal S (\Q_q ^ n,k)\rightarrow\chi\boxtimes\chi_\psi\cdot \chi ^{-1}\\
\text {is the unique projection deduced from Lemma \ref{preliminary calculation for theta}}.
\end{multline}
\end{definition}
The map $f_\chi $
also exists and is unique without the assumption $\chi_i \neq |\cdot | ^ {-\frac {1} {2}} $, so $(C_\chi) $
makes sense for all $\chi $; this is elementary but not needed in our applications.

\subsubsection{}
Let $K \subset G_n$ be the hyperspecial subgroup stabilizing the self-dual lattice $L\subset V_n$ (\ref{local notation for split space}). 
\begin{lemma}\label{lem:theta_criterion_generic}
Let $\chi: (\Q_q ^\times) ^ n\rightarrow k ^\times $
be almost generic and unramified, and suppose $\phi\in\Omega_{n, n} ^ K $
satisfies condition $(C_{\chi}) $. Also let $\mathcal M$ be any admissible $k[G_n']$-module. Then, for any nonzero map $$\theta:\Omega_{n, n}\rightarrow I (\chi)\boxtimes \mathcal M $$
of $k[G_n\times G_n'] $-modules, we have $$\theta (\phi)\neq 0. $$
\end{lemma}
\begin{proof}
Since $I (\chi) $
is irreducible, Corollary \ref{all weyl conjugates isomorphic generic case} implies that $\theta $
factors as
$$\Omega_{n, n}\xrightarrow {\theta_0} I (\chi)\boxtimes\Theta_{n, n} (I (\chi))\cong I (\chi)\boxtimes I' (\chi)\xrightarrow {f} I (\chi)\boxtimes \mathcal M $$
for some map of $G'_n $-modules $$f: I' (\chi)\rightarrow \mathcal M. $$
Then since $I' (\chi) $
is also irreducible, it suffices to show $\theta_0 (\phi)\neq 0. $
Now, by Lemma \ref{intertwining map from theta to induced}, the map
$$\chi\mapsto\left ((h,g)\mapsto f_\chi (\overline {\omega_\psi (h, g)\phi})\right) $$
gives a $(G_n\times G_n') $-intertwining map
$$F_\chi:\Omega_{n, n}\rightarrow I (\chi)\boxtimes I' (\chi ^{-1})\cong I (\chi)\boxtimes I' (\chi). $$
Since $F_\chi $
is not identically zero, Corollary \ref{all weyl conjugates isomorphic generic case} shows that $\theta_0 $
coincides with $F_\chi $
up to a nonzero scalar; in particular $\theta_0 (\phi)\neq 0 $
if and only if $F_\chi (\phi)\neq 0 $. Then because $\phi $
is $K $-spherical, $F_\chi (\phi)\neq 0 $
if and only if there exists $g\in G_n' $
with $f_\chi (\overline {\omega_\psi (1, g)\phi})\neq 0 $, which is condition $(C_\chi) $.
\end{proof}
Similarly, we have:
\begin{lemma}\label{lem:theta_criterion_alrg}
Let $\chi: (\Q_q ^\times) ^ n\rightarrow k ^\times $
be almost level-raising generic and unramified, and suppose $\phi\in\Omega_{n, n} ^ K $
satisfies condition $(C_\chi) $. 
Then, for any nonzero map $$\theta:\Omega_{n, n}\rightarrow J (\chi)\boxtimes \mathcal M $$
of $k[G_n\times G_n']$-modules, we have $$\theta (\phi)\neq 0. $$
\end{lemma}
\begin{proof}
Since $J (\chi) $
is irreducible by Corollary \ref{cor:J_chi_theta}, the map $\theta$  factors
as
$$\Omega_{n, n}\twoheadrightarrow J (\chi)\boxtimes\Theta_{n, n} (J (\chi))\cong J (\chi)\boxtimes I' (\chi)\xrightarrow {f} J (\chi)\boxtimes\mathcal  M $$
for some map of $G'_n $-modules $$f: I' (\chi)\rightarrow \mathcal M. $$
By Lemma \ref{quotients are intertwining operators}, we may assume without loss of generality that $f: I' (\chi)\rightarrow I' (\chi ^ w) $
is an intertwining operator. Then by Lemmas \ref{home space dimension one lemma}  and \ref{lem:two_Weyl_conjugates}, we see that it suffices to show $\phi $
has nonzero image under the map
$$\theta_0:\Omega_{n, n}\rightarrow J (\chi)\boxtimes I' (\chi)\rightarrow J (\chi)\boxtimes I' (\chi ^{-1}). $$
By Lemma \ref{intertwining map from theta to induced}, the map $$\phi\mapsto\left ((h,g)\mapsto f_\chi (\overline {\omega_\psi (h, g)\chi})\right) $$
gives a $(G_n\times G_n') $-intertwining map $$F_\chi:\Omega_{n, n}\rightarrow I (\chi)\boxtimes I' (\chi ^{-1}). $$
As in the proof of Lemma \ref{lem:theta_criterion_generic}, since $\phi $
is $K $-invariant, condition $(C_\chi) $
is equivalent to $F_\chi (\phi)\neq 0. $
Now project to obtain a composite $$F_\chi':\Omega_{n, n}\xrightarrow {F_\chi} I (\chi)\boxtimes I' (\chi ^ {-1})\rightarrow J (\chi)\boxtimes I' (\chi ^{-1}). $$
Now we observe that any $K$-spherical vector in $I(\chi)$ has nonzero image in $J(\chi)$; indeed, by the construction of $J(\chi)$ it suffices to show this when $n = 1$, in which case it is clear from the explicit intertwining operator in \cite[p. 44]{vigneras1989gl2} and the assumption $q^2 - 1 \neq 0 $ in $k$. In particular, we have $F'_\chi (\phi)\neq 0 $
if and only if $F_\chi (\phi)\neq 0. $
On the other hand, $F_\chi' $
must factor as
$$F'_\chi:\Omega_{n, n}\rightarrow J (\chi)\boxtimes\Theta_{n, n} (J (\chi))\rightarrow J (\chi)\boxtimes I' (\chi ^{-1}); $$
the map $\Theta_{n, n} (J (\chi)) = I' (\chi)\rightarrow I' (\chi ^{-1}) $
is unique by Lemma \ref{home space dimension one lemma}, so $F_\chi' $
coincides with $\theta_0 $
up to a nonzero scalar. Hence $\theta_0 (\chi)\neq 0 $
is equivalent to $(C_\chi) $.
\end{proof}
\subsubsection{}
We end this subsection with a convenient shortcut that we will use to check condition $(C_\chi)$ in the characteristic zero, non-level-raising case. 
\begin{lemma}\label{lem:shortcut_C_chi_generic}
Suppose $k = \C$ 
and let $\chi: (\Q_q^\times)^2 \to k^\times$ be generic and unramified. Assume that  $ \phi\in \Omega_{2,2}^K$ satisfies:
    \begin{enumerate}
        \item For all $x,y\in V$, we have $\phi(x,y) \in \Q_{\geq 0}$.\label{C chi criterion positivity}
        \item $\phi$ is supported on $L\times q^{-1}L$ and invariant under translations by $qL \times q^2 L$.\label{C chi criterion support and invariance}
        \item There exist elements $x,y\in V$ with $y\cdot v_1 = 0 $ and $\phi(x, y) \neq 0$. \label{C chi criterion nonzero value}
    \end{enumerate}
    Then  there exists $\chi^w\in W(\chi)$ such that $\phi$ satisfies condition $(C_{\chi^w})$.
\end{lemma}
\begin{proof}
    Let $w_0$ be the Weyl element in (\ref{weil rep}); then $$\overline {\omega_\psi(1, w_0)\phi}\in \left ((|\cdot | ^ {-\frac {1} {2}}) ^ {\boxtimes 2}\boxtimes\chi_\psi\cdot (|\cdot | ^ {\frac {1} {2}}) ^ {\boxtimes 2}\right)\otimes \mathcal S (\Q_q ^ 2,k)$$ is a unit multiple of  the function \begin{align*}
        c(t_1,t_2) &= \integral_{V^2} \integral_{\Q_q}\phi(x, y) \psi\left(t_1 x\cdot  v_1 + t_2 y\cdot v_2+ a y\cdot  v_1 \right) \d a \d x \d y\\
        &= \volume\set{q^{-2} \Z_q}\integral_{V^2} \phi(x, y) \psi\left(t_1 x\cdot v_1 +  t_2y\cdot v_2 \right) \cdot \mathbbm{1}_{ y\cdot v_1 \in q^2\Z_q}\d a \d x \d y.
    \end{align*}
    By condition (\ref{C chi criterion support and invariance}) of the lemma, $c$ is supported on $q^{-1} \Z_q \times q^{-2}\Z_q$ and invariant under translations by $\Z_q \times q\Z_q$. 
    Note that the conditions of the lemma together imply that $c(\Z_q, q\Z_q)\neq 0$. From this, we will deduce that $f_{\chi^w}(c) \neq 0$ for some Weyl conjugate $\chi^w$ of $\chi$, which will show the lemma.

    Indeed, write $\chi = \chi_1\boxtimes \chi_2$ and $\chi^w = \chi_1^w\boxtimes \chi_2^w$ for $w\in W$. Then $$f_{\chi^w}(c) = f_{|\cdot|^{\frac{1}{2}} \chi_1^w} \left(f_{|\cdot|^{\frac{1}{2}} \chi_2^w}c(t_1, \cdot)\right)$$
    (where the  functions on the right are defined in Lemma \ref{preliminary calculation for theta}(\ref{preliminary calculation for theta part two})).
    Now, because
 $c(\Z_q,q\Z_q)\neq 0$, $c(1, \cdot)$ is a nonzero element of the four-dimensional $k$-vector space $ S_{-2, 1}$ from Lemma \ref{preliminary calculation for theta}(\ref{preliminary calculation for theta part three}).  Since $\chi_2^w|\cdot|^{\frac{1}{2}}$ takes on four distinct nontrivial values as $w$ ranges over $W$, 
 we may therefore replace $\chi$ with a Weyl conjugate such that $d\coloneqq f_{|\cdot|^{\frac{1}{2}}\chi_2} c(t_1, \cdot) $ is not identically zero. Since $d$ lies in $ S_{-1,0}$ as a function of $t_1$, Lemma \ref{preliminary calculation for theta}(\ref{preliminary calculation for theta part three}) again implies that either $f_{|\cdot|^{\frac{1}{2}} \chi_1} (d) $ or $f_{|\cdot|^{\frac{1}{2}} \chi_1^{-1}} (d) $ is nonzero. This concludes the proof because $\chi_1^{-1}\boxtimes \chi_2$ is Weyl-conjugate to $\chi$.

\end{proof}
\subsection{ Applications to formal theta lifts}\label{subsec:applications to formal theta lifts}
\subsubsection{}
For this subsection, fix the following data:
\begin{itemize}[$\circ$]
    \item A quadratic space $V$ of trivial discriminant and dimension $2n + 1\geq 3.$
    \item A neat compact open subgroup $K = \prod K_\l \subset \spin(V)(\A)$.
    \item An odd prime $q$ such that $K_q$ is hyperspecial.
    \item A subring $R \subset \C$ which is either $\C$, or a finite flat extension of $\breve \Z_p$ (embedded into $\C$ by a choice of isomorphism $\iota: \overline\Q_p \isomorphism \C$).
     In the latter case we assume  the pro-order of $K_q$ is prime to $p$. Let $\varpi_R\in R$ generate the maximal ideal (so $\varpi_R = 0$ if $R= \C$), and 
 write $k\coloneqq R/\varpi_R$.
    \item A character $\chi: (\Q_q ^\times) ^ n\rightarrow k ^\times $
that is either almost generic or almost level-raising generic (Definition \ref{def:generic_character} and Definition \ref{def:level_raising_generic}). 
\end{itemize}
With these data, we make the following notation:
\begin{notation}\label{notation:m_chi_etc}
    \leavevmode
    \begin{enumerate}
        \item Write $\T_q\coloneqq \T_{\spin_{2n+1}, q, R}$. Let $\m_\chi \subset \T_q$ be the maximal ideal corresponding to $\chi$; explicitly, $\m_\chi$ is the annihilator of the unique spherical vector in $I(\chi)$. 
        \item For any ring $A$, $\mathbbm 1_A$ denotes the trivial $A[K_q]$-module.
       
    \end{enumerate}
\end{notation}

\begin{lemma}\label{lem:compact_is_induced}
We have
$$\cind_{K_q} ^ {\spin (V) (\Q_q)}\blackboardone_R\otimes_{\T_{q, \m_\chi}}k\cong I (\chi) $$
as $k $-linear representations of $\spin (V) (\Q_q) $.
\end{lemma}
\begin{proof}
Write $$\pi_\chi\coloneqq\cind_{K_q} ^ {\spin (V) (\Q_q)}\blackboardone_R\otimes_{\T_q,\mathfrak m_\chi}k. $$
Then we have a map $f:\pi_\chi\rightarrow I (\chi) $, sending the generator of $\pi_\chi $
to the unique spherical vector. 
The first claim is that $f $
is surjective. Indeed, if $\chi $
is almost generic, this is automatic by Corollary \ref{all weyl conjugates isomorphic generic case}. If $\chi $
is almost level-raising generic, then $I (\chi) $
has no $K_q $-spherical submodules: since taking the $K_q $-invariants is exact, $J (\chi) $
is the unique $K_q $-spherical constituent, and it cannot be both a quotient and a submodule by Lemma \ref{home space dimension one lemma}. So indeed $f $
is surjective. 

It remains to prove $f$ is injective.
Because $\pi_\chi$ is generated by $K_q$-spherical vectors, every $\spin(V)(\Q_q)$-stable subspace $V\subset \pi_\chi$ satisfies $R_B(V) \neq 0$; for instance, this follows from \cite[Corollaire II.3.5]{vigneras1996representations} combined with (I.3.15) of \emph{op. cit.} Thus it suffices to show $R_B(f)$ is injective, or equivalently that $$\dim_k R_B(\pi_\chi) = \dim_k R_B(I(\chi)) = |W| = n!\cdot 2^n.$$
To compute the dimension of $R_B(\pi_\chi)$, note that $R_B(\cind_{K_q}^{\spin(V)(\Q_q)} \mathbbm 1_k) = k[X^\bullet(\widehat T)]$  by the Iwasawa decomposition for $\spin(V)(\Q_q)$, and the action of $\T_q\otimes_R k$ is the natural one under the Satake isomorphism
$$\T_q\otimes_R k \isomorphism k[X^\bullet(\widehat T)]^W.$$ Since $k[X^\bullet(\widehat T)]$ is a finite flat $k[X^\bullet(\widehat T)]^W$-algebra of degree $|W|$, we have 
$$\dim R_B(\pi_\chi) = \dim \left(k[X^\bullet(\widehat T)]\otimes_{k[X^\bullet(\widehat T)]^W,\m_\chi} k\right) =  |W|,$$ as desired. 
\end{proof}
\begin{notation}\label{notation:phi_mod_p}
For all $\phi \in \mathcal S (V^n\otimes \Q_q,R)$, let $ \overline\phi $ be its image in $\mathcal S (V^n\otimes \Q_q,k)$.
\end{notation}
(Despite the conflict with Notation \ref{notation:phi_bar}, we hope that the meaning will always be clear from context.)
\subsubsection{}
For the next proposition, recall the notation on formal theta lifts from \S\ref{subsec:formal_theta}.
\begin{prop}\label{prop:change_test_formal}
Let $\alpha\in\Test_K (V, R) $
be a test vector and $n_0\geq 1 $
an integer such that:
\begin{enumerate}
\item $\Theta (\alpha,\phi)\not\equiv 0\pmod{\varpi_R ^ {n_0} }$
for some $\phi =\phi ^ q\otimes\phi_q\in \mathcal S (V ^ n\otimes\A_f, R)^K. $
\item For all $h\in\mathfrak m_\chi\subset\T_q $
and all $\phi_q'\in \mathcal S (V ^ n\otimes\Q_q, R) ^ {K_q} $, $$\Theta (h\cdot\alpha,\phi ^ q\otimes\phi'_q)\equiv 0\pmod{\varpi_R ^ {n_0}}. $$
\end{enumerate}
Then for any $\phi_q ^\circ\in \mathcal S (V ^ n\otimes\Q_q, R) ^ {K_q} $ such that $\overline\phi_q^\circ $ 
satisfies condition $(C_\chi) $, $$\Theta (\alpha,\phi ^ q\otimes\phi_q ^\circ)\not\equiv 0\pmod{\varpi_R ^ {n_0}}. $$
\end{prop}
\begin{proof}
For all $f\in\cind_{K_q} ^ {\spin (V) (\Q_q)}\blackboardone_R $, we can consider the convolution $$f\ast\alpha\in\Test_{K ^ q} (V, R) $$ (notation as in (\ref{subsubsec:test_convolution})).
By Proposition \ref{prop:equivariance_and_modularity}, the map
\begin{equation}\label {representation theoretic equation for varying test vector proof}
\cind_{K_q} ^ {\spin (V) (\Q_q)}\blackboardone_R\otimes \mathcal S (V ^ n\otimes\Q_q, R)\rightarrow M ^ n_{n +\frac {1} {2}, R}
\end{equation}
defined by
$$(f,\phi_q')\mapsto\Theta (f\ast\alpha,\phi ^ q\otimes\phi_q') $$
is $\spin (V) (\Q_q)\times\MP_{2 n} (\Q_q) $-equivariant. By condition (2) of the proposition applied to $h =\varpi_R\in\mathfrak m_\chi $, the image of (\ref {representation theoretic equation for varying test vector proof}) is contained in $\varpi_R ^ {n_0-1} M_{n +\frac {1} {2}, R} ^ n $. Abbreviate $$\mathcal M\coloneqq\frac {\varpi_R ^ {n_0-1} M_{n +\frac {1} {2}, R} ^ n} {\varpi_R ^ {n_0} M_{n +\frac {1} {2}, R} ^ n}. $$
Reducing modulo $\varpi_R$,
(\ref {representation theoretic equation for varying test vector proof}) induces a map
\begin{equation}\label {modulo varpi equation for varying test vector proof}
\cind_{K_q} ^ {\spin (V) (\Q_q)}\blackboardone_{k}\otimes \mathcal S (V ^ n\otimes\Q_q,k) \rightarrow \mathcal M
\end{equation}
which remains a map of $\spin (V) (\Q_q)\times\MP_{2 n} (\Q_q) $-modules by Proposition \ref{prop:integral_structure_stable_metaplectic}. Now note that condition (2) of the proposition implies (\ref {modulo varpi equation for varying test vector proof}) factors through the quotient
$$\cind_{K_q} ^ {\spin (V) (\Q_q)}\blackboardone_{k}\twoheadrightarrow\cind_{K_q} ^ {\spin (V) (\Q_q)}\blackboardone_{k}\otimes_{\T_q,\mathfrak m_\chi}k\cong I (\chi) $$ (Lemma \ref{lem:compact_is_induced}).
By duality 
\cite[p. 96, Propri\'et\'e (vi)]{vigneras1996representations},
(\ref {modulo varpi equation for varying test vector proof}) is equivalent to a nonzero map
$$\theta: \mathcal S (V ^ n\otimes\Q_q,k)\rightarrow I (\chi ^{-1})\otimes\mathcal  M $$
and, for $\phi ^\circ_q\in \mathcal S (V ^ n\otimes\Q_q,R) ^ {K_q} $, we have
$$\Theta (\alpha,\phi ^ q\otimes\phi ^\circ_q)\not\equiv 0\pmod{\varpi_R ^ {n_0}}\IFF\theta (\overline\phi_q ^\circ)\neq 0. $$
If $\chi $
is almost generic, the proposition therefore follows from Corollary \ref{all weyl conjugates isomorphic generic case} and Lemma \ref{lem:theta_criterion_generic}. So assume instead that $\chi $
is almost level-raising generic. 

\begin {claim}
Let $\mathcal M $
be any admissible $k[\MP_{2n}(\Q_q)]$-module. Then every
map of $\spin(V)(\Q_q)\times \MP_{2n}(\Q_q) $-modules $$\mathcal S (V ^ n\otimes\Q_q,k)\rightarrow I (\chi ^ {-1})\boxtimes \mathcal M$$
factors as
$$\mathcal S (V ^ n\otimes\Q_q,k)\rightarrow J (\chi)\boxtimes \mathcal M\rightarrow I (\chi ^{-1})\boxtimes \mathcal M. $$
\end {claim}
Given the claim, the proposition follows from Lemma \ref{lem:theta_criterion_alrg}, because $J (\chi)\hookrightarrow I (\chi ^{-1}) $
is injective (cf. the proof of Corollary \ref{cor:J_chi_theta}).

Let us now prove the claim. Since the statement is purely local in nature, we resume our local abbreviation $\Omega_{n, n} = \mathcal S (V ^ n\otimes\Q_q,k) $. 
The claim is also insensitive to replacing $\chi $
with any $\chi ^ w\in W (\chi) ^ + $ (by Lemma \ref{lem:two_Weyl_conjugates}), so suppose without loss of generality that $\chi = |\cdot | ^ {\frac {1} {2}}\boxtimes\rho $
for some almost generic character $\rho: (\Q_q ^\times) ^ n\rightarrow k ^\times $. Apply Corollary \ref{induced theta corollary}(\ref{cor:induced_theta_one}) with $\chi_0 = |\cdot | ^ {-\frac {1} {2}} $
and $\pi_{m -1} = I (\rho)\cong I (\rho ^ {-1}) $. Since $\Theta_{n -1, n -1} (I (\rho)) = I' (\rho) $
by Corollary \ref{all weyl conjugates isomorphic generic case}, we obtain an isomorphism
\begin{equation}\label {another equation for proof of changing test vector}
\Home (\Omega_{n, n}, I (\chi ^{-1})\boxtimes \mathcal M) =\Home (I' (\chi), \mathcal M)
\end{equation}
that is functorial in $\mathcal M $. So it suffices to show the claim with $\mathcal M = I' (\chi) $. But in this case, (\ref {another equation for proof of changing test vector}) combined with Lemma \ref{home space dimension one lemma} shows that there is a unique non-zero map
\begin{equation}\label {final equation for proof of changing test vector}
\Omega_{n, n}\rightarrow I (\chi ^{-1})\boxtimes I' (\chi).
\end{equation}
Since $J (\chi) $
injects into $I (\chi ^{-1}) $, we also have the map induced by the theta lift $$\Omega_{n, n}\rightarrow J (\chi)\boxtimes\Theta_{n, n} (J (\chi))\cong J (\chi)\boxtimes I' (\chi)\rightarrow I (\chi ^{-1})\boxtimes I' (\chi); $$
this must coincide with (\ref {final equation for proof of changing test vector}) up to a nonzero scalar, which shows the claim.
\end{proof}

\subsection{Main result on changing test functions}
\subsubsection{}
Let $\pi$, $S$, and $E_0$ be as in 
 Notation \ref{notation:pi_basic}, and fix a prime $\p$ of $E_0$, which we suppress from all the notation in this subsection. Fix an isomorphism $\iota: \overline\Q_p \isomorphism \C$ inducing the prime $\p$, and  let $R\subset \overline\Z_p$ be the ring of integers of the maximal unramified extension of $O = O_\p$. 

We apply the results of \S\ref{subsec:applications to formal theta lifts} to study the behavior of  $\lambda_n^D(Q, \phi; K)$ and $\kappa_n^D(Q, \phi; K)$ as $\phi$ changes locally at a prime $q\nmid Q$. 
\begin{prop}\label{prop:alrg_admissible}
    Suppose $q$ is an admissible prime.  Then there exists an unramified character $\chi: (\Q_q^\times)^2 \to \overline\F_p^\times$ such that $\chi$ is almost level-raising generic, and the corresponding maximal ideal $\m_\chi\subset \T_{q, R}$ contains the kernel of the Hecke action on the unique spherical vector of $\pi_q$. 
\end{prop}
\begin{proof}
Because $\pi_q$ is unramified with trivial central character, we have $\pi_q = I(\widetilde \chi)$ for an unramified character
     $\widetilde \chi: (\Q_q^\times )^2 \to \C^\times $, uniquely determined up to Weyl action. Write $\alpha = \iota^{-1} \widetilde \chi(q, 1)$, $\beta =\iota^{-1} \widetilde \chi(1, q)$; then by Theorem \ref{thm:rho_pi_LLC}(\ref{part:rho_pi_LLC1}), $\rho_{\pi}(\Frob_q)$ has eigenvalues $\alpha q^{1/2}, \beta q^{1/2}, \alpha^{-1} q^{1/2}, \beta^{-1} q^{1/2}$, which lie in $\overline\Z_p^\times \subset \overline\Q_p^\times$. By the admissibility of $q$, we may assume without loss of generality that \begin{equation}\label{eq:alrg_for_translation}
         \alpha q^{1/2}\equiv q \pmod \p, \; \beta q^{1/2} \not\equiv\pm  q,\pm 1, q^2, q^{-1}\pmod \p.
     \end{equation}
     We define the character $\chi$ to be the reduction modulo $p$ of $\iota^{-1}\widetilde \chi$, and the conditions (\ref{eq:alrg_for_translation}) exactly correspond to $\chi$ being almost level-raising generic.
\end{proof}
\begin{cor}\label{cor:admissible_C_chi_lambda}
    Suppose $Qq$ is admissible with $\nu(DQ)$ odd, and fix an $S$-level structure $K$ for $\spin(V_{DQ})$. Let
 $\phi_q^\circ \in \mathcal S(V_{DQ}^2 \otimes \Q_q, O)^{K_q}$ 
be a test function whose image in $\mathcal S(V_{DQ}^2\otimes \Q_q, \overline \F_p)$ satisfies condition ($C_\chi$), where $\chi: (\Q_q^\times)^2 \to \overline \F_p^\times$ is the almost level-raising generic character of Proposition \ref{prop:alrg_admissible}. 
    Then for all $n\geq 1$ and all $\phi = \phi^q \otimes \phi_q \in \mathcal S(V_{DQ}^2 \otimes \A_f, O)^K,$ we have
    $$\lambda_n^D(Q, \phi^q\otimes \phi_q^\circ; K) \supset \lambda_n^D(Q,\phi; K).$$
\end{cor}
\begin{rmk}
    The same corollary holds for $\kappa_n^D(Q, - )$ if $\nu(DQ)$ is even, but  this version will not be used for the main results.
\end{rmk}
\begin{proof}
Suppose $\lambda_n^D(Q, \phi; K) =  (\varpi ^ {n_0-1}) $
for some $1\leq n_0 $; without loss of generality we may assume $n_0\leq n $ and that $$\lambda_n^D(Q, \phi^q\otimes \phi_q'; K) \equiv 0 \pmod {\varpi^{n_0 - 1}}$$ for all $\phi_q' \in \mathcal S(V_{DQ}^2\otimes \Q_q, O)^{K_q}.$

Now choose a vector $\alpha\in \Test_K(V_{DQ}, \pi, O/\varpi^n)$ such that $\alpha (Z(T, \phi)_K)$ generates $\lambda_n^D(Q,\phi; K)$ for some $T\in \Sym_2(\Q)_{\geq 0}$.
    Lift $\alpha$ arbitrarily to an $O$-valued test function $\widetilde \alpha\in \Test_K(V_{DQ}, O)$.
    Recall $R\subset \overline \Z_p$ is the ring of integers of the maximal unramified extension of $O$, and let $f_\pi: \T_{q,R} \to R$ be the character associated with the Hecke eigenvalues of $\pi_q$, so that 
$f_\pi (h)\in (\varpi) $
 for all $h\in\mathfrak m_\chi\subset\mathcal \T_{q,R}$.  
Then for $h\in\mathfrak m_\chi $ and $\phi_q'\in  \mathcal S(V_{DQ}^2\otimes \Q_q, O)^{K_q}$, we have $$\Theta (h\cdot\widetilde\alpha,\phi ^ q\otimes\phi_q')\equiv f_\pi (h)\Theta (\widetilde\alpha,\phi ^ q\otimes\phi_q')\equiv 0\pmod{\varpi ^ {n_0}}, $$
so we may apply Proposition \ref{prop:change_test_formal}  to conclude. (Note that because $q^4\not\equiv 1 \pmod p$, $p$ does not divide the pro-order of $K_q$.)

\end{proof}
We now give an analogue of Corollary \ref{cor:admissible_C_chi_lambda} in characteristic zero, which requires $Q = 1$. 
\begin{prop}\label{prop:eigenvalues_generic}
    Suppose $q\not\in S$ is a prime such that $\rho_{\pi}(\Frob_q)$ has distinct eigenvalues. Let $\chi: (\Q_q^\times)^2 \to \C^\times$ be the unramified character, well-defined up to $W$-action, such that $\pi_q$ is a constituent of $I(\chi)$. 
Then $\chi$ is generic.
\end{prop}
\begin{proof}
    The proof is  similar to Proposition \ref{prop:alrg_admissible}, using that $|\alpha| = |\beta| = 1$ because $\pi_q$ is tempered by Theorem \ref{thm:rho_pi_LLC}(\ref{part:rho_pi_LLC1}). 
\end{proof}
\begin{cor} \label{cor:changing_kappa}

    Let $D \geq 1$ be squarefree with $\nu(D)$ even, and suppose $q\not\in S\cup \div(D)$ is a prime such that $\rho_{\pi}(\Frob_q)$ has distinct eigenvalues. Fix an $S$-level structure $K$ for $\spin(V_D)$, and  let $\phi_q^\circ\in \mathcal S(V_D^2 \otimes \Q_q, \Z)^{K_q}$ be a test function whose image in $\mathcal S(V_D^2\otimes \Q_q, \C)^{K_q}$ satisfies the hypotheses of Lemma \ref{lem:shortcut_C_chi_generic}, or more generally condition ($C_\chi$), where $\chi: (\Q_q^\times)^2 \to \C^\times$ is the generic character of Proposition \ref{prop:eigenvalues_generic}.    Then for all $\phi = \phi^q\otimes \phi_q\in \mathcal S(V_D^2\otimes \A_f, O)^K$,  $$\kappa^D(1, \phi; K) \neq 0 \implies \kappa^D(1, \phi^q\otimes \phi_q^\circ; K) \neq 0.$$

\end{cor}
\begin{proof}
    The argument is similar to Corollary \ref{cor:admissible_C_chi_lambda}. First fix  a vector $\alpha\in \Test_K(V_D, \pi, O)$ with $\alpha_\ast \circ \partial_{\AJ,\m} (Z(T,\phi)_K) \neq 0$ for some $T\in \Sym_2(\Q)_{\geq 0}$.   
    Because $ H ^ 1 (\Q, T_\pi) $ is torsion-free by Lemma \ref{lem:H1_tors_free_new}(\ref{lem:H1_tors_free_part}), we may choose a linear functional $\beta: H ^ 1 (\Q, T_\pi)\rightarrow \overline\Q_p$ such that $\beta (\alpha_\ast \circ \partial_{\AJ,\m} (Z(T,\phi)_K) )\neck 0 $.     
    Let $\widetilde \alpha\in \Test_K(V_D, \C)$ denote the composite map
    $$\CH^2(\Sh_K(V_D))  \xrightarrow{\partial_{\AJ,\m}} H^1(\Q, H^3_\et(\Sh_K(V_D)_{\overline \Q}, O(2))_\m )\xrightarrow{\alpha_\ast} H^1(\Q, T_\pi) \xrightarrow{\beta} \overline\Q_p\xrightarrow{\iota}\C.$$ 
    Then $\widetilde \alpha$ is Hecke-equivariant because $\alpha$ is so. 
Let $\phi_q'\in \mathcal S(V_{DQ}^2\otimes \Q_q, \C)^{K_q}$ be any vector. By the Hecke-equivariance of $\widetilde\alpha$, we have $\Theta(h \cdot \widetilde \alpha, \phi^q \otimes \phi_q') = 0$ for all $h \in \mathfrak m_\chi$ 
(Notation \ref{notation:m_chi_etc}), with $R = \C$, $\varpi_R = 0$. 
We can now apply Proposition \ref{prop:change_test_formal}  to conclude.
\end{proof}

With essentially the same proof, we have the following in the endoscopic case:
\begin{cor}\label{cor:endoscopic_changing_kappa}
    With the setup of Corollary \ref{cor:changing_kappa}, suppose $\pi$ is endoscopic associated to a pair $(\pi_1,\pi_2)$. Then for all $\phi = \phi^q\otimes \phi_q\in \mathcal S(V_D^2\otimes \A_f, O)^K$ and $j = 1$ or 2, we have
    $$\kappa^D(1, \phi; K)^{(j)} \neq 0 \implies \kappa^D(1, \phi^q\otimes \phi_q^\circ; K)^{(j)} \neq 0.$$
\end{cor}

\section {The ramified $\spin_5 $
Rapoport-Zink space}\label{sec:RZ}
\subsection {The moduli problem}
\subsubsection {}\label{subsubsection basepoint of RZ space}
Fix a prime $q >2 $, and let $\O_q $
be the unique maximal order in the non-split quaternion algebra $B$ over $\Q_q $. Suppose given a $q $-divisible group $\mathbb X $ 
over $\overline\F_q $ of dimension 4 and height 8,
equipped with an action $\iota_{\mathbb X}:\O_q\hookrightarrow\End (\mathbb X) $
and a principal polarization $\lambda_{\mathbb X}:\mathbb X\xrightarrow {\sim}\mathbb X ^\check $
such that the Rosati involution $\ast $
of $\End (\mathbb X) $
induces a nebentype involution on $\O_q $ of
\emph {unit type} 
(Definition \ref{unit and non-unit type definition}).
\begin {definition}\label{definition RZ space}
Let $\nilpotent $
be the category of schemes over $\breve\Z_q $
on which $q $
is locally nilpotent.
Let $\mathcal N:\nilpotent\to\operatorname{Set}$
be the functor
sending $S\in\nilpotent $
to the set of isomorphism classes of tuples
$(X,\iota,\lambda,\rho) $
where:
\begin{enumerate}
    [label  =(\roman*)]\item $X $
is a $q $-divisible group of dimension 4 and height 8 over $S $.
\item $\iota:\O_q\hookrightarrow\End (X/S) $
is an $\O_q $-action such that $$\debt (T -\iota (\alpha) |\Lie (X)) = (T ^ 2 -\tr (\alpha) T + N (\alpha)) ^ 2,\;\;\forall\alpha\in\O_q. $$
\item $\lambda: X\xrightarrow {\sim} X ^\check $
is a principal polarization such that the Rosati involution on $\End (X/S) $
extends the involution $\ast $
on $\O_q $.
\item If $\overline S = S\times_{\breve\Z_q}\overline\F_q $
denotes the mod $q $
fiber, then $\rho $
is an $\O_q $-linear quasi-isogeny $$\rho: X\times_S\overline S\to\mathbb X\times_{\overline\F_q}\overline S $$
such that $\rho ^ \vee\circ\lambda_{\mathbb X}\circ\rho = c (\rho)\lambda $
for some $c (\rho)\in\Q_q ^\times $.
\end {enumerate}
\end {definition}
The functor $\mathcal N $
is studied in \cite{wang2020bruhat}; in this section, we will recall the key points, and prove additional properties needed for our applications.
\subsubsection {} The functor $\mathcal N $
is represented by a formal scheme over $\SPF\breve\Z_q $, locally formally of finite type, which admits a decomposition into open and closed formal subschemes \begin{equation}\label{decomposition of RZ space open and closed}\mathcal N =\sqcup_{i\in\Z}\mathcal N (i) \end{equation}
according to the $q $-adic valuation of $c (\rho) $.

\subsubsection {}\label{subsubsec:Weil_descent}
Let $\sigma:\breve\Z_q\to\breve\Z_q $
denote the
\emph {arithmetic} Frobenius, lifting the $q $th power map on $\overline\F_q $. The formal scheme $\mathcal N $
is equipped with a canonical Weil descent datum $\phi:\mathcal N\xrightarrow {\sim}\sigma ^\ast\mathcal N $
over $\breve\Z_q $, which we now recall. On $S$-points, $\phi$ is given by the isomorphism $\phi(S): \mathcal N(S) \isomorphism \mathcal N((\sigma^{-1})^\ast S)$, which  sends 
 $(X, \iota, \lambda, \rho)\in \mathcal N(S)$ to $((\sigma ^ {-1}) ^\ast X, (\sigma ^ {-1}) ^\ast\iota, (\sigma ^ {-1}) ^\ast\lambda,\rho\circ F_{X/\overline S}),$ where 
$$F_{X/\overline S}: (\sigma^{-1})^\ast X_{\overline S} \to X_{\overline S}$$ is the relative Frobenius. 
Since $c (\rho\circ F_{X/\overline S}) = c (\rho) +1 $, $\phi $
restricts to an isomorphism $$\phi_i:\mathcal N (i)\xrightarrow {\sim}\sigma ^\ast\mathcal N (i +1) $$
for each $i\in\Z $.

\subsection {The Bruhat-Tits stratification}
\subsubsection {} Let $\mathcal M $
denote the underlying reduced scheme of $\mathcal N (0) $. We now recall the description in \cite{wang2020bruhat} of the stratification of $\mathcal M $
in terms of lattices in the isocrystal $N $
of our fixed $q $-divisible group $\mathbb X $. By the assumption that the involution $\ast $
on $\O_q $
is of unit type, we may choose coordinates
\begin {equation}
\O_q =\Z_{q ^ 2}\oplus\Pi\Z_{q ^ 2},
\end {equation}
where $\alpha ^\ast =\Pi\alpha\Pi ^ {-1} =\sigma\alpha $
for $\alpha\in\Z_{q ^ 2} $, $\Pi ^\ast =\Pi $, and $\Pi ^ 2=q $.
\subsubsection {}
\label{decomposing circ bullet subsubsection}
Label the two embeddings of $\Z_{q ^ 2} $
into $\breve\Z_q $
by $j_\bullet $
and $j_\circ $. Then we have a decomposition
\begin {equation}
N = N_\bullet\oplus N_\circ,
\end {equation}
where $\iota (\alpha) = j_? (\alpha) $
on $N_? $
for $? =\bullet,\circ $.
Each of $N_\bullet $
and $N_\circ $
is an isocrystal of dimension 4 and slope $\frac {1} {2} $.

The polarization $\lambda_{\mathbb X} $
induces a pairing $$\langle\cdot,\cdot\rangle: N\otimes N\to\breve\Q_q =\breve\Z_q\otimes\Q_q, $$
with respect to which $N_\bullet $
and $N_\circ $
are each isotropic (since $\ast $
is nontrivial on $\Z_{q ^ 2}\subset\O_q $). Define a new pairing
\begin {equation}
\langle\cdot,\cdot\rangle_\bullet: N_\bullet\otimes N_\bullet\to\breve\Q_q
\end {equation}
by $$\langle x, y\rangle_\bullet =\langle x,\Pi y\rangle. $$
Then $\langle\cdot,\cdot\rangle_\bullet $
is symplectic and non-degenerate.

The operators $\Pi $
and $V $
on $N $
both interchange $N_\bullet $
and $N_\circ $, so the operator
\begin {equation}
\tau\coloneqq\Pi V ^ {-1}
\end {equation}
stabilizes $N_\bullet $; moreover
\begin {equation}
\langle\tau x,\tau y\rangle_\bullet =\langle x, y\rangle_\bullet ^\sigma,
\end {equation}
where again $\sigma $
is the arithmetic Frobenius of $\breve\Q_q $.
Hence $W\coloneqq N_\bullet ^ {\tau = 1} $
is a 4-dimensional $\Q_q $-vector space equipped with a $\Q_q $-valued symplectic form, such that $N_\bullet = W\otimes_{\Q_q}\breve\Q_q. $
\begin {definition}
For a lattice $\Lambda\subset W $, denote the $\Z_q $-dual lattice by $\Lambda ^\check $.
We define the following families of lattices in $W $:
\begin {align*}
\mathcal L_{\set {0}} & =\set {\text {lattices }\Lambda\subset W\text { s.t. }\Lambda =\Lambda ^\check}.\\
\mathcal L_{\set {2}} & =\set {\text {lattices }\Lambda\subset W\text { s.t. }\Lambda = q\Lambda ^\check}.\\
\mathcal L_{\set {0 2}} & =\set {\text {pairs of lattices }\Lambda_0,\Lambda_2\subset W\text { s.t. } q\Lambda_0 = q\Lambda_0 ^\check\subset_2 q\Lambda_2 ^\check =\Lambda_2\subset_2\Lambda_0}.\\
\mathcal L_{\set {1}} & =\set {\text {lattices }\Lambda_1\subset W\text { s.t. } q\Lambda_1 ^\check\subset_2\Lambda_1\subset_2\Lambda_1 ^\check}.
\end {align*}
\end {definition}
Here, the notation $\Lambda\subset_n\Lambda' $ was defined in (\ref{subsubsec:lattices_n}). 

\begin {thm}
\label{BT stratification first version}
The underlying reduced scheme $\mathcal M $
of $\mathcal N (0) $
admits a stratification $$\mathcal M =\mathcal M ^ 0_{\set {0}}\sqcup\mathcal M ^ 0_{\set {2}}\sqcup\mathcal M ^ 0_{\set {0 2}}\sqcup\mathcal M ^ 0_{\set {1}}, $$
with a decomposition into open and closed subschemes
$$\mathcal M ^ 0_{?} =\bigsqcup_{y\in\mathcal L_{?}}\mathcal M^ 0 _{?}(y) $$
for each $? =\set {0} $, $\set {2} $, $\set {0 2} $, $\set {1} $, satisfying the following conditions
(where $\mathcal M_? $
and $\mathcal M_? (y) $
denote the Zariski closures of $\mathcal M ^ 0_? $
and $\mathcal M ^ 0_? (y) $, respectively):
\begin {enumerate}
\item \label{thm:BT_one}For $? =\set {0} $ or $\set{2}$ and each $\Lambda\in\mathcal L_?$, $\mathcal M_? (\Lambda) $
is isomorphic to the smooth projective hypersurface in $\mathbb P ^ 3_{\overline\F_q} $
defined by the equation $$X ^q_3X_0 - X_0 ^qX_3+ X_2 ^qX_1 - X_1 ^qX_2. $$
The scheme $\mathcal M $
is the union $\mathcal M_{\set {0}}\cup\mathcal M_{\set {2}} $.
\item\label{thm:BT_two} Given $\Lambda_0\in\mathcal L_{\set {0}} $
and $\Lambda_2\in\mathcal L_{\set {2}} $, $\mathcal M_{\set{0}} (\Lambda_0) $
meets $\mathcal M _{\set{2}}(\Lambda_2) $
if and only if $(\Lambda_0,\Lambda_2)\in\mathcal L_{\set {0 2}} $, in which case the intersection is transverse and $$\mathcal M _{\set{0}}(\Lambda_0)\intersection\mathcal M _{\set{2}}(\Lambda_2) =\mathcal M _{\set{02}}(\Lambda_0,\Lambda_2). $$
For each $(\Lambda_0,\Lambda_2)\in\mathcal L_{\set {0 2}} $, $\mathcal M _{\set{02}}(\Lambda_0,\Lambda_2) $
is isomorphic to $\mathbb P ^ 1_{\overline\F_q} $, and both of the resulting embeddings $\mathbb P ^ 1_{\overline\F_q}\hookrightarrow\mathbb P ^ 3_{\overline\F_q} $
are linear.
\item \label{thm:BT_3}For each $\Lambda_1\in\mathcal L_{\set {1}} $, $\mathcal M _{\set{1}}^ 0 (\Lambda_1) =\mathcal M_{\set{1}} (\Lambda_1) $
is an isolated point.
Given $\Lambda_0\in\mathcal L_{\set {0}} $
and $\Lambda_2\in\mathcal L_{\set {2}} $, $\mathcal M_{\set{1}} (\Lambda_1) $
lies on $\mathcal M_{\set{0}} (\Lambda_0) $
if and only if $\Lambda_1\subset\Lambda_0 $,
and on $\mathcal M _{\set{2}}(\Lambda_2) $
if and only if $\Lambda_2\subset\Lambda_1 $.
\item \label{thm:BT_4}For a pair of distinct $\Lambda_0,\Lambda_0'\in\mathcal L_{\set {0}} $, $\mathcal M _{\set {0}}(\Lambda_0) $
meets $\mathcal M _{\set {0}}(\Lambda_0') $
if and only if $\Lambda_0\intersection\Lambda_0'\in\mathcal L_{\set {1}} $; in this case the intersection is transverse and we have $$\mathcal M _{\set {0}}(\Lambda_0)\intersection\mathcal M _{\set {0}}(\Lambda_0') =\mathcal M _{\set {1}}(\Lambda_0\intersection\Lambda_0').$$
\item\label{thm:BT_5} For a pair of distinct $\Lambda_2,\Lambda_2'\in\mathcal L_{\set {2}} $, $\mathcal M_{\set {2}} (\Lambda_2) $
meets $\mathcal M _{\set {2}}(\Lambda_2') $
if and only if $\Lambda_2+\Lambda_2'\in\mathcal L_{\set {1}} $; in this case the intersection is transverse and we have $$\mathcal M _{\set {2}}(\Lambda_2)\intersection\mathcal M _{\set {2}}(\Lambda_2') =\mathcal M _{\set {1}}(\Lambda_2+\Lambda_2'). $$
\item \label{thm:BT_6}
The stratum $\mathcal M_{\set {1}} $
is precisely the nonsmooth locus of $\mathcal N (0) $, and the complete local ring of $\mathcal N (0) $
at each point in $\mathcal M_{\set {1}} $
is isomorphic to $$\breve\Z_q\llbracket X, Y, Z, W\rrbracket/q - XY + ZW. $$
\end {enumerate}
\end {thm}
\begin {proof}
Each point except (6) is contained in \cite{wang2020bruhat}, and (6) is \cite[Corollary 4.2]{oki2022supersingular}.
\end {proof}
\subsubsection {}\label{meaning of BT strata subsubsection}
For later use, we recall the meaning of each of the strata in Theorem \ref {BT stratification first version} on the level of $\overline\F_q $-points.

Given $s = (X,\iota,\lambda,\rho)\in\mathcal N(\overline\F_q) $, the covariant Dieudonn\'e module of $X $
gives rise to an $\O_q $-stable lattice $M\subset N $. Such an $M $
admits a decomposition $$M = M_\bullet\oplus M_\circ\subset N_\bullet\oplus N_\circ $$
as in (\ref{decomposing circ bullet subsubsection}); here we are using $q\neq 2 $. 
For a lattice $\Lambda\subset W $, define $\breve\Lambda\coloneqq\Lambda\otimes\breve\Z_q\subset M_\bullet. $
Then we have, for any $s\in \mathcal M (\overline\F_q) $
and any lattices $\Lambda_0\in\mathcal L_{\set {0}} $, $\Lambda_2\in\mathcal L_{\set {2}} $, $\Lambda_1\in\mathcal L_{\set {1}} $:
\begin {itemize}
\item $s $
lies in $\mathcal M_{\set {1}}$ if and only if $M_\bullet = \tau M_\bullet $, and is the point $\mathcal M_{\set {1}} (\Lambda_1) $
if and only if $M_\bullet=\breve\Lambda_1. $ 
\item $s $
lies in $\mathcal M_{\set {0}} (\Lambda_0) -\mathcal M_{\set{1}}$
if and only if $M_\bullet +\tau M_\bullet =\breve\Lambda_0 $. 
 \item $s $
lies in $\mathcal M_{\set {2}} (\Lambda_2) -\mathcal M_{\set{1}}$
if and only if $M_\bullet\intersection\tau M_\bullet =\breve\Lambda_2$.
\end {itemize}
By Theorem \ref {BT stratification first version}, at least one of these three options occurs for any point $s\in\mathcal M (\overline\F_q) $.

\begin{notation}
From now on, to ease the notation we shall abbreviate $\mathcal M_{\set{0}}(\Lambda_0)$ as $\mathcal M(\Lambda_0)$, etc. 
\end{notation}

\subsection {Deformation theory and the geometry of $\mathcal N $}
\subsubsection {}
Let $(X,\lambda,\iota,\rho) $
be an $S $-valued point of $\mathcal N $, for some $S\in\nilpotent $. To $X $
we associate the (covariant) Dieudonn\'e crystal $\mathbb D (X) $
\cite{mazur1974universal}; thus for any thickening $S\hookrightarrow\widehat S $
in $\nilpotent $
admitting locally nilpotent divided powers, we obtain a locally free sheaf $\mathbb D (\widehat S) $
of $\mathcal O_{\widehat S} $-modules, such that $D (X)\coloneqq\mathbb D (X) (S) $
fits into a canonical exact sequence
\begin {equation}\label{exact sequence for crystal}
0\to\omega_{X ^\check}\to D (X)\to\Lie (X)\to 0
\end {equation}
of locally free $\O_S $-modules.
\subsubsection{}
\label{structures on crystal subsubsection}
As in (\ref{meaning of BT strata subsubsection}) above, the action of $\Z_{q ^ 2}\subset\O_q $
on $\mathbb D (X) $
induces a decomposition $$\mathbb D (X) =\mathbb D (X)_\bullet\oplus\mathbb D (X)_\circ, $$
and likewise for $D (X) $, $\omega_{X ^\check} $, and $\Lie X $; the action of $\Pi $
interchanges the two components in each case.

The polarization $\lambda $
induces a perfect alternating pairing$$\langle\cdot,\cdot\rangle:\mathbb D (X)\otimes\mathbb D (X)\to\O_S ^ {\operatorname {cris}}. $$
Since both $\mathbb D (X)_\bullet $
and $\mathbb D (X)_\circ $
are isotropic, $\langle\cdot,\cdot\rangle $
identifies $\mathbb D (X)_\bullet $
with the dual of $\mathbb D (X)_\circ $.
Finally, the submodule $\omega_{X ^\check} $
of $D (X) $
is also isotropic, so that $\lambda $
induces perfect pairings of locally free $\O_S $-modules:
\begin {align*}
\langle\cdot,\cdot\rangle:\; &\omega_{X ^\check,\bullet}\otimes\Lie X_\circ\to\O_S\\
\langle\cdot,\cdot\rangle:\; &\omega_{X ^\check,\circ}\otimes\Lie X_\bullet\to\O_S.
\end {align*}

If $S =\Spec\overline\F_q $, then $\mathbb D (X) $
is equivalent to the data of the Dieudonn\'e module $M $
of $X $; the exact sequence (\ref{exact sequence for crystal}) becomes $$0\to VM/pM\to M\to M/VM\to 0. $$

\subsubsection {}
Let $S\hookrightarrow\widehat S $
be a thickening in $\nilpotent $
admitting locally nilpotent divided powers, and fix $x = (X,\lambda,\iota,\rho)\in\mathcal N (S) $.
Denote by $\operatorname {Lift} (x) $
the set of isomorphisms classes of lifts of $x $
to
$\widehat x = (\widehat X,\widehat\lambda,\widehat\iota,\widehat\rho)\in\mathcal N (\widehat S) $, and denote by $\overline {\operatorname {Lift}} (x) $
the set of locally free, $\O_q$-stable, totally isotropic  $\O_{\widehat S} $-submodules $\widehat\omega_{X ^\check}\subset\mathbb D (X) (\widehat S) $
lifting $\omega_{X ^\check} $. From the well-known deformation theory of $q$-divisible groups \cite{messing1972crystals}, one has:
\begin {prop}
\label{GM deformation theory prop}
The canonical map $$\widehat x = (\widehat X,\widehat\lambda,\widehat\iota,\widehat\rho)\mapsto\omega_{\widehat X ^\check}\subset D (\widehat X) =\mathbb D (X) (\widehat S) $$
defines a bijection $$\Lift (x)\xrightarrow {\sim}\overline {\operatorname {Lift}} (x). $$
\end {prop}
\subsubsection {}
Let $\mathcal N^ {\smooth} $
denote the formally smooth locus of $\mathcal N(0) $, which by Theorem
\ref {BT stratification first version}(6) is the complement of $\bigsqcup_{\substack \Lambda_1 \in \mathcal L_{\set{1}}}\mathcal M (\Lambda_1). $
Before we can calculate the tangent bundle to the mod $q $
fiber of $\mathcal N ^ {\smooth}$
in Theorem
\ref {tangent bundle theorem} below, we need the following lemma.
\begin {lemma}\label{lemma isomorphisms of line bundles away from smooth locus}
Suppose given $(X,\lambda,\iota,\rho)\in\mathcal N ^{\smooth} (S) $, for some $\overline\F_q $-scheme $S $. Then $\Pi $
induces isomorphisms of line bundles on $S $:
\begin {align*}
\Pi:\frac {\Lie X_\bullet} {\Pi\Lie X_\circ} &\xrightarrow {\sim}\Pi\Lie X_\bullet\subset\Lie X_\circ\\
\Pi:\frac {\omega_{X ^\check,\bullet}} {\Pi\omega_{X ^\check,\circ}} &\xrightarrow {\sim}\Pi\omega_{X ^\check,\bullet}\subset\omega_{X ^\check,\circ},
\end {align*}
and likewise with $\bullet $
and $\circ $
reversed.
\end {lemma}
\begin {proof}
We consider the first map; the second, and the versions with $\bullet$ and $\circ$ reversed, are all similar to this case. Without loss of generality, we may assume that $S = \Spf R$ is affine and formally of finite type, and that $R$ is a local ring with maximal ideal $\m$ such that $R/\m  = \overline \F_q$.
Then $\Lie X_\bullet$ and $\Lie X_\circ$ are each free of rank two over $R$ by the Kottwitz condition, and the map $\Pi_\bullet: \Lie X_\bullet \to \Lie X_\circ$ 
is nonzero modulo $\m$; indeed, this amounts to the assertion that $\Pi M_\bullet \neq VM_\bullet$ for the Dieudonn\'e module $M = M_\bullet \oplus M_\circ$ corresponding to the special fiber of $X$, and this holds because we are away from the nonsmooth locus of $\mathcal N(0)$, cf. (\ref{meaning of BT strata subsubsection}).

In particular, we can choose bases of $\Lie X_\bullet$ and $\Lie X_\circ$ such that $$\Pi_\bullet = \begin{pmatrix} 1 & 0 \\ 0 & d \end{pmatrix}\,: \, \Lie X_\bullet\to \Lie X_\circ$$ for some $d\in R$.
Now, we also know that $\Pi^2 = q = 0$ on $\Lie(X)$, so the matrix $g$ for $\Pi_\circ: \Lie X_\circ \to \Lie X_\bullet$ 
must satisfy
$$g \begin{pmatrix} 1 & 0 \\ 0 & d \end{pmatrix} = \begin{pmatrix} 1 & 0 \\ 0 & d \end{pmatrix} g = 0.$$
A direct calculation shows that $g = \begin{pmatrix} 0 & 0 \\ 0 & w \end{pmatrix}$ for some $w\in R$, where $wd = 0$; but the same reasoning as for $\Pi_\bullet$ shows that $\Pi_\circ$ is nonzero modulo $\m$, so $w$ is a unit and we conclude $d = 0$. From these coordinates for $\Pi_\bullet$ and $\Pi_\circ$, it is clear that $\Pi \Lie X_\bullet$ and $\Lie X_\bullet/\Pi \Lie X_\circ$ are both free of rank one over $R$ and that the map in the lemma is indeed an isomorphism.
\end {proof}
\begin {thm}\label{tangent bundle theorem}
Let $\mathcal T $
denote the tangent bundle on the mod $q $
fiber $\mathcal N ^ {\smooth}_{\overline\F_q} $. Then we have a canonical exact sequence:
$$0\to\mathcal T\to\home(\omega_{\mathcal X ^\check,\bullet},\Lie \mathcal X_\bullet)\to\home(\Pi\omega_{\mathcal X ^\check,\circ},\Lie \mathcal X_\bullet/\Pi\Lie \mathcal X_\circ)\to 0, $$
where $\mathcal X $
is the universal $q $-divisible group.
\end {thm}
\begin {proof}
It suffices to consider deformations of   a point $(X,\iota,\lambda,\rho) \in
\mathcal N ^ {\smooth}(R)$
to points of $\mathcal N ^ {\smooth}(R[\epsilon]/\epsilon ^ 2) $,
 for $R $
 an $\overline\F_q $-algebra.  Let  $S = \Spec R$ and $\widehat S = \Spec R[\epsilon]/\epsilon^2$. By Proposition \ref{GM deformation theory prop}, we need to consider lifts of $\omega_{X ^\check} $
to locally free submodules $\widehat\omega_{X ^\check}\subset\mathbb D (X) (\widehat S) $
which are $\O_q $-stable and isotropic;
this is equivalent to lifting $\omega_{X ^\check,\bullet}\subset D (X)_\bullet $
and $\omega_{X ^\check,\circ}\subset D (X)_\circ $
to locally free submodules $\widehat\omega_{X ^\check,\bullet}\subset\mathbb D (X)_\bullet (\widehat S) $
and $\widehat\omega_{X ^\check,\circ}\subset\mathbb D (X)_\circ (\widehat S) $, subject to the following conditions.
\begin{enumerate}[label= (\roman*)]
    \item $\langle\widehat\omega_{X ^\check,\bullet},\widehat\omega_{X ^\check,\circ}\rangle = 0. $
\item $\Pi\widehat\omega_{X ^\check,\bullet}\subset\widehat\omega_{X ^\check,\circ}. $
\item $\Pi\widehat\omega_{X ^\check,\circ}\subset\widehat\omega_{X ^\check,\bullet} $.
\end{enumerate}
Now, lifts of $\omega_{X ^\check,\bullet} $
correspond to maps of $R $-modules $f_\bullet:\omega_{X ^\check,\bullet}\to\Lie X_\bullet $
via $$f_\bullet\mapsto\operatorname {span}\set {x +\epsilon f_\bullet (x) +\epsilon\omega_{X ^\check,\bullet}\,:\, x\in\omega_{X ^\check,\bullet}}, $$
and likewise for $\omega_{X ^\check,\circ} $. The conditions (i)-(iii) translate to:
\begin {enumerate}
[label= (\roman*)']
\item $\langle x, f_\bullet (y)\rangle =\langle y, f_\circ (x)\rangle $, for all $x\in\omega_{X ^\check,\circ} $, $y\in\omega_{X ^\check,\bullet} $.
\item $\Pi f_\bullet (x) = f_\circ (\Pi x) $
for all $x\in\omega_{X ^\check,\bullet} $.
\item $\Pi f_\circ (x) = f_\bullet (\Pi x) $
for all $x\in\omega_{X ^\check,\circ} $.
\end {enumerate}
If we specify $f_\bullet $, then (i)' can be taken as a
\emph {definition} of $f_\circ $. In terms of $f_\bullet $
only, conditions (ii)' and (iii)' correspond to:
\begin {enumerate}
[label= (\roman*)'']
\setcounter {enumi} {1}
\item $\langle y,\Pi f_\bullet (x)\rangle =\langle x,\Pi f_\bullet (y)\rangle $, for all $x, y\in\omega_{X ^\check,\bullet} $.
\item $\langle y, f_\bullet (\Pi x)\rangle =\langle x, f_\bullet (\Pi y)\rangle $, for all $x, y\in\omega_{X ^\check,\circ} $.
\end {enumerate}
Now, since we are outside the singular locus, Lemma
\ref {lemma isomorphisms of line bundles away from smooth locus} implies that $\Pi\omega_{X ^\check,\bullet} $
and $\Pi\omega_{X ^\check,\circ} $
are locally rank-one direct summands of the rank-two projective $R $-modules $\omega_{X ^\check,\circ} $
and $\omega_{X ^\check,\bullet} $, respectively. In particular, using that (ii)'' and (iii)'' are clearly satisfied when $x$ and $y$ are linearly dependent, it suffices to check (ii)'' and (iii)'' for $x\in\Pi\omega_{X ^\check,\circ} $
and $x\in\Pi\omega_{X ^\check,\bullet} $, respectively. Using $\langle\Pi z,\Pi w\rangle = q\langle z, w\rangle = 0 $, we find that both conditions (ii)'' and (iii)'' are equivalent to $\Pi f_\bullet (\Pi\omega_{X ^\check,\circ}) = 0 $. Again by Lemma
\ref {lemma isomorphisms of line bundles away from smooth locus}, the kernel of $\Pi $
on $\Lie X_\bullet $
is $\Pi\Lie X_\circ $, so we are just requiring $$f_\bullet (\Pi\omega_{X ^\check,\circ})\subset\Pi\Lie X_\circ, $$
as desired.
\end {proof}
\subsubsection {Scheme-theoretic description of the strata $\mathcal M (\Lambda_0)$, $\mathcal M(\Lambda_2)$}\label{subsubsec:scheme_strata}
Fix a lattice $\Lambda\in \mathcal L_{\set{0}} \sqcup \mathcal L_{\set{2}}$;
 we recall the construction of $\mathcal M (\Lambda) $
as a subscheme of $\mathcal N $
given in \cite[\S4]{wang2020bruhat}. 
Let $\breve\Lambda =\Lambda\otimes\breve\Z_q\subset N_\bullet $, and let $Y $
denote the $q $-divisible group over $\overline\F_q $
associated to the lattice $$\breve\Lambda\oplus q^\delta \Pi ^ {-1}\breve\Lambda\subset N_\bullet\oplus N_\circ, $$
where $\delta = 0$ if $\Lambda\in\mathcal L_{\set {0}} $
and $\delta = 1 $
if $\Lambda\in\mathcal L_{\set {2}} $. 
The group $Y$ comes with a natural quasi-isogeny $$t:Y \to \mathbb X. $$
For any $(X,\lambda,\iota,\rho)\in\mathcal N (S) $
with $S $
an $\overline\F_q $-scheme, consider the two quasi-isogenies:
\begin {align*}
\rho_+ &: X\xrightarrow {\rho}\mathbb X_S\xrightarrow {q ^\delta t ^ {-1}} Y_S,\\
\rho_- &: Y_S\xrightarrow {q ^ {1 -\delta} t}\mathbb X_S\xrightarrow {\rho ^ {-1}} X,
\end {align*}
where $\delta = 0 $
if $\Lambda\in\mathcal L_{\set {0}} $
and $\delta = 1 $
if $\Lambda\in\mathcal L_{\set {2}} $. 

Then the scheme $\mathcal M (\Lambda) $
is constructed as the locus where $\rho_+ $
and $\rho_- $
are both isogenies.
 In fact, since the dual lattice to $\breve\Lambda\oplus\Pi ^ {-1}\breve\Lambda $
 is $q ^ {1-2\delta} (\breve\Lambda\oplus\Pi ^{-1}\breve\Lambda) $, 
we may identify $Y\xrightarrow {\sim} Y ^\check $
by $q^{1-2\delta} t^\check\circ\lambda_{\mathbb X}\circ t $; with respect to this polarization, $\rho_+ $
and $\rho_- $
are duals, so $\rho_- $
is an isogeny if and only if $\rho_+ $
is.
\begin {prop}\label{prop identifying O (1)}
Let $\O (1) $
be the line bundle on $\mathcal M (\Lambda) $
corresponding to the embedding into $\mathbb P ^ 3_{\overline\F_q} $
of Theorem
\ref {BT stratification first version} (2). Then, if $(X, \lambda, \iota,\rho)$ is the universal $q$-divisible group over $\mathcal M(\Lambda)$,  we have isomorphisms of line bundles on $\mathcal M (\Lambda) ^ {\smooth}\coloneqq\mathcal M (\Lambda)\intersection\mathcal N ^ {\smooth} $:
$$\O (- q)\cong\Pi\Lie X_\circ\cong\Lie X_\bullet/\Pi\Lie X_\circ. $$
\end {prop}
\begin {proof}
We give the proof in the case $\Lambda \in \mathcal L_{\set{0}}$; for $\Lambda\in \mathcal L_{\set{2}}$, the roles of $\bullet$ and $\circ$ are interchanged. Let $S\coloneqq \mathcal M(\Lambda)^{\sm}$, and let $Y_S$ be the constant $q$-divisible group $Y$ on $S$.

We need to recall the construction of the projective embedding from \cite[\S4]{wang2020bruhat}. First of all, one has the Dieudonn\'e crystal $$\mathbb D (Y_S) =\mathbb D (Y_S)_\bullet\oplus\mathbb D (Y_S)_\circ, $$
with notation as in (\ref{structures on crystal subsubsection}), and
$$D (Y_S)_\circ = D (Y)_\circ\otimes\O_S $$
is a free $\O_S $-module of rank 4. 
The submodule $ \rho_{+,\ast} D(X)_\circ \subset D (Y_S)_\circ $
is locally a free summand of rank $1 $, and this is the tautological bundle $\O(-1)$ under the map $\mathcal M(\Lambda)\hookrightarrow \mathbb P^3_{\overline \F_q}$. 
The canonical Verschiebung map
$$\mathtt V: D(Y_S)_\circ \to D(Y_S^{(q)})_\circ\simeq  D(Y_S)_\bullet\otimes_{\O_S} \O_{S^{(q)}}$$ 
is given by the $q$th power map $\O_S \to \O_{S^{(q)}}$ tensored with the isomorphism $$ V: D(Y)_\circ = \Pi^{-1} \breve \Lambda/  \Pi \breve\Lambda\isomorphism \breve \Lambda/ q\breve\Lambda =  D(Y)_\bullet,$$ so we have an isomorphism
$$\O(-q) \simeq \mathtt V(\rho_{+,\ast}(D(X)_\circ).$$
Now we note that the map 
\begin{equation}\label{eq:D(X)_to_D(Y)}\mathtt V \circ \rho_{+,\ast}: D(X) _\circ \to D(Y_S)_\bullet
\otimes_{\O_S} \O_{S^(q)}\end{equation}
annihilates both $\Pi D( X)_\bullet$ and $\omega_{X^\vee, _\circ}$,
because $\mathtt V( \omega_{Y_S^\vee, \circ}) = \mathtt V \circ \Pi (D( Y_S)_{\circ}) = 0$ by the definition of $Y$. 
In particular, (\ref{eq:D(X)_to_D(Y)}) induces a surjection 
\begin{equation}\label{eq:V_circ_rho}\mathtt V\circ \rho_{+,\ast}: \frac{\Lie X_\circ}{\Pi \Lie X_\bullet} \isomorphism \O(-q),\end{equation}
which is a map of line bundles by Lemma \ref{lemma isomorphisms of line bundles away from smooth locus} and therefore an isomorphism. 
Combined with Lemma \ref{lemma isomorphisms of line bundles away from smooth locus} for the other isomorphism, this completes the proof.  
\end {proof}
 \begin {thm}\label{normal bundle theorem}
For any $\Lambda \in \mathcal L_{\set{0}}\sqcup\mathcal L_{\set{2}},  $
the normal bundle to $\mathcal M (\Lambda) ^ {\smooth} $
inside $\mathcal N _{\overline\F_q}^ {\smooth} $
is isomorphic to $\O (-2q) $
for the embedding into $\mathbb P ^ 3_{\overline\F_q} $
given in Theorem
\ref {BT stratification first version}.
\end {thm}
\begin {proof}
 For simplicity, we  continue to assume  $\Lambda\in\mathcal L_{\set {0}} $; the other case is similar.

The first step is to compute the tangent bundle to $\mathcal M (\Lambda) ^ {\smooth} $. We wish to consider the lifts of $(X,\lambda,\iota,\rho)\in\mathcal M (\Lambda)^{\smooth} (R) $
to points of $\mathcal M (\Lambda)^{\smooth} (R [\epsilon]/\epsilon ^ 2) $ for $R$ an $\overline\F_q$-algebra, which we may take to be reduced since $\mathcal M(\Lambda)$ is reduced. Continuing the notation of (\ref{subsubsec:scheme_strata}),  such lifts  correspond to the pairs $$\begin{cases} f_\bullet:\omega_{X ^\check,\bullet}\to\Lie X_\bullet\\f_\circ:\omega_{X ^\check,\circ}\to\Lie X_\circ
\end{cases} $$
satisfying (i)'-(iii)' from the proof of Theorem \ref{tangent bundle theorem}, subject to the additional condition that $$f_\bullet (\rho_{-,\ast} (\omega_{Y ^\check_R,\bullet})) = f_\circ (\rho_{-,\ast} (\omega_{Y ^\check_R,\circ})) = 0. $$
By the definition of $Y$, we have $\omega_{Y ^\check_R,\circ} = 0 $, so the second condition is automatic. 

\begin{claim}
    On $\mathcal M (\Lambda) ^ {\smooth} $
we have \begin{equation}\label{eq:rho_minus_normal}
    \rho_{-,\ast} (\omega_{Y_R ^\check,\bullet}) =\Pi\omega_{X ^\check,\circ}. 
\end{equation}
\end{claim}

Given the claim, we conclude by comparing with the proof of Theorem \ref{tangent bundle theorem} that
the tangent bundle $\mathcal T_\Lambda $ to $\mathcal M (\Lambda) ^ {\smooth} $
is canonically isomorphic to $$\home (\omega_{X ^\check,\bullet}/\Pi\omega_{X^\vee,\circ},\Lie X_\bullet). $$
The normal bundle is the quotient $\mathcal T/\mathcal T_\Lambda $, which by Theorem
\ref {tangent bundle theorem} is 
\begin{equation*}
    \home (\Pi\omega_{X^\vee,\circ},\Pi\Lie X_\circ). 
\end{equation*}
Since $\Pi\omega_{X ^\check,\circ} $
is dual to $\Lie X_\bullet/\Pi\Lie X_\circ $
on $\mathcal M (\Lambda) ^ {\smooth} $, the theorem then follows from Proposition \ref{prop identifying O (1)}.
Now we turn to the proof of the claim. 
We have $$\rho_{-,\ast} (\omega_{Y_R^\check,\bullet}) = \rho_{-,\ast} (D(Y_R)_\bullet) = \left(\ker\left(\rho_{+,\ast}:D(X)_\circ \to D(Y_R)_\circ\right)\right)^\perp$$ because $\rho_+$ and $\rho_-$ are duals; the orthogonal complement is with respect to the perfect pairing on $D(X)$. 
Because $R$ is reduced, the Verschiebung $\mathtt V: D(Y_R)_\circ \to D(Y_R)_\bullet $ is  injective,  so (\ref{eq:V_circ_rho}) implies that $$\ker \left(\rho_{+,\ast}: D(X)_\circ \to D(Y_R)_\circ\right) = \omega_{X^\vee, \circ} + \Pi D(X)_\circ.$$
Arguments similar to Lemma \ref{lemma isomorphisms of line bundles away from smooth locus} show that $(\Pi D(X)_\circ)^\perp = \Pi D(X)_\bullet$, and so we conclude
$$\rho_{-,\ast} (\omega_{Y_R^\check,\bullet}) = \left(\omega_{X^\vee, \circ} + \Pi D(X)_\bullet\right)^\perp = \omega_{X^\vee, \bullet} \cap \Pi D(X)_\circ.$$
Now it follows from Lemma \ref{lemma isomorphisms of line bundles away from smooth locus} and the snake lemma for the diagram
\begin{center}
\begin{tikzcd}
    0 \arrow[r]& \omega_{X^\vee, \circ} \arrow[r]\arrow[d,"\Pi"] & D(X)_\circ\arrow[r]\arrow[d,"\Pi"] & \Lie X_\circ  \arrow[r]\arrow[d,"\Pi"] & 0\\
    0 \arrow[r]& \omega_{X^\vee, \bullet}  \arrow[r] & D(X)_\bullet  \arrow[r] & \Lie X_\bullet  \arrow[r] & 0
\end{tikzcd}
\end{center}
that $\omega_{X^\vee,\bullet} \cap \Pi D(X)_\bullet = \Pi \omega_{X^\vee, \circ}$, so the proof of the claim is complete.
\end {proof}
\subsection{Regularization of $\mathcal M$ and intersection theory}
\begin{notation}

Let $\widetilde{\mathcal N}(0)$ be the blowup of $\mathcal N(0)$ along $\mathcal M_{\set{1}}$, and let $\widetilde {\mathcal M}$ be the strict transform of the reduced locus $\mathcal M$; then $\widetilde {\mathcal M}$ is smooth. 
 We denote by $C (\Lambda_1) $
the exceptional divisor of $\widetilde {\mathcal M}$ above $\mathcal M (\Lambda_1) $
for each $\Lambda_1\in\mathcal L_{\set {1}} $. For any $\Lambda_0\in\mathcal L_{\set {0}} $
and $\Lambda_2\in\mathcal L_{\set {2}} $, let $\widetilde{\mathcal M} (\Lambda_0) $
and $\widetilde{\mathcal M} (\Lambda_2) $
be the strict transforms of $\mathcal M (\Lambda_0) $
and $\mathcal M (\Lambda_2) $, respectively.

\end{notation}
\begin {lemma}\label{lemma for getting divisor classes on exceptional divisor}
For any $\Lambda_1\in\mathcal L_{\set {1}} $, there exists an isomorphism $C (\Lambda_1)\cong\mathbb P ^ 1_{\overline\F_q}\times\mathbb P ^ 1_{\overline\F_q} $
such that:
\begin {enumerate}
\item For any $\Lambda_0\in\mathcal L_{\set {0}} $
with $\Lambda_1\subset_1\Lambda_0 $, $\widetilde M (\Lambda_0) $
meets $C (\Lambda_1) $
transversely along a divisor with class $(1, 0) $.
\item For any $\Lambda_2\in\mathcal L_{\set {2}} $
with $\Lambda_2\subset_1\Lambda_1 $, $\widetilde M (\Lambda_2) $
meets $C (\Lambda_1) $
transversely along a divisor with class $(0, 1) $.
\end {enumerate}
\end {lemma}
\begin {proof}
By Theorem \ref{BT stratification first version}(6), we may fix one isomorphism $C (\Lambda_1)\cong\mathbb P ^ 1_{\overline \F_q}\times\mathbb P ^ 1_{\overline\F_q} $. Let $\mathcal L_{\set {0}} (\Lambda_1) $
be the set of $\Lambda_0\in\mathcal L_{\set {0}} $
with $\Lambda_1\subset_1\Lambda_0 $, and likewise $\mathcal L_{\set {2}} (\Lambda_1) $. The actions of $\operatorname {Stab} (\Lambda_1)\subset\operatorname {Sp} (W) (\Q_q) $
on $\mathcal L_{\set {0}} (\Lambda_1) $
and $\mathcal L_{\set {2}} (\Lambda_1) $
are transitive, and compatible with the natural $\operatorname{Sp}(W)(\Q_q)$-action on $\mathcal N(0)$ (see (\ref{V versus W})).
For distinct $\Lambda_0,\Lambda_0'\in\mathcal L_{\set {0}} (\Lambda_1) $, it follows that the divisor classes $$D_{\Lambda_0}\coloneqq\widetilde{\mathcal M}(\Lambda_0)\intersection C (\Lambda_1) $$
and $$D_{\Lambda_0'}\coloneqq\widetilde{\mathcal M}(\Lambda_0')\intersection C (\Lambda_1) $$
differ by an automorphism of $C (\Lambda_1) $. In particular, if $D_{\Lambda_0} = (\alpha,\beta) $, then $D_{\Lambda_0'} = (\alpha,\beta) $
or $(\beta,\alpha) $. On the other hand, since $\mathcal M(\Lambda_0) $
meets $\mathcal M (\Lambda_0') $
transversely at $\mathcal M(\Lambda_1) $, we have $$D_{\Lambda_0}\cdot D_{\Lambda_0} = 0 $$
for the intersection product on $C (\Lambda_1) $. Since $(\alpha,\beta)\cdot (\beta,\alpha) =\alpha ^ 2+\beta ^ 2 $, which can only vanish if $\alpha = \beta = 0$, it follows that $D_{\Lambda_0} = D_{\Lambda_0'} = (\alpha,\beta) $
with $\alpha\beta = 0 $; without loss of generality, assume $\beta = 0 $.
By the same reasoning, for any $\Lambda_2\in\mathcal L_{\set {2}} (\Lambda_1) $, $$D_{\Lambda_2}\coloneqq\widetilde{\mathcal M}(\Lambda_2)\intersection C (\Lambda_1) $$
has divisor class $(\gamma,\delta) $
with $\gamma\delta = 0 $. However, $\mathcal M(\Lambda_0) $
meets $\mathcal M (\Lambda_2) $
transversely along $\mathcal M (\Lambda_0,\Lambda_2) $, so we have $$D_{\Lambda_0}\cdot D_{\Lambda_2} = 1. $$
This implies $\alpha\delta = 1 $, so we have $D_{\Lambda_0} = (1, 0) $
and $D_{\Lambda_2} = (0, 1) $, as desired.
\end{proof}
\begin{notation}\label{notation:O(1) pullback RZ space}
Let $\Lambda\in\mathcal L_{\set {0}}\sqcup\mathcal L_{\set {2}} $.
We denote by $\mathcal O (1) $
the line bundle on $\widetilde{\mathcal M} (\Lambda) $
arising from the pullback of $\mathcal O (1) $
along the composite $$\widetilde{\mathcal M} (\Lambda)\to\mathcal M (\Lambda)\hookrightarrow\mathbb P ^ 3_{\overline\F_q}. $$
\end{notation}
\begin {lemma}\label{calculating normal bundle Lemma}
For any $\Lambda\in\mathcal L_{\set {0}}\sqcup\mathcal L_{\set {2}} $, the normal bundle to $\widetilde{\mathcal M} (\Lambda) $
inside $\widetilde{\mathcal N} (0)$
is $\O (- 2q) $.
\end {lemma}
\begin {proof}
Let $\set {\Lambda_1 ^ {(0)},\ldots,\Lambda_1 ^ {(n)}} $
be the set of lattices $\Lambda_1 $
in $\mathcal L_{\set {1}} $
such that $\mathcal M (\Lambda_1) $
lies on $\mathcal M (\Lambda) $. Then the projection $\widetilde{\mathcal M} (\Lambda)\to\mathcal M (\Lambda) $
is the blowup along the points $\mathcal M (\Lambda_1 ^ {(i)}) $, with exceptional divisors $$E_i\coloneqq C (\Lambda_1 ^ {(i)})\intersection\widetilde{\mathcal M} (\Lambda). $$
Since $\mathcal M (\Lambda) $
is a smooth surface, we have $E_i\cdot E_j = -\delta_{ij} $
for the intersection pairing on $\widetilde{\mathcal M} (\Lambda) $.

Now, the normal bundle to $\widetilde{\mathcal M}(\Lambda)$ inside $\widetilde{\mathcal N}(0)$
is locally free of rank one, and Lemma \ref{lemma isomorphisms of line bundles away from smooth locus} implies it is isomorphic to $$\O (- 2q +\alpha_0E_0+\ldots +\alpha_nE_n) $$
for some $\alpha_i\in\mathbb Z $.
On the other hand, as $\widetilde {\mathcal N}(0)$ is formally smooth, we can compute the triple intersection number
$$m_i =\widetilde{\mathcal M} (\Lambda)\cdot\widetilde{\mathcal M} (\Lambda)\cdot C (\Lambda_1 ^ {(i)}) $$
in two ways, for each $0\leq i\leq n $:
\begin {align*}
m_i & =\left (\widetilde{\mathcal M} (\Lambda)\cdot C (\Lambda_1 ^ {(i)})\right)\cdot_{C (\Lambda_1 ^ {(i)})}\left (\widetilde{\mathcal M} (\Lambda)\cdot C (\Lambda_1 ^ {(i)})\right) = 0\;\;\text {(Lemma \ref{lemma for getting divisor classes on exceptional divisor})}\\
& =\left (\widetilde{\mathcal M} (\Lambda)\cdot\widetilde{\mathcal M} (\Lambda)\right)\cdot_{\widetilde{\mathcal M} (\Lambda)}\left (\widetilde{\mathcal M} (\Lambda)\cdot C (\Lambda_1 ^ {(i)})\right) = -\alpha_i.
\end {align*}
So we find $\alpha_i = 0 $
for all $i $, as desired.
\end {proof}
\subsection {The $\spin $
action on $\mathcal N$} \label{spin action subsection}
\subsubsection {}\label{V versus W}  The endomorphism algebra $\End (W) $
is equipped with an involution $\dagger $
given by the adjoint with respect to $\langle\cdot,\cdot\rangle_\bullet $, and $$V\coloneqq\End (W) ^ {\dagger = 1,\tr = 0} =\End (\mathbb X,\iota_{\mathbb X}) ^ {\ast = 1,\tr = 0} $$
is a split orthogonal space of dimension 5, where $\ast$ denotes the Rosati involution. There is a natural projection \begin {equation}\label{projection pi}\pi:\GSP(W)\to  \SO(V)\end {equation}
inducing an isomorphism $\spin (V)(\Q_q)\cong\GSP(W) (\Q_q)$. There is also a canonical action of $\spin (V) (\Q_q) $
on $\mathcal N $ (by modifying $\rho$); when restricted to $\operatorname {Spin} (V) (\Q_q) $, the resulting action of $\operatorname {Sp} (W) (\Q_q) $
on $\mathcal M $
is compatible with the natural actions on $\mathcal L_{\set {0}} $, $\mathcal L_{\set {2}} $, $\mathcal L_{\set {0 2}} $, and $\mathcal L_{\set {1}} $.

\begin{definition}\label{def:lattices_W}
    Define the sets of lattices
\begin {align*}
\mathscr L & =\set {\Lambda\subset W\,:\,\Lambda = q ^ n\Lambda ^\check\text { for some } n\in\Z}\\
\mathscr L_{\paramodular} & =\set {\Lambda_{\paramodular}\subset W\,:\, q ^ {n +1}\Lambda_{\paramodular} ^\check\subset_2\Lambda_{\paramodular}\subset_2p ^ n\Lambda_{\paramodular} ^\check\text { for some } n\in\Z}.
\end {align*}
\end{definition}
\begin{rmk}
    Both $\mathscr L$ and $\mathscr L_{\paramodular}$ are homogeneous spaces for $\GSP (W) (\Q_q) $; the stabilizer of a point in $\mathscr L $
is a hyperspecial subgroup, and the stabilizer of a point in $\mathscr L_{\paramodular} $
is a paramodular subgroup. 
\end{rmk}

Using $\mathscr L $
and $\mathscr L_{\paramodular} $
rather than $\mathcal L_{\set {0}} $, $\mathcal L_{\set {2}} $, and $\mathcal L_{\set {1}} $, we can extend the combinatorial description of $\mathcal M $
to all of $\mathcal N_{\red} $.
\begin {definition}\label{lattice strata on RZ space definition}
For any $\Lambda\in\mathscr L $, choose an arbitrary $g\in\GSP (W) (\Q_q) $
such that $g\Lambda\in\mathcal L_{\set {0}}\subset\mathscr L $, and define
$$\mathcal M_+ (\Lambda)\coloneqq g ^ {-1}\mathcal M_{\set {0}} (g\Lambda). $$
Similarly, choose $g'\in\GSP (W) (\Q_q) $
such that $g'\Lambda\in\mathcal L_{\set {2}}\subset\mathscr L $, and define $$\mathcal M_- (\Lambda)\coloneqq (g') ^ {-1}\mathcal M_{\set {2}} (g'\Lambda). $$

For any $\Lambda_{\paramodular}\in\mathscr L_{\paramodular} $, choose $g\in\GSP (W) (\Q_q) $
such that $g\Lambda_{\paramodular}\in\mathcal L_{\set {1}}\subset\mathscr L_{\paramodular} $, and define $\mathcal M (\Lambda_{\paramodular}) = g ^ {-1}\mathcal M_{\set {1}} (g\Lambda_{\paramodular}) $.
\end {definition}

\begin {prop}\label{proposition decomposition of all of N red}
\begin {enumerate}
\item Definition \ref{lattice strata on RZ space definition} yields bijections
$$\mathscr L\times\set {\PM}\xrightarrow {\sim}\set {\text {irreducible components of }\mathcal N_{\red}} $$
and
$$\mathscr L_{\paramodular}\xrightarrow {\sim}\set {\text {singular points of }\mathcal N_{\red}}. $$
\item Choose any $\Lambda\in\mathscr L $. For the Weil descent datum in (\ref{subsubsec:Weil_descent}), we have
$$\phi \left (\mathcal M_+ (\Lambda)\right) =\sigma ^\ast\mathcal M_- (q\Lambda) $$
and $$\phi\left (\mathcal M_- (\Lambda)\right) =\sigma ^\ast\mathcal M_+ (\Lambda). $$
\end {enumerate}
\end {prop}
\begin {proof}
For (1), it suffices to show there are two $\spin (V) (\Q_q) $-orbits of irreducible components of $\mathcal N_{\red} $, and only one orbit of singular points. However, since $$g\cdot\mathcal N (i) =\mathcal N (i +\ord_q\nu (g)) $$
for $g\in\spin (V) (\Q_q) $, it suffices to show that there are two $\operatorname {Spin} (V) (\Q_q) $-orbits of irreducible components of $\mathcal M $, and one $\operatorname {Spin} (V) (\Q_q) $-orbit of singular points on $\mathcal M $. This follows from the transitivity of the $\SP (W) (\Q_q) $-actions on $\mathcal L_{\set {0}} $, $\mathcal L_{\set {2}} $, and $\mathcal L_{\set {1}} $.

For (2), note that $\phi ^ 2 (\mathcal M_+ (\Lambda)) = (\sigma ^ 2) ^\ast\mathcal M_+ (q\Lambda) $, so it suffices to show the first relation. Without loss of generality, assume $\Lambda\in\mathcal L_{\set {0}} $. By definition, $\phi (\mathcal M _+(\Lambda)) (\overline\F_q) $, viewed as a subset of $\mathcal N (\overline\F_q) $, is the Zariski closure of the set of points corresponding to lattices $M\subset N $
such that $$(V ^ {-1} (M +\tau M))_\bullet =\breve\Lambda $$
and $$\Pi ^ {-1} M_\circ = qM_\bullet ^\check. $$
Since $V\breve\Lambda =\Pi\breve\Lambda $, the first condition is equivalent to $\Pi ^ {-1} M_\circ +\Pi ^ {-1}\tau M_\circ =\breve\Lambda $, or dually $$M_\bullet\intersection\tau M_\bullet = q\breve\Lambda .$$

Now choose any $g\in\spin (V) (\Q_q) $
with $\nu (g) = q ^ {-1} $; we have $qg\Lambda\in\mathcal L_{\set {2}} $. The locus $$g\phi (\mathcal M_+ (\Lambda)) (\overline\F_q)\subset\mathcal M (\overline\F_q) $$
is the Zariski closure of the set of points corresponding to lattices with $M_\bullet\intersection\tau M_\bullet = qg\breve\Lambda $
and $\Pi ^ {-1} M_\circ = M_\bullet ^\check $. But this is exactly the stratum $\mathcal M_{\set {2}} (qg\Lambda) (\overline\F_q)=\mathcal M_- (qg\Lambda) (\overline\F_q) $, so we conclude $\phi (\mathcal M_+ (\Lambda)) = \sigma^\ast g ^ {-1}\mathcal M_- (qg\Lambda) =\sigma ^\ast\mathcal M_- (q\Lambda) $, as desired.
\end {proof}
From Theorem
\ref {BT stratification first version}, we immediately deduce the following relations among the components $\mathcal M_\PM (\Lambda) $
and the points $\mathcal M (\Lambda_{\paramodular}) $.
\begin {corollary}\label{corollary about intersection combinatorics on RZ space}
For any $\Lambda_{\paramodular}\in\mathscr L_{\paramodular} $
and  $\Lambda,\Lambda'\in\mathscr L $, we have:
\begin {enumerate}
\item $\mathcal M (\Lambda_{\paramodular}) $
lies on $\mathcal M_+ (\Lambda) $
if and only if $\Lambda_{\paramodular}\subset_1\Lambda $.
\item $\mathcal M (\Lambda_{\paramodular}) $
lies on $\mathcal M_- (\Lambda) $
if and only if $\Lambda\subset_1\Lambda_{\paramodular} $.
\item $\mathcal M_+ (\Lambda) $
meets $\mathcal M_- (\Lambda') $
if and only if $q\Lambda\subset_2\Lambda'\subset_2\Lambda $.
\item If $\Lambda \neq \Lambda'$, then for $\delta = +$ or $-$, $\mathcal M_\delta (\Lambda) $
and $\mathcal M_\delta (\Lambda') $ can meet only in singular points of $\mathcal N_{\red} $.
\end {enumerate}
\end {corollary}

\section{The first explicit reciprocity law: geometric inputs}\label{sec:1ERL_geom}
\subsection {Abel-Jacobi maps for schemes with ordinary quadratic singularities}\label{AJ section}
\subsubsection {}
Let $R_0 $
be a Henselian discrete valuation ring with uniformizer $\pi $, algebraically closed residue field $k $, and fraction field $K_0$. The inertia subgroup $I_{K_0} $
of $\Gal (\overline K_0/K_0) $
has the canonical tame character $t_p: I_{K_0}\to\Z_p (1) $
for any $p\neq\operatorname {char} (k) $. Let $R $
be the quadratic extension $R_0 [\pi ^ {1/2}] $, and $K $
its field of fractions, with inertia subgroup $I_K\subset I_{K_0}.$  We write $s_0, s,\eta_0,\eta $
for the closed points and the generic points of $\Spec R_0 $ and $\Spec R $, with corresponding geometric points $\overline\eta_0 = \overline\eta $. Let $X $
be an irreducible scheme of finite type and pure relative dimension $2r - 1 $
over $\Spec R_0 $, for some integer $r\geq 1 $. We assume $X $
has
 {ordinary quadratic singularities}: this means that $X $
is smooth outside a finite set of closed points $\set {x_i}_{i\in I} $
in $X_{s_0} $, and,
\'etale locally near each $x_i $, $X $
is isomorphic to $\Spec R_0 [y_0,\ldots, y_{2r -1}]/(Q -\pi) $, with $Q $
the equation of a smooth quadric in $\mathbb P_{R_0} ^ {2r -1} $.
\subsubsection {}
The blowup $Y $
of $X_R $
at the points $\set {x_i}_{i\in I} $
is strictly semistable in the sense of \cite{saito2003weight}, with a particularly simple form  \cite{illusie2000formule}. The irreducible components of the special fiber $Y_s $
of $Y $
are $\widetilde X_s $, the strict transform of $X_s $, and the exceptional divisors $D_i $. Each $D_i $
is isomorphic to the smooth projective quadric in $\mathbb P^{2r}_k = \operatorname{Proj} (k[y_0, \ldots, y_{2r-1}, t])$ cut out by $Q - t^2$, and 
the intersection $C_i = D_i\intersection\widetilde X_s $
is the hyperplane section $t = 0$, so that $C_i$ is a smooth quadric in $\mathbb P^{2r-1}_k$. Since $Y $
is semistable, we have
\begin {equation}\label{normal bundle to C}\mathcal N_{C_i/\widetilde X_s} = -\mathcal N_{C_i/D_i} =\O (-1)\end {equation}
in the Picard group of $C_i $.
\subsubsection {}
\label{definition of nearby cycles subsubsection}
Let $O$ be a finite flat extension of $\Z_p$ with $p$ odd and $p \neq\Char(k)$, and let $\varpi \in O$ be a uniformizer. We fix a coefficient ring $\Lambda = O $
or $O/\varpi ^ m $
for some $m\geq 1. $
We recall the definition of the nearby cycles complex: let $\overline j: Y_{\overline\eta}\to Y $
and $\overline i: Y_s\to Y $
denote the inclusions of the geometric generic and special fibers, respectively. Then $R\Psi \Lambda = \overline i ^\ast\circ\overline j_\ast\circ\overline j ^\ast\Lambda $, an element of the bounded derived category of sheaves on $Y_s $; it has a canonical action of $I_K\subset I_{K_0} $
factoring through the tame character $t_p $, and the (increasing) monodromy filtration $M_\bullet R\Psi\Lambda $. Fix $T\in I_K $
 such that $t_p (T) $
generates $\Z_p (1) $; then the monodromy operator $T -1 $
on $R\Psi\Lambda $
induces compatible maps $$T -1: M_iR\Psi\Lambda\to M_{i -2} R\Psi\Lambda $$
for all $i\in\Z $.

\begin {prop} [Saito]
\label{Prop graded pieces of nearby cycles}
Let $$i_0:\widetilde X_s\sqcup\bigsqcup_{i\in I} D_i\to Y_s $$
and $$i_1:\bigsqcup_{i\in I} C_i\to Y_s $$
be the natural maps. Then the graded pieces of $M_\bullet R\Psi\Lambda $
are given by:
\begin {align*}
\GR_i ^ MR\Psi\Lambda & = 0, & | i | > 1,\\
\GR_1 ^ MR\Psi\Lambda & = i_{1\ast}\Lambda (-1) [-1],\\
\GR_0 ^ MR\Psi\Lambda & = i_{0\ast}\Lambda,\\
\GR_{-1} ^ MR\Psi\Lambda & = i_{1\ast}\Lambda [-1].
\end {align*}
\end {prop}
\begin {proof}
This is \cite[Proposition 2.2.3]{saito2003weight}.
\end {proof}

\begin {lemma}\label{Identifying M0 lemma}
The following composite map is an isomorphism:
$$H ^ {2r} (Y_s, M_0R\Psi\Lambda)\to H ^ {2r} (Y_s,\GR_0 ^ MR\Psi\Lambda) = H ^ {2r} (\widetilde X_s,\Lambda)\oplus\bigoplus_{i\in I} H ^ {2r} (D_i,\Lambda)\twoheadrightarrow H ^ {2r} (\widetilde X_s,\Lambda). $$
\end {lemma}
\begin {proof}
First, consider the tautological distinguished triangle
\begin{equation}\label{second distinguished triangle eqn}M_{-1} R\Psi\Lambda\to M_0R\Psi\Lambda\to\GR_0 ^ MR\Psi\Lambda\to M_{-1} R\Psi\Lambda [+1]. \end{equation}
This yields an exact sequence
$$H ^ {2r} (Y_k, M_{-1} R\Psi\Lambda)\to H ^ {2r} (Y_k, M_0R\Psi\Lambda)\to H ^ {2r} (Y_k,\GR ^ M_0R\Psi\Lambda)\to H ^ {2r +1} (Y_k, M_{-1} R\Psi\Lambda). $$
Applying Proposition \ref{Prop graded pieces of nearby cycles} and using $H^{2r -1} (C_i,\Lambda) = 0 $
for all $i\in I $, we obtain $$0\to H ^ {2r} (Y_k, M_0R\Psi\Lambda)\to H ^ {2r} (\widetilde X_k,\Lambda)\oplus\bigoplus_{i\in I} H ^ {2r} (D_i,\Lambda)\to H ^ {2r} (C_i,\Lambda). $$
However, $H ^ {2r} (D_i,\Lambda)\to H ^ {2r} (C_i,\Lambda) $
is an isomorphism for each $i\in I $
by the Lefschetz hyperplane theorem, and the lemma follows.
\end {proof}
\subsubsection{}
Let $R\Psi_X \Lambda$ be the nearby cycles complex for $X$, defined as in (\ref{definition of nearby cycles subsubsection}); it also coincides with the nearby cycles complex for $X_R$. From now on, we will assume:
\begin {equation}\label{base change assumption equation}\tag {$\text{BC}_X$}
\text {the base change map } H ^ {i} (X_{\overline\eta_0},\Lambda)\to H ^ {i} (X_{s_0}, R\Psi_X\Lambda)\text { is an isomorphism for all }i.
\end {equation}

Since the blowup map $f: Y\to X_R $
is proper and is an isomorphism on generic fibers, we have a canonical isomorphism $$f_\ast R\Psi\Lambda = R\Psi_X\Lambda $$
by \cite[\S2.1.7]{deligne2006groupes}. In particular, (\ref {base change assumption equation}) implies:
\begin {equation}\tag {$\text {BC}_Y $}
\begin{split}
\text {the base change map } H^i(X_{\overline\eta}, \Lambda) = H ^ {i} (Y_{\overline\eta},\Lambda)\to H ^ {i} (Y_s, R\Psi\Lambda)\text { is an isomorphism for all }i.
\end{split}
\end {equation}



\begin {lemma}\label{lemma key diagram for monodromy calculation}
Let $j:\bigsqcup_{i\in I} C_i\hookrightarrow\widetilde X_s $
be the natural embedding. Then the monodromy operator $T - 1 $
on $H ^ {2r -1} (Y_s, R\Psi\Lambda) $
fits into a commutative diagram with exact rows:
\begin {center}
\begin {tikzcd}
H ^ {2r -1} (X_{\overline \eta}, \Lambda)\arrow [r, "\alpha"]\arrow [d, "T -1"] &\oplus_{i\in I} H ^ {2r-2} (C_i,\Lambda (-1))\arrow [r, "j_\ast"]\arrow [d, "t"] & H ^ {2r} (\widetilde X_s,\Lambda)\arrow[r, "\gamma"]&H^{2r}(X_{\overline \eta}, \Lambda)\\
H ^ {2r -1} (X_{\overline\eta}, \Lambda) &\oplus_{i\in I} H ^ {2r-2} (C_i,\Lambda)\arrow [l, "\beta"] & H ^ {2r-2} (\widetilde X_s,\Lambda)\arrow [l, "j ^\ast"]&H^{2r-2}(X_{\overline \eta}, \Lambda)\arrow[l].
\end {tikzcd}
\end {center}
Here, $t $
is the isomorphism $-\otimes t_p (T) $.
\end {lemma}
\begin {proof}
From the vanishing of $\GR_i ^ MR\Psi\Lambda $
for $| i | >1 $
(Proposition \ref{Prop graded pieces of nearby cycles}), we have a canonical factorization
\begin {equation}\label{factoring monodromy equation}
T -1: R\Psi\Lambda\to\GR_1 ^ MR\Psi\Lambda\to M_{-1} R\Psi\Lambda\to R\Psi\Lambda;
\end {equation}
taking cohomology and applying Proposition \ref{Prop graded pieces of nearby cycles}, (\ref{factoring monodromy equation}) induces a commutative diagram:
\begin {center}
\begin {tikzcd}
H ^ {2r -1} (Y_s, R\Psi\Lambda)\arrow [r, "\alpha"]\arrow [d, "T -1"] &\oplus_{i\in I} H ^ {2r -2} (C_i,\Lambda (-1))\arrow [d, "t"]\\
H ^ {2r -1} (Y_s, R\Psi\Lambda) &\oplus_{i\in I} H ^ {2r - 2} (C_i,\Lambda)\arrow [l, "\beta"].
\end {tikzcd}
\end {center}
The description of $t $
is \cite[Corollary 2.2.4.2]{saito2003weight}.

We now explain the exactness of the top row
\begin {equation}\label{repeating top row of diagram equation}
H ^ {2r -1} (X_{\overline \eta}, \Lambda)\xrightarrow {\alpha}\oplus_{i\in I} H ^ {2r -2} (C_i,\Lambda (-1))\xrightarrow {j_\ast} H ^ {2r} (\widetilde X_s,\Lambda)\to H^{2r}(X_{\overline \eta}, \Lambda)
\end {equation}
of the diagram in the lemma. From the tautological distinguished triangle
\begin {equation}\label{1st distinguished triangle equation}
M_0R\Psi\Lambda\to R\Psi\Lambda\to\GR ^ M_1R\Psi\Lambda\to M_0R\Psi\Lambda [+1],
\end {equation}
we deduce the exact sequence
\begin {equation}
H ^ {2 r -1} (Y_s, R\Psi\Lambda)\xrightarrow {\alpha}\oplus_{i\in I} H ^ {2r -2} (C_i,\Lambda (-1))\xrightarrow {\delta} H ^ {2r} (Y_s, M_0R\Psi\Lambda)\to H^{2r}(Y_s, R\Psi\Lambda).
\end {equation}
Combined with Lemma \ref{Identifying M0 lemma} and ($\text{BC}_Y$), it suffices to show that the composite 
\begin{equation}
    \oplus_{i\in I} H ^ {2r -2} (C_i,\Lambda (-1))\xrightarrow {\delta} H ^ {2r} (Y_s, M_0R\Psi\Lambda)\xrightarrow{\sim}H^{2r}(\widetilde X_s,\Lambda)
\end{equation}
is the pushforward map, and this follows from \cite[Proposition 2.2.6]{saito2003weight}. 
The argument for the exactness of the bottom row
$$H^{2r-2} (X_{\overline\eta}, \Lambda) \to H ^ {2r -2} (\widetilde X_s,\Lambda)\xrightarrow {j ^\ast}\bigoplus_{i\in I} H ^ {2r -2} (C_i,\Lambda)\xrightarrow {\beta} H ^ {2r -1} (X_{\overline \eta}, \Lambda) $$
is essentially identical.
\end {proof}
\subsubsection {} By $(\text{BC}_Y)$, the monodromy filtration $M_\bullet $
of $R\Psi\Lambda $
induces filtrations $M_\bullet$ on $$H ^ {2r -1} (X_{\overline\eta},\Lambda)  $$
and on $$H ^ 1 (I_{K_0}, H ^ {2r -1} (X_{\overline\eta},\Lambda)) = H ^ 1 (I_K, H ^ {2r -1} (X_{\overline\eta},\Lambda) )$$
(using $p\neq 2
$).

\begin {prop}
\label{zeta map injection on ramified cohomology prop}
The diagram in Lemma \ref{lemma key diagram for monodromy calculation} induces an exact sequence
$$0\to M_{-1} H ^ 1 (I_{K_0}, H ^ {2r -1} (X_{\overline\eta},\Lambda))\xrightarrow {\zeta}\frac {H ^ {2r} (\widetilde X_s,\Lambda)} {\image (j_\ast\circ t ^ {-1}\circ j ^\ast)}\xrightarrow {\gamma}\frac {H ^ {2r} (X_{\overline\eta},\Lambda)} {\image (\gamma\circ j_\ast\circ t ^ {-1}\circ j ^\ast)}, $$
such that $\zeta(c) = j_\ast\circ t^{-1} (y)$ for any $y\in \oplus_{i\in I} H^{2r-2}(C_i, \Lambda)$ such that $c(T) = \beta(y)$.  \end {prop}
\begin {proof}
By definition, $M_{-1} H ^ {2r -1} (X_{\overline\eta},\Lambda) $
is the image of $$\beta:\bigoplus_{i\in I} H ^ {2r -2} (C_i,\Lambda)\to H ^ {2r -1} (Y_s, R\Psi\Lambda)\cong H ^ {2r -1} (X_{\overline\eta},\Lambda). $$
Then, by the left half of the commutative diagram in Lemma \ref{lemma key diagram for monodromy calculation}, the map $c\mapsto c (T) $
identifies $$M_{-1} H ^ 1 (I_{K_0}, H ^ {2r -1} (X_{\overline\eta},\Lambda))\simeq\frac {\image\beta} {\image (\beta\circ t\circ\alpha)}. $$
Using the exactness of the rows in Lemma \ref{lemma key diagram for monodromy calculation}, we also have the exact sequence
$$0\to\frac {\image\beta} {\image (\beta\circ t\circ\alpha)}\xrightarrow {j_\ast\circ t ^ {-1}}\frac {H ^ {2r} (\widetilde X_s,\Lambda)} {\image (j_\ast\circ t ^ {-1}\circ j ^\ast)}\xrightarrow {\gamma}\frac {H ^ {2r} (X_{\overline\eta},\Lambda)} {\image (\gamma\circ j_\ast\circ t ^ {-1}\circ j ^\ast)} , $$
and the proposition follows.
\end {proof}
\subsubsection {}
Let $\CH ^ r (X_{\eta_0}) ^ 0 $
denote the Chow group of cohomologically trivial algebraic cycles of pure codimension $r $. We have the Abel-Jacobi map $$\partial_{\operatorname {AJ}}:\CH ^ r (X_{\eta_0}) ^ 0\to H ^ 1 (I_{K_0}, H ^ {2r -1} (X_{\overline\eta},\Lambda (r))), $$
which we are now ready to compute in terms of the geometry of the special fiber of $X $.

If $Z_{\eta_0} $
is a closed irreducible subvariety of $X_{\eta_0} $, then write $Z $
for its Zariski closure in $X $, $Z_Y $
for the strict transform of $Z_R $
under the blowup $Y\to X_R $, and $Z_{Y_s} $
for $Z_Y\times_YY_s $. The intersection $Z_{Y_s}\times_Y\widetilde X_s $
is $\widetilde Z_{s} $, the strict transform of $Z_{s_0} $
under the blowup $\widetilde X_s\to X_{s_0} $. Extending this construction linearly, for any algebraic cycle $z _\eta =\Sigma n_jZ_\eta ^ {(j)} $
of pure codimension $r $, we obtain a codimension-$r $
algebraic cycle $$\widetilde z_s =\Sigma n_j\widetilde Z_s ^ {(j)} $$
on $\widetilde X_s $.
\begin {thm}\label{theorem conclusion of AJ section}
Let $z_{\eta_0} $
be an algebraic cycle of codimension $r $
on $X_{\eta_0} $, whose class in $\CH ^ r (X_{\eta_0}) $
is cohomologically trivial, and assume (\ref {base change assumption equation}).  
Then $\partial_{\operatorname {AJ}} (z) $
lies in $M_{-1} H ^ 1 (I_{K_0}, H ^ {2r -1} (X_{\overline\eta},\Lambda (r))) $. If, for each irreducible component $Z_\eta $
of the support of $z_\eta $, $Z_{Y_s} $
is generically smooth, then $$\zeta\left (\partial_{\operatorname {AJ}} (z)\right)\in\frac {H ^ {2r} (\widetilde X_s,\Lambda (r))} {j_\ast\circ t ^ {-1}\circ j ^\ast\left (H ^ {2r -2} (\widetilde X_s,\Lambda (r -1))\right)} $$
coincides with the algebraic cycle class of $\widetilde z_s $, where $\zeta$ is the map from Proposition \ref{zeta map injection on ramified cohomology prop}.
\end {thm}
\begin {proof} In the proof of \cite[Theorem 2.18]{liu2019bounding} there is constructed:
\begin {itemize}
\item An element $F $
of the bounded derived category of abelian sheaves on $Y_s $
fitting into a commutative diagram:
\begin {equation}\label{commutative diagram for the sheaf F}
\begin {tikzcd}
F\arrow [r, "T -1"]\arrow [d] & R\Psi\Lambda\arrow [d]\\
R\Psi\Lambda/M_0R\Psi\Lambda\arrow [r, "T -1"]& R\Psi\Lambda/M_{-2} R\Psi\Lambda.\end {tikzcd}\end {equation}
\item A class $[z ^\sharp]'_0\in H ^ {2r -1} (Y_s, F (r)) $
such that $\partial_{\operatorname {AJ}} (z) |_{I_K} $
is represented by the cocycle that factors through $t_p: I_{K}\twoheadrightarrow\Z_p (1) $
and satisfies $$\partial_{\operatorname {AJ}} (z) (T) = (T -1) [z ^\sharp]'_0\in H ^ {2r -1} (Y_s, R\Psi\Lambda (r)) = H ^ {2r -1} (X_{\overline\eta},\Lambda (r)). $$
\end {itemize}
In our context, since $\GR_i ^ MR\Psi\Lambda = 0 $
for $| i | >1 $, the diagram (\ref {commutative diagram for the sheaf F}) becomes
\begin {center}
\begin {tikzcd}
F\arrow [r, "T -1"]\arrow [d] & R\Psi\Lambda\\
R\Psi\Lambda/M_0R\Psi\Lambda\arrow [r, "T -1"] &M_{-1} R\Psi\Lambda.\arrow [u]
\end {tikzcd}
\end {center}
In particular, $(T -1) H ^ {2r -1} (Y_s, F (r)) $
lies inside $M_{-1} H ^ {2r -1} (X_{\overline\eta},\Lambda (r)), $
and by construction $\zeta\left (\partial_{\operatorname {AJ}} (z)\right) $
is represented by the image of $[z ^\sharp]'_0 $
under the composite map
\begin {equation}\begin {split}
H ^ {2r -1} (Y_s, F (r))\to H ^ {2r -1} (Y_s,\GR_1 ^ MR\Psi\Lambda (r))\to H ^ {2r} (Y_s,\GR_0 ^ MR\Psi\Lambda (r))\\
= H ^ {2r} (\widetilde X_s,\Lambda (r))\oplus\bigoplus_{i\in I} H ^ {2r} (D_i,\Lambda (r)\twoheadrightarrow H ^ {2r} (\widetilde X_s,\Lambda (r)).\end {split}\end {equation}
By \cite[Proposition 2.19]{liu2019bounding}, under the generic smoothness assumption of the theorem, this image is exactly the cycle class of $\widetilde z_s $.
\end {proof}
\subsubsection{}
We conclude this section with a related lemma.
\begin {lemma} \label{lemma compactly supported on special fiber}
In addition to (\ref{base change assumption equation}), assume
\begin {equation}\label{base change assumption equation compact}\tag {$\text{BC}_{X,c}$}
\text {the base change map } H ^ {i}_c (X_{\overline\eta_0},\Lambda)\to H ^ {i}_c (X_{s_0}, R\Psi_X\Lambda)\text { is an isomorphism for all }i.
\end {equation}

Suppose $\mathcal H$ is a commutative $O$-algebra (not necessarily finitely generated) of correspondences on $X$, such that the singular locus of $X$ is stable under $\mathcal H$. Let $\m\subset \mathcal H$ be a maximal ideal such that the natural map induces an isomorphism
$$H^i_{\et,c}(X_{\overline\eta}, \Lambda)_\m \to H^i_{\et} (X_{\overline\eta}, \Lambda)_\m$$ for all $i$. Then for all $i$, the natural map
$$H^i_{\et,c}(\widetilde X_s, \Lambda)_\m \to H^i_{\et}(\widetilde X_s, \Lambda)_\m$$ is an isomorphism as well. 
\end{lemma}
\begin {proof}
The tautological distinguished triangle (\ref{1st distinguished triangle equation}) gives a commutative diagram of long exact sequences
\begin{center}
\begin{tikzcd}
 \cdots \arrow[r]  &   H^i_c(Y_s, M_0R\Psi\Lambda)_\m\arrow[r] \arrow[d] & H^i_c(Y_s,  R\Psi\Lambda)_\m  \arrow[r]\arrow[d,"\sim" {anchor=south, rotate=90}] &
    H^i_c  (Y_s, \operatorname{gr}_1^M R\Psi\Lambda)_\m\arrow[r]\arrow[d, "\sim" {anchor=south, rotate=90}] & \cdots\\
    \cdots \arrow[r]  &   H^i(Y_s, M_0R\Psi\Lambda) _\m\arrow[r] & H^i(Y_s,  R\Psi\Lambda)_\m
    \arrow[r] &
    H^i  (Y_s, \operatorname{gr}_1 ^MR\Psi\Lambda)_\m\arrow[r]& \cdots\\
\end{tikzcd}
\end{center}
The first marked isomorphism is by (\ref{base change assumption equation}) and (\ref{base change assumption equation compact}), and the second is by Proposition \ref{Prop graded pieces of nearby cycles} and the compactness of each $C_i$. By the five lemma, we have an isomorphism 
$$H^i_c(Y_s, M_0 R\Psi\Lambda)_\m \isomorphism H^i(Y_s, M_0 R\Psi\Lambda)_\m$$ for all $i$. Arguing similarly with the distinguished triangle (\ref{second distinguished triangle eqn}), we find 
a natural isomorphism
$$H^i_c(Y_s, \operatorname{gr}_0 ^M R\Psi\Lambda)_\m \isomorphism H^i(Y_s, \operatorname{gr}_0^ M R\Psi\Lambda)_\m $$ for all $i$, which implies the lemma by Proposition \ref{Prop graded pieces of nearby cycles} once again. 
\end {proof}

\subsection {Semistable reduction of $\spin_5 $ Shimura varieties}
\subsubsection {}\label{setup for SST}
Let $D \neq 1$ be a squarefree product of an even number of primes, and fix an odd prime $q|D$. With $V_D$ as in (\ref{subsubsec:B_D_notation}), we suppose fixed a $q$-adic uniformization datum $(\ast, A_0, \iota_0, \lambda_0, i_D, i_{D/q})$ for $V_D$ (Definition \ref{def:unif_datum_general}(\ref{def:unif_datum_general_2})); we will choose this uniformization datum more precisely in Construction \ref{constr:X_diamond_1ERL} below. 
Let $\mathcal D$ and $\mathscr D$ be the associated PEL data and self-dual $q$-integral refinement from Definition 
\ref{def:O_Dtriple}.
\subsubsection {}\label{subsubsec:levels_1ERL} For the entirety of this section, we fix a neat level subgroup $$K ^ q =\prod_{\l\neq q} K_\l\subset\spin (V_D) (\A_f ^ {q}) \cong \prod_{\l\neq q} \spin(V_{D/q})(\A_f^q),$$ with the isomorphism arising from Remark \ref{rmk:unif_datum}(\ref{rmk:unif_datum_two});
then we obtain a flat, quasi-projective scheme $X =\mathcal M_{K ^ q} $
over $\Z_{(q)} $
representing the PEL-type moduli problem defined by $\mathscr D $
at level $K ^ q $.
The generic fiber $X_\Q $
of $X $
is isomorphic to $\Sh_{K ^qK_q^\ramified} (V_D) $, where $K_q^\ramified\subset\spin (V_D) (\Q_q) $
is a paramodular subgroup in the sense of Notation \ref{notation:paramodular_subgroup}.
\begin {lemma}\label{nearby cycles isomorphism}
Let $R\Psi_XO $
denote the nearby cycles complex on $X_{\overline\F_q} $. Then the natural maps $$H ^ i_{c,\et} (X_{\overline\Q}, O)\to H ^ i_{c,\et} (X_{\overline\F_q}, R\Psi_XO) $$
and $$H ^ i_{\et} (X_{\overline\Q}, O)\to H ^ i_{\et} (X_{\overline\F_q}, R\Psi_XO) $$
are isomorphisms for all $i $.
\end {lemma}

\begin {proof}
This is a special case of \cite[Theorem 6.8]{lan2018nearby}.
\end {proof}
Let $O_D\subset B_D$ be the unique maximal $\Z_{(q)}$-order. We now take the $q $-divisible group $\mathbb X\coloneqq \overline A_0 [q ^\infty] $, with its induced polarization and $(O_D\otimes\Z_q) $-action, to be the base point for the Rapoport-Zink space $\mathcal N $
(Definition \ref{definition RZ space}).
\begin {thm}\label{rz uniformization}
\leavevmode
\begin {enumerate}
\item Let $\mathcal X $
be the completion of $X $
along the supersingular locus $X ^ {\ss}_{\overline\F_q} $. Then we have a canonical isomorphism $$\mathcal X\cong\spin (V_{D/q}) (\Q)\backslash\spin (V_{D/q}) (\A_f ^ {q})\times\mathcal N/K ^ q, $$
where $\spin (V_{D/q}) (\Q)\hookrightarrow\spin (V_{D/q}) (\Q_q) $
acts on $\mathcal N $
as described in \S\ref{spin action subsection}.
\item The singular locus of $X_{\overline\F_q} $
is the discrete set of points $$X ^ {\sing}_{\overline\F_q} =\spin (V_{D/q}) (\Q)\backslash\spin (V_{D/q}) (\A_f ^ {q})\times\mathcal N ^ {\sing}_{\red}/K ^ q. $$
The complete local ring of $X $
at any point of $X ^ {\sing}_{\overline\F_q} $
is isomorphic to $$\breve\Z_q\llbracket x, y, z, w\rrbracket/(xy - zw - q). $$
\end{enumerate}
\end {thm}
\begin {proof}
Part (1) is the Rapoport-Zink uniformization theorem for $X $; part (2) follows from (1) and Theorem \ref{BT stratification first version}(\ref{thm:BT_6}), after noting that all singularities of $X_{\overline\F_q} $
lie in the supersingular locus by \cite[Theorem 7.5]{oki2022supersingular}.
\end {proof}
\subsubsection {}
In particular, Theorem \ref{rz uniformization}(2) asserts that $X $
has ordinary quadratic singularities, so that the results of \S\ref{AJ section} apply. Following the notation therein, let $\widetilde X_{\overline\F_q} $
be the blowup of $X_{\overline\F_q} $
along the singular locus; 
note that $\widetilde X_{\overline\F_q} $
inherits an action of the full prime-to-$q $
Hecke algebra.
\subsection {Tate classes}
The goal of this subsection is to show that the full cohomology group $H ^ 2_{\et} (\widetilde X_{\overline\F_q}, O) $
is generated by Tate classes from the supersingular locus, after a Hecke localization.

\begin{notation}   \label{notation:wonky_Sh_sets}
Recall the sets  $\mathscr L$  and $\mathscr L_{\paramodular}$ from   Definition \ref{def:lattices_W}.
\begin{enumerate}
    \item\label{item:notation_wonky_1}  For \begin {equation*}
g = (g ^ q,\Lambda)\in\spin (V_{D/q}) (\Q)\backslash\spin (V_{D/q}) (\A_f ^ {q})\times\mathscr L/K ^ q,
\end {equation*}
  let $B_{\PM} (g) $
be the image of $(g ^ q,\mathcal M_\PM (\Lambda)) $
under the uniformization in Theorem \ref{rz uniformization}, and let $\widetilde B_{\pm}(g)$ be its strict transform under the blowup $\widetilde X_{\overline\F_q} \to X_{\F_q}$. 
\item\label{notation:wonky_2} For \begin {equation*}
g = (g ^ q,\Lambda_{\text {Pa}})\in\spin (V_{D/q}) (\Q)\backslash\spin (V_{D/q}) (\A_f ^ {q})\times\mathscr L_{\text {Pa}}/K ^ q,
\end {equation*}
let $y (g)\in X ^ {\sing}_{\overline\F_q} $
be the image of $(g ^ q,\mathcal M (\Lambda_{\text {Pa}})) $, and let $C(g)$ be the exceptional divisor of $\widetilde X_{\overline \F_q}$ over the point $y(g)$.
\item \label{notation_wonky_3}Recall that $\mathscr L $
and $\mathscr L_{\text {Pa}} $
are homogeneous spaces for $\spin (V_{D/q}) (\Q_q)\cong\GSP_4 (\Q_q) $, with the stabilizer of any point a hyperspecial or paramodular subgroup, respectively. We will therefore abbreviate the two sets in (\ref {item:notation_wonky_1}) and (\ref {notation:wonky_2}) 
by $\Sh_{K ^qK_q} (V_{D/q}) $
and $\Sh_{K ^qK_q ^ {\paramodular}} (V_{D/q}), $ respectively,
even though the identifications actually depend on a non-canonical choice of base point which we do not need to make.
\end{enumerate}
\end{notation}
\begin{rmk}
Notation \ref{notation:wonky_Sh_sets} identifies $\Sh_{K^qK_q}(V_{D/q})\times\set{\pm}$ and $\Sh_{K^qK_q^\paramodular}(V_{D/q})$ with the set of irreducible components of $X_{\overline\F_q}^{\ss}$ and the set of points of $X^{\sing}_{\overline\F_q}$, respectively. 
\end{rmk}
To study the intersections of the divisors $\widetilde B_\pm (g) $
and $C (g) $, we need to
 define some additional Hecke operators.
\begin {definition}
Let $$\delta_\PM:O[\Sh_{K ^qK_q ^{\paramodular}} (V_{D/q})]\to O [\Sh_{K ^qK_q} (V_{D/q})] $$
be the maps defined by
$$\delta_+: (g ^ q,\Lambda_{\paramodular})\mapsto\sum_{\substack {\Lambda_{\paramodular}\subset_1\Lambda\\\Lambda\in\mathscr L}} (g ^ q,\Lambda), $$
$$\delta_-: (g ^ q,\Lambda_{\paramodular})\mapsto\sum_{\substack {\Lambda\subset_1\Lambda_{\paramodular}\\\Lambda\in\mathscr L}} (g ^ q,\Lambda). $$
Similarly, let $$\theta_\PM: O [\Sh_{K ^qK_q} (V_{D/q})]\to O [\Sh_{K ^qK_q ^{\paramodular}} (V_{D/q})]$$ be the maps defined by 
$$\theta_+\,:\, (g^q, \Lambda) \mapsto \sum _{\substack {\Lambda_{\paramodular}\subset_1\Lambda\\\Lambda_\paramodular\in\mathscr L_\paramodular}} (g^q, \Lambda_\paramodular), \;\;\;\;\theta_-\,:\, (g^q, \Lambda) \mapsto \sum _{\substack {\Lambda\subset_1\Lambda_{\paramodular}\\\Lambda_\paramodular\in\mathscr L_\paramodular}} (g^q, \Lambda_\paramodular).$$
\end {definition}
These are incarnations of the level-lowering and level-raising operators in \cite[\S3]{roberts2007local}.
\begin {definition}
We define the natural composite maps
\begin {multline}
\inc ^\ast: H ^ 4_{\et} (\widetilde X_{\overline\F_q}, O (2))\;\;\longrightarrow\\\bigoplus_{g\in\Sh_{K ^qK_q ^{\paramodular} }(V_{D/q})} H ^ 4_{\et} (C (g), O (2))\;\;\;\;\oplus\bigoplus_{g\in\Sh_{K ^qK_q} (V_{D/q})}\left( H ^ 4_{\et} (\widetilde B_+ (g), O (2))\oplus H ^ 4_{\et} (\widetilde B_- (g), O (2))\right)\\
\xrightarrow {\sim} O\left [\Sh_{K ^qK_q ^{\paramodular}} (V_{D/q})\right]\oplus O\left [\Sh_{K ^qK_q} (V_{D/q})\right] ^ {\oplus 2}
\end {multline}
and
\begin {multline}
\inc_{c,\ast}: O\left [\Sh_{K ^qK_q ^ {\paramodular}} (V_{D/q})\right]\oplus O\left [\Sh_{K ^qK_q} (V_{D/q})\right] ^ {\oplus 2}\xrightarrow {\sim}
\\
\bigoplus_{g\in\Shimura_{K ^qK_q ^ {\paramodular}} (V_{D/q})} H ^ 0_{\et} (C (g), O)\oplus\bigoplus_{g\in\Shimura_{K ^qK_q} (V_{D/q})}\left (H ^ 0_{\et} (\widetilde B_+ (g), O)\oplus H ^ 0_{\et} (\widetilde B_- (g), O)\right)\\
\longrightarrow H ^ 2_{c,\et} (\widetilde X_{\overline\F_q}, O (1)).
\end {multline}
We also denote by $\inc_c ^\ast$
the composite $$\inc_c ^\ast: H ^ 4_{c,\et} (\widetilde X_{\overline\F_q}, O (2))\rightarrow H ^ 4_{\et} (\widetilde X_{\overline\F_q}, O (2))\xrightarrow {\inc ^\ast}O\left [\Sh_{K ^qK_q ^ {\paramodular}} (V_{D/q})\right]\oplus O\left [\Sh_{K ^qK_q} (V_{D/q})\right] ^ {\oplus 2} $$
and likewise by $\inc_\ast$
the composite
$$\inc_\ast:O\left [\Sh_{K ^qK_q ^ {\paramodular}} (V_{D/q})\right]\oplus O\left [\Sh_{K ^qK_q} (V_{D/q})\right] ^ {\oplus 2}\xrightarrow {\inc_{c,\ast}} H ^ 2_{c,\et} (\widetilde X_{\overline\F_q}, O (1))\longrightarrow H ^ 2_{\et} (\widetilde X_{\overline\F_q}, O (1)). $$
\end {definition}

\begin{notation}For $\alpha\in O\left [\Shimura_{K ^qK_q} (V_{D/q})\right] $
and $g\in\Shimura_{K ^qK_q} (V_{D/q}) $, let $m (g;\alpha)\in O $
denote the coefficient of $g $
in $\alpha $, and similarly for $\Shimura_{K ^qK_q ^ {\paramodular}} (V_{D/q}) $.
\end{notation}
\begin {lemma}\label{incidence map for lines on C}
Fix $g\in \Sh_{K^qK_q^\paramodular}(V_{D/q})$.
\begin {enumerate}
\item There exists an isomorphism $C (g)\cong\mathbb P^1_{\overline\F_q}\times\mathbb P^1_{\overline\F_q} $
with the following property: for $h\in\hypernormal $, the intersection $\widetilde B_+ (h)\cdot C (g) $
has the cycle class $(m (h;\delta_+ (g)), 0) $
on $C (g) $; and the intersection $\widetilde B_- (h)\cdot C (g) $
has the cycle class $(0, m (h;\delta_- (g))) $.
\item Let $\iota_g: C (g)\hookrightarrow\widetilde X_{\overline\F_q} $
be the natural inclusion. If we fix an isomorphism $C (g)\cong\mathbb P^1_{\overline\F_q}\times\mathbb P^1_{\overline\F_q} $
as in (1), then we have
$$\inc ^\ast(\iota_g)_\ast[(1, 0)] = (- g, 0,\delta_- (g)) $$
and $$\inc ^\ast(\iota_g)_\ast[(0, 1)] = (- g,\delta_+ (g), 0). $$
\end {enumerate}
\end {lemma}
\begin {proof}
(1) is immediate from Lemma
\ref {lemma for getting divisor classes on exceptional divisor} and Corollary \ref{corollary about intersection combinatorics on RZ space}; (2) follows from (1), using that the Chern class of the normal bundle on $C (g) $
is $(-1,-1)$
(by (\ref {normal bundle to C})) to compute the first coordinate.
\end {proof}
\begin{notation}For any $g\in\hypernormal $, let\begin {equation*} [\O_{\widetilde B_\PM (g)} (1)]\in H ^ 4_{\et,c} (\widetilde X_{\overline\F_q}, O (2))\end {equation*}
denote the pushforward of the class of the line bundle $\O (1) $
on $\widetilde B_\PM (g) $ (which is the one induced by Notation \ref{notation:O(1) pullback RZ space}).
\end{notation}
Recall the explicit Hecke algebra generators from (\ref{subsubsec:GSP4_hecke}). 
\begin {lemma}\label{incidence map for O(1)}
For all $g\in \hypernormal$, we have $$\inc_c ^\ast[\O_{\widetilde B_+ (g)} (1)] = (0, - 2q (q +1)\cdot g, T_{q, 2}\cdot g) $$
and $$\inc_c ^\ast [\O_{\widetilde B_- (g)} (1)] = (0,\langle q\rangle ^ {-1} T_{q, 2}\cdot g, - 2q (q +1)\cdot g). $$
\end {lemma}
\begin {proof}
The two calculations are similar, so we consider $\left [\O_{\widetilde B_+  (g)} (1)\right] $. We must calculate the intersection pairings with divisor classes $\left [C (h)\right] $
and $\left [\widetilde B_\PM (g')\right] $, for $h\in\paranormal $
and $g'\in\hypernormal $.

Since $C (h)\cdot\widetilde B_+ (g) $
always lies in the exceptional divisor of the blowup $\widetilde B_+ (g)\to B_+ (g) $, we have
\begin {equation}
\left [\O_{\widetilde B_+ (g)} (1)\right]\cdot\left [C (h)\right] =\left [\O (1)\right]\cdot_{\widetilde B_+ (g)}\left [C (h)\cdot\widetilde B_+ (g)\right] = 0.\end {equation}
Now, $\widetilde B_+ (g) $
and $\widetilde B_+ (g') $
meet only if $g = g' $, in which case we find
\begin {equation}
\left [\O_{\widetilde B_+ (g)} (1)\right]\cdot\left [\widetilde B_+ (g)\right] = [\O (1)]\cdot_{\widetilde B_+ (g)}\left [\mathcal N_{\widetilde B_+ (g)/\widetilde X_{\overline\F_q}}\right] = - 2q (q +1),
\end {equation}
since the normal bundle to $\widetilde B_+ (g) $
is $\O (- 2q) $ (Lemma
\ref{calculating normal bundle Lemma})
and $B_+ (g) $
has degree $q +1 $ (Theorem
\ref{BT stratification first version}(\ref{thm:BT_one})).
Finally, we compute
\begin {equation}
\left [\O_{\widetilde B_+ (g)} (1)\right]\cdot\left [\widetilde B_- (g')\right] =\left [\O (1)\right]\cdot_{\widetilde B_+ (g)}\left [\widetilde B_+ (g)\cdot\widetilde B_- (g')\right] = m (g'; T_{q, 2}\cdot g),
\end {equation}
since $B_+ (g)\cap B_- (g') $
consists of $m (g'; T_{q, 2}\cdot g)$
linearly embedded copies of $\mathbb P^1_{\overline\F_q} $ inside $B_+ (g)\subset\mathbb P^3_{\overline\F_q} $ (by Theorem \ref{BT stratification first version}(\ref{thm:BT_two}) and Corollary \ref{corollary about intersection combinatorics on RZ space}).
Combining these calculations gives $$\inc_c ^\ast\left [\O_{\widetilde B_+ (g)} (1)\right] = (0, - 2q (q +1)\cdot g, T_{q, 2}\cdot g), $$
as desired.
\end {proof}
\begin {definition}\label{def:weakly_q_gen}
Fix a finite set $S$ of places of $\Q$ containing  $q$, and all primes $\l\neq q$ such that $K_\l$ is not hyperspecial.   
A maximal ideal $\m\subset\T ^ S_O$
will be called
\emph {weakly} $q $\emph {-generic} if the map
$$\langle q\rangle ^ {-1} T_{q, 2} ^ 2 - 4q^ 2 (q +1) ^ 2: O\left [\hypernormal\right]_\m\to O\left [\hypernormal\right]_\m $$
is an isomorphism.
\end {definition}
 \begin {remark}\label{rmk:weakly_q_generic}
 Using (\ref{eq:hecke_eigenvalues_satake}), one calculates the following:
if $\pi$ is a relevant automorphic representation of $\GSP_4(\A)$ unramified at $q$ with trivial central character and $\iota:\overline \Q_p \isomorphism \C$ is any isomorphism with $p \neq q$, then $\langle q\rangle^{-1} T_{q,2}^2 - 4q^2(q+1)^2$ acts on the spherical vector of $\pi_q$ with eigenvalue
$$q^2\left(\iota\tr(\Frob_q |V_{\pi,\iota})^2 - 4(q+1)^2\right).$$
\end {remark}
\begin {lemma}\label{surjectivity of incidence map}
Let $\m $
be a weakly $q $-generic maximal ideal of $\T ^ S_{O} $. Then the map $$\inc ^\ast_{c,\m}: H ^ 4_{c,\et} (\widetilde X_{\overline\F_q}, O (2))_\m\to O\left [\paranormal\right]_\m\oplus O\left [\hypernormal\right] ^ {\oplus 2}_\m $$
is surjective.
\end {lemma}
We note that Lemma
\ref {surjectivity of incidence map} is slightly stronger than the corresponding statement for $\inc ^\ast_\m $.
\begin {proof}
By Lemma \ref{incidence map for lines on C}, it suffices to show that the image of $\inc ^\ast_{c,\m} $
contains $O\left [\hypernormal\right]_\m ^ {\oplus 2} $. Define a map
\begin {equation}
\mu = (\mu_+,\mu_-): O\left [\hypernormal\right]_\m ^ {\oplus 2}\to H ^ 4_{\et, c} (\widetilde X_{\overline\F_q}, O (2))_\m
\end {equation}
by linearly extending
\begin {equation}
\mu_\PM (g) =\left [\O_{\widetilde B_\PM (g)} (1)\right].
\end {equation}
Then by Lemma \ref{incidence map for O(1)}, the composite $\inc ^\ast_{c,\m}\circ\mu $
is given by the matrix map $$\begin {pmatrix} 0 & 0\\- 2q (q +1) &\langle q\rangle ^ {-1} T_{q, 2}\\T_{q, 2} & - 2q (q +1).\end {pmatrix} $$
Since $$\debt\begin {pmatrix} - 2q (q +1) &\langle q\rangle ^ {-1} T_{q, 2}\\
T_{q, 2} & - 2q (q +1)\end {pmatrix} = 4q ^ 2 (q +1) ^ 2 -\langle q\rangle ^ {-1} T ^ 2_{q, 2} $$
and $\m $
is weakly $q $-generic, we have $$\image (\inc ^\ast_{c,\m})\supset\image (\inc ^\ast_{c,\m}\circ\mu) = O\left [\hypernormal\right] ^ {\oplus 2}_\m, $$
as desired.
\end {proof}
In Theorem \ref{Tate classes main result}, we will see that $\inc ^\ast_{c,\m} $
also has torsion kernel.
\begin {lemma}\label{lemma cohomology generated by stuff from C}
Let $\m $
be a generic and non-Eisenstein maximal ideal of $\T^S_O $. There is a canonical injection induced by pullback $$H ^ 2 (\widetilde X_{\overline\F_q}, O(1))_\m\hookrightarrow\bigoplus_{g\in\paranormal} H ^ 2_{\et} (C (g), O(1))\simeq O\left [\paranormal\right]_\m ^ {\oplus 2} $$
and a canonical surjection induced by pushforward $$O\left [\paranormal\right]_\m ^ {\oplus 2}\simeq\bigoplus_{g\in \paranormal} H ^ 2_{\et} (C (g), O (1))_\m\twoheadrightarrow H ^ 4_{\et} (\widetilde X_{\overline\F_q}, O(2))_\m. $$
\end {lemma}
\begin {proof}
This follows from  Lemma \ref{lemma key diagram for monodromy calculation} and Theorem \ref{thm:generic}(\ref{part:thm_generic_two}).
\end {proof}
\begin {corollary} [Ihara's Lemma]\label{Ihara}
Let $\m $
be a generic, non-Eisenstein, and weakly $q $-generic maximal ideal of $\T ^ S_{O} $. Then the degeneracy map $$(\delta_+, \delta_-): O\left [\paranormal\right]_\m\longrightarrow O\left [\hypernormal\right]_\m ^ {\oplus 2} $$
is surjective.
\end {corollary}
\begin {proof}
Combining Lemmas \ref{lemma cohomology generated by stuff from C} and \ref{surjectivity of incidence map}, we see that the composite map
\begin{equation*}
\begin{split}
O\left [\paranormal\right]_\m ^ {\oplus 2}\simeq\oplus H ^ 2_{\et} (C (g), O (1))_\m\longrightarrow H_{\et, c} ^ 4 (\widetilde X_{\overline\F_q}, O (2))_\m\xrightarrow {\inc ^\ast_{c,\m}}\\ O\left [\paranormal\right]_\m\oplus O\left [\hypernormal\right] ^ {\oplus 2}_\m 
\end{split}
\end{equation*}
is surjective. On the other hand, by Lemma
\ref {incidence map for lines on C}, this composite is given as a matrix by $$\begin {pmatrix} -1 & 0 & \delta_- \\ -1 & \delta_+ & 0 \end{pmatrix}, $$
and the corollary follows by restricting to the preimage of $O\left [\hypernormal\right]_\m ^ {\oplus 2}. $
\end {proof}
\begin {thm}\label{Tate classes main result}
Let $\m $
be a generic, non-Eisenstein, and weakly $q $-generic
maximal ideal of $\T ^ S_O $. Then $$\inc ^\ast_{\m}: H_{\et} ^ 4 (\widetilde X_{\overline\F_q}, O (2))_\m\to O\left [\paranormal\right]_\m\oplus O\left [\hypernormal\right]_\m ^ {\oplus 2} $$
and $$\inc_{\ast,\m}: O\left [\paranormal\right]_\m\oplus O\left [\hypernormal\right]_\m ^ {\oplus 2} \to H^2_{\et} (\widetilde X_{\overline\F_q}, O(1))_\m$$
are both surjective. Moreover $\inc_{\ast,\m} $
is injective, and $\inc ^\ast_{\m} $
is injective modulo $O $-torsion.
\end {thm}
In fact, only the surjectivity of $\inc_{\ast,\m} $
is needed for the main result.
\begin {proof}
We claim that it suffices to show
\begin {equation}\label{dimension equality goal}\dim H ^ 4_{\et} (\widetilde X_{\overline\F_q}, \overline\Q_p)_\m \leq \dim\left (\overline\Q_p\left [\paranormal\right]_\m\right) + 2\dim\left (\overline\Q_p\left [\hypernormal\right]_\m\right);\end {equation}
indeed, this combined with Lemma \ref{surjectivity of incidence map} implies that $\inc ^\ast_{\m} $
is injective modulo torsion as well as surjective, and the other assertions follow by duality along with Lemma \ref{lemma compactly supported on special fiber} and Theorem \ref{thm:generic}(\ref{part:thm_generic_two}). (We also use that $H ^ 2 (\widetilde X_{\overline\F_q}, O)_\m $
is $O $-torsion-free by Lemma
\ref {lemma cohomology generated by stuff from C}.) 
Inspecting the diagram in Lemma \ref{lemma key diagram for monodromy calculation} and using Theorem \ref{thm:generic}(\ref{part:thm_generic_two}), we see that 
$$\dim H^4(\widetilde X_{\overline \F_q}, \overline \Q_p)_\m = 2 \dim \left(\overline\Q_p\left[\paranormal\right]_\m\right) - \operatorname{rank} \left(T- 1| H^3_\et\left(X_{\overline\Q}, \overline \Q_p\right)_\m\right),$$ where $T \in I_{\Q_q}$ is a generator of tame inertia, so we wish to show
\begin{equation}
    \dim \left(\overline\Q_p\left[\paranormal\right]_\m\right) \leq 2 \dim \left(\overline\Q_p\left[\hypernormal\right]_\m\right) +
    \operatorname{rank} \left(T- 1| H^3_\et\left(X_{\overline\Q}, \overline \Q_p\right)_\m\right).
\end{equation}
Fix an isomorphism $\iota: \overline\Q_p \isomorphism \C$. Applying
 Lemma \ref{lem:yucky_eisenstein_lemma}, it suffices to show
\begin{equation}\label{eq:numbered_ineq_goal}
\begin{split}
    \dim\left( \overline \Q_p\left[\paranormal\right]\left[\iota^{-1}\pi_f^q\right]\right) \leq 2 \dim \left(\overline\Q_p\left[\hypernormal\right]\left[\iota^{-1}\pi_f^q\right] \right)+ \\
    \operatorname{rank} \left(T- 1| H^3_\et\left(X_{\overline\Q}, \overline \Q_p\right)\left[\iota^{-1}\pi_f^q\right]\right)
    \end{split}
\end{equation}
for all relevant automorphic representations $\pi$ of $\spin(V_{D/q})(\A)$ such that $\pi^{K^qK_q^\paramodular}_f \neq 0$ and the Hecke action on $\iota^{-1}\pi^{K^qK_q^\paramodular}_f$ factors through $\T^S_{O,\m}$. 


By Lemma \ref{lem:IIa_new}(\ref{part:IIa_type},\ref{part:IIa_completion}), $\pi_q $
is uniquely determined by $\pi ^ q_f $, and is either spherical or of type IIa.  Assume first that $\pi_q $
is spherical. Then $\pi_f ^ q $
cannot be completed to a relevant automorphic representation of $\spin (V_D)(\A) $ (by Corollary \ref{cor:JL_general}), so by Corollary \ref{cor:coh_relevant} the final term in (\ref{eq:numbered_ineq_goal}) vanishes. Since $\dim \pi_q^{K_q} = 1$ and $\dim \pi_q^{K_q^\paramodular} = 2$ by \cite[Table A.13]{roberts2007local}, both sides of (\ref{eq:numbered_ineq_goal}) are 2.

On the other hand, suppose that $\pi_q $
is of type IIa.

\textbf {Case 1:} $\pi $
is non-endoscopic. Then $\pi_f ^ q $
can be completed to a relevant automorphic representation $\pi' $
of $\spin (V_D)(\A) $
by Theorem \ref{thm:JL}.
By Lemma \ref{lem:IIa_new}(\ref{part:IIa_Galois}), $\pi'_q $
has a unique fixed vector for $K_q^\ramified\subset \spin (V_D) (\Q_q) $, so by Corollary \ref{cor:coh_relevant},
$$H ^ 3_{\et} (X_{\overline\Q},\overline\Q_p)\left [\iota^{-1}\pi_f ^ q\right] = \rho_{\pi,\iota}(-2) \otimes \iota^{-1}(\pi_f ^ q) ^ {K ^ q}. $$
By Lemma \ref{lem:IIa_new}(\ref{part:IIa_Galois}), $\rho_{\pi,\iota}$ is tamely ramified at $q $
with monodromy of rank one. On the other hand, $\pi_q$ has a unique paramodular fixed vector, and $\pi$ has automorphic multiplicity one for $\spin(V_{D/q})(\A)$ by Theorem \ref{thm:JL}(\ref{part:JL_four}). So in this case we see that both sides of (\ref{eq:numbered_ineq_goal}) are 1. 

\textbf {Case 2:} $\pi $
is endoscopic, associated to a pair of cuspidal automorphic representations $(\pi_1,\pi_2) $
of $\GL_2 $
with discrete series archimedean components of weights 2 and 4, respectively. Theorem \ref{thm:endoscopic_packets} implies that there exist (uniquely determined) quaternion algebras $B_1 $
and $B_2 $
such that $\pi $
is the theta lift $\Theta (\pi_1 ^ {B_1}\boxtimes\pi_2 ^ {B_2}) $, with $\pi_i ^ {B_i} $
the Jacquet-Langlands transfers. Moreover, $B_1\otimes\R $
is ramified and $B_2\otimes\R $
is split. Since $\pi_q $
is of type IIa with a paramodular fixed vector, we can conclude from Lemma \ref{lem:IIa_new}(\ref{part:IIa_Galois}) and Theorem \ref{thm:rho_GL2_LLC}(\ref{part:rho_GL2_LLC_1}) that exactly one of $\pi_{i, q} $
is a twist of a Steinberg representation, and the other is unramified. Let $B_i' $
be the quaternion algebras obtained from $B_i $
by changing invariants at $q $
and $\infty $. Then $\pi_f ^ q $
has a unique completion to an automorphic representation of $\spin (V_D) $, which is
\begin {equation}
\begin {cases}
\Theta (\pi_1 ^ {B_ 1'}\boxtimes\pi_2 ^ {B_2}), &\pi_{1, q}\text { ramified},\\
\Theta (\pi_1 ^ {B_1}\boxtimes\pi_2 ^ {B_2'}), &\pi_{2, q}\text { ramified.}
\end {cases}
\end {equation}
We therefore have (applying Corollary \ref{cor:coh_relevant})
$$H_{\et} ^ 3 (X_{\overline\Q},\overline\Q_p )\left [\pi_f ^ q\right]\cong (\pi_f ^ q) ^ {K ^ q}\otimes\rho, $$
with $$\rho =\begin{cases}\rho_{\pi_1,\iota}(-2), &\pi_{1, q}\text { ramified},\\
\rho_{\pi_2,\iota}(-2), &\pi_{2, q}\text { ramified.}
\end{cases}$$
In particular, the monodromy at $q $
has rank one in either case, so again both sides of (\ref{eq:numbered_ineq_goal}) are 1.\end {proof}

\subsection {Level-raising and potential map}\label{level raising section}
\begin {lemma}\label{composites of Hecke operators paramodular}
The Hecke operators $\theta_\PM $, $\delta_\PM $
satisfy:
\begin {align*}
\delta_+\circ\theta_+ & =\delta_-\circ\theta_- = T_{q, 1} + (q +1) (q ^ 2+1)\\
\delta_-\circ\theta_+ & = (q +1) T_{q, 2}\\
\delta_+\circ\theta_- & =\langle q\rangle ^ {-1} (q +1) T_{q, 2}.
\end {align*}
\end{lemma}
\begin {proof}
By definition, $\delta_+\circ\theta_+ $
is induced by
\begin {align*}
(g ^ q,\Lambda) &\mapsto\delta_+\left (\sum_{\substack {\Lambda_{\paramodular}\subset_1\Lambda\\\Lambda_{\paramodular}\in\mathscr L_{\paramodular}}} (g ^ q,\Lambda_{\paramodular})\right)\\
& =\sum_{\substack {\Lambda_{\paramodular}\subset_1\Lambda'\\\Lambda'\in\mathscr L}}\sum_{\substack {\Lambda_{\paramodular}\subset_1\Lambda\\\Lambda_{\paramodular}\in\mathscr L_{\paramodular}}} (g ^ q,\Lambda')\\
& =\sum_{\Lambda'\in\mathscr L} e (\Lambda,\Lambda') (g ^ q,\Lambda'),
\end {align*}
where $$e (\Lambda,\Lambda') =\#\set {\Lambda_{\paramodular}\in\mathscr L_{\paramodular}\,:\,\Lambda_{\paramodular}\subset_1\Lambda,\,\Lambda_{\paramodular}\subset_1\Lambda'}. $$
If $e (\Lambda,\Lambda')\neq 0 $, then either $\Lambda =\Lambda' $
or $\Lambda'\in T_{q, 1}\cdot\Lambda $. In the latter case, $\Lambda_{\paramodular} =\Lambda\intersection\Lambda' $
is uniquely determined, so $e (\Lambda,\Lambda') = 1. $
On the other hand, $e (\Lambda,\Lambda) $
is the number of lattices $\Lambda_{\paramodular}\subset_1\Lambda $, or equivalently the number of rational 3-planes in the symplectic space $\Lambda/q\Lambda $. Thus $$e (\Lambda,\Lambda) =\#\mathbb P^3 (\F_q) = (q +1) (q ^ 2+1). $$
This shows $$\delta_+\circ\theta_+ = T_{q, 1} + (q +1) (q ^ 2+1), $$
and the calculation for $\delta_-\circ\theta_- $
is similar. We now compute $\delta_-\circ\theta_+ $, which is induced by
\begin {align*}
(g ^ q,\Lambda) &\mapsto\delta_-\left (\sum_{\substack {\Lambda_{\paramodular}\subset_1\Lambda\\\Lambda_{\paramodular}\in\mathscr L_{\paramodular}}} (g ^ q,\Lambda_{\paramodular})\right)\\
& =\sum_{\substack {\Lambda_{\paramodular}\subset_1\Lambda\\\lambda_{\paramodular}\in\mathscr L_{\paramodular}}}\sum_{\substack {\Lambda'\subset_1\Lambda_{\paramodular}\\\Lambda'\in\mathscr L}} (g ^ q,\Lambda')\\
& =\sum_{\Lambda'\in\mathscr L} e' (\Lambda,\Lambda') (g ^ q,\Lambda'),
\end {align*}
where $$e' (\Lambda,\Lambda') =\#\set {\Lambda_{\paramodular}\in\mathscr L_{\paramodular}\,:\,\Lambda'\subset_1\Lambda_{\paramodular}\subset_1\Lambda}. $$
If $e' (\Lambda,\Lambda')\neq 0 $, then $\Lambda'\in T_{q, 2}\cdot\Lambda. $
On the other hand, given $\Lambda'\in T_{q, 2}\cdot\Lambda $, then the choices of $\Lambda_{\paramodular} $
with $\Lambda'\subset_1\Lambda_{\paramodular}\subset_1\Lambda $
are in bijection with rational lines in the 2-dimensional $\F_q $-vector space $\Lambda/\Lambda' $; hence $$e' (\Lambda,\Lambda') =\#\mathbb P^1 (\F_q) = q +1. $$
This shows $\delta_-\circ\theta_+ = (q +1) T_{q, 2} $, and the computation of $\delta_+\circ\theta_- $
is similar.
\end {proof}
\subsubsection {}
Recall from Lemma
\ref {lemma key diagram for monodromy calculation}
the natural embedding $$j:\bigsqcup C (g)\hookrightarrow\widetilde X_{\overline\F_q}. $$
\begin {lemma}\label{matrix composite for level raising}
The composite map \begin{equation*}
\begin{split}\inc ^\ast\circ j_\ast\circ j ^\ast\circ\inc_\ast: O\left [\paranormal\right]\oplus O\left [\hypernormal\right] ^ {\oplus 2}\longrightarrow \\O\left [\paranormal\right]\oplus O\left [\hypernormal\right] ^ {\oplus 2} \end{split}
\end{equation*}
is given by the matrix
$$\begin {pmatrix} 2 & -\theta_+ & -\theta_-\\-\delta_+ & 0 & (q +1)\langle q\rangle ^ {-1} T_{q, 2}\\-\delta_- & (q +1) T_{q, 2} & 0.\end {pmatrix} $$
\end{lemma}
\begin {proof}
We begin by calculating $\inc ^\ast\circ j_\ast\circ j ^\ast\left [C (g)\right] $, for $g\in\paranormal $. Let $\iota_g: C (g)\hookrightarrow\widetilde X_{\overline\F_q} $
be the natural embedding. Since the $C (g) $
are all disjoint,
$$j _\ast\circ j^\ast\left [C (g)\right]\in H ^ 4_{\et, c} (\widetilde X_{\overline\F_q}, O (2)) $$
is the pushforward of the class of the normal bundle, i.e. $\iota_{g\ast} [(-1, -1)] $
in the notation of Lemma
\ref {incidence map for lines on C}. Then $$\inc ^\ast\iota_{g\ast} [(-1, -1)] = -\inc ^\ast\iota_{g\ast} [(1, 0)] -\inc ^\ast\iota_{g\ast} [(0, 1)] = (2\cdot g, -\delta_+ (g), -\delta_- (g)) $$
by Lemma
\ref {incidence map for lines on C} (2), which gives the first column of the matrix.

For the second column, we must calculate $\inc ^\ast\circ j_\ast\circ j ^\ast\left [\widetilde B_+ (g)\right] $, for $g\in\hypernormal $.
By Lemma
\ref {incidence map for lines on C} (1), the class $$j_\ast\circ j ^\ast\left [\widetilde B_+ (g)\right]\in H ^ 4_{\et} (\widetilde X_{\overline\F_q}, O (2)) $$
is $$\sum_{h\in\paranormal} m (h;\theta_+ (g))\iota_{h\ast}\left [(1, 0)\right]. $$
Then by Lemma
\ref {incidence map for lines on C} (2), $$\inc ^\ast\circ j_\ast\circ j ^\ast\left [\widetilde B_+ (g)\right] =\left (-\theta_+ (g), 0,\delta_-\circ\theta_+ (g)\right). $$
By Lemma
\ref{composites of Hecke operators paramodular}, $\delta_-\circ\theta_+ = (q +1) T_{q, 2}, $
so this gives the second column of the matrix; the third column is similar.
\end {proof}

\begin {definition}\label{potential map definition}
 Define the potential map $$\nabla: H ^ 4_{\et} (\widetilde X_{\overline\F_q}, O (2))\to O\left [\hypernormal\right] $$
as the composite
$$H ^ 4_{\et} (\widetilde X_{\overline\F_q}, O (2))\xrightarrow {\inc ^\ast} O\left [\paranormal\right]\oplus O\left [\hypernormal\right] ^ {\oplus 2}\xrightarrow {M} O\left [\hypernormal\right], $$
with $M $
the matrix map $$\begin {pmatrix}\delta_+ +\delta_- & 2 & 2\end {pmatrix}. $$
\end {definition}
\begin {definition}\label{level raising Hecke operators definition}
The level-raising Hecke operator $\funnyT_q^\lr $
is defined by $$\funnyT_q^\lr\coloneqq T_{q, 1} + (q +1) (q ^ 2+ 1) - 
T_{q, 2} (q +1). $$
\end {definition}
\begin{rmk}\label{rmk:interpret_funny_T}
Using (\ref{eq:hecke_eigenvalues_satake}), one calculates the following: 
  If $\pi$ is a relevant automorphic representation of $\GSP_4(\A)$ unramified at $q$ with trivial central character and $\iota: \overline\Q_p \isomorphism \C$ is any isomorphism  with $p \neq q$, then $\funnyT_q^\lr$ acts on the spherical vector of $\pi_q$ with eigenvalue $$q^{-1} \iota\det(\Frob_q - q|V_{\pi,\iota}).$$
\end{rmk}
\begin {thm}
\label{main arithmetic level raising semistable}
Let $\m $
be a generic, non-Eisenstein, and weakly $q $-generic maximal ideal of $\T ^ S_{O} $. The composite map $$\nabla\circ j_\ast\circ t ^ {-1}\circ j ^\ast:\, H ^ 2_{\et} ( X_{\overline\Q}, O (2))_\m\to O\left [\hypernormal\right]_\m $$
has image contained in $$\left (\langle q\rangle ^ {-1} -1,\funnyT_q^\lr\right)\cdot O\left [\hypernormal\right]_\m. $$
In particular, $\nabla\circ\zeta $
gives a well-defined surjection
$$\nabla\circ\zeta :\, M_{-1} H ^ 1\left (I_{\Q_q}, H ^ 3_{\et} (X_{\overline\Q}, O (2))_\m\right)\twoheadrightarrow\frac {O\left [\hypernormal\right]_\m} {\left (\langle q\rangle ^ {-1} -1,\funnyT_q^\lr\right)},$$
where $\zeta$ is the map from Proposition
\ref {zeta map injection on ramified cohomology prop}.
\end {thm}
\begin {proof}
By Theorem \ref{Tate classes main result}, the image of $\nabla\circ j_\ast\circ t ^ {-1}\circ j ^\ast$
coincides with the image of $$\nabla\circ j_\ast\circ j ^\ast\circ\inc_{\ast,\m}: O\left [\paranormal\right]_\m\oplus O\left [\hypernormal\right)_\m ^ {\oplus 2}\to O\left [\hypernormal\right]_\m. $$
By Lemma \ref{matrix composite for level raising}, this is the composite of matrix maps
\begin{equation*}
\begin{split}
\begin {pmatrix}\delta_+ +\delta_- & 2 & 2\end {pmatrix}\circ\begin {pmatrix} 2 & -\theta_+ & -\theta_-\\-\delta_+ & 0 & (q +1)\langle q\rangle ^ {-1} T_{q, 2}\\-\delta_- & (q +1) T_{q, 2} & 0\end {pmatrix} \\=\begin {pmatrix} 0 & -\funnyT_q^\lr & -\funnyT_q^\lr +\left (\langle q\rangle ^ {-1} -1\right) (q +1) T_{q, 2}\end {pmatrix}\end{split}\end{equation*}
(using Lemma \ref{composites of Hecke operators paramodular}    to compute $\delta_\PM\circ\theta_\PM $).
\end {proof}
\subsection {Siegel cycles on the special fiber}
\begin{definition}\label{def:L_siegel}
Recall the set $\mathscr L$ from Definition \ref{def:lattices_W}.
We define $$\mathscr L_{\Siegel}\coloneqq\set {\text {pairs } (\Lambda_+,\Lambda_-)\in\mathscr L ^ 2\,:\, q\Lambda_+\subset_2\Lambda_-\subset_2\Lambda_+}. $$
\end{definition}
\begin{notation}
Note that $\mathscr L_{\Siegel} $
is a homogeneous space for $\spin (V_{D/q}) (\Q_q) $, where the stabilizer of any point is a Siegel parahoric subgroup. As in Notation \ref{notation:wonky_Sh_sets}, we abbreviate
\begin {equation}
\Siegelnormal =\spin (V_{D/q}) (\Q)\backslash\spin (V_{D/q}) (\A_f ^ q)\times\mathscr L_{\Siegel}/K ^qK_q ^ {\Siegel},
\end {equation}
although the identification depends on a choice of base point of $\mathscr L_{\Siegel}$ which is not necessary for our discussion.
\end{notation}

\begin {definition}\label {Siegel degeneracy maps definition}
\leavevmode
\begin {enumerate}
\item We define degeneracy maps $$\delta_\pm ^ {\Siegel}:\mathscr L_{\Siegel}\to\mathscr L $$
by $$\delta_\PM ^ {\Siegel} (\Lambda_+,\Lambda_-) =\Lambda_\pm. $$
\item We define the operator $$\theta_{\Siegel} ^ {\paramodular}:\mathscr L_{\Siegel}\to \Z\left [\mathscr L_{\paramodular}\right] $$
by $$\theta_{\Siegel} ^ {\paramodular} (\Lambda_+,\Lambda_-) =\sum_{\substack {\Lambda_{\paramodular}\in\mathscr L_{\paramodular}\\\Lambda_-\subset\Lambda_{\paramodular}\subset\Lambda_+}}\left [\Lambda_{\paramodular}\right]. $$
\item We extend these maps linearly to $$\delta_\pm ^ {\Siegel}: O\left [\Siegelnormal\right]\to O\left [\hypernormal\right] $$
and $$\theta_{\Siegel} ^ {\paramodular}: O\left [\Siegelnormal\right]\rightarrow O\left [\paranormal\right]. $$
\end {enumerate}
\end {definition}
\begin{notation}
For each $g\in\Siegelnormal $, let $$D (g) = B_+ (\delta_1 ^ {\Siegel} (g))\intersection B_- (\delta_2 ^ {\Siegel} (g)),$$
which is a closed subscheme of $X_{\overline\F_q} $ isomorphic to
$\mathbb P^1_{\overline\F_q} $.
 We write  $\widetilde D (g)\hookrightarrow\widetilde X_{\overline\F_q} $
for the strict transform of $D(g)$ under the blowup $\widetilde X_{\overline \F_q} \to X_{\overline \F_q}$, and   $[\widetilde D (g)]\in H ^ 4_{\et} (\widetilde X_{\overline\F_q}, O (2)) $ for its algebraic
  cycle class.
\end{notation}
\begin {lemma}\label{potential map on Siegel cycle lemma}
For $g\in\Siegelnormal $, we have $$\nabla\left [\widetilde D (g)\right] = (\delta_+ +\delta_-)\circ\theta_{\Siegel} ^ {\paramodular} (g) - 4q\delta_+ ^ {\Siegel} (g) - 4q\delta_- ^ {\Siegel} (g). $$
\end {lemma}
\begin {proof}
We first calculate $\inc ^\ast\left [\widetilde D (g)\right]. $
For $g'\in\hypernormal $, we have
$$\left [\widetilde D (g)\right]\cdot\left [\widetilde B_+ (g')\right] =\left [\widetilde B_+ (\delta_1 ^ {\Siegel} (g))\right]\cdot\left [\widetilde B_- (\delta_2 ^ {\Siegel} (g))\right]\cdot\left [\widetilde B_+ (g')\right] = 0 $$
unless $g' =\delta_1 ^ {\Siegel} (g) $, in which case the intersection number is $- 2q $
(cf. the proof of Lemma
\ref{incidence map for O(1)}).
Similarly, $$\left [\widetilde D (g)\right]\cdot\left [\widetilde B_- (g')\right] =\begin {cases} - 2q, & g' =\delta_2 ^{\Siegel} (g),\\0, &\text {else}.\end {cases} $$
Now consider the intersections with $C (h) $, for $h\in\paranormal $. We see from Lemma
\ref{incidence map for lines on C} that $\widetilde D (g) $
meets $C (h) $
transversely with multiplicity $m (h;\theta_{\Siegel} ^ {\paramodular} (g)) $. Hence
$$\inc ^\ast\left [\widetilde D (g)\right] =\left (\theta_{\Siegel} ^ {\paramodular} (g), - 2 q\delta_1 ^ {\Siegel} (g), - 2q\delta_2 ^ {\Siegel} (g)\right). $$
The claimed formula then follows from the formula for  $\nabla $
in Definition \ref {potential map definition}.
\end{proof}
\subsection {Special cycles on ramified $\spin_5$ Shimura varieties}
 The goal of this section is to compute the local ramification of Abel-Jacobi images of special cycles $Z(T,\phi)$ on the generic fiber of $X$, by applying the results of \S\ref {AJ section}. However, our first task is to make a good choice of the uniformization datum from (\ref{setup for SST}). 
\begin{notation}\label{notation:sst_sp_new}
Fix a matrix $T\in \symmetric_ 2 (\Z_{(q)})_{\geq 0} $
such that $$T\quiver\begin {pmatrix} 0 &\alpha\\\alpha & 0\end {pmatrix}\pmod q, $$
for some $\alpha\in\F_q ^\times .$
Then we make the following notations. 
\begin{enumerate}
    \item     Let $V_\special$ be the two-dimensional quadratic space over $\Q_q$ with  basis $\set{e_1^\special, e_2^\special}$ and pairing matrix given by $T$.
    \item Let $B_\special$ be the quaternion algebra $C(V_\special)$, with its natural positive nebentype involution $\ast$. 
    \item Let $d $
be the discriminant of the unique quaternion algebra $B_d $
such that $B_D\otimes B_\special\simeq M_2 (B_d)$, and let $O_d\subset B_d$ be the unique maximal $\Z_{(q)}$-order.
\end{enumerate}
\end{notation}
 
\begin{rmk}\label{rmk:B_d_indef}
    Because $B_\special$ is split at $q$ and $\infty$, $B_d$ is ramified at $q$ and split at $\infty$. In particular, $O_d$ is well-defined. 
\end{rmk}
\begin {prop}\label{integral embedding of quaternion algebras}
Fix a nebentype involution $\ast$ on $O_D$, of unit type. Then there exists an isomorphism $$\beta: B_D\otimes B_\special\xrightarrow {\sim} M_2 (B_d) $$
such that:
\begin {enumerate}
\item $\beta (O_D)\subset M_2 (O_d) $.
\item The induced involution $\dagger $
on $M_2 (B_d) $
stabilizes $M_2 (O_d) $, and is of non-unit type.
\item If $\Pi_D\in O_D $
and $\Pi_d\in  O_d $
are uniformizers, then $$\beta (\Pi_D)\cdot\begin {pmatrix}\Pi_d ^ {-1} & 0\\0 &\Pi_d ^ {-1}\end {pmatrix} $$
lies in $\GL_2 (O_d) $.
\item The $\Z_{(q)} $-lattice $\Lambda_\special\coloneqq M_2 (O_d) ^ {\dagger = 1,\trace = 0}\intersection B_\special\subset V_\special $
has basis $\set {e_1 ^\special, qe_2 ^\special} $.
\end {enumerate}
\end {prop}
\begin {proof}
Let $$X =\operatorname {Isom} (B_D\otimes B_\special, M_2 (B_d)), $$
viewed as algebraic variety over $\Q $; $X $
is a (split) torsor for the algebraic group $\GL_2 (B_d) $. Note that all the conditions of the proposition can be checked after tensoring with $\Z_q $, and define an open subset $U\subset X (\Q_q) $
in the $q $-adic topology. Since $X (\Q) $
is dense in $X (\Q_q) $, it suffices to show $U\neq\emptyset $.

Since the involution $\ast$ on $B_D$ is nebentype and of unit type, we can fix  a unit $j \in O_D^\times$ such that $j^{\ast_D} = - j^{\ast_D}$ and $\alpha^\ast = j(\alpha^{\ast_D})j^{-1}$ for $\alpha\in B_D$ (recall $\ast_D$ is the canonical involution on $B_D$). 
Also choose a uniformizer $\Pi\in O _D$ satisfying $\tr \Pi=0$ and  $\Pi j = -j\Pi$, and let $K \coloneqq \Q(\Pi)\subset B_D$. We have a decomposition 
$$B_D = K \oplus j\cdot K,$$ which 
defines an embedding $$\iota: B_D\hookrightarrow M_2 (K)\hookrightarrow M_2 (B_D), $$  satisfying\ $\alpha(O_D)\subset M_2 (O_D) $. Now let $\dagger $
be the non-unit type involution on $O_D$
defined by $$j ^\dagger = j,\;\Pi ^\dagger =\Pi,\;\text {and}\, (\Pi j) ^\dagger = -\Pi j. $$
We extend $\dagger $
to an involution of non-unit type on $M_2 (B_D) $
by $$\begin {pmatrix}\alpha &\beta\\\gamma &\delta\end {pmatrix}\mapsto\begin {pmatrix}\alpha ^\dagger & -\gamma ^\dagger j ^ 2\\-\beta ^\dagger/j ^ 2 &\delta ^\dagger\end {pmatrix}. $$
A simple calculation shows $\iota (\alpha) ^\dagger =\iota(\alpha ^\ast), $
for all $\alpha\in B_D $. Moreover, the centralizer $Z $
of $\iota (B_D) $
inside $M_2 (B_D) $
satisfies $$L\coloneqq Z\intersection M_2 (O_D) ^ {\dagger = 1,\trace = 0} =\Z_{(q)}\cdot\begin {pmatrix}\Pi & 0\\0 &\Pi\end {pmatrix}\oplus\Z_{(q)}\cdot\begin {pmatrix} 0 &\Pi j\\\Pi j ^ {-1} & 0\end {pmatrix}.$$
We have a natural quadratic form 
$x \mapsto x^2$ on $L$, 
which is represented by  $\begin {pmatrix} q\alpha & 0\\0 & - q\alpha\end {pmatrix} $ in the basis above, 
 for a unit $\alpha\in \Z_{(q)}^\times$. Let $L_\Q = L \otimes_{\Z_{(q)}} \Q$.

We may then choose the following two isomorphisms:
\begin{enumerate}[label = (\roman*)]
    \item An isomorphism $\beta_{1,q}:  Z\otimes \Q_q = C(L_\Q) \otimes \Q_q \isomorphism C(V_\special) \otimes\Q_q= B_\special \otimes \Q_q $ compatible with the involutions, 
such that $\beta_{1,q}(L\otimes \Z_q)\subset V_\special\otimes \Q_q $
is the lattice spanned by $e_1 ^\special$ and $ qe_2 ^\special $.
\item An isomorphism $\beta_{2,q}: M_2(O_D) \otimes \Z_q \isomorphism M_2(O_d)\otimes \Z_q$. 
\end{enumerate}. One checks readily that the induced isomorphism $$\beta_q: (B_D\otimes\Q_q)\otimes (B_\special\otimes\Q_q)
\xrightarrow[\sim]{\operatorname{id} \otimes \beta_{1,q}^{-1}} B_D \otimes Z\otimes \Q_q \xrightarrow[\sim]{\iota} M_2(B_D) \otimes \Q_q \xrightarrow[\sim]{\beta_{2,q}} M_2(B_d)\otimes \Q_q$$
lies in $U $, so indeed $U\neq\emptyset $.\end {proof}

Now we use Proposition \ref{integral embedding of quaternion algebras} to construct some particular integral models of special cycles on $X_\Q = \Sh_K(V_D)$. 
\begin{construction}\label{constr:X_diamond_1ERL}
    Fix once and for all a positive, unit type involution on $B_D$, and a choice of $\beta$ as in Proposition \ref{integral embedding of quaternion algebras} for this choice. We will now also write $\ast$ for the induced involution on $M_2(B_d)$. 
\begin{enumerate}
\item\label{constr:X_diamond_unifdatum}
By Corollary \ref{cor:exists_basepoint} and Remark \ref{rmk:B_d_indef}, we fix an abelian scheme $A_0$ over $\breve \Z_q$ with supersingular reduction, 
equipped with an embedding $\iota_0^\diamond: M_2 (O_d)\hookrightarrow\End (A_0)\otimes\Z_{(q)} $
and a polarization $\lambda_0:A_0 \to A_0^\vee$ 
such that $$ \iota_0^\diamond(\alpha)^\vee \circ \lambda_0 = \lambda_0 \circ \iota_0^\diamond(\alpha^\ast).$$ We  choose our $q$-uniformization datum for $V_D$ in (\ref{setup for SST}) to be of the form $(\ast, A_0,\iota_0 ^\diamond\circ\beta, \lambda_0, i_D, i_{D/q})$.
With notation as in Definition \ref{def:unif_datum_general}, we also obtain the PEL datum $\mathcal D^\diamond = (M_2(B_d), \dagger, H, \psi)$, with self-dual $q $-integral refinement $\mathscr D ^\diamond = (M_2 (O_d),\dagger,\Lambda, \psi) $.
\item\label{constr:X_diamond_1ERL_orthogonaldecomps} We  observe that, following Remark \ref{rmk:unif_datum}(\ref{rmk:unif_datum_two}), our choice of $q$-adic uniformization datum defines an orthogonal decomposition
\begin{equation*}
    V_D \xrightarrow[\sim]{i_D}\End(H, B_D)^{\dagger = 1, \tr = 0} = V_\special \oplus V_d^\diamond
\end{equation*}
where $$V_d^\diamond \coloneqq \End(H, M_2(B_d))^{\dagger = 1, \tr = 0}.$$
\item 
Now fix an element $$g ^ q =\product_{\l\neq q} g_\l\in\spin (V_D) (\A ^ {q}_f). $$
Let $$K ^\diamond_{\l}\coloneqq g_\l K_\l g_\l ^ {-1}\intersection\spin (V_d ^\diamond) (\Q_\l)$$
 for each $\l\neq q $; we define $$K ^ {q\diamond} =\product_{\l\neq q} K ^\diamond_{\l},$$ a neat compact open subgroup by \cite[\S0.1]{pink1990arithmetical}. We let $ \boldsymbol Z(g^q, V_\special, V_D)$ be the $\Z_{(q)}$-scheme representing the moduli functor $\mathcal M^\diamond_{K^{q\diamond}}$ associated to $\mathscr D^\diamond$ at level $K^{q\diamond}$, with  special fiber $Z(g^q, V_\special, V_D)_{\F_q}$.
 \end{enumerate}
\end{construction}
\begin{rmk}\label{rmk:X_diamond_integral_model}
Note that (by Corollary \ref{cor:changing_polarizations_PEL} combined with Proposition \ref{prop:positive involutions}), $ \boldsymbol Z(g^q, V_\special, V_D)$ is the usual integral model of the quaternionic Shimura curve associated to $B_d$ at level $K^{q\diamond}$, with maximal compact level structure at $q$.  
\end{rmk}
\subsubsection{}
We have the obvious forgetful finite morphism  \begin {equation}\label{eq:forgetful_1ERL} j:\boldsymbol Z (g^q, V_\special, V_D)\rightarrow X\end {equation}
 defined on the moduli problems by $(A,\iota^\diamond,\lambda,\eta ^ q)\mapsto (A,\iota^\diamond\circ\beta,\lambda, g ^ q\cdot\eta ^ q). $
\begin{lemma}
The generic fiber of (\ref{eq:forgetful_1ERL})
 coincides with the special cycle
 $$Z(g^q, V_\special, V_D)_{K^qK_q^\ram} \to \Sh_{K^qK_q^\ram}(V_D).$$
\end{lemma}

\begin {proof}
See the proof of \cite[Proposition 2.5]{kudla2000siegel}.
\end {proof}
\begin{notation}
We now consider the special fiber $Z(g^q, V_\special, V_D)_{\F_q}$.  
    \begin{enumerate}
         \item Similarly to Construction \ref{constr:X_diamond_1ERL}(\ref{constr:X_diamond_1ERL_orthogonaldecomps}), we obtain a natural orthogonal decomposition 
\begin {equation*}
V_{D/q} \xrightarrow[\sim]{i_{D/q}} \End ^ 0 (\overline A_0, B_D) ^ {\dagger = 1,\trace = 0} = V_\special \oplus  V_{d/q} ^\diamond,
\end {equation*}
with
\begin {equation*}
V_{d/q} ^\diamond\coloneqq\End ^ 0 (\overline A_0, M_2 (B_d)) ^ {\dagger = 1,\trace = 0}
\end {equation*}
a three-dimensional quadratic space whose Hasse invariant coincides with that of $B_{d/q} $. 
\item Let $\mathcal N^\diamond$ be the Rapoport-Zink space  parametrizing framed, polarized deformations of the $q$-divisible group $A_0[q^\infty]$, with action of $M_2(O_d\otimes \Z_q)$. (The details are analogous to Definition \ref{definition RZ space}). Let $j^{\loc}: \mathcal N^\diamond \hookrightarrow \mathcal N$ be the natural closed immersion induced by $\beta\otimes \Z_q: O_D\otimes \Z_q \hookrightarrow M_2(O_d\otimes \Z_q)$, and let  \begin {equation*}
\mathcal N ^\diamond =\sqcup_{i\in\Z}\mathcal N ^\diamond (i)\end {equation*}
be the decomposition defined analogously to (\ref{decomposition of RZ space open and closed}), or equivalently defined by $\mathcal N ^\diamond (i) =\mathcal N ^\diamond\intersection\mathcal N (i). $
    \end{enumerate}
\end{notation}

\begin{rmk}\label{rmk:N_diamond}
    By the local analogue of Corollary \ref{cor:changing_polarizations_PEL}, $\mathcal N^\diamond$ is isomorphic to the formal scheme considered in \cite[\S I.3]{boutot1991uniformisation}, which is the one encountered in the  well-known \v{C}erednik-Drinfeld uniformization of quaternionic Shimura curves.
\end{rmk}
\begin {lemma}
\label{embedding contained in Siegel locus RZ space}
The image of $\mathcal M ^\diamond\coloneqq\mathcal N ^\diamond_{\red} (0) $
under $j ^ {\local} $
is contained in $\mathcal M_{\set{02}} $
(cf. Theorem
\ref {BT stratification first version}).\end {lemma}
\begin {proof}
It suffices to consider $\overline\F_q $-valued points, so suppose given $x = (X_x,\iota_x,\lambda_x,\rho_x)\in\mathcal M ^\diamond (\overline\F_q) $, with $\iota: M_2 (O_d)\otimes\Z_q\hookrightarrow\End (X_x) $
an embedding compatible with involutions.
By (\ref {meaning of BT strata subsubsection}), it suffices to show the Dieudonn\'e module $M $
of $X_x $
satisfies $M +\tau M =\tau M +\tau ^ 2M $
for $\tau = (\iota\circ\beta (\Pi_D))\cdot V ^ {-1}. $
(Since $M $
is self-dual, $M\intersection\tau M $
is then $\tau $-stable as well.)
If $\tau' =\iota\begin{pmatrix}\Pi_d & 0\\0 &\Pi_d\end {pmatrix}\cdot V ^ {-1} $, then we have $$(\tau') ^ nM=\tau ^ nM $$
for all $n\in\Z $
by Proposition \ref{integral embedding of quaternion algebras}(3). So it suffices to show $M +\tau' M $
is $\tau' $-stable. The action of
$$M_2 (\Z_q)\subset M_2 (O_d)\otimes\Z_q $$
on $M $
defines a decomposition $M = M_0\oplus M_0 $, where $M_0 $
inherits an action $\iota_0 $
of $O_d $. We can further decompose $M_0 = M_{0,\bullet} \oplus M_{0,\circ}$ according to the eigenvalues of the action of $\Z_{q^2} \subset O_d\otimes \Z_q$. Then $M_{0,\bullet}$ and $M_{0,\circ}$ are both free of rank two over $\breve \Z_q$ and stable under  $\tau_0 \coloneqq\iota_0 (\Pi_d) V ^ {-1} $. Hence $M_0 + \tau_0 M_0$ is $\tau_0$-stable by \cite[Proposition 2.17]{rapoport1996period}; so $M +\tau' M $
is $\tau' $-stable, as desired.
\end {proof}
\begin{definition}\label{def:L_diamond}
We define a subset $\mathscr L^\diamond \subset \mathscr L_{\Siegel}$ by 
$$j ^ {\local} (\mathcal N ^\diamond_{\red}) =\bigcup_{(\Lambda_+,\Lambda_-)\in\mathscr L ^\diamond}\left (\mathcal M_+ (\Lambda_+)\intersection\mathcal M_- (\Lambda_-)\right),$$ which makes sense 
by Lemma \ref{embedding contained in Siegel locus RZ space}.
\end{definition}
By  \cite[Proposition III.2]{boutot1991uniformisation} (and Remark \ref{rmk:X_diamond_integral_model}), the special fiber $Z(g^q, V_\special, V_D)_{\F_q}$ is purely supersingular. Hence by Rapoport-Zink uniformization, we have:
\begin {prop}\label{RZ uniformization for SST embedding}
The special fiber $Z(g^q, V_\special, V_D)_{\overline\F_q} $
is isomorphic to $$\spin (V_{d/q} ^\diamond) (\Q)\backslash\spin (V_{d/q} ^\diamond) (\A ^ {q}_f)\times\mathcal N ^\diamond_{\red}/K ^ {q\diamond}_{g ^ q}, $$
in such a way that the special fiber of $j:\boldsymbol Z(g^q, V_\special, V_D)\to X $
is given by $(h ^ q, x)\mapsto (h ^ qg ^ q, j ^ {\local} (x)) $
under the uniformization of Theorem
\ref {rz uniformization}.\end {prop}
\begin{notation}
Let
\begin {equation}
\widetilde j:\widetilde Z(g^q, V_\special, V_D)_{\overline\F_q}\to\widetilde X_{\overline\F_q}
\end {equation}
be the strict transform of $$Z(g^q, V_\special, V_D)_{\overline\F_q}\to X_{\overline\F_q}, $$
and
\begin {equation}
\left [\widetilde Z(g^q, V_\special, V_D)_{\overline\F_q}\right]\in H ^ 4_{\et, c} (\widetilde X_{\overline\F_q}, O (2))\end {equation}
its algebraic cycle class.
\end{notation}
\begin {corollary}\label{formula for potential map on special cycle corollary}
We have $$\nabla\left [\widetilde Z(g^q, V_\special, V_D)_{\overline\F_q}\right] =\sum_{\substack {[(h ^ q,\Lambda)]\in\\\spin (V_{d/q} ^\diamond) (\Q)\backslash\spin (V ^\diamond_{d/q}) (\A ^ {q}_f)\times\mathscr L/K ^ {q\diamond}_{g ^ q}}} m (\Lambda) [(h ^ qg ^ q,\Lambda)], $$
where $$m (\Lambda) =\sum_{y\in\mathscr L ^\diamond}\multiplicity (\Lambda, - 4q\delta_+ ^ {\Siegel} (y) - 4q\delta_- ^ {\Siegel} (y) + (\delta_+ +\delta_-)\circ\theta ^ {\paramodular}_{\Siegel} (y)). $$
\end {corollary}
\begin {proof}
This is immediate from Lemma \ref{potential map on Siegel cycle lemma}, Proposition \ref{RZ uniformization for SST embedding}, and 
 Definition \ref{def:L_diamond}.
\end {proof}
To compute the multiplicities $m (\Lambda) $
in the Corollary \ref{formula for potential map on special cycle corollary} above, it is better to work with lattices in the split 5-dimensional quadratic space $V\coloneqq V_{D/q}\otimes\Q_q $
rather than the symplectic space $W $
(see (\ref{V versus W})).
\begin{definition}
\leavevmode
\begin{enumerate}
    \item A
 {vertex lattice} $L\subset V $
is a $\Z_q $-lattice satisfying $$L ^\check\supset L\supset qL ^\check. $$
For $0\leq i\leq 2, $
set
\begin {equation*}
\VL (2i)\coloneqq\set {\text {vertex lattices } L\subset V\,:\,\dim_{\F_q} L ^\check/L = 2i}.\end {equation*}
(The analogous sets $\VL (2i +1) $
are all empty.)
\item For any 
 $\Z_q $-lattice $\Lambda\subset W $, we  define
\begin {equation*} L_\Lambda =\End_{\Z_q} (\Lambda)\intersection V,\end {equation*}
which makes sense because $V =\End (W) ^ {\dagger = 1,\trace = 0}.$
\end{enumerate}
\end{definition}
\begin {lemma}\label{projecting onto vertex lattices lemma}
The map $\Lambda\mapsto L_\Lambda $
induces an isomorphism $$\mathscr L/q ^\Z\xrightarrow {\sim}\VL (0) $$
and a surjection $$\mathscr L_{\paramodular}/q ^\Z\twoheadrightarrow\VL (4). $$
Moreover, the map $$(\Lambda_+,\Lambda_-)\mapsto L_{\Lambda_+}\intersection L_{\Lambda_-} $$
induces a surjection $$\mathscr L_{\Siegel}/q ^\Z\twoheadrightarrow\VL (2). $$
\end {lemma}
\begin {proof}
For each $0\leq i\leq 2, $$\VL (2i) $
is a homogeneous space for $\spin (V) (\Q_q) $. Since the maps
\begin {align*}
\mathscr L/q ^\Z &\rightarrow\VL (0)\\
\mathscr L_{\Siegel}/q ^\Z &\rightarrow\VL (2)\\
\mathscr L_{\paramodular}/q ^\Z &\rightarrow\VL (4)
\end {align*}
are all $\spin (V) (\Q_q) $-equivariant, they are automatically surjective. Finally, for the injectivity of the first map, it suffices to note that the stabilizers coincide, i.e. the hyperspecial subgroup of $\SO (V) (\Q_q) $
is the image of a hyperspecial subgroup of $\spin (V) (\Q_q) $.
\end {proof}
\begin {lemma}\label{Vertex lattice commutative diagrams}
The projections of Lemma \ref{projecting onto vertex lattices lemma} fit into the following commutative diagrams:
\begin{center}
\begin {tikzcd}
\mathscr L_{\paramodular}\arrow [r, "\delta_+ +\delta_-"]\arrow [d] & \Z [\mathscr L]\arrow [d]\\\VL (4)\arrow [r, "\overline\delta"] &\Z [\VL (0)]
\end {tikzcd}
\begin {tikzcd}
\mathscr L_{\Siegel}\arrow [r, "\delta_+ ^ {\Siegel} +\delta_- ^ {\Siegel}"]\arrow [d] & \Z [\mathscr L]\arrow [d]\\\VL (2)\arrow [r, "\overline\delta ^ {\Siegel}"] & \Z [\VL (0)]
\end {tikzcd}
\begin {tikzcd}
\mathscr L_{\Siegel}\arrow [d]\arrow [r, "\theta_{\Siegel} ^ {\paramodular}"] & \Z [\mathscr L_{\paramodular}]\arrow [d]\\\VL (2)\arrow [r, "\overline\theta_{\Siegel} ^ {\paramodular}"] & \Z [\VL (4)]
\end {tikzcd}
\end{center}
Here, the bottom maps are:
\begin {align*}
\overline\delta (L_4) & =\sum_{\substack {L_0\in\VL (0)\\L_0\supset L_4}} [L_0], &\text {for } L_4\in\VL (4),\\
\overline\delta ^ {\Siegel} (L_2) & =\sum_{\substack {L_0\in\VL (0)\\L_0\supset L_2}} [L_0], &\text {for } L_2\in\VL (2),\\
\overline\theta ^ {\paramodular}_{\Siegel} (L_2) & =\sum_{\substack {L_4\in\VL (4)\\L_4\subset L_2}} [L_4], &\text {for } L_2\in\VL (2).
\end {align*}
\end {lemma}
\begin {proof}
Given $\Lambda_{\paramodular}\in\mathscr L_{\paramodular} $
with $$q ^ {n +1}\Lambda_{\paramodular} ^\check\subset_2\Lambda_{\paramodular}\subset_2 q ^ n\Lambda ^\check_{\paramodular} $$
and $\Lambda\in\mathscr L $
with $\Lambda\in\delta_+ (\Lambda_{\paramodular}) $, we first claim $L_\Lambda\supset L_{\Lambda_{\paramodular}}. $
Indeed, for any $\l\in L_{\Lambda_{\paramodular}} $, $\l $
induces a self-adjoint, trace-zero endomorphism of the two-dimensional symplectic space $q ^ n\Lambda ^\check_{\paramodular}/\Lambda_{\paramodular} $; hence $\l (q ^ n\Lambda_{\paramodular} ^\check)\subset\Lambda_{\paramodular} $. Since $\Lambda $
fits into a chain $$\Lambda_{\paramodular}\subset_1\Lambda = q ^ n\Lambda ^\check\subset_1 q ^ n\Lambda_{\paramodular} ^\check, $$
we conclude $$\l (\Lambda)\subset\Lambda_{\paramodular}\subset\Lambda, $$
so $\l\in L_\Lambda $.
Similarly, we see that, for $\Lambda\in\mathscr L $
appearing in $\delta_- (\Lambda_{\paramodular}) $, we have $L_\Lambda\supset L_{\Lambda_{\paramodular}} $. Now to prove the first diagram commutes, it suffices to show
\begin {equation}\label {degree count for commuting diagram lemma}
\degree\overline\delta (L_{\Lambda_{\paramodular}}) =\degree (\delta_+ (\Lambda_{\paramodular})) +\degree (\delta_- (\Lambda_{\paramodular})).\end {equation}
The left-hand side is the number of lattices $L\in\VL (0) $
containing $L_{\Lambda_{\paramodular}} $; such lattices are in bijection with isotropic planes in the split 4-dimensional $\F_q $-quadratic space $L ^\check_{\Lambda_{\paramodular}}/L_{\Lambda_{\paramodular}} $, hence there are $2 (q +1) $
of them. Meanwhile $\degree (\delta_\pm (\Lambda_{\paramodular})) = q +1, $
since lattices $\Lambda\supset_1\Lambda_{\paramodular} $
(resp. $\Lambda\subset_1\Lambda_{\paramodular} $)
are in bijection with lines in the 2-dimensional $\F_q $-symplectic space $q ^ n\Lambda ^\check_{\paramodular}/\Lambda_{\paramodular} $
(resp. $\Lambda_{\paramodular}/q ^ {n +1}\Lambda_{\paramodular} ^\check $). This proves (\ref {degree count for commuting diagram lemma}), hence the commutativity of the first diagram; and the rest are similar.
\end {proof}
\begin{notation}
\leavevmode
\begin{enumerate}
\item For convenience, we now abbreviate $$V_{\special, q} = V_\special\otimes\Q_q.$$ 
Let $L_\special\subset V_{\special, q} $
be the $\Z_q $-lattice spanned by $\set {e_1 ^\special, qe_2 ^\special} $, and let $L ^ {(0)}_\special = \Span_{\Z_q}\set{q^{-1} e_1^T, qe_2^T}$
and $L ^ {(1)}_\special = \Span_{\Z_q}\set{e_1^T, e_2^T}$
be the two self-dual lattices in $V_{\special, q} $
containing $L_\special $.
\item 
 We write $\VL (2) ^\diamond\subset\VL (2) $
for the subset consisting of lattices of the form $$L_2 = L_\special\oplus L_\diamond, $$
where $$L_\diamond\subset V_\diamond\coloneqq V_{d/q} ^\diamond\otimes\Q_q $$
is self-dual.
\end{enumerate}
\end{notation}
\begin {lemma}\label{vertex lattice projection of special cycle orbit}
$\mathscr L ^\diamond $
consists of two $\spin (V ^\diamond_{d/q}) (\Q_q) $-orbits; if $\O\subset\mathscr L ^\diamond $
is one orbit, then the map of Lemma \ref{projecting onto vertex lattices lemma} induces an isomorphism $$\O/q ^\Z\isomorphism\VL (2) ^\diamond. $$
\end {lemma}
\begin {proof}
By \cite[Th\'eor\`eme 9.3]{boutot1991uniformisation} and Remark \ref{rmk:X_diamond_integral_model},  $\mathscr L ^\diamond $
consists of two $\spin (V_{d/q} ^\diamond) (\Q_q) $-orbits, and the stabilizer of any point in $\mathscr L ^\diamond $
is a hyperspecial subgroup. On the other hand, for any point in $\mathscr L ^\diamond $
with image $L_2\in\VL (2) $, we know $$L_2\intersection V_{\special, q}\supset L_\special $$
by Proposition \ref{integral embedding of quaternion algebras}(4) and the definition of the strata $\mathcal M_+ (\Lambda_+) $
and $\mathcal M_- (\Lambda_-). $
Since the stabilizer of $L_2 $
in $\SO (V_{d/q} ^\diamond) (\Q_q) $
is hyperspecial, this forces $$L_2 = L_\special\oplus L_\diamond, $$
for some $L_\diamond\subset V_\diamond $
self-dual. Hence the image of either orbit $\O\subset\mathscr L ^\diamond $
in $\VL (2) $
is contained in $\VL (2) ^\diamond $; since $\VL (2) ^\diamond $
is a single orbit, we have a surjection $\O/q ^\Z\twoheadrightarrow\VL (2) ^\diamond. $
Finally, we see that this map is an isomorphism because the stabilizers coincide.
\end {proof}
\begin {lemma}\label{calculation of multiplicity for first explicit reciprocity law lemma}
The multiplicity $m (\Lambda) $
in Corollary \ref{formula for potential map on special cycle corollary} depends only on $L_\Lambda\in\VL (0) $, and is given by:
$$m (\Lambda) =\begin {cases} 4, &\text {if } L_\Lambda\intersection V_{\special, q} = L_\special,\\4 - 4q, &\text {if } L_\Lambda\intersection V_{\special, q} = L_\special ^ {(0)}\text { or } L ^ {(1)}_\special,\\0, &\text {else}.\end {cases} $$
\end {lemma}
\begin {proof}
Let $\O $
be one of the two $\spin (V_{d/q} ^\diamond) (\Q_q) $-orbits in $\mathscr L ^\diamond $; we will compute $$m_\O (\Lambda) =\sum_{y\in\O}\multiplicity\left (\Lambda, - 4q\delta_+ ^ {\Siegel} (y) - 4q\delta_- ^ {\Siegel} (y) + (\delta_+ +\delta_-)\circ\theta_{\Siegel} ^ {\paramodular} (y)\right). $$
Let $m (\Lambda, y) $
be the summand above, so that $m_\O (\Lambda) =\sum_{y\in\O} m (\Lambda, y) $. Fix some $y\in\O $
corresponding to $L_y\in\VL (2) $.
Then $$m (\Lambda, q ^ ny) = m (q ^ {- n}\Lambda, y) $$
is nonzero for at most one $n $, so $$\sum_{n\in\Z} m (\Lambda, q ^ ny) =\multiplicity (L_\Lambda, - 4q\overline\delta ^ {\Siegel} (L_y) +\overline\delta\circ\overline\theta_{\Siegel} ^ {\paramodular} (L_y)) $$
by Lemma \ref{Vertex lattice commutative diagrams}. Since $\O/q ^\Z $
maps isomorphically to $\VL (2) ^\diamond $
by Lemma \ref{vertex lattice projection of special cycle orbit}, we find $$m_\O (\Lambda) =\sum_{L_2\in\VL (2) ^\diamond}\multiplicity\left(L_\Lambda, - 4q\overline\delta ^ {\Siegel} (L_2) +\overline\delta\circ\overline\theta ^ {\paramodular}_{\Siegel} (L_2)\right). $$
Next, we calculate, for any $L_2\in\VL (2) $,
\begin {align*}
\overline\delta\circ\overline\theta ^ {\paramodular}_{\Siegel} (L_2) & =\sum_{\substack {L_4\in\VL (4)\\L_4\subset L_2}}\sum_{\substack {L_0\in\VL (0)\\L_0\supset L_4}} [L_0]\\
& =\sum_{\substack {L_0\in\VL (0)\\L_0\supset L_2}}\#\set {L_4\in\VL (4)\,:\, L_4\subset L_2} +\sum_{\substack {L_0\in\VL (0)\\L_0\intersection L_2\in\VL (4)}} [L_0]\\
& = (q +1)\overline\delta ^ {\Siegel} (L_2) +\sum_{\substack {L_0\in\VL (0)\\L_0\intersection L_2\in\VL (4)}} [L_0].
\end {align*}
(The $q +1 $
choices of $L_4\subset L_2 $
correspond to flags $$qL_2 ^\check\subset qL_4 ^\check\subset L_4\subset L_2, $$
hence to complete isotropic flags in the split 3-dimensional $\F_q $-quadratic space $L_2/qL_2 ^\check $.)
Hence \begin{equation}\label{semifinal diamond multiplicity formula}
m_\O (\Lambda) =\sum_{L_2\in\VL (2) ^\diamond}\multiplicity\left (L_\Lambda, (1 - 3q)\overline\delta ^ {\Siegel} (L_2) +\sum_{\substack {L_0\in\VL (0)\\L_0\intersection L_2\in\VL (4)}} [L_0]\right).  
\end{equation}
Now observe that, for any $L_2\in\VL (2) ^\diamond $
and $L_4\in\VL (4) $
with $L_4\subset L_2 $, we have $$L_4\supset qL_4 ^\check\supset qL_2 ^\check\supset L_\special. $$
Hence, if $L_2\in\VL (2) ^\diamond $, all $L_0
$
appearing in $$(1 - 3q)\overline\delta ^{\Siegel} (L_2) +\sum_{\substack {L_0\in\VL (0)\\L_0\intersection L_2\in\VL (4)}} [L_0] $$
have $L_0\intersection V_{\special, q}\supset L_\special $.
In particular, if $m_\O (\Lambda)\neq 0 $, then $$L_\Lambda\intersection V_{\special, q} = L_\special ^ {(0)},\, L_\special ^ {(1)},\,\text {or}\, L_\special, $$
since these are the only lattices containing $L_\special $
on which the pairing is $\Z_q $-valued.

Suppose first $L_\Lambda\intersection V_{\special, q} = L_\special ^ {(i)} $, for $i = 0 $
or 1. Then we may write $$L_\Lambda = L_\diamond\oplus L_\special ^ {(i)} $$
with $L_\diamond\subset V_\diamond $
self-dual, and $$m_\O (\Lambda) = (1 - 3q)\#\set {L_2\in\VL (2) ^\diamond\,:\, L_\Lambda\in\overline\delta ^{\Siegel} (L_2)} +\#\set {L_2\in\VL (2) ^\diamond\,:\, L_2\intersection L_\Lambda\in\VL (4)}. $$
Recall that all $L_2\in\VL (2) ^\diamond $
are of the form $L_\diamond'\oplus L_\special $, with $L_\diamond'\subset V_\diamond $
self-dual, and $$\overline\delta ^{\Siegel} (L_\diamond'\oplus L_\special) =\left [L_\diamond'\oplus L_\special ^ {(0)}\right] +\left [L_\diamond'\oplus L_\special ^ {(1)}\right]. $$
Hence
\begin {align*}
m_\O (\Lambda) & = (1 - 3q) +\#\set {L'_\diamond\subset V_\diamond\,\text {self-dual}\,:\, (L'_\diamond\intersection L_\diamond)\oplus L_\special\in\VL (4)}\\
& = (1 - 3q) + (q +1)\\
& = 2 - 2q;
\end {align*}
the $(q +1) $
choices of $L_\diamond' $
correspond bijectively to the  isotropic  lines in the 3-dimensional $\F_q $-space $L_\diamond/qL_\diamond $.

Now suppose $L_\Lambda\intersection V_{\special, q} = L_\special $. Then $L_\Lambda $
does not appear in $\overline\delta ^{\Siegel} (L_2)$ for any $L_2\in \VL(2) ^\diamond $, so by (\ref{semifinal diamond multiplicity formula}), we have $$m_\O (\Lambda) =\#\set {L_2\in\VL (2) ^\diamond\,:\, L_\Lambda\intersection L_2\in\VL (4)}. $$
Recall that any $L_4 $
contained in $L_2\in\VL (2) ^\diamond $
satisfies $qL_4 ^\check\supset L_\special $; we first show that there is a unique such $L_4 $
contained in $L_\Lambda $.
Indeed, any such $L_4 $
fits in a chain $$qL_\Lambda\subset_2qL_4 ^\check\subset_1L_4\subset_2L_\Lambda, $$
and so we must have $$qL_4 ^\check = L_\special + qL_\Lambda, $$
which determines $L_4 $. Now, for this $L_4 $, we claim there are exactly two $L_2\in\VL (2) ^\diamond $
with $L_2\supset L_4 $. Such an $L_2 $
fits into a chain $$L_4\subset_1L_2\subset_2L_2 ^\check\subset_1L_4 ^\check. $$
Since $L_\special\subset qL_4 ^\check $, we have $$\frac {1} {q} L_\special\subset L_4 ^\check. $$
Thus $$L_4 ^\check/L_4\cong\frac {1} {q} L_\special/L_\special\oplus\H, $$
with $\H $
a hyperbolic plane over $\F_q $. For $L_2 $
to lie in $\VL (2) ^\diamond $
is equivalent to $\frac {1} {q} L_\special\subset L_2 ^\check $, so choices of $L_2 $
correspond to choices of isotropic lines in $L_4 ^\check/L_4 $
orthogonal to $\frac {1} {q} L_\special/L_\special $; these are isotropic lines in $\H $, so there are exactly two. Hence $m_\O (\Lambda) = 2 $
if $L_\Lambda\intersection V_{\special, q} = L_\special. $
We have now shown $$m_\O (\Lambda) =\begin {cases} 2, &\text {if}\, L_\Lambda\intersection V_{\special, q} = L_\special,\\2 - 2q, &\text {if}\, L_\Lambda\intersection V_{\special, q} = L_\special ^ {(0)} \text {or}\, L_\special ^ {(1)},\\0, &\text {else.}\end {cases} $$
Since $m_\O (\Lambda) $
is evidently independent of the choice of orbit $\O $, we have $$m (\Lambda) = 2m_\O (\Lambda), $$
which completes the proof.\end {proof}
Combining Lemma \ref{calculation of multiplicity for first explicit reciprocity law lemma} and Corollary \ref{formula for potential map on special cycle corollary} gives the following simple formula.
\begin {corollary}\label{cor:nabla_formula_cosets}
We have \begin{equation*}
\begin{split}&\nabla\left [{\widetilde Z}( g ^ q, V_\special, V_D)_{\overline\F_q}\right] = 4Z(g ^ qg_q ^\special, V_\special, V_{D/q})_{K^qK_q} \\ &\hspace{1cm}+4 (1 - q)\left (Z ( g ^ qg ^ {(0)}, V_\special, V_{D/q})_{K^qK_q}) + Z (g ^ qg ^ {(1)}, V_\special, V_{D/q})_{K^qK_q}\right)\in O\left[\hypernormal\right], \end{split}\end{equation*}
where $g_q ^\special $, $g_q ^ {(0)} $, $g_q ^ {(1)}\in\spin (V_{D/q}) (\Q_q) $
represent the cosets in $$\spin (V_{d/q }^\diamond) (\Q_q)\backslash\spin (V_{D/q}) (\Q_q)/K_q $$
corresponding to the lattices $L_\special $, $L_\special ^ {(0)} $, $L_\special ^ {(1)} $ under Proposition \ref{prop:zanarella}.
\end {corollary}
\subsection {Interpretation in terms of test functions}\label{subsec:1ERL_test}
\subsubsection{}\label{subsubsec:1ERL_test_fns}
We now define a specific test function $\phi_q ^{\total}\in \mathcal S (V_{D/q} ^ 2\otimes\Q_q,\Z) $
as follows. First, define a subset $X\subset V ^ { 2}_{D/q}\otimes\Q_q $
by
\begin{equation}\label{eq:def_X_1ERL}
X =\set {(x,y)\in V_{D/q} ^ { 2}\otimes\Q_q\, |\, x\cdot x\in q\Z_q,\, y\cdot y\in q\Z_q, \, x\cdot y\in\Z_q ^\times}.
\end{equation}
Let $L\subset V_{D/q}\otimes\Q_q $
be a self-dual lattice, and let $\phi_q ^ {(0)} $, $\phi_q ^ {(1)} $, $\phi _q ^\star\in \mathcal S (V_{D/q} ^ { 2}\otimes\Q_q,\Z) $
be indicator functions of the following compact open subsets of $V_{D/q} ^ { 2}\otimes\Q_q $:
\begin{align*}
X ^ {(0)} & =\set {(x, y)\in X\,:\, x, y\in L - qL}\\
X ^ {(1)} & =\set {(x, y)\in X\,:\, x\in qL - q ^ 2L,\, y\in q ^ {-1} L - L}\\
X ^\star & =\set {(x, y)\in X\,:\, x\in L - qL,\, y\in q ^ {-1} L - L}.
\end{align*}
Let $$\phi_q ^{\total} =\phi_q ^\star + (1 - q)\left (\phi_q ^ {(0)} +\phi_q ^ {(1)}\right). $$
Corollary \ref{cor:nabla_formula_cosets} can then be reformulated as follows:
\begin{thm}\label{test function version of the first geometric reciprocity law}
Let $K  = \prod K_v \subset \spin (V_{D/q}) (\A_f)$ be neat with $K_q$ hyperspecial, and let $$\phi ^ q\in \mathcal S (V_{D/q}^ { 2}\otimes\A_f ^ q, O) $$
be a $K ^ q $-invariant Schwartz function. Let $\m\subset \T^S_O$ be a generic, non-Eisenstein, and weakly $q$-generic maximal ideal. 

Then, for all $T\in\symmetric ^ 2 (\Q)_{\geq 0} $, 
there exists a choice of uniformization datum for $V_D$
and a test function $\phi_q ^\ramified\in \mathcal S (V_{D} ^ { 2}\otimes\Q_q,\Z) ^{K_q^\ramified}$
such that
$$\nabla \circ \zeta \circ \res_{\Q_q}\circ\;\partial_{\AJ,\m}\left (Z(T, \phi ^ q\otimes\phi_q ^\ramified)_{K^qK_q^\ramified}\right) = 4Z(T, \phi ^ q\otimes\phi_q ^\total)_{K^qK_q}\in\frac {O\left [\hypernormal (V_Q)\right]_\m} {(\langle q\rangle ^{-1} - 1,\funnyT_q^\lr)}. $$
\end{thm}
\begin {rmk}
\leavevmode
\begin{enumerate}
\item The map $\partial_{\AJ,\m}$ is defined as in (\ref{subsubsec:where_AJ}), and we are using Theorem \ref{theorem conclusion of AJ section} to apply $\zeta$ on the left-hand side of the identity. 
    \item The choice of uniformization datum intervenes in two points in the displayed equation: first in the definition of $\nabla$, and second in the isomorphism $V_{D/q}\otimes\A_f ^ q\cong V_{D}\otimes\A ^ q_f $ from Remark \ref{rmk:unif_datum}(\ref{rmk:unif_datum_two}), which we are using to view $\phi^q$ as a test function in $\mathcal S(V_{D}^2\otimes \A_f^q, O)$ and $K^qK_q^\ram$ as a compact open subgroup of  $\spin(V_D)(\A_f)$. 
    \item Without any great difficulty, $\phi_q ^\ramified $
can be chosen not to depend on $T $. Since this is not used in the proofs of the main results, we omit the details.
\end{enumerate}
\end {rmk}
\begin{proof}
Without loss of generality, we may assume $T $
is of the form considered in Notation \ref{notation:sst_sp_new}; otherwise $Z(T, \phi ^ q\otimes\phi_q ^ {\total})_K = 0, $
so $\phi_q ^\ramified = 0 $
satisfies the conclusion of the theorem. We fix the uniformization datum as in Construction \ref{constr:X_diamond_1ERL}(\ref{constr:X_diamond_unifdatum});
in particular, we are identifying $V_\special $
with a subspace of $V_{D} $, so we may choose $\phi_q ^\ramified\in \mathcal S (V_{D} ^ { 2}\otimes\Q_q,\Z) $
such that $\phi_q ^\ramified |_{\Omega_{T, V_D} (\Q_q)} $
is the indicator function of the coset $K_q ^\ramified\cdot (e_1 ^\special, e_2 ^\special). $
 For any $\phi ^ q\in \mathcal S (V_{D/q} ^ { 2}\otimes\A_f ^ q, O) $, write
$$\support(\phi ^ q)\intersection\Omega_{T, V_{D/q}} (\A_f ^ q) =\sqcup\spin (V_{D/q}^\diamond) (\A_f ^ q) g_i ^ q K ^ q. $$
Then we have $$Z (T, \phi ^ q\otimes\phi_q ^\ramified)_{K^qK_q^\ramified} =\sum_iZ ( g ^ q_i,V_\special, V_D)_{K^qK_q^\ramified}\phi ^ q ((g_i ^ q)^{-1} e_1 ^\special, (g_i ^ q)^{-1} e_2 ^\special). $$

Now, by Theorem \ref{theorem conclusion of AJ section}, we conclude that
$$\nabla \circ \zeta \circ \res_{\Q_q} \circ \;\partial_{\AJ,\m} \left( Z (T, \phi ^ q\otimes\phi_q ^\ramified)_{K^qK_q^\ramified} \right) \in \frac{O\left [\hypernormal (V_Q)\right]_\m} {(\langle q\rangle ^{-1} - 1,\funnyT_q^\lr)}$$ coincides with 
$$\sum_i \phi ^ q ((g_i ^ q)^{-1} e_1 ^\special, (g_i ^ q)^{-1} e_2 ^\special) \cdot\nabla \left[\widetilde Z ( g ^ q_i,V_\special, V_D)_{K^qK_q^\ramified}\right].$$
Then the theorem follows from Corollary \ref{cor:nabla_formula_cosets} and the construction of $\phi_q^\total$.
\end{proof}

\section{First explicit reciprocity law}\label{sec:1erl_continued}
For this section, let $\pi$, $S$, and $E_0$ be as in Notation \ref{notation:pi_basic},  fix a 
prime $\p$ of $E_0$ satisfying Assumption \ref{assumptions_on_p}, and put $\m \coloneqq \m_{\pi,\p}$ as usual. Our goal for this section is to combine Theorem \ref{test function version of the first geometric reciprocity law} with Corollary \ref{cor:admissible_C_chi_lambda} to prove Theorems \ref{thm:1ERL} and \ref{thm:1ERL_endoscopic} below. First we make some deformation-theoretic preparations in \S\ref{subsec:typic}-\S\ref{subsec:def_endo}; then we check the criterion from Corollary \ref{cor:admissible_C_chi_lambda} in \S\ref{subsec:test_fn_calculation}; and we complete the proofs in \S\ref{subsec:1ERL_conclusions}.

\subsection{ Typic modules}\label{subsec:typic}

The following definition is  a  generalization of \cite[Definition 5.2]{scholze2018lubintate}.
\begin{definition}\label{typic definition}
       Let $G$ be a group, $R$ a Noetherian local ring with maximal ideal $\m_R$, and $\sigma_1,\cdots, \sigma_m$ a finite collection of representations $\sigma_i: G \to \GL_{n_i}(R)$ such that the residual representations $\overline\sigma_i \coloneqq \sigma\otimes_R R/\m_R$ are distinct and absolutely irreducible for each $i = 1, \cdots, m$. 

       An $R[G]$-module $M$ is called $\sigma_i$-typic if it is isomorphic  to $\sigma_i\otimes_R M_0$ for an $R$-module $M_0$, and $(\sigma_1, \cdots, \sigma_m)$-typic if it is isomorphic  to a direct sum $\oplus M_i$ with each $M_i$ $\sigma_i$-typic.

\end{definition}

\begin{prop}\label{typic inclusion prop}
    With notation as in Definition \ref{typic definition}, let $N \subset M$ be an inclusion of $ R[G]$-modules. 
    \begin{enumerate}
        \item\label{typic inclusion prop part one} If $M$ is $\sigma_i$-typic for some $i\in \set{1,\cdots, m}$, then $N$ is $\sigma_i$-typic.
        \item\label{typic inclusion prop part two} If $M$ is $(\sigma_1,\cdots, \sigma_m)$-typic, then $N$ is $(\sigma_1,\cdots, \sigma_m)$-typic.
    \end{enumerate}
\end{prop}
\begin{proof}
Part (\ref{typic inclusion prop part one}) is proved in \cite[Proposition 5.4]{scholze2018lubintate}. For (\ref{typic inclusion prop part two}), as in \emph{loc. cit.} we may assume without loss of generality that $M$ and $N$ are finitely generated. Let $M = \oplus M_i$ be the decomposition of $M$ into $\sigma_i$-typic parts, and let $\pi_i: M \to M_i$ be the projection map for each $i = 1, \cdots, m$. Without loss of generality, we may assume $\pi_i(N) = M_i$. We claim that the natural injection
$$\oplus \pi_i: N \hookrightarrow \oplus M_i$$
is an isomorphism, i.e.
 $\oplus \pi_i$ is surjective. Let $\overline N \coloneqq N \otimes_R R/\m_R$ and $\overline M _i \coloneqq M_i\otimes_R R/\m_R$; by Nakayama's lemma, it suffices to show the induced map
$$\oplus\overline \pi_i: \overline N \to \oplus \overline M_i$$ is surjective. Because the residual representations $\overline \sigma_i$ are all distinct, any $R[G]$-stable submodule of $\oplus \overline M_i$ is a direct sum $\oplus \overline M_i'$ for some $\overline M_i'\subset \overline M_i$. So if $\oplus \overline \pi_i$ is not surjective, then $\overline \pi_i$ is not surjective for some $i = 1,\cdots, m$; but this contradicts our assumption that $\pi_i(N) = M_i$ for all $i$, so $\oplus \pi_i$ is surjective, which shows (\ref{typic inclusion prop part two}).
\end{proof}
\subsection{Deformation theory: non-endoscopic case}\label{subsec:non_endo_def}
We assume for this subsection that $\pi$ is not endoscopic. We will apply the results and notations of Appendix \ref{sec:appendix_def_theory} to $\rho_\pi=\rho_{\pi,\p}$, which we view as valued in $\GSP_4(O)$ via Remark \ref{rmk:rho_O_valued}. First note:
\begin{lemma}\label{lem:assumption_appendix_ok}
The representation $\rho_\pi $
satisfies Assumptions \ref{ass:appendix_main} and \ref{ass:appendix_smooth} from Appendix \ref{sec:appendix_def_theory}.
\end{lemma}
\begin{proof}
By Assumption \ref{assumptions_on_p}, $\overline\rho_\pi $
is absolutely irreducible, so $H ^ 0 (\Q,\ad^0\overline \rho_\pi) = 0. $
Also, $\overline\rho_\pi\not\cong\overline\rho_\pi (1) $
by considering the similitude characters (since $p >3 $). So, again using the absolute irreducibility, $$H ^ 0 (\Q, \ad^0 \overline\rho_{\pi} (1))\subset\Home_{G_\Q} (\overline T_\pi,\overline T_\pi (1)) = 0 $$
as well. This shows Assumption \ref{ass:appendix_main}(\ref{ass_B1_noresidual}). Assumptions \ref{ass:appendix_main}(\ref{ass_B2_odd},\ref{ass_B3_HT}) are clear from Theorem \ref{thm:rho_pi_LLC}. 
We now consider Assumption \ref{ass:appendix_smooth}. By Lemma \ref{smoothness lemma appendix} it suffices to show
$H ^ 0 (\Q_v,  \WD(\ad^0\rho_\pi)) = 0$ for all non-archimedean $v$.
But $$H ^ 0 (\Q_v, \WD(\ad^0\rho_\pi))\subset\Hom_{\WD_v} (\WD(V_\pi),\WD( V_\pi (1))), $$
with $\WD_v$ the local Weil-Deligne group,
which vanishes by purity (Theorem \ref{thm:rho_pi_LLC}(\ref{part:rho_pi_LLC1})).
\end{proof}


\subsubsection{}
Suppose $q$ is an admissible prime, and let  $\mathcal D_q$ and $R_q$ be as in Notation \ref{notation:appendix_def_rings}. 
For any $A\in \CNL_O$ and $\rho_A\in\mathcal D_q (A) $, let $M_A $
be the free, rank-four $A $-module with $G_{\Q_q} $
action determined by $\rho_A $. Then 
by \cite[Lemma 3.21]{liu2024deformation},
$M_A $
admits a $G_{\Q_q} $-stable decomposition
\begin{equation}\label{admissible local decomposition equation}
M_A = M_0\oplus M_1,\end{equation}
where:
\begin{itemize}
\item Each of $M_0 $
and $M_1 $
is free of rank two over $A $.
\item $M_0\otimes_A k $
has $\Frobenius_q $ eigenvalues $1 $
and $q $.
\item $M_1\otimes_A k $
has generalized $\Frobenius_q $
eigenvalues $a $
and $q/a $, with $a\neq 1, q $.
\end{itemize}
(Here $k = O/\varpi$, and $\Frobenius_q\in G_{\Q_q} $ is any lift of Frobenius.)
\begin{definition}\label{def:R_q_ord}
Let 
$\mathcal D_q ^\ordinary\subset\mathcal D_q $
be the subfunctor of lifts $\rho_A $
such that $$\debt (\rho_A (\Frobenius_q) - T | M_0) = (1 - T) (q - T). $$
\end{definition}
\begin{lemma}\label{lem:R_q_ord_smooth}
The functor $\mathcal D_q ^\ordinary $
is represented by a formally smooth quotient  $R_q ^\ordinary $
of $R_q $, of relative dimension 10 over $O $.
\end{lemma}
\begin{proof}
    This follows the same argument as \cite[Proposition 3.35]{liu2024deformation} (cf. also \cite[Proposition 3.8]{wang2022deformation}).
\end{proof}
\subsubsection{}\label{subsubsec:translate_appendix_1ERL}
In light of Lemma \ref{lem:R_q_ord_smooth}, we take the admissible primes in Notation \ref{notation:admissible_appendix} to be the ones of Definition \ref{def:admissible}, and the notion of $R_q ^\ordinary $
to be the one from Lemma \ref{lem:R_q_ord_smooth}; the definitions of $n $-admissible primes in Definition \ref{def:admissible} and Definition \ref{def:appendix_admissible_etc}(\ref{def:appendix_n_admissible_part}) then coincide. 



\begin{notation}\label{notation:global_def_ring}
    \leavevmode
    \begin{enumerate}
        \item Let $Q\geq 1 $
be admissible, and denote by $\widetilde R ^ Q_\m $
the global $\GSP_4 $-valued deformation ring of $\overline\rho_\pi $
as a representation of $G_{\Q, S\cup \div(Qp)}$, with fixed similitude character $\chi_p ^\sick $.
 Let $R^Q$ and $R_Q$ be the quotients of $\widetilde R^Q_\m$ defined in Notation \ref{notation:appendix_global_rings} (identifying $Q$ with $\div(Q)$ for notational convenience).
\item Let $$\rho_Q ^\universal: G_\Q\rightarrow\GSP_4 (R_Q) $$
be a framing of the universal deformation, and let $M_Q^\universal$ be the free $R_Q$-module of rank four with $G_\Q$-action defined by $\rho_Q^\universal.$
    \end{enumerate}
\end{notation}

\subsubsection{}
Let $\pr_p: I_{\Q_q} \to \Z_p(1)$ be the maximal pro-$p$ quotient. 
\begin{lemma}\label{lem:t_q_1ERL}
    Suppose $Q$ is admissible and $q|Q$ is a prime. In the decomposition $M_Q^\univ|_{G_{\Q_q}} = M_0 \oplus M_1$ of (\ref{admissible local decomposition equation}), $M_1$ is unramified. Moreover, there exists a  basis of $M_0$ and an element $t_q\in R_Q$ such that 
    $$\rho_Q^\univ|_{M_0} = \begin{pmatrix}
    \chi_{p,\cyc} & \ast \\ 0 & 1
\end{pmatrix},$$
and $$\rho_Q^\univ(g)|_{M_0} = \begin{pmatrix}
    1 & t_q\pr_p(g) \\ 0 & 1
\end{pmatrix},\; \forall g\in I_{\Q_q}.$$
\end{lemma}
\begin{proof}
Since $M_0$ and $M_1$ are $G_{\Q_q}$-stable, this follows from  \cite[Propositions 5.3, 5.5]{shotton2016local}.
\end{proof}

\begin{definition}
Suppose $Q$ and $q\nmid Q$ are admissible. Then:
\begin{enumerate}
    \item We set $P_q(T) = \det(\rho_Q^\univ(\Frob_q) - T| M_Q^\univ) \in R_Q[T].$  
    \item We set  $R_{Q, q} ^\congruent\coloneqq R_Q\otimes_{R ^ {Qq}} R_{Qq}. $
\end{enumerate}  
\end{definition}
\begin{lemma}\label{lem:R_cong}
    Suppose $Q\geq 1$ is admissible, and $q\nmid Q$ is an admissible prime. Then $$R_{Q,q}^\congruent = R_Q/(P_q(q)) = R_{Qq}/(t_q).$$
\end{lemma}
\begin{proof}
    We have $R_{Q,q}^\congruent = R_Q/(P_q(q))$ because an unramified deformation of $\overline \rho_\pi|_{G_{\Q_q}}$ is ordinary if and only if $q$ is an eigenvalue of $\Frob_q$; on the other hand, it is clear from Lemma \ref{lem:t_q_1ERL}
    that $R_{Q,q}^\congruent = R_{Qq}/(t_q)$.
\end{proof}
\begin{lemma}\label{ordinary local condition is standard Lemma}
Suppose $q $
is $n $-admissible. Then:
\begin{enumerate}
\item \label{lem:ord_std_one}$H ^ 1_{\ordinary} (\Q_q, \ad^0\rho_{\pi, n}) + H ^ 1_{\unramified} (\Q_q, \ad^0\rho_{\pi, n}) = H ^ 1 (\Q_q, \ad^0\rho_{\pi, n}). $
\item \label{lem:ord_std_two}The quotients $$\frac {H ^ 1 (\Q_q, \ad^0\rho_{\pi, n})} {H ^ 1_{\unramified} (\Q_q, \ad^0\rho_{\pi, n})}\:,\;\;\frac {H ^ 1 (\Q_q, \ad^0\rho_{\pi, n})} {H ^ 1_{\ordinary} (\Q_q, \ad^0\rho_{\pi, n})} $$
are both free of rank one over $O/\varpi^ n $.
\end{enumerate}
In particular, $q$ is standard in the sense of Definition \ref{appendix definition standard}.
\end{lemma}
\begin{proof}
First note that $$\frac {H ^ 1 (\Q_q, \ad^0\rho_{\pi, n})} {H ^ 1_{\unramified} (\Q_q, \ad^0\rho_{\pi, n})} =\Home_{\Frobenius_q} (\Z_p(1), \ad^0\rho_{\pi, n}) $$
is free of rank one over $O/\varpi^ n $
since $q $
is $n $-admissible. On the other hand, $H ^ 1_{\ordinary} (\Q_q, \ad^0\rho_{\pi, n}) $
clearly surjects onto this quotient by definition, which shows (\ref{lem:ord_std_one}). Since $Z ^ 1_{\ordinary} (\Q_q, \ad^0\rho_{\pi, n}) $
and $Z ^ 1_{\unramified} (\Q_q, \ad^0\rho_{\pi, n}) $
both contain all coboundaries, we see that $$\frac {Z ^ 1_{\ordinary} (\Q_q, \ad^0\rho_{\pi, n})} {Z ^ 1_{\ordinary} (\Q_q, \ad^0\rho_{\pi, n})\intersection Z ^ 1_{\unramified} (\Q_q, \ad^0\rho_{\pi, n})} = 
\frac {H ^ 1_{\ordinary} (\Q_q, \ad^0\rho_{\pi, n})} {H ^ 1_{\ordinary} (\Q_q, \ad^0\rho_{\pi, n})\intersection H ^ 1_{\unramified} (\Q_q, \ad^0\rho_{\pi, n})} =\frac {H ^ 1 (\Q_q, \ad^0\rho_{\pi, n})} {H ^ 1_{\unramified} (\Q_q, \ad^0\rho_{\pi, n})}$$ is free of rank 1 over $O/\varpi^n$.
On the other hand, $Z ^ 1_{\ordinary} (\Q_q, \ad^0\rho_{\pi, n}) $
and $Z ^ 1_{\unramified} (\Q_q, \ad^0\rho_{\pi, n}) $
are both free of rank 10 over $O/\varpi^ n $
 because  $R_q^\ord$ and the unramified local deformation ring $R_q^\unr$ are formally smooth, so we conclude that $Z ^ 1_{\ordinary} (\Q_q, \ad^0\rho_{\pi, n})\intersection Z ^ 1_{\unramified} (\Q_q, \ad^0\rho_{\pi, n}) $
is free of rank 9 over $O/\varpi^ n $. In particular, $$\frac {Z ^ 1_{\unramified} (\Q_q, \ad^0\rho_{\pi, n})} {Z ^ 1_{\ordinary} (\Q_q, \ad^0\rho_{\pi, n})\intersection Z ^ 1_{\unramified} (\Q_q, \ad^0\rho_{\pi, n})} = 
\frac {H ^ 1_{\unramified} (\Q_q, \ad^0\rho_{\pi, n})} {H ^ 1_{\ordinary} (\Q_q, \ad^0\rho_{\pi, n})\intersection H ^ 1_{\unramified} (\Q_q, \ad^0\rho_{\pi, n})} =\frac {H ^ 1 (\Q_q, \ad^0\rho_{\pi, n})} {H ^ 1_{\ordinary} (\Q_q, \ad^0\rho_{\pi, n})}$$
is also free of rank one over $O/\varpi^ n $, as desired.
\end{proof}
\subsubsection{}\label{subsubsec:R_to_T}
To state the next lemma, we establish some temporary notation. 
Suppose $Q$ is admissible and let $K$ be an $S $-tidy level structure $K $ for $\spin(V_{DQ})$ (Definition \ref{def:S_tidy}). 
Abbreviate $\T\coloneqq\T ^ {S\cup\divisors (Q)}_{K, V_{DQ},O,\m}, $ which may be the zero ring. Also  fix an isomorphism $\iota: \overline\Q_p \isomorphism \C$ inducing $\p$. 
Then we write $\mathcal T$ for the set of relevant automorphic representations $\Pi$ of $\spin (V_{DQ})(\A) $
with $\Pi ^ K_f\neq 0 $ such that the Hecke action on $\iota^{-1} \Pi_f^K$ factors through $\T$. 
Recall from Corollary \ref{cor:T_embedding} that we have an embedding of $\T$-algebras
\begin{equation}\label{decomposing Hecke algebra equation}\T\hookrightarrow \bigoplus_{\Pi\in \mathcal T}\overline\Q_p(\Pi),\end{equation}
where  $\overline\Q_p(\Pi)$ is $\overline\Q_p$ with  Hecke action through the eigenvalues on $\iota^{-1}\Pi_f^K$. 
 By the same argument as \cite[Theorem 7.9.4]{boxer2021abelian}, there exists a Galois representation $$\rho: G_{\Q, S\cup \div(Qp)} \rightarrow\GSP_4 (\T) $$ 
such that, for each $\Pi\in \mathcal T $, the composite
$$G_{\Q, S\cup \div(Qp)}\xrightarrow {\rho}\GSP_4 (\T)\rightarrow\GSP_4 (\overline\Q_p (\Pi)) $$
is conjugate to the Galois representation $\rho_{\Pi,\iota} $ from Remark \ref{rmk:rho_Pi}.
\begin{lemma}\label{lem:RtoT_non_end}
With notation as in (\ref{subsubsec:R_to_T}), we have:
\begin{enumerate}
\item \label{lem:RtoT_non_end_one}The
composite 
$$G_\Q\xrightarrow{\rho}\GSP_4(\T) \xrightarrow{\nu} \T^\times$$ is given by $\chi_{p,\cyc}$, and the corresponding $O$-algebra map $r_\rho: \widetilde R^Q_\m \to \T$ factors through $R_Q$.
\item Suppose $\sigma(DQ) $ is  {even.}
Then $H ^ 3_{\et} (\Sh_K (V_{DQ})_{\overline\Q}, O (2))_\m $
is $\rho_Q ^\universal $-typic when viewed as a $R_Q [G_\Q] $-module via (\ref{lem:RtoT_non_end_one}).\label{lem:RtoT_non_end_two}
\end{enumerate}
\end{lemma}
\begin{proof}
We have $\nu \circ \rho =\chi_{p,\cyc}$ because 
 each $\Pi\in \mathcal T$ has trivial similitude character by Lemma \ref{lem:S_tidy_upshot}. To complete the proof of (\ref {lem:RtoT_non_end_one}),  it suffices to show that all the composite maps $$ \widetilde R ^ Q_\m\xrightarrow{r_\rho}\T\rightarrow\overline\Q_p(\Pi) $$
factor through $R _Q$ for $\Pi\in \mathcal T$. Because $K_p$ is hyperspecial,
 each $\rho_{\Pi,\iota}|_{G_{\Q_p}}$ is crystalline with Hodge-Tate weights $\set {-1, 0, 1, 2} $ by Theorem \ref{thm:rho_pi_LLC}(\ref{part:rho_pi_LLC_HT}). So it suffices to check that $\rho_{\Pi,\iota}|_{G_{\Q_q}}$ is ordinary for all $q|Q$. 
Indeed, $\rho_{\Pi,\iota}|_{G_{\Q_q}}$ is tamely ramified because $\overline\rho_\pi$ is unramified, but ramified by Corollary \ref{cor:JL_general}(\ref{cor:JL_general_two}) and Theorem \ref{thm:rho_pi_LLC}(\ref{part:rho_GL2_LLC_1}).
In particular, for any lift of $\Frobenius_q\in G_{\Q_q} $, 
$\rho_{\Pi,\iota} (\Frobenius_q) $
has a pair of eigenvalues of ratio $q$. It follows that $\rho_{\Pi,\iota}|_{G_{\Q_q}}$ is ordinary, and this shows (\ref{lem:RtoT_non_end_one}).

For (\ref {lem:RtoT_non_end_two}), by Proposition \ref{typic inclusion prop}(\ref{typic inclusion prop part one})  and Theorem \ref{thm:generic}(\ref{part:thm_generic_two}) it suffices to show $H ^ 3_{\et} (\Sh_K (V_{DQ})_{\overline\Q},\overline\Q_p(2))_\m $
is $\rho ^ {\universal}_Q $-typic. However, this is immediate from Corollary \ref{cor:coh_relevant} and the construction of the map in (\ref {lem:RtoT_non_end_one}); note that each $\Pi\in \mathcal T$ is non-endoscopic because $\overline\rho_{\pi}$ is absolutely irreducible. 
\end{proof}
\begin{definition}\label{def:H_Q}
    For any admissible $Q$ with $\sigma(DQ)$  {even} and any $S$-tidy level structure $K$ for $\spin(V_{DQ})$, we define $H_Q(K)= \Hom_{R_Q[G_\Q]}\left(M_Q^\univ, H^3_\et ( \Sh_K(V_{DQ})_{\overline \Q}, O(2))_\m\right)$. 
\end{definition}
\begin{rmk}\label{rmk:typic_decomps}
In the context of Definition \ref{def:H_Q}: 
\begin{enumerate}
    \item \label{rmk:typic_decomps_one}
    By \cite[Proposition 5.3]{scholze2018lubintate} and Lemma \ref{lem:RtoT_non_end}(\ref{lem:RtoT_non_end_two}), we have 
\begin{equation*}
    H^3_\et (\Sh_K(V_{DQ})_{\overline \Q}, O(2))_\m \simeq M_Q^\univ \otimes_{R_Q} H_Q(K) 
\end{equation*}
as $R_Q[G_\Q]$-modules. 
\item\label{rmk:typic_decomps_two} Under the isomorphism of (\ref{rmk:typic_decomps_one}), we have, for all $q|Q$:
$$H^1\left(I_{\Q_q}, H^3_\et (\Sh_K(V_{DQ})_{\overline \Q}, O(2))_\m \right)\simeq H^1(I_{\Q_q}, M_Q^\univ) \otimes_{R_Q} H_Q(K).$$
\end{enumerate}
\end{rmk}

\begin{lemma}\label{lem:new_torsion_inertia_1ERL}
    Suppose $Q$ is admissible with $\sigma(DQ)$ even, and $K$ is an $S$-tidy level structure for $\spin(V_{DQ})$. Then for any $q|Q$, the $\varpi$-power-torsion of $H^1\left(I_{\Q_q}, H^3_\et(\Sh_K(V_{DQ})_{\overline\Q}, O(2))_\m\right)$ is contained in 
    $$H^1\left(I_{\Q_q}, H^3_\et(\Sh_K(V_{DQ})_{\overline\Q}, O(2))_\m\right)^{\Frob_q = 1} \simeq 
    H^1(I_{\Q_q}, M_Q^\univ) ^{\Frob_q = 1}\otimes_{R_Q} H_Q(K)\simeq H_Q(K)/(t_q).$$
\end{lemma}
(The element $t_q\in R_Q$ was defined in Lemma \ref{lem:t_q_1ERL}.)
\begin{proof}
    By Lemma \ref{lem:t_q_1ERL}, we see that $$H ^ 1 (I_{\Q_q},M ^ {\universal}_{Q}) = R_{Q}/(t_q)\oplus R_{Q} (-1)\oplus M_1(-1)$$ as $R_{Q} $-modules with $\Frobenius_q $-action, 
with $M_0\oplus M_1$ the decomposition of (\ref{admissible local decomposition equation}) for $M ^ {\universal}_{Q} $. In particular, $\Frob_q -1$ acts invertibly on $R_Q(-1) \oplus M_1(-1)$. 
Since $M_1 $
is free over $R_{Q} $
and $H_Q(K)$
is $\varpi $-torsion-free by Theorem \ref{thm:generic}(\ref{part:thm_generic_two}), it follows that the $\varpi $-torsion part of $H ^ 1 (I_{\Q_q}, M ^ {\universal}_{Q})\otimes _{R_Q}H_{Q}(K) $
is contained in $$ \left(H^1(I_{\Q_q}, M_Q^\universal) \otimes_{R_Q} H_Q(K)\right)^{\Frob_q = 1} = H ^ 1 (I_{\Q_q}, M ^ {\universal}_{Q}) ^ {\Frobenius_q = 1}\otimes H_{Q}(K) = H_{Q}(K)/(t_q), $$
as claimed.

\end{proof}

\begin{lemma}\label{lem:new_typic_upshot}
    Suppose $Q$ is $n$-admissible with $\sigma(DQ)$ even, and $K$ is an $S$-tidy level structure for $\spin(V_{DQ})$. Then for all $q|Q$, all $\alpha_0 \in \Hom_{R_Q}(H_Q(K)/(t_q), O/\varpi^n)$, and all $z\in \SC^2_K(V_{DQ},O)$, we have
    $$\alpha_0 \circ \res_{\Q_q} \circ \, \partial_{\AJ,\m}(z) \in \partial_q(\kappa_n(Q; K)).$$
\end{lemma}
Here $\res_{\Q_q} \circ\, \partial \AJ_\m (z)$ is viewed as an element of $H^1\left(I_{\Q_q}, H^3_\et(\Sh_K(V_{DQ})_{\overline{\Q}}, O(2))_\m \right)^{\Frob_q = 1}$, which we identify with $H_Q(K)/(t_q)$ by Lemma \ref{lem:new_torsion_inertia_1ERL};  $O/\varpi ^ n $
is viewed as  an $R_{Q} $-algebra  via the map  corresponding to $\rho_{\pi,n} $; and $\partial_q$ was defined in Notation \ref{notation:loc_q}.
\begin{proof}
By construction, $T_{\pi, n}$ and $M_Q^\universal \otimes_{R_{Q}} O/\varpi ^ n$ are isomorphic as $O[G_\Q]$-modules.
Given a map $\alpha_0: H_{Q}(K)/(t_q)\rightarrow O/\varpi ^ n $, we obtain a corresponding map of Galois modules $$\alpha =\identity\otimes\alpha_0: H ^ 3_{\et} (\Shimura_{K} (V_{DQ})_{\overline\Q}, O (2))_\m\simeq M ^ {\universal}_{Q}\otimes_{R_{Q}} H_{Q}(K)\rightarrow M_Q^\universal \otimes_{R_{Q}} O/\varpi ^ n\simeq T_{\pi,n}. $$
Let $\alpha_\ast: H ^ 1\left (\Q, H ^ 3_{\et} (\Shimura_K (V_{DQq})_{\overline\Q}, O (2))_\m\right)\rightarrow H ^ 1 (\Q, T_{\pi, n}) $
be the map induced by $\alpha $. For any $z\in\SC_K^2 (V_{DQq}, O), $
$\kappa_n (Q; K) $
contains $\alpha_\ast(\partial_{\AJ,\m}) (z)\in H ^ 1 (\Q, T_{\pi, n}) $. So the lemma follows from the  commutative diagram
\vspace{0.4cm}
\begin{center}
\begin {tikzcd}[transform canvas={scale=0.7}]
H ^ 1\left (\Q, H ^ 3_{\et} (\Shimura_{{K}} (V_{DQ})_{\overline\Q}, O (2))_\m\right)\arrow [r]\arrow [d, "\alpha_\ast"] & H ^ 1\left (I_{\Q_q}, H ^ 3_{\et} (\Shimura_{K} (V_{DQ})_{\overline\Q}, O (2))_\m\right) ^ {\Frobenius_q = 1}\arrow [d]\arrow [r, "\sim"] & H ^ 1 (I_{\Q_q},M_Q^\universal) ^ {\Frobenius_q = 1}\otimes_{R _{Q}} H_Q(K)\arrow [d, "\alpha_0"]\arrow [r, "\sim"] & H_Q(K)/(t_q)\arrow [d, "\alpha_0"]\\
H ^ 1 (\Q, T_{\pi, n})\arrow [r] & H ^ 1 (I_{\Q_q}, T_{\pi, n}) ^ {\Frobenius_q = 1}\arrow [r, "\sim"] & H ^ 1 (I_{\Q_q},M ^ {\universal}_{Q}) ^ {\Frobenius_q = 1}\otimes_{R_{Q}} O/\varpi ^ n\arrow [r, "\sim"] & O/\varpi ^ n.
\end {tikzcd}
\end{center}
\vspace{0.4cm}
\end{proof}

\subsection{Deformation theory: endoscopic case}\label{subsec:def_endo}
For this subsection, we assume $\pi $
is endoscopic, associated to a pair $(\pi_1,\pi_2)$
of automorphic representations of $\GL_2 $
with discrete series archimedean components of weights 2 and 4, in some order. In particular, we have $\rho_{\pi} =\rho_{\pi_1} \oplus \rho_{\pi_2}$.
\begin{notation}
\leavevmode
\begin{enumerate}
\item Set $S_\pi\coloneqq T_{\pi_1}\otimes T_{\pi_2} (-1) $
with the diagonal Galois action; for any $n\geq 1 $, we also write $S_{\pi, n}\coloneqq S_\pi\otimes_{O} O/\varpi ^ n $. 
 Let
\begin{equation}
H ^ 1_{\Crystal} (\Q_p, S_{\pi, n})\subset H ^ 1 (\Q_p, S_{\pi, n})
\end{equation}
be the subspace of cocycles corresponding to extensions $$0\rightarrow T_{\pi_1, n}\rightarrow \mathcal E\rightarrow T_{\pi_2, n} \rightarrow 0 $$
such that $\mathcal E $
is torsion crystalline with Hodge-Tate weights in $[-1, 2] $, cf. (\ref{subsubsec:torsion_cryst}). 
\item For any squarefree $Q\geq 1$ with $p\nmid Q$, define 
$$\Sel_{\mathcal G^Q}(\Q, S_{\pi,n}) \coloneqq \ker\left( H^1(\Q, S_{\pi,n}) \to \prod_{\l\not\in S\cup \div(Qp)} H^1(I_{\Q_\l}, S_{\pi,n}) \times \frac{H^1(\Q_p, S_{\pi,n})}{H^1_\Crystal(\Q_p, S_{\pi,n})}\right).$$
When $Q =1$ we drop it from the notation.
\end{enumerate}
\end{notation}
\begin {lemma}\label{lem:Q_RS_endoscopic}
Suppose $Q \geq 1$ is admissible and let $n\geq 1 $
be any integer. Then $\Summer_{\mathcal G ^ Q} (\Q, S_{\pi, n}) =\Summer_{\mathcal G} (\Q, S_{\pi, n}). $
\end{lemma}
\begin{proof}
Fix $q | Q $. Since $q $
is admissible, the eigenvalues of $\Frobenius_q $
on $S_{\pi, n}\otimes O/\varpi =\overline T_{\pi_1}\otimes\overline T_{\pi_2} (-1) $
are of the form $\set {\alpha,\alpha ^ {-1}, q\alpha ^{-1},\alpha q ^{-1}} $
for some $\alpha\in\overline\F_p ^\times $
with $\alpha\neq 1, q, q^2, q^{-1} $.
Thus $H ^ 0 (\Q_q, S_{\pi, n}) = H ^ 0 (\Q_q, S_{\pi, n} (1)) = 0 $
for all $n $, so $H ^ 1 (\Q_q, S_{\pi, n}) = 0 $
for all $n $
by the local Euler characteristic formula, and the lemma follows.
\end{proof}



\begin{notation}
For an integer $n \geq 1$,
recall the notion of pseudorepresentations of degree $n$ from \cite[Definition 2.1.1]{wake2019deformation}. We use \emph{op. cit.} as our basic reference for pseudodeformation theory, although some of the relevant results are due to Bella\"iche and Chenevier \cite{bellaiche2009pseudo}.
\leavevmode
\begin{enumerate}
\item Let $G$ be a group and $R$ be a ring. If $\rho: G\to \GL_n(R)$ is any representation, we write $D_\rho: G\to R$ for the associated degree-$n$ pseudorepresentation. A pseudorepresentation $D: G \to R$ of degree $n$ is called \emph{reducible} if it is equal to $D_\rho$ for a reducible representation $\rho: G\to \GL_n(R)$. 
    \item  If $Q \geq 1$ is squarefree,    let 
$$\widetilde\pseudodeformations_\m ^Q: \CNL_O \to \operatorname{Set}$$ be the functor defined by
$$A \mapsto \set{\text{pseudorepresentations } D: G_{\Q, S\cup \div(Qp)} \rightarrow A\text{ of degree }4\,:\, D\otimes_A k = D_{\overline\rho_\pi}},$$
where $k = O/\varpi$ and $\overline\rho_\pi$ is viewed as valued in $\GL_4(k)$. 
Let $\pseudodeformations_\m ^Q\subset \widetilde \pseudodeformations_\m ^Q $  be the subfunctor of pseudorepresentations that are torsion crystalline at $p $
with Hodge-Tate weights in $[-1, 2] $ in the sense of \cite[Definition 2.5.4]{wake2019deformation}.
 By Theorems 2.2.5 and 2.5.5 of \emph{op. cit.}, $\widetilde \pseudodeformations_\m^Q$ and $ \pseudodeformations_\m^Q$ are representable by universal pseudodeformation rings $\widetilde R^Q_\m$ and $R^Q_\m$, respectively. 
\item Let $\widetilde J^Q_\red \subset \widetilde R^Q_\m$ be the reducibility ideal defined in \cite[Proposition 4.2.2(2)]{wake2019deformation}, and $J^Q_\red\subset R_\m^Q$ its image under the natural projection $\widetilde R^Q_\m \twoheadrightarrow R^Q_\m$. 
\end{enumerate}
\end{notation}

\begin{rmk}
\leavevmode
    \begin{enumerate}
        \item The ideal $\widetilde J^Q_\red$ is characterized by the property that, for any morphism $f: \widetilde R^Q_\m \to A$ in $\CNL_O$ corresponding to a pseudorepresentation $D: G_{\Q, S\cup \div(Qp)} \to A$, $D$ is reducible if and only if $f(\widetilde J^Q_\red) = 0$. 
        \item In what follows, we we will apply the results of \cite[\S4]{wake2019deformation}. Although the discussion there is carried out for residual representations which are a sum of two characters, as noted in Remark 4.3.6 of \emph{op. cit.}, the results also apply for any residual representation which is multiplicity free with exactly 2 irreducible constituents, which includes $\overline\rho_{\pi} $ by Lemma \ref{lem:residually distinct}.
    \end{enumerate}
\end{rmk}

For all squarefree $Q \geq 1$,  let $p_\pi^Q: R^Q_\m \to O$ be the augmentation corresponding to $\rho_\pi$.
\begin {lemma}\label{lem:RS_BK_endoscopic}
Suppose the Bloch-Kato Selmer group $H ^ 1_f (\Q, V_{\pi_1}\otimes V_{\pi_2}(-1)) $
vanishes. 
Then there is a constant $C_{\operatorname{RS}}\geq 0 $
such that, for any admissible $Q \geq 1$, there exists $j\in\annihilator_{ R ^ Q_\m} (J_\red ^ Q) $
with $p_\pi ^ Q (j)\not\equiv 0\pmod   {\varpi ^ C}. $
\end{lemma}
\begin{proof}
We know $\fitting_{R ^ Q_\m} (J_\red ^ Q)\subset\annihilator_{R_\m ^ Q} (J_\red ^ Q) $. Then since Fitting ideals are stable under base change, it suffices to show there exists $C $
with $\varpi ^ C\in\fitting_O (J ^ Q_\red\otimes_{p_\pi ^ Q} O) $
for all $Q $, or equivalently that $J ^ Q_\red\otimes _{p_\pi ^ Q }O$
is finite with uniformly bounded cardinality. 

Let $B^Q$ and $C^Q$ be the finitely generated $R^Q_\m$-modules appearing in \cite[p. 38]{wake2019deformation} for the deformation problem $\pseudodeformations^Q_\m$. By construction in \cite[Proposition 4.2.2]{wake2019deformation}, 
 we have a surjection $B ^ Q\otimes C ^ Q\twoheadrightarrow J ^ Q_\red $, so it suffices to show in turn that $B ^ Q\otimes_{p_\pi ^ Q} O $
and $C ^ Q\otimes_{p_\pi ^ Q} O $
are finite of uniformly bounded cardinality.
Let $M $
be a finitely generated $O $-module. Because $p_\pi ^ Q $
corresponds to the reducible Galois representation $\rho_\pi $, we have $p_\pi ^ Q (J ^ Q_\red) = 0 $, so \cite[Theorem 4.3.5]{wake2019deformation} gives canonical isomorphisms
\begin{equation}
\begin{split}
\Home_O (B ^ Q\otimes_{p_\pi ^ Q} O, M) &\cong\Extension ^ 1_{O [G_\Q],\mathcal C ^ Q} (T_{\pi_2},T_{\pi_1}\otimes_O M)\\
\Home_O (C ^ Q\otimes_{q ^ Q_\pi} O, M) &\cong\Extension ^ 1_{O [G_\Q],\mathcal C ^ Q} (T_{\pi_1},T_{\pi_2}\otimes_O M).
\end{split}
\end{equation}
Here $\mathcal C ^ Q $
is the full subcategory of finitely generated $O [G_\Q] $-modules which are unramified outside $S\union \div(Qp) $
and all of whose finite subquotients are torsion crystalline with Hodge-Tate weights in $[-1, 2] $.

In particular, for all $n\geq 1 $, $$\Home_O (B ^ Q\otimes_{p_\pi ^ Q} O, O/\varpi ^ n) =\Home_O (C ^ Q\otimes_{p ^ Q_\pi} O, O/\varpi ^ n) = \Sel_{\mathcal G ^ Q} (\Q, S_{\pi, n}). $$
We conclude that $B ^ Q\otimes_{p_\pi ^ Q} O $
and $C ^ Q\otimes_{p_\pi ^ Q} O $
are isomorphic, and by Lemma \ref{lem:Q_RS_endoscopic}, they are also independent of $Q $; so take $Q= 1$ without loss of generality. 
Then by Proposition \ref{prop torsion crystalline iff crystalline},
$$\Hom_O(B^1\otimes_{p_\pi^1} O, O) \otimes \Q_p = \Extension ^ 1_{O [G_\Q],\mathcal C^1} (T_{\pi_2}, T_{\pi_1}) \otimes \Q_p = H ^ 1_f (\Q, V_{\pi_1} \otimes V_{\pi_2} (-1)) = 0. $$
Since $B^1\otimes_{p_\pi^1} O$ is a finitely generated $O$-module, it follows that $B^1\otimes_{p_\pi^1} O$ is finite, as desired. 
\end{proof}

\subsubsection{}
We now study deformations of each $\overline\rho_{\pi_i}$. As in the non-endoscopic case, we will apply the results and notations of Appendix \ref{sec:appendix_def_theory} to  $\rho_{\pi_i}$, which we view as valued in $\GL_2(O)$.
In the same way as Lemma \ref{lem:assumption_appendix_ok}, we obtain:
\begin{lemma}\label{lem:ass_appendix_ok_endo}
   For $i = 1,2$, the representation $\rho_{\pi_i}$ satisfies Assumptions \ref{ass:appendix_main} and \ref{ass:appendix_smooth} from Appendix \ref{sec:appendix_def_theory}.
\end{lemma}

\subsubsection{}
For each prime $\l$, let $R_{\l, i}$ be the local deformation ring for $\overline\rho_{\pi_i}$ as in Notation \ref{notation:appendix_def_rings}.
If $q $
is BD-admissible for $\rho_{\pi_i} $ (Definition \ref{def:admissible_endoscopic}), let  $R_{q, i} ^\ordinary $
be the   Steinberg quotient of $R_{q,i}$ in the sense of \cite[\S2]{manning2021patching}.
\begin{lemma}\label{lem:R_q_ord_sm_endo}
    The ring $R_{q,i}^\ordinary$ is formally smooth over $O$ of dimension 3.
\end{lemma}
\begin{proof}
    Immediate from \cite[Proposition 5.5]{shotton2016local}.
\end{proof}
\subsubsection{}\label{subsubsec:translate_appendix_1ERL_endo}
    In light of Lemma  \ref{lem:R_q_ord_sm_endo}, 
 we take the ``admissible'' primes in Notation \ref{notation:admissible_appendix} to be the BD-admissible ones for $\rho_{\pi_i}$. Then the notion of $n$-admissible  in Definition \ref{def:appendix_admissible_etc}(\ref{def:appendix_n_admissible_part}) coincides with the notion of $n$-BD-admissible from Definition \ref{def:admissible_endoscopic}; we will always say $n $-BD-admissible to avoid confusion. 

\begin{notation}
\leavevmode
\begin{enumerate}
    \item For $i = 1, 2 $
and a squarefree integer $Q \geq 1$ coprime to $p$, let $\widetilde R_{\m, i} ^Q $
be the global deformation ring of  $$\overline\rho_{\pi_i}: G_{\Q, S\union\divisors (Qp)}\rightarrow\GL_2 (O/\varpi), $$
with fixed determinant $\chi_p ^\sick $. 
\item  Let $R^Q_i$ and -- when $Q$ is BD-admissible for $\rho_{\pi_i}$ -- $R_{Q,i}$
be the quotients of $\widetilde R_{\m, i}^Q$ defined in Notation \ref{notation:appendix_global_rings}(\ref{notation:appendix_global_rings_firstdef},\ref{notation:appendix_global_rings_orddef}). (We are identifying $Q$ with $\div(Q)$ for notational convenience.)
\item Let $\rho_{Q,i}^\univ: G_{\Q, S\cup \div(Qp)} \to \GL_2(R_{Q,i})$ be a framing of the universal deformation, with underlying $R_{Q,i}[G_\Q]$-module $M_{Q,i}^\univ$. 
\end{enumerate}
\end{notation}
\subsubsection{}
As in the non-endoscopic case,
let $\pr_p: I_{\Q_q} \to \Z_p(1)$ be the maximal pro-$p$  quotient. By the construction of $R_{q,i}^\ord$, we have:
\begin{lemma}\label{lem:t_q_endo}
    Suppose $Q\geq 1$ is BD-admissible for $\rho_{\pi_i}$, and $q|Q$ is a prime. Then there exists a basis for $M_{Q,i}^\univ$ and   an element $t_q\in R_{Q,i}$ such that 
    $$\rho^\univ_{Q,i}|_{G_{\Q_q}} = \begin{pmatrix}
         \chi_p^\cyc & \ast \\ 0 & 1
    \end{pmatrix},$$
    and $$\rho^\univ_{Q,i} (g) = \begin{pmatrix}
        1 & \pr_p(g) t_q \\ 0 & 1
    \end{pmatrix},\;\forall g\in I_{\Q_q}.$$
\end{lemma}
\begin{definition}
    If $Q$ and  $q\nmid Q$ are BD-admissible for $\rho_{\pi_i}$:
    \begin{enumerate}
        \item Define $P_{q,i} (T) = \det(\rho_Q^\univ(\Frob_q)  - T)\in R_{Q,i}[T].$
        \item Define
$
R_{ Q, q,i} ^\congruent\coloneqq R_{Q,i}\otimes_{R ^ {Qq}_{ \m,i}} R_{Qq,i}.$
    \end{enumerate}
    
\end{definition}
The same proof as Lemma \ref{lem:R_cong} shows:
\begin{lemma}
    Suppose $Q$ and $q\nmid Q$ are BD-admissible for $\rho_{\pi_i}$. Then $$R_{Q,q,i}^\congruent = R_{Q,i}/(P_{q,i} (q)) = R_{Qq,i}/(t_q).$$
\end{lemma}
\begin{lemma}
    Suppose $q$ is $n$-BD-admissible for $\rho_{\pi_i}$. Then
$$H^1(\Q_q, \ad^0 \rho_{\pi_i, n}) = H^1_\ord(\Q_q, \ad^0 \rho_{\pi_i,n}) \oplus H^1_\unr(\Q_q, \ad^0\rho_{\pi_i, n}),$$ in which each factor is free of rank one over $O/\varpi^n$.
    In particular, $q$ is standard in the sense of Definition \ref{appendix definition standard}.
\end{lemma}
\begin{proof}
Let $\Frob_q\in G_{\Q_q}$ be a lift of Frobenius, and let $\tau_q\in I_{\Q_q}$ be an element such that $\pr_p(\tau_q) = 1$. Then, with respect to the basis in Lemma  \ref{lem:t_q_endo}, $H^1(\Q_q, \ad^0 \rho_{\pi_i,n})$ is spanned by the following two cocycles:
\begin{align*}\Frob_q&\mapsto \begin{pmatrix}
     1 & 0 \\  0& -1 
\end{pmatrix}, \;\;\tau_q\mapsto 0 \\
\Frob_q  &\mapsto 0, \;\;\;\;\;\;\;\;\;\;\;\;\;\;\;\;\tau_q \mapsto \begin{pmatrix}
     0 & 1 \\ 0 & 0 
\end{pmatrix}.\end{align*}
The first generates $H^1_\unr(\Q_q, \ad^0 \rho_{\pi_i, n})$ and the second $H^1_\ord (\Q_q, \ad^0 \rho_{\pi_i,n})$; also, both cocycles are clearly not $\varpi^{n-1}$-torsion, and the lemma follows.
\end{proof}

\subsubsection{}\label{subsubsec:pseudo_hecke}
For the next two lemmas, we introduce some temporary notation. Let $Q \geq 1$
be  admissible,  let $K$ be an $S$-tidy level structure for $\spin(V_{DQ})$, and abbreviate $\T \coloneqq \T^{S\cup \div(Q)}_{K, V_{DQ},O, \m}$, which may be the zero ring.
Also fix an isomorphism $\iota: \overline\Q_p \isomorphism \C$ inducing the prime $\p$ of $E_0$.
  Then we write $\mathcal T$ for the set of relevant automorphic representations $\Pi$ of $\spin(V_D)(\A)$ such that $\Pi_f^K \neq 0$, and the Hecke action on $\iota^{-1} \Pi_f^K$ factors through $\T$. 
  By Corollary \ref{cor:T_embedding}, we have the embedding of $\T$-algebras
  \begin{equation}\label{eq:embedding_of_T_algebras}
\T\hookrightarrow\bigoplus_{\Pi\in \mathcal T} \overline\Q_p(\Pi),
  \end{equation}
where $\overline\Q_p(\Pi)$ has Hecke action through the eigenvalues on $\iota^{-1} \Pi_f^K$. 
Then by \cite[Corollary 1.14]{chenevier2014pseudocharacters}, there is a canonical pseudorepresentation
\begin{equation}
D_Q: G_{\Q,S\union\divisors (Qp)}\rightarrow\T,
\end{equation}
such that for all $\Pi\in \mathcal T$, the composite of $D_Q$ with  the character $\lambda_\Pi: \T \to \overline \Q_p(\Pi)$ is the pseudorepresentation associated to $\rho_{\Pi,\iota}$. 
Finally, let $\mathcal T_{\textrm{end}}\subset \mathcal T$ be the subset of endoscopic representations, and
let $\overline \T$ be the quotient of $\T$ defined by the actions on $\Pi\in\mathcal T_{\textrm{end}}$. 
\begin{lemma}\label{lem:pseudo_properties}
\leavevmode
\begin{enumerate}
    \item \label{lem:pseudo_properties_one} The pseudodeformation $D_Q $
 is induced by an $O $-algebra morphism $R ^ Q_\m\rightarrow\T $.
\item\label{lem:pseudo_properties_two}
The composite  $R^Q_\m \to \T \to\overline \T $ factors through $R^Q_\m/J^Q_\red$.
\end{enumerate}
\end{lemma}
\begin{proof}
By definition, $D_Q $
is induced by an $O $-algebra morphism $d:\widetilde R_\m ^ Q\rightarrow\T $.
By (\ref{eq:embedding_of_T_algebras}), to prove (\ref{lem:pseudo_properties_one}) it suffices to show each composite map 
$\widetilde R^Q_\m  \xrightarrow{d} \T\xrightarrow{\lambda_\Pi} \overline \Q_p$ factors through $R_\m^Q$, with $\lambda_\Pi$ as in (\ref{subsubsec:pseudo_hecke}); but this is clear because $K_p$ is hyperspecial, so $\rho_{\Pi,\iota}$ is crystalline at $p$ for all such $\Pi$. 
Similarly, for (\ref{lem:pseudo_properties_two}) it suffices to note that $\lambda_\Pi \circ d$ annihilates $J_\red^Q$ for all $\Pi \in \mathcal T_{\textrm{end}}$, because $\rho_{\Pi,\iota}$ is reducible.
\end{proof}
\subsubsection{}
Let \begin{equation}\label{eq:reducible_to_tensor}R_\m^Q/J_\red^Q \to R_{Q_1,1}\otimes_O R_{Q_2,2}\end{equation} be the canonical map, defined on moduli problems by sending a pair of deformations $\rho_1,\rho_2$ of $\overline\rho_{\pi_1}$ and $\overline\rho_{\pi_2}$ to the pseudorepresentation $D_{\rho_1\oplus \rho_2}$. It follows from \cite[Proposition 4.2.6]{wake2019deformation} that (\ref{eq:reducible_to_tensor}) is surjective.
\begin{lemma}\label{lem:typic_endoscopic}
Write $Q = Q_1\cdot Q_2$ such that each $q|Q_i$ is BD-admissible for $\rho_{\pi_i}$ (Remark \ref{rmk:BD_adm}(\ref{rmk:BD_adm_dichotomy_part})). Then:
\begin{enumerate}
\item\label{lem:typic_endoscopic_one} The map $R^Q_\m /J^Q_\red \to \overline \T$ induced by Lemma \ref{lem:pseudo_properties}(\ref{lem:pseudo_properties_two}) factors through the surjection (\ref{eq:reducible_to_tensor}).
    In particular, $\overline \T$ is an $R_{Q_i,i}$-module for $i = 1,2$.
    \item\label{lem:typic_endoscopic_two} Assume $\sigma(DQ)$ is {even}. For any $j \in \Ann_{R^Q_\m} ({ J_{\red}^Q})$, the 
 $\T$-action on 
  $$jH^3_\et (\Shimura_K(V_{DQ})_{\overline{\Q}}, O(2))_\m$$ factors through $\overline \T$, and as a   $\overline \T$-module, $jH^3_\et (\Shimura_K(V_{DQ})_{\overline{\Q}}, O(2))_\m$  
  is $$(\rho^\univ_{Q_1,1}\otimes_{R_{Q_1,1}} \overline \T,  \rho^\univ_{Q_2, 2}\otimes_{R_{Q_2,2}}\overline \T)\text{-typic}.$$
\end{enumerate}

\end{lemma}

\begin{proof}
Arguing as in Lemma \ref{lem:pseudo_properties},
for (\ref{lem:typic_endoscopic_one}) it suffices to consider the pseudorepresentations attached to $\Pi\in \mathcal T_{\textrm{end}}$. Suppose $\Pi$ is associated to a pair $(\tau_1,\tau_2)$ of automorphic representations of $\GL_2(\A)$, and fix an isomorphism $\iota: \overline \Q_p \isomorphism \C$ inducing $\p$.  Then $\rho_{\Pi,\iota} =  \rho_{\tau_1,\iota} \oplus  \rho_{\tau_2,\iota}$. We have 
$\overline\rho_{\tau_1,\iota} \oplus \overline \rho_{\tau_2,\iota} \cong \overline\rho_{\pi_1,\iota} \oplus \overline \rho_{\pi_2,\iota}$, and since $\overline\rho_{\pi_i,\iota}$ are both absolutely irreducible, without loss of generality we may assume $\overline \rho_{\tau_i,\iota} \cong \overline \rho_{\pi_i,\iota}$, $i = 1$, 2. By Lemma \ref{lem:when_endoscopic_relevant} and Fontaine-Laffaille theory, $\pi_{i,\infty}$ and $\tau_{i,\infty}$ have the same weight.

 By Corollary \ref{cor:JL_general}(\ref{cor:JL_general_two}) and Theorem \ref{thm:rho_pi_LLC}(\ref{part:rho_pi_LLC1}), $(\rho_{\tau_1,\iota}\oplus \rho_{\tau_2,\iota})|_{G_{\Q_q}}$ is ramified for all $q|Q$.
 It then follows from \cite[Propositions 5.3, 5.5]{shotton2016local} that $\rho_{\tau_i,\iota}|_{G_{\Q_q}}$ is ordinary for all $q|Q_i$, and unramified for all $q| Q/Q_i$. It is also clear from  Lemma \ref{lem:S_tidy_upshot} and Theorem \ref{thm:rho_pi_LLC}(\ref{part:rho_pi_char})
 that $\det \rho_{\tau_i,\iota} = \chi_p^\cyc$, 
 and $\rho_{\pi_i,\iota}$ and $\rho_{\tau_i,\iota}$ have the same Hodge-Tate weights by Theorem \ref{thm:rho_pi_LLC}(\ref{part:rho_pi_LLC_HT}). Hence $\rho_{\tau_i, \iota}$ arises from a deformation parametrized by $R_{Q_i, i}$, and this proves (\ref{lem:typic_endoscopic_one}).

 For (\ref{lem:typic_endoscopic_two}), 
by Proposition \ref{typic inclusion prop}(\ref{typic inclusion prop part two}) and Theorem \ref{thm:generic}(\ref{part:thm_generic_two}), it suffices to show that $jH^3_\et (\Shimura_K(V_{DQ})_{\overline{\Q}}, \overline\Q_p(2))_\m$ 
is  
$(\rho^\univ_{Q_1,1}\otimes_{R_{Q_1,1}} \overline \T,  \rho^\univ_{Q_2, 2}\otimes_{R_{Q_2,2}}\overline \T)\text{-typic}.$

    For this, we use the decomposition of Corollary \ref{cor:coh_relevant}:
$$H^3_\et (\Shimura_K(V_{DQ})_{\overline{\Q}}, \overline\Q_p(2))_\m = \bigoplus_{\Pi_f} \iota^{-1}\Pi_f^K \otimes \rho_{\Pi_f},$$
as $\Pi_f$
ranges over finite parts of automorphic representations $\Pi\in \mathcal T$. This  is a decomposition of $ R^Q_\m[G_\Q]$-modules, where $ R^Q_\m$ acts on the factor $\iota^{-1}\Pi_f^K\otimes \rho_{\Pi_f}$
via the map $\lambda_\Pi:R^Q_\m \to \T \to \overline\Q_p(\Pi)$
 corresponding to the pseudorepresentation of  $\rho_{\Pi,\iota}$ (equivalently, to the Hecke eigenvalues of $\Pi_f^K$).
In particular, $\lambda_{\Pi}(J_\red^Q) = 0$ if and only if $\rho_{\Pi,\iota}$ is reducible, which by Lemma \ref{lem:reducible_endoscopic} and Remark \ref{rmk:p>5} occurs if and only if $\Pi$ is endoscopic. Because
 the element $j \in R^Q_\m$ annihilates $J_\red^Q$, it then suffices to show that, for all relevant \emph{endoscopic} automorphic representations $\Pi\in \mathcal T_{\textrm{end}}$ associated to a pair $(\tau_1,\tau_2)$, $\rho_{\Pi_f}$ is either $\rho^\univ_{Q_1, 1}\otimes_{ R_{Q_1,1}} \overline \T$-typic or  $\rho^\univ_{Q_2, 2}\otimes_{ R_{Q_2,2}} \overline \T$-typic as a $\overline \T[G_\Q] $-module.
 However, as $\rho_{\Pi_f}= \rho_{\tau_1,\iota}$ or $\rho_{\tau_2,\iota}$ by Corollary \ref{cor:coh_relevant},
 this is clear from the construction of the map $R_{Q_1,1}\otimes_OR_{Q_2,2} \to \overline \T$.
\end{proof}
 \begin{definition}\label{def:H_Q_endo}
 Let $Q$ be admissible, with a factorization $Q = Q_1 \cdot Q_2$ such  that $\sigma(DQ)$ is even and all $q|Q_i$ are BD-admissible for $\rho_{\pi_i}$. Fix an $S$-tidy level structure $K$ for $\spin(V_{DQ})$, and an element $j \in \Ann_{R_\m^Q}(J^Q_\red)$. Then 
    we define
     $$H_Q(K, j) ^{(i)} = \Hom_{R_{Q_i,i}[G_\Q]}(M_{Q_i, i}^\univ, jH^3_\et(\Sh_K(V_{DQ})_{\overline\Q}, O(2))_\m)$$
     for $i = 1,2$. 
 \end{definition}
\begin{rmk}\label{rmk:typic_decomps_endo}
In the context of Definition \ref{def:H_Q_endo},
    by \cite[Proposition 5.3]{scholze2018lubintate} and Lemma \ref{lem:typic_endoscopic}(\ref{lem:typic_endoscopic_two}), we have
$$jH^3_\et(\Sh_K(V_{DQ})_{\overline\Q}, O(2))_\m \simeq M_{Q_1,1}^\univ \otimes_{R_{Q_1,1}}   H_Q(K, j)^{(1)} \oplus M_{Q_2,2}^\univ \otimes_{R_{Q_2,2}}   H_Q(K, j)^{(2)}.$$
\end{rmk}
\begin{lemma}\label{lem:torsion_inertia_endo}
    In the context of Definition \ref{def:H_Q_endo}, let $q|Q_i$ be a prime. Then under the natural isomorphism 
    \begin{equation*}
        \begin{split}
            H^1\left(I_{\Q_q}, H^3_\et(\Sh_K(V_{DQ})_{\overline \Q}, O(2))_\m\right) \simeq H^1(I_{\Q_q}, M_{Q_1,1}^\univ)\otimes_{R_{Q_1,1}} H_Q(K,j)^{(1)} \oplus \\H^1(I_{\Q_q}, M_{Q_2,2}^\univ)\otimes_{R_{Q_2,2}} H_Q(K,j)^{(2)},
        \end{split}
    \end{equation*}
    the $\varpi$-power torsion of $H^1\left(I_{\Q_q}, jH^3_\et(\Sh_K(V_{DQ})_{\overline \Q}, O(2))_\m\right)$ is contained in 
    \begin{equation*}
        \begin{split}
H^1\left(I_{\Q_q}, jH^3_\et(\Sh_K(V_{DQ})_{\overline \Q}, O(2))_\m\right)^{\Frob_q = 1} \simeq H^1(I_{\Q_q}, M_{Q_i,i}^\univ)^{\Frob_q = 1}\otimes_{R_{Q_i,i}} H_Q(K,j)^{(i)} \\ \simeq H_Q(K, j)^{(i)}/ (t_q).
        \end{split}
    \end{equation*}
\end{lemma}
(The element $t_q \in R_{Q_i, i}$ was defined in Lemma \ref{lem:t_q_endo}.)

\begin{proof}
    Without loss of generality, suppose $i = 1$. Then $\rho_{Q_2,2}^\univ|_{G_{\Q_q}}$ is unramified, so by Lemma \ref{lem:t_q_endo}, we have
    $$H^1(I_{\Q_q}, M_{Q_i,i}^\univ) = \begin{cases}
        R_{Q_1,1}/(t_q) \;\oplus R_{Q_1,1}(-1), & i = 1,\\
        M_{Q_2,2}^\univ(-1), & i = 2.
    \end{cases}$$
    Since $H_Q(K, j)^{(1)}$ and  $H_Q(K, j)^{(2)}$ are $\varpi$-torsion-free by Theorem 
    \ref{thm:generic}(\ref{part:thm_generic_two}), the lemma follows as in the proof of Lemma \ref{lem:new_torsion_inertia_1ERL}.
\end{proof}

\begin{lemma}\label{lem:new_typic_upshot_endo}
In the context of Definition \ref{def:H_Q_endo}, for 
    all
    $j \in \Ann_{R_\m^Q}(J_\red^Q)$, all  $z\in \SC_K^2(V_{DQ},O)$,  all $q|Q_i$, and all $\alpha_0 \in \Hom_{R_{Q_i,i}} (H_Q(K, j)^{(i)}/(t_q), O/\varpi^n)$, we have
    $$\alpha_0 \circ j_\ast \circ \res_{\Q_q} \circ\;\partial_{\AJ,\m}(z) \in \partial_q(\kappa_n(Q; K)).$$
\end{lemma}
Here, 
\begin{equation}\label{eq:j_star}
j_\ast: H ^ 1 (I_{\Q_q}, H ^ 3_\et (\Shimura_{K} (V_{DQ})_{\overline\Q}, O (2))_\m)\rightarrow H ^ 1 (I_{\Q_q}, jH ^ 3_\et (\Shimura_{K} (V_{DQ})_{\overline\Q}, O (2))_\m)
\end{equation}
is the natural map, and $O/\varpi^n$ is viewed as an $R_{Q_i,i}$-algebra through the map corresponding to $\rho_{\pi_i,n}$. 
 \begin{proof}
     For any $\alpha_0\in\Home_{R_{Q_i,i}} (H_Q(K; j)^{(i)}/(t_q), O/\varpi ^ n), $
we obtain a corresponding induced map of Galois modules
$$\alpha =\identity\otimes\alpha_0:M ^\universal_{Q_i,i}\otimes_{R_{Q_i,i}} H_Q(K; j)^{(i)}\rightarrow M ^\universal_{Q_i,i}\otimes_{R_{Qi,i}} O/\varpi ^ n =T_{\pi_i, n}. $$
Then by the decomposition from Remark \ref{rmk:typic_decomps_endo}, we can also view $\alpha $
as a map of Galois modules
$$j H ^ 3_\et (\Shimura_{K} (V_{DQ})_{\overline\Q}, O (2))_\m\rightarrow T_{\pi_i, n}. $$
Let $(\alpha\circ j)_\ast$
be the induced map
$$H ^ 1 (\Q, H ^ 3_\et (\Shimura_{K} (V_{DQ})_{\overline \Q}, O (2))_\m)\to H^1(\Q, T_{\pi_i,n}). $$
For any $z\in\SC_{K}^2 (V_{DQ}, O) $, $\kappa_n^D (Q;K) $
contains $$(\alpha\circ j)_\ast(\partial_{\AJ,\m} (z))\in H ^ 1 (\Q,T_{\pi_i, n})\hookrightarrow H^1(\Q, T_{\pi,n}). $$
The lemma now follows as in the proof of Lemma \ref{lem:new_typic_upshot}.
 \end{proof}

\subsection {A test function calculation}\label{subsec:test_fn_calculation}
For this subsection, we fix the following additional data:
\begin{itemize}[label = $\circ$]
\item An integer $n \geq 1$.
\item A squarefree integer $D\geq 1$ with $\div(D) \subset S$, and an $n$-admissible $Q \geq 1$ coprime to $D$, such that $\sigma(DQ)$ is {odd.}
\item An $n$-admissible prime $q\nmid Q$. 
    \item An $S$-level structure $K$ for $\spin(V_{DQ})$.    
    Let $L\subset V_{DQ}\otimes \Q_q$ be the unique self-dual lattice stabilized by $K_q$. 
    \item A $\Z_q$-basis $\set{v_0, v_1, v_2, v_1^\ast, v_2^\ast}$ for $L$ as in (\ref{local notation for split space}), which identifies $V_{DQ}\otimes \Q_q$ with the standard split five-dimensional quadratic space over $\Q_q$.
\end{itemize}
We will apply the results and notations of \S\ref{sec:1ERL_geom}, with the $D$ therein always replaced by $DQq$. However, we do not yet specify the choice of $q$-adic uniformization datum for $V_{DQq}$.  
The goal of this subsection is the crucial Lemma \ref{lem:for_1ERL}.

\subsubsection{}

Let $\phi_q^{(0)}, \phi_q^{(1)},  \phi_q^\star,  \phi_q^\total\in \mathcal S(V_{DQ}^2\otimes \Q_q, \Z)$ be as in (\ref{subsubsec:1ERL_test_fns}).
We also let $\overline\phi_q^?\in \mathcal S(V_{DQ}^2\otimes \Q_q, \overline \F_p)$ 
be the reduction of $\phi_q^?$ for $? = (0)$, $(1)$, $\star$, $\total$. 
\begin{notation}
Without loss of generality, we write the almost level-raising generic character $\chi$ from Proposition \ref{prop:alrg_admissible} as
\begin{equation}\chi = |\cdot|^{1/2} \boxtimes \alpha: (\Q_q^\times)^2 \to \overline\F_p^\times,\end{equation} where $\alpha^2 \neq |\cdot|^{\pm 1}$. 
\end{notation}
Then we can consider condition $(C_\chi) $
from Definition \ref{def:C_chi}, for  $\overline\phi_q ^\total $:
\begin{multline}\tag {$C_\chi $}
\text{
There exists $g\in \MP_4(\Q_q) $
such that $f_\chi (\overline {\omega_\psi (1, g)\overline\phi_q^\total})\neq 0 $, where}\\
f_\chi:\left ((|\cdot | ^ {-\frac {1} {2}}) ^ {\boxtimes 2}\boxtimes\chi_\psi\cdot (|\cdot | ^ {\frac {1} {2}}) ^ {\boxtimes 2}\right)\otimes \mathcal S (\Q_q ^ 2,\overline\F_p)\rightarrow\chi\boxtimes\chi_\psi\cdot \chi ^{-1}\\
\text {is the unique projection deduced from Lemma \ref{preliminary calculation for theta}}.
\end{multline}

\begin{lemma}\label{checking C chi four first ERL}
The test function $\overline\phi_q ^\total $
satisfies condition $(C_\chi) $.
\end{lemma}
\begin{proof}
Let $g\in\metaplectic_4 (\Q_q) $
be a lift of the Weyl element $$\begin{pmatrix} 0 & 0 & 1 & 0\\0 & 1 & 0 & 0\\-1 & 0 & 0 & 0\\0 & 0 & 0 & 1\end{pmatrix}\in\SP_4 (\Q_q) $$
in the standard basis $\set {e_1, e_2, e_1 ^\ast, e_2 ^\ast} $. Then for some unit $u\in\overline\F_p ^\times $,
we have
$$\omega_\psi (1, g)\phi (x, y) =u\integral_{z\in V_{DQ}\otimes\Q_q}\phi (z, y)\psi (x\cdot z)\d z $$
for all $$\phi\in S (V ^ { 2}_{DQ}\otimes\Q_q,\overline\F_p). $$
Notice that, since $\chi = |\cdot | ^ {\half}\boxtimes\alpha, $
the projection
$$f_\chi:\left ((|\cdot | ^ {-\frac {1} {2}}) ^ {\boxtimes 2}\boxtimes\chi_\psi\cdot (|\cdot | ^ {\frac {1} {2}}) ^ {\boxtimes 2}\right)\otimes \mathcal S (\Q_q ^ 2,\overline\F_p)\rightarrow\chi\boxtimes\chi_\psi\cdot \chi ^{-1}$$
is the composite of the integration map
\begin{align*}\operatorname {int}: \mathcal S (\Q_q ^ 2,\overline\F_p) &\rightarrow \mathcal S (\Q_q,\overline\F_p)\\\phi &\mapsto\left (t\mapsto\integral\phi (t_1, t)\d t_1\right)\end{align*}
and the projection $$f_{\alpha|\cdot|^{1/2}}:\mathcal  S (\Q_q,\overline\F_p)\rightarrow\alpha|\cdot|^{1/2}\boxtimes\alpha ^{-1}|\cdot|^{-1/2}. $$
Abbreviate $$s ^? =u^{-1}\operatorname {int}\left (\overline {\omega_\psi (1, g)\overline\phi_q ^?}\right)\in\mathcal S (\Q_q,\overline\F_p) $$
for $? = (0), (1),\star,\total $. We will compute $s ^\total $ explicitly
to show $f_{\alpha|\cdot|^{1/2}} (s ^\total)\neq 0 $, which will prove the lemma.
By definition we have
$$s ^? (t) =\integral_{t_1\in\Q_q}\integral_{a\in\Q_q}\integral_{z\in V_{DQ}\otimes\Q_q}\overline\phi_q ^? (z, t v_2+ a v_1)\psi (t _1z\cdot v_1)\d z\d a\d t_1. $$
Note that $\overline\phi_q ^? (z, t v_2+ a v_1) $
depends on $z $
modulo $q ^ 2L $
only, so the inner integral is nonzero only on the set $\set {t_1\in q ^ {-2}\Z_q} $. We may therefore reorder the integrals and obtain
\begin{align*} s ^? (t) & =\integral_{z\in L}\integral_{a\in q ^{-1}\Z_q}\integral_{t_1\in q ^ {-2}\Z_q}\overline\phi_q ^? (z, t v_2+ a v_1)\psi (t_1 z\cdot v_1)\d t_1\d a\d z\\
& = q ^ 2\integral_{z\in L}\integral_{a\in q ^{-1}\Z_q}\overline\phi_q ^? (z, t v_2+ a v_1)\cdot\blackboardone_{z\cdot v_1\in q ^ 2\Z_q}\d a\d z.
\end{align*}

For the inner integral to be nonzero, we must have $ z\cdot t v_2\in\Z_q ^\times $
and $t\in q ^{-1}\Z_q $. In particular, since $z $
and $v_2 $
lie in $L $, $s ^? (t) $
is supported on $q ^{-1}\Z_q ^\times\sqcup\Z_q ^\times $. At this point, we are ready to compute the following table of values for $s ^? (t) $:
\begin{equation*}
\begin{array}{c|c|c|c|c|}

  & s ^ {(0)} (t) & s ^ {(1)} (t) & s ^\star (t) & s ^\total (t)\\
 \hline
q ^{-1}\Z_q ^\times & 0 &\frac {q-1} {q ^ 4} &\frac {(q ^ 2-1) (q -1)} {q ^ 4} &\frac {(q -1) ^ 2} {q ^ 3}\\
\hline
\Z_q ^\times &\frac {q -1} {q ^ 2} & 0 &\frac {(q -1) ^ 2} {q ^ 2} & 0\\
\hline
\end{array}
\end{equation*}

The fourth column is determined by the first three by $$s ^\total = s ^\star + (1 - q) (s ^ {(0)} + s ^ {(1)}). $$
On the other hand, the given values for $s ^\total (t) $
immediately imply $f_{\alpha|\cdot|^{1/2}} (s ^\total)\neq 0 $
for any $\alpha\neq |\cdot | ^ {-\frac {1} {2}} $. It remains to explain the calculation of the first three columns.
First, since $\overline\phi_q ^ {(0)} (z, t v_2+ a v_1) = 0 $
unless $t v_2+ a v_1\in L $, we have $s ^ {(0)} (t) = 0 $
for $t\in q ^{-1}\Z_q^\times $. For $t\in\Z_q ^\times $, we have
\begin{equation}\label{eq:s_zero_volume}
    s ^ {(0)} (t) = q^2 \volume\set {z\in L - qL\,:\, z\cdot v_2\in\Z_q ^\times,\, z\cdot v_1\in q ^ 2\Z_q, z\cdot z\in q\Z_q}.
\end{equation}

Label the set in (\ref{eq:s_zero_volume}) by $S ^ {(0)} $
and let $\overline S ^ {(0)} $
be its image in $L/q L $. 
Write $\overline z,\overline v_1,\overline v_2 $
for the reductions in $L/q L $. 

Then $\overline {S} ^ {(0)} $
is the set of $\overline z\in L/q L $
that are isotropic, and orthogonal to $\overline v_1 $
but not $\overline v_2 $.
Now, there are $q ^ 3 = q ^ 2\cdot q$
isotropic vectors in $L/q L $
orthogonal to $\overline v_1 $
since $\overline v_1 ^\perp/\overline v_1 $
is a split quadratic space over $\F_q $ of dimension three. Of these, $q ^ 2 $
are also orthogonal to $\overline v_2 $; these are just the vectors in $\operatorname{span}_{\F_q}\set {\overline v_1,\overline v_2} $, since the latter is a maximal isotropic subspace. So $$\#\overline S ^ {(0)} = q ^ 3 - q ^ 2. $$
On the other hand, given any $\overline z_0\in\overline S ^ {(0)} $, we have $$\volume\set {z\in S ^ {(0)}\,:\,\overline z =\overline z_0} =\frac {1} {q ^ 6}. $$
(It is almost the full coset $\overline z_0+ q L $, which has volume $\frac {1} {q ^ 5} $, except that we must ensure that $z $
remains orthogonal to $v_1 $
modulo $q ^ 2 $; this cuts down the volume by a factor of $q$.) So $$\volume (S ^ {(0)}) =\frac {q ^ 3 - q ^ 2} {q ^ 6} =\frac {q -1} {q ^ 4}, $$
and $$s ^ {(0)} (t) = q ^ 2\volume (S ^ {(0)}) =\frac {q -1} {q ^ 2}\;\text {for } t\in\Z_q ^\times. $$
Next let us consider $s ^ {(1)} (t). $
Since $\phi_q ^ {(1)} (z, t v_2+ a v_1) 
\cdot\blackboardone_{z\cdot v_1\in q ^ 2\Z_q} $, 
 when restricted to $a\in q ^{-1}\Z_q $, 
is supported on $z\in q L $
such that $z\cdot t v_2\in\Z_q ^\times $, we see that $s ^ {(1)} (t) = 0 $
for $t\in\Z_q ^\times $. 
For $t\in q ^{-1}\Z_q ^\times $, we have
$$s ^ {(1)} (t) = q ^ 2\volume\set {z\in q L, a\in q ^{-1}\Z_q\,:\, z\cdot v_2\in q\Z_q ^\times,\, z\cdot v_1\in q ^ 2\Z_q}. $$
(Note that the condition $z\cdot z\in q\Z_q $
in the definition of $X $
(\ref{eq:def_X_1ERL}) is automatic from $z\in q L $.)
Then replacing $z $
with $\frac {z} {q} $, we have
\begin{align*}
s ^ {(1)} (t) & =\frac {q ^ 3} {q ^ 5}\volume\set {z\in L\,:\, z\cdot v_2\in\Z_q ^\times,\, z\cdot v_1\in q\Z_q}\\
& =\frac {q ^ 3} {q ^ {10}}\#\set {\overline z\in L/q L\,:\,\overline z\perp\overline v_1,\,\overline z\not\perp\overline v_2}\\
& =\frac {q ^ 3} {q ^ {10}}\cdot (q ^ 4 - q ^ 3) =\frac {q- 1} {q ^ 4}.
\end{align*}
Finally, consider $s ^\star (t) $. For $t\in\Z_q ^\times $, we have
\begin{align*}
s ^\star (t) & = q ^ 2\volume\set {z\in L - q L,\, a\in q ^{-1}\Z_q ^\times\,:\, z\cdot v_2\in\Z_q ^\times,\, z\cdot v_1\in q ^ 2\Z_q}\\
& = q ^ 2 (q -1)\volume\set {z\in L\,:\, z\cdot z\in q\Z_q,\, z\cdot v_2\in\Z_q ^\times,\, z\cdot v_1\in q ^ 2\Z_q}.
\end{align*}
This is the same set that appeared for $s ^ {(0)} (t) $, so we have
$$s ^\star (t) = (q -1) s ^ {(0)} (t) =\frac {(q -1) ^ 2} {q ^ 2}\;\,\text {for } t\in\Z_q ^\times. $$
For $t\in q ^{-1}\Z_q ^\times $, we have
\begin{equation}\label{eq:s_start_volume}
\begin{split}
s ^\star (t) & = q ^ 2\volume\set {z\in L - q L,\, a\in q ^{-1}\Z_q\,:\, z\cdot z\in q\Z_q,\, z\cdot v_2\in q\Z_q ^\times,\, z\cdot v_1\in q ^ 2\Z_q}\\
& = q ^ 3\volume\set {z\in L - q L\,:\, z\cdot z\in q\Z_q,\, z\cdot v_2\in q\Z_q ^\times,\, z\cdot v_1\in q ^ 2\Z_q}.
\end{split}
\end{equation}
To compute this, we use the same technique as for $s ^ {(0)} $. Let $S ^\star $
be the set in the second line of (\ref{eq:s_start_volume}) and let $\overline S ^\star $
be its image in $L/q L $.
Then $\overline S ^\star $
consists of nonzero isotropic vectors in $L/q L $
orthogonal to $\overline v_2 $
and $\overline v_1 $, of which there are $q ^ 2-1 $. On the other hand, for $\overline z_0\in\overline S ^\star $, we have $$\volume\set {z\in S ^ {\star}:\overline z =\overline z_0} =\frac {q -1} {q ^ 7}: $$
this is because, out of the coset $\overline z_0 + q L $, we must take only those vectors with $z\cdot v_2\in q\Z_q ^\times $
and $z\cdot v_1\in q ^ 2\Z_q $. So $$s ^\star (t) = (q ^ 2-1)\cdot\frac {q -1} {q ^ 7}\cdot q ^ 3 =\frac {(q -1) (q ^ 2-1)} {q ^ 4}\;\text {for}\, t\in q ^{-1}\Z_q ^\times, $$
as desired.
\end{proof}
\subsubsection{}
Now we return to the geometric setting of \S\ref{sec:1ERL_geom}. We will write $\m\coloneqq \m_{\pi,\p}^{S\cup \div(Qq)}\subset \T^{S\cup \div(Qq)}_O$.
\begin{lemma}\label{lem:check_weak_generic}
    The maximal ideal $\m\subset \T^{S\cup \div(Qq)}_O$ is generic and non-Eisenstein, and weakly $q$-generic (Definition \ref{def:weakly_q_gen}).
\end{lemma}
\begin{proof}
That $\m$ is generic and non-Eisenstein follows from Lemma \ref{lem:m_is_nonEisgeneric}.
    Then by Corollary \ref{cor:T_embedding}, it suffices to show that $\langle q\rangle T_{q,2}^2 - 4q^2 (q+1)^2 \not\in \m_{\pi,\p}^S\subset \T^S_O$. Indeed, this holds by Remark \ref{rmk:weakly_q_generic}, because the admissibility of $q$ implies $\tr (\Frob_q|\overline \rho_\pi) \neq \pm 2 (q+1).$
\end{proof}
In particular, we 
 can consider, for any choice of  $q$-adic uniformization datum for $V_{DQq}$, the map
\begin{equation}\xi\coloneqq  \nabla\circ \zeta: M_{-1} H^1\left(I_{\Q_q}, H^3_\et(\Sh_{K^qK_q^\ram}(V_{DQq})_{\overline\Q}, O(2))_\m \right)\twoheadrightarrow \frac{O\left[\Sh_{K^qK_q} (V_{DQ})\right]_\m}{(\funnyT_{\lr})}\end{equation}
from Theorem \ref{main arithmetic level raising semistable}; here we use that $\langle q \rangle = 1$ on $O\left[\Sh_{K^qK_q} (V_{D/q})\right]_\m $ by Lemma \ref{lem:S_tidy_upshot}. (The map $\xi$, and the identification of $K^q$ with a compact open subgroup of $\spin(V_{DQq})(\A_f^q)$, both depend on the choice of  uniformization datum.)

The following is the crucial lemma for the proof of Theorems \ref{thm:1ERL} and \ref{thm:1ERL_endoscopic}.

\begin{lemma}\label{lem:for_1ERL}
Suppose $Qq$ is $n$-admissible. Then  there exists a
test function $\alpha\in \Test_K(V_{DQ}, \pi, O/\varpi^n)$, a $q$-adic uniformization datum for $V_{DQq}$, and a
special cycle $z\in \SC_{K^qK_q^\ram}(V_{DQq},O)$
such that
$$\alpha\circ\xi\circ\res_{\Q_q}(\partial_{\AJ,\mathfrak m} (z))\in O/{\varpi^n} $$
generates $\lambda_n (Q; K) $.
\end{lemma}
Here we are using Theorem \ref{theorem conclusion of AJ section} to apply $\xi$ to $\res_{\Q_q} (\partial_{\AJ,\m}(z))$.
\begin{proof}
By Corollary \ref{cor:admissible_C_chi_lambda} and Lemma \ref{checking C chi four first ERL}, we conclude
 $$\lambda_n^D(Q; K) = \lambda_n^D(Q, \phi^q\otimes\phi_q^\total;K)$$
 for some $\phi^q\in \mathcal S(V_{DQ}^2\otimes \A_f^q, O)^{K^q}$. Then by definition, there exists a test vector
  $$\alpha\in \Test_K(V_{DQ}, \pi, O/\varpi^n)$$ such that $\alpha (Z(T, \phi^q\otimes \phi_q^\total)_K)$ generates $\lambda_n(Q; K)$, for some $T\in \Sym_2(\Q)_{\geq 0}$. 

Now note that, for any choice of uniformization datum, $\alpha\circ \xi$ gives a well-defined map
$$M_{-1} H^1\left(I_{\Q_q}, H^3_\et(\Sh_{K^qK_q^\ram}(V_{DQq})_{\overline\Q}, O(2))_\m\right) \twoheadrightarrow \frac{O\left[\Sh_{K^qK_q} (V_{DQ})\right]_\m}{(\funnyT_{\lr})} \twoheadrightarrow O/\varpi^n$$
because $\alpha(\funnyT_{\lr}) \subset (\varpi^n)$ by Remark \ref{rmk:interpret_funny_T}.

  Then the lemma is immediate from Theorem \ref{test function version of the first geometric reciprocity law}.

\end{proof}

\subsection{Conclusions}\label{subsec:1ERL_conclusions}
Finally, we are ready to prove the main results for this section. 
We start with the non-endoscopic case.
\begin {thm}\label{thm:1ERL}
Suppose $\pi$ is not endoscopic.
Fix an integer $m\geq 1 $, and let $n_0 = n_0 (m,\rho_\pi) $
satisfy the conclusion of Lemma \ref{lem:eta_appendix}.  Suppose $Q \geq $
is  $n $-admissible  and $q\nmid Q$
is an $n $-admissible prime, where $n\geq \max\set{3m, n_0}$, such that $\sigma(DQ)$ is odd. 
\begin {enumerate}
\item \label{thm:1ERL_one}Suppose $\overline\Selmer_{\mathcal F (Q)} (\Q,\ad ^ 0\rho_m) = 0. $
Then 
$$\partial_q \kappa_n^D(Qq) \supset \lambda_n^D(Q) \cdot (\varpi^{C}),$$
where $$C = 2\lg_O \Sel_{\mathcal F(Q)^\rel}(\Q, \ad^0\rho_{n-m+1}) + m -1.$$

\item\label{thm:1ERL_two} Suppose there exists $q' | Q $
which is
\emph {not} $(n +1) $-admissible, such that $Q/q' $
is $(n + m) $-admissible and $$\overline\summer_{\mathcal F (Q/q')} (\Q,\adjoint ^ 0\rho_m) =\overline\summer_{\mathcal F (Qq)} (\Q,\adjoint ^ 0\rho_m) = 0, $$
but $$\overline\summer_{\mathcal F (Q)} (\Q,
\adjoint ^ 0\rho_{2m-1})\neq 0. $$
Then $$\partial_q\kappa ^ D_n (Qq)\supset\lambda_n ^ D (Q)\cdot\varpi ^ C, $$
where $C = 2 (m -1) + \length_\O\summer_{\mathcal F (Qq)^\rel} (\Q,\adjoint ^ 0\rho_{n-m+1}). $
\end{enumerate}
\end{thm}
\begin{proof}
Choose an $S$-tidy level structure $K$ for $\spin(V_{DQ})$ such that $\lambda_n^D(Q; K) = \lambda_n^D(Q)$ (possible by Lemma \ref{lem:tidy_suffices}),
and fix a $q$-adic uniformization datum for $V_{DQq}$, a special cycle $z\in \SC^2_{K^qK_q^\ram}(V_{DQq})$, 
and a test function $\alpha\in\Test_K (V_{DQ}, \pi, O/\varpi^n) $
satisfying the conclusion of Lemma \ref{lem:for_1ERL}; in particular, we have\begin{equation}\label{eq:beta_xi_generates}\left(\alpha\circ\xi \circ \res_{\Q_q}(\partial_{\AJ,\m} (z)) \right) = \lambda_n^D(Q).\end{equation}

Now note that $$M_{-1} H^1\left(I_{\Q_q}, H^3_\et(\Shimura_{K^qK_q^\ram} (V_{DQq})_{\overline\Q}, O (2))_\m\right) $$ is $\varpi$-power-torsion because $H^3_\et(\Shimura_{K^qK_q^\ram} (V_{DQq})_{\overline\Q}, \overline \Q_p (2))_\m$ is pure as a $G_{\Q_q}$ representation by Corollary \ref{cor:coh_relevant} and Theorem \ref{thm:rho_pi_LLC}(\ref{part:rho_pi_LLC1}). Hence by 
Lemma \ref{lem:new_torsion_inertia_1ERL} and Theorem \ref{main arithmetic level raising semistable}, we have a  diagram:

\begin{equation}\label{eq:cd_1ERL}
\begin{tikzcd}
M_{-1} H ^ 1\left (I_{\Q_q}, H ^ 3_{\et} (\Shimura_{K^qK_q^\ram} (V_{DQq})_{\overline\Q}, O (2))_\m\right) \arrow [r, hook]\arrow [d, twoheadrightarrow, "\xi"] & H_{Qq}/(t_q) \\
H_Q/(\funnyT^\lr_q),
\end{tikzcd}
\end{equation}
 where we set $$H_Q \coloneqq O\left[\Sh_K(V_{DQ})\right]_\m,\;\;H_{Qq}\coloneqq H_{Qq}(K^qK_q^\ram)$$
 (Definition \ref{def:H_Q}).
By Remark \ref{rmk:interpret_funny_T}, $(\funnyT^\lr_q) = (P_q(q))$ as ideals of $\T^{S\cup \div(Q)}_{K, V_{DQ},\m}$. Hence
by Lemmas \ref{lem:R_cong}  and \ref{lem:RtoT_non_end}, the  diagram (\ref{eq:cd_1ERL}) is a  diagram of $R_{Q,q} ^ {\congruent} $-modules. Let $Q' = Q $
in case (\ref{thm:1ERL_one}) and $Q' = Qq $
in case (\ref{thm:1ERL_two}), and define an element $a\in R_{Q'} $
as follows. Let $\tau_{Q'}: G_\Q\rightarrow\GSP_4 (O) $
be the representation constructed by Theorem \ref{main appendix theorem}, and let $I_{Q'}\subset R_{Q'} $
be the kernel of the corresponding homomorphism $f_{Q'}: R_{Q'}\rightarrow O $. By Lemma \ref{lem:eta_appendix}, we may fix an element $a\in\annihilator_{R_{Q'}} (I_{Q'}) $
such that \begin{equation}\label{eq:eta_for_1ERL_fake}\ord_{\varpi} f_{Q'} (a)\leq \lg_O \Sel_{\mathcal F(Q')^\rel}(\Q, \ad^0\rho_{n-m+1}). \end{equation}
By the definition of $C$ in each case of the theorem, we may assume without loss of generality that
$$\lg_O \Sel_{\mathcal F(Q')^\rel}(\Q, \ad^0 \rho_{n-m +1}) < n - m + 1.$$  Let $f_\pi: R^{Q}_\m \to O$ be the map corresponding to $\rho_\pi$. 
Since $f_{Q'} \equiv f_\pi \pmod{\varpi^{n - m+1}}$ by Theorem \ref{main appendix theorem}(\ref{thm:appendix_main_one}), we then have 
\begin{equation}\label{eq:eta_for_1ERL}\ord_{\varpi} f_\pi (a)\leq \lg_O \Sel_{\mathcal F(Q')^\rel}(k, \ad^0\rho_n). \end{equation}

Applying $a $
to the  diagram (\ref{eq:cd_1ERL}) of $R_{Q'}$-modules, we obtain a diagram
\begin{equation}\label{eq:cd_1ERL_2}
\begin {tikzcd}
a\cdot M_{-1} H ^ 1\left (I_{\Q_q}, H ^ 3_{\et} (\Shimura_{K^qK_q^\ram} (V_{DQq})_{\overline\Q}, O (2))_\m\right) \arrow [r, hook]\arrow [d, "\xi"] & a\cdot \left(H_{Qq}/(t_q)\right)\\
a\cdot\left( O\left [\Shimura_{K} (V_{DQ})\right]_\m/(\funnyT^\lr_q)\right) = aH_Q/(aH_Q \cap f_{Q'}(\funnyT^\lr_q)H_Q).
\end{tikzcd}
\end{equation}

Suppose first we are in case (\ref{thm:1ERL_one}). Because $a$ annihilates $I_{Q}$, Lemma \ref{lem:R_cong} implies that (\ref{eq:cd_1ERL_2}) is 
 a  diagram of $R_{Q,q} ^ {\congruent}\otimes_{R_{Q, f_Q}} O = O/f_Q(\funnyT^\lr_q)$-modules.
Note that $$\frac{a H_Q}{af_Q (\funnyT^\lr_q) H_Q}$$ is free over $O/f_Q(\funnyT^\lr_q)$ because $H_Q$, hence $aH_Q$, is $\varpi$-torsion-free. 
Since the natural surjection
\begin{equation}
\frac{a H_Q}{a f_Q (\funnyT^\lr_q) H_Q}\twoheadrightarrow \frac{a H_Q}{a H_Q\intersection f_Q (\funnyT^\lr_q) H_Q }
\end{equation} has kernel annihilated by $f_Q(a)$, we conclude
that \begin{equation} f_Q(a) \cdot \Extension^ 1_{O/f_Q (\funnyT^\lr_q)} (-, aH_Q/(aH_Q\intersection f_Q (\funnyT^\lr_q)) = 0. \end{equation}
In particular, by (\ref{eq:cd_1ERL_2}), there exists a map $\widetilde\xi: a\cdot (H_{Qq}/(t_q))\rightarrow aH_Q/(a H_Q\intersection f_Q (\funnyT^\lr_q) H_Q) $
fitting into the following   commutative diagram.
\begin{center}
\begin {tikzcd}
a\cdot M_{-1} H ^ 1\left (I_{\Q_q}, H ^ 3_{\et} (\Shimura_{K^qK_q^\ram} (V_{DQq})_{\overline\Q}, O (2))_\m\right) \arrow [r, hook]\arrow [d, "f_Q(a)\cdot \xi"] &\arrow[dl, "\widetilde\xi"] a\cdot \left(H_{Qq}/(t_q)\right)\\
a H_{Q }/(a H_{Q}\intersection\funnyT^\lr_qH_{Q})  &
\end{tikzcd}
\end{center}
Recall the test function $\alpha\in \Test_K(V_{DQ}, \pi, O/\varpi^n)$ fixed above, and let $\beta $
denote the composite map $$H_{Qq}/(t_q)\xrightarrow {a} a\cdot (H_{Qq}/(t_q))\xrightarrow {\widetilde\xi} a H_{Q }/(a H_{Q}\intersection\funnyT^\lr_qH_{Q})   \xrightarrow {\alpha} O/\varpi ^ n. $$
By Lemma \ref{lem:new_typic_upshot} and (\ref{eq:beta_xi_generates}), 
we conclude \begin{equation}\label{eq:1ERL_almost}
    \partial_q\kappa_n^D (Qq) \supset\left( \beta (\partial_{\AJ,\m} (z))\right) = \left(\alpha\circ a^2\xi\circ \res_{\Q_q} (\partial_{\AJ,\m}(z))\right) = f_\pi(a)^2 \lambda_n^D(Q).
\end{equation}
 Then the theorem  follows from (\ref{eq:eta_for_1ERL}).

For case (\ref{thm:1ERL_two}), 
(\ref{eq:cd_1ERL_2}) is  a  diagram of $R_{Q,q}^\congruent\otimes_{R_{Qq}, f_{Qq}}O = O/f_{Qq}(t_q)$-modules. By Lemma \ref{ordinary local condition is standard Lemma}, all admissible primes are standard in the sense of Definition \ref{appendix definition standard}, so by Lemma \ref{lem:std_upshot}, 
we have
\begin{equation}
    f_{Qq} (t_q)\not\equiv 0\pmod{\varpi ^ {n + m}}.
\end{equation}
 In particular, (\ref{eq:cd_1ERL_2}) is a commutative diagram of $O/\varpi^{n + m -1}$-modules.  Because $$\varpi^{m-1}\Extension^1_{O/\varpi^{n+m - 1}}(-,O/\varpi^n)= 0,$$  we conclude that there exists a map $\widetilde \alpha: a\cdot (H_{Qq}/(t_q)) \to O/\varpi^n$ fitting into the following commutative diagram:
\begin{center}
    
 \begin{tikzcd}[column sep = large]
     a\cdot M_{-1} H ^ 1\left (I_{\Q_q}, H ^ 3_{\et} (\Shimura_{K^qK_q^\ram} (V_{DQq})_{\overline\Q}, O (2))_\m\right) \arrow [d, "\alpha \circ \xi\circ \varpi^{m- 1}"] \arrow [r, hook] &a\cdot (H_{Qq}/(t_q))\arrow[dl, "\widetilde\alpha"] 
\\
O/\varpi^n
 \end{tikzcd}
 \end{center}

 A priori, $\widetilde\alpha$ is only equivariant with respect to $f_{Qq}$, but $f_\pi \equiv f_{Qq} \pmod {\varpi^{n - m + 1}}$. Multiplying by $\varpi^{m -1}$, we then obtain an $f_\pi$-equivariant map
 $$\varpi^{m - 1}(\widetilde \alpha \circ a): H_{Qq}/(t_q) \to O/\varpi^n.$$

Hence by Lemma \ref{lem:new_typic_upshot} and (\ref{eq:beta_xi_generates}), we conclude
\begin{align*}\partial_q\kappa_n^D(Qq) \supset 
\varpi^{m-1}\left(\widetilde \alpha\circ a\circ \res_{\Q_q}(\partial_{\AJ, \m}(z))\right)
&=  \varpi^{2(m-1)}\left(\alpha \circ \xi \left( a\cdot \res_{\Q_q} \partial_{\AJ, \m}(z) \right)\right)\\
&= \varpi^{2(m-1)} f_\pi(a) \left(\alpha\circ \xi \circ \res_{\Q_q} \partial_{\AJ, \m}(z)\right)\\
&= \varpi^{2(m-1)}f_\pi(a) \lambda_n^D(Q).
\end{align*}

Combined with  (\ref{eq:eta_for_1ERL}), this completes the proof in case (\ref{thm:1ERL_two}).
\end{proof}

\begin{thm}\label{thm:1ERL_endoscopic}
Suppose $\pi$ is endoscopic associated to a pair $(\pi_1,\pi_2)$. 
Assume $H ^ 1_f (\Q, V_{\pi_1,\p}\otimes V_{\pi_2,\p}(-1)) = 0, $ and let $C_{\operatorname{RS}}$ be as in Lemma \ref{lem:RS_BK_endoscopic}. 

Fix an integer $m\geq 1 $,  let $n_0 = n_0(m,\rho_{\pi_1}) $
satisfy the conclusion of Lemma \ref{lem:eta_appendix}, and let $n \geq \max\set{3m, n_0}$ be an integer. Suppose $Q = Q_1\cdot Q_2 $
is  $n $-admissible
such that $Q_1$ is BD-admissible for $\rho_{\pi_1}$, $Q_2$ is 
 BD-admissible for $\rho_{\pi_2} $, and $\sigma(DQ)$ is odd. Let $q\nmid Q$ be an  $n$-admissible prime, BD-admissible for $\rho_{\pi_i}$.
\begin{enumerate}
\item\label{thm:1ERL_endoscopic_one} Suppose $\overline\summer_{\mathcal F (Q_i)} (\Q,\adjoint ^ 0\rho_{\pi_i, m}) = 0 $. Then $$\partial_q\kappa_n^D (Qq)\supset\lambda_n^D (Q)\cdot\varpi ^ C, $$
where $$C = 2\length_O\summer_{\mathcal F (Q_i)^\rel} (\Q,\adjoint ^ 0\rho_{\pi_i, n - m + 1}) + 2C_{\operatorname{RS}} + m - 1.$$
\item \label{thm:1ERL_endoscopic_two}Suppose there exists a prime $q' | Q_i $
which is
\emph {not} $(n +1) $-admissible such that $Q_i/q'$
is $(n + m) $-admissible and $$\overline\summer_{\mathcal F (Q_i/q')} (\Q,\adjoint ^ 0\rho_{\pi_i, m}) =\overline\summer_{\mathcal F (Q_iq)} (\Q,\adjoint ^ 0\rho_{\pi_i, m}) = 0 $$
but $$\overline\summer_{\mathcal F (Q_i)} (\Q,\adjoint ^ 0\rho_{\pi_i, 2 m -1})\neq 0. $$
Then $$\partial_q\kappa_n^D (Qq)\supset\lambda_n^D (Q)\cdot\varpi ^ C, $$
where $$C =  \length_O\summer_{\mathcal F (Q_iq)^\rel} (\Q,\adjoint ^ 0\rho_{\pi_i, n-m + 1})+ C_{\operatorname {RS}}  + 2 (m -1). $$
\end{enumerate}
\end{thm}
\begin{proof}
First, choose the $S$-tidy level structure $K$ for $\spin(V_{DQ})$ such that $\lambda_n^D(Q; K) = \lambda_n^D(Q)$ (possible by Lemma \ref{lem:tidy_suffices}). 
Then fix a $q$-adic uniformization datum for $V_{DQq}$, a special cycle $z\in \SC^2_{K^qK_q^\ram}(V_{DQq})$, and a test function $\alpha\in \Test_K(V_{DQ},\pi,O/\varpi^n)$ satisfying the conclusion of Lemma \ref{lem:for_1ERL}; in particular, we have $$\left(\alpha\circ\xi \circ \res_{\Q_q} (\partial_{\AJ,\m} (z))\right)= \lambda_n^D(Q).$$ We also fix $j \in \Ann_{R^{Qq}_\m}(J^{Qq}_\red)$ satisfying the conclusion of Lemma \ref{lem:RS_BK_endoscopic}.

As in the proof of Theorem \ref{thm:1ERL}, $M_{-1}H^1\left(I_{\Q_q}, H^3_\et(\Sh_{K^qK_q^\ram}(V_{DQq})_{\overline \Q}, O(2))_\m\right)$ is $\varpi$-power-torsion, hence  $j_\ast M_{-1}H^1\left(I_{\Q_q}, H^3_\et(\Sh_{K^qK_q^\ram}(V_{DQq})_{\overline \Q}, O(2))_\m\right)$ is as well, where $j_\ast$ is as in (\ref{eq:j_star}). Since the kernel of (\ref{eq:j_star}) is  $j$-torsion, we obtain the following 
  diagram arising from Theorem \ref{main arithmetic level raising semistable} and Lemma \ref{lem:torsion_inertia_endo}:
\begin {center}
\begin {tikzcd}
j_\ast M_{-1} H ^ 1 (I_{\Q_q}, H ^ 3_\et (\Shimura_{K^qK_q^\ram} (V_{DQq})_{\overline\Q}, O (2))_\m)\arrow [r, hook]\arrow [d, twoheadrightarrow, "\xi"] & H ^ 1 (I_{\Q_q},M_{Q_iq,1}^\univ) ^ {\Frobenius_q = 1}\otimes_{R_{Q_iq,1} }H^{(i)}_{Qq}\simeq H^{(i)}_{Qq}/(t_q)\\j\cdot\left (H_Q/(\funnyT^\lr_q) \right) = j H_Q/(jH_Q \cap \funnyT^\lr_q (H_Q)
\end {tikzcd}
\end {center}
where we abbreviate 
 $$H_Q \coloneqq O \left[\Sh_K(V_{DQ})\right]_\m,\; H^{(i)}_{Qq} \coloneqq H_{Qq}(K, j)^{(i)}.$$ 
From here, the argument is entirely analogous to Theorem \ref{thm:1ERL}, replacing Lemma \ref{lem:new_typic_upshot} with Lemma \ref{lem:new_typic_upshot_endo}.


\end{proof}
\section{Main result:  rank zero case}\label{sec:main_rk0}
\subsection{Chebotarev primes and proof of the main result}
\subsubsection{}
Throughout this section, we let $\pi$, $S$, and $E_0$ be as in Notation \ref{notation:pi_basic}, and fix for now a prime $\p$ of $E_0$.
\begin{lemma}\label{lem:loc_cheb_rk0}
    Suppose  that $\pi$ is non-endoscopic, that $\p$ satisfies Assumption \ref{assumptions_on_p}(\ref{assume:irreducible}),  and that there exist admissible primes for $\rho_\pi= \rho_{\pi,\p}$. Let $C \geq 0$ be the constant from Corollary \ref{cor:Galois_coh_restr} applied to $T_\pi$. Then for all integers $m \geq n \geq 1$ and 
    for any  cocycle $c\in H^1(\Q, T_{\pi,n})$,   there are infinitely many $m$-admissible primes $q$ such that
    $$\ord_\varpi\loc_q c \geq \ord_\varpi c  - C.$$
    
\end{lemma}
Recall here that $\loc_q$ was defined in Notation \ref{notation:loc_q}.
\begin{proof}
Let $g\in G_\Q$ be an admissible element  for $\rho_{\pi}$, which is possible by Lemma \ref{lem:admissible_TFAE}.  
By Corollary \ref{cor:Galois_coh_restr}, we have $\varpi^C H^1(\Q(T_{\pi})/\Q, T_{\pi,n})=0$, so by inflation-restriction
there exists an element $h\in G_{\Q(T_{\pi})}$ such that $\ord_\varpi c(h) \geq \ord_\varpi c - C$. Because $\overline T_{\pi}$ is absolutely irreducible, we can assume without loss of generality that the component of  $c(h)$ in the 1-eigenspace for $g$ is nonzero modulo $\varpi^{n - \ord_\varpi c + C + 1}$. 
Then since
 $$ c(gh) = gc(h) + c(g),$$ after possibly replacing $g$ by $gh$ we may assume without loss of generality that the same is true for the component of $c(g)$ in the 1-eigenspace for $g$ (which is independent of the choice of cocycle representative).
Then  any prime $q\not\in S\cup \set{p}$ with Frobenius conjugate to $g$ in $\Gal(\Q(T_{\pi, m}, c))$ satisfies the conclusion of the lemma. 
\end{proof}

The following theorem is a corollary of the work of Newton-Thorne \cite{newton2023thorne} and Thorne \cite{thorne2022vanishing}.
\begin{thm}\label{theorem of thorne}
    Suppose $\pi$ is  non-endoscopic, and $\p$ is a prime of $E_0$ of residue characteristic $p > 3$ such that $\pi_p$ is unramified. 
    Then $$H^1_f(\Q, \ad^0\rho_{\pi,\p}) = 0.$$ 
\end{thm}
    \begin{proof}
    By \cite[Theorem 6.2]{thorne2022vanishing}, it suffices to show \begin{equation}\label{goal equation rho pi weak Taylor Wiles}     V_{\pi,\p}|_{G_{\Q(\mu_{p^\infty})}} \text{ is absolutely irreducible.} 
    \end{equation}
    By Lemmas \ref{lem:reducible_endoscopic} and \ref{lemma distinct Hodge Tate weights implies induction}, we can write
 $$V_{\pi,\p}\cong \Ind_{G_K}^{G_\Q}V_0$$ for a finite extension $K/\Q$, where $V_0$ is a strongly irreducible representation of $G_K$. By \cite[Lemma 2.2.9]{patrikis2019variations} and the assumption that $\pi_p$ is unramified, we conclude $K$ is unramified at $p$; hence (\ref{goal equation rho pi weak Taylor Wiles}) follows from Lemma \ref{lemma with restriction to abelian Galois extension}. 
\end{proof}
\begin{thm}\label{thm:rk_zero_main}
    Let $\pi$ be non-endoscopic. Suppose $\p$ satisfies Assumption \ref{assumptions_on_p}, and that there exist admissible primes for $\rho_{\pi,\p}$. Suppose as well that there exists a prime $\l_0$ such that $\pi_{\l_0}$ is transferrable (Definition \ref{def:transferrable}).
Then $$L( \pi,\operatorname{spin},1/2) \neq 0 \implies H^1_f(\Q, V_{\pi,\p})= 0.$$
\end{thm}
\begin{proof}
Set $D\coloneqq \l_0$, so that, by Theorem \ref{thm:JL}, $\pi_f^D$ can be completed to an automorphic representation of $\spin(V_D)(\A)$. 
By Proposition \ref{L values related to lambda elements}, if $L(\pi, \operatorname{spin}, 1/2) \neq 0$ then we have $\lambda^D(1) \neq 0$, so there exists a constant $C_0 \geq 0$  such that \begin{equation}
    \ord_\varpi \lambda_m^D(1) \geq m - C_0,\;\;\forall m \geq 1.
    \end{equation}

   Suppose for contradiction that there exists a  non-torsion element $ c\in H^1_f(\Q, T_\pi)$, and let $c_n$ be the image of $c$ in $H^1_f(\Q, T_{\pi,n})$ for all $n \geq 1$. We fix a large integer $N$ to be specified later.     
   Let $M\geq N$ be the integer of Lemma \ref{lemma kill pairing on crystalline}(\ref{lemma kill pairing on crystalline part two}) for $M = T_\pi$ and $n = N$. 
  By Lemma \ref{lem:loc_cheb_rk0} and Lemma \ref{lem:H1_tors_free_new}(\ref{lem:H1_tf_order}), we may choose an $M$-admissible prime $q$ such that \begin{equation}\label{unramified lower bound for proof of rank zero}
       \ord _{\varpi} \loc_q c_N \geq N - C_1 
   \end{equation}
   for a constant $C_1\geq 0$. 
   Now by Theorem \ref{theorem of thorne} and Lemma \ref{lemma for using theorem of thorne} (which applies to $\rho_{\pi}$ by Lemma \ref{lem:assumption_appendix_ok}),   for some $m_0 \geq 1$ we have $$\overline \Sel_{\mathcal F} (\Q, \ad^0 \rho_{\pi,m_0}) = 0.$$ Moreover, by Corollary \ref{cor:relaxed_bound_appendix},   $\lg_O \Sel_{\mathcal F^\rel}(\Q, \ad^0\rho_{\pi,n})$ is uniformly bounded in $n$. 
 Hence by Theorem \ref{thm:1ERL}(\ref{thm:1ERL_one}), as long as $M$ is sufficiently large depending on $m_0$ -- which we can ensure by choosing $N$ sufficiently large --
 there exists a constant $C_2 \geq 0 $  and an element $\kappa^D_M(q)_0\in \kappa^D_M(q)$ such that 
   \begin{equation}
       \label{ramification lower bound for proof of rank zero}
       \ord_\varpi \partial_q \kappa^D_M(q)_0 \geq \ord_\varpi \lambda^D_M(1) - C_2\geq M - C_2 - C_0.
   \end{equation}
Let $\kappa^D_N(q)_0$ be the image of $\kappa^D_M(q)_0$ in $H^1(\Q, T_{\pi_1,N})$. 
We now consider the global Tate pairing
\begin{equation}\label{global tate pairing for proof of rank zero}
   0 = \langle \kappa_N^D(q)_0, c_N \rangle = \sum_v \langle \kappa_N^D(q)_0, c_N\rangle_v. 
\end{equation}
For $v\not\in S\cup \set{q}$, the local Tate pairing vanishes by Proposition \ref{prop: check local conditions for kappa}(\ref{part:check_local_one}) and Lemma \ref{lemma kill pairing on crystalline}(\ref{lemma kill pairing on crystalline part two}) -- recall here that the local Tate pairing of two unramified classes is always trivial. By the same argument as \cite[Lemma 4.3(1)]{liu2016hirzebruch}, we may also pick a constant $C_3 \geq 0$ independent of $N$ such that, for all $v\in S$, $\varpi^{C_3}H^1(\Q_v, T_{\pi_1,N}) = 0$; hence $$\ord_\varpi \langle \kappa_N^D(q)_0, c_N\rangle_v \leq C_3, \; \forall v\in S.$$
It then follows from (\ref{global tate pairing for proof of rank zero}) that $$\ord_\varpi \langle \kappa_N^D(q)_0, c_N\rangle_q \leq C_3.$$
On the other hand, by Proposition \ref{prop:local_adm_free}, (\ref{unramified lower bound for proof of rank zero}) and (\ref{ramification lower bound for proof of rank zero}) together imply
$$\ord_\varpi \langle \kappa_N^D(q)_0, c_N\rangle_q \geq N - C_0  - C_1- C_2,$$
so we obtain a contradiction when $$N > C_0 + C_1 + C_2 + C_3.$$    
\end{proof}

\subsection{The endoscopic case}
\subsubsection{}
For completeness, we include an analogue of Theorem \ref{thm:rk_zero_main} in the endoscopic case. First, we require an analogue of Lemma \ref{lem:loc_cheb_rk0}.
\begin{lemma}\label{lem:loc_cheb_rk0_endo}
     Suppose $\pi$ is endoscopic associated to a pair $(\pi_1,\pi_2)$ of automorphic representations of $\GL_2(\A)$, and fix $i = 1$ or 2. Let $\p$ be a prime of $E_0$  such that there exist admissible primes for $\rho_\pi = \rho_{\pi,\p}$ which are BD-admissible for $\rho_{\pi_i}=\rho_{\pi_i,\p}$. Then there is a constant $C$ with the following property.

     For all integers $m \geq n \geq 0$ and for any cocycle $c\in  H^1(\Q, T_{\pi_i})$, there are infinitely many $m$-admissible primes $q$, BD-admissible for $\rho_{\pi_i}$, such that
    $$\ord_\varpi\loc_q c \geq \ord_\varpi c - C.$$
\end{lemma}

\begin{proof}
Without loss of generality, we may assume $i = 1$ (after possibly relabeling $\pi_1$ and $\pi_2$).
    \begin{claim}
        There exists a constant $C \geq 0$ such that $\varpi^C H^1(\Q(\rho_\pi)/\Q, T_{\pi_1,n}) = 0$ for all $n \geq 1$.
    \end{claim}
\begin{proof}[Proof of claim]
    By Corollary \ref{cor:Galois_coh_restr} and inflation-restriction, it suffices to show $$\Hom_{G_\Q} (\Gal(\Q(\rho_\pi)/\Q(\rho_{\pi_1})), T_{\pi_1,n})$$ is uniformly bounded in $n$. Note that, unless 
 $\pi_2$ has CM by an imaginary quadratic subfield  $K$ of $\Q(\rho_{\pi_1})$,  the abelianization of $\Gal(\Q(\rho_{\pi})/\Q(\rho_{\pi_1}))$ is finite. Indeed, if $\pi_2$ is CM, this can be checked by hand; otherwise one can use Theorem \ref{thm:nekovar} and Corollary \ref{corollary subalgebra of simple Lie algebra with base change}(\ref{corollary subalgebra of simple Lie algebra with base change part two})
    to see that no closed normal subgroup of $\Gal(\Q(\rho_{\pi_2})/\Q(\mu_{p^\infty}))$ can have an infinite abelian quotient.

    So we assume without loss of generality that $\pi_2$ has CM by $K \subset \Q(\rho_{\pi_1})$.
        In this case, one can calculate directly that complex conjugation acts on $\Gal(\Q(\rho_{\pi})/\Q(\rho_{\pi_1}))$ by $-1$, so the image of any Galois-invariant group homomorphism $\Gal(\Q(\rho_{\pi}, )/\Q(\rho_{\pi_1}))\to T_{\pi_1,n}$ lies in the $-1$ eigenspace of complex conjugation;  since $\overline T_{\pi_1}$ is absolutely irreducible and odd, we conclude that all such homomorphisms vanish.
\end{proof}

    Let $g\in G_\Q$ be an element which is admissible for $\rho_\pi$ and BD-admissible for $\rho_{\pi_1}$, which is possible by (the argument of) Lemma \ref{lem:admissible_TFAE}.
    By the claim and inflation-restriction, there exists $h\in G_{\Q(T_\pi)}$ such that $\ord_\varpi c(h) \geq \ord_\varpi c - C$; as in the proof of Lemma \ref{lem:loc_cheb_rk0} above, after possibly replacing $g$ with $gh$, we may assume without loss of generality that the component of $c(g)$ in the 1-eigenspace for $g$ is nonzero modulo $\varpi^{n - \ord_\varpi c + C + 1}$. Then any $q\not\in S \cup\set{p}$ with Frobenius conjugate to $g$ in $\Gal(\Q(T_{\pi, m}, c_n))$ satisfies the conclusion of the lemma. 
\end{proof}
The same proof of Theorem \ref{theorem of thorne} also shows:
\begin{prop}\label{prop:thorne_endoscopic}
    Suppose $\pi$ is endoscopic associated to a pair $(\pi_1,\pi_2)$. Then if $\p$ is any prime of $E_0$ of residue characteristic $p$ such that $\pi_p$ is unramified,
    $$H^1_f(\Q, \ad^0 \rho_{\pi_1,\p}) = H^1_f(\Q, \ad^0 \rho_{\pi_2,\p}) = 0.$$
\end{prop}
\begin{thm}\label{thm:rk_zero_endoscopic}
    Let $\pi$ be endoscopic associated to a pair $(\pi_1,\pi_2)$ of automorphic representations of $\GL_2(\A)$, which we order so that $\pi_{1,\infty}$ and $\pi_{2,\infty}$ have weights 2 and 4, respectively. Assume there exists a prime $\l_0$ such that $\pi_{1,\l_0}$ is discrete series.

   Fix $i = 1$ or 2. Then for any prime $\p$ satisfying Assumption \ref{assumptions_on_p}, such that there exist admissible primes for $\rho_{\pi,\p}$ which are BD-admissible for $\rho_{\pi_i,\p}$,
   if $H^1_f(\Q, V_{\pi_1,\p}\otimes V_{\pi_2,\p}(-1)) = 0$ then we have
   $$L(\pi,\operatorname{spin}, 1/2)  \neq 0  \implies H^1_f(\Q, V_{\pi_i,\p}) = 0.$$
\end{thm}
\begin{rmk}
    Note that $L(\pi, \operatorname{spin},1/2)$ is the product of central $L$-values for $\pi_1$ and $\pi_2$; in particular, Theorem \ref{thm:rk_zero_endoscopic} recovers (with extra conditions) the  result of Kato \cite{kato2004padic}.
\end{rmk}
\begin{proof}
The proof of Theorem \ref{thm:rk_zero_main} applies almost verbatim, with the following substitutions: Theorem \ref{thm:endoscopic_packets} for Theorem \ref{thm:JL}; Lemma \ref{lem:loc_cheb_rk0_endo} for Lemma \ref{lem:loc_cheb_rk0}; Proposition \ref{prop:thorne_endoscopic} for Theorem \ref{theorem of thorne};  and Theorem \ref{thm:1ERL_endoscopic}(\ref{thm:1ERL_endoscopic_one}) for Theorem \ref{thm:1ERL}(\ref{thm:1ERL_one}).
\end{proof}
\subsection{Rigidity and $\p$-integral vanishing of the Selmer group}
\subsubsection{}\label{subsubsec:where_rigid}
We can also give a more precise result on the vanishing of the dual Bloch-Kato Selmer group $H^1_f(\Q, V_{\pi,\p}/T_{\pi,\p})$, under some stronger conditions. Assume for this subsection that $\pi$ is not endoscopic.

We consider the following additional assumptions on $\p$ (the R stands for      ``rigid''). 
\begin{enumerate}[label = (R\arabic*)]
    \item \label{ass_A5_H1} The image of the $G_\Q$ action on $\overline T_{\pi,\p}$ contains a nontrivial scalar.
    \item\label{ass_R2} We have $\Sel_{\mathcal F^\rel}(\Q, \ad^0 \overline \rho) = 0$, with notation as in Definition \ref{def:loc_const_appendix}. 
    \item \label{ass_R3}For all $\l\in S$, $H^1(\Q_\l, \overline T_{\pi,\p})=0 $.   
\end{enumerate}
It is proved in Theorem \ref{thm:appendix_A5} that \ref{ass_A5_H1} holds for cofinitely many $\p$. We will show 
 in Proposition \ref{prop:rigid} below that the same is true of \ref{ass_R2} and \ref{ass_R3}.
\subsubsection{}\label{subsubsec:notation_for_rigid}
Under Assumption \ref{assumptions_on_p}(\ref{assume:irreducible}), 
let $R_\p^\univ$ be the universal $\GSP_4$-valued  deformation ring of $\overline\rho_{\pi,\p}$ denoted $R^1$ in Notation \ref{notation:global_def_ring}, with corresponding universal deformation $\rho_\p^{\univ}: G_{\Q, S\cup \set{p}} \to \GSP_4(R_\p^\univ)$.
\begin{lemma}\label{lemma subring of deformation ring generated by}
    Let $(R_\p^\univ)^0 \subset R_\p^\univ$ be the subring generated by the coefficients of the characteristic polynomials of elements of $G_\Q$ under the composite
    $$\rho_\p^\univ: G_\Q \to \GSP_4(R^\univ_\p) \hookrightarrow \GL_4(R^\univ_\p).$$ Then $$(R_\p^\univ)^0 = R_\p^\univ.$$
\end{lemma}
\begin{proof}
Recall the category $\CNL_O$ from (\ref{subsubsec:CNLO_notation}). 
    By the universal property of $\rho_\p^\univ$, it suffices to show the following: given two  morphisms $$f_1,f_2: R_\p^\univ \to A$$ in $\CNL_O$, if $f_1 = f_2 $ on $(R_\p^\univ)^0$, then  the associated deformations
    $$\rho_1,\rho_2: G_\Q \to \GSP_4(A) $$ are $\GSP_4(A)$-conjugate.     
    By \cite[Th\'eor\`eme 1]{carayol1994anneau}, $ \rho_1$ and $\rho_2$ are  $\GL_4(A)$-conjugate, say by a matrix $a\in \GL_4(A)$. Then 
    \begin{equation}\label{first identity for lemma with trace subring}
     \rho_1(g) = a \cdot \rho_2(g) \cdot a^{-1} \;\; \forall g \in G_\Q.
    \end{equation}
    On the other hand, because $\rho_1$ and $\rho_2$ are valued in $\GSP_4$ with similitude character $\chi_p^\cyc$, we also have
    \begin{equation}\label{second identity for lemma with trace subring}
       \rho_i(g) = -\chi_{p,\cyc}(g)\cdot  \Omega \cdot  \rho_i(g)^{-t} \cdot \Omega\;\;\forall g \in G_\Q,\; i = 1,2,
    \end{equation}
    where $\Omega$ is the matrix from (\ref{subsubsec:coordinates_gsp2n}) with $n = 2$.
    Combining (\ref{first identity for lemma with trace subring}) and (\ref{second identity for lemma with trace subring}) and using Schur's lemma, it follows that $\Omega a^{t} \Omega a$ is scalar, or equivalently $a \in \GSP_4(A)$, and this proves the lemma.
\end{proof}
\begin{lemma}\label{taylor wiles hypothesis lemma}
    For all but finitely many primes $\p$ of $E_0$, $\overline \rho_{\pi,\p}|_{G_{\Q(\zeta_p)}}$ is absolutely irreducible.
\end{lemma}
\begin{proof}
First, recall that $\overline\rho_{\pi,\p}$ is absolutely irreducible for all but finitely many $\p$ by Lemma \ref{lem:easy_assumptions_cofinite}. 
We restrict our attention to these $\p$;
   the argument is based on \cite[Proposition 4.5]{liu2024deformation}, but we are able to take advantage of the low-dimensionality of $\overline\rho_{\pi,\p}$. 
   The representation $\overline\rho_{\pi,\p}|_{G_{\Q(\zeta_p)}}$ is semisimple, so, after possibly extending scalars, we may write
   $$\overline\rho_{\pi,\p}|_{G_{\Q(\zeta_p)}} = \bigoplus_{i = 1}^M \rho_i^{\oplus m_i}$$
   for distinct absolutely irreducible representations $\rho_i$ of $G_{\Q(\zeta_p)}$ and multiplicities $m_i \geq 1$. The same argument as \cite[Lemma 2.1]{gelbart1982indistinguishability} (which applies here because $\Q(\zeta_p)/\Q$ has degree coprime to $p$) shows that the integers $m_i = m$ are all equal, and all $\rho_i$ have the same dimension $M$, such that
   $m^2 M $ is the number of characters $\nu$ of $\Gal(\Q(\zeta_p)/\Q)$ such that $\overline\rho_{\pi,\p} \cong \overline\rho_{\pi,\p} \otimes \nu$. Considering the self-duality of $\overline\rho_{\pi,\p}$, we see that this can occur only if $\nu^2 = 1$, so that $m^2 M \leq 2$. In particular, $m = 1$, so $\overline\rho_{\pi,\p}|_{G_{\Q(\zeta_p)}}$ is multiplicity-free, and the result now follows by the same argument as \cite[Proposition 4.5]{liu2024deformation}.
\end{proof}
\begin{prop}\label{prop:rigid}
 All but finitely many primes $\p$ of $E_0$ satisfy \ref{ass_R2} and \ref{ass_R3}.
\end{prop}
\begin{proof}
   First, choose an imaginary quadratic field $K$ such that all $\l\in S$ are split in $K$, and restrict attention to the cofinitely many $\p$ such that the underlying rational prime $p$ is unramified in $K$ and $\overline\rho_{\pi,\p}|_{G_{\Q(\zeta_p)}}$ is absolutely irreducible. 
   Then $\Q(\rho_{\pi, \p})\cap K$ is unramified at all finite primes, hence equal to $\Q$, and so \begin{equation}\label{image over G_K equals image over G_Q}
       \rho_{\pi,\p}(G_K) = \rho_{\pi, \p}(G_\Q).
   \end{equation}
Note as well that $\BC(\pi)\otimes \omega_{K/\Q} \not\cong \BC(\pi)$ (with $\BC(\pi)$ as in Lemma \ref{lem:BC}), for otherwise we would have $\rho_{\pi,\p} \otimes \omega_{K/\Q} \cong \rho_{\pi,\p}$, and it is easy to check using Schur's lemma that this implies $\rho_{\pi,\p}|_{G_K}$ is not absolutely irreducible for any $\p$, contradicting (\ref{image over G_K equals image over G_Q}) and Lemma \ref{lem:reducible_endoscopic} as long as $p > 3$. 
In particular, the base change $\Pi$ of $\BC(\pi)$ to $\GL_4(\A_K)$ is a regular algebraic, cuspidal automorphic representation of $\GL_4(\A_K)$ by \cite[Chapter 3, Theorems 4.2(a), 5.1]{arthur1989basechange}. It is also clear that $E_0$ is a strong coefficient field for $\Pi$. 

Recall the universal deformation $\rho_\p^\univ: G_\Q \to \GSP_4(R_\p)$ from (\ref{subsubsec:notation_for_rigid}).
Note that $K \cap \Q(\rho_\p^\univ)$ is unramified at all finite primes, hence equal to $\Q$; in particular, there exists $\gamma \in G_\Q$ such that $\rho_\p^\univ(\gamma) = 1$ and $\gamma$ acts nontrivially on $K$. 

Recall the group scheme $$\mathscr G \coloneqq (\GL_4\times \GL_1)\rtimes \set{1, \mathfrak c}$$ 
from \cite[\S2.1]{clozel2008automorphy}. By the recipe of \cite[Lemma 2.3(2)]{liu2024deformation}, $\rho^\univ_\p$ gives rise to a Galois representation
\begin{equation}\label{universal deformation in mathscr G}
    \widetilde\rho^\univ_\p: G_\Q \to \mathscr G(R^\univ_{\p}),
\end{equation}
with residual representation $\widetilde{\overline\rho}_\p: G_\Q \to \mathscr G(O_\p/\varpi_\p).$ 
Note
\begin{equation}\label{universal deformation in mathscr G after restriction to K}
    \widetilde\rho^\univ_\p |_{G_K} = (\rho^\univ_\p|_{G_K}, \chi_{p,\cyc}, 1): G_K\to (\GL_4 \times \GL_1)(R_\p^\univ) \rtimes \set{1,\mathfrak c},
\end{equation}
where $p$ is the residue characteristic of $\p$. 
By the argument of \cite[Theorem 4.8]{liu2024deformation}  applied to $\Pi$, we have for all but finitely many $\p$:
\begin{enumerate}
    \item \label{properties in rigidity prop one}$\widetilde{\overline{\rho}}_\p: G_\Q \to \mathscr G(O_\p/\varpi_\p)$ satisfies \cite[Definition 3.36]{liu2024deformation} for the pair $(S, \emptyset)$. 
    \item \label{properties in rigidity prop two}
    There are no regular algebraic conjugate self-dual cuspidal automorphic representations $\Pi'$ of $\GL_N(\A_K)$ such that $\Pi'$ is unramified outside primes above $S$, $\Pi'$ has the same archimedean weights as $\Pi$, and there is a congruence of associated Hecke eigensystems with respect to any isomorphism $\iota_p:  \overline \Q_p\isomorphism \C$ inducing $\p$. 
\end{enumerate}
For all but finitely many $\p$, we also have:
\begin{enumerate}
  \setcounter{enumi}{2}
    \item \label{properties in rigidity prop three}$\widetilde{\overline\rho}_{\pi,\p}|_{G_{K(\zeta_p)}}$ is absolutely irreducible
\end{enumerate}
by Lemma \ref{taylor wiles hypothesis lemma} and  (\ref{image over G_K equals image over G_Q}).

Restrict to the $\p$ satisfying the above properties with $p \not\in S$, and such that 
$\rho_{\p}^\univ|_{G_{\Q_p}}$ is Fontaine-Lafaille, i.e. $p \geq 5$. 
Let $\mathsf R_{\mathscr S}^\univ$  be the universal Fontaine-Laffaille deformation ring of $\widetilde{\overline\rho}_\p$ considered in \cite[p. 1630]{liu2024deformation} with $\Sigma_{\operatorname{min}}^+ = S$ and $\Sigma_{\lr}^+ = \emptyset$; more explicitly, $\mathsf R_{\mathscr S}^\univ$ classifies deformations that are unramified outside primes above $S\cup \set{p}$ and Fontaine-Laffaille at primes above $p$. 
It follows from the definition that (\ref{universal deformation in mathscr G}) corresponds to an $O_\p$-algebra morphism
\begin{equation}\label{deformation morphism in rigidity proof}
\mathsf R_{\mathscr S}^\univ \to R_\p^\univ.
\end{equation}

Now let $V$ be the unique four-dimensional Hermitian space over $K$ which is positive definite and split at all finite places. Then \cite[Theorem 3.38]{liu2024deformation} applied to $V$, together with properties (\ref{properties in rigidity prop one})-(\ref{properties in rigidity prop three}) above, shows that $\mathsf R_{\mathscr S}^\univ = O_\p$ for all but finitely many $\p$. We claim that these $\p$ satisfy \ref{ass_R2} and \ref{ass_R3}.

Now, 
by Remark \ref{rmk:appendix_iso_criterion}, \ref{ass_R2} is equivalent to requiring that the map $R_\p^\univ \to O_\p$ arising from $\rho_{\pi,\p}$ is an isomorphism, and so 
it suffices to  show that the morphism (\ref{deformation morphism in rigidity proof}) is surjective.   By (\ref{universal deformation in mathscr G after restriction to K}), the image of (\ref{deformation morphism in rigidity proof}) contains the coefficients of the characteristic polynomials of $\rho_\p^\univ(g)$  for $g\in G_K$, and so the desired surjectivity follows from Lemma \ref{lemma subring of deformation ring generated by} and the analogue of (\ref{image over G_K equals image over G_Q}) for $\rho_\p^\univ$. 

For \ref{ass_R3}, 
the local deformation ring of 
$$\widetilde{\overline \rho}_{\pi,\p}|_{G_{\Q_\l}}: G_{\Q_\l} \to \mathscr G(O_\p/\varpi_\p)$$
is formally smooth of relative dimension 16 over $O_\p$ for all $\l \in S$: this follows from 
 \cite[Proposition 3.33(3)]{liu2024deformation}, where Definition 3.36(1)  of \emph{op. cit.} is satisfied by property (\ref{properties in rigidity prop one}) above. In fact, because $\l$ splits in $K$, this is just the local deformation ring of $$\overline\rho_{\pi,\p}|_{G_{\Q_\l}}: G_{\Q_\l} \to \GL_4(O_\p/\varpi_\p)$$ (without any self-duality condition). The usual tangent space computation together with the formal smoothness then implies
$$16 = \dim_{O_\p/\varpi_\p} Z^1(\Q_\l, \End \overline T_\pi ) = 16 - \dim_{O_\p/\varpi_\p} H^0(\Q_\l, \End\overline T_\pi )+ H^1(\Q_\l, \End\overline T_\pi).$$
In particular, by the local Euler characteristic formula and local Poitou-Tate duality, we conclude 
\begin{equation}\label{eq:in_rigid}H^0(\Q_\l, \End \overline T_\pi(1))=0.\end{equation}
We then claim that $H^0(\Q_\l,\overline T_\pi) = 0$: otherwise $\overline\rho_{\pi,\p}$ has a $G_{\Q_\l}$-invariant line, so by  duality, it also has a quotient on which $G_{\Q_\l}$ acts by $\chi_{p,\cyc}$, and so there is a line in $ \End T_\pi$ on which $G_{\Q_\l}$ acts by $\chi_{p,\cyc}^{-1}$, contradicting (\ref{eq:in_rigid}). 

On the other hand, again by the local Euler characteristic formula and local duality, we see that
$\dim H^1(\Q_\l, \overline T_\pi) = 2\dim H^0(\Q_\l, \overline T_\pi) $, so \ref{ass_R3} follows.


\end{proof}
\begin{prop}\label{prop:chebotarev_mod_p}
    Suppose $\p$ satisfies Assumption \ref{assumptions_on_p}(\ref{assume:irreducible}) and
\ref{ass_A5_H1}. Then for any nonzero class $c\in H^1(\Q, \overline T_{\pi,\p})$ and any $N \geq 0$, there exist infinitely many $N$-admissible primes $q$ such that $\loc_q c\neq 0$.
\end{prop}
\begin{proof}
By the same argument as \cite[Proposition 9.1]{gross1989kolyvagin}, assumption \ref{ass_A5_H1} implies that $H^1(\Q(\overline T_\pi)/\Q, \overline T_\pi) = 0$, so by inflation-restriction $$c|_{G_{\Q(\overline T_\pi)}}: G_{\Q(\overline T_\pi)} \to \overline T_\pi$$ is nontrivial. Let $g\in G_\Q$ be an admissible element for $\rho_\pi$. 
Arguing as in Lemma \ref{lem:loc_cheb_rk0}, there exists $h \in G_\Q$ such that $h$ acts as $g$ on $\overline T_\pi$ and $c(h)$ has nonzero component in the 1-eigenspace for $h$. 
\begin{claim}
    We have $\Q(T_{\pi,N})\cap \Q(\overline T_\pi, c) = \Q(\overline T_\pi)$.
    \end{claim}
    \begin{proof}[Proof of claim]
It suffices to show any group $G$ with $G_\Q$-action which is a quotient of both the Galois groups
        $\Gal(\Q(\overline T_\pi, c)/\Q(\overline T_\pi))$ and $\Gal(\Q(T_{\pi,N})/\Q(\overline T_{\pi}))$, must be trivial. Note that $\Gal(\Q(\overline T_\pi, c)/\Q(\overline T_\pi))$ is an $\F_p[G_\Q]$-submodule of $\overline T_\pi$, so any element $z\in G_\Q$ that acts as a nontrivial scalar on $\overline T_\pi$ acts nontrivially on  $G$ unless $G=1$.

        On the other hand, the group
   $\Gal(\Q(T_{\pi,N})/\Q(\overline T_\pi))$ has a $G_\Q$-stable filtration in which each quotient is abelian and isomorphic to an $\F_p[G_\Q]$-subquotient of $\ad^0 \overline \rho_\pi$. In particular, $z$ acts trivially on $G$, which is an abelian quotient of $\Gal(\Q(T_{\pi,N})/\Q(\overline T_\pi))$. We conclude $G$ is trivial.
       \end{proof}
   By the claim, there exists  $\widetilde g\in G_\Q$ that acts as $g$ on $T_{\pi,N}$ and has image $h$ in $\Gal(\Q(\overline T_\pi,c)/\Q(\overline T_\pi))$. Any prime $q\not\in S\cup \set{p}$ with Frobenius conjugate to $\widetilde g$ in $\Gal(\Q(T_{\pi,N}, c)/\Q)$ satisfies the conclusion of the lemma. 
    
\end{proof}
\begin{definition}\label{def:divisible_Selmer}
Let $\p$ be a prime of $E_0$ such that $\overline T_{\pi,\p}$ is absolutely irreducible (so $T_{\pi,\p}$ is well-defined up to scalars). 

\begin{enumerate} \item We write $$W_{\pi,\p} = V_{\pi,\p}/T_{\pi,\p}.$$
As usual, we drop the subscript $\p$ when clarity permits.
\item For each rational prime $\l$, let $H^1_f(\Q_\l, W_{\pi})$ be the  annihilator of $H^1_f(\Q_\l, T_{\pi})$ under the perfect local Tate pairing
$$H^1(\Q_\l, W_{\pi}) \times H^1(\Q_\l, T_{\pi}) \to E/O,$$ and let $$H^1_f(\Q, W_{\pi}) \coloneqq \ker\left(H^1(\Q, W_{\pi}) \to \prod_\l \frac{H^1(\Q_\l, W_{\pi})}{H^1_f(\Q_\l, W_{\pi})}\right).$$
\end{enumerate}
\end{definition}
\begin{thm}\label{thm:rk_zero_integral}
    Suppose $\pi$ is non-endoscopic, and $\p$ is a prime of $E_0$ such that:
\begin{enumerate}
    \item\label{part:rk_zero_integral_ass} Assumption \ref{assumptions_on_p} and \ref{ass_A5_H1}-\ref{ass_R3} hold for $\p$, and there exist admissible primes for $\rho_{\pi,\p}$.
    \item\label{part:rk_zero_integral_lambda} There exists a squarefree $D > 1$ with $\nu(D)$ odd, such that $\pi_\l$ is transferrable for all $\l|D$ (Definition \ref{def:transferrable}) and $$\lambda^D(1)_\p \not\equiv 0 \pmod \p.$$
\end{enumerate}
Then    $$H^1_f(\Q, W_{\pi,\p}) = 0.$$
\end{thm}
\begin{proof}
We follow the proof of Theorem \ref{thm:rk_zero_main}, but with some integral refinements. Fix  $\p$ and $D$ satisfying (\ref{part:rk_zero_integral_ass}) and (\ref{part:rk_zero_integral_lambda}), and omit $\p$ from the notation for the rest of the proof.  Suppose for contradiction that $H^1_f(\Q, W_{\pi}) \neq 0$. By the long exact sequence in Galois cohomology associated to $$0 \to \overline T_\pi \to W_\pi \to W_\pi \to 0,$$
we have $$H^1(\Q, \overline T_\pi) = H^1(\Q, W_\pi)[\varpi];$$ in particular, there exists a class $0 \neq c\in H^1(\Q, \overline T_\pi)$
with the following property:
\begin{equation}\label{eq:c_v_rk0}
    \Res c_v\in H^1(\Q_v, \overline T_\pi) \text{ has trivial image in } \frac{H^1(\Q_v, W_\pi)}{H^1_f(\Q_v,W_\pi)} \text{ for all primes }v. 
\end{equation}

We fix an integer $N\geq 3$ sufficiently large to satisfy the conclusion of Lemma  \ref{lemma kill pairing on crystalline}(\ref{lem:kill_cryst_1.5}) for $n = 1$ and $M = T_\pi$, and greater than the number $n_0(1, \rho_\pi)$ from Theorem \ref{main appendix theorem}.\footnote{For cofinitely many $\p$, 
 all the local deformation rings of $\overline \rho_{\pi,\p}|_{G_{\Q_\l}}$ for $\l\in S\cup \set{p}$ will be formally smooth (by the same argument of Proposition \ref{prop:rigid} for $\l\in S$ and by \cite[Theorem A]{booher2023Gvalued} for $\l = p$), so one can take   $n_0(1, \rho_{\pi,\p}) = 1$. The assumption $N \geq 3$ is to conform to the statement of Theorem \ref{thm:1ERL} with $m = 1$, but in reality it is not needed since $N \geq 3m$ is used only in Theorem \ref{thm:1ERL}(\ref{thm:1ERL_two}). In particular, the argument would work with $N = 1$ in practice.}

    By Proposition \ref{prop:chebotarev_mod_p} and the assumption \ref{ass_A5_H1}, we may choose an $N$-admissible prime $q\not\in S$ such that 
    \begin{equation}\label{eq:in_rk0_rigid_loc}
        \loc_q c \neq 0.
    \end{equation}
    
    Then by Theorem \ref{thm:1ERL}(\ref{thm:1ERL_one}) and the assumption \ref{ass_R2}, we  have an element $\kappa_N^D(q)_0 \in \kappa^D_N(q)$ such that its image $\kappa_1^D(q)_0$ in $H^1(\Q, \overline T_\pi)$ satisfies
    \begin{equation}\label{eq:in_rk0_rigid_q} \partial_q \kappa_1^D(q)_0\neq 0.\end{equation}
    We now consider the global Tate pairing
    \begin{equation}\label{eq:in_rk0_rigid}0 = \sum_v\langle c, \kappa_1^D(q)_0 \rangle_v.\end{equation}   
    By the assumption \ref{ass_R3}, the local terms vanish for all $v \neq q, p$. 
The local term at $v = q$ is nonzero by Proposition \ref{prop:local_adm_free} combined with (\ref{eq:in_rk0_rigid_loc}), (\ref{eq:in_rk0_rigid_q}). So to obtain a contradiction with (\ref{eq:in_rk0_rigid}), it suffices to show the local pairing at $p$ vanishes.
Indeed, the maps 
$$\alpha: H^1(\Q_p, T_\pi)\to H^1(\Q_p, \overline T_\pi),\;\; \beta: H^1(\Q_p, \overline T_\pi) \to H^1(\Q_p, W_\pi)$$
are adjoint with respect to the local Tate pairings, and by Lemma \ref{lemma kill pairing on crystalline}(\ref{lem:kill_cryst_1.5}) combined with Proposition \ref{prop: check local conditions for kappa}, there exists $d\in H^1_f(\Q_p, T_\pi)$ such that $\alpha(d) = \Res_p \kappa_1^D(q)_0.$ Hence indeed
$$\langle \kappa_1^D(q)_0, c\rangle_p = \langle d, \beta(\Res_p c)\rangle_p = 0$$ by (\ref{eq:c_v_rk0}).
\end{proof}

\begin{cor}\label{cor:rk_zero_integral}
    Suppose $\pi$ is relevant and non-endoscopic, and there exists a place $\l_0$ such that $\pi_{\l_0}$ is transferrable (Definition \ref{def:transferrable}).
    If $L( \pi, \operatorname{spin},1/2) \neq 0$, then for all but finitely many primes $\p$   such that admissible primes exist for $\rho_{\pi,\p}$, $H^1_f(\Q, W_{\pi,\p}) = 0$. 
\end{cor}

\begin{proof}
By Theorem \ref{thm:JL}, $\pi_f^{\l_0}$ can be completed to an automorphic representation of $\spin(V_{\l_0})(\A)$.
    Thus the corollary  follows from Theorem \ref{thm:rk_zero_integral} combined with Proposition \ref{L values related to lambda elements}, Lemma \ref{lem:easy_assumptions_cofinite},  Theorem \ref{thm:appendix_A5}, and Proposition \ref{prop:rigid}.  
\end{proof}
Conditions are given in Theorem \ref{thm:when_adm_primes} under which admissible primes exist for $\rho_{\pi,\p}$ for cofinitely many $\p$. In particular: 
\begin{cor}\label{cor:rk_zero_integral_IIa}
        Suppose $\pi$ is relevant and non-endoscopic, and there exists a place $\l$ such that $\pi_\l$ is type IIa in the sense of \cite{roberts2007local}.
    If $L( \pi, \operatorname{spin},1/2) \neq 0$, then for all but finitely many primes $\p$, $H^1_f(\Q, W_{\pi,\p}) = 0$. 
\end{cor}
\subsection{Applications to automorphic inductions}\label{subsec:corollaries_rk0}
In this section, we give some  corollaries of Theorem \ref{thm:rk_zero_main} which may be of independent interest. Both of them could be upgraded to statements about dual Selmer groups using Theorem \ref{thm:rk_zero_integral};   we omit the details only for concision.

In the next two corollaries, when $\pi$ is a cuspidal unitary automorphic representation of $\GL_2(\A)$ such that $\pi_\infty$ is discrete series of weight $k\geq 2$ and $\iota_p: \overline\Q_p \isomorphism \C$ is an isomorphism, the associated $p$-adic Galois representations $\rho_{\pi,\iota_p}$ are normalized so that $\det \rho_{\pi,\iota_p} = \chi_{p,\cyc}^{k-1}  \omega_\pi$, where $\omega_\pi$ is the central character. We also have the usual $p$-adic Galois representation $\chi_{\iota_p}$ associated to any algebraic automorphic character $\chi$ of $\A_K^\times$, with $K/\Q$ a number field. We write the (semisimplified) reductions mod $p$ as $\overline\rho_{\pi,\iota_p}$ and $\overline \chi_{\iota_p}$.
\begin{cor}\label{cor:IQ_ind_rk0}
Let $\pi$ be a non-CM cuspidal unitary automorphic representation of $\GL_2(\A)$ with $\pi_\infty$ discrete series of weight 3, and with central character $\omega_\pi$. Let $K$ be an imaginary quadratic field and $\chi: \A_K^\times \to \C^\times$  an automorphic character of infinity type $(-1,0)$ such that $\chi|_{\A_{\Q}^\times} = |\cdot|\omega_\pi^{-1}.$

Fix an isomorphism $\iota_p: \overline\Q_p \isomorphism \C$ and assume:
\begin{enumerate}
    \item $p$ splits in $K$ and is coprime to the conductor of $f$ and $\chi$.
    \item\label{item:cor_IQ_indecomp} For some inert nonarchimedean place $v$ of $K$, $\WD(\rho_{\pi,\iota_p}|_{G_{K_v}})$ is indecomposable.
    \item  \label{item:cor_IQind_residual}$\overline\rho\coloneqq \overline\rho_{\pi,\iota_p} \otimes \Ind_{G_K}^{G_\Q} \overline\chi_{\iota,p}$ satisfies:
\begin{enumerate}
\item $\overline\rho$ is absolutely irreducible and generic (Definition \ref{def:generic}).
    \item\label{item:cor_IQind_adm} There exists a prime $q$ such that $\overline\rho|_{G_{\Q_q}}$ is unramified, $q^4\not\equiv 1\pmod p$,  and $\overline\rho (\Frob_q)$ has eigenvalues $\set{q,1,\alpha, q/\alpha}$ with $\alpha\not\in \set{\pm 1, \pm q, q^2, q^{-1}}$.
    \end{enumerate}
\end{enumerate}
Then $$L(f, \chi, 1/2) \neq 0 \implies H^1_f(K, \rho_{\pi,\iota_p}\otimes \chi_{\iota_p}) = 0.$$

Moreover, the conditions in (\ref{item:cor_IQind_residual})  hold for all but finitely many $p$ split in $K$ and all choices of $\iota_p$.
\end{cor}

\begin{rmk}
The condition that $p$ split in $K$ is actually necessary for (\ref{item:cor_IQind_adm}) to hold.
\end{rmk}

\begin{proof}
    Let $\BC_K(\pi)$ be the base change to $\GL_2(\A_K)$. By \cite[Chapter 3, Theorems 4.2(e), 5.1]{arthur1989basechange}, there is a (strong) automorphic induction $\Pi$ of $\BC_K(\pi)\otimes \chi$ to $\GL_4(\A)$. Then by \cite[Theorem C]{ramakrishnan2007siegeldescent}, there is an automorphic representation $\widetilde{\pi}$ of $\GSP_4(\A)$ with trivial central character, such that  $\Pi$ is the base change of $\widetilde{\pi}$ as in Lemma \ref{lem:BC}. Note that $\widetilde{\pi}$ is relevant by a direct calculation with archimedean $L$-parameters.
    We have $$\rho_{\widetilde{\pi}, \iota_p} = \rho_{\pi,\iota_p} \otimes \Ind_{G_K}^{G_\Q} \chi_{\iota_p},$$ where the symplectic structure is by viewing $\rho_{\pi,\iota_p}$ as symplectic and $\Ind_{G_K}^{G_\Q} \chi_{\iota_p}$ as orthogonal. 
The ``moreover'' assertion of the corollary therefore follows from Lemma \ref{lem:image_IQ} combined with Lemma \ref{lem:easy_assumptions_cofinite}.

For the rest, by Shapiro's Lemma we have $H^1_f(\Q, \rho_{\widetilde \pi,\iota_p}) = H^1_f(K, \rho_{\pi,\iota_p} \otimes \chi_{\iota_p})$. So by 
Theorem \ref{thm:rk_zero_main} applied to $\widetilde{\pi}$, it suffices to 
  show that there exists a prime $\l$ such that $\widetilde{\pi}_\l$ is transferrable.  

    Let $v$ be the place in (\ref{item:cor_IQ_indecomp}), and let $\l$ be the rational prime underlying $v$. 
Comparing Definition \ref{def:transferrable} with the explicit local Langlands parameters in \cite[Table A.7]{roberts2007local} (and using as well Theorem \ref{thm:rho_pi_LLC}(\ref{part:rho_pi_LLC1}) to see that $\widetilde{\pi}_\l$ is tempered), we see that it suffices to show the associated Weil-Deligne representation 
$$\tau_\l: W_{\Q_\l} \times \SL(2,\C) \to \GSP_4(\C)$$ of 
$\rho_{\widetilde{\pi},\iota_p}|_{G_{\Q_\l}}$ does not factor through a Siegel parabolic subgroup; equivalently, $\tau_\l$ does not stabilize an isotropic plane. 
Let $W$ be the underlying two-dimensional complex symplectic space of the  Weil-Deligne representation
$\tau_{\pi,\l}: W_{\Q_\l} \times \SL(2, \C) \to \GL_2(\C)$
corresponding to $\rho_{\pi,\iota_p}|_{G_{\Q_\l}}$, and let $V= \Ind_{W_{K_v}}^{W_{\Q_\l}} \tau_{\chi, v},$ where $\tau_{\chi,v}: W_{K_v} \to \C^\times$ is the character corresponding to $\chi_{\iota_p}|_{G_{K_v}}$. 

 In particular, there is an isotropic basis $\set{e_1,e_2}$ for $V$ stable (as a set) under $W_{K_v}$. Suppose for contradiction that $I \subset W\otimes V$ is a $W_{\Q_\l} \times \SL(2, \C)$-stable isotropic plane. Because $v$ is inert, $I \neq W\otimes e_1,W\otimes e_2$. In particular, it follows that $$I = \set{w\otimes e_1 + g(w) \otimes e_2,\; w\in W}$$ for some $g\in \GL_2(\C)$. Then $g$ commutes with $\tau_{\pi,\l}(W_{K_v} \times \SL(2, \C)) $, hence is scalar by (\ref{item:cor_IQ_indecomp}); but clearly such an $I$ is not isotropic, so we have obtained a contradiction. 
\end{proof}

\begin{cor}
    Let $K$ be a real quadratic field, and let $\pi$ be a cuspidal automorphic representation of $\PGL_2(\A_K)$ with $\pi_v$ discrete series of weights 2 and 4, in some order, for the two places $v|\infty$ of $K$.

Let $p$ be a prime  and let $\iota_p: \overline\Q_p \isomorphism \C$ be an isomorphism such that:
 \begin{enumerate}
     \item $p$ is unramified in $K$ and  coprime to the conductor of $\pi$.
     \item $\pi_v$ is discrete series for some nonarchimedean place $v$ of $K$, and if $v$ is split, then $\pi_v\not\cong \pi_{\overline v}$. 
     \item $\overline\rho\coloneqq \Ind_{G_K}^{G_\Q} \overline \rho_{\pi,\iota_p}$ satisfies:
     \begin{enumerate}
\item $\overline\rho$ is absolutely irreducible and generic (Definition \ref{def:generic}).
    \item There exists a prime $q$ such that $\overline\rho|_{G_{\Q_q}}$ is unramified, $q^4 \not\equiv 1\pmod p$, and $\overline\rho (\Frob_q)$ has eigenvalues $\set{q,1,\alpha, q/\alpha}$ with $\alpha\not\in \set{\pm 1, \pm q, q^2, q^{-1}}$.
    \end{enumerate} \end{enumerate}
Then $$L(\pi, 1/2) \neq 0 \implies H^1_f(K, \rho_{\pi,\iota_p}) = 0.$$

Moreover:
\begin{itemize}
\item 
If $\pi$ is non-CM and not Galois-conjugate to a twist of $\pi\circ \tau$, where $\tau\in \Gal(K/\Q)$ is a generator, then the conditions in
(\ref{item:cor_IQind_residual}) hold for all but finitely many $p$  and all choices of $\iota_p$.
\item If $\pi$ is non-CM and Galois-conjugate to a twist of $\pi\circ\tau$, then then the conditions in
(\ref{item:cor_IQind_residual}) hold for all but finitely many $p$ split in $K$. 
\item If $\pi$ has CM by a totally imaginary quadratic extension $F/K$, then the conditions in (\ref{item:cor_IQind_residual}) hold for all but finitely many $p$ split completely in $F$.
\end{itemize}
\end{cor}

\begin{proof}
   By the same argument as for Corollary \ref{cor:IQ_ind_rk0}, there exists an automorphic representation $\widetilde{\pi}$ of $\GSP_4(\A)$ such that the base change of $\widetilde{\pi}$ to $\GL_4(\A)$ is the automorphic induction of $\pi$. Again it suffices to show that $\widetilde\pi_\l$ is transferrable, where $\l$ is the rational prime underlying $v$ from (\ref{item:cor_IQ_indecomp}); in this case, the ``moreover'' assertions follow from Lemmas \ref{lem:image_RQ} and \ref{lem:image_quartic}, combined with Lemma \ref{lem:easy_assumptions_cofinite}.

   Now, let $W$ be the 
   underlying symplectic space of the
   associated Weil-Deligne representation
   $$\tau_\l: W_{\Q_\l} \times \SL(2, \C) \to \GSP_4(\C)$$ to $\rho_{\widetilde{\pi}, \iota_p}|_{G_{\Q_\l}}$. As in the proof of Corollary \ref{cor:IQ_ind_rk0}, we wish to show that $\tau_\l$ does not factor through a Siegel parabolic subgroup, so  suppose for contradiction that $W$ contains a $W_{\Q_\l}\times \SL(2,\C)$-stable isotropic plane $I$. 
   We have a $W_{K_v}\times \SL(2, \C)$-stable decomposition $$W = W_1 \oplus W_2,$$ where $W_{K_v}\times \SL(2, \C)$ acts on $W_1$ through the Weil-Deligne representation corresponding to $\rho_{\pi_v,\iota_p}|_{\Q_\l}$ -- in particular, irreducibly because $\pi_v$ is discrete series. The symplectic form is nondegenerate on each of $W_1$ and $W_2$, so  we conclude that  $$I = \set{w + \l(w) \,:\, w\in W_1}$$ for some linear isomorphism $\l: W_1 \isomorphism W_2$. This $\l$ is necessarily $W_{K_v}\times \SL(2,\C)$-intertwining, so the assumption in (\ref{item:cor_IQ_indecomp}) means $v$ is an inert prime. Hence $W = W_1\otimes V$ where $V =\Ind_{W_{K_v}}^{W_{G_{\Q_\l}}} \C$, and the same argument as in Corollary \ref{cor:IQ_ind_rk0} shows again that no such $I$ can exist.
\end{proof}
\section{The second explicit reciprocity law: geometric inputs}\label{sec:2ERL_geometry}
\subsection{Setup and notation}
\subsubsection{}\label{subsubsec:where_T_assumptions}
Let $D\geq 1$ be squarefree with $\sigma(D)$ even, and recall the quadratic space $V_D$ from (\ref{subsubsec:B_D_notation}). For this section, we suppose fixed a matrix $T\in \symmetric^2(\Q)_{>0}$ satisfying:
\begin{enumerate}[label = (T\arabic*)]
    \item\label{ass:T1} $T_{11}\in \Q^\times\setminus (\Q^\times)^2$.
    \item \label{ass:T2}The two-dimensional quadratic space defined by $T$ has nontrivial local Hasse invariant for some prime $\l\nmid D$.
\end{enumerate}

\subsubsection{}
Choose a base point $(e_1^T,e_2^T)\in \Omega_{T, V_D}(\Q)$ (Construction  \ref{constr:Z_T_phi}(\ref{constr:Z_T_phi_omega})), and let \begin{equation}
    V_D^\circ \coloneqq (e_1^T)^\perp \subset V_D, \;\;  V_T \coloneqq \operatorname{span}_\Q\set{e_1^T,e_2^T}\subset V_D,\;\;V_D^\diamondsuit \coloneqq V_T^\perp \subset V_D^\circ.
\end{equation}
Then $V_D^\circ$ is a four-dimensional quadratic space with discriminant field $F\coloneqq \Q(\sqrt{T_{11}})$.

\subsubsection{}\label{subsubsec:g0_setup_2ERL} Let $K = \prod K_v \subset \spin(V_D)(\A_f)$ be a neat compact open subgroup, 
 and fix, throughout this section, an element $g_0 = \prod_v g_{0,v} \in\spin(V_D)(\A_f)$. 
 We write 
$$\begin{cases}
    K_v^\diamondsuit = g_{0,v}K_vg_{0,v}^{-1} \cap \spin (V^\diamondsuit_D)(\Q_v) \\
    K_v^\circ = g_{0,v}K_vg_{0,v}^{-1} \cap \spin (V^\circ_D)(\Q_v)
\end{cases}$$
for all finite places $v$ of $\Q$, and $K^{?} \coloneqq \prod K_v^?$ for $? = \diamondsuit, \circ$. The special cycle $Z(g_0, V_T, V_D)_K$ factors as:
\begin{equation}\label{factoring special cycle for 2nd ERL}
 \Sh_{K^\diamondsuit} (V_D^\diamondsuit)  \to \Sh_{K^\circ}(V_D^\circ) \xrightarrow{\cdot g_0}\Sh_K(V_D).
\end{equation}

\subsection{Integral models at good primes}
Fix a prime $q\nmid D$ satisfying the following:
\begin{assumption}\label{assume:q_2ERL}
\leavevmode
\begin{enumerate}
\item $ T$ lies in $\GL_2(\Z_{(q)})\subset M_2(\Q)$.
    \item $T_{11}$ lies in $(\Z_q^\times \setminus (\Z_q^\times)^2 )\cap \Q.$
    \item $K_q$ is hyperspecial and  $g_q\in \spin(V_D)(\Q_q)$ lies in $K_q$.
\end{enumerate}
\end{assumption}

\begin{notation}\label{notation:algebras_involutions_2ERL}
\leavevmode
\begin{enumerate}
\item 
Let $O_D\subset B_D$ be a maximal $\Z_{(q)}$-order. 
\item 
The lattice $\Span_{\Z_{(q)}}\set{e_1^T,e_2^T}$ defines a maximal $\Z_{(q)}$-order $\O_T \subset C(V_T)$, with the natural positive nebentype involution.
\item 
Let $\O_F\subset \O_T$ be the subalgebra generated by $e_1^T$, which is the unique maximal 
$\Z_{(q)}$-order in $F$. 
\item Fix an arbitrary positive involution $\ast$ of $O_D$ (necessarily nebentype). The Clifford involution $\ast$ is positive and nebentype on $\O_T$, and stabilizes $\O_F$. 
\end{enumerate}
\end{notation}
\subsubsection{}
Under Notation \ref{notation:algebras_involutions_2ERL}, we have the chain of embeddings of $\Z_{(q)}$-algebras  with positive involutions:
\begin{equation}\label{embeddings of maximal orders in algebras for 2nd ERL}
O_D\hookrightarrow O_D\otimes \O_F \hookrightarrow O_D\otimes \O_T.
\end{equation}
We now use  (\ref{embeddings of maximal orders in algebras for 2nd ERL}) to describe $q$-integral models for the cycles (\ref{factoring special cycle for 2nd ERL}). 
\begin{construction}
    \label{constr:integral models 2nd ERL}
Using Corollary \ref{cor:exists_basepoint}, we fix a four-dimensional abelian scheme $A_0$ over $\breve \Z_q$ of supersingular reduction, equipped with:
\begin{enumerate}
    \item An embedding $\iota_0^\diamondsuit: O_D\otimes_{\Z_{(q)}} \O_T\hookrightarrow \End(A_0) \otimes_{\Z} \Z_{(q)}.$
    \item A prime-to-$q$ quasi-polarization $\lambda_0: A_0 \to A_0^\vee$ such that 
    $$\iota_0^\diamondsuit(\alpha)^\vee \circ \lambda_0 = \lambda_0 \circ\iota_0^\diamondsuit(\alpha^\ast),\;\; \forall \alpha\in O_D\otimes \O_T.$$
\end{enumerate}
By restricting along (\ref{embeddings of maximal orders in algebras for 2nd ERL}), we also have $\iota_0: O_D\hookrightarrow\End(A_0)\otimes \Z_{(q)}$ and $\iota_0^\circ: O_D\otimes_{\Z_{(q)}} \O_F \hookrightarrow\End(A_0)\otimes \Z_{(q)}$. 
Note that $(A_0,\iota_0, \lambda_0)$ is an $(O_D, \ast)$-triple in the sense of Definition 
\ref{def:O_Dtriple}.  Using Remark \ref{rmk:unif_datum}(\ref{rmk:unif_datum_one}), we extend this to a $q$-adic uniformization datum $(A_0, \iota_0, \lambda_0, i_D, i_{Dq})$ for $(O_D, \ast)$. 
Consider the three PEL data with self-dual $q$-integral refinements:
\begin{align*}
    \mathcal D^\diamondsuit &= (B_D\otimes_{\Q} C(V_T), \ast, H, \psi),\;\; &\mathscr D^\diamondsuit = (O_D\otimes_{\Z_{(q)}} \O_T, \ast, \Lambda, \psi) \\
      \mathcal D^\circ &= 
      (B_D\otimes_{\Q} F, \ast, H, \psi),\;\; &\mathscr D^\circ =
      (O_D\otimes_{\Z_{(q)}} \O_F, \ast, \Lambda, \psi) \\
        \mathcal D &= (B_D, \ast, H, \psi), \;\; &\mathscr D =
        (O_D, \ast, \Lambda, \psi) 
\end{align*}
arising from $A_0$. Also write 
 $K^{q?} \coloneqq \prod_{v\neq q} K_v^?$ for $? = \diamondsuit, \circ$.

For $? = \diamondsuit, \circ$, or $\emptyset$,
let $ X^?$ be the smooth quasiprojective scheme  over $\Z_{(q)}$ representing
the moduli functor $\mathcal M^?_{K^{q?}}$ associated to $\mathscr D^?$ at level $K^{q?}$. 
\end{construction}

\begin{lemma}\label{lem:proper}
    The scheme $ X^\diamondsuit$ is proper over $\Spec \Z_{(q)}$.
\end{lemma}
\begin{proof}
    The biquaternion algebra $B_D \otimes C(V_T)$ is nonsplit by \ref{ass:T2}. If $B_D \otimes C(V_T) = M_2(B_d)$ for some squarefree $d > 1$, then $\sigma(d)$ is necessarily even. By  Corollary \ref{cor:changing_polarizations_PEL} combined with Propositions \ref{prop:positive involutions} and \ref{prop:conj_OF},  $ X^\diamondsuit$ can be identified with the canonical (smooth) integral model of the Shimura curve attached to $B_d$ at level $K^{\diamondsuit}$. The latter is well-known to be proper. 
\end{proof}
\subsubsection{}

We have natural finite maps
\begin{equation}\label{integral models factoring special cycle for 2nd ERL}
     X^\diamondsuit\to  X^\circ \to  X,
\end{equation}
defined on the level of moduli problems by 
\begin{align*}
(A, \iota^\diamondsuit, \lambda, \eta) &\mapsto (A, \iota^\diamondsuit|_{O_D\otimes_{\Z_{(q)}}\O_F}, \lambda, \eta)\\
(A, \iota^\circ, \lambda, \eta) &\mapsto (A, \iota^\circ|_{O_D}, \lambda, g_0^q\cdot\eta).
\end{align*}
Let $X^?_\Q$ denote the generic fiber of $ X^?$, for $? = \diamondsuit,\circ$, or $\emptyset$, and let $X^?_{\F_q}$ denote the special fiber.
\begin{prop}\label{moduli interpretation recovers factoring of special cycle on generic fiber}
    There are isomorphisms
    $$X^?_\Q \cong \Sh_{K^?} (V_D^?)$$
    for $? = \diamondsuit,\circ$, or $\emptyset$, such that the generic fiber of (\ref{integral models factoring special cycle for 2nd ERL}) recovers (\ref{factoring special cycle for 2nd ERL}).
\end{prop}
\begin{proof}
This follows from the discussion in \cite[\S2]{kudla2000siegel}; note the isomorphisms depend on our choice of $q$-adic uniformization datum in Construction \ref{constr:integral models 2nd ERL}.
\end{proof}

\subsubsection{}
We let $O$ be the ring of integers of a finite extension of $\Q_p$, where $q\neq p$, and let $\varpi \in O$ be a uniformizer. 
\begin{lemma}\label{base change for 2nd ERL}
    For all $i$ and for $? = \diamondsuit, \circ$, or $\emptyset$, there are  canonical $G_{\Q_q}$-equivariant isomorphisms \begin{equation}
        \tag{$\operatorname{BC}_{X^?}$}
    \begin{split}
        H^i_{\et}(\Sh_{K^?}(V_D^?)_{\overline\Q}, O) &\cong H^i_{\et} (X^?_{\overline\F_q}, O) \\
         H^i_{\et,c}(\Sh_{K^?}(V_D^?)_{\overline\Q}, O) &\cong H^i_{\et,c} (X^?_{\overline\F_q}, O).
         \end{split}
    \end{equation}
    These isomorphisms commute with the actions of prime-to-$q$ Hecke correspondences and with the pullback and pushforward maps induced by (\ref{integral models factoring special cycle for 2nd ERL}).
\end{lemma}
\begin{proof}
Let $R\Psi_{X^?}O$ denote the nearby cycles complex on $X^?_{\overline \F_q}$. 
    Since $ X^?$ is smooth over $\Z_{(q)}$, the natural map $O \to R\Psi_{X^?}O$ is an isomorphism. On the other hand, by \cite[Corollary 5.20]{lan2018nearby}, the base change map 
    $$R\Gamma_\et (X^?_{\overline \Q}, O) \to R\Gamma_\et(X^?_{\overline \F_q}, R\Psi_{X^?}O)$$
    is also an isomorphism, and the lemma follows.
\end{proof}
\subsection{Unramified Rapoport-Zink spaces}
\begin{notation}\label{notation:V_Dq_etc}
\leavevmode
\begin{enumerate}\item \label{item:V_Dq_etc}Recall from Remark \ref{rmk:unif_datum}(\ref{rmk:unif_datum_two}) that our choice of $q$-adic uniformization datum in Construction \ref{constr:integral models 2nd ERL} entails a choice of isomorphism
$$\End^0(\overline A_0, \overline \iota_0)^{\dagger = 1, \tr = 0} \isomorphism V_{Dq},$$ hence an inclusion $V_T \hookrightarrow V_{Dq}$. Let $$V_{Dq}^\diamondsuit
\coloneqq V_T^\perp \subset V_{Dq}$$ and $$V_{Dq}^\circ  = (e_1^T)^\perp \subset V_{Dq}.$$
\item For each $? = \diamondsuit, \circ,$ or $\emptyset$, let $\mathcal N^?$ denote the Rapoport-Zink space over $\Spf \breve \Z_q$ parametrizing framed polarized deformations $(X, \iota^?, \lambda, \rho)$ of $(\overline A_0[q^\infty], \overline \iota_0^?\otimes \Z_q, \lambda_0)$, where $\overline\iota_0^\diamondsuit \otimes \Z_q : O_D\otimes \O_T \otimes \Z_q \hookrightarrow \End^0 (\overline A_0[q^\infty])$ is the induced embedding, and likewise for $\overline\iota_0^\circ$, $\overline\iota_0$. We let $\mathcal M^?$ denote the underlying reduced scheme of $\mathcal N^?$. 
\item\label{notation:V_Dq_etc_phi} Let $  \phi: \mathcal N^? \isomorphism\sigma^\ast \mathcal N^?$ be the natural Weil descent datum as in (\ref{subsubsec:Weil_descent}).
\end{enumerate}
\end{notation}
\subsubsection{}
 We have natural closed embeddings 
\begin{equation}\label{eq:RZ_embeddings_2ERL}
    \mathcal N^\diamondsuit\hookrightarrow\mathcal N^\circ\hookrightarrow\mathcal N
\end{equation}
compatible with the actions of $$\spin(V_{Dq}^\diamondsuit)(\Q_q) \subset \spin(V_{Dq}^\circ)(\Q_q) \subset\spin(V_{Dq})(\Q_q).$$ 
From the Rapoport-Zink uniformization theorem, we deduce:
\begin{prop}\label{RZ uniformization for 2nd ERL}
    For each $? = \diamondsuit, \circ$, or $\emptyset$, let $X^{\ss?}_{\overline\F_q}$ denote the supersingular locus. Then there is a canonical isomorphism
    $$X^{\ss?}_{\overline\F_q} \cong \spin(V_{Dq}^?)(\Q)\backslash \spin(V_{Dq}^?) (\A^q_f)\times \mathcal M^? /K^{q?},$$
    compatible with prime-to-$q$ Hecke correspondences,  Frobenius action, and  the maps arising from (\ref{integral models factoring special cycle for 2nd ERL}), (\ref{eq:RZ_embeddings_2ERL}).
\end{prop}
Here $K^{q?}$ is viewed as a subgroup of $\spin(V_{Dq}^?)(\A^q_f)$ by Remark \ref{rmk:unif_datum}(\ref{rmk:unif_datum_two}).
\begin{prop}\label{prop:KR_M_structure}
\leavevmode
\begin{enumerate}
\item Each irreducible component of $\mathcal M^\circ$ or $\mathcal M$ is isomorphic to $\mathbb P^1_{\overline\F_q}$; in particular, $\mathcal M^\circ$ is a union of irreducible components of $\mathcal M$. 
    \item \label{prop:KR_M_struct_full} The group $\spin(V_{Dq})(\Q_q)$ acts transitively on the set of irreducible components of $\mathcal M$, and the stabilizer of each component is a paramodular subgroup.
\item \label{prop:KR_M_struct_circ} There are two $\spin(V_{Dq}^\circ)(\Q_q)$-orbits of irreducible components of $\mathcal M$ interchanged by $\phi$, and the stabilizer of each component is a hyperspecial subgroup. 
\item   \label{prop:KR_M_struct_phi}  For any irreducible component $A \subset \mathcal M$, we have $\phi^2(A) = (\sigma^2)^\ast \left(\langle q\rangle \cdot A\right). $
\end{enumerate}
\end{prop}
\begin{proof}
    See \cite[\S4]{kudla2000siegel} and \cite[\S4]{kudla1999hirzebruch} for the structure of $\mathcal M$ and $\mathcal M^\circ$, respectively.
\end{proof}

\begin{notation}\label{notation:irr_comps_labels_2ERL}
    \leavevmode
    \begin{enumerate}
        \item Fix an irreducible component $\mathcal M(1)\subset \mathcal M^\circ$ as a basepoint, and let $K_q^\paramodular\subset \spin(V_{Dq})(\Q_q)$ be the stabilizer of $\mathcal M(1)$.
        \item For all  $g\in \spin(V_{Dq})(\Q_q)/K_q^\paramodular$, define $$\mathcal M(g)\coloneqq g\cdot \mathcal M(1)\subset \mathcal M.$$ By Proposition \ref{prop:KR_M_structure}(\ref{prop:KR_M_struct_full}), this defines a bijection between $\spin(V_{Dq})(\Q_q)/K_q^\paramodular$ and the irreducible components of $\mathcal M$. 
        \item Let $F\in \spin(V_{Dq})(\Q_q)$ be an element normalizing $K_q^\paramodular$ such that $F^2 = \langle q\rangle$ and 
        $\phi(\mathcal M(g)) = \sigma^\ast \mathcal M(gF)$ for all $g\in \spin(V_{Dq})(\Q_q)$; such an $F$ exists by Proposition \ref{prop:KR_M_structure}(\ref{prop:KR_M_struct_phi}).
        \item Let $K_q^\circ \coloneqq K_q^\paramodular\cap \spin(V_{Dq}^\circ)(\Q_q)$, which
is hyperspecial by Proposition \ref{prop:KR_M_structure}(\ref{prop:KR_M_struct_circ}).\footnote{Since the subgroup $K_q^\circ \subset \spin(V_D^\circ)(\Q_q)$ from (\ref{subsubsec:g0_setup_2ERL}) is also hyperspecial, we hope this will not produce any confusion.} We also set $K^\circ = K^{q\circ}K_q^\circ \subset \spin(V_{Dq}^\circ)(\A_f)$. 
    \end{enumerate}
\end{notation}
\begin{rmk}
    By Proposition \ref{prop:KR_M_structure}(\ref{prop:KR_M_struct_circ}), under Notation \ref{notation:irr_comps_labels_2ERL}  
    the irreducible components of $\mathcal M^\circ$ are labeled by $\mathcal M(g)$ and $\mathcal M(gF)$ for $g\in \spin(V_{Dq}^\circ)(\Q_q)/K_q^\circ$. 
    \end{rmk}
\subsection{Tate classes on $X^\circ_{\overline \F_q}$}

\begin{definition}
\leavevmode
        \begin{enumerate}
        \item We label the irreducible components of $X_{\overline \F_q}^{\ss \circ}$ as $B_\delta^\circ(g)$ for $$(g, \delta) \in \Sh_{K^{\circ}}(V_{Dq}^\circ) \times \set{0,1},$$ by defining 
             $B_\delta^\circ(g) $ to be the image of $(g^q, \mathcal M(g_q F^\delta))$ under the uniformization of Proposition \ref{RZ uniformization for 2nd ERL}.
            \item We define the incidence map  \begin{equation}
            \begin{split}
    \inc^{\circ\ast}: H^2(X^\circ_{\overline\F_q}, O(1)) \to \bigoplus_{(g, \delta) \in \Sh_{K^\circ}(V_{Dq}^\circ)\times \set{0,1}} H^2(B_\delta^\circ(g), O(1)) \\\cong O[\Sh_{K^\circ}(V_{Dq}^\circ)]^{\oplus 2}
    \end{split}
\end{equation}
and dually
\begin{equation}
\begin{split}
    \inc^\circ_\ast: O[\Sh_{K^\circ}(V_{Dq}^\circ)]^{\oplus 2}\cong \bigoplus_{(g, \delta) \in \Sh_{K^\circ}(V_{Dq}^\circ)\times \set{0,1}} H^2(B_\delta^\circ(g), O(1)) \\\to H^2_c(X^\circ_{\overline\F_q}, O(1)).\end{split}
\end{equation}
        \end{enumerate}
\end{definition}

\subsubsection{}
Let $S^\circ$ be the set of primes $\l$ such that $K_\l^\circ$ is not hyperspecial.
For the rest of this section, we shall apply the results and notations of Appendix \ref{sec:appendix_spin4}, with the added superscript $\circ$  for consistency. In particular, we obtain a compact open subgroup $\widetilde K^\circ = \prod\widetilde K^\circ_\l\subset C^+(V_D^\circ)^\times(\A_f)$, where $\widetilde K^\circ_\l$ is hyperspecial for $\l\not\in S^\circ$.

\begin{prop}
    The integral model $ {X}^{\circ}$ for $\Sh_{K^\circ}(V_D^\circ)$ extends to a smooth canonical model $\widetilde { X}^\circ$ over $\Spec \Z_{(q)}$ for $\widetilde \Sh_{\widetilde K^\circ}(V_D^\circ)$, with an
 open and closed embedding
\begin{equation}
     X^\circ\hookrightarrow\widetilde{{X}}^\circ 
\end{equation}
extending the map on generic fibers. Moreover, the universal abelian variety on $ X$ extends naturally to $\widetilde{ X}$.
\end{prop}
\begin{proof}
    Under the condition (\ref{third condition on widetilde K}) from Appendix \ref{sec:appendix_spin4}, this follows from the construction of the embedding on generic fibers in \cite[Proposition 2.10, Remark 2.11]{liu2020supersingular}. 
\end{proof}

\begin{lemma}\label{base change for tilde 2nd ERL}
    For each $i$, there are  canonical 
    $G_{\Q_q}$-equivariant isomorphisms 
\begin{equation}
    \tag{$\operatorname{BC}_{\widetilde X^\circ}$}
    \begin{split}
        H^i_{\et}(\widetilde\Sh_{\widetilde K^\circ}(V_D^\circ)_{\overline\Q}, O) &\cong H^i_{\et} (\widetilde X^\circ_{\overline\F_q}, O) \\
         H^i_{\et,c}(\widetilde\Sh_{\widetilde K^\circ}(V_D^\circ))_{\overline\Q}, O) &\cong H^i_{\et,c} (\widetilde X^\circ_{\overline\F_q}, O).
         \end{split}
    \end{equation}
    compatible with those of Lemma \ref{base change for 2nd ERL}.
\end{lemma}
\begin{proof}
See \cite[Lemma 4.4]{liu2020supersingular}.

\end{proof}

\subsubsection{}
The Clifford algebra $C^+(V_{Dq}^\circ)$ is a totally definite quaternion algebra over $F$, whose local invariants coincide with those of $C^+(V_D^\circ)$ at all finite places; let $Q^\times$ be the $\Q$-algebraic group of  units of $C^+(V_{Dq}^\circ)$.  Then by \cite[Theorem 3.13]{liu2020supersingular}, we have a uniformization
\begin{equation}
    \widetilde X^{\ss\circ}_{\overline \F_q}\simeq Q^\times(\Q)\backslash Q^\times(\A_f^q)\times \mathcal M^\circ / \widetilde K^{q\circ}
\end{equation}
compatible with Proposition \ref{RZ uniformization for 2nd ERL} for $? = \circ$. Here  $\widetilde K^{q\circ}$ is viewed as a compact open subgroup of $Q^\times(\A_f^q)$ using the isomorphism$$V_{Dq}^\circ\otimes_\Q \A_f^q \simeq V_D^\circ\otimes_\Q \A_f^q$$ that follows from the choice of $q$-adic uniformization datum, cf. Notation \ref{notation:V_Dq_etc}(\ref{item:V_Dq_etc}).

We may therefore extend $\inc^{\circ\ast}$ to a map $\widetilde\inc^{\circ\ast}$ fitting into the following commutative diagram:

\begin{equation}\label{inc pullback commutative diagram}
    \begin{tikzcd}
        H^2_\et(X^\circ_{\overline\F_q}, O(1)) \arrow[r,  "\inc^{\circ\ast}"]\arrow[d, hook]
        & O\left[\Sh_{K^\circ}(V_{Dq}^\circ) \right]^{\oplus 2} \arrow[d, hook]\\
         H^2_\et(\widetilde X^\circ_{\overline\F_q}, O(1)) \arrow[r, "\widetilde \inc^{\circ\ast}"] & O\left[\widetilde \Sh_{\widetilde K^{\circ} }(V_{Dq}^\circ)\right]^{\oplus 2}.
    \end{tikzcd}
\end{equation}
On the bottom right, the level subgroup is $\widetilde K^\circ = \widetilde K^{q\circ} \widetilde K_q^\circ $, where $\widetilde K_q^\circ$ is the unique hyperspecial subgroup of $Q^\times(\Q_q)$ containing $K_q^\circ$, and by definition
$$\widetilde \Sh_{\widetilde K^\circ}(V_{Dq}^\circ) = Q^\times(\Q) \backslash Q^\times(\A_f)/ \widetilde K^\circ.$$
 Similarly, we have a  map $\widetilde\inc^\circ_\ast$ fitting into a commutative diagram
\begin{equation}\label{inc push forward commutative diagram}
    \begin{tikzcd}
      O\left[\Sh_{K^\circ}(V_{Dq}^\circ) \right]^{\oplus 2} \arrow[d, hook] \arrow[r,  "\inc^{\circ}_{\ast}"]&
      H^2_{\et,c}(X^\circ_{\overline\F_q}, O(1)) \arrow[d, hook]
         \\
       O\left[\widetilde \Sh_{\widetilde K^{\circ}}(V_{Dq}^\circ)\right]^{\oplus 2}   \arrow[r, "\widetilde \inc^{\circ}_{\ast}"] &H^2_{\et,c}(\widetilde X^\circ_{\overline\F_q}, O(1)).
    \end{tikzcd}
\end{equation}
\subsubsection{Hecke actions}
Recall the local and global  Hecke algebras $\T_\l^\circ$, $\widetilde \T_\l^\circ$, $\T^{\circ S^\circ}$, and $\widetilde \T^{\circ S^\circ} $ from (\ref{subsubsec:Hecke_gspin4_appendix}).
  We define actions of the local Hecke algebras $\T_q^\circ \cong \widetilde \T_q^\circ$ on $H^i_{\et} (X^\circ_{\overline\F_q}, O)$, $H^i_{\et,c}(X^\circ_{\overline\F_q}, O)$, $H^i_{\et}(\widetilde X^\circ_{\overline\F_q}, O)$, and $H^i_{\et,c}(\widetilde X^\circ_{\overline\F_q}, O)$ via the isomorphisms of Lemmas \ref{base change for 2nd ERL} and \ref{base change for tilde 2nd ERL}.
\begin{lemma}\label{Hecke equivariance for inc 2nd ERL}
\leavevmode
\begin{enumerate}
    \item The maps $\inc^{\circ \ast}$ and $\widetilde\inc^{\circ\ast}$ are equivariant for $\T^{\circ S^\circ}$ and $\widetilde{\T}^{\circ S^\circ}$, respectively.\label{Hecke equivariance for inc 2nd ERL part one}
    \item The maps $\inc^\circ_\ast$ and $\widetilde \inc^\circ_\ast$ are equivariant for $\T^{\circ S^\circ}$ and $\widetilde{\T}^{\circ S^\circ}$, respectively, after extending scalars to $\overline\Q_p$.\label{Hecke equivariance for inc 2nd ERL part two}
\end{enumerate}
\end{lemma}

\begin{proof}
    It is clear that $\inc^{\circ\ast}$, $\widetilde\inc^{\circ\ast}$, $\inc^\circ_\ast$, and $\widetilde \inc^\circ_\ast$ are equivariant for all prime-to-$q$ Hecke operators, so it suffices to 
     consider the action of $\T_q^\circ \cong \widetilde \T_q^\circ$.
       Also, by the commutative diagrams (\ref{inc pullback commutative diagram}) and (\ref{inc push forward commutative diagram}), it suffices to consider $\widetilde \inc ^{\circ\ast}$ and $\widetilde \inc^\circ_\ast$.
        The final reduction is that we may prove both statements of the lemma after extending scalars to $\overline\Q_p$, since the target of $\widetilde \inc^{\circ\ast}$ is $O$-torsion-free.

Applying Lemma \ref{base change for tilde 2nd ERL} and the \'etale comparison theorem, for (\ref{Hecke equivariance for inc 2nd ERL part one}) it therefore suffices to show that any map
$$H^2(\widetilde \Sh_{\widetilde K^\circ} (V_D)(\C), \C) \to \C\left[\widetilde \Sh_{\widetilde K^{\circ}} (V_{Dq}^\circ)\right]$$ 
which is equivariant for $\widetilde{\T}^{\circ S^\circ\cup\set{q}}$  is also equivariant for $\widetilde {\T}^\circ_q$; but this is clear from the Jacquet-Langlands correspondence and strong multiplicity one for $\GL_2$. 
The proof of (\ref{Hecke equivariance for inc 2nd ERL part two}) is the same. 
\end{proof}
\begin{definition}\label{def:T_q_circ}
   Let $T_q^\circ \in \widetilde{\T}_q^\circ \cong  \T_q^\circ$ be the double coset operator represented by $\begin{pmatrix}
             q & 0 \\ 0 & 1 
        \end{pmatrix}$ in any basis such that $\widetilde K_q^\circ = \GL_2(\Z_{q^2})$. 
\end{definition}
\begin{lemma}\label{lemma inc pullback and Hecke action on X circ}
    \begin{enumerate}
        \item\label{matrix composite of inc maps on X circ} The composite maps
        $$O\left[\Sh_{K^\circ}(V_{Dq}^\circ)\right] ^{\oplus 2} \xrightarrow{\inc^\circ_\ast} H^2_{\et, c} (X^\circ_{\overline \F_q}, O(1)) \to H^2_{\et} (X^\circ_{\overline \F_q}, O(1))\xrightarrow{\inc^{\circ\ast}}O\left[\Sh_{K^\circ}(V_{Dq}^\circ)\right] ^{\oplus 2}$$
        and         $$O\left[\widetilde\Sh_{\widetilde K^\circ}(V_{Dq}^\circ)\right] ^{\oplus 2} \xrightarrow{\widetilde\inc^\circ_\ast} H^2_{\et, c} (\widetilde X^\circ_{\overline \F_q}, O(1)) \to H^2_{\et} (\widetilde X^\circ_{\overline \F_q}, O(1))\xrightarrow{\widetilde \inc^{\circ\ast}}O\left[\widetilde\Sh_{\widetilde K^\circ}(V_{Dq}^\circ)\right] ^{\oplus 2}$$
        are both given by the matrix $$\begin{pmatrix} -2q & T_q^\circ\langle q\rangle^{-1} \\ T_q^\circ & -2q\end{pmatrix}.$$
        \item \label{injectivity of inc pullback}The restricted map
        $$\inc^{\circ\ast} : (T_q^{\circ 2} \langle q \rangle^{-1} - 4q^2) H^2_{\et, !}(X^\circ_{\overline \F_q}, O(1))^{\Frob_q^2 = \langle q \rangle} \to O\left[\Sh_{K^\circ}(V_{Dq}^\circ)\right] ^{\oplus 2}$$
        has $\varpi$-power-torsion kernel. 
    \end{enumerate}
\end{lemma}
\begin{proof}
    For (\ref{matrix composite of inc maps on X circ}), see \cite[Proposition 2.21(4)]{liu2016hirzebruch}; the off-diagonal entries follow from the intersection combinatorics of $\mathcal M^\circ$ described in \cite[\S4]{kudla1999hirzebruch}. For (\ref{injectivity of inc pullback}), it suffices to prove the analogous statement for $\widetilde \inc^{\circ\ast}$ due to the commutative diagram (\ref{inc pullback commutative diagram}). Let us fix an isomorphism $\iota: \overline \Q_p \isomorphism \C$. Extending 
scalars to $\overline \Q_p$ and applying Lemma \ref{Hecke equivariance for inc 2nd ERL} and Proposition \ref{prop decompose hecke action on quaternionic shimura variety}, it suffices to show that
$$\widetilde\inc^{\circ\ast}: H^2_{\et, !} (\widetilde X^\circ_{\overline \F_q} , \overline\Q_p(1))^{\Frob_q^2 = \langle q \rangle}[\iota^{-1}\tau_f] \to \overline \Q_p \left[\widetilde\Sh_{\widetilde K^\circ}(V_{Dq}^\circ)\right][\iota^{-1}\tau_f]^{\oplus 2}$$ is injective for all discrete automorphic representations $\tau$ of $C^+(V_D^\circ)^\times(\A)$ such that \begin{equation}\tag{$\operatorname{Hecke}_q$}
    T_q^{\circ 2} \langle q \rangle^{-1} - 4q^2\neq 0 \text{   on   }\tau_f^{\widetilde K^\circ}.
\end{equation}

Now note that $\Frob_q^2 = \langle q \rangle$ on the image of $$\widetilde\inc^\circ_\ast: O\left[\widetilde\Sh_{\widetilde K^\circ}(V_{Dq}^\circ)\right]^{\oplus 2}\to 
 H^2_{\et, c} (\widetilde X^\circ_{\overline{\F}_q}, O(1));$$
 this follows from Proposition \ref{prop:KR_M_structure}(\ref{prop:KR_M_struct_phi}). In particular, we have a well-defined composite map 
     $$\overline \Q_p\left[\widetilde\Sh_{\widetilde K^\circ}(V_{Dq}^\circ)\right]^{\oplus 2} [\iota^{-1}\tau_f] \xrightarrow{\widetilde\inc^\circ_\ast} H^2_{\et, !} (\widetilde X^\circ_{\overline{\F}_q}, \overline \Q_p(1))^{\Frob_q^2 = \langle q \rangle} [\iota^{-1}\tau_f]
    \xrightarrow{\widetilde \inc^{\circ\ast}} \overline \Q_p\left[\widetilde\Sh_{\widetilde K^\circ}(V_{Dq}^\circ)\right]^{\oplus 2} [\iota^{-1}\tau_f],$$
    given by the matrix in part (\ref{matrix composite of inc maps on X circ}), which is invertible by the assumption ($\operatorname{Hecke}_q$).

    It therefore suffices to show that 
    $$\dim _{\overline \Q_p} H^2_{\et, !} (\widetilde X^\circ_{\overline{\F}_q}, \overline \Q_p(1))^{\Frob_q^2 = \langle q \rangle} [\iota^{-1}\tau_f]\leq 2\dim \Q_p\left[\widetilde\Sh_{\widetilde K^\circ}(V_{Dq}^\circ)\right] [\iota^{-1}\tau_f]$$
    for all $\tau$ satisfying ($\operatorname{Hecke}_q$). This dimension count follows from \cite[Proposition 2.25]{tian2019tate} when $\tau$ is cuspidal; when $\tau$ is not cuspidal, it is clear from Proposition \ref{prop decompose hecke action on quaternionic shimura variety}.
    \end{proof}
\subsubsection{}
Consider the algebraic cycle class
\begin{equation}
    [X^\diamondsuit_{\overline\F_q}] \in H^2_{\et,c}(X^\circ_{\overline\F_q}, O(1)).
\end{equation}
which makes sense by Lemma \ref{lem:proper}.
\begin{lemma}\label{inc pullback of special cycle is special for 2nd ERL on X circ}
    There exists a special cycle $Z\in \SC_{K^\circ}^1(V_{Dq}^\circ)$ such that $$\inc^{\circ\ast} ([X^\diamondsuit_{\overline\F_q}]) = (Z, Z) \in O \left[\Sh_{K^\circ}(V_{Dq}^\circ) \right].$$
\end{lemma}
\begin{proof}
Since $[X^\diamondsuit_{\overline\F_q}]$ is Frobenius-invariant, it suffices to consider the first coordinate of $\inc^{\circ\ast} ([X^\diamondsuit_{\overline\F_q}])$, which we write as $\inc^{\circ\ast} _1([X^\diamondsuit_{\overline\F_q}])$. Since $\mathcal M^\diamondsuit$ is zero-dimensional,  all the intersections of $X^\diamondsuit_{\overline\F_q}$ with supersingular curves on $X^\circ_{\overline\F_q}$ are proper. Hence $\inc^{\circ\ast} _1([X^\diamondsuit_{\overline\F_q}])$ is computed by
\begin{equation}\label{inexplicit equation for inc pullback special cycle 2nd ERL}
    \sum_{ \substack{[(g^q,g_q)]\in \spin(V_{Dq}^\diamondsuit)(\Q) \backslash \spin(V_{Dq}^\diamondsuit) (\A_f^q)\times \spin(V_{Dq}^\circ) (\Q_q) / K^{\diamondsuit q} \times K_q^\circ}} m(g_q) [(g^q, g_q)],
\end{equation}
where, for $g_q \in \spin(V_{Dq}^\circ)(\Q_q)$, $m(g_q)$ is the degree of the divisor $\mathcal N^\diamondsuit \cap \mathcal M(g_q)$ on $\mathcal M(g_q) \cong \mathbb P^1_{\overline\F_q}$. In particular, $m(g_q)$ depends only on the $\spin(V_{Dq}^\diamondsuit)(\Q_q)$-orbit of $g_q$. Moreover, $g_q \mapsto m(g_q)$ is a compactly supported function on $\spin(V_{Dq}^\diamondsuit)(\Q_q) \backslash \spin(V_{Dq}^\circ)(\Q_q)$ since each point of $\mathcal M$ lies on only finitely many irreducible components.\footnote{One can see this using that $\mathcal M$ is locally of finite type over $\overline \F_q$, or more concretely from \cite[\S4]{kudla1999hirzebruch}.}  Hence (\ref{inexplicit equation for inc pullback special cycle 2nd ERL}) is a finite linear combination of special cycles
$$Z(h_q^{(i)}, V_T\cap V_{Dq}^\circ, V_{Dq}^\circ)_{K^\circ}\in \Z\left[\Sh_{K^\circ}(V_{Dq}^\circ)\right]$$ for some elements $$h_q^{(i)}\in \spin(V_{Dq}^\diamondsuit)(\Q_q) \backslash \spin(V_{Dq}^\circ)(\Q_q)/ K_q^\circ,$$ which completes the proof.
\end{proof}
\begin{rmk}
    It would not be difficult to make Lemma \ref{inc pullback of special cycle is special for 2nd ERL on X circ} more explicit, but it is unnecessary for the main results.
\end{rmk}
\begin{thm}\label{2nd geometric ERL version on X circ}
 For any $h\in  \T^{\circ S^\circ}_O$, there exists a cycle $Z_{Dq}^\circ\in \SC^1_{K^\circ}(V_{Dq}^\circ,O)$ such that 
    $$\inc^\circ_\ast [(Z_{Dq}^\circ, Z_{Dq}^\circ)] - (T_q^{\circ 2} - 4q^2\langle q \rangle)^2 h [X^\diamondsuit_{\overline \F_q}] \in H^2_{\et, c} (X^\circ_{\overline\F_q}, O(1))$$
    has $\varpi$-power-torsion image in $H^2_{\et,!}(X^\circ_{\overline\F_q}, O(1))$.
\end{thm}
\begin{proof}
    Let $Z\in \SC^1_{K^\circ} (V_{Dq}^\circ)$ be the special cycle in Lemma \ref{inc pullback of special cycle is special for 2nd ERL on X circ}, so that $$\inc^{\circ\ast} ([X^\diamondsuit_{\overline \F_q}] ) = (Z, Z),$$ and define $$Z_{Dq}^\circ\coloneqq (T_q^{\circ 2} - 4q^2\langle q \rangle) \cdot (T_q^\circ + 2q) \cdot h \cdot Z\in \SC^1_{K^\circ} (V_{Dq}^\circ,O).$$
    By Lemma \ref{lemma inc pullback and Hecke action on X circ}(\ref{injectivity of inc pullback}), it suffices to show that
    \begin{equation}\label{goal equation for pf of 2nd ERL on X circ}
        \inc^{\circ\ast}\inc^\circ_\ast [(Z_{Dq}^\circ, Z_{Dq}^\circ)] = \inc^{\circ\ast} (T_q^{\circ2} - 4q^2\langle q \rangle)^2 h [X^\diamondsuit_{\overline \F_q}].
    \end{equation}
    By Lemmas \ref{lemma inc pullback and Hecke action on X circ}(\ref{matrix composite of inc maps on X circ}) and  \ref{Hecke equivariance for inc 2nd ERL}, the left-hand side of (\ref{goal equation for pf of 2nd ERL on X circ}) is\begin{equation*}(T_q^{\circ 2} - 4q^2\langle q\rangle) \cdot (T_q^\circ + 2q)\cdot h \cdot \begin{pmatrix}
        -2q & T_q^\circ \langle q \rangle^{-1} \\ T_q^\circ & -2q 
    \end{pmatrix} \cdot \begin{pmatrix}
        Z \\ Z
    \end{pmatrix} \\ = (T_q^{\circ2}- 4q^2\langle q \rangle)^2 \cdot h \cdot \begin{pmatrix} Z \\ Z \end{pmatrix}
    \end{equation*}
because the Hecke operator $\langle q\rangle$ acts trivially on special cycles. This coincides with the right-hand side of (\ref{goal equation for pf of 2nd ERL on X circ}) by Lemma \ref{Hecke equivariance for inc 2nd ERL} and the definition of $Z$, so the proof is complete.
\end{proof}
\subsection{Pushing forward from $ X^\circ$ to $ X$}
\subsubsection{}\label{subsubsec:m_transition_2ERL}
Now let $S$ be a finite set of primes of $\Q$ such that $K_\l$ is hyperspecial for $\l\not\in S$, and let $\m \subset \T^{S\cup \set{q}}_{O}$ be a generic, non-Eisenstein maximal ideal. 
\begin{notation}
We denote by $j$ the closed embedding $ X^\circ\hookrightarrow X$ of (\ref{integral models factoring special cycle for 2nd ERL}).
\end{notation}
\begin{notation}
    Let $$\partial_{\AJ, \m, \F_q}: \CH^2(X_{\F_q}) \to H^1(\F_q, H^3_\et (X_{\overline \F_q}, O(2))_\m)$$ denote the local Abel-Jacobi map constructed analogously to (\ref{subsubsec:where_AJ}).
\end{notation}

\begin{lemma}\label{pushforward AJ factors through for 2nd ERL}
The composite map 
$$\CH^1(X^\circ_{\F_q}) \xrightarrow{j_\ast}\CH^2(X_{\F_q})\xrightarrow{\partial_{\AJ, \m, \F_q}} H^1(\F_q, H^3_\et(X_{\overline\F_q}, O(2))_\m)$$
factors through the specialization map  $$H^2(X_{\F_q}^\circ, O(1)) \to H^2(X_{\overline\F_q}^\circ, O(1))^{\Frob_q = 1}.$$

\end{lemma}
\begin{proof}
For any variety $Y$ defined over $\F_q$, we have the Hochschild-Serre spectral sequence:
\begin{equation}\label{AJ over Fp in 2nd ERL}
    E^{i,j}_2 = H^i(\F_q, H^j_\et(Y_{\overline\F_q}, O(n)) \implies H^{i+j}_\et(Y, O(n)),\;\; \forall n \in \Z.
\end{equation}
Since $\F_q$ has cohomological dimension one, it follows immediately that the map $$H^i(Y, O(n)) \to H^i(Y_{\overline\F_q}, O(n))^{\Frob_q = 1}$$ is surjective. Let $H^i(Y, O(n))^0\subset H^i(Y, O(n))$ be the kernel of this map, and let $$\partial: H^i(Y, O(n))^0 \to H^1(\F_q, 
H^{i-1}(Y_{\overline\F_q}, O(n))$$ be the edge map from (\ref{AJ over Fp in 2nd ERL}). 

It then suffices to show that the map 
$$H^2(X_{\F_q}^\circ, O(1)) \xrightarrow{j_\ast} H^4(X_{\F_q}, O(2)) \xrightarrow{\partial} H^1(\F_q, H^3(X_{\overline\F_q}, O(2))_\m)$$
factors through the surjection $H^2(X_{\F_q}^\circ, O(1)) \twoheadrightarrow H^2(X_{\overline\F_q}^\circ, O(1))^{\Frob_q = 1},$ i.e. is trivial on $H^2(X_{\F_q}^\circ, O(1)) ^0$. 

Consider the commutative diagram arising from the functoriality of (\ref{AJ over Fp in 2nd ERL}):
\begin{equation*}
    \begin{tikzcd}
        H^2_\et(X^\circ_{\F_q}, O(1))^0 \arrow[r, "\partial"]\arrow[d, "j_\ast"] & H^1(\F_q, H^1_\et(X^\circ_{\overline\F_q}, O(1))) \arrow[d, "j_\ast"] \\ H^4_\et(X_{\F_q}, O(2))^0 \arrow[r, "\partial"] \arrow[d, "\loc_\m"] & H^1(\F_q, H^3_\et(X_{\overline \F_q}, O(2)))\arrow[d, "\loc_\m"]\\
        H^4_\et(X_{\F_q}, O(2))^0 _\m\arrow[r, "\partial"]  & H^1(\F_q, H^3_\et(X_{\overline \F_q}, O(2))_\m).
    \end{tikzcd}
\end{equation*}

Now note that the composite map $$H^1_\et (X^\circ_{\overline \F_q}, O(1)) \xrightarrow{j_\ast} H^3_\et (X_{\overline\F_q}, O(1)) \xrightarrow{\loc_\m } H^3_\et (X_{\overline\F_q}, O(1))_\m $$ is identically zero: indeed, the source is $\varpi$-power-torsion by \cite[Chapter IX, Corollary 7.15(iii)]{margulis1991discrete}, and the target is $\varpi$-torsion-free by Theorem \ref{thm:generic}(\ref{part:thm_generic_two}). In particular, the composite from the top left to the bottom right of the commutative diagram vanishes, which proves the lemma.
\end{proof}
\subsubsection{}
We now construct a  map $$\partial_{\ss,\m}: O[\Sh_{K^qK_q^\paramodular} (V_{Dq})]_\m \to H^1_\unr(\Q_q, H^3_\et(\Sh_K(V_D)_{\overline\Q}, O(2))_\m)$$
in several steps.
\begin{construction}\label{constr:partial_ss_2ERL}
Assume $\langle q \rangle = 1$ in $\T^S_{K, V_D, O,\m}$. 
    \begin{enumerate}
        \item For $g\in\Sh_{K^qK_q^\paramodular}(V_{Dq})$, let $B(g)\subset X^{\ss}_{\overline \F_q}$  be the image of $(g^q, \mathcal M(g_q))$ under the uniformization of Proposition \ref{RZ uniformization for 2nd ERL}.
\item \label{constr:partial_ss_2ERL_cycle_class}Define an action of $\Frob_q$ on $\Sh_{K^qK_q^\paramodular}(V_{Dq})$ by $g\mapsto gF$. Then using Proposition \ref{prop:KR_M_structure}(\ref{prop:KR_M_struct_phi}), we obtain a map\begin{equation*}
    O[\Sh_{K^qK_q^\paramodular} (V_{Dq})]^{\Frob_q = 1} \to \CH^2(X_{\F_q}, O).
\end{equation*}
\item\label{constr:partial_ss_2ERL_iso} By our assumption  $\langle q \rangle = 1$, the natural map gives an isomorphism
\begin{equation*}
     O[\Sh_{K^qK_q^\paramodular} (V_{Dq})]_\m^{\Frob_q^2 = 1}\isomorphism  O[\Sh_{K^qK_q^\paramodular} (V_{Dq})]_\m.
\end{equation*}
\item
Finally, we define  the map 
\begin{equation*}
    \partial_{ss, \m}: O[\Sh_{K^qK_q^\paramodular} (V_{Dq})]_\m \to H^1(\F_q, H^3_\et(\Sh_K(V_D)_{\overline\Q}, O(2))_\m)
\end{equation*}

to be the composite 
\begin{equation*}\begin{split}O[\Sh_{K^qK_q^\paramodular} (V_{Dq})]_\m \xrightarrow{\text{(\ref{constr:partial_ss_2ERL_iso})}^{-1}}  O[\Sh_{K^qK_q^\paramodular} (V_{Dq})]_\m^{\Frob_q^2 = 1}\xrightarrow{[g]\mapsto [g] + [gF]} O[\Sh_{K^qK_q^\paramodular} (V_{Dq})]_\m^{\Frob_q = 1}  \\ \xrightarrow{\text{(\ref{constr:partial_ss_2ERL_cycle_class})}} \CH^2(X_{\F_q},O)_\m \xrightarrow{\partial_{\AJ, \F_q, \m}} H^1(\F_q, H^3_\et(X_{\overline\F_q}, O(2))_\m) \xrightarrow[\operatorname{BC}_X] {\sim}H^1_\unr(\Q_q, H^3_\et(X_{\overline\Q}, O(2))_\m).
\end{split}
\end{equation*}
\end{enumerate}
\end{construction}
\begin{thm}\label{thm:2ERL_geometric}
    Suppose $\m\subset \T^{S\cup \set{q}}_{O}$ satisfies:
    \begin{enumerate}
        \item $\m$ is non-Eisenstein and generic.
        \item\label{item:tidy_substitute_2ERL} $\langle q \rangle = 1$ in $\T^S_{K, V_D, O, \m}$. 
    \end{enumerate}
  
    Let $h \in \T^{\circ S^\circ}_O$ be any element, and let $C\geq 0$ be an integer such that $\varpi^C$ annihilates the $\varpi$-power-torsion in $H^2(X^\circ_{\overline\Q}, O(1))$. Then there exists a special cycle $$Z_{Dq} \in\SC^2_K(V_{Dq},O)$$ such that
    $$\varpi^C\cdot \Res_{\Q_q}  \partial _{\AJ, \m} \left(j_\ast\left( (T_q^{\circ2} - 4q^2\langle q \rangle)^2 \cdot h \cdot [X^\diamondsuit]\right)\right) = \varpi^C \cdot \partial_{\ss,\m} (Z_{Dq})\in H^1_\unr(\Q_q, H^3_\et(X_{\overline\Q}, O(2))_\m).$$
\end{thm}
\begin{proof}
Let 
\begin{equation*}\label{equation with overline j for 2nd ERL}\overline j: \Sh_{K^\circ}(V_{Dq}^\circ) \xrightarrow{\cdot g_0^q} \Sh_{K^qK_q^\paramodular} (V_{Dq})\end{equation*}
 be the natural map, induced by the embedding $V_{Dq}^\circ\hookrightarrow V_{Dq}$ from Notation \ref{notation:V_Dq_etc}(\ref{item:V_Dq_etc}).        
In particular, we have
\begin{equation}\label{compatibility of pushforwards ss locus for 2nd ERL}
    j(B_\delta^\circ(g)) = B(\overline j(g) F^\delta),\;\; (g, \delta) \in \Sh_{K^\circ}(V_{Dq}^\circ) \times \set{0,1}.
\end{equation}
    Then we take $$Z_{Dq}^\circ = \sum_{g\in \Sh_{K^\circ}(V_{Dq}^\circ)}n(g)[g]\in \SC^1_{K^\circ} (V_{Dq}^\circ,O)$$
   to be the special cycle provided by Theorem \ref{2nd geometric ERL version on X circ}, and let $$Z_{Dq} = \overline j\left( Z_{Dq}^\circ\right) \in \SC_K^2 (V_{Dq},O)$$ be the pushforward. 
If we let $$\cl_\ss^\circ(Z_{Dq}^\circ) = \sum_{g\in \Sh_{K^\circ}(V_{Dq}^\circ)} n(g) \left(\left[B_0^\circ(g)\right] + \left[B_1^\circ(g)\right]\right) \in \CH^1(X^\circ_{\F_q}, O),$$
then by (\ref{compatibility of pushforwards ss locus for 2nd ERL}),
$$\partial_{\ss,\m} (Z_{Dq}) = \operatorname{BC}_X\left(\partial_{\AJ,\m, \F_q} j_\ast \cl_\ss^\circ(Z_{Dq}^\circ) \right)\in H^1_\unr(\Q_q, H^3(X_{\overline\Q}, O(2))_\m).$$
    In light of Lemma \ref{pushforward AJ factors through for 2nd ERL} and the local-global compatibility of the Abel-Jacobi map, it then suffices to show that $$\cl_\ss^\circ(Z_{Dq}^\circ) - (T_q^{\circ2} - 4q^2\langle q \rangle)^2 \cdot h \cdot [X^\diamondsuit_{\F_q}] \in \CH^1(X^\circ_{\F_q})$$ has $\varpi$-power-torsion image in $H^2_\et(X^\circ_{\overline\F_q}, O(1))$. But this is precisely the content of Theorem \ref{2nd geometric ERL version on X circ}. 
    \end{proof}
    \subsubsection{}
    Now let $\pi$, $S$, and $E_0$ be as in  Notation \ref{notation:pi_basic}, and let $\p$ be a prime of $E_0$ satisfying Assumption \ref{assumptions_on_p}. We  set $\m \coloneqq \m_{\pi,\p}$ as usual.
    \begin{cor} \label{cor:usable_geometric_2ERL}
    Suppose $q\nmid D$ is $n$-admissible and $K$ is an $S$-tidy level structure for $\spin(V_D)$. Then 
for any $\alpha\in \Test_K(V_D, \pi, O/\varpi^n)$ and any $h\in \T^{\circ S^\circ}_O$, 
        $$\ord_\varpi \lambda_n^D(q) \geq \ord_ \varpi \loc_q \circ \alpha_\ast\circ \partial _{\AJ, \m} \left(j_\ast\left( (T_q^{\circ2} - 4q^2\langle q \rangle)^2 \cdot h \cdot [X^\diamondsuit]\right)\right) - C,$$
        where $C$ is a constant independent of $n$ and $q$. 
    \end{cor}
    \begin{proof}
    Recall the conditions on $\m$ from Theorem \ref{thm:2ERL_geometric} are satisfied by  Lemmas \ref{lem:m_is_nonEisgeneric} and \ref{lem:S_tidy_upshot}.
        Let $\alpha'$ be the composite map $$O\left[\Sh_{K^qK_q^\paramodular}(V_{Dq})\right]_\m \xrightarrow{\partial_{\ss,\m}}H^1_\unr(\Q_q, H^3_\et(\Sh_K(V_D)_{\overline \Q}, O(2))_\m ) \xrightarrow{\alpha_\ast} H^1_\unr(\Q_q, T_{\pi,n}) \simeq O/\varpi^n,$$ where the last isomorphism is from Proposition \ref{prop:local_adm_free}.
        Then $\lambda^D_n(q)$ contains $\alpha' (Z_{Dq})$, where $Z_{Dq}\in \SC_K^2(V_{Dq},O)$ is the cycle from Theorem \ref{thm:2ERL_geometric}, and the corollary follows. 
    \end{proof}
\section{Second explicit reciprocity law}\label{sec:2ERL_contd}
The goal of this section is to prove Theorem \ref{ultimate theorem for second TRL} below.
\subsection{Setup and notation}
Let $\pi$, $S$, and $E_0$ be as in Notation \ref{notation:pi_basic}, and  fix the following data:
\begin{itemize}[label = $\circ$]
    \item A  prime $\p$ of $E_0$ satisfying Assumption \ref{assumptions_on_p}, with residue characteristic $p$.
    \item A squarefree product $D\geq 1$ of primes in $S$, with $\sigma(D)$ even.
    \item An $S$-tidy level structure $K$ for $\spin(V_D)$. 
\item A Schwartz function $\phi = \bigotimes'_\l \phi_\l\in \mathcal S(V_D^2\otimes \A_f, O)^K$.
\end{itemize}
\subsubsection{}For all $\l\not\in S$, 
let $L_\l \subset V_D\otimes \Q_\l$ be the unique self-dual lattice stabilized by $K_\l$.
\subsection{ Modifying the test function}
In several steps, we now construct a new Schwartz function $\phi'$ which coincides with $\phi$ at cofinitely many primes.
\begin{construction}\label{constr:2ERL_modifying}
\leavevmode
    \begin{enumerate}
        \item Label the finitely many imaginary quadratic fields contained in $\rho_\pi = \rho_{\pi,\p}$ as $E_1,\ldots, E_s$ for some $s\geq 0$. 
        \item 
For each $ 1\leq i\leq s $, fix an odd prime $ \l_i\not\in S\union\set {p,\l_1,\ldots, \l_{i -1}} $ inert in $ E_i $ such that $\rho_{\pi} (\Frobenius_{\l_i}) $ has distinct eigenvalues (possible \emph{a fortiori} by Assumption \ref{assumptions_on_p}(\ref{assume:pi_generic})). 
\item Let
\begin{equation*}
X_{\l_i} =\set {(x, y)\in V_D ^ 2\otimes\Q_{\l_i}\,:\, x\cdot x\in (\Z_{\l_i} ^\times) ^ 2,\, x\cdot y\in \Z_{\l_i},\, y\cdot y\in\Z_{\l_i} ^\times - (\Z_{\l_i} ^\times) ^ 2}.
\end{equation*}
\item 
 We define a test function
\begin{equation*}
\phi'_{\l_i} = (\l_i +1)\cdot\blackboardone_{\set {(x, y)\in L_{\l_i} ^ 2\intersection X_{\l_i}}} +\blackboardone_{\set {(x, y)\in L_{\l_i}\times (\l_i ^{-1} L - L_{\l_i})\intersection X_{\l_i}}}\in\mathcal S (V_D ^ 2\otimes\Q_{\l_i},\Z) ^ {K_{\l_i}}.
\end{equation*}
\item 
Now fix a prime $ \l_0\not\in S\union\set {p,\l_1,\ldots, \l_s} $ such that $\rho_{\pi} (\Frobenius_{\l_0}) $ has distinct eigenvalues.
\item Let 
\begin{equation*}
\phi'_{\l_0}\in\mathcal S (V_D ^ 2\otimes\Q_{\l_0},\Z) ^ {K_{\l_0}}
\end{equation*}
be the indicator function of the set
\begin{equation*}
\set {(x, y)\in L_{\l_0} ^ 2\,:\, x\cdot x\in\Z ^\times_{\l_0} - (\Z_{\l_0} ^\times) ^ 2,\, x\cdot y\in \l_0\Z_{\l_0},\, y\cdot  y\in \l_0\Z_{\l_0} ^\times}.
\end{equation*}
\item 
Finally, we define
\begin{equation*}
\phi'\coloneqq\bigotimes_{\l\neq \l_i}\phi_\l\otimes\bigotimes_{i = 0} ^ s\phi'_{\l_i}\in\mathcal S (V_D ^ 2\otimes\A_f, O) ^ K.
\end{equation*}    \end{enumerate}
\end{construction}


\begin{lemma}\label{Lemma non-vanishing with modified Schwartz function or second euro}
Suppose $\kappa^D(1,\phi; K) \neq 0$.
Then
there exists  a test function $\alpha\in \Test_K(V_D, \pi, O)$ and a matrix 
 $T\in\symmetric_2 (\Q)_{>0} $ satisfying \ref{ass:T1} and \ref{ass:T2} of (\ref{subsubsec:where_T_assumptions}) such that $$\alpha_\ast\circ\partial_{\AJ,\m}\left (Z (T,\phi')_K\right)\neq 0. $$
Moreover, for all $1 \leq i \leq s$,  $ \l_i $ is split in the quadratic field $ F\coloneqq\Q (\sqrt {T_{1 1}}) $  and $T$ lies in $\GL_2(\Z_{(\l_i)})$.
\end{lemma}
\begin{proof}
Repeatedly applying Corollary \ref{cor:changing_kappa}, we conclude 
$\kappa^D(1, \phi'; K) \neq 0$; i.e.,
there exists $ T\in\symmetric_2 (\Q)_{\geq 0} $ and $\alpha\in \Test_K(V_D, \pi, O)$ such that $$\alpha_\ast\circ\partial_{\AJ,\m} \left(Z (T,\phi')_K\right)\neq 0.$$

By definition of $ Z (T,\phi')_K $, it follows that, for all $ 0\leq i\leq s $, there exists $ x, y\in\Omega_{T, V_D} (\Q_{\l_i}) $ such that $\phi'_{\l_i} (x, y)\neq 0 $. Now the choice of $\phi'_{\l_0} $ implies that $ T $ is positive definite, $ T_{1 1} $ is not a rational square, and the quadratic space defined by $ T $ has nontrivial Hasse invariant at $ \l_0\nomad D $, and this proves the first claim of the lemma.  Similarly, for each $1 \leq i \leq s$,  $\l_i$ is split in $F$ and $T$ lies in $\GL_2(\Z_{(\l_i)})$ by the choice of $\phi'_{\l_i}$.
\end{proof}
\subsubsection{}
We will now assume $\kappa^D(1, \phi; K) \neq 0$, and fix $\alpha$ and $ T $ satisfying the conclusion of Lemma \ref{Lemma non-vanishing with modified Schwartz function or second euro}.
We choose a base point $(e_1^T, e_2^T) \in \Omega_{T, V_D} (\Q)$ such that $e_1^T, e_2^T\in L_{\l_i}$ for all $1 \leq i \leq s$ (possible by the last point of Lemma \ref{Lemma non-vanishing with modified Schwartz function or second euro}), and we adopt the setup of \S\ref{sec:2ERL_geometry} for this choice of $T$, $(e_1^T,e_2^T)$, and $K$; the choice of $g_0$ is postponed to Proposition \ref{prop:intermediate_2ERL}.
\begin{definition}\label{def:T_l_circ}
    For each $\l$ split in $F$ and each hyperspecial subgroup $K_\l^\circ\subset\spin(V_D^\circ)(\Q_\l)$, we define a Hecke operator $T_\l^\circ\in \T_\l^\circ$ as follows. Choose an isomorphism $$\spin(V_D^\circ)(\Q_\l)\simeq G_\l\coloneqq \set{(g_1,g_2)\in \GL_2(\Q_\l)\times \GL_2(\Q_\l)\,:\,\det g_1=\det g_2}$$
    mapping $K_\l^\circ$ to $G_\l\cap\left(\GL_2(\Z_\l)\times \GL_2(\Z_\l)\right)$. Then $T_\l^\circ$ is the double coset operator represented by $$\left(\begin{pmatrix} \l&0\\ 0&1 \end{pmatrix}, \begin{pmatrix} \l&0\\ 0&1 \end{pmatrix}\right)\in G_\l.$$
\end{definition}
\begin{prop} \label{prop:intermediate_2ERL} There exists an element $ g_0\in\spin (V_D) (\A_f ^ {\set{\l_1,\ldots, \l_s}}) $ such that, adopting the notation of (\ref{subsubsec:g0_setup_2ERL}) for this choice of $ g_0 $:
\begin{enumerate}
\item \label{prop:intermediate_2ERL_1} The compact open subgroups $ K ^\circ_{\l_i}\subset\spin (V_D ^\circ) (\Q_{\l_i}) $ are hyperspecial for $ 1\leq i\leq s $.
\item \label{prop:intermediate_2ERL_2} We have $$\alpha_\ast\circ\partial_{\AJ,\m} j_\ast\left (\product_{i = 1} ^ s T ^\circ_{\l_i}\cdot\left [\shimmer_{K ^\diamond} (V_D ^\diamond)\right]\right)\neq 0, $$
where $\left [\shimmer_{K ^\diamond} (V_D ^\diamond)\right]\in\CH ^ 1 (\shimmer_{K ^\circ} (V ^\circ_D)) $ is the class of the cycle from (\ref{factoring special cycle for 2nd ERL}).
\end{enumerate}
\end{prop}
\begin{proof}
By definition of $ Z (T,\phi')_K $ and the assumption $\alpha \circ \partial_{\AJ,\m}(Z (T,\phi')_K)\neq 0$, there exists $ g_0\in\spin (V_D) (\A_f ^ {\set{\l_1,\ldots, \l_s}}) $ such that
\begin{equation}\label{equation for picking apart cycle in second TRL rock}
Z\coloneqq\sum_{g =\prod g_{\l_i}\in\prod_{i = 1} ^ s\spin (V_D ^\diamond)(\Q_{\l_i})\backslash\spin (V_D) (\Q_{\l_i})/K_{\l_i}}\left (\prod_{i = 1} ^ s\phi'_{\l_i} (g ^{-1}_{\l_i}\cdot (e_1 ^ T, e_2 ^ T))\right)\cdot Z (g_0 g, V_T, V_D)_K
\end{equation}
satisfies $\alpha_\ast\partial_{\AJ,\m} (Z)\neq 0 $. We will check the claimed properties for this choice of $g_0$. First, if $\phi_{\l_i}' (g ^{-1}_{\l_i} (e_1 ^ T, e_2 ^ T))\neq 0 $, then $ g_{\l_i} ^{-1} e_1 ^ T \cdot g ^{-1}_{\l_i} e_1 ^ T = e_1 ^ T\cdot e_1 ^ T\in\Z ^\times_{(\l_i)} $, so in particular
$$ L_{\l_i}\intersection\Q_{\l_i} g ^{-1}_{\l_i} e_1 ^ T =\Z_{\l_i} g ^{-1}_{\l_i} e_1 ^ T. $$
On the other hand, we chose $(e_1^T, e_2^T)$ so that $L_{\l_i} \cap e_1^T = \Z_{\l_i} e_1^T$.
Hence by Proposition \ref{prop:zanarella}, we have $ g_{\l_i}\in\spin (V_D ^\circ) (\Q_{\l_i})\cdot K_{\l_i} $. In particular, we can rewrite (\ref{equation for picking apart cycle in second TRL rock}) as
\begin{equation}\label{eq:picking_apart_2ERL}
\begin{split}
Z & =\sum_{g =\prod g_{\l_i}\in\prod_{i = 1} ^ s\spin (V_D ^\diamond) (\Q_{\l_i})\backslash\spin (V_D ^\circ) (\Q_{\l_i})/K_{\l_i}\intersection\spin (V_D ^\circ) (\Q_{\l_i})}\product_{i = 1} ^ s\phi'_{\l_i} (g ^{-1}_{\l_i} (e_1 ^ T, e_2 ^ T))\cdot Z ( g_0 g,V_T,V_D)_K\\
& =\sum_{g =\prod g_{\l_i}\in\prod_{i = 1} ^ s\spin (V_D ^\diamond) (\Q_{\l_i})\backslash\spin (V_D ^\circ) (\Q_{\l_i})/K_{\l_i}\intersection\spin (V_D ^\circ) (\Q_{\l_i})}\product_{i = 1} ^ s\phi'_{\l_i} (g ^{-1}_{\l_i} (e_1 ^ T, e_2 ^ T))\cdot j_\ast Z (g,V_T\intersection V_D ^\circ, V_D ^\circ)_{K^\circ}.
\end{split}
\end{equation}
Again because $ e_1 ^ T\cdot e_1 ^ T\in\Z ^\times_{(\l_i)} $, $ L ^\circ\coloneqq L_{\l_i}\intersection (V ^\circ_D\otimes\Q_{\l_i}) $ is self-dual, and so its stabilizer $ K ^\circ_{\l_i}\coloneqq K_{\l_i}\intersection\spin (V_D ^\circ) (\Q_{\l_i}) $ in $\spin (V_D ^\circ) (\Q_{\l_i}) $ is hyperspecial, which proves (\ref{prop:intermediate_2ERL_1}). To prove  (\ref{prop:intermediate_2ERL_2}), by definition we have
\begin{equation}\label{equation values of Schwartz function in second TRL}
\begin{split}
\phi'_{\l_i} (e_1 ^ T, g ^{-1}_{\l_i} e_2 ^ T) &=\begin{cases} \l_i +1, & g ^{-1}_{\l_i} e_2 ^ T\in L_{\l_i},\\1, & g ^{-1}_{\l_i} e_2 ^ T\in \l ^{-1}_i L_{\l_i} - L_{\l_i},\\0, & g ^{-1}_{\l_i} e_2 ^ T\not\in \l_i ^{-1} L_{\l_i}\end{cases} \\&=\begin{cases} \l_i +1, & g ^{-1}_{\l_i} e_2\in L ^\circ,\\1, & g ^{-1}_{\l_i} e_2\in \l ^{-1}_i L ^\circ - L ^\circ,\\0, & g ^{-1}_{\l_i} e_2\not\in \l ^{-1}_i L ^\circ,\end{cases}
\end{split}
\end{equation}
where  $$ e_2\coloneqq e_2 ^ T -\frac { e_2 ^ T\cdot  e_1 ^ T} { e_1 ^ T\cdot  e_1 ^ T} e_1^T\in V_D^\circ. $$
Note that $ e_2\cdot e_2\in\Z_{(\l_i)} ^\times $ because, for $ (e_1 ^ T, g ^{-1}_{\l_i} e_2 ^ T) $ to lie in the support of $\phi'_{\l_i} $, we automatically have $ e_1 ^ T\cdot e_1 ^ T\in\Z ^\times_{(\l_i)} - (\Z ^\times_{\l_i}) ^ 2 $, $ e_1 ^ T\cdot e_2 ^ T\in\Z_{(\l_i)} $, and $ e_2 ^ T\cdot  e_2 ^ T\in\Z ^\times_{(\l_i)}\intersection (\Z ^\times_{\l_i}) ^ 2 $.

Let $\phi ^\circ_{\l_i}\in\mathcal S (V_D ^\circ\otimes\Q_{\l_i},\Z) ^ {K ^\circ_{\l_i}} $ be the indicator function of $ L ^\circ $; by our specification $e_2^T\in L_{\l_i}$ and Proposition \ref{prop:zanarella}, we see that $\phi^{\circ}_{\l_i} (g_{\l_i}^{-1} e_2)$ is the indicator function of $\spin(V_D^\diamond) (\Q_{\l_i})\cdot K_{\l_i}^\circ$. 
The proposition therefore follows from (\ref{eq:picking_apart_2ERL}), (\ref{equation values of Schwartz function in second TRL}), and the following:
\begin{claim} When restricted to $\set {y\in V ^\circ_D\otimes\Q_{\l_i}\,:\, y\cdot  y\in\Z ^\times_{\l_i}}, $ we have $ T ^\circ_{\l_i}\cdot\phi ^\circ_{\l_i} = (\l_i +1)\phi ^\circ_{\l_i} +\blackboardone_{\l_i ^{-1} L ^\circ - L ^\circ}. $
\end{claim}
To prove the claim, observe first that $$ T ^\circ_{\l_i}\cdot\phi ^\circ_{\l_i} =\sum_{L'\sim L^\circ}\blackboardone_{L'}, $$ where $ L' $ runs over self-dual lattices in $ V ^\circ_D\otimes\Q_{\l_i} $ such that
\begin{equation}\label{equation lattice chains in V circ}
\l_i L ^\circ\subset_1 \l_i L' + \l_i L ^\circ\subset_2L'\intersection L ^\circ\subset_1L ^\circ.
\end{equation}
Note that such a chain (\ref{equation lattice chains in V circ}) uniquely determines $ L' $ because $ L' $ and $ L ^\circ $ correspond to the two isotropic lines in the split two-dimensional $\F_{\l_i} $-space $ L ^\circ + L'/L ^\circ\intersection L' $. We can also write the chain (\ref{equation lattice chains in V circ}) as
\begin{equation}\label{eq:lattice_chains_V_circ2}
    \l_i L ^\circ\subset_1 L_1 \subset_2 L_1^\vee \subset_1L ^\circ,
\end{equation}
i.e. such chains correspond bijectively to isotropic lines $L_1/{\l_i}L^\circ$ in the $\F_{\l_i}$-quadratic space $L^\circ/\l_i L^\circ$.

Now for any $ y\in V_D ^\circ\otimes\Q_{\l_i} $ with $ y\cdot y\in\Z_{\l_i} ^\times $, we wish to calculate $$ T ^\circ_{\l_i}\cdot\phi ^\circ_{\l_i} (y) =\#\set {L'\sim L ^\circ\,:\, y\in L'}. $$ If $ y\in L ^\circ $, then the choices of $ L' $ correspond to isotropic lines orthogonal to $ y $ in the four-dimensional $\F_{\l_i} $-space $ L ^\circ/\l_i L ^\circ $; there are $ \l_i +1 $ such lines because $ y ^\perpendicular\subset L ^\circ/\l_i L ^\circ $ is a non-degenerate three-dimensional $\F_{\l_i} $-space. Hence $ T ^\circ_{\l_i}\cdot\phi ^\circ_{\l_i} (y) = \l_i +1 $
for $ y\in L ^\circ $.

On the other hand, since any $ L'\sim L ^\circ $ is contained in $ \l ^{-1}_i L ^\circ $, if $ T ^\circ_{\l_i}\cdot\phi_{\l_i} ^\circ (y)\neq 0 $ we must have $ y\in \l_i ^{-1} L ^\circ $. So suppose $ y\in \l ^{-1}_i L ^\circ - L ^\circ $. Now, lattices $ L'\sim L ^\circ $ with $ y\in L' $ (equivalently $\l_i y \in \l_i L'$) are in bijection with chains (\ref{eq:lattice_chains_V_circ2}) satisfying  $\l_i y \in L_1$; since $\l_i y \not \in \l_iL^\circ$, the unique such chain is given by $L_1 = \l_i L^\circ + \l_iy$, and this makes sense because $\l_iy$ is isotropic in $L^\circ/\l_i L^\circ$.
We therefore conclude that $T^\circ_{\l_i} \cdot \phi^\circ_{\l_i}(y) = 1$ for $y\in \l_i^{-1} L^\circ - L^\circ$, which completes the proof of the claim.
\end{proof}
\begin{theorem}\label{ultimate theorem for second TRL}
Suppose $\pi$ is non-endoscopic, $\kappa^D (1)\neq 0 $, and admissible primes exist for $\rho_\pi$. Then there exists a constant $ C\geq 0 $ such that, for all $ N $, there exist infinitely many admissible primes $ q $ with $n(q)\geq N$, $$\order_\varpi\local_q\kappa^D_{n(q)} (1)\geq n(q) - C,$$ and $$\order_\varpi\lambda^D_{n(q)} (q)\geq n(q)- C. $$
\end{theorem}
\begin{proof}
Since $\kappa^D(1) \neq 0$, we can choose the data $\phi$ and $K$ in the beginning of this section so that $\kappa^D(1,\phi; K)\neq 0$ (where $K$ can be made $S$-tidy by Lemma \ref{lem:tidy_suffices}). 
Then fix $ g_0\in\spin (V_D) (\A_f ^ {\set{\l_1,\ldots, \l_s}}) $ as in Proposition \ref{prop:intermediate_2ERL}, so that
\begin{equation} \label{eq:in_ultimate_2ERL}
\alpha_\ast\circ\partial_{\AJ,\m} j_\ast\left (\product_{i = 1} ^ s T ^\circ_{\l_i}\left [\shimmer_{K ^\diamond} (V_D ^\diamond)\right]\right)\neq 0;
\end{equation}
equivalently, by Lemma \ref{lem:H1_tors_free_new}, (\ref{eq:in_ultimate_2ERL}) is non-torsion. Let $S^\circ$ be the set of primes $\l$ such that $K^\circ_\l $ is not hyperspecial.  
 Without loss of generality, we may replace $ O $ by a finite extension such that all the eigenvalues of $\T^{\circ S^\circ}_O$ acting on $H^2_{\et}(\Sh_{K^\circ} (V_D^\circ)_{\overline \Q}, \overline \Q_p)$ lie in $O$; let $\boldsymbol h_1, \ldots, \boldsymbol {h}_m: \T^{\circ S^\circ}_O \to O$ be the finitely many characters that appear in the action on $H^2_{\et, !}(\Sh_{K^\circ} (V_D^\circ)_{\overline \Q}, \overline\Q_p(1))^{G_\Q}$, cf. Corollary \ref{cor:gspin4_cohomology_decomposes}. Then we may choose elements $h_j \in \T^{\circ S^\circ}_O$ such that $\T^{\circ S^\circ}_{O}$ acts on $h_j H^2_{\et, !}(\Sh_{K^\circ} (V_D^\circ)_{\overline \Q}, \overline\Q_p(1))^{G_\Q}$ through $\boldsymbol h_j$, and moreover $$\sum_{j = 1}^m h_j = \varpi^{C_0} \text{ on } H^2_{\et, !} (\Sh_{K^\circ} (V_D^\circ)_{\overline \Q}, \overline\Q_p (1))^{G_\Q}$$
for a constant $C_0 \geq 0$. 
In particular, by Lemma \ref{pushforward AJ factors through for 2nd ERL} and Proposition \ref{prop:intermediate_2ERL}, 
$$c\coloneqq \alpha\circ \partial_{\AJ,\m} j_\ast\left(h_j \cdot\left( \prod_{i = 1}^s T_{\l_i}^\circ \right) \left[\Sh_{K^\diamond}(V_D^\diamond)\right]\right) $$
is non-torsion for some $1 \leq j \leq m$. Write $\boldsymbol h \coloneqq \boldsymbol h_j$ and $h\coloneqq h_j$. Let us fix an isomorphism $\iota: \overline\Q_p\isomorphism \C$. By Corollary \ref{cor:gspin4_cohomology_decomposes}, we are in one of the following two cases:
\begin{case}
    The character $\iota \circ \boldsymbol h$ is given by the action of $\T^{\circ S^\circ}$ on a global newform in an automorphic representation $\sigma$ of $B_F(\A_F)^\times$ unramified outside primes above $S^\circ$. Moreover  $\JL(\sigma) = \BC_{F/\Q} (\sigma_0) \otimes \chi$ where $\sigma_0$ is a cuspidal automorphic representation of $\GL_2(\A_\Q)$ of weight 2 and $\chi$ is a finite order character of $F^\times \backslash \A_F^\times$. 
\end{case}
\begin{case}
     The character $\iota\circ \boldsymbol h$ is given by the action of $\T^{\circ S^\circ}$ on the automorphic representation $\chi_0 \circ \nu$, where $\nu: \spin(V_D^\circ) \to \mathbb G_m$ is the restriction of the norm character in (\ref{subsubsec:appendix_spin4_start}) and $\chi_0$ is a quadratic character of $\Q^\times \backslash \A^\times / \nu(K^\circ)$. 
\end{case}
\begin{claimnumbered}
    In Case 1, $\sigma_0$ does not have CM by any imaginary quadratic field $K_0 \subset F(\rho_{\pi})$. 
\end{claimnumbered}
\begin{proof}[Proof of claim 1]
   If so, then $K_0 \neq F$, and $K_0$ is contained in the compositum of $F$ and one of the quadratic fields $E_i\subset \Q(\rho_{\pi})$ from Construction \ref{constr:2ERL_modifying}; in particular, since $\l_i$ is inert in $K_i$ but split in $F$ by Lemma \ref{Lemma non-vanishing with modified Schwartz function or second euro}, $\l_i$ is inert in $K_0$. 
     It follows that $\tr\rho_{\sigma,\iota}(\Frob_{\l_i}) = 0$. On the other hand, it is not difficult to compute using the Satake transform and Theorem \ref{thm:rho_GL2_LLC}(\ref{part:rho_GL2_LLC_1}) that
$$T_{\l_i}^\circ = \tr\rho_{\sigma,\iota}  (\Frob_{\l_i}) \cdot \tr \rho_{\sigma,\iota} (\tau\Frob_{\l_i}\tau^{-1})$$ on a local newform in $\sigma_{\l_i}$, where  $\tau\in G_\Q$ projects to the nontrivial element of $\Gal(F/\Q)$. Hence $h\cdot T_{\l_i}^\circ = \boldsymbol h(T_{\l_i}^\circ) \cdot h = 0$, which contradicts the nontriviality of $c$. 
\end{proof}
Now fix an element $g\in \Gal(F(\rho_{\pi})/\Q)$ such that:
\begin{enumerate}
    \item $g$ is admissible for $\rho_{\pi}$ and has nontrivial image in $\Gal(F/\Q)$. 
    \item $c(g)$ has nonzero component in the 1-eigenspace for $g$ for any cocycle representative of $c$.
    \item In Case 1 above, $\rho_{\sigma_0,\iota} (g^2)$ has distinct eigenvalues.
\end{enumerate}
This is possible by Proposition \ref{prop:choosing_g_2ERL} because $c$ is nontorsion. Now fix constants $C_1,C_2, C_3\geq 0$ such that:
\begin{enumerate}
    \item The component of $c(g)$ in the 1-eigenspace for $g$ is nonzero modulo $\varpi^{C_1}$.
    \item In Case 1 above, $\tr(\rho_{\sigma,\iota}(g^2)) - 4\det \rho_{\sigma,\iota}(g^2) \neq 0 \pmod{ \varpi^{C_2}}$.
    In Case 2, $C_2 = 0$. 
    \item $\varpi^{C_3}$ annihilates the $O$-torsion in $H^2_{\et, !} (\Sh_{K^\circ}(V_D^\circ)_{\overline{\Q}}, O)$. 
\end{enumerate}
Now fix $N \geq C_1 + C_2 + C_3$ and let $c_N \in H^1(\Q, T_{\pi, N})$ be the reduction of $c$ modulo $\varpi^N$. All but finitely many primes with Frobenius conjugate to $g$ in $\Gal(F(T_{\pi,N}, c_N)/\Q)$ are $N$-admissible and satisfy Assumption 
\ref{assume:q_2ERL}; let $q$ be one such and
abbreviate $n\coloneqq n(q)\geq N$. If $\alpha_n \in \Test_K(V_D, \pi, \O/\varpi^n)$ is the image of $\alpha$, then:
\begin{equation}\label{inequality for kappa proof of 2nd ERL}
    \ord_{\varpi} \loc_q \alpha_{n,\ast} \circ \partial_{\AJ, \m} j_\ast \left(h\cdot \left(\prod_{i = 1}^s T_{\l_i}^\circ\right) \left[\Sh_{K^\diamond}(V_D^\diamond)\right]\right) \geq n - C_1.
\end{equation}
We have:
\begin{claimnumbered}\label{claim:2_for_2ERL}
    $\ord_{\varpi} \boldsymbol h(T_q^{\circ 2} - 4q^2\langle q \rangle) \leq C_2.$
\end{claimnumbered}
\begin{proof}[Proof of Claim 2]
    In Case 1, we have 
    $$\boldsymbol h(T_q^{\circ 2} - 4q^2\langle q \rangle) = \tr \rho_{\sigma,\iota} (\Frob_q^2) - 4 \det \rho_{\sigma,\iota}(\Frob_q^2) \equiv \tr \rho_{\sigma,\iota}(g^2) - 4\det \rho_{\sigma,\iota}(g^2) \pmod {\varpi^n},$$
    so this follows from the choice of $C_2$. In Case 2, we have $\boldsymbol h(T_q^\circ) = (q^2 + 1) \chi_0(\langle q\rangle )$ and $\boldsymbol h(\langle q \rangle) = \chi_0(\langle q\rangle)$, so $ \boldsymbol h(T_q^{\circ 2} - 4q^2\langle q \rangle) = (q^2 - 1)^2\chi_0(\langle q\rangle)$, which is a $\varpi$-adic unit because $q$ is admissible.
\end{proof}
Now, by Corollary \ref{cor:usable_geometric_2ERL} we have
\begin{align*}
    \ord_\varpi \lambda_n^D(q) &\geq \ord _\varpi \loc_q \alpha_{n,\ast}\circ \partial_{\AJ, \m} j_\ast\left((T_q^{\circ 2} - 4q^2\langle q \rangle) h \left(\prod_{i =1}^s T_{\l_i}^\circ\right) \left[\Sh_{K^\diamond}(V_D^\diamond)\right]\right) - C_3 \\
    &= \ord _\varpi \loc_q \alpha_{n,\ast}\circ \partial_{\AJ, \m} j_\ast\left(\boldsymbol h(T_q^{\circ 2} - 4q^2\langle q \rangle) h \left(\prod_{i =1}^s T_{\l_i}^\circ\right) \left[\Sh_{K^\diamond}(V_D^\diamond)\right]\right) - C_3\\
    &\geq \ord _\varpi \loc_q \alpha_{n,\ast}\circ \partial_{\AJ, \m} j_\ast\left( h \left(\prod_{i =1}^s T_{\l_i}^\circ\right) \left[\Sh_{K^\diamond}(V_D^\diamond)\right]\right) - C_3 - \ord_\varpi \boldsymbol h(T_q^{\circ 2} - 4q^2\langle q \rangle).
\end{align*}
By Claim \ref{claim:2_for_2ERL} and (\ref{inequality for kappa proof of 2nd ERL}), we conclude that
\begin{equation}\label{inequality for lambda proof of 2nd ERL}
    \ord_\varpi \lambda_n^D(q) \geq n - C_1 - C_2 - C_3.
\end{equation}
Since  $j_\ast \left(h\cdot \left(\prod_{i = 1}^s T_{\l_i}^\circ\right) \left[\Sh_{K^\diamond}(V_D^\diamond)\right]\right)\in \CH^2(\Sh_K(V_D), O)$ lies in $\SC^2_K(V_D,O)$ by Remark \ref{rmk:SC_definitions}, (\ref{inequality for kappa proof of 2nd ERL}) and (\ref{inequality for lambda proof of 2nd ERL}) together show the theorem.
\end{proof}
In the endoscopic case, we similarly obtain: 
\begin{theorem}\label{thm:2ERL_endo_ultimate}
Suppose $\pi$ is endoscopic associated to a pair $(\pi_1,\pi_2)$, such that $\pi_1$ and $\pi_2$ are not both CM for the same imaginary quadratic field.  For $j = 1$ or 2, assume
 $\kappa^D (1)^{(j)}\neq 0 $ and admissible primes exist which are BD-admissible for $\rho_{\pi_j}$. Then there exists a constant $ C\geq 0 $ such that, for all $ N $, there exist infinitely many admissible primes $ q $ which are $N$-BD-admissible for $\rho_{\pi_j}$, such that $$\order_\varpi\local_q\kappa^D_{n(q)} (1)^{(j)}\geq n(q) - C$$ and $$\order_\varpi\lambda^D_{n(q)} (q)\geq n(q)- C. $$
\end{theorem}
\begin{proof}
    By using Corollary \ref{cor:endoscopic_changing_kappa} in place of Corollary \ref{cor:changing_kappa}, we can refine Lemma \ref{Lemma non-vanishing with modified Schwartz function or second euro} to obtain the same conclusion where $\alpha: H^3_\et(\Sh_K(V_D)_{\overline \Q}, O(2))_\m \to T_{\pi}$ is required to factor through $T_{\pi_j}$. From here, the rest of the proof follows that of Theorem \ref{ultimate theorem for second TRL}, substituting Proposition \ref{prop:choosing_g_2ERL_endo} for Proposition \ref{prop:choosing_g_2ERL}. 
\end{proof}
 \section{Main result: rank one case}\label{sec:rk1}
 \subsection{Setup and notation}
Let $\pi$, $S$, and $E_0$ be as in Notation \ref{notation:pi_basic}, and fix a prime $\p$ of $E_0$ satisfying Assumption \ref{assumptions_on_p}, with residue characteristic $p$.
\begin{definition}
    An automorphic representation $\pi$ of $\GL_4(\A)$ is of \emph{general type} if it is neither an automorphic induction nor a  symmetric cube lift. 
\end{definition}
 For this section, we shall  assume:
\begin{hyp}[$\bigstar$]
    If $\pi$ is non-endoscopic, then $\BC(\pi)$ is   
    either of general type, or a symmetric cube lift of a non-CM automorphic representation $\pi_0$ of $\GL_2(\A)$, or induced from a non-CM automorphic representation $\pi_0$ of $\GL_2(\A_K)$  with $K/\Q$ real quadratic;  in the latter two cases $E_0$ is also a strong coefficient field of $\pi_0$. 

\end{hyp}

\subsection{Choosing Chebotarev primes}
\subsubsection{}
Let $L_\pi\subset \End_O(T_\pi)$ be the underlying $O$-module of $\ad^0 \rho_{\pi}= \ad^0 \rho_{\pi,\p}$, which is free of rank 10 over $O$. In general, $L_\pi\otimes \Q_p$ is not absolutely irreducible, even if $V_\pi$ is. In the endoscopic case, we also let $L_{\pi_i}$ be the underlying $O$-module of $\ad^0 \rho_{\pi_i}$, for $i = 1,2$. 
\begin{prop}\label{prop:cases_L_pi}
Suppose $\pi$ is  not endoscopic. We have the following cases for $L_\pi$:
    \begin{enumerate}
        \item\label{prop:cases_L_pi_induced} If $\BC(\pi)$ is an automorphic induction of $\pi_0$ as in Hypothesis $(\bigstar)$, then 
        $$L_\pi = \Ind_{G_K}^{G_\Q} \ad^0 T_{\pi_0} \oplus\left( \otimes-\Ind_{G_K}^{G_\Q} T_{\pi_0}\right)(-1).$$ Both direct summands are absolutely irreducible after inverting $p$.
        \item \label{prop:cases_L_pi_cube}If $\BC(\pi)$ is a symmetric cube lift of $\pi_0$ as in Hypothesis $(\bigstar)$, then $$L_\pi = \ad^0 T_{\pi_0} \oplus \Sym^6(T_{\pi_0})(-3),$$ with each summand absolutely irreducible after inverting $p$.

        \item \label{prop:cases_L_pi_gen}If $\BC(\pi)$ is of general type, then $L_\pi\otimes \Q_p$ is absolutely irreducible. 
    \end{enumerate}

\end{prop}
\begin{proof}
    In the first case, $T_{\pi_0}$ exists by Lemma \ref{lem:RQ_induction_relevant}, and   
    we have $$T_\pi = \Ind_{G_K}^{G_\Q} T_{\pi_0};$$ so the claimed decomposition follows from, e.g., the discussion in \cite[\S7.5.16]{boxer2021abelian}. For the irreducibility, if $\pi_0^\tw$ is the $\Gal(K/\Q)$ twist, then one checks using Hodge-Tate weights that $V_{\pi_0} \not\simeq V_{\pi_0^\tw}\otimes \chi$ for any character $\chi$ of $G_K$. 
    It follows that $\otimes-\Ind_{G_K}^{G_\Q} V_{\pi_0}$ is absolutely irreducible, since if it reduces it must have a one-dimensional constituent. Similarly, if $\Ind _{G_K}^{G_\Q} \ad^0 V_{\pi_0}$ is not absolutely irreducible, then $\ad^0 V_{\pi_0} \cong \ad^0 V_{\pi_0^\tw}$, so by Corollary \ref{cor:nekovar_pair_fail}, $V_{\pi_0}|_{G_L} = V_{\pi_0^\tw}|_{G_L}$ for a finite extension $L/K$, and this  is a contradiction by Lemma \ref{lem:nekovar_twist}.

    In the second case, $T_{\pi_0}$ exists by Lemma \ref{lem:sym_cube_relevant}, and by definition, $T_\pi = \Sym^3 T_{\pi_0}(-1)$. Then we have $$\End(T_\pi) = \Sym^6(T_{\pi_0})(-3) \oplus \Sym^4(T_{\pi_0})(-2)\oplus \Sym^2(T_{\pi_0}) (-1) \oplus O$$ as an $O[G_\Q]$-module, and by dimension counting,  $L_\pi$ consists of the first and third summands.
    The irreducibility of each summand of $L_\pi$ after inverting $p$ is clear because the
 Zariski closure (over $E$) of the projective image of $\rho_{\pi_0}$ is $\PGL_2(E)$ by Theorem \ref{thm:nekovar}.

    The third case follows from Lemma \ref{lem:large_image_general}, noting that $p > 3$ by Remark \ref{rmk:p>5}.
\end{proof}
\subsubsection{}
We write $L_{\pi,n} = L_\pi/\varpi^n L_\pi$ for all $n \geq 1$ and likewise for $L_{\pi_i,n}$ when $\pi$ is endoscopic. 
We will be applying the results and notations of Appendix \ref{sec:appendix_def_theory} to $\rho_\pi$ (and $\rho_{\pi_i}$ when $\pi$ is endoscopic), using Lemmas \ref{lem:assumption_appendix_ok} and \ref{lem:ass_appendix_ok_endo} to check the relevant assumptions.
The notion of admissible primes for Notation \ref{notation:admissible_appendix} is as explained in  (\ref{subsubsec:translate_appendix_1ERL}) and  (\ref{subsubsec:translate_appendix_1ERL_endo}). 
In particular, if $\pi$ is not endoscopic and $q$ is an $n$-admissible prime for $\rho_{\pi}$, let $H^1_\ord(\Q_q, L_{\pi,n})\subset H^1(\Q_q, L_\pi)$ be the subspace defined in Definition \ref{def:appendix_admissible_etc}.
\begin{lemma}\label{lem:heart_circ}
Suppose $\pi$ satisfies Hypothesis $(\bigstar)$ and is not endoscopic. There exists a $G_\Q$-stable decomposition $L_\pi = L_\pi^\heartsuit \oplus L_\pi^\circ$  such that $L_\pi^\heartsuit\otimes \Q_p$ and $L_\pi^\circ \otimes \Q_p$ are absolutely irreducible and distinct, and moreover, if $q$ is an $n$-admissible prime:
\begin{enumerate}
    \item \label{lem:heart_circ_sum}We have $H^1(\Q_q, L_{\pi,n}^\heartsuit) + H^1_\ord(\Q_q, L_{\pi,n}) = H^1(\Q_q, L_{\pi,n})$.
    \item\label{lem:heart_circ_unr} We have $H^1_\unr (\Q_q,L_{\pi,n}^\circ) = H^1 (\Q_q,L_{\pi,n}^\circ).$
    \item\label{lem:heart_circ_M0}
    If $T_{\pi,n}|_{G_{\Q_q}} = M_{0,n} \oplus M_{1,n}$ is the decomposition of Lemma \ref{lem:M0_adm_def}, then the composite
$$L_{\pi,n}^\heartsuit\hookrightarrow L_{\pi,n} \twoheadrightarrow \ad^0 M_{0,n}$$ is surjective.
\end{enumerate}
\end{lemma}
Here $L_{\pi,n}^\heartsuit = L_\pi^\heartsuit/\varpi^n L_\pi^\heartsuit$, and likewise for $L_{\pi,n}^\circ$.
\begin{proof}
In case (\ref{prop:cases_L_pi_gen}) of Proposition \ref{prop:cases_L_pi}, the lemma is clear, taking $L_\pi^\circ=0$.
Suppose we are in case (\ref{prop:cases_L_pi_induced}) of Proposition \ref{prop:cases_L_pi}.
Then if $q$ is admissible for $\rho_{\pi}$, we have $\tr\rho_{\pi}(\Frob_q) \not\equiv 0 \pmod \varpi$, which implies $q$ splits in $K$. 
    Let $\pi_0^\tw$ be the $\Gal(K/\Q)$-twist. Because $\det\rho_{\pi_0} = \chi_p^\cyc$, comparing the decomposition $T_{\pi,n} = M_{0,n} \oplus M_{1,n}$ with $T_\pi|_{G_K} = T_{\pi_0} \oplus T_{\pi_0^\tw}$ shows that, up to replacing $\pi_0$ with $\pi_0^\tw$, we have $M_{0,n} = T_{\pi_0,n}.$ 
    Now note that $H^1\left(\Q_q, \left(\otimes-\Ind_{G_K}^{G_\Q} T_{\pi_0,n}\right)(-1)\right) = 0$ because the $\Frob_q$ eigenvalues on $\left(\otimes-\Ind_{G_K}^{G_\Q} \overline T_{\pi_0}\right)(-1)$ are $q/\alpha$, $\alpha/q$, $\alpha$, and $1/\alpha$
for some $\alpha\neq q^2, q^{-1}, \pm 1, \pm q$. 
Hence the lemma holds in this case with $L_\pi^\heartsuit = \Ind_{G_K}^{G_\Q} \ad^0 T_{\pi_0}$ and $L_\pi^\circ =\left( \otimes-\Ind_{G_K}^{G_\Q}  T_{\pi_0}\right)(-1)$.

Now suppose we are in case (\ref{prop:cases_L_pi_cube}) of Proposition \ref{prop:cases_L_pi}, and take $L_\pi^\heartsuit= \Sym^6(T_{\pi_0})(-3)$, $L_\pi^\circ = \ad^0 T_{\pi_0}$.  
 If $\Frob_q$ acts on $\overline T_{\pi_0}$ with generalized eigenvalues $\set{\alpha,\beta}$, then the admissibility of $\Frob_q$ implies that (up to reordering) we have  $\beta^3 = q$ and $\alpha = \beta^2$. 
 Choose a basis for $T_\pi$ such that $$\Frob_q = \begin{pmatrix}
    \alpha^2/\beta & && \\ & \alpha& & \\ && \beta&\\ &&&\beta^2/\alpha
\end{pmatrix} = \begin{pmatrix}
    q & && \\ & \beta^2 & & \\ && \beta&\\ &&&1
\end{pmatrix}.$$
One sees immediately that the eigenvalues of $\Frob_q$ on $\ad^0 \overline T_{\pi_0}$ are $\set{1,\beta, \beta^{-1}}$, which are all distinct from $q^{-1}$, so $H^1_\unr(\Q_q, L_{\pi,n}^\circ(1)) = 0$, and this implies (\ref{lem:heart_circ_unr}) by local Poitou-Tate duality.

To prove (\ref{lem:heart_circ_sum}) and (\ref{lem:heart_circ_M0}), we use the following:
\begin{claim}
    The two decompositions into one-dimensional $O/\varpi$-vector spaces
    $$L_{\pi,1}^{\Frob_q = 1} = (L_{\pi,1}^\circ)^{\Frob_q = 1} \oplus (L_{\pi,1}^\heartsuit)^{\Frob_q = 1} = (\ad^0 M_{0,1})^{\Frob_q = 1} \oplus (\ad^0 M_{1,1})^{\Frob_q = 1}$$
    are both orthogonal with respect to the Killing form, and not the same. 
\end{claim}

Before proving the claim, we show it implies (\ref{lem:heart_circ_sum}) and (\ref{lem:heart_circ_M0}). 
Indeed, by (\ref{lem:heart_circ_unr}), we can rewrite (\ref{lem:heart_circ_sum}) 
as asserting the surjectivity of
$$
    H^1_\unr(\Q_q, \ad^0 M_{1,n}) = H^1_\unr(\Q_q, L_{\pi,n}) \cap H^1_\ord(\Q_q, L_{\pi,n}) \to H^1_\unr(\Q_q, L_{\pi,n}^\circ).
$$
Since $\Frob_q$ acts on $L_{\pi}$ with eigenvalues that are all either 1 or not congruent to 1 modulo $\varpi$, this is equivalent to the surjectivity of 
$$(\ad^0 M_{1,1})^{\Frob_q = 1} \to (L_{\pi,1}^\circ)^{\Frob_q = 1},$$
which follows from the claim. 

For (\ref{lem:heart_circ_M0}), we can take $n = 1$. 
 Since $\Frob_q$ acts with distinct eigenvalues $\set{1, q, q^{-1}}$ on $\ad^0 M_{0,1}$, and the eigenvalues $q$ and $q^{-1}$ do not appear in $L_{\pi,1}^\circ = \ad^0 \overline T_{\pi_0}$, it again suffices to consider the $\Frob_q = 1$ eigenspaces. In particular, it suffices to show that $(L_{\pi, 1}^\heartsuit)^{\Frob_q = 1}$ is not contained in the kernel of the map $L_{\pi,1}^{\Frob_q = 1} \to (\ad^0 M_{0,1})^{\Frob_q = 1}$, which again follows from the claim. 

Now we return to the proof of the claim. In our basis of $T_\pi$,  $(\ad^0M_{1,1})^{\Frob_q = 1}$  consists of the matrices  $\begin{pmatrix}
     0 &&& \\ &x&&\\&&-x&\\&&&0
\end{pmatrix}$, and $(\ad^0 M_{0,1})^{\Frob_q = 1}$ consists of the matrices 
$\begin{pmatrix}
     x &&& \\ &0&&\\&&0&\\&&&x
\end{pmatrix}$. This decomposition is plainly orthogonal for the Killing form. 

Meanwhile, $(L_{\pi,1}^\circ)^{\Frob_q = 1} = (\ad^0 \overline T_{\pi_0})^{\Frob_q = 1}$ consists 
 of  the matrices of the form $$\begin{pmatrix} 3y &&& \\ &y&&\\&&-y&\\&&&-3y\end{pmatrix}.$$  The decomposition $L_\pi^\circ \oplus L_\pi^\heartsuit$ is necessarily orthogonal for the Killing form because the form is Galois-invariant, and the claim follows.

\end{proof}
\begin{lemma}\label{lemma action contains scalar}
    Suppose $\pi$ is not endoscopic and satisfies Hypothesis $(\bigstar)$. Then:
    \begin{enumerate}
        \item \label{lemma action contains scalar part}The action of $G_\Q$ on $T_\pi$ contains a scalar of infinite order.
        \item \label{lemma action contains scalar semisimple}The projective image of $\rho_{\pi}$ is a compact $p$-adic Lie group with semisimple Lie algebra.
    \end{enumerate}
\end{lemma}
\begin{proof}
    If $\BC(\pi)$ is not an automorphic induction, then the lemma follows from Corollary \ref{cor:general_type_irr} combined with Lemma \ref{lemma with all the strongly irreducible statements}(\ref{lemma strongly irreducible part lie algebra is semisimple},\ref{lemma strongly irreducible part center contains open subgroup of determinant}). 
    If $\BC(\pi)$ is an automorphic induction of the kind in
Hypothesis ($\bigstar$), it follows from Corollary \ref{cor:nekovar_pair_fail}. 
\end{proof}
\begin{lemma}\label{lem:restr_for_adjt}
    Suppose $\pi$ satisfies Hypothesis $(\bigstar)$. Then:
    
    \begin{enumerate}
        \item \label{lem:restr_for_adjt_1}If $\pi$ is not endoscopic, there exists a constant $C\geq 0$ such that
    $$\varpi^CH^1(\Q(\rho_{\pi})/\Q, L_{\pi,n})=\varpi^CH^1(\Q(\rho_{\pi})/\Q, L_{\pi,n}(1))= 0$$ for all $n \geq 1$.
    \item\label{lem:restr_for_adjt_endo} If $\pi$ is endoscopic, then for $j = 1$ or 2 such that $\pi_j$ is non-CM, there exists a constant $C \geq 0$ such that $$\varpi^CH^1(\Q(\rho_{\pi})/\Q, L_{\pi_j,n})=\varpi^CH^1(\Q(\rho_{\pi})/\Q, L_{\pi_j,n}(1))= 0$$ for all $n \geq 1$.
    \end{enumerate}
\end{lemma}
\begin{proof}
We start with the non-endoscopic case. 
For $i =0$ or 1, consider the inflation-restriction exact sequence:
\begin{equation}\label{inflation restriction for restriction on adjoint}
\begin{split}
    0 \to H^1(\Q(\ad^0\rho_\pi)/\Q, L_{\pi,n}(i)^{G_{\Q(\ad^0\rho_\pi)}}) \to H^1(\Q(\rho_\pi)/\Q, L_{\pi,n}(i)) \to\\ \Hom_{G_\Q}(\Gal(\Q(\rho_\pi)/\Q(\ad^0\rho_\pi)), L_{\pi,n}(i)).
    \end{split}
\end{equation}
We claim the third term vanishes. Indeed, the $G_\Q$ action on $\Gal(\Q(\rho_\pi)/\Q(\ad^0\rho_\pi))$ by conjugation is trivial, so the third term
    is $$\Hom(\Gal(\Q(\rho_\pi)/\Q(\ad^0\rho_\pi)), L_{\pi,n}(i)^{G_\Q}) = 0$$ because the absolute irreducibility of $\overline T_{\pi}$ implies  $\overline L_{\pi}(i)^{G_\Q} = 0$. 
If $i = 1$, the first term is uniformly bounded in $n$ by Lemma \ref{lemma action contains scalar}(\ref{lemma action contains scalar part}), because any $g\in G_\Q$ that acts as a scalar $z$ on $T_\pi$ lies in $G_{\Q(\ad^0\rho_\pi)}$ and acts by $z^2$ on $L_{\pi}(1)$, so the lemma is proved in this case.

    For the case $i = 0$, we have $L_{\pi,n}^{G_{\Q(\ad^0\rho_\pi)}} = L_{\pi,n}$. By Lemma \ref{lemma action contains scalar}(\ref{lemma action contains scalar semisimple}) combined with  \cite[Lemma B.1]{fakhruddin2021relative}, the first term in (\ref{inflation restriction for restriction on adjoint}) is uniformly bounded in $n$, which proves the lemma when $\pi$ is non-endoscopic.

    Now consider the endoscopic case. Using Theorem \ref{thm:nekovar} to replace Lemma \ref{lemma action contains scalar}, the same argument as above shows $$\varpi^CH^1(\Q(\rho_{\pi_j})/\Q, L_{\pi_j,n}) = \varpi^CH^1(\Q(\rho_{\pi_j})/\Q, L_{\pi_j,n}(1)) = 0$$ for some constant $C \geq 0$. By inflation-restriction, it suffices to show $$\Hom_{G_\Q}(\Gal(\Q(\rho_{\pi})/\Q(\rho_{\pi_j})), L_{\pi_j, n} (i))$$ is uniformly bounded in $n$ for $i = 0, 1$. 
    By the same argument as for the 
   claim in Lemma \ref{lem:loc_cheb_rk0_endo}, any Galois-invariant group homomorphism $\Gal(\Q(\rho_{\pi})/\Q(\rho_{\pi_j}))\to L_{\pi_j,n}(i)$ lies in the proper subspace on which complex conjugation acts by $-1$; since $L_{\pi_j}\otimes \Q_p$ is absolutely irreducible by Theorem \ref{thm:nekovar} again, this suffices by \cite[Lemma 2.3.3]{liu2022beilinson}.
    \end{proof}
    \begin{lemma}\label{lem:breakup}
Let $L = L_1\oplus L_2$ be a free $O$-module of finite rank with $G_\Q$ action, where $L_i \otimes \Q_p$ are absolutely irreducible and distinct. Then there is a constant $C\geq 0$ with the following property: for any $G_\Q$-stable $O$-submodule $H \subset L/\varpi^n L$, we have
$$H \supset \varpi^C\pr_1(H)\oplus \varpi^C \pr_2(H).$$
    \end{lemma}
    Here $\pr_i: L/\varpi^n L \to L_i/\varpi^nL_i$ is the natural projection.
    \begin{proof}
    Write $L_{i,n} = L_i/\varpi^n$ and
        note that we have isomorphisms
        $$\frac{\pr_1(H)}{H \cap L_{1,n}} \xleftarrow{\sim} \frac{H}{H \cap L_{2,n} \oplus H \cap L_{1,n}}\xrightarrow{\sim} \frac{\pr_2(H)}{H \cap L_{2,n}}.$$

        In particular, it is enough to show that any isomorphic $O[G_\Q]$-module subquotients of $L_1$ and $L_2$ are $\varpi^C$-torsion for some universal constant $C$. Suppose on the contrary that for all integers $m \geq 0$, there exist submodules $B_i^m \subset A_i^m\subset L_i$ with $A_1^m/B_1^m \cong A_2^m/B_2^m$ not $\varpi^m$-torsion.
        Rescaling, we may assume without loss of generality that $A_i^m$ has nonzero image in $\varpi L_i$,  so by \cite[Lemma 2.3.3]{liu2022beilinson} we know $\varpi^{C_0} L_i \subset A_i^m$ for some constant $C_0$ depending on $L_1$ and $L_2$. 

For each $m$, let $N_i(m)$ be the maximal integer such that $$B_i^m \subset \varpi^{N_i(m)}L,$$
which implies
\begin{equation}
  \varpi^{N_i(m) + C_0} L\subset  B_i^m \subset \varpi^{N_i(m)}L.
\end{equation}
In particular, for $A_i^m/B_i^m$ not to be $\varpi^m$-torsion, we must have
\begin{equation}\label{lower bound Nim}
    N_i(m) + C_0 > m
\end{equation}
for all $m$. 
Now consider the chain of maps:
\begin{equation}\label{long chain of maps}
    \varpi^{C_0} L_1/\varpi^{N_1(m) + C_0} L_1 \twoheadrightarrow \varpi^{C_0} L_1/ (B_1^m\cap \varpi^{C_0}L_1) \hookrightarrow A_1^m/B_1^m \xrightarrow{\sim} A_2^m/B_2^m \hookrightarrow L_2/B_2^m \twoheadrightarrow L_2/\varpi^{N_2(m)}L_2.
\end{equation}
        The two injections in the diagram have cokernel annihilated by $\varpi^{C_0}$, so in particular the composite has cokernel annihilated by $\varpi^{2C_0}$. By (\ref{lower bound Nim}), after reindexing, we may assume without loss of generality that $N_i(m)$ is increasing in $m$. Then by compactness of $\Hom_O(L_1,L_2)$, up to passing to a subsequence, the maps in (\ref{long chain of maps}) fit together and give in the inverse limit a Galois-equivariant map
        $$\varpi^{C_0} L_1 \to L_2$$ with cokernel annihilated by $\varpi^{2C_0}$. This is absurd because there are no nontrivial maps $L_1\otimes \Q_p \to L_2\otimes \Q_p$, so we have a contradiction and the lemma is proved. 
    \end{proof}
\begin{lemma}\label{lem:adjoint_chebotarev}
Suppose $\pi$ is not endoscopic and satisfies Hypothesis $(\bigstar)$, and admissible primes exist for $\rho_\pi$. There is a constant $C\geq 0$ with the following property:
    suppose given integers $n, m$ and cocycles
    $c\in H^1(\Q, T_{\pi,n})$, $\phi\in H^1(\Q, L_{\pi,m})$, $\psi\in H^1(\Q, L_{\pi,m}(1))$.  Let $(\phi^\heartsuit, \phi^\circ)$ be the decomposition of $\phi$ with respect to 
    $$H^1(\Q, L_{\pi,m}) = H^1(\Q, L_{\pi,m}^\heartsuit)\oplus H^1(\Q, L_{\pi,m}^\circ),$$ and likewise for $\psi$. 
    Then for any $N \geq \max\set{n,m}$, there are infinitely many $N$-admissible primes $q$ such that all of the cocycles are unramified at $q$ and:
    \begin{itemize}
        \item $\ord_\varpi \loc_q c \geq \ord_\varpi c - C.$
        \item We have $$\ord_\varpi \left(\Res_q \phi, \frac{H^1_\unr(\Q_q,L_{\pi,m})}{H^1_\unr(\Q_q,L_{\pi,m})\cap H^1_\ord(\Q_q,L_{\pi,m})} \right) \geq \ord_\varpi \phi^\heartsuit - C.$$
        \item We have $$\ord_\varpi \left(\Res_q \psi, \frac{H^1_\unr(\Q_q,L_{\pi,m}(1))}{H^1_\unr(\Q_q,L_{\pi,m}(1))\cap H^1_\ord(\Q_q,L_{\pi,m}(1))} \right) \geq \ord_\varpi \psi^\heartsuit - C.$$
    \end{itemize}
\end{lemma}
\begin{proof}
  By Lemma \ref{lem:restr_for_adjt}(\ref{lem:restr_for_adjt_1}), Corollary \ref{cor:Galois_coh_restr}, and inflation-restriction,  the restrictions of $c$, $\phi$, and $\psi$ correspond to $G_\Q$-invariant homomorphisms
    \begin{align*}\Res(c)&: G_{\Q(T_{\pi,N})} \to T_{\pi,n},\\
    \Res(\phi) &= \Res(\phi^\heartsuit) \oplus \Res(\phi^\circ) : G_{\Q(T_{\pi,N})} \to L_{\pi,m},\\
    \Res(\psi)&= \Res(\psi^\heartsuit) \oplus \Res(\psi^\circ):G_{\Q(T_{\pi,N})} \to L_{\pi,m}(1)
    \end{align*}
    
    satisfying
    \begin{equation}\label{order of restriction is large for chebotarev shit}
        \ord_{\varpi} \Res(c) \geq \ord_\varpi (c) - C_0, \;\; \ord_{\varpi} \Res(\phi^?) \geq \ord _\varpi(\phi^?) - C_0,\;\;\ord_{\varpi} \Res(\psi^?) \geq \ord _\varpi(\psi^?) - C_0
    \end{equation}
where $? = \heartsuit$ or $\circ$ and $C_0\geq 0$ is a constant.
    We  combine these homomorphisms into a map
    $$H: G_{\Q(T_{\pi,N})} \to T_{\pi,n}\oplus L_{\pi,m} \oplus L_{\pi,m}(1).$$
Let $\pr_1,$ $\pr_2$, $\pr_3$ be the projections onto each of the three factors. 

   For any $g\in G_\Q$ that acts on $T_\pi$ as a scalar $z$, we have
   \begin{align*}
          \im (H) &\supset (g - 1)(g - z^2) \im (H) + (g-z)(g- z^2) \im (H) + (g - 1)(g- z)\im (H) \\
          &= (z - 1)(z^2 - 1) \pr_1\im (H) \oplus (z - 1)(z^2 - 1) \pr_2\im(H) \oplus (z^2 - 1)(z-1)\pr_3\im(H) \\
          &= (z-1)(z^2-1) \left(\im \Res(c) \oplus \im \Res(\phi) \oplus \im \Res(\psi).\right)
   \end{align*}
   By Lemma \ref{lem:breakup} combined with Lemma \ref{lemma action contains scalar}(\ref{lemma action contains scalar part}), this also implies 
   $$\im(H) \supset \varpi^{C_1} \left(\im \Res(c) \oplus \im \Res(\phi^\heartsuit)\oplus \im \Res (\phi^\circ) \oplus \im \Res(\psi^\heartsuit) \oplus \im \Res(\psi^\circ)\right) $$
   for some constant $C_1 \geq 0$ indepedent of $n$ amd $m$.
   In particular, by (\ref{order of restriction is large for chebotarev shit}) combined with \cite[Lemma 2.3.3]{liu2022beilinson}, there exists a constant $C \geq 0$ independent of $n$ and $m$, such that 
   \begin{equation}\label{eq:Ospan_H}
       O\cdot \im (H) \supset\varpi^{n-\ord_\varpi(c) + C} T_{\pi,n} \oplus \varpi^{m - \ord_\varpi(\phi^\heartsuit) + C} L^\heartsuit_{\pi,m} \oplus \varpi^{m - \ord_\varpi(\psi^\heartsuit) + C} L^\heartsuit_{\pi,m} (1).
   \end{equation}
   Now fix an admissible element  $g\in G_\Q$, which is possible by Lemma \ref{lem:admissible_TFAE}. By repeatedly raising $g$ to $p$th powers and taking the limit, we may assume without loss of generality that 
    $\rho_{\pi}(g)$ has finite order coprime to $p$.   Choose a basis $\set{e_1,e_2,e_3,e_4}$ for $T_\pi$ with respect to which we have
   $$\rho_\pi(g) = \begin{pmatrix}
       \chi_p^\cyc(g) &&& \\ & 1 && \\ && \ast & \ast \\&& \ast & \ast
   \end{pmatrix},$$
   and set $M_0 = \Span_O\set{e_1,e_2}$, $M_1 = \Span_O\set{e_3, e_4}$. 
   Note that, by Lemma
   \ref{lem:heart_circ}(\ref{lem:heart_circ_M0}),
 $L_\pi^\heartsuit$ surjects onto the direct summand $\ad^0 M_0$
of $L_\pi$. 
   
   Hence, using (\ref{eq:Ospan_H}), we may choose $h \in G_{\Q(\rho_{\pi})}$ as follows:
   \begin{enumerate}
       \item The $e_2$-component of $c(h)$ has order at least $\ord_\varpi(c) - C_.$
       \item The component of $\phi(h)$ in the $g$-invariant line 
       $\set{\begin{pmatrix}
           x &&& \\ & -x && \\ &&0 &\\ &&&0
       \end{pmatrix}}\subset L_\pi$ has order at least $\ord_\varpi(\phi^\heartsuit) - C$.
       \item The component of $\psi(h)$ in the $g$-invariant line
       $\set{\begin{pmatrix}
           0 &0&& \\ \ast &0  &&\\ &&&\\ &&&
       \end{pmatrix}}\subset L_\pi(1)$
       has order at least $\ord_\varpi(\psi^\heartsuit) - C$. 
   \end{enumerate}
In particular, because $g$ has finite order coprime to $p$, the same is also true for the corresponding  components of $c(gh)$, $\phi(gh)$, and $\psi(gh)$, with respect to any choice of cocycle representatives.

Now suppose $q\not\in S\cup \set{p}$ has Frobenius conjugate to $g$ in $\Gal(\Q(T_{\pi,N}, c, \phi, \psi))$. To show that $q$ satisfies the conclusion of the lemma, it suffices to observe that the $O$-modules $$\frac{H^1_\unr(\Q_q, L_{\pi,m})}{H^1_\unr(\Q_q, L_{\pi,m}) \cap H^1_\ord(\Q_q, L_{\pi,m})}, \;\; \frac{H^1_\unr(\Q_q, L_{\pi,m}(1))}{H^1_\unr(\Q_q, L_{\pi,m}(1)) \cap H^1_\ord(\Q_q, L_{\pi,m}(1))},$$
which are both free of rank one over $O/\varpi^m$, are generated by the cocyles
$$\Frob_q \mapsto \begin{pmatrix}
    1 & && \\ & -1 && \\ &&0 &\\ &&&0
\end{pmatrix},\;\;\Frob_q \mapsto \begin{pmatrix}
    0 &0&& \\ 1& 0 &&\\ &&&\\&&&
\end{pmatrix},$$
respectively.
\end{proof}
The Selmer groups in the next lemma (and the rest of this section) are the ones from  Definition \ref{def:appendix_admissible_etc}.
\begin{lemma}\label{two cases for funky killing Selmer lemma}
    Suppose $\pi$ is non-endoscopic and satisfies Hypothesis $(\bigstar)$, and let $C$ be the constant of Lemma \ref{lem:adjoint_chebotarev}. If $$\overline\Sel_{\mathcal F}(\Q, L_{\pi,m}) = 0$$ for some $m\geq \max\set{1, C}$, and  $q_1$ is an $N$-admissible prime for some $N \geq 5m$, then either:
\begin{enumerate}
    \item\label{first case for funky killing Selmer} $\overline\Sel_{\mathcal F(q_1)}(\Q, L_{\pi, 5m}) = 0$; or:
    \item \label{second case for funky killing Selmer}For any cocycle $c\in H^1(\Q, T_{\pi,n})$ with $n \leq N$, and any $M \geq N$, there exist infinitely many $M$-admissible primes $q_2$ such that
    $$\overline\Sel_{\mathcal F(q_1q_2)} (\Q, L_{\pi,5m}) = 0$$
    and $$\ord_{\varpi} \loc_{q_2} c\geq \ord_{\varpi} c - C.$$
\end{enumerate}
\end{lemma}
\begin{proof}
Without loss of generality, assume (\ref{first case for funky killing Selmer}) does not hold,
and let $\phi \in \Sel_{\mathcal F(q_1)}(\Q, L_{\pi,5m})$ be an element whose image in $\overline \Sel_{\mathcal F(q_1)} (\Q, L_{\pi,5m})$ is nonzero; it follows from Lemma \ref{stuff like lemma 6.1 FKP}(\ref{lemma 6.1 FKP},\ref{lemma 6.1 part III}) that $\ord_\varpi \phi = 5m$. Similarly, by Lemma \ref{balanced residual selmer ranks lemma}, there exists an element $\psi\in \Sel_{\mathcal F(q_1)}(\Q, L_{\pi,5m}(1))$ whose image in $\overline \Sel_{\mathcal F(q_1)}(\Q, L_{\pi,5m}(1))$ is nonzero, and we have $\ord_\varpi \psi = 5m$.
Write $$\phi = \phi^\heartsuit \oplus \phi^\circ, \;\; \psi = \psi^\heartsuit \oplus \psi^\circ$$ as in the statement of Lemma \ref{lem:adjoint_chebotarev}.
\setcounter{claimnumbered}{0}
\begin{claimnumbered}\label{claim:one_adjt}
        We have $\ord_\varpi \phi^\heartsuit\geq 4m$ and $ \ord_\varpi \psi^\heartsuit \geq 4m$.
\end{claimnumbered}
\begin{proof}[Proof of claim 1]
    Suppose that $\ord_\varpi \phi^\heartsuit < 4m$; then by Lemma \ref{stuff like lemma 6.1 FKP}(\ref{lemma 6.1 FKP},\ref{lemma 6.1 part III}), the images of $\phi$ and $\phi^\circ$ coincide in $H^1(\Q, L_{\pi,m})$. However, by Lemma \ref{lem:heart_circ}(\ref{lem:heart_circ_unr})  we have
    $$H^1_\unr (\Q_{q_1}, L_{\pi,m}^\circ) = H^1 (\Q_{q_1}, L_{\pi,m}^\circ) ,$$ so it follows that the image of $\phi$ lies in $\Sel_{\mathcal F}(\Q, L_{\pi,m})$; this contradicts the assumption $\overline \Sel_{\mathcal F}(\Q, L_{\pi,m})=0$.

    Similarly, if $\ord_\varpi \psi^\heartsuit < 4m$, then $\psi$ and $\psi^\circ$ have the same image in $H^1(\Q, L_{\pi,m}(1))$. However, the local Tate dual of Lemma \ref{lem:heart_circ}(\ref{lem:heart_circ_sum}) shows that $H^1(\Q_{q_1}, L_{\pi,m}^\circ(1)) \cap H^1_\ord (\Q_{q_1}, L_{\pi,m}(1)) = 0,$ so then the image of $\psi$ lies in $\Sel_{\mathcal F}(\Q, L_{\pi,m}(1))$, and this contradicts Lemma \ref{balanced residual selmer ranks lemma}. 
\end{proof}
By Claim 1 combined with Lemma \ref{lem:adjoint_chebotarev}, there are infinitely many $M$-admissible primes $q_2$ such that:
\begin{itemize}
    \item $\ord_\varpi \loc_{q_2} c \geq \ord_\varpi c - C.$
    \item We have $$\ord_\varpi\left(\Res_{q_2} \phi, \frac{H^1_\unr(\Q_{q_2}, L_{\pi,5m})}{H^1_\unr(\Q_{q_2}, L_{\pi,5m}) \cap H^1_\ord (\Q_{q_2}, L_{\pi,5m})}\right) \geq \ord_\varpi \phi^\heartsuit - C \geq 3m.$$
    \item We have $$\ord_\varpi\left(\Res_{q_2} \psi, \frac{H^1_\unr(\Q_{q_2}, L_{\pi,5m}(1))}{H^1_\unr(\Q_{q_2}, L_{\pi,5m}(1)) \cap H^1_\ord (\Q_{q_2}, L_{\pi,5m}(1))} \right)\geq \ord_\varpi \psi^\heartsuit - C \geq 3m.$$
\end{itemize}
Put $$\Sel_{\mathcal F_{q_2}(q_1)}(\Q, L_{\pi,5m}) = \Sel_{\mathcal F(q_1)}(\Q, L_{\pi,5m})\cap \Sel_{\mathcal F(q_1q_2)}(\Q, L_{\pi,5m})$$ and $$
\Sel_{\mathcal F_{q_1}}(\Q, L_{\pi,5m}) = \Sel_{\mathcal F}(\Q, L_{\pi,5m})\cap \Sel_{\mathcal F(q_1)}(\Q, L_{\pi,5m}).$$
Our next claim is:
\begin{claimnumbered}
    We have 
    $$\varpi^{3m-1} \Sel_{\mathcal F_{q_2}(q_1)}(\Q, L_{\pi, 5m}) =0.$$
\end{claimnumbered}
\begin{proof}[Proof of claim 2]
    First, we have the exact sequence
    \begin{equation}\label{first exact sequence in proof of funky killing selmer}
        0 \to \Sel_{\mathcal F_{q_1}}(\Q, L_{\pi,5m}) \to \Sel_{\mathcal F(q_1)}(\Q, L_{\pi, 5m}) \to \frac{H^1_\ord(\Q_{q_1}, L_{\pi,5m})}{H^1_\ord(\Q_{q_1}, L_{\pi,5m})\cap H^1_\unr(\Q_{q_1}, L_{\pi,5m})} \simeq O/\varpi^{5m},
    \end{equation}
    where the final isomorphism is by Lemma \ref{ordinary local condition is standard Lemma}.
    Because $\Sel_{\mathcal F_{q_1}}(\Q, L_{\pi,5m})$ is $\varpi^{m-1}$-torsion by Lemma \ref{stuff like lemma 6.1 FKP}(\ref{lemma 6.1 part II residual selmer groups and torsion}), but $\varpi^{5m-1}\Sel_{\mathcal F(q_1)}(\Q, L_{\pi,5m})\neq 0$ by assumption, from (\ref{first exact sequence in proof of funky killing selmer}) we  have an isomorphism of $O$-modules
    \begin{equation}
        \Sel_{\mathcal F(q_1)}(\Q, L_{\pi,5m}) \simeq O/\varpi^{5m} \oplus T
    \end{equation}
    where $T$ is $\varpi^{m-1}$-torsion. In particular, for any $x\in \Sel_{\mathcal F_{q_2}(q_1)}(\Q, L_{\pi,5m})$, $\varpi^{m-1} x = \varpi^a \phi$ for some $a \geq 0$. We conclude
    \begin{equation}
        0 = \ord_\varpi \left(\Res_{q_2} (\varpi^{m-1} x), \frac{H^1_\unr(\Q_{q_2}, L_{\pi,5m})}{H^1_\unr(\Q_{q_2}, L_{\pi,5m}) \cap H^1_\ord (\Q_{q_2}, L_{\pi,5m})}\right)  \geq 3m - a,
    \end{equation}
    so $a \geq 3m$, hence $\varpi^{3m - 1} x = 0.$
\end{proof}
Now we are ready to complete the proof of the lemma. We have the exact sequence
\begin{equation}\label{second exact sequence in proof of funky killing selmer}
        0 \to \Sel_{\mathcal F_{q_2}(q_1)}(\Q, L_{\pi,5m}) \to \Sel_{\mathcal F(q_1q_2)}(\Q, L_{\pi, 5m}) \to \frac{H^1_\ord(\Q_{q_2}, L_{\pi,5m})}{H^1_\ord(\Q_{q_2}, L_{\pi,5m})\cap H^1_\unr(\Q_{q_2}, L_{\pi,5m})} \simeq O/\varpi^{5m}.
\end{equation}
On the other hand, for any $\phi' \in \Sel_{\mathcal F(q_1q_2)}(\Q, L_{\pi,5m})$, we can compute the global Tate pairing
$$0 = \sum _v \langle \phi', \psi\rangle_v = \langle\phi',  \psi\rangle_{q_2}$$ by the definition of the local conditions for $\mathcal F(q_1q_2)$ and $\mathcal F(q_1)$. Now, the induced local Tate pairing 
\begin{equation*}
    \begin{split}
        &O/\varpi^{5m} \times O/\varpi^{5m} \simeq \\&\frac{H^1_\ord(\Q_{q_2}, L_{\pi,5m})}{H^1_\ord(\Q_{q_2}, L_{\pi,5m})\cap H^1_\unr(\Q_{q_2}, L_{\pi,5m})}\times \frac{H^1_\ord(\Q_{q_2}, L_{\pi,5m}(1))+ H^1_\unr(\Q_{q_2}, L_{\pi,5m}(1))}{H^1_\ord(\Q_{q_2}, L_{\pi,5m}(1))} \to O/\varpi^{5m}
    \end{split}
\end{equation*}
 is perfect, so we conclude
 \begin{equation*}
     \begin{split}
         \ord_\varpi &\left( \Res_{q_2} \phi', \frac{H^1_\ord(\Q_{q_2}, L_{\pi,5m})}{H^1_\ord(\Q_{q_2}, L_{\pi,5m})\cap H^1_\unr(\Q_{q_2}, L_{\pi,5m})} \simeq O/\varpi^{5m}\right) \\&\leq 5m - \ord_\varpi \left(\Res_{q_2} \psi, \frac{H^1_\ord(\Q_{q_2}, L_{\pi,5m}(1))+ H^1_\unr(\Q_{q_2}, L_{\pi,5m}(1))}{H^1_\ord(\Q_{q_2}, L_{\pi,5m}(1)) }\right)\leq 2m.
     \end{split}
 \end{equation*}
In particular, the image of the final map in (\ref{second exact sequence in proof of funky killing selmer}) is $\varpi^{2m}$-torsion, so we have
$$\varpi^{5m-1} \Sel_{\mathcal F(q_1q_2)}(\Q, L_{\pi, 5m}) = 0$$ by Claim 2. Hence by Lemma \ref{stuff like lemma 6.1 FKP}(\ref{lemma 6.1 FKP},\ref{lemma 6.1 part III}), $\overline{\Sel_{\mathcal F(q_1q_2)}(\Q, L_{\pi, 5m})} = 0$, as desired. 
\end{proof}
We also have an endoscopic analogue:
   \begin{lemma}\label{lem:endoscopic_funky} Suppose $\pi$ is endoscopic associated to a pair $(\pi_1,\pi_2)$, and  let $j = 1$ or 2. Suppose $\pi_j$ is non-CM. 
    Then there exists a constant $C$  with the following property: if $$\overline \Sel_{\mathcal F}(\Q, L_{\pi_j,m}) = 0$$ for some $m \geq \max\set{1, C}$, and $q_1$ is an $N$-admissible prime for some $N \geq 5m$ which is BD-admissible for $\pi_j$, then either:
 \begin{enumerate}
     \item $\overline \Sel_{\mathcal F(q_1)} (\Q, L_{\pi_j,5m} ) = 0$; or:
     \item For any cocycle $c\in H^1(\Q, T_{\pi_j,n})$ with $n \leq N$, and for any $M \geq N$, there exist infinitely many $M$-admissible primes $q_2$, BD-admissible for $\pi_j$, such that 
     $$\overline \Sel_{\mathcal F(q_1q_2)} (\Q, L_{\pi_j,5m}) = 0$$ and $$\ord_\varpi \loc_{q_2} c \geq \ord _\varpi c - C.$$
 \end{enumerate}
\end{lemma}
\begin{proof}
The same argument used to 
prove Lemma \ref{two cases for funky killing Selmer lemma} applies formally, taking $L_\pi^\circ \coloneqq  0$ and $L_{\pi}^\heartsuit \coloneqq L_{\pi_j}$. When proving the appropriate analogue of Lemma \ref{lem:adjoint_chebotarev}, one uses the claim in the proof of Lemma \ref{lem:loc_cheb_rk0_endo} in place of Corollary \ref{cor:Galois_coh_restr}, Lemma \ref{lem:restr_for_adjt}(\ref{lem:restr_for_adjt_endo}) in place of Lemma \ref{lem:restr_for_adjt}(\ref{lem:restr_for_adjt_1}), and  Theorem \ref{thm:nekovar} in place of Lemma \ref{lemma action contains scalar}(\ref{lemma action contains scalar part}).
\end{proof}

\subsection{Proof of the main results}
\begin{thm}\label{thm:rk1_non_endoscopic_ultimate}
    Suppose $\pi$ is non-endoscopic and satisfies Hypothesis $(\bigstar)$, 
  and let $\p$ be a prime of $E_0$ satisfying Assumption \ref{assumptions_on_p}, such that admissible primes exist for $\rho_{\pi,\p}$.
Then for all $D\geq 1$ 
 squarefree with $\sigma(D)$ even, 
    $$\kappa^D(1)_\p \neq 0\implies \dim H^1_f(\Q, V_{\pi,\p}) = 1.$$
\end{thm}
\begin{rmk}
    To deduce Theorem \ref{thm:intro_rk1} as stated in the introduction from Theorem \ref{thm:rk1_non_endoscopic_ultimate}, one needs the following additional observation.  Let $\widetilde \kappa^D(1)_\p \subset H^1(\Q, T_{\pi,\p})$ be defined like $\kappa^D(1)_\p$, but allowing all neat level structures, not only those hyperspecial outside our fixed set $S$ of Notation \ref{notation:pi_basic}: then $\widetilde \kappa^D(1)_\p \neq 0 \implies \kappa^D(1)_\p \neq 0$. This is a consequence of Proposition \ref{prop:equivariance_and_modularity} and Corollary \ref{cor:coh_relevant}.  
\end{rmk}
\begin{proof}
    Let $N$ be a large integer to be specified later, and let $M\geq N$ be the number from Lemma \ref{lemma kill pairing on crystalline}(\ref{lemma kill pairing on crystalline part two}) applied to $T_\pi$ and $n=N$. By Theorem \ref{ultimate theorem for second TRL}, there exists a constant $C_0 \geq 0$ independent of $N$ and an $M$-admissible prime $q_1$ such that
    $$\ord_\varpi \loc_{q_1} \kappa^D_{n(q_1)} (1) \geq n(q_1) - C_0$$ and
    $$\ord_\varpi \lambda^D_{n(q_1)}(q_1)  \geq n(q_1) - C_0.$$
Because $\kappa^D(1) \subset H^1(\Q, T_\pi)$ is a free $O$-module by Lemma \ref{lem:H1_tors_free_new}(\ref{lem:H1_tors_free_part}), we may fix a class $\kappa^D(1)_0 \in \kappa^D(1)$, with image $\kappa^D_m(1)_0 $ in $H^1(\Q, T_{\pi,m})$ for all $m \geq 1$, such that 
\begin{equation}\label{eq:loc_kappa1_rk1}
    \ord _\varpi \loc_{q_1} \kappa^D_m(1)_0\geq m - C_0
\end{equation} for all $m \leq n(q_1)$. 

\begin{claim}
    Suppose $\dim H^1_f(\Q, V_\pi) > 1$. Then there exists a class $c\in H^1_f(\Q, T_\pi)$, with images $c_m \in H^1_f(\Q, T_{\pi,m})$ for all $m \geq 1$, such that:
    \begin{enumerate}
        \item\label{part:claim_one_rk1} $\ord _\varpi c_m = m$ for all $m$.
        \item \label{part:claim_two_rk1}$\ord_\varpi \loc _{q_1} c_m \leq C_0$ for all $m \leq n(q_1)$. 
    \end{enumerate}
\end{claim}
\begin{proof}[Proof of claim]
By the assumption $\dim H^1_f(\Q, V_\pi) > 1$,   we may choose $c\in H^1_f(\Q, T_\pi)$ such that
\begin{equation}\label{eq:for_claim_rk1}
    c + \alpha \kappa^D(1)_0\not\in \varpi H^1_f(\Q, T_\pi), \;\; \forall \alpha\in O.
\end{equation}
Adjusting $c$ by an $O$-multiple of $\kappa^D(1)_0$ and using (\ref{eq:loc_kappa1_rk1}), we can ensure that (\ref{part:claim_two_rk1}) holds. 
    By definition, we have an injection $$\frac{H^1(\Q, T_\pi)}{H^1_f(\Q, T_\pi)} \hookrightarrow \prod_v \frac{H^1(\Q_v, V_\pi)}{H^1_f(\Q_v, V_\pi)},$$ hence the quotient $H^1(\Q, T_\pi)/H^1_f(\Q, T_\pi)$ is $O$-torsion-free, and in particular $c\not\in \varpi H^1(\Q, T_\pi)$ by (\ref{eq:for_claim_rk1}). 
Then (\ref{part:claim_one_rk1}) holds as well by Lemma \ref{lem:H1_tors_free_new}(\ref{lem:H1_tf_order}).

\end{proof}

By Theorem \ref{theorem of thorne} and Lemma \ref{lemma for using theorem of thorne}, there exists a constant $m_0 \geq 1$ such that $\overline \Sel_{\mathcal F} (\Q,\ad^0 \rho_{m_0}) = 0$. Without loss of generality, we assume $N \geq 10m_0$ and $m_0 \geq \max\set{1,C}$, for the constant $C$ of Lemma \ref{lem:adjoint_chebotarev}. Now consider the following two cases:
\begin{enumerate}
    \item If $\overline \Sel_{\mathcal F(q_1)}(\Q, L_{\pi, 10m_0 - 1}) = 0$, then we choose (by Lemma \ref{lem:loc_cheb_rk0} and (\ref{part:claim_one_rk1}) of the claim) an $n(q_1)$-admissible prime $q_2$ such that $\ord_\varpi \loc_{q_2} c_{n(q_1)} \geq n(q_1) - C_1$ for a constant $C_1 \geq 0$ independent of $N$, $q_1$, and $q_2$.
    \item If $\overline\Sel_{\mathcal F(q_1)}(\Q, L_{\pi, 10m_0 - 1}) \neq 0$, then \emph{a fortiori} we have 
    $\overline\Sel_{\mathcal F(q_1)}(\Q, L_{\pi, 5m_0}) \neq 0$.
    We choose (by Lemma \ref{two cases for funky killing Selmer lemma}) an $(n(q_1) + 5m_0)$-admissible prime $q_2 \neq q_1$ such that
    $\overline \Sel_{\mathcal F(q_1q_2)}(\Q, L_{\pi, 5m_0}) = 0$ and $\ord_\varpi \loc_{q_2} c_{n(q_1)} \geq \ord_\varpi n(q_1) - C_1$
    for a constant $C_1 \geq 0$ independent of $N$, $q_1$, and $q_2$.
\end{enumerate}
By Theorem \ref{thm:1ERL} combined with Corollary \ref{cor:relaxed_bound_appendix},
in either case we can conclude -- as long as $N$ is sufficiently large in a manner depending only on $\pi$, $\p$, and $m_0$ -- that
$$\partial_{q_2} \kappa_{n(q_1)}^D(q_1q_2) \supset \lambda_{n(q_1)}^D (q_1) \cdot (\varpi^{C_2})$$
for a constant $C_2$ that is independent of $N$, $q_1$, and $q_2$. Hence we may choose an element $\kappa_{n(q_1)}^D(q_1q_2)_0\in  \kappa_{n(q_1)}^D(q_1q_2)$,
with images $\kappa_m^D(q_1q_2)_0 \in H^1(\Q, T_{\pi,m})$ for all $m \leq n(q_1)$, 
such that
\begin{equation}\label{eq:partialz_rk1}
    \ord_\varpi \partial_{q_2} \kappa_m^D(q_1q_2)_0\geq m- C_0 - C_2.
\end{equation}
We now compute the Tate pairing
\begin{equation}\label{eq:Tate_rk1}
    \langle c_{N}, \kappa_N^D(q_1q_2)_0\rangle = \sum_{v} \langle c_{N}, \kappa_N^D(q_1q_2)_0\rangle_v= 0 \in O/\varpi^{N}.
\end{equation}
Arguing as in the proof of Theorem \ref{thm:rk_zero_main} and using that $M \leq n(q_1)$,
there is a constant $C_3 \geq 0$ 
 so that
$$\ord_\varpi \langle c_{N}, \kappa_N^D(q_1q_2)_0\rangle_v \leq C_3,\; \;\forall v\not\in \set{q_1,q_2}.$$
By (\ref{part:claim_two_rk1}) of the claim, we also have
$$\ord_\varpi \langle c_{N}, \kappa_N^D(q_1q_2)_0 \rangle_{q_1} \leq C_0.$$
In particular, (\ref{eq:Tate_rk1}) implies
$$\ord_\varpi  \langle c_{N}, \kappa_N^D(q_1q_2)_0\rangle_{q_2} \leq \max\set{C_0, C_3}.$$
However, by Proposition \ref{prop:local_adm_free}, (\ref{eq:partialz_rk1}) combined with the choice of $q_2$ implies
$$\ord_\varpi \langle  c_{N}, \kappa_N^D(q_1q_2)_0\rangle_{q_2} \geq N -C_0 - C_1 - C_2.$$
This is a contradiction if we choose $N > C_0 + C_1 + C_2 + \max\set{C_0,C_3},$ and the proof of the theorem is complete.
    \end{proof}
    
\begin{thm}\label{thm:endoscopic_rk1}
Suppose $\pi$ is endoscopic, associated to a pair $(\pi_1,\pi_2)$ of  automorphic representations of $\GL_2(\A)$ (in any order), and $\p$ is a prime of $E_0$ satisfying Assumption \ref{assumptions_on_p}, such that $H^1_f(\Q, V_{\pi_1,\p}\otimes V_{\pi_2,\p}(-1)) = 0$.
Assume as well that $\kappa^D(1)_\p^{(1)} \neq 0$ for some squarefree $D$ with $\sigma(D)$ even. Then the following hold:
\begin{enumerate}
    \item\label{thm:rk_one_endo_one} If $\pi_1$ is non-CM and there exist admissible primes  which are BD-admissible for $\rho_{\pi_1,\p}$, then 
    $$\dim H^1_f(\Q, V_{\pi_1,\p}) = 1.$$
    \item\label{thm:rk_one_endo_two} If for each $j = 1,2$, there exist admissible primes  which are BD-admissible for  $\rho_{\pi_j,\p}$, then $$\dim H^1_f(\Q, V_{\pi_2,\p}) = 0.$$
\end{enumerate}
In particular, if  for each $j = 1,2$, $\pi_j$ is non-CM and there exist admissible primes  which are BD-admissible for  $\rho_{\pi_j,\p}$, then $$\kappa^D(1)_\p \neq 0 \implies \dim H^1_f(\Q, V_{\pi,\p}) = 1.$$

\end{thm}
\begin{rmk}
    The existence of admissible primes which are BD-admissible for each $\rho_{\pi_j, \p}$ is considered in Proposition \ref{prop:adm_endosc}.
\end{rmk}
\begin{proof}
  Let $N$ be a large integer to be specified later, and let $M \geq N$ be the number from Lemma \ref{lemma kill pairing on crystalline}(\ref{lemma kill pairing on crystalline part two}) for $T_\pi$ and $n = N$. By Theorem \ref{thm:2ERL_endo_ultimate}, there exists a constant $C_0 \geq 0$ independent of $N$ and an $M$-admissible prime $q_1$, BD-admissible for $\rho_{\pi_1}$, such that
  $$\ord_\varpi \loc_{q_1}\kappa_{n(q_1)}^D(1)^{(1)} \geq n(q_1) - C_0$$ and $$\ord_\varpi \lambda^D_{n(q_1)} (q_1) \geq n(q_1) - C_0.$$ As in Theorem \ref{thm:rk1_non_endoscopic_ultimate}, we fix a class 
  $\kappa^D(1)_0\in \kappa^D(1)^{(1)}$, with image $\kappa_m^D(1)_0 \in H^1(\Q, T_{\pi_1,m})$ for all $m \geq 1$, such that \begin{equation}
      \label{eq:loc_kappa1_endo}
      \ord_\varpi \loc_{q_1} \kappa_m^D(1)_0\geq m - C_0
  \end{equation}
  for all $m \leq n(q_1)$.

  Now we suppose we are in case (\ref{thm:rk_one_endo_one}) of the theorem; assume for contradiction $\dim H^1_f(\Q, V_{\pi_1}) > 1$.
  By the same argument as for the claim in Theorem \ref{thm:rk1_non_endoscopic_ultimate},  we have a class $c\in H^1_f(\Q, T_{\pi_1})$, with images $c_m \in H^1_f(\Q, T_{\pi_1,m})$, such that:
  \begin{enumerate}
      \item $  \ord_\varpi c_m = m$ for all $ m \geq 1.$
      \item   $ \ord\loc_{q_1} c_m \leq C_0$ for all $ 1\leq m\leq n(q_1).$
  \end{enumerate}
   
By Proposition \ref{prop:thorne_endoscopic} and Lemma \ref{lemma for using theorem of thorne}, there exists a constant $m_0\geq 1$ such that $$\overline \Sel_{\mathcal F}(\Q, \ad^0 \rho_{\pi_1,m_0}) = 0.$$ Increasing $N$ if necessary, we assume $N > 10m_0$. 
  Now consider the following cases:
  \begin{enumerate}
      \item If  $\overline \Sel_{\mathcal F(q_1)} (\Q, \ad^0 T_{\pi_1, 10 m_0 - 1}) = 0$, 
      then choose (by Lemma \ref{lem:loc_cheb_rk0_endo}) $q_2$ to be $n(q_1)$-admissible and BD-admissible for $\rho_{\pi_1}$ such that $$\ord_\varpi \loc_{q_2} c_{n(q_1)} \geq n(q_1) - C_1$$ for a constant $C_1 \geq 0$ independent of $N$, $q_1$, and $q_2$. 
      \item If  $\overline \Sel_{\mathcal F(q_1)}(\Q, \ad^0 T_{\pi_1, 10m_0 -1}) \neq 0$, then \emph{a fortiori} we have $$\overline \Sel_{\mathcal F(q_1)}(\Q, \ad^0 T_{\pi_1, 5m_0}) \neq 0.$$ We choose, by Lemma \ref{lem:endoscopic_funky}, an $(n(q_1) + 5m_0)$-admissible prime $q_2 \neq q_1$, BD-admissible for $\pi_1$, such that $$\overline \Sel_{\mathcal F(q_1q_2)} (\Q, \ad^0 T_{\pi_1, 5m_0}) = 0$$and $\ord_\varpi \loc_{q_2} c_{n(q_1)} \geq n(q_1) - C_1$ for a constant $C_1 \geq 0$ independent of $N$, $q_1$, and $q_2$. 
  \end{enumerate}
By Theorem \ref{thm:1ERL_endoscopic} and Corollary \ref{cor:relaxed_bound_appendix}, in either case we can conclude
$$\partial_{q_2} \kappa^D_{n(q_1)} (q_1q_2) \supset \lambda_{n(q_1)}^D(q_1) \cdot (\varpi^{C_2})$$ for all $N$ sufficiently large in a manner depending only on $\pi$, $\p$, and $m_0$, and
for a constant $C_2\geq 0$ that is independent of $N$, $q_1$, and $q_2$. The remainder of the proof of (\ref{thm:rk_one_endo_one}) is now identical to Theorem \ref{thm:rk1_non_endoscopic_ultimate}.

Now we suppose we are in case (\ref{thm:rk_one_endo_two}), and assume for contradiction that there exists a non-torsion class $c\in H^1_f(\Q, T_{\pi_2})$ with images $c_m \in H^1(\Q, T_{\pi_2,m})$ for all $m \geq 1$. By the proof of the claim in Theorem \ref{thm:rk1_non_endoscopic_ultimate}, we may assume $\ord_\varpi c_m = m$ for all $m$. 
Then we choose (by Lemma \ref{lem:loc_cheb_rk0_endo}) $q_2$ to be $n(q_1)$-admissible and BD-admissible for $\rho_{\pi_2}$, such that $\ord_\varpi \loc_{q_2} c_{n(q_1)} \geq n(q_1) - C_1$ for a constant $C_1 \geq 0$ independent of $N$, $q_1$, and $q_2$. 
By Proposition \ref{prop:thorne_endoscopic} and Lemma \ref{lemma for using theorem of thorne}, there exists a constant $m_0 \geq 1$ such that $$\overline \Sel_{\mathcal F}(\Q, \ad^0 \rho_{\pi_2,m_0}) = 0.$$
By Theorem \ref{thm:1ERL_endoscopic}(\ref{thm:1ERL_endoscopic_one}) and Corollary \ref{cor:relaxed_bound_appendix}, we can conclude that $$\partial_{q_2} \kappa^D_{N}(q_1q_2) \supset \lambda_N^D(q_1) \cdot (\varpi^{C_2})$$  
 for all $N$ sufficiently large  and for a constant $C_2 \geq 0$ that is independent of $N$, $q_1$, and $q_2$. 
The remainder of the proof of (\ref{thm:rk_one_endo_two}) now follows the  proof of Theorem \ref{thm:rk1_non_endoscopic_ultimate}, using that $\loc_{q_1} c_N = 0$ because $H^1_f(\Q_{q_1}, T_{\pi_2,N}) = 0$.

    \end{proof}

\appendix
\section{Cohomology of $\spin_4$ Shimura varieties}\label{sec:appendix_spin4}
\subsection{ The auxiliary quaternionic Shimura variety}
\subsubsection{}\label{subsubsec:appendix_spin4_start}
For this section, let $V $ be  a quadratic space over $\Q $ of signature $ (2, 2) $ and nontrivial discriminant character $\chi $. Then by \cite[Appendix A]{knus1988notes}, $ C ^ + (V) = B\otimes_\Q F\eqqcolon B_F $, where $ B $ is an indefinite quaternion algebra over $\Q $ and $ F/\Q $ is the real quadratic field associated to $\chi $.
We abbreviate $G\coloneqq \spin(V)$, $\widetilde G \coloneqq \Res_{F/\Q} B_F^\times$. 
Then we have an exact sequence of algebraic groups over $\Q $:
\begin{equation}
\label{ exact sequence of algebraic groups for spin 4}
1\rightarrow G\rightarrow\widetilde G\rightarrow(\Restriction_{F/\Q}\mathbb G_m)/\mathbb G_m\rightarrow 1.
\end{equation}
Let $ K =\product K_\l \subset G (\A_f) $ be a neat compact open subgroup, and let $ S $ be a nonempty set of primes such that $ K_\l  $ is hyperspecial for $ \l\not\in S $.
Fix a neat compact open subgroup $\widetilde K  =\product\widetilde K_\l \subset \widetilde G (\A_f)$ satisfying the following conditions:
\begin{enumerate}
    \item\label{first condition on widetilde K} For all $\l \not\in S$, $\widetilde K_\l$ is hyperspecial.
    \item \label{second condition on widetilde K}$\widetilde K \cap \spin(V)(\A_f) = K$. 
    \item \label{third condition on widetilde K}We have $\nu(\widetilde K)\cap O_{F,+}^\times = (O_F^\times \cap \widetilde K)^2$, where $\nu: \widetilde G \to \Res_{F/\Q} \mathbb G_m$ is the norm character and  $O_{F,+}^\times$ is the group of totally positive units of $F$.
\end{enumerate}
Such a $\widetilde K$  exists because $S$ is nonempty. Let $\widetilde \Sh_{\widetilde K}(V)$ be the Shimura variety for $\widetilde G$ at level $\widetilde K$.
\begin{prop}
    Under conditions (\ref{first condition on widetilde K}) - (\ref{third condition on widetilde K}) above, the natural map 
\begin{equation}\label{eq:map of shimura varieties at level widetilde K}
    \Shimura_K (V)\rightarrow\widetilde \Shimura_{\widetilde K} (V)
\end{equation} 
is an open and closed embedding. 
\end{prop}
\begin{proof}
    This follows from \cite[Proposition 2.10, Remark 2.11]{liu2020supersingular}.
\end{proof}

\subsubsection{Hecke algebras} \label{subsubsec:Hecke_gspin4_appendix}

For a prime $\l\not\in S$ and a ring $R$, let $\T_{\l, R}$ 
 (resp. $\widetilde{\T}_{\l, R}$) denote the spherical Hecke algebra of $K_\l$-biinvariant (resp. $\widetilde K_\l$-biinvariant) $R$-valued functions on $G(\Q_\l)$ (resp. $\widetilde G(\Q_\l)$). If $S'\supset S$ is a finite set of primes, then we set
$$\T_R^{S'} \coloneqq \otimes'_{\l\not\in S'} \T_{\l,R},\;\;\widetilde{\T}_R^{S'} \coloneqq \otimes'_{\l\not\in S'} \widetilde{\T}_{\l,R}.$$ When $R = \Z$ we drop it from the notation.
\begin{prop}\label{prop decompose hecke action on quaternionic shimura variety}
Fix an isomorphism $\iota: \overline\Q_p \isomorphism \C$.
    There is a decomposition of $\widetilde \T^{S}$-modules:
    $$H^2_{\et,!}(\widetilde \Sh_{\widetilde K}(V)_{\overline{\Q}}, \overline \Q_p) = \bigoplus_{\pi_f} H^2_{\et,!}(\widetilde \Sh_{\widetilde K}(V)_{\overline{\Q}}, \overline \Q_p)_{\pi_f}\otimes \iota^{-1}\pi_f^{\widetilde K} \oplus \bigoplus_\chi H^2_{\et,!}(\widetilde \Sh_{\widetilde K}(V)_{\overline{\Q}}, \overline \Q_p)_{\chi\circ \det} \otimes \iota^{-1}(\chi\circ \det)^{\widetilde K},$$
    where $\pi_f$ runs over finite parts of cuspidal, infinite-dimensional automorphic representations $\pi$ of $B_F(\A_F)^\times$ with discrete series archimedean components of parallel weight 2, and $\chi$ runs over finite order characters of $F^\times \backslash \A_F^\times$. 
    As $G_\Q$-representations, we have $$H^2_{\et,!}(\widetilde \Sh_{\widetilde K}(V)_{\overline{\Q}}, \overline \Q_p)_{\pi_f} ^{ss}= \otimes-\Ind_{G_F}^{G_\Q} \rho_{\JL(\pi),\iota}$$ and $$H^2_{\et,!}(\widetilde \Sh_{\widetilde K}(V)_{\overline{\Q}}, \overline \Q_p)_{\chi\circ \det}^{ss} = \iota^{-1}\rec(\chi|_{\A_\Q^\times})(-1) \oplus \iota^{-1}\rec(\chi|_{\A_\Q^\times})\cdot \omega_{F/\Q}(-1),$$ where $\omega_{F/\Q}$ is the quadratic character of $G_\Q$ associated to $F$.
    
\end{prop}
\begin{proof}
When $B_F$ is split, this follows from the discussion in \cite[\S XI.2]{geer1988hilbert}. In the nonsplit case, the Hecke module decomposition is clear from Matsushima's formula, and the Galois actions follow from \cite{Langlands1979zeta}.
\end{proof}

\subsubsection{}
Let $\l\not\in S$ be a prime;
then we may identify $B\otimes \Q_\l\simeq M_{2}(\Q_\l)$ in such a way that $\widetilde K_\l = \GL_2(O_F\otimes \Z_\l )$ and
$G(\Q_\l)\subset \widetilde G(\Q_\l) \simeq \GL_2(F\otimes \Q_\l )$ consists of those matrices having determinant in $\Q_\l ^\times$.
Let $T(\Q_\l)$ and $\widetilde T(\Q_\l)$ be the standard diagonal tori in $G(\Q_\l)$ and $\widetilde G(\Q_\l)$, respectively, and let $B(\Q_\l)$ and $\widetilde B(\Q_\l)$ be the upper triangular Borel subgroups. 

\begin{prop}
    Let $\chi_0: \widetilde T(\Q_\l)\to \C^\times$ be an unramified character. Then $\left(\Ind_{\widetilde B(\Q_\l)}^{\widetilde G(\Q_\l)} \chi_0\right)|_{G(\Q_\l)}$ has a unique $K_\l $-spherical constituent, which is isomorphic to the unique spherical constituent of $\Ind_{B(\Q_\l)}^{G(\Q_\l)} (\chi_0|_{T(\Q_\l)})$. 
\end{prop}
\begin{proof}
    This is clear from the observation that $\widetilde B(\Q_\l)\cdot K_\l  = \widetilde B(\Q_\l)\cdot \widetilde K_\l  = \widetilde G(\Q_\l)$; indeed, $K_\l  \cdot (\widetilde K_\l  \cap \widetilde T(\Q_\l)) = \widetilde K_\l $ because $\nu(\widetilde T(\Q_\l)\cap \widetilde K_\l ) = \nu(\widetilde K_\l )$. 
\end{proof}
It follows immediately that:
\begin{cor}\label{corollary local Hecke actions}
    Let $\widetilde \pi$ be an irreducible admissible representation of $\widetilde G(\Q_\l)$ which is $\widetilde K_\l $-spherical, and let $\chi_{\widetilde \pi}: \widetilde{\T}_\l \to \C$ be the character giving the Hecke action on $\widetilde \pi^{\widetilde K_\l}$. Then, viewing $\widetilde \pi$ as an admissible representation of $G(\Q_\l)$, $\T_\l$ stabilizes the one-dimensional space $\widetilde \pi^{\widetilde K_\l }$ and acts on it via the composite of $\chi_{\widetilde\pi}$ with the homomorphism $\T_\l \to \widetilde \T_\l$ determined by the Satake transform and the map of dual groups ${}^L\Res_{F/\Q} B_F^\times \twoheadrightarrow {}^L \spin(V).$
\end{cor}

\subsection{Hecke action on Tate classes for $\spin_4$}
For an automorphic represention $\pi$ of $\GL_2(\A_\Q)$, let $\BC_{F/\Q}(\pi)$ denote the base change to $\GL_2(\A_F)$. 
\begin{lemma}\label{lemma which have tate classes on quaternionic shimura variety}
Continue the notation of Proposition \ref{prop decompose hecke action on quaternionic shimura variety}.

    \begin{enumerate}
        \item\label{lemma which have tate classes on quaternionic shimura variety part one} If $H^2_{\et,!} (\widetilde \Sh_{\widetilde K}(V)_{\overline\Q}, \overline \Q_p(1))^{G_\Q}_{\pi_f} \neq 0$, then there exists an automorphic representation $\pi_0$ of $\GL_2(\A_\Q)$, with $\pi_{0,\infty}$  discrete series of weight 2, and a finite-order character $\chi$ of $F^\times \backslash \A_F^\times$, such that $\JL(\pi) = \BC_{F/\Q}(\pi_0) \otimes \chi$. 
        \item\label{lemma which have tate classes on quaternionic shimura variety part two} If $H^2_{\et,!}(\widetilde \Sh_{\widetilde K}(V)_{\overline\Q}, \overline \Q_p(1))^{G_\Q}_{\chi\circ \det} \neq 0$, then $\chi|_{\A_\Q^\times} = \omega_{F/\Q}$ or $\blackboardone$. 
    \end{enumerate}
\end{lemma}
\begin{proof}
    Part (\ref{lemma which have tate classes on quaternionic shimura variety part two}) is obvious from Proposition \ref{prop decompose hecke action on quaternionic shimura variety}. Part (\ref{lemma which have tate classes on quaternionic shimura variety part one}) follows from  the proof of \cite[Theorem XI.4.6(i)]{geer1988hilbert}, except for the assertion about $\pi_{0,\infty}$; but this is clear by \cite[Chapter 3, Theorem 5.1]{arthur1989basechange} and the archimedean condition on $\pi$ in Proposition \ref{prop decompose hecke action on quaternionic shimura variety}. 
\end{proof}
\begin{cor} \label{cor:gspin4_cohomology_decomposes}
    The $\T^S$-module $H^2_{\et,!}(\Sh_{K}(V), \overline\Q_p(1))^{G_\Q} $ decomposes completely into a direct sum of characters $\boldsymbol h: \mathcal \T^S\to\overline\Q_p$, each of which arises from the action of $\T^S$ on either:
    \begin{itemize}
        \item A newform in an automorphic representation $\pi$ of $B_F(\A_F)^\times$, unramified outside $S$, such that $\pi$ satisfies the conclusion of Lemma \ref{lemma which have tate classes on quaternionic shimura variety}(\ref{lemma which have tate classes on quaternionic shimura variety part one}).
        \item The automorphic character $\chi_0\circ \nu$ of $\spin(V)(\A_\Q)$, where $\chi_0$ is either trivial or the Hecke character associated to $F/\Q$. 
    \end{itemize}
\end{cor}
\begin{proof}
Because (\ref{eq:map of shimura varieties at level widetilde K}) is an open and closed embedding, we have a split inclusion of $\T^S$-modules $$H^2_{\et,!}(\Sh_{K}(V)_{\overline \Q}, \overline \Q_p(1))\subset H^2_{\et,!}(\widetilde\Sh_{\widetilde K}(V)_{\overline \Q}, \overline \Q_p(1))$$  and so the corollary is immediate from Lemma \ref{lemma which have tate classes on quaternionic shimura variety} and Corollary \ref{corollary local Hecke actions}.
\end{proof}
\section{ Relative deformation theory and level-raising}\label{sec:appendix_def_theory}
In this appendix, we recall the relative deformation theory of Fakhruddin, Khare, and Patrikis \cite{fakhruddin2021relative} in a format useful for characteristic zero level-raising. In the hope that this discussion will be helpful in future work, we work with representations valued in more general groups than are needed for the main text.
\subsection {Notation}
\subsubsection{The group}
Let $G $
be a smooth, split group scheme over the ring of integers $O $
of a finite extension $E $ of $\Q_p$, such that the neutral component of $G $
is a connected reductive group. Let $G ^\derived $
be the derived subgroup of $G $.
Write $\varpi $
for the uniformizer of $O $. We suppose $p >2 $
and $G $
satisfy \cite[Assumption 2.1]{fakhruddin2021relative}. When $G = \GSP_4$ or $\GL_2$, which are the cases relevant for the main text,  \emph{loc. cit.} is satisfied for all odd $p$. Let $d_G$ be the dimension of $\Lie G^\der$. 
\subsubsection{The Galois representation}
Let $k $
be a number field, and fix a Galois representation $\rho: G_k\rightarrow G (O). $
For an integer $n\geq 1 $, let $\rho_n: G_k\rightarrow G (O/\varpi ^ n) $
be the reduction of $\rho $, and let $\overline\rho\coloneqq\rho_1 $. Also write $\adjoint ^ 0\rho $, $\adjoint ^ 0\rho_n $, and $\adjoint ^ 0\overline\rho $
for the natural $G_k $-representations on $\Leo G ^\derived $, $\Leo G ^\derived\otimes_OO/\varpi ^ n $, and $\Leo G ^\derived\otimes_OO/\varpi $,  respectively. Let $\Sigma_p$ and $\Sigma_\infty$ be the set of places of $k$ lying above $p$ and $\infty$, respectively. We will always suppose fixed a finite set $S$ of nonarchimedean places of $k$ such that $\Sigma_p \cap S = \emptyset$ and $\rho|_{G_{k_v}}$ is unramified for $v\not\in S\cup \Sigma_p$. We 
make the following assumptions on $\rho $:
\begin{assumption}
    \label{ass:appendix_main}
    \leavevmode
    \begin{enumerate}
    \item \label{ass_B1_noresidual}
    $H ^ 0 (k,\adjoint ^ 0\overline\rho) = H ^ 0 (k,\adjoint ^ 0\overline\rho (1)) = 0 $.
    \item\label{ass_B2_odd}
    $\rho $
is odd in the sense of \cite[Definition 1.2]{fakhruddin2021relative}.

\item \label{ass_B3_HT}
For all primes $v\in\Sigma_p $, $\rho |_{G_{k_v}} $
is potentially semistable with regular Hodge-Tate cocharacter $\mu_v:\mathbb G_m \to G $.
    \end{enumerate}
\end{assumption}

\begin{notation}\label{notation:appendix_def_rings}
 Recall the category $\CNL_O$ from (\ref{subsubsec:CNLO_notation}).
Let $\mu: G \to H= G/G^\derived$ be the maximal abelian quotient of $G$, and let 
$$\chi: G_k\xrightarrow {\rho} G (O)\xrightarrow{\mu} H(O)$$
be the multiplier character of $\rho $.
For all primes $v $
of $k $, let $\mathcal D_v$ be the functor on $\CNL_O$ defined by 
\begin{equation}
    \mathcal D_v(A) = \set{\rho_A: G_{k_v} \to G(A)\,:\, \rho_A \otimes_A (O/\varpi) = \overline \rho,\, \mu\circ \rho_A = \chi}. 
\end{equation}
The functor $\mathcal D_v$ is represented by a universal deformation ring $\widetilde R_v$. For $v\in \Sigma_p$, 
let $R_v$ be the quotient of $\widetilde R_v$ corresponding to potentially semistable deformations with  fixed Hodge type $\mu_v $ \cite[Proposition 3.0.12]{balaji2012Gvalued}; for $v\not\in \Sigma_p$,  set $R_v \coloneqq \widetilde R_v$. 
\end{notation}

Our final assumption on $\rho $
is:
\begin{assumption}\label{ass:appendix_smooth}
    For all primes $v\in S\cup\Sigma_p $,  the point $y_v$ of $\Spec R_v[1/\varpi]$ defined by
$\rho|_{G_{k_v}}: G_{k_v}\rightarrow G (O)\rightarrow G (E) $
is  formally smooth.
    
\end{assumption}

In particular, Assumption \ref{ass:appendix_smooth} implies that $y_v $
lies on a unique irreducible component of $\Spec R_v [1/\varpi] $. Let $R_v\twoheadrightarrow\overline R_v $
be the quotient corresponding to the Zariski closure of this irreducible component.
We have the following simple criterion for Assumption \ref{ass:appendix_smooth} to hold:
\begin{lemma}\label{smoothness lemma appendix}
    For all $v$, $y_v$ is  a formally smooth  point of $\Spec R_v[1/\varpi]$ if and only if 
    $$H^0(\WD (\ad^0 \rho|_{G_{k_v}})(1)) = 0.$$
\end{lemma}
\begin{proof}
    This is immediate from \cite[Corollary 3.3.4]{bellovin2019Gvalued}; note that, in the notation of \emph{loc. cit.}, $\ad \WD(\rho|_{G_{k_v}})$ is by definition the Weil-Deligne representation associated to $\ad \rho|_{G_{k_v}}$, cf. \cite[\S2]{bellovin2019Gvalued}.
\end{proof}
\subsection{Selmer groups and relative deformation theory}
Recall that a \emph{Selmer structure} $\mathcal F$ for an $O[G_k]$-module $M$ is a collection of $O$-submodules (the ``local conditions'')
$$H^1_{\mathcal F}(k_v, M) \subset H^1(k_v, M)$$ for all nonarchimedean\footnote{Since $p \neq 2$, for all $v|\infty$ we have $H^1(k_v, M) = 0$.} places $v$ of $k$, such that $H^1_{\mathcal F}(k_v, M) = H^1_{\unr} (k_v, M)$ for all but finitely many $v$. The associated \emph{Selmer group} is
$$H^1_{\mathcal F}(k, M) = \ker\left(H^1(k, M) \to \prod_v \frac{H^1(k_v, M)}{H^1_{\mathcal F}(k_v, M)}\right).$$ 

If $M$ is finite and $M' = \Hom(M, (E/O)(1))$ is the Cartier dual, then $H^1(k_v, M)$ and $H^1(k_v, M')$ are dual under the local Tate pairing.
The \emph{dual Selmer structure} $\mathcal F^\ast$ to $\mathcal F$ is the Selmer structure for $M'$ defined by the orthogonal complement local conditions.

\begin{prop}\label{proposition making selmer conditions deformation appendix}
Let $v $
be a nonarchimedean place of $k $. There exists a nonempty open set $Y_v\subset\Spec\overline R_v (O) $
containing the point corresponding to $\rho_v $, and a collection of submodules $Z_{r, v}\subset Z ^ 1 (G_{k_v},\adjoint ^ 0\rho_r) $
with the following properties.
\begin{enumerate}
\item\label{proposition making selmer conditions deformation appendix part I} $Z_{r, v} $ is free
over $O/\varpi ^ r $
of rank $\dimension\Spec R_v [1/\varpi] $ ($=d_G$ if $v\not\in \Sigma_p$).
\item\label{proposition making selmer conditions deformation appendix part II}  Let $Y_{n} ^ v $
be the image of $Y_v $
in $\Spec R_v (O/\varpi ^ n) $
and denote by $\phi ^ {Y_v}_{n, r}: Y ^ v_{n + r}\rightarrow Y ^ v_n $
the reduction maps for $n, r\geq 1 $.
Then given $r_0\geq 1 $, there exists $n_0\geq 1 $
such that, for all $n\geq n_0 $
and all $0\leq r\leq r_0 $, the fibers of $\phi ^ {Y_v}_{n, r} $
are nonempty principal homogeneous spaces for $Z_{r, v} $.
\item \label{proposition making selmer conditions deformation appendix part III} The natural $O $-module maps $\adjoint ^ 0\rho_r\twoheadrightarrow\adjoint ^ 0\rho_{r -1} $
and $\adjoint ^ 0\rho_{r -1}\hookrightarrow\adjoint ^ 0\rho_r $
induce surjections $Z_{r, v}\twoheadrightarrow Z_{r -1, v} $
and inclusions $Z_{r -1, v}\hookrightarrow Z_{r, v} $.
\item\label{proposition making selmer conditions deformation appendix part IV}  $Z_{r, v} $
contains all coboundaries in $Z ^ 1 (G_{k_v},\adjoint ^ 0\rho_r) $.
\end {enumerate}
\end{prop}
\begin{proof}
See \cite[Proposition 4.7]{fakhruddin2021relative}. 
\end{proof}
\begin{rmk}
Although $Y_v\subset\Spec\overline R_v (O) $
is not uniquely determined by the properties in Proposition \ref{proposition making selmer conditions deformation appendix}, the property (\ref{proposition making selmer conditions deformation appendix part II}) shows that $Z_{r, v} $
depends only on $\rho |_{G_{k_v}} $ (by considering the fiber over $\rho_n|_{G_{k_v}}$).
\end{rmk}

\begin{definition}
For all $n \geq 1$, we define a Selmer structure $\mathcal F$ for $\ad^0\rho_n$ by
$$H^1_{\mathcal F}(k_v, \ad^0 \rho_n) = \begin{cases} \im\left(Z_{n, v} \to H^1(k_v, \ad^0\rho_n) \right), & v\in S\cup \Sigma_p,\\ H^1_{\unr} (k_v, \ad^0\rho_n), & v\not\in S \cup \Sigma_p \cup \Sigma_\infty.\end{cases}$$
\end{definition}

    

\begin{notation}\label{notation:admissible_appendix}
Now suppose given a set $\mathfrak Q $
of finite primes $\q $
of $k $
called
\emph {admissible}, and, for each $\q\in\mathfrak Q $, a quotient $R_\q ^\ordinary $
of $R_\q $
with the following properties:
\begin{enumerate}
\item $R_\q ^\ordinary $
is formally smooth of dimension $d_G$.
\item $R_\q ^\ordinary $
is stable under the conjugation action by $$\kernel\left (G (O)\rightarrow G (O/\varpi)\right). $$
\end{enumerate}
We also suppose $\mathfrak Q\intersection (S\cup\Sigma_p) =\emptyset $.
\end{notation}
\begin{definition}\label{def:appendix_admissible_etc}
\leavevmode
\begin{enumerate}
    \item A lift $\tau_\q: G_{k_\q}\rightarrow G (A) $
of $\overline\rho |_{G_{k_\q}} $, for a complete local Noetherian $O $-algebra $A $, is called
\emph {ordinary} 
if the corresponding map $R_\q\rightarrow A $
factors through $R_\q ^\ordinary $. 
\item For an admissible prime $\q\in \mathfrak Q$, a global lift $\tau: G_k\rightarrow G (A) $
of $\overline\rho $
is called $\mathdutchbcal q $-ordinary if $\tau |_{G_{k_\q}} $
is ordinary.
\item\label{def:appendix_n_admissible_part} For $\q\in\mathfrak Q $, we say $\q $
is $n $-\emph {admissible}
if $\rho_n $
is $\q $-ordinary. 
\item\label{def:appendix_admissible_part_H1ord} If $\q $
is $n $-admissible, then we define $Z ^ 1_\ordinary (G_{k_\q},\adjoint ^ 0\rho_n)\subset Z^1(G_{k_\q}, \ad^0 \rho_n) $
as the relative tangent space to $\Spec R_\q ^\ordinary\otimes_O O/\varpi ^ n $
at the point corresponding to $\rho_n |_{G_{k_\q}} $. Let $$H^1_{\ord}({k_\q}, \ad^0\rho_n)= \image\left( Z^1_\ord(G_{k_\q}, \ad^0\rho_n)\rightarrow H^1({k_\q}, \ad^0\rho_n)\right),$$
and let $H^1_\ord({k_\q}, \ad^0\rho_n(1))\subset H^1(k_\q, \ad^0\rho_n(1))$ be the orthogonal complement of $H^1_\ord(k_\q, \ad^0\rho_n)$.
\item 
If $\mathdutchbcal Q $
is a finite set of $n $-admissible primes, then we define a Selmer structure $\mathcal F (\mathdutchbcal Q) $
for $\adjoint ^ 0\rho_n $
by
$$H ^ 1_{\mathcal F (\mathdutchbcal Q)} (k_v,\adjoint ^ 0\rho_n) =\begin{cases} H ^ 1_{\mathcal F} (k_v,\adjoint ^ 0\rho_n), & v\not\in\mathdutchbcal Q,\\H ^ 1_\ordinary (k_\q,\adjoint ^ 0\rho_n), & v =\q\in\mathdutchbcal Q.\end{cases}$$
\item If $\mathdutchbcal Q$ is a finite set of $n$-admissible primes, then we define the \emph{relative Selmer groups} by
$$\overline\summer_{\mathcal F (\mathdutchbcal Q)} (k,\adjoint ^ 0\rho_n) =\image\left (\Selmer_{\mathcal F (\mathdutchbcal Q)} (k,\adjoint ^ 0\rho_n)\rightarrow\summer_{\mathcal F (\mathdutchbcal Q)} (k,\adjoint ^ 0\overline\rho)\right)
$$and, dually,
$$\overline\summer_{\mathcal F (\mathdutchbcal Q) ^\ast} (k,\adjoint ^ 0\rho_n (1)) =\image\left (\summer_{\mathcal F (\mathdutchbcal Q) ^\ast} (k,\adjoint ^ 0\rho_n (1))\rightarrow\Selmer_{\mathcal F (\mathdutchbcal Q) ^\ast} (k,\adjoint ^ 0\overline\rho (1))\right).
$$\end{enumerate}
    
\end{definition}

\begin{prop}\label{prop:appendix_ord}
Suppose $\q$ is $n$-admissible. 
    \begin{enumerate}
        \item\label{prop:appendix_ord_free} $Z ^ 1_\ordinary (G_{k_\q},\adjoint ^ 0\rho_n) $
is  free of rank $d_G$
over $O/\varpi ^ n $
and contains all coboundaries.
\item\label{prop:appendix_ord_compatible} For all $1< r \leq n$, the natural maps $\ad^0\rho_r \twoheadrightarrow \ad^0\rho_{r-1}$ and $\ad^0 \rho_{r-1} \hookrightarrow \ad^0 \rho_r$ induce surjections $Z^1_\ordinary(G_{k_\q}, \adjoint^0 \rho_r) \twoheadrightarrow Z^1_\ordinary(G_{k_\q}, \ad^0 \rho_{r-1})$ and injections $Z^1_\ordinary(G_{k_\q}, \adjoint^0 \rho_{r-1}) \hookrightarrow Z^1_\ordinary(G_{k_\q}, \ad^0 \rho_{r})$.
\item\label{prop:appendix_ord_fibers} Let $Y_{\q,n,\ord}\subset \Spec R_\q^\ord (O)$ be the set of points reducing to $\rho_n|_{G_{k_\q}}$ modulo $\varpi^n$, and let $Y_{m,n,\ord}^{\q}$ be the image in $\Spec R_\q^\ord(O/\varpi^m)$ for all $m \geq n$. Then for any $1\leq r\leq n$, the fibers of $Y_{m+r,n,\ord}^{\q} \to Y_{m,n,\ord}^{\q}$ are nonempty principal homogeneous spaces over $Z^1_\ord(G_{k_\q}, \ad^0 \rho_r).$ 
    \end{enumerate}
\end{prop}
\begin{proof}
    Parts (\ref{prop:appendix_ord_free}) and (\ref{prop:appendix_ord_fibers}) are immediate from the conditions on $R_\q^\ord$ in Notation \ref{notation:admissible_appendix}. For (\ref{prop:appendix_ord_compatible}), it is clear from the definition that $$Z^1_\ord(G_{k_\q}, \ad^0 \rho_{r-1})\supseteq \image \left(Z^1_\ord(G_{k_\q}, \ad^0 \rho_r) \to Z^1_\ord(G_{k_\q}, \ad^0 \rho_{r-1})\right),$$ and equality holds by (\ref{prop:appendix_ord_free}) and counting. A similar argument shows the compatibility with $$Z^1 (G_{k_\q}, \ad^0 \rho_{r-1})\hookrightarrow Z^1(G_{k_\q}, \ad^0 \rho_{r}).$$
\end{proof}

\begin {lemma}\label{stuff like lemma 6.1 FKP} Let $\mathdutchbcal Q$ be a finite set of $n$-admissible primes. Then:
\begin{enumerate}
    \item\label{lemma 6.1 FKP} For all $a, b\geq 0$ with $a+b\leq n$, there are natural exact sequences
    $$0 \to \Selmer_{\mathcal F(\mathdutchbcal Q)}(k, \adjoint^0\rho_a)\to \Selmer_{\mathcal F(\mathdutchbcal Q)}(k, \adjoint^0\rho_{a+b}) \to \Selmer_{\mathcal F(\mathdutchbcal Q)}(k, \adjoint^0\rho_b)$$
    and
     $$0 \to \Selmer_{\mathcal F(\mathdutchbcal Q)^\ast}(k, \adjoint^0\rho_a(1))\to \Selmer_{\mathcal F(\mathdutchbcal Q)^\ast}(k, \adjoint^0\rho_{a+b}(1)) \to \Selmer_{\mathcal F(\mathdutchbcal Q)^\ast}(k, \adjoint^0\rho_b(1)).$$
    \item \label{lemma 6.1 part III} The exact sequences in (\ref{lemma 6.1 FKP}) identify $$\Sel_{\mathcal F(\mathdutchbcal Q)}(k, \ad^0 \rho_a) = \Sel_{\mathcal F(\mathdutchbcal Q)}(k, \ad^0\rho_n)[\varpi^a]$$ and $$\Sel_{\mathcal F(\mathdutchbcal Q)^\ast}(k, \ad^0 \rho_a(1)) = \Sel_{\mathcal F(\mathdutchbcal Q)^\ast}(k, \ad^0\rho_n(1))[\varpi^a]$$  for all $a \leq n$.
    \item \label{lemma 6.1 part II residual selmer groups and torsion}If $\overline \Selmer_{\mathcal F(\mathdutchbcal Q)}(k, \adjoint^0\rho_m) = 0$ for some integer $m \leq n$, then for all $m'$ with $m -1\leq m'\leq n$,
the natural map  induces an isomorphism
   $$\Selmer_{\mathcal F(\mathdutchbcal Q)} (k, \adjoint^0\rho_{m-1})\xrightarrow{\sim} \Selmer_{\mathcal F(\mathdutchbcal Q)} (k, \adjoint^0\rho_{m'}).$$
\end{enumerate}
\end{lemma}
\begin{proof}
Using Proposition \ref{prop:appendix_ord}(\ref{prop:appendix_ord_free},\ref{prop:appendix_ord_compatible}), the first part is \cite[Lemma 6.1]{fakhruddin2021relative}, except for the injectivity on the left, which follows from Assumption \ref{ass:appendix_main}(\ref{ass_B1_noresidual}). 
Part (\ref{lemma 6.1 part III}) is a corollary of (\ref{lemma 6.1 FKP}), because the kernel of $\varpi^a: \Sel_{\mathcal F(\mathdutchbcal Q)}(k, \ad^0 \rho_n)\to \Sel_{\mathcal F(\mathdutchbcal Q)} (k, \ad^0 \rho_n)$ coincides with the kernel of $\Sel_{\mathcal F(\mathdutchbcal Q)}(k, \ad^0 \rho_n)\to \Sel_{\mathcal F(\mathdutchbcal Q)} (k, \ad^0 \rho_{n-a})$, and likewise for dual Selmer groups. 
So we show (\ref{lemma 6.1 part II residual selmer groups and torsion}). For any $m'$ with $m-1 < m' \leq n$, the map $$\Selmer_{\mathcal F(\mathdutchbcal Q)} (k, \adjoint^0\rho_{m'}) \to \Selmer_{\mathcal F(\mathdutchbcal Q)} (k, \overline\adjoint^0\rho)$$
factors through $\overline \Selmer_{\mathcal F(\mathdutchbcal Q)} (k, \adjoint^0\rho_m)$, hence vanishes; in particular, we have an isomorphism $$\Selmer_{\mathcal F(\mathdutchbcal Q)} (k, \adjoint^0\rho_{m'-1})\isomorphism\Selmer_{\mathcal F(\mathdutchbcal Q)} (k, \adjoint^0\rho_{m'})$$
by (\ref{lemma 6.1 FKP}). The claim follows by  downwards induction on $m'$. 

\end{proof}
\begin {lemma} \label{balanced residual selmer ranks lemma}Suppose $\mathdutchbcal Q $
is a finite set of $n $-admissible primes. Then for all $m\leq n $,
$$\dimension_{O/\varpi}\overline\summer_{\mathcal F (\mathdutchbcal Q)} (k,\adjoint ^ 0\rho_m) =\dimension_{O/\varpi}\overline\summer_{\mathcal F (\mathdutchbcal Q) ^\ast} (k,\adjoint ^ 0\rho_m). $$
\end{lemma}
\begin{proof}
This is \cite[Lemma 6.3]{fakhruddin2021relative}. Note that the local conditions are balanced in the sense of
\emph {loc. cit.}: for $v\not\in\mathdutchbcal Q $, this is \cite[Proposition 4.7(3)]{fakhruddin2021relative}, and for $\q\in\mathdutchbcal Q $
the same calculation applies because by Proposition \ref{prop:appendix_ord}(\ref{prop:appendix_ord_free}).
\end{proof}
\begin{rmk}
    The proof of \cite[Lemma 6.3]{fakhruddin2021relative} uses Assumption  \ref{ass:appendix_main}(\ref{ass_B1_noresidual}).
\end{rmk}
\begin{definition}For any finite set of primes $\mathdutchbcal Q$ disjoint from $S\cup \Sigma_p$, we define the Shafarevich-Tate groups:
\begin{equation}
\Sha ^ 2_{\mathdutchbcal Q} (\adjoint ^ 0\rho_n)\coloneqq\kernel\left (H ^ 2 (k^{S\cup\Sigma_p\cup\mathdutchbcal Q}/k,\adjoint ^ 0\rho_n)\rightarrow\product_{v\in S \cup\Sigma_p\cup\mathdutchbcal Q} H ^ 2 (k_v,\adjoint ^ 0\rho_n)\right)
\end{equation}
and
\begin{equation}
\Sha_{\mathdutchbcal Q} ^ 1 (\adjoint ^ 0\rho_n (1))\coloneqq\kernel\left (H ^ 1 (k^{S\cup\Sigma_p\cup\mathdutchbcal Q}/k,\adjoint ^ 0\rho_n (1))\rightarrow\product_{v\in S\cup\Sigma_p\cup\mathdutchbcal Q} H ^ 1 (k_v,\adjoint ^ 0\rho_n (1))\right),
\end{equation}
for all $n\geq 1 $.
\end{definition}
\begin{lemma}\label{residual obstructions vanish mod n lemma}
Suppose given a finite set $\mathdutchbcal Q $
of $n $-admissible primes such that $$\overline\summer_{\mathcal F (\mathdutchbcal Q)} (k,\adjoint ^ 0\rho_n) = 0. $$
Then the natural map $$\Sha ^ 2_{\mathdutchbcal Q} (\adjoint ^ 0\overline\rho)\rightarrow\Sha ^ 2 _{\mathdutchbcal Q}(\adjoint ^ 0\rho_n) $$
is identically zero.
\end{lemma}
\begin{proof}
In the commutative diagram
\begin{center}
\begin{tikzcd}
\Sha ^ 1_{\mathdutchbcal Q} (\adjoint ^ 0\rho_n (1))\arrow [d, hook]\arrow [r] &\Sha ^ 1_{\mathdutchbcal Q} (\adjoint ^ 0\overline\rho (1))\arrow [d, hook]\\\Sel_{\mathcal F (\mathdutchbcal Q) ^\ast} (k,\adjoint ^ 0\rho_n (1))\arrow [r] & \Sel_{\mathcal F (\mathdutchbcal Q) ^\ast} (k,\adjoint ^ 0\overline\rho (1)),
\end{tikzcd}
\end{center}
the bottom map is identically zero because $\overline\summer_{\mathcal F (\mathdutchbcal Q) ^\ast} (k,\adjoint ^ 0\rho_n (1)) = 0 $
by Lemma \ref{balanced residual selmer ranks lemma}. Hence the top map is identically zero as well. But by global Poitou-Tate duality, the top map is canonically dual to the map
$$\Sha_{\mathdutchbcal Q} ^ 2 (\adjoint ^ 0\overline\rho)\rightarrow\Sha_{\mathdutchbcal Q} ^ 2 (\adjoint ^ 0\overline\rho_n), $$
so this shows the lemma.
\end{proof}
\begin{theorem} [Fakhruddin-Khare-Patrikis, \cite{fakhruddin2021relative}]
\label{main appendix theorem}Fix $m\geq 1 $. Then there exists a constant $n_0 = n_0 (m,\rho)\geq 1 $
with the following property. For any $N\geq n_0 $, if $\mathdutchbcal Q $
is a finite set of $N $-admissible primes and $$\overline\summer_{\mathcal F (\mathdutchbcal Q)} (k,\adjoint ^ 0\rho_m) = 0, $$
then there exists a representation
$$\rho ^ {\mathdutchbcal Q}: G_k\rightarrow G (O) $$
satisfying:
\begin{enumerate}
\item\label{thm:appendix_main_one} $\rho\equiv\rho ^ {\mathdutchbcal Q}\pmod {\varpi ^ {N - m +1}}. $
\item\label{thm:appendix_main_two} $\rho ^ {\mathdutchbcal Q} $
is unramified outside $S\cup\Sigma_p\cup\mathdutchbcal Q $.
\item \label{thm:appendix_main_three}For all $v\in S\cup\Sigma_p $, the points of $\Spec R_v [1/\varpi] $
corresponding to $\rho ^ {\mathdutchbcal Q} |_{G_{k_v}} $
and $\rho |_{G_{k_v}} $
lie on the same irreducible component.
\item\label{thm:appendix_main_four} For all $\mathdutchbcal q\in\mathdutchbcal Q $, $\rho ^ {\mathdutchbcal Q} $
is $\mathdutchbcal q $-ordinary.
\end{enumerate}
\end{theorem}
\begin{proof}
For all primes $v\in S\cup\Sigma_p $, apply Proposition \ref{proposition making selmer conditions deformation appendix} with $r_0 = m$
to obtain an integer $n_{0, v} $ and a subset $Y_v\subset \overline R_v(O)$. Then let $$n_0 =\Max\set {\Max_{v\in S\cup\Sigma_p}\set {n_{0, v}}+m, 2 m}. $$

We construct $\rho ^ {\mathdutchbcal Q} $
as the inverse limit of representations $\rho ^ {\mathdutchbcal Q}_n: G_k\rightarrow G (O) $, compatible under reduction maps, with the following properties for all $n\geq 1 $:
\begin{enumerate}[(i)]
\item For all $v\not\in S\cup\Sigma_p \cup \mathdutchbcal Q $, $\rho_n ^ {\mathdutchbcal Q} |_{G_{k_v}} $
is unramified.
\item For $\q\in\mathdutchbcal Q $, $\rho_n ^ {\mathdutchbcal Q}|_{G_{k_\q}} $
is ordinary.
\item For $v\in S\cup\Sigma_p $, $\rho_n ^ {\mathdutchbcal Q} |_{G_{k_v}} $
lies in the set $Y_n ^ v $ (cf. Proposition \ref{proposition making selmer conditions deformation appendix}(\ref{proposition making selmer conditions deformation appendix part II})).
\end{enumerate}
This suffices because $Y_v\subset\Spec \overline R_v (O)\subset \Spec R_v(O) $ and $\Spec R_\q^\ord (O)\subset \Spec R_\q(O)$ are both closed in the $\varpi$-adic topology, for all  $v\in S\cup \Sigma_p$ and $\q\in \mathdutchbcal Q$.
The representations $\rho_n ^ {\mathdutchbcal Q} $
are constructed inductively,
\emph {but}, when constructing $\rho_{n +1} ^ {\mathdutchbcal Q} $, we will allow ourselves to modify the representations $$\rho_{n - m +2} ^ {\mathdutchbcal Q},\cdots,\rho ^ {\mathdutchbcal Q}_{n -1},\rho ^ {\mathdutchbcal Q}_n. $$

The base case of the induction is $\rho_n ^ {\mathdutchbcal Q} \equiv\rho_n $
for $n\leq N $.

For the inductive step, first fix local lifts
$\rho_{n +1, v} ^ {\mathdutchbcal Q} $
of $\rho ^ {\mathdutchbcal Q}_{n -  m +1} |_{G_{k_v}} $
for $v\in S\cup\Sigma_p\cup\mathdutchbcal Q $, with the following property: if $v\in S\cup\Sigma_p $, then $\rho_{n +1, v} ^ {\mathdutchbcal Q} $
lies in $Y ^ v_{n +1} $, and if $v =\q\in\mathdutchbcal Q $, then $\rho_{n +1,\q} ^ {\mathdutchbcal Q} $
lies in $\Spec R ^\ordinary_\q (O/\varpi ^ {n +1}) $. Such choices are possible by Proposition \ref{proposition making selmer conditions deformation appendix}(\ref{proposition making selmer conditions deformation appendix part II}) and the formal smoothness of $\Spec R^\ordinary_\q$. Now let $c\in H ^ 2 (k^{S\cup \Sigma_p\cup \mathdutchbcal Q}/k,\adjoint ^ 0\overline\rho) $
denote the obstruction class defined by choosing a set-theoretic lift $$\widetilde\rho_{n +1} ^ {\mathdutchbcal Q}: G_{k,S\cup\Sigma_p\cup\mathdutchbcal Q} \rightarrow G (O/\varpi ^ {n +1}) $$
of $\rho_n ^ {\mathdutchbcal Q} $; we have 
\begin{equation}
    c\in \Sha_{\mathdutchbcal Q} ^ 2 (\adjoint ^ 0\overline\rho)
\end{equation}
since the local lifts $\rho_{n +1, v} ^ {\mathdutchbcal Q} $
exist for all $v\in S\cup\Sigma_p\cup\mathdutchbcal Q $. 

Then by Lemma \ref{residual obstructions vanish mod n lemma}, $c $
has trivial image $c_m $
in $\Sha_{\mathdutchbcal Q} ^ 2 (\adjoint ^ 0\rho_m) $. But, because $n - m +1\geq N - m +1\geq m $, $c_m$
is precisely the obstruction to lifting $\rho ^ {\mathdutchbcal Q}_{n - m +1} $
modulo $\varpi ^ {n +1} $. Thus we may choose a lift $$\widetilde\rho_{n +1} ^ {\mathdutchbcal Q}: G_{k,S\cup\Sigma_p\cup  \cup \mathdutchbcal Q}\rightarrow G (O/\varpi ^ {n +1}) $$
with $$\widetilde\rho_{n +1} ^ {\mathdutchbcal Q}\equiv\rho ^ {\mathdutchbcal Q}_{n - m +1}\pmod {\varpi ^ {n - m +1}}. $$
For each place $v \in S\cup\mathdutchbcal Q \cup\Sigma_p $, comparing $\widetilde \rho_{n+1}^ {\mathdutchbcal Q} |_{G_{k_v}} $
to $\rho ^ {\mathdutchbcal Q}_{n +1, v} $
as lifts of $\rho ^ {\mathdutchbcal Q}_{n - m +1} |_{G_{k_v}} $ 
produces a collection of local classes
$$(f_v)\in\prod_{v\in S\cup\mathdutchbcal Q\cup\Sigma_p}\frac {H ^ 1 (k_v,\adjoint ^ 0\rho_m)} {H ^ 1_{\mathcal F (\mathdutchbcal Q)} (k_v,\adjoint ^ 0\rho_m)}. $$
Now consider the commutative diagram with exact rows coming from the Poitou-Tate long exact sequence:
\begin{equation}
\tag{$\ast$}
\begin{tikzcd}
H ^ 1 (k ^ {S\cup\mathdutchbcal Q\cup\Sigma_p}/k,\adjoint ^ 0\rho_m)\arrow [d]\arrow [r, "\local_{m}"] &\Prod_{v\in S\cup\mathdutchbcal Q\cup\Sigma_p}\frac {H ^ 1 (k_v,\adjoint ^ 0\rho_m)} {H ^ 1_{\mathcal F (\mathdutchbcal Q)} (k_v,\adjoint ^ 0\rho_m)}\arrow [r]\arrow [d] &\summer_{\mathcal F (\mathdutchbcal Q)^\ast} (k,\adjoint ^ 0\rho_m (1)) ^\check\arrow [d, hook]\\
H ^ 1 (k ^ {S\cup\mathdutchbcal Q\cup\Sigma_p}/k,\adjoint ^ 0\rho_{m -1})\arrow [r, "\local_{m -1}"] &\Prod_{v\in S\cup\mathdutchbcal Q\cup\Sigma_p}\frac {H ^ 1 (k_v,\adjoint ^ 0\rho_{m -1})} {H ^ 1_{\mathcal F (\mathdutchbcal Q)} (k_v,\adjoint ^ 0\rho_{m -1})}\arrow [r] &\summer_{\mathcal F (\mathdutchbcal Q)^\ast} (k,\adjoint ^ 0\rho_{m -1} (1)) ^\check,
\end{tikzcd}
\end{equation}
where the superscript $\check $
denotes Pontryagin duality. The injectivity of the rightmost map, or equivalently the surjectivity of its dual $\summer_{\mathcal F (\mathdutchbcal Q)^\ast} (k,\adjoint ^ 0\rho_{m -1} (1))\rightarrow\summer_{\mathcal F (\mathdutchbcal Q)^\ast} (k,\adjoint ^ 0\rho_m (1)) $,  follows from Lemma \ref{stuff like lemma 6.1 FKP}(\ref{lemma 6.1 FKP}),   Lemma \ref{balanced residual selmer ranks lemma}, and  the vanishing of $\overline\summer_{\mathcal F (\mathdutchbcal Q)^\ast} (k,\adjoint ^ 0\rho_m (1)) $.

Our next claim is that the image of $(f_v) $
under the central vertical map of  ($\ast$) lies in the image of $\local_{m -1} $. Indeed, because $\rho ^ {\mathdutchbcal Q}_n $
satisfies (i)-(iii) above, it follows from Propositions \ref{proposition making selmer conditions deformation appendix}(\ref{proposition making selmer conditions deformation appendix part II}) and \ref{prop:appendix_ord}(\ref{prop:appendix_ord_fibers}) that the image of $(f_v) $
coincides with the image of the global cocycle formed by comparing $\widetilde\rho ^ {\mathdutchbcal Q}_{n +1}\pmod {\varpi ^ {n}} $
and $\rho ^ {\mathdutchbcal Q}_n $
as lifts of $\rho ^ {\mathdutchbcal Q}_{n - m +1} $. Then because the rows of ($\ast$) are exact, we conclude that $(f_v) $
lies in the image of $\local_m$. Picking a preimage, we may then modify $\widetilde\rho ^ {\mathdutchbcal Q}_{n +1} $
to a representation $\rho'_{n +1}: G_k\rightarrow G (O/\varpi ^ {n +1}) $
with $$\rho'_{n +1}\equiv\rho ^ {\mathdutchbcal Q}_{n - m +1}\pmod {\varpi ^ {n - m +1}}. $$
Properties (i)-(iii) hold for $\rho_{n+1}'$ 
by Propositions \ref{proposition making selmer conditions deformation appendix}(\ref{proposition making selmer conditions deformation appendix part II}), \ref{prop:appendix_ord}(\ref{prop:appendix_ord_fibers}) again.
To complete the inductive step, we set $\rho ^ {\mathdutchbcal Q}_{n+ 1}\coloneqq\rho_{n +1}' $, and
\emph {relabel}
$\rho ^ {\mathdutchbcal Q}_{n - m +2},\cdots,\rho ^ {\mathdutchbcal Q}_n $
to be the reductions of $\rho'_{n +1} $.
\end{proof}
\subsection{Relation to Bloch-Kato Selmer groups}
\begin{notation}
    Let $\Sigma_p ^\crystal \subset \Sigma_p$ be the set of places $v|p$ of $k$ such that $\rho|_{G_{k_v}}$ is crystalline. For $v\in \Sigma_p^\crystal$, we write $R_v^\crystal$ for the crystalline quotient of $R_v$ (constructed in \cite{balaji2012Gvalued}).
\end{notation}
\begin{rmk}
Recall that $\Spec R_v^\crystal[1/\varpi]$ is a union of irreducible components of $\Spec R_v[1/\varpi]$; in particular, $\overline R_v$ is a quotient of $R_v^\crystal$, and $\Spec R_v^\crystal[1/\varpi]$ is equidimensional of the same dimension as $\Spec R_v[1/\varpi]$. 
\end{rmk}
\begin{definition}
For all finite places $v$ of $k$, and all $r \geq 0$,  define $$Z^{\rel}_{r,v} \subset Z^1(k_v, \ad^0 \rho_r)$$ to be the subspace of cocycles $c$ corresponding to lifts $\rho_c: G_{k_v} \to G(O[\epsilon]/(\epsilon^2, \varpi^r \epsilon))$ such that the corresponding map $f_c: \widetilde R_v \to O[\epsilon]/(\epsilon^2, \varpi^r\epsilon)$ factors through $R_v$ (resp. $R_v^\crystal$) if $v\not\in \Sigma_p^\crystal$ (resp. $v\in \Sigma_p^\crystal$).
In particular, $Z^{\rel}_{r,v} = Z^1(k_v, \ad^0 \rho_r)$ if $v\not\in \Sigma_p$.
\end{definition}
\begin{prop}\label{prop identify Z and relaxed new}
    Fix a place $v$ of $k$.
    \begin{enumerate}
        \item\label{prop identify Z and relaxed new part one} For all $r\geq 1$, we have $Z_{r,v} \subset Z^\rel_{r,v}.$
        \item \label{prop identify Z and relaxed new part two}The cardinality of $Z^{\rel}_{r,v}/ Z_{r,v}$ is uniformly bounded in $r$.
    \end{enumerate}
\end{prop}
\begin{proof}
Let $I\subset \widetilde R_v$ be the kernel of the map to $R_v$ (resp. $R_v^\crystal$) if $v\not\in \Sigma_p^\crystal$ (resp. $v\in \Sigma_p^\crystal$). 
For (\ref{prop identify Z and relaxed new part one}), suppose given a cocycle $c\in Z_{r,v}$ which corresponds to a lift $$f_c:\widetilde R_v \to 
O[\epsilon]/(\epsilon^2, \varpi^r\epsilon)$$
of the map $f: \widetilde R_v \to O$ determined by $\rho|_{G_{k_v}}$. Because $c\in Z_{r,v}$, for all $n$ sufficiently large the map 
$$\widetilde R_v \xrightarrow{f_c}O[\epsilon]/(\epsilon^2, \varpi^r\epsilon)\xrightarrow{\epsilon\mapsto \varpi^n} O/\varpi^{n+r}$$
factors through $\widetilde R_v/I$; hence $$f_c(I) \subset (\varpi^n - \epsilon) \cap (\epsilon)\text{ in }O[\epsilon]/(\epsilon^2,\varpi^r \epsilon)$$
for all $n$ sufficiently large. We conclude $f_c(I) = 0$, hence $c$ lies in $Z^\rel_{r,v}$, which proves (\ref{prop identify Z and relaxed new part one}). 

For (\ref{prop identify Z and relaxed new part two}), let $\mathfrak p \subset \widetilde R_v$ be the kernel of $f$, and note that
$Z^\rel_{r,v}$ is canonically identified with $$\Hom_O (\mathfrak p/(\mathfrak p ^2, I), O/\varpi^r).$$
Since $\widetilde R_v$ is Noetherian, $\mathfrak p/(\mathfrak p ^2, I)$ is a finitely-generated $O$-module, hence
$$\lg_O Z^\rel_{r,v} = r \cdot \rank_O (\mathfrak p/(\mathfrak p^2, I)) + \textrm{O}(1)$$ as $r$ varies. But because $\rho|_{G_{k_v}}$ defines a formally smooth point of $\Spec R_v[1/\varpi]$, the $O$-rank of $\mathfrak p /(\mathfrak p^2, I)$ is also $\dim \Spec R_v[1/\varpi]$ (which equals $\dim \Spec R_v^\crystal [1/\varpi]$ for $v\in \Sigma_p^\crystal$). Moreover $$\lg_O Z_{r,v} = r \dim \Spec R_v[1/\varpi]$$ by Proposition \ref{proposition making selmer conditions deformation appendix}(\ref{proposition making selmer conditions deformation appendix part I}), so (\ref{prop identify Z and relaxed new part two}) follows. 

\end{proof}


\begin{prop}\label{prop compare BK and FKP Selmer}
    For all places $v$ of $k$,
    let $$H^1_{\mathcal F}(k_v, \ad^0 \rho) = \varprojlim_n H^1_{\mathcal F}(k_v, \ad^0 \rho_n) \subset H^1(k_v, \ad^0 \rho).$$
   Then  we have 
    $$H^1_{\mathcal F}(k_v, \ad^0 \rho)\otimes _O E = H^1_f(k_v, \ad^0\rho\otimes_O E).$$
    \end{prop}
    \begin{proof}
Suppose first that $v\not\in \Sigma_p$. Then $$\dim H^1(k_v, \ad^0\rho\otimes_O E) - \dim H^1_f(k_v, \ad^0\rho\otimes_O E) = \dim H^0 (k_v, \ad^0 \rho (1)\otimes_OE)= 0$$
by the local Euler characteristic formula, local duality, and Lemma \ref{smoothness lemma appendix} (under Assumption \ref{ass:appendix_smooth}). By Proposition \ref{proposition making selmer conditions deformation appendix}(\ref{proposition making selmer conditions deformation appendix part I},\ref{proposition making selmer conditions deformation appendix part III}), we also have  \begin{equation*}
    \begin{split}
        \dim H^1_{\mathcal F}(k_v, \ad^0 \rho) \otimes_O E&= d_G  - \dim H^0(k_v, \ad^0 \rho\otimes_O E)  \\
        &=   \dim_E \Lie G^\derived\otimes_O E  - \dim H^0(k_v, \ad^0 \rho\otimes_O E)     \\  
        &= \dim H^1_f (k_v, \ad^0 \rho\otimes_O E),     
    \end{split}
\end{equation*}and the proposition follows.

Now we consider the case $v\in \Sigma_p$. By Proposition \ref{proposition making selmer conditions deformation appendix}(\ref{proposition making selmer conditions deformation appendix part I}), we have \begin{align*} \dim H^1_{\mathcal F} (k_v, \ad^0 \rho) \otimes_O E &= \dim \Spec R_v[1/\varpi] + \dim \ad^0 \rho\otimes_O E -  \dim H^0(k_v, \ad^0\rho\otimes_O E) \\&= \dim D_\dR (\ad^0\rho\otimes_O E)/\Fil^0 D_\dR(\ad^0 \rho\otimes_O E) - \dim H^0(k_v, \ad^0\rho\otimes_O E),\end{align*}
where the latter equality is by the proof of \cite[Theorem 3.3.2]{bellovin2019Gvalued} and by Assumption  \ref{ass:appendix_smooth}. In particular, by \cite[Corollary 3.8.4]{bloch1990lfunctions}, we have $$\dim H^1_{\mathcal F}(k_v, \ad^0 \rho)\otimes_O E = \dim H^1_f(k_v, \ad^0 \rho\otimes_O E).$$ 
It therefore suffices to show that $H^1_{f} (k_v, \ad^0\rho\otimes_O E) \subset H^1_{\mathcal F}(k_v, \ad^0 \rho)\otimes_O E$.

By
Proposition \ref{prop identify Z and relaxed new}, we have $$(\varprojlim_r Z_{r,v})\otimes_OE = (\varprojlim_r Z_{r,v}^\rel)\otimes_O E \subset Z^1(k_v, \ad^0\rho\otimes_O E).$$ In particular, by the definition of $Z^\rel_{r,v}$, $H^1_{\mathcal F}(k_v, \ad^0 \rho)\otimes_O E$ consists of cocycles $c$ such that the corresponding $G$-valued   deformation $\rho_c$ of $\rho$ 
to $E[\epsilon]/(\epsilon^2)$ is  crystalline (resp. potentially semistable) with Hodge type $\mu_v$ if $v\in  \Sigma_p^\crystal$ (resp. $v\in \Sigma_p - \Sigma_p^\crystal$). 
By \cite[Proposition 2.3.2]{balaji2012Gvalued}, it suffices to check this condition for the representation $\sigma\circ\rho_{c}: G_{k_v} \to \GL_n(E[\epsilon]/(\epsilon^2))$ obtained by composing $\rho_c$ with any faithful algebraic representation $\sigma: G \to \GL_{n, E}$.

Consider the case when $v\in \Sigma_p - \Sigma_p^\crystal$. 
Suppose given a cocycle $c\in H^1_f(k_v, \ad^0 \rho\otimes_O E)$; \emph{a fortiori}, $c$ lies in the kernel of the map
$$H^1(k_v, \ad^0\rho\otimes_O E) \to H^1(k_v, \ad^0\rho\otimes_O  B_{\dR}),$$ hence the cocycle corresponding to the deformation $\sigma \circ\rho_{c}$ of $\sigma\circ \rho$ lies in the kernel of the map 
$$H^1(k_v, \ad^0(\sigma\circ\rho)) \to H^1(k_v, \ad^0(\sigma\circ\rho)\otimes B_{\dR}).$$ In particular, $\sigma\circ\rho_{c}$ is potentially semistable by the argument of \cite[Lemma 1.2.5]{allen2016deformations}, and this completes the proof. 
When $v\in \Sigma_p^\crystal$, an analogous argument applies, using that $c$ lies in the kernel of the map $$H^1(k_v, \ad^0\rho\otimes_O E) \to H^1(k_v, \ad^0\rho\otimes_O  B_{\crystal})$$
by definition of $H^1_f(k_v, \ad^0 \rho\otimes_O E)$. 
    \end{proof}
\begin{lemma}\label{lemma for using theorem of thorne}
    Suppose $H^1_f(k, \ad^0 \rho\otimes_O E) = 0$. Then for all $n$ sufficiently large,
    $$\overline \Sel_{\mathcal F}(k, \ad^0 \rho_n) = 0.$$ 
\end{lemma}
\begin{proof}
We first claim:
\begin{claim}
    We have $\Sel_{\mathcal F}(k, \ad^0 \rho) = \varprojlim_n \Sel_{\mathcal F}(k, \ad^0 \rho_n).$
\end{claim}
\begin{proof}[Proof of claim]
    Set $Z_{n, v} = Z^1_{\unr}(G_{k_v}, \ad^0 \rho_n)$ for all finite $v\not\in S\cup \Sigma_p $, and set $Z_v \coloneqq \varprojlim_n Z_{n,v}$ for all finite $v$. It follows from Proposition \ref{proposition making selmer conditions deformation appendix}(\ref{proposition making selmer conditions deformation appendix part I},\ref{proposition making selmer conditions deformation appendix part III}) and a direct calculation in the unramified case that  $Z^1(G_{k_v}, \ad^0 \rho)/Z_v$ is torsion-free and $Z_{n, v}$ is the image of the map $Z_v \to Z^1(G_{k_v}, \ad^0 \rho_n)$ for all $n \geq 1$. Then the claim follows from \cite[Lemma 3.7.1]{mazur2004kolyvaginsystems}.
\end{proof}
Now we return to the proof of the lemma. 
By Proposition \ref{prop compare BK and FKP Selmer}, the assumption $H^1_f(k, \ad^0 \rho\otimes_O E) = 0$ implies that $\Sel_{\mathcal F}(k, \ad^0\rho)$ is torsion, hence trivial because $H^1(k, \ad^0\rho)$ is $\varpi$-torsion-free by  Assumption \ref{ass:appendix_main}(\ref{ass_B1_noresidual}). 
   So by the claim, we have $$\varprojlim_n \Sel_{\mathcal F}(k, \ad^0\rho_n) = 0,$$ which implies $$\varprojlim_n \overline  \Sel_{\mathcal F}(k, \ad^0\rho_n) = \cap_n \overline  \Sel_{\mathcal F}(k, \ad^0\rho_n) = 0.$$ Hence $\overline  \Sel_{\mathcal F}(k, \ad^0\rho_n) = 0$ for $n$ sufficiently large. 
\end{proof}
\subsection{Controlling congruences for level-raised representations}

\begin{definition}\label{def:loc_const_appendix}
\begin{enumerate}
    \item     For each place $v$ of $k$, we define $$C_v \coloneqq \sup_r \#\left(Z_{r,v}^\rel/Z_{r,v}\right),$$ which is finite by Proposition \ref{prop identify Z and relaxed new}.
\item If $\mathdutchbcal Q$ is a finite set of $n$-admissible primes for some $n \geq 1$, define  a Selmer structure $\mathcal F(\mathdutchbcal Q)^\rel$ for $\ad^0\rho_n$ by
$$H ^ 1_{\mathcal F (\mathdutchbcal Q)^ \rel} (k_v,\adjoint ^ 0\rho_n) =\begin{cases}\im\left(Z^\rel_{n, v}\to H^1(k_v, \ad^0\rho_n)\right), & v\in S \cup \Sigma_p, \\ H^1_{\mathcal F(\mathdutchbcal Q)}(k_v, \ad^0\rho_n), &\text{otherwise.}\end{cases}$$
\end{enumerate}
\end{definition}

\begin{prop}\label{prop:const_relaxed}
Let $C_1 = \sum_{v\in S\cup \Sigma_p}C_v$, where $C_v$ is as in Definition \ref{def:loc_const_appendix}. 
   Then
    $$\lg_O \Sel_{\mathcal F(\mathdutchbcal Q)^ \rel}(k, \ad^0\rho_n) \leq \lg_O  \Sel_{\mathcal F(\mathdutchbcal Q)}(k, \ad^0\rho_n) + C_1$$
    for all $n \geq 1$ and all $n$-admissible $\mathdutchbcal Q$.
\end{prop}
\begin{proof}
This follows from the exactness of the sequence
\begin{equation}\label{relax the local conditions exact sequence in appendix} 0\rightarrow\summer_{\mathcal F (\mathdutchbcal Q)} (k,\adjoint ^ 0\rho_n)\rightarrow\summer_{\mathcal F ^ \rel (\mathdutchbcal Q)} (k,\adjoint ^ 0\rho_n)\rightarrow\product_{v\in S\cup \Sigma_p}\frac {H ^ 1_{\mathcal F^\rel} (k_ v,\adjoint ^ 0\rho_n)} {H ^ 1_{\mathcal F} (k_v,\adjoint ^ 0\rho_n)}.\end{equation}
\end{proof}

\begin{cor}\label{cor:relaxed_bound_appendix}
   Fix $c\geq 0$.  There exists a constant $C_2\geq 0$, depending only on $c$ and $\rho$, with the following property: for all $n \geq m-1\geq 0$ and all $n$-admissible $\mathdutchbcal Q$ with $|\mathdutchbcal Q| = c$, 
    $$\overline \Sel_{\mathcal F(\mathdutchbcal Q)}(k, \ad^0 \rho_m ) = 0 \implies \lg_O \Sel_{\mathcal F(\mathdutchbcal Q)^\rel} (k, \ad^0 \rho_n) \leq C_2  (m - 1) + C_1,$$
    where $C_1$ is the constant in Proposition \ref{prop:const_relaxed}. 
\end{cor}
\begin{proof}
    By Lemma \ref{stuff like lemma 6.1 FKP}(\ref{lemma 6.1 part II residual selmer groups and torsion}), if $\overline \Sel_{\mathcal F(\mathdutchbcal Q)}(k, \ad^0 \rho_m) = 0$ then $$\varpi^{m-1}\Sel_{\mathcal F(\mathdutchbcal Q)}(k, \ad^0\rho_n) = 0.$$ Hence
    \begin{equation}\label{eq:cor_relaxed_bound}\begin{split}\lg_O \Sel_{\mathcal F(\mathdutchbcal Q)} (k, \ad^0\rho_n) &\leq 
\left(\dim_{O/\varpi} \Sel_{\mathcal F(\mathdutchbcal Q)}(k, \ad^0\rho_n)[\varpi]\right)(m-1)\\
    &=   
    \left(\dim_{O/\varpi} \Sel_{\mathcal F(\mathdutchbcal Q)}(k, \ad^0\overline\rho) \right) (m- 1) \\
    &\leq \left(\dim_{O/\varpi} \Sel_{\mathcal F}(k, \ad^0\overline\rho) + \sum_{\mathdutchbcal q\in\mathdutchbcal Q} \dim_{O/\varpi} \frac{H^1 (k_{\mathdutchbcal q}, \ad^0 \overline\rho)}{H^1_{\mathcal F}(k_{\mathdutchbcal q}, \ad^0\overline\rho)}\right)(m-1);
    \end{split}
    \end{equation}
    in the second line we have used Lemma \ref{stuff like lemma 6.1 FKP}(\ref{lemma 6.1 part III}).
    By the local Euler characteristic formula, $\dim_{O/\varpi}H^1(k_{\mathdutchbcal q}, \ad^0\overline\rho)$ is uniformly bounded in $\mathdutchbcal q$, so (\ref{eq:cor_relaxed_bound}) becomes
    \begin{equation}
        \lg_O \Sel_{\mathcal F(\mathdutchbcal Q)} (k, \ad^0\rho_n)\leq C_2 (m-1)
    \end{equation}
    for a constant $C_2$ depending only on $c = |\mathdutchbcal Q|$ and $\rho$. Combined with Proposition \ref{prop:const_relaxed}, this proves the corollary. 
\end{proof}
\begin{notation}\label{notation:appendix_global_rings}
Let $\Sigma$ be a finite set of places of $k$. 
    \begin{enumerate}
\item Let $A\in \CNL_O$.
 A lift $\rho_A: G_k \to G(A)$ of $\overline \rho$ is called $\Sigma$\emph{-good} if:
\begin{enumerate}[label = (\roman*)]
\item $\mu \circ \rho_A = \chi$ (notation as in Notation \ref{notation:appendix_def_rings});
    \item  $\rho_A$ is 
unramified outside $S\cup\Sigma_p\cup\Sigma_\infty \cup \Sigma $;
\item For all $v\in \Sigma_p^\crystal$ (resp. $v\in \Sigma_p - \Sigma_p^\crystal$), the map $\widetilde R_v\to A$ defined by $\rho_A|_{G_{k_v}}$ factors through $R_v^\crystal$ (resp. $R_v$).
\end{enumerate}

\item \label{notation:appendix_global_rings_firstdef}Let $\mathcal D_{\Sigma}^{\operatorname{global}}$ be the functor on $\CNL_O$ defined by 
$$\mathcal D_{\Sigma}^{\operatorname{global}} (A) = \set{\rho_A: G_k \to G(A)\,:\, \rho_A \otimes_A (O/\varpi) = \overline\rho \text{ and } \rho_A \text{ is }\Sigma\text{-good}}/\sim,$$
    where the equivalence relation is $\ker\left(G(A) \to G(O/\varpi)\right)$-conjugacy. 
By (the same argument of) \cite[Proposition 2.2.9]{clozel2008automorphy}, $\mathcal D_{\Sigma}^{\operatorname{global}}$ is represented by a global deformation ring which we denote $R^\Sigma$. 
\item \label{notation:appendix_global_rings_orddef}Now suppose $\Sigma = \mathdutchbcal Q$ for some finite subset $\mathdutchbcal Q \subset \mathfrak Q$. Let $\mathcal D_{\mathdutchbcal Q\text{-ord}}^{\operatorname{global}} \subset \mathcal D_{\mathdutchbcal Q}^{\operatorname{global}}$ be the subfunctor consisting of deformations which are $\mathdutchbcal q$-ordinary for all $\q\in \mathdutchbcal Q$. 
By (the same argument of) \cite[Proposition 2.2.9]{clozel2008automorphy}, $\mathcal D_{\mathdutchbcal Q\text{-ord}}^{\operatorname{global}}$ is represented by the $\mathdutchbcal Q$-\emph{ordinary quotient}   of $R^{\mathdutchbcal Q}$, which we denote $R_{\mathdutchbcal Q}$.

\item Given a homomorphism $f_ {\mathdutchbcal Q}: R _ {\mathdutchbcal Q}\rightarrow O $, define the
\emph {congruence ideal}
\begin{equation}
\eta_{f_ {\mathdutchbcal Q}}\subset O\coloneqq f_ {\mathdutchbcal Q} (\annihilator_{R_ {\mathdutchbcal Q}} (\kernel f_ {\mathdutchbcal Q}))\subset O.
\end{equation}

    \end{enumerate}
\end{notation}
\begin{lemma}\label{lem:eta_appendix}
Let $m \geq 1$ be an integer. Then we may choose the integer $n_0(m, \rho)\geq 1$ in Theorem \ref{main appendix theorem} such that the following holds: 
suppose $n \geq n_0(m, \rho)$, and   $\mathdutchbcal Q$ is an $n$-admissible set
such that $$\overline\summer_{\mathcal F (\mathdutchbcal Q)} (\adjoint ^ 0\rho_m) = 0. $$
Let 
$f_ {\mathdutchbcal Q}: R_ {\mathdutchbcal Q}\rightarrow O $
be the homomorphism corresponding to the representation $\rho^{\mathdutchbcal Q}$ from Theorem \ref{main appendix theorem}. 
Then $$\Ord_\varpi\eta_{f_ {\mathdutchbcal Q}}\leq \length_O\summer_{\mathcal F(\mathdutchbcal Q)^\rel} (\adjoint ^ 0\rho_{n - m + 1}). $$

\end{lemma} 
\begin{proof}
Let $n_1(m, \rho) \geq 1$ satisfy the conclusion of Theorem \ref{main appendix theorem}, and let $C_1$ be the constant from Proposition \ref{prop:const_relaxed}. We set $n_0(m, \rho) \coloneqq\max\set{n_1(m,\rho), C_1 + 2m-1}$, and check the claimed property.  Write $I\coloneqq\kernel f_ {\mathdutchbcal Q} $. We have
$$\fitting_{R_ {\mathdutchbcal Q}} (I)\subset\annihilator_{R_ {\mathdutchbcal Q}} (I), $$
so by base change for Fitting ideals,
$$\fitting_{R_ {\mathdutchbcal Q}/I} (I/I ^ 2)\subset\eta_{f_ {\mathdutchbcal Q}}. $$
Because $R_ {\mathdutchbcal Q}/I = O$, it therefore suffices to bound $\length_OI/I ^ 2. $

Now note that $O $-module maps $I/I ^ 2\rightarrow O/\varpi ^ s $, for any integer $s\geq 1 $, are canonically in bijection with lifts
$R_{\mathdutchbcal Q}\rightarrow O [\epsilon]/(\epsilon ^ 2,\varpi ^ s\epsilon) $
of $f_{\mathdutchbcal Q} $. Taking $s = n- m +1$ and using that $\rho^{\mathdutchbcal Q} \equiv \rho \pmod {\varpi^{n-m+1}}$, such lifts are in bijection with classes in $\summer_{\mathcal F  (\mathdutchbcal Q)^\rel} (k,\adjoint ^ 0\rho_{n-m+1}) $. Hence
\begin{equation}\label{eq:IandSel}
     \Hom(I/I^2, O/\varpi^{n-m+1}) \cong \Sel_{\mathcal F(\mathdutchbcal Q)^\rel}(k, \ad^0 \rho_{n-m+1}).
\end{equation}

Now, by Lemma \ref{stuff like lemma 6.1 FKP}(\ref{lemma 6.1 part II residual selmer groups and torsion}),
$\summer_{\mathcal F (\mathdutchbcal Q)} (k,\adjoint ^ 0\rho_{n-m+1}) $
is $\varpi^{m-1}$-torsion. In particular, (\ref {relax the local conditions exact sequence in appendix}) shows that $\summer_{\mathcal F  (\mathdutchbcal Q)^\rel} (k,\adjoint ^ 0\rho_{n-m+1}) $
is $\varpi ^ {m + C_1-1} $-torsion, hence \emph{a fortiori} $\varpi^{n -m }$-torsion. Since $I$ is finitely generated over $R_{\mathdutchbcal Q}$, we conclude that $I/I^2$ is $\varpi^{n-m}$-torsion. 

Thus $$\lg_O I/I^2 = \length_O\Home_O (I/I ^ 2, O/\varpi ^ {n-m+1}),$$ and the lemma follows from (\ref{eq:IandSel}). 
\end {proof}
We remark that essentially the same argument shows:
\begin{rmk}\label{rmk:appendix_iso_criterion}
    The map $f_{\mathdutchbcal Q}: R_{\mathdutchbcal Q} \to O$ is an isomorphism if and only if $\Sel_{\mathcal F(\mathdutchbcal Q)^\rel} (k, \ad^0 \overline \rho) = 0$.
\end{rmk}
\begin{definition}\label{appendix definition standard}
Let $\q $
be $n $-admissible. We say that $\q $
is
\emph {standard} if:
\begin{enumerate}
\item \label{appendix definition standard part I}
There exists a representation $\tau_\q: G_{k_\q}\rightarrow G (O) $
which is both ordinary and unramified.
\item \label{appendix definition standard part II}For all $m\leq n $, both $$\frac {H ^ 1_\unramified (k_\q,\adjoint ^ 0\rho_m) + H ^ 1_\ordinary (k_\q,\adjoint ^ 0\rho_m)} {H ^ 1_\unramified (k_\q,\adjoint ^ 0\rho_m)} $$
and $$\frac {H ^ 1_\unramified (k_\q,\adjoint ^ 0\rho_m) + H ^ 1_\ordinary (k_\q,\adjoint ^ 0\rho_m)} {H ^ 1_\ordinary (k_\q,\adjoint ^ 0\rho_m)} $$
are free of rank one over $O/\varpi ^ m $.
\end{enumerate}
\end{definition}
\begin{lemma}\label{lem:std_upshot}
Fix an integer $m\geq 1 $ and let $n_0 = n_0 (m,\rho) $
be the integer of Theorem \ref{main appendix theorem}. Let $n \geq \max\set{n_0, 3m}$ be an integer and suppose given a finite set $\mathdutchbcal Q $
of $(n + m) $-admissible primes and two additional $n $-admissible primes ${\mathdutchbcal{p}},\q\not\in\mathdutchbcal Q $, such that:
\begin{enumerate}
\item \label{item q is std not admissible in appendix lemma} $\q $
is standard and
\emph {not} $(n +1) $-admissible.
\item \label{item nonzero Qq selmer group in appendix lemma}$\overline\summer_{\mathcal F (\mathdutchbcal Q)} (k,\adjoint ^ 0\rho_m) =\overline\summer_{\mathcal F (\mathdutchbcal Q\q{\mathdutchbcal{p}})} (k,\adjoint ^ 0\rho_m) = 0 $
but $\overline\summer_{\mathcal F (\mathdutchbcal Q\q)} (k,\adjoint ^ 0\rho_{2 m -1})\neq 0. $
\end{enumerate}
Then the representation $\rho ^ {\mathdutchbcal Q\q{\mathdutchbcal{p}}} $
constructed in Theorem \ref{main appendix theorem} is ramified at ${\mathdutchbcal{p}}$
modulo $\varpi ^ {n + m} $.
\end{lemma}
\begin{proof}
Let $\rho ^ {\mathdutchbcal Q} $
and $\rho ^ {\mathdutchbcal Q\q{\mathdutchbcal{p}}} $
be the representations afforded by Theorem \ref{main appendix theorem}, so that $\rho ^ {\mathdutchbcal Q}\equiv\rho\pmod {\varpi ^ {n +1}} $
and $\rho ^ {\mathdutchbcal Q\q{\mathdutchbcal{p}}}\equiv\rho\pmod {\varpi ^ {n - m +1}}. $
Modulo $\varpi ^ {n + m} $, $\rho_{\mathdutchbcal Q} $
and $\rho_{\mathdutchbcal Q\q{\mathdutchbcal{p}}} $
differ by a cocycle $c\in H ^ 1 (k,\adjoint ^ 0\rho_{2 m -1}) $. (This make sense because $2 m -1\leq n - m +1 $.) Also let $d\in\summer_{\mathcal F (\mathdutchbcal Q\q)} (k,\adjoint ^ 0\rho_{2 m -1}(1)) $
be an element whose image in $\overline\summer_{\mathcal F (\mathdutchbcal Q\q)^\ast} (k,\adjoint ^ 0\rho_{2 m -1}(1)) $
is nonzero, which exists Lemma \ref{balanced residual selmer ranks lemma}. By global Poitou-Tate duality, we have
\begin{equation}\label{appendix equation global PT pairing zero}
    \sum_v\langle c, d\rangle_v = 0, 
\end{equation}
where $\langle c, d\rangle_v $
is the local Tate pairing.
By Proposition \ref{proposition making selmer conditions deformation appendix}(\ref{proposition making selmer conditions deformation appendix part II}) and by the choice of $n_0 $, $\local_v c $
lies in $H ^ 1_{\mathcal F (\mathdutchbcal Q\q)} (k_v,\adjoint ^ 0\rho_{2 m -1}) $
for all $v\neq{\mathdutchbcal{p}},\q $. In particular, 
\begin{equation}\label{eq:appendix_PT_upshot}
    \langle c, d\rangle_{\mathdutchbcal p} \neq 0 \iff \langle c, d\rangle _{\mathdutchbcal q} \neq 0.
\end{equation}Our next claim is that:
\begin{equation}\label{nonzero local tate pairing appendix equation}
\langle c, d\rangle_\q\neq 0. 
\end{equation}
Indeed, $\Res_\q c $
lies in $H ^ 1_{\unramified} (k_\q,\adjoint ^ 0\rho_{2 m -1}) + H ^ 1_\ordinary (k_\q,\adjoint ^ 0\rho_{2 m -1}) $; one can see this by comparing both $\rho ^ {\mathdutchbcal Q} $
and $\rho ^ {\mathdutchbcal Q\q{\mathdutchbcal{p}}} $
to the representation $\tau_\q $
in Definition \ref{appendix definition standard}(\ref{appendix definition standard part I}).  Because $\q $
is not $(n +1) $-admissible, $\rho ^ {\mathdutchbcal Q} $
is not ordinary at $\q $
modulo $\varpi ^ {n +1} $, so 
$$\Res_\q c \in \frac{H^1_\unr(k_\q, \ad^0 \rho_{2m-1}) + H^1_\ord(k_\q, \ad^0\rho_{2m-1})}{H^1_\ord(k_\q, \ad^0 \rho_{2m-1})} \approx O/\varpi^{2m-1}$$
is nonzero modulo $\varpi^m$.
On the other hand, 
\begin{equation*}
    \begin{split}
        \Res_\q d\in \frac {H ^ 1_\ordinary (k_\q,\adjoint ^ 0\rho_{2 m -1} (1))} {H ^ 1_\ordinary (k_\q,\adjoint ^ 0\rho_{2 m -1} (1))\intersection H ^ 1_\unramified (k_\q,\adjoint ^ 0\rho_{2 m -1} (1))}\\ =\frac {H ^ 1_\ordinary (k_\q,\adjoint ^ 0\rho_{2 m -1} (1)) + H ^ 1_\unramified (k_\q,\adjoint ^ 0\rho_{2 m -1} (1))} {H ^ 1_\unramified (k_\q,\adjoint ^ 0\rho_{2 m -1} (1))}\approximate O/\varpi ^ {2 m -1}
    \end{split}
\end{equation*}
is also nonzero modulo $\varpi ^ m $. Otherwise, the image of $d $
modulo $\varpi ^ m $
would lie in $\summer_{\mathcal F (\mathdutchbcal Q)^\ast} (k,\adjoint ^ 0\rho_m (1)) $, which contradicts the assumption that $\overline\summer_{\mathcal F (\mathdutchbcal Q)^\ast} (k,\adjoint ^ 0\rho_m (1)) = 0. $
Since local Poitou-Tate duality gives a perfect pairing
\begin{equation*}
\begin{split}\frac {H ^ 1_\unramified (k_\q,\adjoint ^ 0\rho_{2 m -1}) + H ^ 1_\ordinary (k_\q,\adjoint ^ 0\rho_{2 m -1})} {H ^ 1_\ordinary (k_\q,\adjoint ^ 0\rho_{2 m -1})}\times \frac {H ^ 1_\ordinary (k_\q,\adjoint ^ 0\rho_{2 m -1} (1))} {H ^ 1_\ordinary (k_\q,\adjoint ^ 0\rho_{2 m -1} (1))\intersection H ^ 1_\unramified (k_\q,\adjoint ^ 0\rho_{2 m -1} (1))}\to \\O/\varpi^{2m-1},
\end{split}
\end{equation*}
we indeed have (\ref{nonzero local tate pairing appendix equation}). Then by (\ref{eq:appendix_PT_upshot}), we conclude $$\langle c, d\rangle_{\mathdutchbcal{p}}\neq 0. $$
Since $\local_{\mathdutchbcal{p}} d $
is unramified, we must have $$\Res_{\mathdutchbcal p}c\not\in H ^ 1_\unramified (k_{\mathdutchbcal{p}},\adjoint ^ 0\rho_{2 m -1}). $$
Since $\rho ^ {\mathdutchbcal Q} |_{G_{k_{\mathdutchbcal{p}}}} $
is unramified, and $\Res_{\mathdutchbcal{p}} c $
measures the difference between $\rho ^ {\mathdutchbcal Q} |_{G_{k_{\mathdutchbcal{p}}}} $
and $\rho ^ {\mathdutchbcal Q\q{\mathdutchbcal{p}}} |_{G_{k_{\mathdutchbcal{p}}} }$
modulo $\varpi ^ {n + m} $, this proves the lemma.
\end{proof}
\section{Large image results}\label{sec:large_image}
Throughout this appendix, let $E$ be a finite extension of $\Q_p$, with ring of integers $O_E \subset E$.
\subsection{ Generalities on $p $-adic Lie groups}\label{subsec:appendixC_generalities}

\begin{lemma}\label{facts about simple Lie algebras}
Let $\mathfrak h $ be a simple Lie algebra over $ E $.
\begin{enumerate}
\item\label{facts about simple Lie algebras part one} If $\mathfrak g\subset\mathfrak h ^ {\oplus n} $
is a Lie subalgebra that surjects onto each factor, then $\mathfrak g $ is isomorphic to $\mathfrak h ^ {\oplus m} $ for some integer $ m\leq n $. Up to an automorphism of $\mathfrak h ^ {\oplus n} $, the map $\mathfrak g\cong\mathfrak h ^ {\oplus m}\rightarrow\mathfrak h ^ {\oplus n} $ is given by $$(h_1,\ldots, h_m)\mapsto (\underbrace {h_1,\ldots, h_1}_{n_1\text { times}},\underbrace {h_2,\ldots, h_2}_{n_2\text { times}},\ldots,\underbrace {h_m,\cdots, h_m}_{n_m\text { times}} )$$ with $ n_1+\cdots + n_m = n $.
\item\label{facts about simple Lie algebras part two} The only ideal $I\subset\mathfrak h ^ {\oplus n} $ that surjects onto each factor is $ I =\mathfrak h ^ {\oplus n} $.
\end{enumerate}
\end{lemma}
\begin{proof}
We prove (\ref{facts about simple Lie algebras part one}) by induction on $ n $, with the case $ n = 1 $ being trivial. Supposing we know (\ref{facts about simple Lie algebras part one}) for $ n -1 $, let $\mathfrak g\subset\mathfrak h ^ {\oplus n} $ be a subalgebra surjective onto each factor, and let $\mathfrak g' $ be the image of $\mathfrak g $ under the projection $\mathfrak h ^ {\oplus n} =\mathfrak h ^ {\oplus (n -1)}\oplus\mathfrak h\rightarrow\mathfrak h ^ {\oplus (n -1)} $. Then by the inductive hypothesis, $\mathfrak g'\cong\mathfrak h ^ {\oplus m} $
for some integer $ m\leq n -1 $. Now, $\mathfrak g\subset\mathfrak g'\oplus\mathfrak h $
is a subalgebra surjective onto each factor, so by Goursat's Lemma for Lie algebras, $\mathfrak g $ is either $\mathfrak g'\oplus\mathfrak h $ or the graph of isomorphism between $\mathfrak h $ and a simple factor of $\mathfrak g' $. In particular, $\mathfrak g $ is isomorphic to either $\mathfrak g'\cong\mathfrak h ^ {\oplus m} $ or $\mathfrak g'\oplus\mathfrak h\cong\mathfrak h ^ {\oplus (m +1)} $, and it is easy to check that the embedding $\mathfrak g\rightarrow\mathfrak h ^ {\oplus n} $ is of the desired form using that $\mathfrak g'\rightarrow\mathfrak h ^ {\oplus (n -1)} $ is.
For  (\ref{facts about simple Lie algebras part two}), it suffices to check that the subalgebras in  (\ref{facts about simple Lie algebras part one}) are never ideals unless $ n = m $ (and hence $ n_1 = n_2 =\cdots = n_m = 1 $). Indeed, it suffices to check that the diagonal subalgebra $\mathfrak h\subset\mathfrak h\oplus\mathfrak h $ is not an ideal, but this is clear: since $\mathfrak h $ is simple, it is not abelian, so for some $ h_1, h_2\in\mathfrak h $ we have $ [h_1, h_2]\neq 0 $. In particular, the bracket $[(h_1, h_1), (h_2, 0)] $ is not contained in the diagonal subalgebra, which witnesses that the latter is not an ideal.
\end{proof}
\begin{corollary}\label{corollary subalgebra of simple Lie algebra with base change}
Let $\mathfrak h $ be an absolutely simple Lie algebra over $\Q_p $. Then for any finite extension $ E/\Q_p $:
\begin{enumerate}
\item \label{corollary subalgebra of simple Lie algebra with base change part one}The base change $\mathfrak h_E\coloneqq\mathfrak h\otimes_{\Q_p} E $ is simple as a Lie algebra over $\Q_p $.
\item \label{corollary subalgebra of simple Lie algebra with base change part two}For any $\Q_p $-Lie subalgebra $\mathfrak g\subset\mathfrak h_E $ such that $ E\cdot\mathfrak g =\mathfrak h_E $, $\mathfrak g $ is simple.
\end{enumerate}
\end{corollary}
\begin{proof}
For any subalgebra $\mathfrak g\subset\mathfrak h_E $, consider the extension of scalars $$\mathfrak g\otimes_{\Q_p}\overline\Q_p\subset\mathfrak h_E\otimes_{\Q_p}\overline\Q_p\cong\mathfrak h_{\overline\Q_p} ^ {[E:\Q_p]}. $$ For (\ref{corollary subalgebra of simple Lie algebra with base change part one}), suppose $\mathfrak g $ is an ideal; then the image of $\mathfrak g\otimes_{\Q_p}\overline\Q_p $ in each factor of $\mathfrak h_{\overline\Q_p} ^ {[E:\Q_p]} $ is a $\overline\Q_p $-stable ideal, hence either 0 or $\mathfrak h_{\overline\Q_p} $. Now, $\mathfrak g\otimes_{\Q_p}\overline\Q_p\subset\mathfrak h_E\otimes_{\Q_p}\overline\Q_p $ is stable under the action of $ G_{\Q_p} $, which transitively permutes the factors of $\mathfrak h_{\overline\Q_p} ^ {[E:\Q_p]} $. Hence if $\mathfrak g\neq 0 $, then $\mathfrak g\otimes_{\Q_p}\overline\Q_p $ surjects onto each factor of $\mathfrak h_{\overline\Q_p} ^ {[E:\Q_p]} $. Then by Lemma \ref{facts about simple Lie algebras}(\ref{facts about simple Lie algebras part two}), $\mathfrak g\otimes_{\Q_p}\overline\Q_p =\mathfrak h_E\otimes_{\Q_p}\overline\Q_p $, so $\mathfrak g =\mathfrak h_E $. This proves (\ref{corollary subalgebra of simple Lie algebra with base change part one}).
For (\ref{corollary subalgebra of simple Lie algebra with base change part two}), if $\mathfrak g\cdot E =\mathfrak h_E $, then $\mathfrak g\otimes_{\Q_p}\overline\Q_p $ generates $\mathfrak h_E\otimes_{{\Q_p}}\overline\Q_p \cong  \mathfrak h_{\overline\Q_p} ^ {[E:{\Q_p}]} $
under the action of $ E\otimes_{{\Q_p}}\overline\Q_p \cong\overline\Q_p ^ {[E:{\Q_p}]}$, so $\mathfrak g\otimes_{{\Q_p}}\overline\Q_p $ surjects onto each factor. By Lemma \ref{facts about simple Lie algebras}(\ref{facts about simple Lie algebras part one}), we conclude $\mathfrak g\otimes_{{\Q_p}}\overline\Q_p\cong\mathfrak h_{\overline\Q_p} ^ {\oplus m} $ for some $ m\leq [E:{\Q_p}] $.

If $ I\subset\mathfrak g $ is a nonzero ideal, then $ I\cdot E $ is a nonzero ideal of $\mathfrak g\cdot E =\mathfrak h_E $, so $ I\cdot E =\mathfrak h_E $. But then $I\otimes_{{\Q_p}}\overline\Q_p $
is an ideal of $\mathfrak g\otimes_{{\Q_p}}\overline\Q_p $ that surjects onto each factor of $\mathfrak h_{\overline\Q_p} ^ {[E:{\Q_p}]} $, and, inspecting the possible embeddings $$\mathfrak g\otimes_{{\Q_p}}\overline\Q_p\cong\mathfrak h ^ {\oplus m}_{\overline\Q_p}\hookrightarrow\mathfrak h_{\overline\Q_p} ^ {[E:{\Q_p}]} $$ from Lemma \ref{facts about simple Lie algebras}(\ref{facts about simple Lie algebras part one}), we conclude that $ I\otimes_{{\Q_p}}\overline\Q_p $ surjects onto each factor of $\mathfrak g\otimes_{\Q_p}\overline\Q_p\cong\mathfrak h ^ {\oplus m}_{\overline\Q_p} $. But by Lemma \ref{facts about simple Lie algebras}(\ref{facts about simple Lie algebras part two}), $ I\otimes_{{\Q_p}}\overline\Q_p =\mathfrak g\otimes_{{\Q_p}}\overline\Q_p $, so then $ I =\mathfrak g $. This proves (\ref{corollary subalgebra of simple Lie algebra with base change part two}).
\end{proof}
\subsection{ Strongly irreducible representations}\label{subsec:appendixC_strongly_irr}
For the following definition only, we allow $E$ to be an arbitrary algebraic extension of $\Q_p$. 
\begin{definition}\label{definition strongly irreducible}
Suppose $V$ is a finite-dimensional $E$-vector space, $G$ is a group, and $\rho: G\to \GL_E(V)$ is a representation. 
 Then $V$ (or $\rho$) is said to be
 {strongly irreducible} if, for any finite-index subgroup $ H\subset G $, $ (\rho, V) $ is absolutely irreducible as a representation of $ H $.
\end{definition}
\begin{lemma}\label{lemma with all the strongly irreducible statements}
Let $ V $ be an $ E $-vector space of finite dimension, and let $ G\subset\GL_E (V) $ be a compact $p $-adic Lie subgroup. If $ V $ is strongly irreducible as a representation of $ G $, then:
\begin{enumerate}
\item \label{lemma strongly irreducible part trivial center}No nontrivial element of $ G/Z_G $ is fixed under conjugation by an open subgroup $ U\subset G $; in particular $ G/Z_G $ has trivial center.
\item \label{lemma strongly irreducible part no finite normal}The group $ G/Z_G $ contains no finite normal subgroup.
\item\label{lemma strongly irreducible part unipotents have only one eigenvalue} If $ g\in G $ acts unipotently on $\Lie G $, then $ g $
has only one eigenvalue on $ V $.
\item\label{lemma strongly irreducible part lie algebra is semisimple} The Lie algebra $\Lie G/Z_G$ is semisimple.

\item \label{lemma strongly irreducible part center contains open subgroup of determinant}
The natural maps $ G\intersection\SL_E (V)\rightarrow G/Z_G $ and $Z_G \to \det G$ induce  isomorphisms on Lie algebras.
\item\label{lemma strongly irreducible part lie algebra representation is irreducible} $ V $ is absolutely irreducible as a representation of $\Lie (G\intersection\SL_E (V)) $.
\end{enumerate}
\end{lemma}
\begin{proof}
\begin{enumerate}\item Let $ h\in G $ be an element whose image in $ G/Z_G $ is invariant under conjugation by $ U $. Then for all $ g\in U $, $ h g h ^{-1} g ^{-1} $ lies in $ Z_G $, so by Schur's Lemma 
\begin{equation}\label{Scherz Lemma strongly irreducible equation}
h g h ^{-1} = g \lambda_h(g)\;\;\text {for a scalar $\lambda_h (g)\in E ^\times $}.\end{equation}
It is easy to check that $ g\mapsto\lambda_h (g) $ is a group homomorphism $ U\rightarrow E ^\times $. On the other hand, if $\dimension_E V = n $, then (\ref {Scherz Lemma strongly irreducible equation}) implies that $\lambda_h (g) $
lies in $\mu_n (E) $ for all $ g\in U $. In particular, the homomorphism $ g\mapsto\lambda_h (g) $ has open kernel, so $ h $ commutes with an open subgroup of $ U $. By strong irreducibility and Schur's Lemma again, $ h $ is scalar, so has trivial image in $ G/Z_G $.

\item

Let $ H\subset G/Z_G $ be a finite normal subgroup. Then the map $ G\rightarrow\automorphisms (H) $ has open kernel, so (\ref{lemma strongly irreducible part trivial center}) implies that $ H $ is trivial.


\item Let $g = g ^{ss} g ^ u $ be the Jordan decomposition in $\GL_E (V) $. Then $$\Adjoint (g)-1 = (\Adjoint (g ^ {ss}) -1) (\Adjoint (g ^ u) -1) + (\Adjoint (g ^ u )-1) + (\Adjoint (g ^ {ss}) -1) $$ as operators on $\mathfrak {gl}_E (V) $. In particular, if $\Adjoint (g) -1 $
is nilpotent on $\mathfrak g\subset\mathfrak {gl}_E (V) $, then for $ N $ sufficiently large, $\mathfrak g $ lies in the kernel of $ (\Adjoint (g ^ {ss}) -1) ^ N $.

But since $\Adjoint g ^ {s s} -1 $ is diagonalizable over $\overline\Q_p $ as an operator on $\mathfrak {gl}_E (V) $, we conclude $\mathfrak g $ lies in the kernel of $\Adjoint (g ^ {ss}) -1 $; hence $ g ^ {ss} $ commutes with an open subgroup of $ G $, so by Schur's Lemma $ g ^ {ss} $ is a scalar in $\GL_E (V) $. In particular, $ g = g ^ {ss} g ^ u $ has a single eigenvalue on $ V $.
\item  By \cite[\S6, Proposition 5]{bourbaki1960algebres}, $\mathfrak g$ is  a direct sum $\mathfrak g = \mathfrak h \oplus \mathfrak s$, with $\mathfrak h$ semisimple and $\mathfrak s$ abelian. Since $\Lie(G/Z_G)$ has trivial center by (\ref{lemma strongly irreducible part trivial center}), it follows that the natural maps induce isomorphisms
 $\mathfrak h \isomorphism \Lie (G/Z_G)$ and $\Lie Z_G \isomorphism \mathfrak s$. 
In particular, $\Lie (G/Z_G)$ is semisimple.  \item 
The map $\Lie (G\cap \SL_E(V)) \to\Lie (G/Z_G) = \mathfrak h$ is injective with abelian cokernel; hence it is an isomorphism. Since $\mathfrak g = \mathfrak h \oplus \mathfrak s$, it follows that the determinant identifies  $ \Lie Z_G = \mathfrak s\isomorphism \Lie (\det(G))$. 
\item By (\ref{lemma strongly irreducible part center contains open subgroup of determinant}), 
$\mathfrak g $ is a direct sum $\mathfrak g =\Lie (Z_G)\oplus\Lie (G\intersection\SL_E (V)) $. So if $\Lie (G\intersection\SL_E (V)) $ stabilized any subspace of $ V $ after extending scalars, $\mathfrak g $ would as well, which contradicts strong irreducibility.
\end{enumerate}
\end{proof}
\begin{lemma}\label{lemma strongly irreducible in dimension 2 or 4}
Let $V $ be a symplectic $ E $-vector space of dimension 2 or 4, and let $ G\subset\GSP_E (V) $ be a compact $p $-adic Lie subgroup. If $ V $ is strongly irreducible as a representation of $ G $, then every nontrivial closed normal subgroup of $ G/Z_G $ has finite index.
\end{lemma}
\begin{proof}
Abbreviate $\overline G = G/Z_G $
and $\overline {\mathfrak g} =\Lie\overline G $. Then $\overline {\mathfrak g} $
is a Lie subalgebra (over ${\Q_p} $) of $\mathfrak {sp}_{n, E} $, with $ n = 2 $ or 4. After replacing $ E $ with a finite extension, we may assume $ E\cdot\overline {\mathfrak g}\subset\mathfrak {s p}_{n, E} $ is split. It is also semisimple (by Lemma \ref{lemma with all the strongly irreducible statements}(\ref{lemma strongly irreducible part lie algebra is semisimple})) of rank at most 2, so $ E\cdot\overline {\mathfrak g} $ is isomorphic to $\mathfrak {s l}_{2, E} $, $\mathfrak {s l}_{2, E}\times\mathfrak {s l}_{2, E} $, or $\mathfrak {s p}_{4, E} $. The second case is impossible by Lemma \ref{lemma with all the strongly irreducible statements}(\ref{lemma strongly irreducible part center contains open subgroup of determinant}, \ref{lemma strongly irreducible part lie algebra representation is irreducible}), since $\mathfrak {s l}_{2, E}\times\mathfrak {s l}_{2, E} $ admits no faithful irreducible two- or four-dimensional symplectic representation. Hence $ E\cdot\overline {\mathfrak g} $ is simple, so $\overline {\mathfrak g} $ is simple by Corollary \ref{corollary subalgebra of simple Lie algebra with base change}(\ref{corollary subalgebra of simple Lie algebra with base change part two}). Now if $ H\subset\overline G $ is any closed normal subgroup, it has the structure of a compact $p $-adic Lie subgroup by \cite[Ch. III, Th\'eor\`eme 3.2.3]{lazard1965groupes}. In particular, $\mathfrak h\coloneqq\Lie H $ is an ideal of $\overline {\mathfrak g} $, hence $\mathfrak h = 0 $ or $\mathfrak h =\overline {\mathfrak g} $. By arguing with the exponential map, we see that $ H $ is either finite or has finite index; but if finite it is trivial by Lemma \ref{lemma with all the strongly irreducible statements}(\ref{lemma strongly irreducible part no finite normal}), so the lemma is proved.
\end{proof}
In fact, we extract the following more precise statement in the four-dimensional case.
\begin{prop}\label{prop cases for subgroups of GSp4}
Let $G\subset\GSP_4 (E) $ be a compact $p $-adic Lie subgroup, such that the defining representation is strongly irreducible. Write $\mathfrak h\coloneqq\Lie (G/Z_G)\subset\mathfrak {s p}_{4, E} $. Then after a finite extension of $ E $:
\begin{enumerate}
\item \label{prop cases for subgroups of GSp4 part one}$ E\cdot\mathfrak h $ is isomorphic to either $\mathfrak {s p}_{4, E} $ or $\mathfrak {s l}_{2, E} $.
\item \label{prop cases for subgroups of GSp4 part two}In the latter case $ G $ is contained in the image of the symmetric cube representation $\symmetric^3:\GL_2 (E)\rightarrow\GL_4 (E) $ up to $\GL_4 (E) $-conjugacy.
\end{enumerate}
\end{prop}
\begin{proof}
We have seen (\ref{prop cases for subgroups of GSp4 part one}) in the proof of Lemma \ref{lemma strongly irreducible in dimension 2 or 4}, so we prove (\ref{prop cases for subgroups of GSp4 part two}). Let $V$ be the four-dimensional defining representation of $G$; as a representation of $E\cdot \mathfrak h \cong\mathfrak{sl}_{2,E}$, $V$ is isomorphic to the symmetric cube. In particular, after extending $E$ if necessary, the embedding $E\cdot \mathfrak h \cong \mathfrak {sl}_{2,E} \hookrightarrow
\mathfrak{sp}_{4,E}$ is conjugate to $\Lie \symmetric^3$ by some $g\in \GL_4(E)$. We may assume without loss of generality that $g = 1$. Let $S\subset \GSP_4(E)$ be the image of the symmetric cube embedding over $\overline\Q_p$; we first claim that $G$ is contained in $S\cdot E^\times \subset \GL_4(E)$. Indeed, for any $g\in G$, $\Ad(g)$ preserves $\mathfrak{sl}_{2,E} = E\cdot \mathfrak h$. Since the automorphism group of $\mathfrak{sl}_{2,E}$ is $\PGL_2(E)$, for each $g\in G$ there exists $h\in S$ such that $\Ad(h) = \Ad(g)$ on $\mathfrak{sl}_{2, E}$. In particular $h^{-1}g\in \GL_4(E)$ commutes with an open subgroup of $G$ (by arguing with the exponential map), so by Schur's Lemma and strong irreducibility $h^{-1}g$ is scalar. So $G\subset S\cdot E^\times$, as desired.

If $E'$ denotes the compositum of the finitely many cubic extensions of $E$, then $S \cdot E^\times$ is contained in the image of the symmetric cube map $\GL_2(E') \to \GL_4(E')$, and this completes the proof.
\end{proof}
The following lemma is a corollary of \cite[Lemma 4.3]{calegari2013irreducibility}.
\begin{lemma}\label{lemma distinct Hodge Tate weights implies induction}
    Fix a number field $F$, and let $\rho:  G_F \to \GL_E(V)$ be a continuous,  absolutely irreducible  representation of $G_F$. Assume there exists a place $\p|p$ of $F$ such that $V|_{G_{F_\p}}$ is Hodge-Tate with distinct weights. Then, after possibly replacing $E$ by a finite extension, there exists a number field $K\supset F$ and a strongly irreducible, continuous  representation $\rho_0 : G_K \to \GL_E(V_0)$ such that $\rho \cong \Ind_{G_K}^{G_F} \rho_0$. 
\end{lemma}
\begin{proof}
    After taking a finite extension of $E$, there exists a Galois extension $K$ of  $ F$ such that each constituent of $\rho|_{G_K}$ is strongly irreducible. Write $$\rho|_{G_K}^{ss} = \bigoplus_{i = 0}^j \rho_i$$
    for some $0 \leq j < n$. Then the $\rho_i$ are all distinct because $\rho$ has distinct Hodge-Tate weights, so there is a well-defined action of $\Gal(K/F)$ on the set of $\rho_i$'s. This action must be transitive or else $\rho$ would be reducible; hence $\rho|_{G_K}$ is semisimple and each $\rho_i$ has the same dimension $m$. Replacing $K$ by the fixed field of the stabilizer of $\rho_0$, it follows that $\rho = \Ind_{G_K}^{G_F} \rho_0$.
\end{proof}

\begin{corollary}\label{cor:general_type_irr}
Let $\pi $ be a relevant, non-endoscopic automorphic representation of $\GSP_4 (\A_\Q) $  such that $\BC(\pi)$ (Lemma \ref{lem:BC}) is not an automorphic induction. Then for each isomorphism $\iota: \overline\Q_p \isomorphism \C$  with $p > 3$,  $V_{\pi,\iota} $
is strongly irreducible.
\end{corollary}
\begin{proof}
By  Lemma \ref{lem:reducible_endoscopic}, $ V_{\pi,\iota}$ is absolutely irreducible. Suppose for contradiction that it is not strongly irreducible. By Lemma \ref{lemma distinct Hodge Tate weights implies induction}, we may assume that $ V_{\pi,\iota} =\induction_{G_K} ^ {G_\Q} \rho_0 $, where $ K/\Q $ is either quartic or quadratic and $ \rho_0 $ is strongly irreducible and Hodge-Tate. If $ K $ is quartic,  $ \rho_0 $ corresponds to Hecke character of $K$ with algebraic infinity type via the usual recipe, and $\BC(\pi) $ is the automorphic induction of $\chi |\cdot|^{-1/2}$, a contradiction.

If $ K $ is quadratic, then $V_{\pi,\iota}\cong V_{\pi,\iota}\otimes\omega_{K/\Q} $ where $\omega_{K/\Q} $ is the quadratic character of $ G_\Q $ corresponding to $ K $. Hence $\BC (\pi)\cong\BC (\pi)\otimes\omega_{K/\Q} $ by strong multiplicity one for $\GL_4 $, so by \cite[Theorem 4.2(b)]{arthur1989basechange} $\BC (\pi) $ is an automorphic induction. This is a contradiction, so indeed $V_{\pi,\iota} $ is strongly irreducible, as desired.
\end{proof}
\subsubsection{}
For the next corollary, we use the following notation. Let $\varpi\in O_E$ be a uniformizer; then for an $O_E$-lattice $T$ in an $E$-vector space $V$, we write $T_n \coloneqq T/\varpi^n T$ for all $n \geq 1$.
\begin{corollary}\label{cor:Galois_coh_restr}
    Let $F$ be a number field, and let $(\rho, V)$
    be as in Lemma \ref{lemma distinct Hodge Tate weights implies induction} above. Assume, if $\dim V = 1$, that there exists a place $\p | p$ of $F$ such that  the character $\rho|_{G_{F_\p}}$ has nonzero Hodge-Tate weight. Then  for any Galois-stable $O_E$-lattice $T\subset V$, there exists a constant $C\geq 0$ such that $$\varpi^CH^1(F(\rho)/F, T_n) = 0$$ for all $n \geq 1$. 
\end{corollary}
\begin{proof}

 Without loss of generality, we can extend $E$ so that the conclusion of Lemma \ref{lemma distinct Hodge Tate weights implies induction} holds, for some finite extension $K/F$ and some  $\rho_0: G_K \to \GL_E(V_0)$.  Let $K^c$ be the Galois closure of $K$; then
 by inflation-restriction, it suffices to show \begin{equation}
    \label{first goal for restriction vanishing corollary}
    H^1(K^c(\rho)/K^c, T_n) \text{ is uniformly bounded in $n$.}
\end{equation}
 
 We label the $\Gal(K^c/F)$-conjugates of $(\rho_0,V_0)$ as $(\rho_i, V_i)$, for $0 \leq i < \dim V/\dim V_0,$ 
and let $$G\subset \prod_i \GL_E(V_i)$$ be the image of $G_{K^c}$ under $\rho$. We  assume without loss of generality that $T = \oplus T_i$ for Galois-stable $O_E$-lattices $T_i \subset V_i$. Hence to show (\ref{first goal for restriction vanishing corollary}), it suffices to show $H^1(G, T_{0,n})$ is uniformly bounded in $n$.
\setcounter{case}{0}
\begin{case}
    $Z_G$ contains an element $z$ that acts nontrivially on $V_0$.
\end{case}
Then by inflation-restriction, we have an exact sequence $$0 \to H^1(G/\langle z \rangle, T_{0,n}^z) \to H^1(G,T_{0,n}) \to H^1(\langle z \rangle, T_{0,n})^{G/\langle z \rangle}.$$ The outer terms are clearly uniformly bounded, so we are done in this case.
\begin{case}
    $Z_G$ acts trivially on $V_0$.
\end{case}
For this case, note that $G/ Z_G$ is a compact $p$-adic Lie group with semisimple Lie algebra; 
indeed, if $G_i = \rho_i(G_{K^c})$ with center $Z_{G_i} \subset G_i$, then we have an injection
\begin{equation}\label{eq:Lie_factors_ind}\Lie (G/Z_G) \hookrightarrow\prod_i \Lie(G_i/Z_{G_i}),\end{equation} and the semisimplicity of $\Lie(G/Z_G)$ follows by Goursat's Lemma and Lemma \ref{lemma with all the strongly irreducible statements}(\ref{lemma strongly irreducible part lie algebra is semisimple}).
 In particular, by \cite[Lemma B.1]{fakhruddin2021relative}, 
$H^1(G/Z_G, T_{0,n})$ is uniformly bounded in $n$.
By inflation-restriction again, it then suffices to show 
$$H^1(Z_G,T_{0,n})^{G/Z_G} = \Hom_G (Z_G, T_{0,n})$$ 
is uniformly bounded. Since $G$ acts trivially on $Z_G$ and $V_0$ is strongly irreducible,  it suffices to ensure $V_0$ is not the trivial representation of $G$. However, if this occurs then $\rho_0: G_K \to \GL_E(V_0)$ is a finite-order character, so all its Galois conjugates $\rho_i$ also have finite order. This would mean that $\rho|_{G_{F_v}}$ has trivial Hodge-Tate weights, which is ruled out by our assumptions on $\rho$. 

\end{proof}

\begin{lemma}\label{lemma with restriction to abelian Galois extension}
    Let $F$ be a number field, and let $\rho: G_F \to \GL_E(V)$ be as in Lemma \ref{lemma distinct Hodge Tate weights implies induction} above, so after an extension of scalars we can write 
    $$\rho\cong \Ind_{G_K}^{G_F}\rho_0$$ for a number field $K \supset F$ and a strongly irreducible representation $\rho_0: G_K \to \GL_E(V_0)$. Let $L$ be an abelian Galois extension of $F$ (possibly infinite) which is disjoint from $K$; then $\rho|_{G_L}$ is absolutely irreducible. 
\end{lemma}

\begin{proof}
    Let $K^c$ be the Galois closure of $K$; then  
    $$\rho|_{G_{K^c}} = \bigoplus \rho_0^\sigma|_{G_{K^c}},$$ where $\rho_0^\sigma$ runs over the distinct $\Gal(K^c/K)$-conjugates of $\rho_0$. In particular, if $v$ is a prime of $K^c$ lying over the $\p$ from Lemma \ref{lemma distinct Hodge Tate weights implies induction}, then    \begin{equation}\label{equation rho sigma and rho have distinct HT weights}
        \rho_0^\sigma|_{G_{K^c_v}}\text{ and }\rho_0 |_{G_{K^c_v}}\text{ have distinct Hodge-Tate weights if } \rho_0^\sigma\not\cong \rho_0.
    \end{equation}

    If $\rho|_{G_L}$ is reducible (after replacing $E$ by any finite extension), then
    \begin{align*}
        1 &< \dim_E \Hom_{E[G_L]} \left((\Ind_{G_K}^{G_F} \rho_0)|_{G_L}, (\Ind_{G_K}^{G_F} \rho_0)|_{G_L}\right) \\
        &= \dim_E \Hom_{E[G_L]} \left(\Ind_{G_{KL}}^{G_L} \rho_0|_{G_{KL}}, \Ind_{G_{KL}}^{G_L} \rho_0|_{G_{KL}}\right) \\
        &= \dim_E \Hom_{E[G_{KL}]} \left(\rho_0|_{G_{KL}}, \Res_{G_{KL}}^{G_L} \Ind_{G_{KL}}^{G_L} \rho_0|_{G_{KL}}\right) \\
        &\leq \dim_E\Hom_{E[G_{K^cL}]} \left(\rho_0|_{G_{K^cL}}, \bigoplus \rho_0^\sigma|_{G_{K^cL}}\right).
    \end{align*}
    In particular, we may fix $\sigma\in \Gal(K^c/K)$ such that $\rho_0^\sigma \not\cong \rho_0$ but $$\Hom_{E[G_{K^cL}]}(\rho_0|_{G_{K^cL}}, \rho_0^\sigma|_{G_{K^cL}}) \neq 0.$$
    We claim $\rho_0|_{G_{K^cL}}$ is absolutely irreducible; indeed, if $G_0 = \rho_0(G_{K^c})$, then $H = \rho_0 (G_{K^c L})$ is a normal subgroup of $G_0$ with abelian cokernel. Then $\Lie (H \cap \SL_E(V_0) )\subset \Lie( G_0 \cap \SL_E(V_0))$ has abelian cokernel, which implies $\Lie (H \cap \SL_E(V_0) )= \Lie( G_0 \cap \SL_E(V_0))$ by Lemma \ref{lemma with all the strongly irreducible statements}(\ref{lemma strongly irreducible part lie algebra is semisimple}, \ref{lemma strongly irreducible part center contains open subgroup of determinant}). Then $H \cap \SL_E(V_0)$ acts strongly irreducibly on $V_0$ by Lemma \ref{lemma with all the strongly irreducible statements}(\ref{lemma strongly irreducible part lie algebra representation is irreducible}), so \emph{a fortiori} $\rho_0|_{G_{K^cL}}$ is absolutely irreducible, as desired. 

    This implies that $\Hom_{E[G_{K^cL}]} (\rho_0, \rho_0^\sigma)$, which is nonzero by assumption, is in fact one-dimensional. It is also preserved by the natural action of $G_{K^c}$ on $\Hom_E(\rho_0, \rho_0^\sigma)$, because $G_{K^cL}$ is normal in $G_{K^c}$.  Hence $\Gal(K^c L/ K^c)$ acts on $\Hom_{E[G_{K^cL}]} (\rho_0, \rho_0^\sigma)$ by scalars, and in particular we conclude that 
    \begin{equation}
        \label{rho zero is a twist of rho zero sigma}
        \rho_0|_{G_{K^c}} \cong \rho_0^\sigma|_{G_{K^c}}\otimes \chi
    \end{equation}
    for a character $\chi$ of $\Gal(K^c L/K^c) \subset \Gal(L/F)$. Because $\Gal(K^c/K)$ acts trivially by conjugation on $\Gal(L/F)$ and $\sigma \in \Gal(K^c/K)$ has finite order, (\ref{rho zero is a twist of rho zero sigma}) implies 
    $$\rho_0 \cong \rho_0 \otimes \chi^n \text{ for some } n \geq 1,$$
    hence $\chi$ has finite order. But then (\ref{rho zero is a twist of rho zero sigma}) contradicts (\ref{equation rho sigma and rho have distinct HT weights}), so the lemma is proved.
\end{proof}

\subsection{Galois representations associated to Hilbert modular forms}
\subsubsection{}
Fix a totally real field $F$.
The following result is due to Nekovar:
\begin{thm}\label{thm:nekovar}
    Let $\pi$ be an automorphic representation of $\GL_2(\A_F)$ corresponding to a non-CM Hilbert modular form of weight $(2k_v)_{v|\infty}$, with each $k_v\geq 1$. If $E_0$ is a strong coefficient field for $\pi$, then there exists a subfield $E_1 \subset E_0$ and a quaternion algebra $D$ over $E_1$, along with a finite abelian extension $K$ of $F$, such that for all primes $\p$ of $E_0$, with residue characteristic $p$:
    \begin{enumerate}
        \item \label{thm:nekovar_one}The image of $\rho_{\pi,\p}$ contains an open subgroup of  $$H_\p\coloneqq \set{x \in \left(D\otimes_{E_1} E_{1,\p}\right)^\times\,:\, \Nm(x) \in \Q_p^\times},$$ where the embedding $H_\p \hookrightarrow \GL_2(E_{0,\p})$
is induced by the natural  embedding  $D\otimes_{E_1} E_{1,\p}\hookrightarrow D\otimes_{E_1} E_{0,\p} $ and an isomorphism $D\otimes_{E_1} E_{0,\p} \simeq M_2(E_{0,\p})$.
        \item The image $\rho_{\pi,\p}(G_K)$ is contained in $H_\p$.
    \end{enumerate}
    Moreover, for any  finite abelian extension $K'/K$ and all but finitely many $\p$, the image of $\rho_{\pi,\p}({G_{K'}})$ is a  conjugate of $$\set{g\in \GL_2(O_{E_{1,\p}})\,:\, \det g\in \Z_p^\times}.$$ 
\end{thm}
\begin{proof}
Let $K$ be the field written $F_\Gamma$ in \cite[Theorem B.5.2]{nekovar2012levelraising}. Then the first two claims, and the last part when $K' = K$, follow from \emph{loc. cit.}

For the general case, note that $\det\rho_{\pi,\p}|_{G_K} = \chi_{p,\cyc}$. If $K'/K$ is a finite abelian extension, restrict to those primes $\p$  such that $\rho_{\pi,\p}({G_K})$ contains $\SL_2(O_{E_1,\p})$ and $K'\cap \Q(\mu_{p^\infty}) = K \cap \Q(\mu_{p^\infty}).$ Then 
    $\det\rho_{\pi,\p} (G_K) = \det\rho_{\pi,\p}(G_{K'})$. On the other hand,
    $\rho_{\pi,\p}(G_{K'})$ is a normal subgroup of $\rho_{\pi,\p} (G_K)$ with abelian cokernel, which necessarily contains $\SL_2(O_{E_1,\p})$; it follows that  $\rho_{\pi,\p} (G_K) = \rho_{\pi,\p}(G_{K'})$, which completes the proof. 
\end{proof}
\subsubsection{}\label{subsubsec:nekovar_pair_Setup}
Now let $\pi_1$ and $\pi_2$ be two automorphic representations as in Theorem \ref{thm:nekovar}, with $E_0$ a common strong coefficient field. For the rest of this section we will always write $p$ for the residue characteristic of a prime $\p$ of $E_0$. Also,
let $E_1$, $D_1$, $K_1$, $H_{1,\p}$, $E_2$, $D_2$,  $K_2$, and $H_{2,\p}$  be as in the conclusion of Theorem \ref{thm:nekovar} applied to $\pi_1$ and $\pi_2$, respectively; we can and do fix a finite abelian extension $K \supset K_1\cdot K_2$ such that \begin{equation}\det\rho_{\pi_1,\p}|_{G_K} = \det\rho_{\pi_2,\p}|_{G_K} = \chi_{p,\cyc}\end{equation} for all $\p$,
and  the conclusions of Theorem \ref{thm:nekovar} hold for this choice of $K$. 

We will also consider the joint representation 
\begin{equation}\label{eq:jt_repn_app}\rho_{\pi_1,\pi_2, \p}:G_F \xrightarrow{\rho_{\pi_1,\p}\times \rho_{\pi_2,\p}} \GL_2(E_{0,\p}) \times \GL_2(E_{0,\p}).\end{equation}
\begin{lemma}\label{lem:nekovar_twist}
    Suppose there exists a prime $\p$ of $E_0$ and a finite extension $L/F$ such that
    $$\rho_{\pi_1,\p}|_{G_L} \cong \rho_{\pi_2,\p}|_{G_L}.$$
    Then $\pi_1$ is the twist of $\pi_2$ by a finite-order automorphic character of $\A_F^\times$.
\end{lemma}
\begin{proof}
By \cite[Theorem 2]{rajan2003strong}, we have $$\rho_{\pi_1,\p} = \rho_{\pi_2,\p} \otimes \chi$$ for some character $\chi$ of $G_F$, which is of finite order because it vanishes on $G_L$. 
Viewing $\chi$ as a finite-order character of $F^\times \backslash \A_F^\times$ via class field theory, we conclude that $\rho_{\pi_1, \p} = \rho_{\pi_2\otimes \chi, \p},$
and hence $\pi_1 = \pi_2 \otimes \chi$.
    \end{proof}
\subsubsection{}
If $\pi$ is as in Theorem \ref{thm:nekovar}, corresponding to a holomorphic Hilbert modular form $f$, then for any $\sigma \in \Gal(\overline \Q/\Q)$ we write $\pi^\sigma$ for the automorphic representation corresponding to $f^\sigma$. 
The following result generalizes
\cite{loeffler2017images} to the setting of Theorem \ref{thm:nekovar}. 

\begin{thm}\label{thm:nekovar_pair}
    In the setting of (\ref{subsubsec:nekovar_pair_Setup}), suppose $\pi_1 \neq \pi_2^\sigma\otimes \chi$ for any $\sigma\in \Gal(\overline \Q/\Q)$  and any finite-order automorphic character $\chi$ of $\A_F^\times$. 
Then:
\begin{enumerate}
    \item \label{thm:nekovar_pair_open}For all primes $\p$ of $E_0$, $\rho_{\pi_1,\pi_2,\p}(G_K)$ contains an open subgroup of $H_{1,\p}\times_{\Q_p^\times} H_{2,\p}$.
    \item \label{thm:nekovar_pair_full}For all but finitely many $\p$, the image of $\rho_{\pi_1,\pi_2,\p} (G_K)$ contains a conjugate of
$$\set{(g_1,g_2)  \in \GL_2(O_{E_1,\p})\times \GL_2(O_{E_2,\p}):\, \det g= \det h \in \Z_p^\times }.$$
\end{enumerate}
\end{thm}
\begin{proof}
The proof is analogous to \cite[Theorem 3.2.2, Proposition 3.3.2]{loeffler2017images}, where we replace Lemma 3.1.1 of \emph{op. cit.} with Lemma \ref{lem:nekovar_twist} above. For completeness, we recall the argument.
    Let  $G_\p = \rho_{\pi_1,\pi_2,\p}(G_K)$,
    so that we have a natural embedding
$$G_\p \hookrightarrow H_{1,\p}\times_{\Q_p^\times} H_{2,\p}.$$
It follows from Goursat's Lemma for Lie algebras that $\Lie G_\p$ is either $\Lie\left(H_{1,\p} \times_{\Q_p^\times} H_{2\p}\right) $ or $\Lie G_{1,\p} = \Lie H_{1,\p}$, diagonally embedded by an isomorphism $\Lie H_{1,\p} \isomorphism \Lie H_{2,\p}$ that preserves the linearized determinant maps to $\Q_p$. To prove (\ref{thm:nekovar_pair_open}), we assume that we are in the latter case, and aim to show $\pi_1$ is a conjugate twist of $\pi_2$.

By \cite[Lemma 1.1.4]{loeffler2017images}, any isomorphism $\Lie H_{1,\p} \isomorphism \Lie H_{2,\p}$ is induced by an isomorphism $i: E_{1,\p} \isomorphism E_{2,\p}$ and an $i$-linear isomorphism $D_1 \otimes_{E_1} E_{1,\p} \isomorphism D_2\otimes_{E_2} E_{2,\p} $. In particular, assuming without loss of generality that $E_0$ is Galois, there exists an automorphism $\sigma \in \Gal(E_0/\Q)$ that preserves $\p$ and induces $i: E_{1,\p} \isomorphism E_{2,\p}$. 
Since all automorphisms of $M_2(E_{0,\p})$ are inner, it follows from the description of the embedding in Theorem \ref{thm:nekovar}(\ref{thm:nekovar_one}) that, 
after conjugating $\rho_{\pi_2,\p}$, $\Lie G_\p\subset \mathfrak{gl}_2(E_{0,\p})\times \mathfrak {gl}_2(E_{0,\p})$ is contained in a subalgebra of the form $$\set{(X, \sigma X)\,: X\in \mathfrak{gl}_2(E_{0,\p})}.$$ Exponentiating, 
for some finite extension $L/K$ we have
$$\rho_{\pi_1,\p} |_{G_L} = \sigma \circ \rho_{\pi_2,\p} |_{G_L}.$$
Since $\sigma \circ \rho_{\pi_2,\p} = \rho_{\pi_2^\sigma, \p}$, Lemma \ref{lem:nekovar_twist} concludes the proof of (\ref{thm:nekovar_pair_open}).

For (\ref{thm:nekovar_pair_full}),  we  restrict our attention to those $\p$ such that $\rho_{\pi_i,\p} (G_K) = \GL_2(O_{E_i,\p})$ (after conjugating) for $i = 1,2$, which eliminates only finitely many primes of $E_0$ by Theorem \ref{thm:nekovar}.
Let $S$ be the set of primes $\p$ as above such that the conclusion of (\ref{thm:nekovar_pair_full}) does \emph{not} hold; we assume for contradiction that $S$ is infinite. 
By \cite[Proposition 3.2.1]{loeffler2017images}, for all $\p\in S$, we have an element $\sigma \in \Gal(E_0/\Q)$ preserving $\p$ such that, after  conjugating $\rho_{\pi_2,\p}$:
\begin{equation}\label{eq:tr_cong_pair}\rho_{\pi_1, \p} (g) = \pm \sigma \circ \rho_{\pi_2,\p} (g) \pmod \p,\;\; \forall g\in G_K.\end{equation}
If  $S$ is infinite,  then there exists a single $\sigma\in \Gal(E_0/\Q)$  such  that (\ref{eq:tr_cong_pair}) holds for infinitely many of the $\p\in S$ fixed by $\sigma$.

Let $\Sigma$ be the set of primes $v$ of $F$ such that either $\pi_{1,v}$ or $\pi_{2,v}$ is ramified, and for $v\not\in \Sigma$,  let $a_{1,v}$ and $a_{2,v}$ be the eigenvalues of the standard Hecke operator at $v$ on the spherical vectors of $\pi_{1,v}$ and $\pi_{2,v}$, respectively. By (\ref{eq:tr_cong_pair}), for all  $v\not\in \Sigma$ that split completely in $K/F$, $a_{1,v}^2 - \sigma(a_{2,v})^2 \in O_{E_0}$ is divisible by infinitely many primes $\p\in S$, hence vanishes. Now take a single $\p\in S$ fixed by $\sigma$ such that (\ref{eq:tr_cong_pair}) holds, and assume without loss of generality that $p\neq 2$. 
Let $K'$ be the compositum of the fixed fields of $\overline\rho_{\pi_1,\p}|_{G_K}$ and $\overline\rho_{\pi_2,\p}|_{G_K}$. Then for all $v\not\in \Sigma$ that split completely in $K'$, $$a_{1,v}\equiv \sigma (a_{2,v}) \not\equiv 0 \pmod \p,$$ so  the identity $a_{1,v}^2 - \sigma (a_{2,v})^2 = 0$ implies $a_{1,v} = \sigma (a_{2,v})$. In particular, the traces of $\rho_{\pi_1,\p} |_{G_{K'}}$ and $\sigma \circ \rho_{\pi_2,\p}|_{G_{K'}}$  coincide,  so 
we have
$\rho_{\pi_1,\p} |_{G_{K'}} \cong \sigma \circ \rho_{\pi_2,\p} |_{G_{K'}}.$ Now we can conclude by Lemma \ref{lem:nekovar_twist}.

\end{proof}

We also have the following complementary result.
\begin{corollary}\label{cor:nekovar_pair_fail}
    In the setting of (\ref{subsubsec:nekovar_pair_Setup}), there exists a finite abelian extension $L$ of $K$ with the following property:
    \begin{enumerate}
        \item\label{cor:nekovar_pair_fail_1} For all primes $\p$ of $E_0$, either $\rho_{\pi_1,\pi_2,\p}(G_L)$ is an open subgroup of $H_{1,\p} \times_{\Q_p^\times} H_{2,\p}$, or there exists an isomorphism $\sigma_\p: E_{1,\p} \isomorphism E_{2,\p}$  and a $\sigma_\p$-linear isomorphism $i_\p:D_1\otimes_{E_1} E_{1,\p} \isomorphism D_2\otimes_{E_2} E_{2,\p}$ such that $\rho_{\pi_1,\pi_2,\p}(G_L)$ is an open subgroup of 
$$H_{1,\p}\xhookrightarrow{\operatorname{id}, i_\p} H_{1,\p}\times_{\Q_p^\times} H_{2,\p}.$$
    \item    
\label{cor:nekovar_pair_fail_2}For all but finitely many $\p$, 
        either the image of $\rho_{\pi_1,\pi_2,\p}(G_L)$ is a conjugate
        of $$\set{(g_1,g_2)  \in \GL_2(O_{E_1,\p})\times \GL_2(O_{E_2,\p}):\, \det g= \det h \in \Z_p^\times },$$ or there exists an isomorphism $\sigma_\p: E_{1,\p} \isomorphism E_{2,\p}$ such that $\rho_{\pi_1,\pi_2,\p}(G_L)$ is a conjugate of $$\set{(g, \sigma_\p(g))\in \GL_2(O_{E_{1,\p}}) \times \GL_2(O_{E_{2,\p}})\,:\, g \in \GL_2(O_{E_{1,\p}}), \, \det g \in \Z_p^\times}.$$
    \end{enumerate}
\end{corollary}
\begin{proof}
By Theorem \ref{thm:nekovar_pair}, we may assume without loss of generality that $\pi_1 \cong \pi_2 ^\sigma \otimes \chi$ for some $\sigma \in \Gal(\overline \Q/\Q)$ and some finite-order automorphic character $\chi$ of $\A_F^\times$, which we also view as a character of $G_F$ via class field theory. 
Then for all primes $\p$ of $E_0$, 
$$\rho_{\pi_2,\p}\otimes_{E_{0,\p}, \sigma} E_{0, \sigma(\p)} \cong \rho_{\pi_2^\sigma,\sigma(\p)} \cong \rho_{\pi_1,\sigma(\p)} \otimes \chi.$$
By \cite[Theorem B.4.10]{nekovar2012levelraising}, for all but finitely many $\p$ the image of $\rho_{\pi_2^\sigma,\pi_2, \p}(G_{K_0})$ is one of the groups listed in part (\ref{cor:nekovar_pair_fail_2}), for a certain abelian extension $K_0$ of $F$ such that $\det_{\rho_{\pi_2}, \p}|_{G_{K_0}} = \chi_{p,\cyc}$. If
 $L$ is the compositum of $K_0$ with $K$ and with the fixed field of the kernel of $\chi$, 
it follows from the same argument as in Theorem \ref{thm:nekovar} that the image of $\rho_{\pi_2^\sigma,\pi_2, \p}(G_L) 
 = \rho_{\pi_1, \pi_2,\p}(G_L)$ coincides with that of $\rho_{\pi_2^\sigma,\pi_2, \p}(G_{K_0})$ for all but finitely many $\p$, and this proves (\ref{cor:nekovar_pair_fail_2}).
\end{proof}

\subsection{Large image for relevant representations}
\subsubsection{}

Fix a relevant automorphic representation $\pi$ of $\GSP_4(\A_\Q)$, with trivial central character and with strong coefficient field $E_0$. In this subsection, we prove some results on the image of the Galois representation $\rho_{\pi,\p}$ associated to $\pi$, with an eye towards studying the existence of admissible elements (Definition \ref{def:admissible}) and assumption  \ref{ass_A5_H1} from (\ref{subsubsec:where_rigid}). Throughout this section we write $p$ for the residue characteristic of a prime $\p$ of $E_0$. 

\begin{lemma}\label{lem:image_cube}
    Suppose $\pi$ is not endoscopic, and $\BC(\pi)$ is the symmetric cube lift of a non-CM automorphic representation $\pi_0$ of $\GL_2(\A_\Q)$.
    Consider the map of algebraic groups $$f = \Sym^3\otimes \debt^{-1}: \GL_2 \to \GSP_4.$$
  For all but finitely many primes $\p$ of $E_0$, the image of $\rho_{\pi,\p}$ contains a conjugate of $f(\GL_2(\Z_p))$. In particular, for all but finitely many $\p$, admissible elements exist for $\rho_{\pi,\p}$. 
\end{lemma}
\begin{proof}
By Lemma \ref{lem:sym_cube_relevant}, $\pi_{0,\infty}$ is discrete series of weight 2.
    Without loss of generality, extend $E_0$ so that it is also a strong coefficient field for $\pi_0$.
    Then for all primes $\p$ of $E_0$, 
 $\rho_{\pi,\p} = \Sym^3\rho_{\pi_0,\p} (-1)$. Comparing similitude characters, we see that the central character of $\pi_0$ is cubic; by twisting, we may assume without loss of generality that it is trivial. 
    Then the claim about the image of $\rho_{\pi,\p}$ follows from \cite[Theorem 3.1]{ribet1985largeimage}. 
    Restricting to these $\p$, if $p$ is sufficiently large we may fix $z\in \Z_p^\times$ with  
    \begin{equation}\label{eq:z_for_adm_largeimage}
            z\not\equiv \pm 1, \pm z^3, z^6, z^{-3}\pmod p, \;\; z^{12}\not\equiv 1 \pmod p.
    \end{equation}
  Then applying $f$ to the diagonal matrix $\begin{pmatrix} z & 0\\0 & z^2 
  \end{pmatrix}$, it follows that
  the image of $\rho_{\pi,\p}$ contains a matrix with eigenvalues $\set{1, z, z^2, z^3}$; in particular, admissible elements exist for $\rho_{\pi,\p}$. 
\end{proof}

\begin{lemma}\label{lem:large_image_general}
    Suppose $\pi$ is not endoscopic, and $\BC(\pi)$ is neither a (weak) symmetric cube lift nor a (weak) automorphic induction. For all but finitely many primes $\p$ of $E_0$:
    \begin{enumerate}
        \item\label{lem:lg_im_gen_1} The image of $\overline\rho_{\pi,\p}$ contains a conjugate of $\SP_4(\F_p)$.
        \item \label{lem:lg_im_gen_2}If $E_{0,\p} = \Q_p$, the image of $\overline\rho_{\pi,\p}$ is a conjugate of $\GSP_4(\F_p)$. 
    \end{enumerate}
    Moreover, for all primes $\p$ with $p > 3$, the Zariski closure (over $E_{0,\p}$) of the image of $\rho_{\pi,\p}$ is equal to $\GSP_4(E_{0,\p})$. 
\end{lemma}
\begin{proof}
    Part (\ref{lem:lg_im_gen_1}) is \cite[Theorem 1.2(ii)]{weiss2022images}. Part (\ref{lem:lg_im_gen_2}) follows immediately. 
  For the final claim, by  Proposition \ref{prop cases for subgroups of GSp4}(\ref{prop cases for subgroups of GSp4 part two}) and Corollary \ref{cor:general_type_irr}, it suffices to rule out the case that $\rho_{\pi,\p}$ factors through the image of the symmetric cube representation $\GL_2(E_{0,\p})/\mu_3(E_{0,\p}) \hookrightarrow \GSP_4(E_{0,\p})$. If so, then by a theorem of Tate \cite[Theorem 2.1.1]{patrikis2019variations},  after extending $E_0$ if necessary there exists a representation $\rho_0: G_\Q \to \GL_2(E_{0,\p})$ such that $\Sym^3\rho_0= \rho_{\pi,\p}(1) \otimes \tau$ for a character $\tau: G_\Q \to E_{0,\p}^\times$. Comparing the similitude factors, $(\det \rho_0)^3 = \chi_{p,\cyc}^3 \otimes \tau^2$; replacing $\rho_0$ with  $ \rho_0\otimes (\tau^{-1} \det\rho_0 / \chi_{p,\cyc})$, we can assume 
  $\Sym^3\rho_0 = \rho_{\pi,\p}(1)$. 
  Now it follows that 
 $\det \rho_0 / \chi_{p,\cyc}$ is cubic; so after twisting, we may assume without loss of generality that $\det \rho_0 = \chi_{p,\cyc}$. 
 For all but finitely many primes $\l$, we have $\rho_0(I_\l) \subset \mu_3$, so the determinant condition implies $\rho_0$ is unramified almost everywhere.
 We also claim $\rho_0|_{G_{\Q_p}}$ is de Rham: indeed, by 
  \cite[Corollary 3.2.13]{patrikis2019variations} there exists a character $\tau$ of $G_{\Q_p}$ such that $\rho_0\otimes \tau$ is de Rham, so $\rho_{\pi,\p}(1) \otimes \tau^3$ is de Rham, hence $\tau^3$ is de Rham; also $\det \rho_0\otimes \tau^2 = \chi_{p,\cyc} \otimes \tau^2$ is de Rham, so $\tau^2$ is de Rham. It follows that $\tau$ is de Rham, so $\rho_0$ is as well. 
  Since $\rho_0$ is also clearly odd, \cite[Theorem 1.0.4]{pan2022fontaine} implies $\rho_0$ arises from a modular form, hence $\BC(\pi)$ is a symmetric cube lift, and this concludes the proof of the final claim. 
\end{proof}
\subsubsection{}
For an automorphic representation $\pi_0$ of $\GL_2(\A_K)$ as in Theorem \ref{thm:nekovar} with $K/\Q$ real quadratic, let $\pi_0^\tw$ denote the $\Gal(K/\Q)$-twist.  We say $\pi_0$ is \emph{exceptional}
if there exists $\sigma\in \Gal(\overline \Q/\Q)$ and a finite-order automorphic character $\chi$ of $\A_K^\times$ such that 
$$\pi_0^\tw \cong \pi_0^\sigma \otimes \chi.$$ 
\begin{lemma}\label{lem:image_RQ}
    Suppose $\pi$ is not endoscopic, and $\BC(\pi)$ is the automorphic induction of a non-CM automorphic representation $\pi_0$ of $\GL_2(\A_K)$ with $K/\Q$ real quadratic. Then
     for all but finitely many primes $\p$ of $E_0$, the following hold.
    \begin{enumerate}
    \item  \label{lem:image_RQ_1}The image of $\rho_{\pi,\p}$ contains a conjugate of $\GL_2(\Z_p)$, embedded diagonally via  
        $$\GL_2\hookrightarrow \GL_2 \times_{\mathbb G_m}\GL_2 \hookrightarrow \GSP_4.$$
    \item\label{lem:image_RQ_2} If  $p$ splits in $K$ or $\pi_0$ is not exceptional,
   the image of $\rho_{\pi,\p}$ contains a conjugate of $\GL_2(\Z_p) \times_{\Z_p^\times} \GL_2(\Z_p)$.
    \end{enumerate}
\end{lemma}
\begin{proof}
    Recall from Lemma \ref{lem:RQ_induction_relevant} that $\pi_0$ is the automorphic representation associated to a Hilbert modular form of weights $(2,4)$. Assume without loss of generality that $E_0$ is Galois and is also a strong coefficient field for $\pi_0^\tw$; then $$\rho_{\pi,\p}|_{G_K} = \rho_{\pi_0,\p}\oplus \rho_{\pi_0^\tw, \p}.$$ Hence part (\ref{lem:image_RQ_1}) follows from    Corollary \ref{cor:nekovar_pair_fail}(\ref{cor:nekovar_pair_fail_2}). For part (\ref{lem:image_RQ_2}), the non-exceptional case is immediate from Theorem \ref{thm:nekovar_pair}(\ref{thm:nekovar_pair_full}), so suppose without loss of generality that $p$ splits in $K$. 
    Then for any $\sigma \in \Gal(E_{0,\p}/\Q_p)$, and any fixed embedding $j:K\hookrightarrow \Q_p$, the Hodge-Tate weights of $\rho_{\pi_0, \p}$ and $\sigma\circ \rho_{\pi_0,\p}$ with respect to $j$ coincide. This rules out that $\rho_{\pi_0^\tw, \p}|_{G_L} \cong \sigma \circ \rho_{\pi_0, \p}|_{G_L}$
for any finite extension $L/K$, so by Corollary \ref{cor:nekovar_pair_fail}(\ref{cor:nekovar_pair_fail_2}), we obtain (\ref{lem:image_RQ_2}).
\end{proof}

\begin{lemma}\label{lem:image_IQ}
    Suppose $\BC(\pi)$ is the automorphic induction of an automorphic representation $\pi_0$ of $\GL_2(\A_K)$  with $K/\Q$ imaginary quadratic, and $\pi_0$ is not an automorphic induction. Then for all but finitely many primes $\p$ of $E_0$:
    \begin{enumerate}
        \item \label{lem:image_IQ_1}The image of $\rho_{\pi,\p}$ contains a conjugate of
    $\SL_2(\Z_p)$, where $\SL_2 \hookrightarrow \SP_4$ is embedded into the Levi factor of a Siegel parabolic.
    \item \label{lem:image_IQ_2}
  If $p$ splits in $K$, then admissible elements exist for $\rho_{\pi,\p}$. 
    \end{enumerate}
\end{lemma}
\begin{proof}
    By Lemma \ref{lem:IQ_induction_relevant},
after possibly extending $E_0$ we can write
\begin{equation}\label{eq:IQ_decompose_foradm}\rho_{\pi,\p}|_{G_K} = \rho_{f,\p} \otimes \chi_\p \oplus \rho_{f,\p} \otimes \chi_\p^\tw\end{equation} for all $\p$, where:
\begin{itemize}
    \item $\rho_{f,\p}$ is the Galois representation attached to a classical modular form $f$ of weight $k = 2$ or 3 with coefficients in $E_0$; here the normalization is as usual, i.e. $\det \rho_{f,\p} = \omega_f \cdot \chi_{p,\cyc}^{k-1}$ where $\omega_f$ is the nebentype character of $f$, viewed as a character of $G_\Q$ via class field theory. 
    \item $\chi_\p$ is the $G_K$-representation attached to an algebraic Hecke character $\chi'$ of infinity type $(-1, 3-k)$.
    \item $\chi_\p^\tw$ is the $\Gal(K/\Q)$-twist of $\chi_\p$, which is also associated to the twist $(\chi')^\tw$.
\end{itemize}
The symplectic form in (\ref{eq:IQ_decompose_foradm}) is given by the natural pairing $$\rho_{f,\p} \otimes \rho_{f,\p} \otimes \chi_\p \otimes \chi_\p^\tw \to \det \rho_{f,\p} \otimes \chi_\p\chi_\p^\tw = \chi_{p,\cyc}.$$ 

Let $L$ be the fixed field of $\omega_f$, which is independent of $\p$.
After discarding finitely many primes $\p$ and changing basis, we may assume by \cite[Theorem 3.1]{ribet1985largeimage} that $$\rho_{f,\p}(G_L) = \set{g\in \GL_2(O)\,:\, \det g\in (\Z_p^\times)^{k-1}},$$
where $O$ is the ring of integers of 
a subfield of $E_{0,\p}$. 
Then
the Galois group $$\Gal(L(\rho_{f,\p}) \cap LK(\chi_{p,\cyc}, \chi_\p, \chi_\p^\tw, \omega_f)/L(\det(\rho_{f,\p})))$$ is a solvable quotient of $\SL_2(O)$, hence trivial if $p$ is sufficiently large; so we have
\begin{equation}\label{eq:disjoint_for_IQ_image}
    L(\rho_{f,\p}) \cap LK(\chi_{p,\cyc}, \chi_\p, \chi_\p^\tw, \omega_f) = L(\det(\rho_{f,\p})).
\end{equation}
In particular, this immediately implies (\ref{lem:image_IQ_1}).

For (\ref{lem:image_IQ_2}), we further restrict to those $\p$  such that $\chi_\p$ is crystalline at all primes above $p$. 
Fix a prime $v|p$ of $K$, and let $\overline v$ be its complex conjugate, with inertia subgroups $I_v, I_{\overline v} \subset G_K^{\ab}$; these are disjoint and each naturally identified with $\Z_p^\times$ since we are assuming $p$ is split (and unramified) in $K$. When restricted to inertia, the characters $\chi_p^\cyc$, $\chi_\p$, and $\chi_\p^\tw$ have the form:
\begin{align*}
    \chi_p^\cyc|_{I_v\times I_{\overline v}} \,: \,\Z_p^\times \times \Z_p^\times &\to \Z_p^\times \\
    (z_1, z_2) &\mapsto z_1z_2,\\
    \chi_\p|_{I_v\times I_{\overline v}}\,: \, \Z_p^\times \times \Z_p^\times &\to \Z_p^\times\subset O_{E,\p}^\times\\
    (z_1, z_2) &\mapsto z_1^{-1} z_2^{3-k} \\
     \chi_\p^\tw|_{I_v\times I_{\overline v}}\,: \, \Z_p^\times \times \Z_p^\times &\to \Z_p^\times\subset O_{E,\p}^\times\\
    (z_1, z_2) &\mapsto z_1^{3-k} z_2^{-1}.
\end{align*}
In particular, one can calculate that, for $p$ unramified in $L$, the image of
$$(\chi_\p, \chi_\p^\tw, \chi_{p,\cyc}): G_L \to O_{E_0,\p}^\times \times O_{E_0,\p}^\times \times \Z_p^\times$$
contains a subgroup of 
$$\set{(a, b, c) \in (\Z_p^\times)^3\,:\, ab= c^{2-k}}$$ 
with index at most 2. 
Comparing with (\ref{eq:disjoint_for_IQ_image}), we see that the image of \begin{equation}\label{eq:uber_dumb_IQ}(\rho_{f,\p}, \chi_{\p}, \chi_\p^\tw): G_L \to \GL_2(O) \times O_{E_0,\p}^\times \times O_{E_0,\p}^\times\end{equation} contains a subgroup of
\begin{equation}\label{eq:dumb_for_IQ}\set{(g, x, y)\in \GL_2(O) \times \Z_p^\times \times \Z_p^\times \,:\,\det g \in (\Z_p^\times)^{k-1}, \, (xy)^{(k-1)} = (\det g)^{2-k}}\end{equation}
with index at most $2$. 
Let $n \coloneqq 2(k-1)+ 2(k-2) = 4k - 6$, and set \begin{equation*}S_n \coloneqq \set{(g, c) \in \GL_2(\Z_p) \times \Z_p^\times\,:\, c, \det g\in (\Z_p^\times)^n}.\end{equation*}

Then for any $(g, c) \in S_n$, there exists $\lambda\in \Z_p^\times$ satisfying $$\lambda ^{n} = (\det g)^{k-2} c^{1-k}.$$
It follows from (\ref{eq:dumb_for_IQ}) that $(g\lambda^{-1}, \lambda, c\lambda)^2$ lies in the image of (\ref{eq:uber_dumb_IQ}); hence $(g, cg)^2$ lies  in the image of $(\rho_{f,\p} \otimes \chi_\p, \rho_{f,\p} \otimes \chi_\p^\tw): G_L \to \GL_2(E_{0,\p})\times \GL_2(E_{0,\p})$ for all $(g,c)\in S_n$. If $p$ is sufficiently large, this immediately implies that admissible elements exist for $\rho_{\pi,\p}$. 
\end{proof}

\begin{lemma}\label{lem:image_quartic}
 Suppose $\pi$ is not endoscopic, and $\BC(\pi)$ is the (weak) automorphic induction of a Hecke character $\chi_0$ of a quartic field $K \subset \C$. Then there exists a constant $n$ such that, for all but finitely many primes $\p$ of $E_0$, the following holds:
 \begin{enumerate}
     \item\label{lem:image_quartic_1} $\rho_{\pi,\p}(G_\Q)$ contains the scalar subgroup $(\Z_p^\times)^n \subset \GSP_4(\Z_p)$. 
     \item\label{lem:image_quart_2} If $p$  splits completely in the Galois closure $K^c$ of $K$, then
 $\rho_{\pi,\p} (G_\Q)$ contains a conjugate of 
 $$\set{\begin{pmatrix}
     x &&& \\ & y && \\ &&z & \\ &&& xz/y
 \end{pmatrix}\,:\, x, y, z \in (\Z_p^\times)^n}\subset \GSP_4(\Z_p).$$
 \end{enumerate}
\end{lemma}
\begin{proof}
Let $\chi \coloneqq \chi_0 |\cdot|^{1/2}$. From Theorem \ref{thm:rho_pi_LLC}(\ref{part:rho_pi_LLC1}), we see that the local component $\chi_v$ of $\chi$ takes algebraic values on $K_v^\times$ for cofinitely many primes $v$ of $K$; hence $\chi_\infty$ is algebraic \cite[Th\'eor\`eme 3.1]{waldschmidt1982caracteres}. Extending $E_0$ if necessary, for all primes $\p$ of $E_0$ we have the $\p$-adic character $\chi_\p$ associated to $\chi$, and $\rho_{\pi,\p} = \Ind_{G_K}^{G_\Q} \chi_\p$ for all $\p$. We restrict to those $\p$ such that $K^c/\Q$ is unramified at $p$, and $\chi_\p$ is crystalline at all primes $v|p$. 
The Hodge-Tate weights of $\chi_\p$ with respect to the four embeddings $i: K\hookrightarrow\overline\Q_p$ are $\set{-1,0,1,2}$ in some order by Theorem \ref{thm:rho_pi_LLC}(\ref{part:rho_pi_LLC_HT}); hence on the subgroup
$$\Z_p^\times \hookrightarrow ( O_K \otimes \Z_p)^\times\hookrightarrow G_K^{\ab},$$ $\chi_\p$ is given by $z \mapsto z^{-1 + 0 + 1 + 2} = z^2$.
In particular, on the subgroup
$$\Z_p^\times \hookrightarrow (O_{K^c}\otimes \Z_p)^\times \hookrightarrow G_{K^c}^{\ab},$$ $\chi_\p$ is given by $z\mapsto z^{2[K^c:K]}$. The same is true for all $G_\Q$-conjugates of $\chi_\p$, so the image of $\rho_{\pi,\p}|_{G_{K^c}}$ contains the scalar subgroup $(\Z_p^\times)^{2[K^c:K]}$, proving (\ref{lem:image_quartic_1}).

For (\ref{lem:image_quart_2}), 
we decompose 
\begin{equation}
    \rho_{\pi,\p}|_{G_{K^c}} = \chi_1 \oplus \chi_2 \oplus \chi_3 \oplus \chi_4,
\end{equation}
where all of the characters $\chi_j$ are Galois conjugates of $\chi_\p|_{G_{K^c}}$ and
\begin{equation}\label{eq:chis_multiply}
    \chi_1\cdot \chi_2 = \chi_3\cdot \chi_4 = \chi_{p, \cyc}.
\end{equation}
For each $g\in G_\Q$, $\rho_{\pi,\p}(G_{K^c})$
contains the image of
\begin{equation}\label{eq:for_quartic_image}\begin{pmatrix}
    g\cdot \chi_1 &&& \\ & g\cdot \chi_2 && \\ && g\cdot \chi_3 \\ &&& g\cdot \chi_4
\end{pmatrix}.\end{equation}

Fix an embedding $i: K^c\hookrightarrow\Q_p$, and
for any Hodge-Tate character $\rho$ of $G_{K^c}$, let $\HT(\rho)$ denote the Hodge-Tate weight with respect to $i$. Let $I_p \subset G_{K^c}$ be the inertia subgroup for the prime induced by $i$. 
In particular, restricting (\ref{eq:for_quartic_image}) to $I_p$ and using that each $\chi_j$ is crystalline at primes above $p$, $\rho_{\pi,\p}(G_{K^c})$
contains the image of 
\begin{equation*}
    \begin{split}
        \Z_p^\times &\to \GL_4(\Z_p) \\
        z &\mapsto \begin{pmatrix}
            z^{\HT(g\cdot \chi_1)}&&& \\ &             z^{\HT(g\cdot \chi_2)} && \\ &&             z^{\HT(g\cdot \chi_3)} & \\ &&&            z^{\HT(g\cdot \chi_4)}
        \end{pmatrix}.
    \end{split}
\end{equation*}
Let \begin{equation}\label{eq:sublattice_for_quartic}L \subset \set{(x, y, z, w)\in \Z^4\, : \, x + y= z + w}\end{equation}be the sublattice  spanned by the vectors $(\HT(g\cdot\chi_1), \HT(g\cdot\chi_2), \HT(g\cdot\chi_3), \HT(g\cdot\chi_4))$ for $g\in G_\Q$.
\begin{claim}
For a constant $n\geq 1$ independent of $\p$,     the lattice $L$ contains $$n\cdot \set{(x, y, z, w)\in \Z^4\, : \, x + y= z + w}.$$
\end{claim}
Note that the claim implies the lemma, because, as  long as $p$ is sufficiently large, there exists $z\in (\Z_p^\times)^n$ 
satisfying (\ref{eq:z_for_adm_largeimage}); 
 an element $h\in G_\Q$ such that $\rho_{\pi,\p} (h)$ has eigenvalues $\set{1, z^3, z, z^2}$ is admissible for $\rho_{\pi,\p}$, and the claim implies such elements exist.

 Now we prove the claim.  Let $\pr: \Z^4 \to \Z^3$ be the projection onto the first three factors, and note that it suffices to show $\pr(L)$ contains $n\cdot \Z^3$. 
 Without loss of generality, suppose the Hodge-Tate weights of $\chi_1,\chi_2$, $\chi_3$, and $\chi_4$ are 1, 0, 2, and $-1$ (in order). 
 Because the action of $G_{K^c}$ on the set $\set{\chi_1,\chi_2,\chi_3,\chi_4}$ is transitive, for each $j \in 1,\ldots, 4$ we have some $g_j \in G_\Q$ such that $\HT(g_j \chi_j) = 1$. In particular, using (\ref{eq:chis_multiply}), $\pr(L)$ contains $(1,0,2)$; a vector  $e = (0,1, 2)$ or  $ e = (0, 1, -1)$; and a vector  $ f = (2, -1, 1)$ or $ f= (-1, 2, 1)$. In particular, the set $\set{(1, 0, 2), e, f}$ is always linearly independent; and, since there are only four total possibilities for this set, there exists $n \in \Z$ such that the $\Z$-span of $(1,0,2)$, $e$, and $f$ always contains $n\Z$. 

\end{proof}
Now we are ready to consider assumption \ref{ass_A5_H1} from the main text (see (\ref{subsubsec:where_rigid})). 
\begin{thm}\label{thm:appendix_A5}
    Let $\pi$ be a relevant, non-endoscopic automorphic representation of $\GSP_4(\A_\Q)$, with strong coefficient field $E_0$. Then \ref{ass_A5_H1} holds for all but finitely many primes $\p$ of $E_0$.
\end{thm}
The theorem is also true in the endoscopic case, but not used in the main text; the proof uses Lemma \ref{lem:endoscopic_independence} below.
\begin{proof}
This is an immediate consequence of
 Lemmas  \ref{lem:image_cube} through \ref{lem:image_quartic}.

\end{proof}

\begin{prop}\label{prop:adm_IIa}
    Suppose $\pi$ is not endoscopic, and there exists a prime $\l$ such that $\pi_\l$ is of type IIa. Then for all but finitely many primes $\p$ of $E_0$, admissible elements exist for $\rho_{\pi,\p}$.
\end{prop}
\begin{proof}
    The Weil-Deligne representation $\rec_{\operatorname{GT}} (\pi_\l)$ is tamely ramified; under the embedding $\GSP_4\hookrightarrow\GL_4$ from (\ref{subsubsec:coordinates_gsp2n}), it is given by
    $$\Frob_\l = \begin{pmatrix} \pm\l^{1/2} & && \\ & \alpha && \\ &&\pm\l^{-1/2} & \\ &&& \l/\alpha\end{pmatrix}\in \GSP_4(\C), \;\; N = \begin{pmatrix} &  & 1 &  \\  &  &  & \\ &  &  &  \\  & & &  \end{pmatrix}\in \GSP_4(\C)$$
    By the purity assertion in Theorem \ref{thm:rho_pi_LLC}(\ref{part:rho_pi_LLC1}) (for any prime $\p$ of $E_0$), we know $|\alpha| = 1$. Extend $E_0$ if necessary so that $\alpha^2\in E_0$. Then for all but finitely many primes $\p$ of $E_0$, we have:
    \begin{align*}
        \l^8 &\not\equiv 1 \pmod \p,\\
        \alpha^2\l &\not\equiv \pm 1, \pm \l^2, \l^{-2}, \l^4\pmod \p.
    \end{align*}
Suppose $\p$ satisfies the above conditions, and let $\Frob_\l\in G_\Q$ be any lift of Frobenius. 
   By Theorem \ref{thm:rho_pi_LLC}(\ref{part:rho_pi_LLC1}),  $\rho_{\pi,\p}(\Frob_\l^2)$ has eigenvalues  $\set{\l^2, \alpha^2\l, 1, \l^2/\alpha^2}$, hence $\Frob_\l^2$ is an admissible element for $\rho_{\pi,\p}$. 
\end{proof}

Combining Lemmas \ref{lem:image_cube} through \ref{lem:image_quartic} with Proposition \ref{prop:adm_IIa}, we obtain:
\begin{thm}\label{thm:when_adm_primes}
    Let $\pi$ be a relevant, non-endoscopic automorphic representation of $\GSP_4(\A_\Q)$, with strong coefficient field $E_0$.
        There is a set $S$ of rational primes of positive Dirichlet density such that for all $p \in S$ and all $\p|p$, admissible elements exist for $\rho_{\pi,\p}$.
    There exists such an $S$ containing all but finitely many $p$ if $\pi$ satisfies any of the following:
    \begin{enumerate}[(i)]
        \item          There exists a prime $\l$ such that $\pi_\l$ is of type IIa.
        \item $\BC(\pi)$ is a symmetric cube lift.
        \item $\BC(\pi)$ is the automorphic induction of a non-CM automorphic representation $\pi_0$ of $\GL_2(\A_K)$ with $K$ real quadratic, and $\pi_0^\tw \neq \pi_0^\sigma \otimes \chi$ for all $\sigma\in \Gal(\overline \Q/\Q)$ and all quadratic Hecke characters $\chi$ of $K$.
    \end{enumerate}
\end{thm}
\qed

Finally, we handle the endoscopic case separately. 
\begin{prop}\label{prop:adm_endosc}
    Suppose $\pi$ is endoscopic, associated to a pair $(\pi_1,\pi_2)$ of automorphic representations of $\GL_2(\A_\Q)$ (in any order). Then:
    \begin{enumerate}
        \item\label{prop:adm_endoscopic_one} If $\pi_1$ does not have CM, then for all but finitely many $p$ and all $\p|p$, there exist admissible primes for $\rho_{\pi,\p}$ that are BD-admissible for $\rho_{\pi_1,\p}$. 
        \item\label{prop:adm_endo_two} If $\pi_1$ has CM by a field $K$ and $\pi_2$ does not have CM by $K$, then for all but finitely many $p$ split in $K$ and all $\p|p$, there exist admissible primes for $\rho_{\pi,\p}$ that are BD-admissible for $\rho_{\pi_1,\p}$. 
    \end{enumerate}
\end{prop}
\begin{proof}
Let $S$ be the set of all rational primes in case (\ref{prop:adm_endoscopic_one}) and all rational primes $p$ split in $K$ in case (\ref{prop:adm_endo_two}). 
Then there exists a constant $n \geq 1$ such that,  for all but finitely many $p \in S$ and all $\p|p$, $\rho_{\pi_1,\p}(G_\Q)$ contains the diagonal subgroup
$$\set{\begin{pmatrix}
    x & \\ & y 
\end{pmatrix}\,:\, x, y\in (\Z_p^\times)^n}.$$
In the non-CM case this follows from Theorem \ref{thm:nekovar} (with $n = 1$), and in the CM case it follows from either \cite[Proposition B.6.3]{nekovar2012levelraising} or a similar argument to Lemma \ref{lem:image_quartic}. 
On the other hand, there exists a constant $n \geq 1$ such that, for all but finitely many $p$ and all $\p|p$, $\rho_{\pi_2,\p}(G_\Q)$ contains the scalar subgroup $(\Z_p^\times)^n\subset \GL_2(\Z_p)$.  

By Lemma \ref{lem:endoscopic_independence} below, after enlarging $n$ if necessary, for all but finitely many $p\in S$ and all $\p|p$, $\left(\rho_{\pi_1,\p}\times \rho_{\pi_2,\p}\right)(G_\Q)$ contains
$$\set{\begin{pmatrix}
    x& 0\\ 0 & y 
\end{pmatrix}, \begin{pmatrix}
     z & 0 \\ 0 & z
\end{pmatrix}\,:\, x,y,z\in (\Z_p^\times)^n, z^2 = xy};$$
and this implies the proposition.
\end{proof}


\begin{lemma}\label{lem:endoscopic_independence}
    Suppose $\pi$ is endoscopic, associated to a pair $(\pi_1,\pi_2)$ of automorphic representations of $\GL_2(\A_\Q)$ which do not both have CM by the same imaginary quadratic field. Then as $\p$ varies over primes of $E_0$, $\Q(\overline \rho_{\pi_1,\p}) \cap \Q(\overline \rho_{\pi_2,\p})$ has bounded degree over $\Q(\mu_p)$.
\end{lemma}
\begin{proof}
    If both $\pi_1$ and $\pi_2$ are non-CM, then the lemma follows from \cite[Theorem 3.2.2]{loeffler2017images}, or equivalently from Theorem \ref{thm:nekovar_pair}(\ref{thm:nekovar_pair_full}) above. 
Now suppose $\pi_1$ is CM and $\pi_2$ is not. Since $\SL_2(\F_q)$ is simple for $q$ sufficiently large and $\overline\rho_{\pi_1,\p}$ is dihedral for all $\p$, it follows from \cite[Theorem 3.1]{ribet1985largeimage} that $\Q(\overline\rho_{\pi_1,\p})\cap \Q(\overline \rho_{\pi_2,\p}) = \Q(\mu_p)$ for all but finitely many $\p$. 

If $\pi_1$ and $\pi_2$ are CM with respect to two different imaginary quadratic fields $K_1$ and $K_2$, fix elements $\tau_1, \tau_2\in G_\Q$ such that $\tau_i$ is a complex conjugation on $K_i$ but acts trivially on $K_j$, $i \neq j$. 
The abelian group $H \coloneqq \Gal(\Q(\overline\rho_{\pi_1,\p})/K_1)$
is a subgroup of  $k^\times$ for a quadratic \'etale algebra $k$ over the residue field of $\p$; in particular, $H$ is the product of at most two cyclic groups. Let $G\coloneqq \Gal(K_1K_2 (\overline \rho_{\pi_1,\p}) \cap K_1K_2 (\overline \rho_{\pi_2,\p})/K_1K_2)$. Because $G$ is a subquotient of $H$, the
conjugation actions on $G$ of both $\tau_1$ and $\tau_1\tau_2$ are by inversion, so the conjugation action of $\tau_2$ is trivial; arguing symmetrically, the conjugation action of $\tau_1$ is also trivial, so $G$ is 2-torsion and generated by at most two elements. We conclude that $|G|$ is uniformly bounded, which implies the lemma.
\end{proof}

\subsection{Complements for the second reciprocity law}
In this subsection, we prove some auxiliary results needed in \S\ref{sec:2ERL_contd}. Let us fix a relevant 
automorphic representation $\pi$ of $\GSP_4$, and an isomorphism $\iota: \overline \Q_p \isomorphism \C$ with  $p > 3$. 
\begin{lemma}\label{lem:unentangled}
   Let $\tau$ be a cuspidal automorphic representation of $\GL_2(\A_\Q)$ whose archimedean component is discrete series of even weight $k\geq 2$. Let $F$ be a number field such that $\tau$ does not have CM by any quadratic imaginary subfield $K\subset F(\rho_{\pi,\iota})$.
      If $F(\rho_{\pi,\iota}) \cap F(\ad^0\rho_{\tau,\iota})$ is infinite, 
    then $F(\ad^0\rho_{\tau,\iota}) \subset F(\rho_{\pi,\iota})$,
and moreover one of
 the following occurs:
        \begin{enumerate}[label = (\roman*)]
            \item$\pi$ is not  endoscopic, and if $g\in G_\Q$ is admissible for $\rho_{\pi,\iota}$, then $\rho_{\tau,\iota}(g^2)$ has distinct eigenvalues.
            \item $\pi$ is  endoscopic  associated to a pair $(\pi_1,\pi_2)$ of automorphic representations of $\GL_2(\A_\Q)$, and for $j = 1$ or 2, $\pi_j\cong \tau^\sigma\otimes \chi$ for some finite-order Hecke character $\chi$ and automorphism $\sigma\in \operatorname{Aut}(\overline\Q/\Q)$. 
        \end{enumerate}
\end{lemma}
\begin{proof}
    First, we claim that  $\tau$ does not have CM. Indeed, if $\tau$    
    has CM by a quadratic imaginary field $K$, then it is easy to check that any infinite subfield of $F(\ad^0\rho_{\tau,\iota})$ contains $K$;  so if $F(\rho_{\pi,\iota}) \cap F(\ad^0\rho_{\tau,\iota})$ is infinite then $K \subset F(\rho_{\pi,\iota})$, which contradicts the hypotheses of the lemma.

    Since $\tau$ does not have CM, $\rho_{\tau,\iota}$ is strongly irreducible by Theorem \ref{thm:nekovar}. Hence 
 by Lemma \ref{lemma strongly irreducible in dimension 2 or 4}, the normal, infinite-index subgroup $$\Gal\left(F(\ad^0\rho_{\tau,\iota})/F(\ad^0\rho_{\tau,\iota})\cap F(\rho_{\pi,\iota})\right) \mathrel{\unlhd} \Gal\left(F(\ad^0\rho_{\tau,\iota})/F\right)$$ must be trivial; equivalently, we have $F(\ad^0\rho_{\tau,\iota}) \subset F(\rho_{\pi,\iota})$.

Suppose first that $\pi$ is not endoscopic, so $V_{\pi,\iota}$ is absolutely irreducible by Lemma \ref{lem:reducible_endoscopic}.
 Without loss of generality, we assume that $F$ is Galois and that \begin{equation}\label{eq:for_complements}V_{\pi,\iota}|_{G_F} = \bigoplus_i V_i\end{equation}for some strongly irreducible representations $V_i$, all of the same dimension $n$ (Lemma \ref{lemma distinct Hodge Tate weights implies induction}). Let $G= \rho_{\pi,\iota} (G_F)$, and let $H = \ad^0\rho_{\tau,\iota}(G_F)$; then the inclusion $F(\ad^0\rho_{\tau,\iota}) \subset F(\rho_{\pi,\iota})$ corresponds to a surjection $G/Z_G \twoheadrightarrow H$ (recall here that $H$ has trivial center by Lemma \ref{lemma with all the strongly irreducible statements}(\ref{lemma strongly irreducible part trivial center})).
 In particular, $n > 1$. 
 Write $\mathfrak g = \Lie (G/Z_G)$ and $\mathfrak h = \Lie H$, and recall that $\mathfrak h$ is simple by Theorem \ref{thm:nekovar} and Corollary \ref{corollary subalgebra of simple Lie algebra with base change}(\ref{corollary subalgebra of simple Lie algebra with base change part two}). We have a surjection
 \begin{equation}\label{eq:for_complements_2}
     \mathfrak g\twoheadrightarrow \mathfrak h 
 \end{equation}
which identifies $\mathfrak h$ with a simple factor of $\mathfrak g$.

Suppose first that $n = 2$.
Then the decomposition (\ref{eq:for_complements}) has exactly two factors, and $\mathfrak g$ is either simple, or isomorphic to $\mathfrak g_0 \times \mathfrak g_0$, with $\mathfrak g_0$ simple and the factors interchanged by the action of $G_\Q$. Since (\ref{eq:for_complements_2}) is $G_\Q$-equivariant, $\mathfrak g$ is simple, and (\ref{eq:for_complements_2}) is an isomorphism. 
If $\rho_{\tau, \iota} (g^2)$ has only one eigenvalue for some $g\in G_\Q$, then $g^2$ acts unipotently on $\mathfrak h$, hence also on $\mathfrak g$. 
Since $g^2$ is a square, it preserves the decomposition (\ref{eq:for_complements}), so by
  Lemma \ref{lemma with all the strongly irreducible statements}(\ref{lemma strongly irreducible part unipotents have only one eigenvalue}), we see that $g^2$ has at most two eigenvalues on $V_{\pi,\iota}$, which contradicts $g$ being admissible.

We are now reduced to the case $n = 4$, i.e. $\BC(\pi)$ is not an automorphic induction.
Hence $\mathfrak g$ is simple by Proposition \ref{prop cases for subgroups of GSp4}(\ref{prop cases for subgroups of GSp4 part one}) and Corollary \ref{corollary subalgebra of simple Lie algebra with base change}(\ref{corollary subalgebra of simple Lie algebra with base change part two}), and again (\ref{eq:for_complements_2}) is an isomorphism. 
If $\rho_{\tau,\iota}(g^2)$ had only one eigenvalue, then $g^2\in G_\Q$ would act unipotently on $\mathfrak h$, hence also on $\mathfrak g$; but by Lemma \ref{lemma with all the strongly irreducible statements}(\ref{lemma strongly irreducible part unipotents have only one eigenvalue}), this contradicts the admissibility of $g$.

It remains to consider the case when
 $\pi$ is  endoscopic  associated to a pair $(\pi_1,\pi_2)$ of cuspidal automorphic representations of $\GL_2(\A_\Q)$. We let $G_{\pi} = \rho_{\pi,\iota}(G_\Q),$ $G_{\pi_j} = \rho_{\pi_j,\iota}(G_\Q)$, $\mathfrak g_\pi = \Lie(G_\pi/Z_{G_\pi})$,  $\mathfrak g_{\pi_j} = \Lie(G_{\pi_j}/Z_{G_{\pi_j}})$, for $j= 1,2$.
 The assumption $F(\ad^0\rho_{\tau,\iota}) \subset F(\rho_{\pi,\iota})$ implies that   $G_\Q$ does not have open image in the product $G_\pi/Z_{G_\pi}\times  H$. By Goursat's Lemma, the Lie algebra of the image of $G_\Q$ is the graph of an isomorphism between simple factors of $\mathfrak g_\pi\subset \mathfrak g_{\pi_1}\oplus \mathfrak g_{\pi_2}$ and  $\mathfrak h$. Since $\tau$ is non-CM, we conclude that for $j = 1$ or 2, $\pi_j$ is non-CM and $G_\Q$ has non-open image in $G_{\pi_j}/Z_{\pi_j} \times H$. 
Hence by \cite[Proposition 3.3.2]{loeffler2017images}, there exists an automorphism $\sigma \in \operatorname{Aut}(\overline\Q/\Q)$ and a Hecke character $\chi$ such that  $\tau \cong \pi_j^\sigma \otimes \chi$, for $j = 1$ or 2; this concludes the proof of the lemma. 
    
\end{proof}
For the rest of the section, we fix a strong coefficient field $E_0$ for $\pi$ and let $\p$ be the prime of $E_0$ induced by $\iota$. 
\begin{lemma}\label{lem:inflation with tau}
   Let $\tau$ be a cuspidal automorphic representation of $\GL_2(\A_\Q)$ whose archimedean component is discrete series of even weight $k \geq 2$. If $\tau$ does not have CM by any quadratic imaginary subfield $K\subset \Q(\rho_{\pi,\iota})$,   then for any number field $F$ and any $O_\p$-stable lattice $T_\pi \subset V_{\pi,\p}$, $H^1(\Gal(F(\rho_{\pi,\iota}, \ad^0\rho_{\tau,\iota})/\Q), T_\pi)$ is finite.
\end{lemma}
\begin{proof}
By inflation-restriction, we may assume without loss of generality that $F = \Q$. Applying Corollary \ref{cor:Galois_coh_restr} and inflation-restriction again,
to  prove the lemma it suffices to show
    \begin{equation}\label{goal for vanishing of Lie group cohomology for 2nd ERL chebotarev}
H^1(\Gal(\Q(\rho_{\pi,\iota},\ad^0\rho_{\tau,\iota})/\Q(\rho_{\pi,\iota}), T_\pi) = \Hom_{G_\Q}(\Gal(\Q(\rho_{\pi,\iota},\ad^0\rho_{\tau,\iota})/\Q(\rho_{\pi,\iota})), T_\pi) = 0.
    \end{equation}
    By Lemma \ref{lem:unentangled}, we may assume without loss of generality that $\Q(\rho_{\pi,\iota})\cap \Q(\ad^0\rho_{\tau,\iota})$ is finite. Since $\chi_{p,\cyc}$ has infinite order, there exists $g\in G_\Q$ such that $\ad^0\rho_{\tau,\iota}(g) = 1$ and $ \chi_{p,\cyc}(g)$ has infinite order, meaning in particular that $\rho_{\pi,\iota} (g) \neq 1$. Then
    $g$ acts trivially by conjugation on $\Gal\left(\Q(\rho_{\pi,\iota}, \ad^0\rho_{\tau,\iota})/\Q(\rho_{\pi,\iota})\right)\hookrightarrow \ad^0\rho_{\tau,\iota}(G_\Q).$ In particular,  any $G_\Q$-invariant homomorphism $h: \Gal\left(\Q(\rho_{\pi,\iota},\ad^0\rho_{\tau,\iota})/\Q(\rho_{\pi,\iota})\right)\to T_\pi$ has image contained in $T_{\pi}^{g=1}\subsetneq T_\pi$. If $\pi$ is non-endoscopic, this shows $h = 0$ by Lemma \ref{lem:reducible_endoscopic}; if $\pi$ is endoscopic associated to $(\pi_1,\pi_2)$, the same argument applies because we cannot have $\rho_{\pi_1,\iota}(g) = 1$ or $\rho_{\pi_2,\iota} (g) = 1$ under the assumption that $\chi_{p,\cyc}(g)$ has infinite order. 
This shows (\ref{goal for vanishing of Lie group cohomology for 2nd ERL chebotarev}).
\end{proof}
\begin{prop} \label{prop:choosing_g_2ERL}
    Let $\pi$, $\iota$, $E_0$, and $\p$
    be as above with $\pi$ non-endoscopic, and suppose admissible primes exist for $\rho_\pi = \rho_{\pi,\p}$.
Suppose given the following data:
\begin{itemize}[label = $\circ$]
    \item A quadratic field $F\not\subset \Q(\rho_{\pi})$.
    \item A cuspidal automorphic representation $\tau$ of $\GL_2(\A_\Q)$ whose archimedean component is discrete series of even weight $k \geq 2$, such that $\tau$ does not have CM by any quadratic field $K \subset F(\rho_{\pi})$. 
\item A $G_\Q$-stable $O_\p$-lattice $T_\pi \subset V_{\pi,\p}$, and a non-torsion cocycle $c\in H^1(\Q, T_{\pi})$.
\end{itemize}
    Then there exists an element $g\in G_\Q$ such that:
    \begin{enumerate}
        \item \label{first property choosing g for 2nd ERL prop}$g$ is is admissible for $\rho_{\pi}$ and has nontrivial image in $\Gal(F/\Q)$.
        \item \label{second property choosing g for 2nd ERL prop}$\rho_{\tau, \iota}(g^2)$ has distinct eigenvalues.
        \item \label{third property choosing g for 2nd ERL prop}$c(g)$ has nonzero component in the 1-eigenspace for $g$. 
    \end{enumerate}    
\end{prop}
Note the last condition is independent of the choice of cocycle representative for $c$.
\begin{proof}
First choose $g$ satisfying (\ref{first property choosing g for 2nd ERL prop}), and with the additional property that $g$ has trivial image in $\Gal(K/\Q)$ 
if $\tau$ has CM by a quadratic field $K$. (This choice is possible because we have $F\not\subset \Q(\rho_{\pi})$ and $K\not\subset F(\rho_{\pi})$.) 
\begin{claim}
    There exists $h\in G_{F(\rho_{\pi})}$ such that $hg$ satisfies  (\ref{second property choosing g for 2nd ERL prop}).
\end{claim}
\begin{proof}[Proof of claim]
    If $F(\rho_{\pi}) \cap F(\ad^0 \rho_{\tau,\iota})$ is infinite then taking $h = 1$ suffices by  Lemma \ref{lem:unentangled}, so we may assume without loss of generality that $F(\rho_{\pi}) \cap F(\ad^0 \rho_{\tau,\iota})$ is finite. 
If $\tau$ is non-CM, then because $F(\rho_{\pi}) \cap F(\ad^0 \rho_{\tau,\iota})$ is finite, 
    Theorem \ref{thm:nekovar} implies that the image of $\rho_{\tau,\iota}|_{G_{F(\rho_{\pi})}}$ contains a compact open subgroup of $\set{x \in D^\times\,: \, \Nm(x) \in \Q_p^\times} \hookrightarrow \GL_2(\overline\Q_p)$, for a quaternion algebra $D$ over a finite extension $E$ of $\Q_p$. Since $x\rho_{\tau,\iota}(g) x\rho_{\tau,\iota}(g) $ having distinct eigenvalues is an open condition on $x\in D^\times$, the claim follows when $\tau$ is non-CM.

    If on the other hand $\tau$ has CM by an imaginary quadratic field $K$, then because $\rho_{\tau,\iota}|_{G_{\Q_p}}$ has distinct Hodge-Tate weights, there exists $h_0 \in G_K$ such that $\rho_{\tau,\iota}(h_0)$ has eigenvalues whose ratio is of infinite order. After replacing $h_0$ with a finite power, it acts trivially on $F(\rho_{\pi}) \cap F(\ad^0\rho_{\tau,\iota})$; thus there exists $h\in G_{K\cdot F(\rho_{\pi})}$ such that $\rho_{\tau,\iota}(h^2)$ has distinct eigenvalues. 
     Since $g$ has trivial image in $\Gal(K/\Q)$, $\rho_{\tau,\iota}(g)$ and $\rho_{\tau,\iota}(h)$ commute; in particular, if $\rho_{\tau,\iota}(g^2)$ is scalar, then  $\rho_{\tau,\iota}(hghg) = \rho_{\tau,\iota}(h^2)\rho_{\tau,\iota}(g^2)$ has distinct eigenvalues. Hence either $g$ or $hg$ satisfies (\ref{second property choosing g for 2nd ERL prop}), which shows the claim.
\end{proof}
Replacing $g$ with $hg$ as in the claim, we may now assume $g$ satisfies both (\ref{first property choosing g for 2nd ERL prop}) and (\ref{second property choosing g for 2nd ERL prop}).
By Lemma \ref{lem:inflation with tau} and inflation-restriction, $c$ has nonzero image in $$H^1(F(\rho_{\pi}, \ad^0\rho_{\tau,\iota}), T_\pi)^{G_\Q} = \Hom_{G_\Q}(\Gal(\overline \Q/F(\rho_{\pi}, \ad^0\rho_{\tau,\iota})), T_\pi),$$ and because $V_{\pi,\p}$ is absolutely irreducible, 
 there exists $h\in G_F$ such that $\ad^0\rho_{\tau,\iota}(h) = \rho_{\pi}(h) = 1$ and $c(h)$ has nonzero component in the 1-eigenspace for $g$. Then  either $g$ or $hg$ satisfies (\ref{first property choosing g for 2nd ERL prop}), (\ref{second property choosing g for 2nd ERL prop}), and (\ref{third property choosing g for 2nd ERL prop}), which proves the proposition.

\end{proof}

Finally, we have the endoscopic analogue of Proposition \ref{prop:choosing_g_2ERL}.
\begin{prop} \label{prop:choosing_g_2ERL_endo}
 Let $\pi$, $\iota$, $E_0$, and $\p$ be as above, with    $\pi$  endoscopic associated to  pair $(\pi_1,\pi_2)$ of automorphic representations of $\GL_2(\A_\Q)$; and assume that $E_0$ is a common strong coefficient field of 
 $\pi_1$ and $\pi_2$. Let $j = 1$ or 2, and suppose there exist admissible primes for $\rho_{\pi,\p}$ which are BD-admissible for $\rho_{\pi_j} = \rho_{\pi_j,\p}$.
  
  Suppose given the following data:
\begin{itemize}[label = $\circ$]
    \item A quadratic field $F\not\subset \Q(\rho_{\pi})$.
    \item A cuspidal automorphic representation $\tau$ of $\GL_2$ whose archimedean component is discrete series of weight at least 2, such that $\tau$ does not have CM by any quadratic field $K \subset F(\rho_{\pi})$. 
\item A $G_\Q$-stable $O_\p$ lattice $T_{\pi_j} \subset V_{\pi_j, \p}$, and a non-torsion cocycle $c\in H^1(\Q, T_{\pi_j})$.
\end{itemize}
    Then there exists an element $g\in G_\Q$ such that:
    \begin{enumerate}
        \item $g$ is is admissible for $\rho_{\pi}$ and BD-admissible for $\rho_{\pi_j}$, and has nontrivial image in $\Gal(F/\Q)$.
        \item $\rho_{\tau, \iota}(g^2)$ has distinct eigenvalues.
        \item $c(g)$ has nonzero component in the 1-eigenspace for $g$. 
    \end{enumerate}    
\end{prop}

\begin{proof}
Without loss of generality, suppose $j = 1$. Clearly there exists $g\in G_\Q$ satisfying (\ref{first property choosing g for 2nd ERL prop}), such that, if $\tau$ has CM by an imaginary quadratic field $K$, $g$ has trivial image in $\Gal(K/\Q)$. 
    We next claim:
    \begin{claim}
        There exists $g\in G_\Q$ satisfying (\ref{first property choosing g for 2nd ERL prop}) and (\ref{second property choosing g for 2nd ERL prop}).
    \end{claim}
\begin{proof}
    If $F(\rho_{\pi}) \cap F(\ad^0 \rho_{\tau, \iota})$ is finite, then we conclude using the same argument as for the claim in the proof of Proposition \ref{prop:choosing_g_2ERL}. 
    By Lemma \ref{lem:unentangled}, we may therefore assume that, for $i = 1$ or 2,
    $\pi_i \cong \tau^\sigma \otimes \chi$ for some finite-order Hecke character $\chi$ and automorphism $\sigma \in \automorphisms(\overline \Q/\Q)$. In this case, $\tau$ is necessarily non-CM (because its CM field would be contained in $\Q(\rho_{\pi_i})$), so $\pi_1$ and $\pi_2$ cannot both be CM; and it suffices to show there exists $g\in G_\Q$ 
satisfying (\ref{first property choosing g for 2nd ERL prop}), such that $\rho_{\pi_2}(g^2)$ has distinct eigenvalues. Hence it suffices to show that $\Q(\rho_{\pi_1}) \cap \Q(\rho_{\pi_2})$ is finite over $\Q(\mu_{p^\infty})$; and this follows from an argument very similar to Lemma \ref{lem:endoscopic_independence}, using that $\pi_1$ and $\pi_2$ are not both CM. 
\end{proof}
Now take $g$ as in the claim. By Lemma \ref{lem:inflation with tau}, $c$ has nonzero image in
$$H^1(F(\rho_{\pi}, \ad^0 \rho_{\tau,\iota}), T_{\pi_1}).$$ Arguing as in Proposition \ref{prop:choosing_g_2ERL} and using the absolute irreducibility of $\rho_{\pi_1}$, the proposition follows. 
\end{proof}

\bibliographystyle{plain} 
\bibliography{mybib}

\end{document}